\documentclass{memo-l}
\usepackage{tikz-cd}
\usepackage{amsmath}
\usepackage{fourier}
\usepackage{amssymb}
\usepackage{amscd}
\usepackage{amsthm}
\usepackage[centertags]{amsmath}
\usepackage{amsfonts}
\usepackage{newlfont}
\usepackage{graphicx}
\usepackage{amsfonts, amssymb}
\usepackage{mathrsfs}
\usepackage{latexsym}
\usepackage{tikz}
\usepackage{verbatim}
\usepackage[all]{xy}
\usepackage{enumitem}
\usepackage[colorlinks=true,linkcolor=colorref,citecolor=colorcita,urlcolor=colorweb]{hyperref}
\definecolor{colorcita}{RGB}{21,86,130}
\definecolor{colorref}{RGB}{5,10,177}
\definecolor{colorweb}{RGB}{177,6,38}

\numberwithin{section}{chapter}
\numberwithin{subsection}{section}

\newtheorem{theorem}{Theorem}[section]

\newtheorem{proposition}[theorem]{Proposition}
\newtheorem{corollary}[theorem]{Corollary}
\newtheorem{lemma}[theorem]{Lemma}

\theoremstyle{definition}
\newtheorem{remark}[theorem]{Remark}
\newtheorem{example}[theorem]{Example}

\theoremstyle{remark}

\newcommand{\dis}{\displaystyle}

\newcommand{\Bb}{\mathcal B}
\newcommand{\Jj}{\mathcal J}

\newcommand{\PP}{\mathcal P}

\newcommand{\NN}{\mathbb N}
\newcommand{\EE}{\mathbb E}
\newcommand{\zN}{\mathbb N}

\newcommand{\CC}{\mathbb C}
\newcommand{\DD}{\mathbb D}
\newcommand{\RR}{\mathbb R}
\newcommand{\jj}{\mathbf{j}}
\newcommand{\ii}{\mathbf{i}}

\usepackage{mathtools}

\DeclareMathOperator{\mon}{mon}
\DeclareMathOperator{\card}{card}
\DeclareMathOperator{\re}{Re}

\DeclareMathOperator{\id}{\mathrm{id}}
\DeclarePairedDelimiter\floor{\lfloor}{\rfloor}

\newcommand{\bi}{\mathbf i}
\newcommand{\bj}{\mathbf j}

\newcommand{\chimon}{\chi_{\mon}}

\newcommand{\Pp}{\mathcal{P}}

\newcommand\restrict[1]{\raisebox{-.5ex}{$\vert$}_{#1}}

\newcommand{\sinc}{\operatorname{sinc}}

\DeclareMathOperator{\sign}{sign}

\def\N{\mathbb{ N}}
\def\R{\mathbb{ R}}

\newcommand{\ccc}{\mathbf c}

\begin{document}
\title{Projection constants for spaces of multivariate 
polynomials}

\author{A. Defant, D. Galicer, M. Mansilla, M. Masty{\l}o and S. Muro}

\date{}

\thanks{The research of the fourth author was supported by the National Science Centre (NCN), Poland, Grant no. 2019/33/B/ST1/00165; the second, third and fifth author were supported by CONICET-PIP 11220200102336 and PICT 2018-4250. The research of the fifth author is additionally supported by ING-586-UNR.}

\begin{abstract}
\smallskip

\noindent 
The general problem we address  is to develop new methods in the study of projection constants of Banach spaces of multivariate polynomials. The relative projection constant $\boldsymbol{\lambda}(X,Y)$ of a~subspace $X$ of a~Banach $Y$  is the smallest norm among all possible projections on $Y$ onto $X$, and  the projection constant $\boldsymbol{\lambda}(X)$ is the supremum of all relative projection constants of $X$ taken with respect to all possible super spaces $Y$. This is one of the most significant notions of modern Banach space theory  and one that has been intensively studied since the birth of abstract operator theory. We focus on  projection constants of Banach spaces of multivariate polynomials formed either by trigonometric polynomials $f(g)=\sum_{\gamma \in E} \widehat{f}(\gamma) \gamma(g)$ defined on a~given compact topological group $G$, which have Fourier coefficients $\widehat{f}(\gamma)$ supported in an a~priori given finite set $E$  of characters; or analytic polynomials $P(z)=\sum_{\alpha\in J}c_\alpha(P)\,z^\alpha$, which are  defined on a~fixed Banach space $X_n = (\mathbb{C}^n, \|\cdot\|)$ and have monomial coefficients $c_\alpha(P)$
supported in an a priori given finite set $J \subset \mathbb{N}_0^n$ of multi indices. Depending on the underlying structure (of the group, Banach space or index set), the primary  goal is to prove precise formulas or asymptotically optimal estimates. Our general setting of multivariate polynomials is flexible enough to handle a~wide variety of Banach spaces of polynomials, including analytic polynomials on polydiscs, Dirichlet polynomials on the complex plane, and polynomials on Boolean cubes $\{-1,+1\}^n$. Moreover, we  combine our techniques with the theory of harmonic polynomials to get an explicit formula for the projection constant of  the trace class on the $n$-dimensional complex Hilbert space together with their precise asymptotic  whenever $n$ tends to infinity.
The modern theory of projection constants is intimately linked with various important invariants of local Banach space theory.  The methods developed here enable us to prove new estimates for particularly important invariants such as the unconditional basis constant and the Gordon-Lewis constant for Banach spaces of multivariate polynomials. Techniques from different areas of analysis are used -- including local Banach space theory, probability theory, harmonic and complex analysis, analytic number theory, interpolation theory and combinatorics. Our work is divided into
ten chapters which are presented independently but interact closely with one another.
\end{abstract}


\maketitle

\tableofcontents

\chapter*{Introduction}
\label{Introduction}

The study of complemented subspaces $X$ of a  Banach space $Y$ and their projection constants  has a~long history going  back to the beginning of operator theory in Banach spaces. At first recall that Banach in his famous treatise \cite{Banach} of 1932, asked whether every '$B$-space' (Banach space) has an uncomplemented subspace. The first example of such an uncomplemented subspace was given in 1933 by Banach and Mazur. They proved in their joint article \cite{Banach-Mazur} that any copy of $L_1[0,1]$ in $C[0, 1]$ is uncomplemented in $C[0,1]$. In 1934 Fichtenholz and Kantorovitch \cite{fichtenholz1934operations} showed that $C[0, 1]$ is uncomplemented in $L_\infty[0, 1]$. Later more concrete spaces with uncomplemented subspaces were given by Murray \cite{Murray} and  Sobczyk \cite{Sobczyk}. Regarding the problem of the existence of non-trivial complemented subspaces in  Banach spaces, let us recall the famous result due to  Lindenstrauss and Tzafriri \cite{LT1}, which says that an infinite dimensional~Banach space $X$ is isomorphic to a~Hilbert space if and only if any closed subspace of $X$ is complemented in $X$. In general, the description of the complemented subspaces of an infinite dimensional Banach space, which is not isomorphic to a~Hilbert space, is a~difficult problem. We point out that the first complete characterization of all complemented subspaces of $c_0$ and $\ell_p$ with $1\leq p <\infty$ with $p\neq 2$ was obtained by Pe{\l}czy\'nski \cite{pelczynski1960projections}; for the space $\ell_{\infty}$ such a~description was given by Lindenstrauss \cite{Lindenstrauss}.
 In these cases, the results state that every  complemented subspace of $\ell_p$ (resp. $c_0$) is isomorphic to $\ell_p$ (resp. $c_0$). Within the study of projection constants of Banach spaces, it  is particularly important to know the nature of the complemented subspaces of $C(K)$-spaces.
We point out that even in the case of $C[0, 1]$ a complete  description of all complemented subspaces is unknown. Let us mention the result due to Rosenthal which states that every complemented subspace of $C[0, 1]$ with a~non-separable dual is isomorphic to $C[0, 1]$ (see \cite{Rosenthal1972}, and also the  interesting survey by Rosenthal \cite [Chapter 5] {Rosenthal2003}).

In connection with Banach's question an important notion was introduced, which has roots in Murray's article \cite{Murray}. If $X$ is
a~complemented subspace of a~Banach space $Y$, then the relative projection constant of $X$ in $Y$ is defined by
\[
\boldsymbol{\lambda}(X, Y) = \inf\big\{\|P\|: \,\, P\in L(Y, X),\,\, P|_{X} = \id_X\big\}\,,
\]
where $\id_X$ is the identity operator on $X$.
The following straightforward result shows the intimate link  between projection constants and extensions of linear operators: For every Banach space $Y$ and its subspace $X$ one has
\begin{align*}
\boldsymbol{\lambda}(X, Y) = \inf\big\{c>0: \,\,\text{$\forall\, T \in \mathcal{L}(X, Z)$ \,\, $\exists$\,\, an extension\,
$\widetilde{T}\in \mathcal{L}(Y, Z)$\, with $\|\widetilde{T}\| \leq c\,\|T\|$}\big\}\,.
\end{align*}
It is worth noting that projections play an important role in approximation theory. In fact, if $Y$ is a~Banach space and $P\colon Y \to Y$ a~projection onto a subspace $X$, then the error $\|y - Py\|_Y$ of approximation of an element $y\in Y$ by $Py$ satisfies
\[
\|y- Py\|_Y \leq \|\id_Y - P\| \, \text{dist}(y, X) \leq (1 + \|P\|)\, \text{dist}(y, X)\,,
\]
where $\text{dist}(y, X) =\inf\{\|y-x\|_Y \,: \, x\in X\}$. This estimate motivates the problem of minimizing $\|P\|$, and any
projection  $P_0\colon Y \to Y$ onto $X$ such that $\|P_0\| = \boldsymbol{\lambda}(X, Y)$, is said to be a minimal projection of $Y$ onto $X$.

The (absolute) projection constant of $X$ is given by
\[
\boldsymbol{\lambda}(X) := \sup \,\,\boldsymbol{\lambda}(I(X),Y)\,,
\]
where the supremum is taken over all Banach spaces $Y$ and isometric embeddings $I\colon X \to Y$.

Any Banach space $X$ can be embedded isometrically into $\ell_\infty(S)$, where  $S$ is some  nonempty set which in general depends on $X$. For  every separable $X$, we may choose $S=\mathbb{N}$, and  if $X$ is finite dimensional, then there clearly is 
a~finite set~$S$. Throughout the paper, we  use the fact that, if $S$ is a~nonempty set for which the Banach space $X$ is  isometrically isomorphic to a~subspace $Z$ of $\ell_{\infty}(S)$, then 
\[
\boldsymbol{\lambda}(X) = \boldsymbol{\lambda}(Z, \ell_\infty(S))\,.
\]
Indeed, this is due to the fact that $\ell_{\infty}(S)$ is $1$-injective.
So finding $\boldsymbol{\lambda}(X)$ is equivalent to finding the norm of a~minimal projection from $\ell_\infty(S)$ onto an isometric copy of
$X$ in $\ell_\infty(S)$. However, this is a~non-trivial problem in general.

\noindent {\bf Objects of desire.}
Since ever polynomials play an outstanding role in various branches of mathematics, and this development continues unabated in numerous
modern developments in mathematics -- among others in approximation theory, analytic number theory, harmonic analysis, differential
geometry, learning theory, quantum information or coding theory.

We also point out that  in recent years  multivariate polynomials play an important role in combinatorics.  
See the  excellent book  \cite{Guth2015} by Guth  which carefully explains  grounbreaking 
progress in combinatorial geometry coming from an unexpected connection of  polynomials with  algebraic 
geometry. We note here that Guth and Katz in \cite{GuthKatz} introduce
so-called polynomial partitioning techniques 
to solve almost completely the Erd\"os distinct distances problem. They prove that the number of distinct 
distances determined by $n$ points in the plane is at least $O\big(\frac{n}{\log n}\big)$. 
Among others this technique is used 
in \cite{Guth2016, Guth2018} to get improvements  of the known restriction estimates for the paraboloid 
in high dimensions related to Stein’s famous restriction conjecture in Fourier analysis.

As indicated by the title of this work, we try to study projection constants of spaces of   multivariate polynomials. What do we mean
by Banach spaces of multivariate polynomials? In fact we basically look at two types of such spaces.

The first type consists of linear spaces of trigonometric polynomials defined on a compact topological group $G$, which all have Fourier
coefficients supported in an a priori given finite set $E$  of characters in the dual group $\widehat{G}$. That is, polynomials of the form
\[
P(g) = \sum_{\gamma \in E} \widehat{P}(\gamma) \gamma(g) , \,\quad  g \in G\,.
\]
The space of all such polynomials equipped with the uniform norm on $G$ forms a~finite dimensional Banach space and is denoted by
$$\text{Trig}_E(G)\,.$$ 
To see a concrete example, look at the $n$-dimensional circle group (= torus) $\mathbb{T}^n$. In this case
every character in $\widehat{\mathbb{T}^n}$ may be identified  with a unique  multi index $\alpha = (\alpha_1, \ldots, \alpha_n)\in \mathbb{Z}^n$,
and consequently every polynomial $P \in \text{Trig}_{\mathbb{Z}^n}(\mathbb{T}^n)$ has the form
\[
P(z) = \sum_{\alpha \in \mathbb{Z}^n} \widehat{P}(\alpha) z^\alpha,\, \quad  z \in \mathbb{T}^{n}\,.
\]
If we restrict to the 'index set' $E$ of all multi indices $ \alpha \in \mathbb{N}_0^n$ such that $|\alpha| = \sum \alpha_j \leq m$,
then 
$$\text{Trig}_{\leq m}(\mathbb{T}^n)$$ 
consists of all so-called analytic trigonometric polynomials of degree $\leq m$. Why analytic?
Recall that $\mathbb{T}^n$ is the distinguished boundary of the $n$-dimensional polydisc $\mathbb{D}^n$. Then it is obvious that the
Banach space $\text{Trig}_{\leq m}(\mathbb{T}^n)$ by the maximum modulus theorem identifies  isometrically and coefficient preserving
with the linear space of all analytic polynomials
\[
P(z) = \sum_{ \alpha \in \mathbb{N}_0^n: |\alpha| \leq m} \frac{\partial^\alpha P(0)}{\alpha!} \,z^\alpha, \,\quad  z \in \mathbb{C}^n
\]
endowed with supremum norm taken on $\mathbb{D}^n$. In other terms,
\[
\text{Trig}_{\leq m}(\mathbb{T}^n) = \mathcal{P}_{\leq m} (\ell_\infty^n)\,,
\]
where the latter symbol stands for all  polynomials $P\colon \ell_\infty^n \to \mathbb{C}$ of degree at most $m$ and the norm is the
supremum norm with respect to the unit ball of $\ell_\infty^n$.

This, in a generic way, leads us  to the second type of finite dimensional Banach spaces of multivariate polynomials we are interested in. Given a finite index set $J \subset \mathbb{N}_0^n$ and a~Banach space $X_n = (\mathbb{C}^n,\|\cdot\|)$
with open unit ball $B_{X_n}$, we consider the~Banach space
$$\mathcal{P}_J(X_n)$$ of all polynomials $P\colon X_n \to \mathbb{C}$ 'supported on $J$' -- that is polynomials of the form
\[
P(z) =  \sum_{ \alpha \in J} c_\alpha(P) z^\alpha, \,\quad  z \in \mathbb{C}^n\,,
\]
equipped with the norm
\[
\|P\|_{B_{X_n}} := \sup\big\{|P(z)|\,: \, z\in B_{X_n}\big\}\,.
\]
We will see that this abstract approach  to  spaces of polynomials supported on certain index sets incorporates  many apparently unrelated
settings -- like e.g., Dirichlet polynomials $\sum_{n} a_n e^{-\lambda_n s}$ with respect to arbitrary frequencies $\lambda=(\lambda_n)$ or functions
$f\colon \{-1, 1\}^N \to \mathbb{R}$ on $N$-dimensional Boolean cubes under the restriction of various different complexity measures.

\smallskip

\noindent {\bf Aim.}
Our primary goal is to provide concrete formulas and asymptotically correct  estimates for the projection constants  $\boldsymbol{\lambda}\big(\text{Trig}_E(G)\big)$
as well as $\boldsymbol{\lambda}\big(\mathcal{P}_J(X_n)\big)$. Doing this, we  as a by-product obtain a number of independently interesting results -- in particular, results on the unconditional basis constants of $\text{Trig}_E(G)$
and $\mathcal{P}_J(X_n)$.

More precisely, this aim forces us to use techniques from different areas of analysis -- among others from harmonic analysis, probability theory, complex analysis, analytic number theory, local Banach space theory (in particular, the metric theory of tensor products and operator ideals) and interpolation theory.
In fact, we relate the projection constants of various Banach spaces of multivariate polynomials with (a priori unrelated) characteristics like e.g.,
Sidon constants, $\Lambda(2)$-sets, unconditional basis constants, Gordon-Lewis constants, concavity and convexity constants, or multi dimensional Bohr radii.

Describing our goal in more detail, we  mention three examples that in fact served as a significant inspiration for our work. To our best knowledge,
 these three
asymptotically optimal estimates for projection constants of multivariate polynomials 
 seem to be the so far only known  overall results we aim for.

The first one is a theorem of  Lozinski and Kharshiladze (see \cite[IIIB. Theorem 22]{wojtaszczyk1996banach} and \cite{natanson1961constructive}),
which shows a  precise estimate of the projection constant of $\text{Trig}_{\leq m}(\mathbb{T})$, the space of trigonometric polynomials on
$\mathbb{T}$ of degree at most $m$ equipped with the sup-norm:
\begin{equation} \label{LoKa}
\boldsymbol{\lambda}\big(\text{Trig}_{\leq m}(\mathbb{T})\big)= \frac{4}{\pi^2} \log(m+1) + o(1).
\end{equation}

The second result is a~beautiful  formula from the eighties due to  Ryll and Wojtaszczyk~\cite{ryll1983homogeneous} (see also \cite[IIIB. Theorem 15]{wojtaszczyk1996banach}) giving the exact value of the projection constant of the space $\mathcal{P}_{m}(\ell^n_2)$ of all $m$-homogeneous
polynomials on the $n$-dimensional complex Hilbert space $\ell_2^n$:
\begin{equation} \label{RWRW}
    \boldsymbol{\lambda}\big(\mathcal{P}_{m}(\ell^n_2)\big) = \frac{\Gamma(n+m) \Gamma(1+\frac{m}{2})}{\Gamma(1+m)\Gamma(n+\frac{m}{2})}.
\end{equation}
This result in fact was the key step to answer some questions of the time within the theory of the Hardy spaces  and Bloch spaces on the
complex ball.

Motivated by the study of multivariate variants of Bohr's famous power series theorem,  the first named author together with Frerick
\cite{defant2011bohr} (see also \cite[Section 22.2]{defant2019libro}) provided (implicitly) the following asymptotically sharp estimate
\begin{equation}\label{eq:proj const pols on lp}
\boldsymbol{\lambda}\big(\mathcal P_m(\ell_r^n)\big)  \sim_{C^m} \left( 1+\frac{n}{m} \right)^{m\left(1-\frac{1}{\min\{r,2\}}\right)}
\end{equation}
of the projection constant of $\mathcal P_m(\ell_r^n)\,, \,\,1 \leq r \leq \infty$,
where $\sim_{C^m}$ stands for equivalence up to $C^m$ with a constant $C$ only depending on $r$.

Before we give a brief outlook on the main results we obtained, let us continue with a short history on a few selected highlights raised
since the days of Banach and Mazur in connection with projection constants of
Banach spaces.

\smallskip

\noindent {\bf Some  history.}
Today Banach space theory offers a~rich toolbox which under certain topological or geometrical assumptions on a~given Banach space $Y$
allows to construct non-complemented subspaces of this space -- even non-complemented subspaces with additional properties. For example,
any Banach space $Y$ with the Dunford-Pettis property (that is, a~space such that every weakly compact operator from $Y$ into any Banach space maps weakly compact sets onto compact sets) does not contain any infinite dimensional complemented reflexive subspace. The classical examples of Banach spaces with the Dunford-Pettis property are $C(K)$-spaces and $L_1(\mu)$-spaces (see \cite{Diestel}).  Another important class of Banach spaces are Grothendieck spaces. A Banach space $X$ is a Grothendieck space whenever  weak$^{*}$-convergence and weak-convergence of sequences  in the dual space $X^{*}$ coincide.  Grothendieck in his seminal paper \cite{Grothendieck1953} proved that for any nonempty set $S$ the space of bounded functions on $S$, and more generally, the space $C(K)$ of continuous functions on a~compact, extremally disconnected Hausdorff space, have
this property. Other examples of Grothendieck spaces include the
Hardy space $H^{\infty}(\mathbb{D})$ (see \cite{Bourgain1983}) and von Neumann algebras (see \cite[Corollary 7]{Pfitzner1994}).
It is well-known that separable Grothendieck spaces are reflexive, and that the class of all Grothendieck spaces  inherits quotients. As a~consequence, separable non-reflexive subspaces of Grothenedick spaces are non-complemented. Examples of Grothendieck spaces with the Dunford-Pettis are $\ell_\infty(S)$-spaces for any nonempty set $S$. In particular, this result recovers the well-known and important fact that $C[0, 1]$ (resp. $c_0$) is not complemented in $\ell_\infty[0,1]$ (resp. $\ell_\infty$).
Observe that from the above discussion, it follows immediately that if $X$ is a~reflexive or separable Banach space with  $\lambda(X)<\infty$, then $X$ is finite dimensional.

We once again mention  the complemented subspace theorem due to Lindenstrauss and Tzafriri \cite{LT71}, which states that a~Banach space $X$ is isomorphic to a~Hilbert space if and only if any closed subspace of $X$ is complemented in $X$. A much more recent highlight of modern Banach theory is the solution to a famous problem asking whether every infinite dimensional Banach space $X$ admits a~non-trivial projection $P$, that is, some  projection $P$ such  that $P(X) $ and $X/P(X)$ are both infinite dimensional. It is an implausible deep result due to Gowers and Maurey in \cite{gowers1993unconditional} and \cite{gowers1997banach} showing that there is an example of a~reflexive Banach space $X$ which is hereditarily indecomposable. This means that every closed subspace of $X$ is indecomposable (i.e., there are no infinite dimensional closed subspaces $M$ and $N$ for which
$M \cap N=\{0\}$ and $M+N$ is closed). In particular, this space does not admit a~non-trivial projection, so it solves the mentioned problem.

The Hahn-Banach theorem shows that all $1$-dimensional subspaces of a given Banach space are $1$-complemented. As proved by Gr\"unbaum
\cite{grunbaum}  the projection constant of a~$2$-dimensional real Banach space, whose unit ball is the regular hexagon, equals  $\frac{4}{3}$,
and he  conjectured that any $2$-dimensional real Banach space admits a~projection constant of at most $\frac{4}{3}$. A positive answer
to this long standing conjecture was given by  Chalmers and Lewicki in \cite{ChL}. In the case of finite dimensional spaces, the property of
having norm-one projections is very rare. We note that Bohnenblust \cite{Bohnenblust} proved that there are subspaces of $\ell_p^n$, which do
not have any non-trivial $1$-complemented subspace. It is perhaps worth noting the following general result proved by Boszny and Garay \cite{bosznay1986norms}:
For any $\varepsilon \in (0, 1)$ and  any finite dimensional real Banach space $(X, \|\cdot\|)$ there exists a~norm $\|\cdot\|^{*}$ such that
\[
(1-\varepsilon)\|x\| \leq \|x\|^{*} \leq (1 + \varepsilon) \|x\|, \quad\, x\in X
\]
and $(X, \|\cdot\|^{*})$ does not have any non-trivial $1$-complemented subspace.

Finite-rank projections naturally appear in the study of Banach spaces with basis. Recall that a Banach spaces $X$ has a~basis if and only if there is
a~bounded sequence $\{P_n\} \subset L(X)$  of finite-rank projections on $X$ with $\text{rank}(P_n) = n$ such that $P_{n-1}(X) \subset P_n(X)$
for each $n>1$ and $\bigcup_n P_n(X)$ is dense in $X$. Clearly, every sequence $\{e_n\} \subset P_n(X) \cap \text{ker} P_{n-1}$ with $e_n\neq 0$
forms a~basis of $X$.

General bounds for projection constants of various finite dimensional Banach spaces were studied by many authors. The most fundamental general
upper bound is due to Kadets and Snobar \cite{kadecsnobar}: For every $n$-dimensional Banach space $X_n$ one has
\begin{equation} \label{kadets1}
\boldsymbol{\lambda}(X_n) \leq \sqrt{n}\,.
\end{equation}
In contrast, K\"onig and Lewis \cite{koniglewis} showed that for any Banach space $X_n$ of dimension $n \ge 2$ the strict
inequality $\boldsymbol{\lambda}(X_n) < \sqrt{n}$ holds, and this estimate was improved by Lewis \cite{lewis} showing
\[
\boldsymbol{\lambda}(X_n) \leq \sqrt{n}\,\left(1 - \frac{1}{n^2} \bigg(\frac{1}{5}\bigg)^{2n + 11}\right)\,.
\]
Regarding the Kadets-Snobar estimate~\eqref{kadets1}, we also mention a~remarkable result of\,Pisier \cite[Corollary 10.8]{pisier1986factorization}, who 
constructed an infinite dimensional Banach space $X$ such that for some  $\delta \in (0, 1)$ and all  finite rank projections 
$P\colon X \to X$
\[
\|P\| \geq \delta \sqrt{\text{rank}\,P}\,.
\]
K\"onig \cite{konig1985spaces} proved that the Kadets-Snobar estimate~\eqref{kadets1} is asymptotically  best possible. More precisely, he proved that
there exists a sequence $(X_{n_k})_{k=1}^{\infty}$ of finite-dimensional real Banach spaces such that $\text{dim}(X_{n_k})
= n_k$, where $n_k\to \infty$ as $k\to \infty$, and
\[
\lim_{k\to \infty} \frac{\boldsymbol{\lambda}(X_{n_k})}{\sqrt{n_k}} = 1\,.
\]
Another  surprising result in this direction is due to  Figiel, Lindenstrauss and Milman \cite{FLM}. It guarantees  that there is a~universal constant
$c\in (0, 1)$ such that every $n$-dimensional Banach space $E$ contains some subspace $E_0$ such that $\boldsymbol{\lambda}(E_0) \geq c\,\sqrt{n}$.

Nevertheless, it is often very hard to compute the projection constant $\boldsymbol{\lambda}(X_n)$ of a~concrete $n$-dimensional
Banach space $X$.

The exact values of $\boldsymbol{\lambda}(\ell_2^n)$ and $\boldsymbol{\lambda}(\ell_1^n)$ were computed by Gr\"unbaum
\cite{grunbaum} and Rutovitz \cite{rutovitz}: In the complex case
\begin{align}\label{grunbuschC-A}
\boldsymbol{\lambda}\big(\ell_2^n(\mathbb{C})\big)  = n \int_{\mathbb{S}_n(\mathbb{C})} |x_1|\,d\sigma
= \frac{\sqrt{\pi}}{2}   \frac{n!}{\Gamma(n + \frac{1}{2})}\,,
\end{align}
where $d\sigma$ stands for the normalized surface measure on the sphere  $\mathbb{S}_n(\mathbb{C})$
in $\mathbb{C}^n$, and
\begin{align}\label{grunbuschC-B}
\boldsymbol{\lambda}\big(\ell_1^n(\mathbb{C})\big)  = \int_{\mathbb{T}^n} \Big|\sum_{k=1}^{n} z_k\Big|\, dz  
= \int_{0}^{\infty} \frac{1 -J_0(t)^n}{t^2} dt\,,
\end{align}
where $dz$ denotes the normalized Lebesgue measure on the distinguished boundary $\mathbb{T}^n$ in $\mathbb{C}^n$
and $J_0$ is the zero Bessel function defined by
$
J_0(t) = \frac{1}{2\pi} \int_{0}^{\infty} \cos( t \cos \varphi) d \varphi\,.
$
The corresponding real constants are different:
\begin{align}\label{grunbuschR}
& \boldsymbol{\lambda}\big(\ell_2^n(\mathbb{R})\big)  =  n \int_{\mathbb{S}_n(\mathbb{R})} |x_1|\,d\sigma
= \frac{2}{\sqrt{\pi}}   \frac{\Gamma(\frac{n+2}{2})}{\Gamma(\frac{n+1}{2})} \\
& \boldsymbol{\lambda}\big(\ell_1^n(\mathbb{R})
\big)  =
\begin{cases}
\boldsymbol{\lambda}\big(\ell_2^n(\mathbb{R})\big),  &  \text{$n$ odd}\\[2mm]
\boldsymbol{\lambda}(\ell_2^{n-1}(\mathbb{R})),  &  \text{$n$ even}\,.\\[2mm]
\end{cases}
\end{align}
Gordon \cite{gordon} and Garling-Gordon \cite{garlinggordon} determined the asymptotic growth of
$\boldsymbol{\lambda}\big(\ell_p^n\big)$ for $1<p<\infty$ with $p\neq 1,2,\infty$:
\begin{equation} \label{gordongarling}
\boldsymbol{\lambda}\big(\ell^n_{p}\big) \sim n^{\min\big\{\frac{1}{2}, \frac{1}{p} \big\}}\,.
\end{equation}
K\"onig, Sch\"utt and Tomczak-Jagermann \cite{konig1999projection} proved that for  $1 \le p \le 2$
\begin{align}\label{koenigschuetttomczak}
  \lim_{n\to \infty} \frac{\boldsymbol{\lambda}\big(\ell_p^n\big)}{\sqrt{n}} = \gamma\,,
\end{align}
where $\gamma = \sqrt{\frac{2}{\pi}}$ in the real  and $\gamma= \frac{\sqrt{\pi}}{2}$ in the complex case.
We refer to \cite{lewickimastylo} for extensions of these results within the setting of some classes of rearrangement invariant
Banach sequence spaces including Orlicz spaces. For an extensive treatment on all of this we  refer  to the excellent  monograph
\cite{tomczak1989banach}
of Tomczak-Jaegermann .

\smallskip

\noindent{\bf Outlook.}
We now give a brief description of some of the main results.  Our work is divided into ten chapters. On the one hand, all chapters
are of independent interest and up to some point even self contained -- on the other hand, they interact with each other, and our hope is that all contributions together
result in an interesting overall picture on projection constants of Banach spaces of  multivariate polynomials.

\noindent {\bf Chapter~\ref{chapter: general preliminaries}.}
We start with some  general preliminaries which are used all over the text.
Those preliminaries that are not in constant use, are only introduced locally.

\noindent {\bf Chapter~\ref{trigo}.}
We consider the space $\text{Trig}_E(G)$ of trigonometric polynomials over a compact abelian  group $G$,
which all have  Fourier transforms supported on a fixed set $E$ of characters. Following a classical
averaging technique developed by Rudin, we provide  (what we call) an  integral formula:
\[
\boldsymbol{\lambda}\big(\text{Trig}_E(G)\big)
= \int_G \, \Big|\sum_{\gamma \in E} \gamma (x)\Big|\,d\mathrm{m}(x)\,,
\]
where $m$ stands for the Haar measure on $G$.
As a first consequence, we  prove a couple of variants of the
 Lozinski-Kharshiladze  result from~\eqref{LoKa} for spaces of trigonometric polynomials on
 products of compact abelian groups, including the $n$-dimensional torus $\mathbb{T}^n$.
This chapter paves the way for later applications to many more concrete  spaces of polynomials (as e.g., spaces of Dirichlet polynomials, trace class operators or functions on  Boolean cubes $\{-1,1\}^{N}$).

\noindent {\bf Chapter~\ref{eukli}.}
Continuing Rudin's ideas from the preceding chapter, we mainly focus on spaces of multivariate polynomials on finite dimensional
Euclidean spaces. Given some finite index sets $J\subset  \mathbb{N}_0^n $, we  prove an exact integral formula for $\boldsymbol{\lambda}\big(\mathcal{P}_J(\ell_2^n)\big)$ 
under the assumption  that  $\mathcal{P}_{J} (\ell_2^n)$ is invariant under the action of the unitary group  $\mathcal{U}_n$
(all unitary $n \times n$ matrices):
\[
\boldsymbol{\lambda} \big(\mathcal{P}_{J}(\ell_2^n)\big)
= \sup_{z \in \mathbb{S}_n}
\int_{\mathbb{S}_n} \bigg| \sum_{k=1}^{m} \,\, 
\frac{(n-1 + k)!}{(n-1)!\,k!}
\sum_{\substack{\alpha \in J\\|\alpha|=k }} \frac{k!}{\alpha !} z^{\alpha} \overline{\xi}^{\alpha}
\bigg|\,d\sigma(\xi)\,,
\]
where $m= \max_{\alpha \in J}|\alpha|$
and $\sigma$ stands for the normalized Lebesgue
measure on the unit sphere $\mathbb{S}_n$ in $\ell_2^n$. Among others, we this way in the special case of $m$-homogeneous polynomials recover and extend the formula~\eqref{RWRW} of Ryll and Wojtaszczyk.

\noindent {\bf Chapter~\ref{Trace class operators}.}
We  study the space of harmonic polynomials on the group $\mathcal U_n$ of unitary $n \times n$ matrices. We show that this space of
polynomials is dense in the space $C(\mathcal U_n)$. Combining previous ideas (in particular, Rudin's averaging technique), we show an
integral formula for the projection constant of the space $\mathcal S_1(n)$, the trace class on the $n$-dimensional complex Hilbert space; namely
\[
\boldsymbol{\lambda}\big(\mathcal S_1(n)\big) = n \int_{\mathcal U_n} \vert \text{tr}(V) \vert \,d\mu \,,
\]
where $\mu$ stands for the Haar
probability measure on the group $\mathcal U_n$. This in fact is a non-commutative analogue of \eqref{grunbuschC-B}. Using a probabilistic
approach (and the  so-called Weingarten calculus), we get the limit formula
\[
\dis\lim_{n\to \infty} 
\frac{\boldsymbol{\lambda}\big(\mathcal S_1(n)\big)}{n}
= \frac{\sqrt{\pi}}{2}\,,
\]
which for $p=1$ is a  non-commutative analog of~\eqref{koenigschuetttomczak}.

\noindent {\bf Chapter~\ref{tensor-pro-me}.}
We show that for a large variety of index sets $J\subset  \mathbb{N}_0^n $ of degree at most $m$ and for a rather
large class of Banach spaces $X_n = (\mathbb C^n, \|\cdot\|)$, the projection constant of $\mathcal{P}_{J}(X_n)$,
up to a constant $C(m)$ only depending on $m$
and not on the dimension $n$, equals the $m$th power of
the  projection constant of the dual of $X_n$, that is,
\[
\boldsymbol{\lambda}\big(\mathcal{P}_{J}(X_n)\big)
\sim_{C(m)} \boldsymbol{\lambda}(X^\ast_n)^{m}\,.
\]
On this way, we collect independently interesting information on $\boldsymbol{\lambda}\big(\mathcal{P}_{J}(X_n)\big)$.
Among others, using factorization and polarization, we connect the projection constant of $\mathcal{P}_{J}(X_n)$ with the projection constant of its 'homogeneous building blocks'
$\mathcal{P}_{J_m}(X_n)$.
The main proofs are based on the theory of (multilinear) tensor products and operator ideal norms.

\noindent {\bf Chapter~\ref{Unconditionality}.}
We use the concepts of unconditionality and probability in Banach spaces as a  systematic tool to find lower bounds
for projection constants of spaces of multivariate polynomials. Doing this, we  relate the projection constant of  $\mathcal{P}_J(X_n)$
  with various important invariants
of local Banach space theory -- among others, the Gordon-Lewis constant ${\mbox{gl}}\big( \mathcal{P}_{J}(X_n)\big)$ and
the unconditional basis constant $\boldsymbol{\chimon}\big( \mathcal{P}_{J}(X_n)\big)$ of the monomial basis $(z^\alpha)_{\alpha \in J }$ of $\mathcal{P}_{J}(X_n)$. 
To see a very first sample, we mention that for every Banach space $X_n = (\mathbb C^n, \|\cdot\|)$
and index sets 
$J\subset  \Lambda(m,n) = \{\alpha \in \mathbb{N}_0^n \colon |\alpha|=m\} $ 
\[
   \boldsymbol{\chimon}\big( \mathcal{P}_{I}(X_n)\big)
 \,\le \,e 2^m \, \|\mathbf{Q}_{\Lambda(m,n),I}\|\,\,
 \boldsymbol{\lambda}\big( \mathcal{P}_{J}(X_n)\big)\,,
\]
where $I = \{\alpha \in \Lambda(m+1,n)\colon \exists \beta \in \Lambda(1,n) \textrm{ such that }
\alpha-\beta\in J\}$
and
$\mathbf{Q}_{\Lambda(m+1,n),I}: \mathcal{P}_{\Lambda(m+1,n)}(X_n) \to \mathcal{P}_{I}(X_n)$ is
 the canonical projection which annihilates monomial  coefficients with indices outside of $I$. The
 main advantage then is that we may use probabilistic estimates to find lower bounds of  
$ \boldsymbol{\chimon}\big( \mathcal{P}_{I}(X_n)\big)$
which in turn give lower bounds for 
$\boldsymbol{\lambda}\big( \mathcal{P}_{J}(X_n)\big)$.
Under further convexity assumptions on the underlying Banach spaces $X_n$ all this, for a large variety of Banach spaces $X_n$ and index sets $J$, leads to asymptotically
optimal estimates for $\boldsymbol{\lambda}\big(\mathcal{P}_J(X_n)\big)$ .
We in particular focus on the dependence of  $\boldsymbol{\lambda}\big(\mathcal{P}_{J}(X_n)\big)$  simultaneously on the degree $m$  and the dimension $n$, and prove several asymptotic results as these parameters go to infinite.

\noindent
{\bf Chapter~\ref{Part: Dirichlet polynomials and polynomials on the Boolean cube}.}
Given a frequency sequence $\lambda=(\lambda_n)$ and a finite subset  $J \subset \mathbb{N}$, we define the Banach space $\mathcal{H}_{\infty}^{J}(\lambda)$ of all Dirichlet polynomials supported on $J$. That is, all finite Dirichlet
polynomials
$D(s) = \sum_{n \in J} a_n e^{-\lambda_n s}$ with
complex coefficients $a_n$ and  the complex variable $s$, endowed  with the supremum norm of $D$ as a function on the right half plane in $\mathbb{C}$.
Based on  harmonic analysis on  so-called $\lambda$-Dirichlet groups, we get  the  integral formula
\[
\boldsymbol{\lambda}\big(\mathcal{H}_\infty^{J}(\lambda\big) ~ = ~ \dis\lim_{T \to \infty} \frac{1}{2T} \int_{-T}^T \Big|\sum_{n \in J} e^{-i\lambda_n t}\Big|\,dt\,,
\]
which we then apply to various concrete frequencies $\lambda$ and index sets $J$. Combining with a recent deep result of Harper from probabilistic analytic number theory, we for the space $\mathcal{H}_\infty^{\leq x}$ of all ordinary Dirichlet polynomials $D(s) = \sum_{n \leq x} a_n n^{-s}$ of length $x$ prove   the following asymptotically correct order of the projection constant:
\[
\boldsymbol{\lambda}\big(\mathcal{H}_\infty^{\leq x}\big) \sim \frac{\sqrt{x}}{(\log \log x)^{\frac{1}{4}}}\,.
\]

\noindent
{\bf Chapter~\ref{Polynomials on the Boolean cube}.}
Every real-valued function $f$ on the $N$-dimen\-sional Boolean cube $\{-1,1\}^{N}$
has  a
Fourier-Walsh expansion $f(x) = \sum_{S\subset \{1,\ldots,N\}}\widehat{f}(S) \prod_{k \in S} x_k \,.$
Given a~set $\mathcal{S}$ of subsets in  $\{1,\ldots,N\}$,  we study the  projection constant of $\mathcal{B}^N_{\mathcal{S}}$, denoting the Banach space of all functions $f: \{-1,1\}^{N} \to \mathbb{R}$ with  Fourier-Walsh expansions supported on $\mathcal{S}$ together with the supremum norm on $\{-1,1\}^{N}$. Based on an integral formula for $\boldsymbol{\lambda}\big(\mathcal{B}_{\mathcal{S}}^N\big)$, we outline how   to compute the exact asymptotic behaviour of $\boldsymbol{\lambda}\big(\mathcal{B}_{\{ S \colon  |S|= d\}}^N\big)$ as well as $\boldsymbol{\lambda}(\mathcal{B}_{\{ S \colon  |S|\leq  d\}}^N)$ for $d\in \mathbb{N}$ as $N \to \infty$.
As an example we mention the following asymptotic formula:
\[
\dis \lim_{N \to \infty} \frac{\boldsymbol{\lambda}\big(\Bb_{\{ S \colon  |S|=  2\}}^N\big) }{N} = \dis \lim_{N \to \infty} \frac{\boldsymbol{\lambda}\big(\Bb_{\{ S \colon  |S| \leq   2\}}^N\big) }{N} =  \sqrt{\frac{2}{\pi e}}.
\]
Another  by-product of our techniques are  independently interesting estimates for the Sidon constants of collections of Fourier-Walsh monomials $x^S, S \in \mathcal{S}$.

\noindent
{\bf Chapter~\ref{Part: Polynomial projection constants}.
}
We invent a   new tool  to estimate  the projection constant of  spaces $\mathcal{P}_J(X_n)$ of polynomials, which are defined on a Banach lattice $X_n = (\mathbb{C}^n, \|\cdot\|)$ and supported on an index set $J \subset \mathbb{N}_0^n$.
The so-called  polynomial projection constant
 of $\mathcal P_J(X_n)$ is given by 
\[
\widehat{\boldsymbol{\lambda}}\big(\mathcal P_J(X_n)\big)=\sup_{z \in B_{X_n}} \sum_{\alpha \in J} c_{X_n}(\alpha) \vert z^{\alpha}\vert\,,
\]
where $c_{X_n}(\alpha)$ 
is the reciprocal of the norm of the monomial $z^{\alpha}$ in $\mathcal P_J(X_n)$.
This  quantity leads to the  upper  bound 
\[
\boldsymbol{\lambda}\big(\mathcal{P}_{J}(X_n)\big) 
\leq \widehat{\boldsymbol{\lambda}}\big(\mathcal{P}_{J}(X_n)\big)
\]
of the  projection constant of $\mathcal P_{J}(X_n)$, which in many concrete situations is  good enough.
In fact, it turns out that the polynomial projection constant  as a  substitute of the projection constant
of $\mathcal{P}_J(X_n)$  allows to extend estimates like \eqref{eq:proj const pols on lp}
in a~systematic and comfortable way to wider ranges of spaces $X_n$  and index sets $J$. Underpinned by abstract theory, we examine this new invariant by primarily looking at concrete examples.

\noindent
{\bf Chapter~\ref{Part: Polynomials on Lorentz sequences spaces}.
}
We  focus on  polynomials, which are  defined on finite dimensional Lorentz sequence spaces $\ell_{r,s}^n$ and supported on a given finite index set $J$ of multi indices. Our aim is to obtain asymptotically  correct estimates  for  the projection constant and the unconditional basis constant of the  spaces $\mathcal P_J(\ell_{r,s}^n))$, which depend both, on the dimension $n$ and on the structure of $J$.
A natural question one may ask  is, whether  estimates like~\eqref{eq:proj const pols on lp} for spaces of $m$-homogeneous polynomials on  $\ell_r^n$, remain valid for $\ell_{r,s}^n$ and more general index sets $J \subset \NN_0^{(\mathbb{N})}$ regardless of the second parameter $s$
 -- that is,
\[
\boldsymbol{\lambda}\big(\mathcal{P}_{J_{\leq m}}(\ell_{r,s}^n)\big) 
\sim_{C^m} \left( 1+\frac{n}{m} \right)^{m\left(1-\frac{1}{\min\{r,2\}}\right)} \,.
\]
We see that if we restrict ourselves to tetrahedral indices, then this is in fact true. But it turns out that under  certain technical restrictions, this  even holds  in the general case. To do so we
have to overcome subtle  
technical difficulties, which are far more involved than those used  in the case $r=s$ and 
 require most  of the previously provided techniques .

\noindent
{\bf Chapter~\ref{Part: Bohr radii}.
}
Based on the result collected in the previous chapters, we  study  multivariate  Bohr radii.
 More precisely, given   a  Banach  sequence lattice $X$ and some  index set $J \subset \NN_0^{(\mathbb{N})}$, we denote by
 $H_{\infty}^{J}(B_{X_n})$ the Banach space of all bounded holomorphic functions $f: B_{X_n} \to \mathbb{C}$ with monomial coefficients
$\partial^\alpha f(0)/\alpha!$ supported on $J$, where $X_n = (\CC^n, \| \cdot \|)$ stands for the $n$th sections  of $X$. Extending the usual definitions, the Bohr radius of $B_{X_n}$ with respect to $J$ is then given by
\begin{equation*}
    K(B_{X_n},J) := \sup \Big\{ 0 < r < 1 \colon \sup_{z \in r  B_{X_n}}  \sum_{\alpha \in J} \Big|\frac{\partial^\alpha f(0)}{\alpha!} z^\alpha\Big| \le \|f\|_\infty \,\,\, \text{for all $f \in H^J_\infty (B_{X_n})$} \Big\}\,.
       \end{equation*}
       We recover and extend various results from the literature on  Bohr radii in high dimensions. Following  previous themes, we in particular study
       the behavior of such Bohr radii versus the unconditional basis and projection constants of  characteristic 'homogeneous building blocks'
       of $ H^J_\infty (B_{X_n})$,  but also versus the convexity/concavity of the underlying Banach sequence space $X$.
       Among others, we isolate a class of Banach sequence lattices $X$ and index sets $J$ such that
       $$
              \dis\lim_{n \to \infty} 
       \frac{K(B_{X_n},J)}{\sqrt{\log n/n}}  = 1\,,
             $$
            and we
 provide the asymptotically correct decay of the Bohr radii $K(B_{\ell_{r,s}^n},J)$ for a wide range of values of $1 \leq r, s \leq \infty$
and  index sets $J$. Compared to the case $r=s$, this seems far more challenging.

\bigskip

\chapter{General preliminaries} \label{chapter: general preliminaries}

\subsection{Comparison of sequences and index sets}\label{index-sets}
Given two sequences $(a_n)$ and $(b_n)$ of non-negative real numbers we write $a_n \prec b_n$, if there is a constant
$c>0$ such that $a_n \leq c\,b_n$ for all $n\in \mathbb{N}$, while $a_n \sim b_n$ means that $a_n \prec b_n$ and
$b_n \prec a_n$ holds. In the case that an extra parameter $m$ is also involved, for two sequences of non-negative real
numbers $(a_{n,m})$ and $(b_{n,m})$, we write $a_{n,m} \prec_{C(m)} b_{n,m}$ whenever there is a constant $C(m)>0$
(which depends exclusively on $m$ and not on $n$) such that $a_{n,m} \leq C(m) b_{n,m}$ for all $n,m \in \mathbb{N}$.
We use the notation $a_{n,m} \sim_{C(m)} b_{n,m}$ if $a_{n,m} \prec_{C(m)} b_{n,m}$ and $b_{n,m} \succ_{C(m)} a_{n,m}$.
We also write $a_{n,m} \prec_{C^m} b_{n,m}$ when there is a hypercontractive comparision, i.e., there is an absolute
constant $C>0$ such that $a_{n,m} \leq C^m b_{n,m}$ for all $n,m \in \mathbb{N}$. Of course, if $a_{n,m} \prec_{C^m}
b_{n,m}$ and $b_{n,m} \succ_{C(m)} a_{n,m}$ we simply write $a_{n,m} \sim_{C^m} b_{n,m}$.

\smallskip

We write $\mathbb{N}_0^{(\mathbb{N})}$ for all sequences $\alpha \in \mathbb{N}_0^{\mathbb{N}}$ which have finite support, and  as usual we call all these
 sequences multi indices.
 Obviously,
\[
\mathbb{N}_0^{(\mathbb{N})} = \bigcup_{n \in \mathbb{N}} \mathbb{N}_0^n\,,
\]
where  $\mathbb{N}_0^n$ is  interpreted as a subset of $\mathbb{N}_0^{\mathbb{N}}$; we say  that   $\mathbb{N}_0^n$  forms  all multi indices of length $n$.
 For each multi index $\alpha = (\alpha_i) \in \mathbb{N}_0^{(\mathbb{N})}$, we call $$|\alpha | = \sum \alpha_i$$ the order of $\alpha$, and define the two important index sets
\begin{align*}
\Lambda(m,n) = \{ \alpha \in \mathbb{N}_0^n \colon |\alpha | = m \}
\,\,\, \,\,\,\text{ and } \,\,\,\,\,\,
\Lambda(\leq m,n) = \{ \alpha \in \mathbb{N}_0^n\colon |\alpha | \leq  m\}\,.
\end{align*}
The following  formula and estimate
\begin{equation} \label{cardi}
    |\Lambda(m,n)| = \dbinom{n+m-1}{m}\leq e^m \Big( 1 + \frac{n}{m}\Big)^m
\end{equation}
for the cardinality of $\Lambda(m,n)$ is crucial for our purposes.
Given an index set $J \subset \mathbb{N}_0^{(\mathbb{N})}$ and $m,n \in \mathbb{N}$, we call
\[
J^n:= J \cap \mathbb{N}_0^n
\]
the $n$-dimensional section of $J$, and
\[
J_m = J \cap \Lambda(m,n)
\]
its $m$-homogeneous part.
By
\[
J_{\leq d} = \big\{ \alpha \in J \colon |\alpha| \leq d \big\}
\]
we denote the index set of all multi indices in $J$ of degree $\leq d$.
For any finite index set $J \subset \mathbb{N}_0^{(\mathbb{N})}$, we define the degree of $J$ by
$\text{deg} \; J = \max\{|\alpha|\,:\, \alpha\in J\}\,.$
An index set  $J \subset \mathbb{N}_0^{(\mathbb{N})}$ is said to have degree at most $d$ whenever $J = J_{\leq d}$, and $m$-homogeneous whenever $J = J_{m}$.

A multi index $\alpha = (\alpha_i)\in \mathbb{N}_0^{(\mathbb{N})}$ is called  tetrahedral whenever
each entry $\alpha_i$ is either $0$ or $1$. We denote the index set of all   tetrahedral multi indices by $\Lambda_T$, and put for each $m \leq n$
\begin{align*}
\Lambda_T(m,n) = \big\{ \alpha \in \Lambda_T^n \colon |\alpha | = m \big\}
\,\,\, \,\,\,\text{ and } \,\,\,\,\,\,
\Lambda_T(\leq m,n) = \{ \alpha \in \Lambda_T^n \colon |\alpha | \leq\,m\}\,.
\end{align*}
Moreover, $J$ is called  tetrahedral, if all indices in $J$ are tetrahedral, i.e., $J \subset \Lambda_T$.
Given  $J \subset  \Lambda(m,n)$, we  define the index set
\[
J^\flat \subset \Lambda(m-1,n)\,,
\]
which consists of all $\alpha \in \Lambda(m-1,n)$ for which there is $1 \leq k \leq n$ and $\beta \in J$ such
$\beta_i = \alpha_i$ for all $1 \leq i\neq  k\leq n$ and $\beta_k = \alpha_k +1$. We call this set $J^\flat$ the  reduced set of $J$.

At a few occasions it will turn out to be convenient to use an equivalent description of $\Lambda(m,n)$. We write
\begin{align*}
& \mathcal{M}(m,n) =\{1, \ldots, n\}^{m}\,,
\\
&
\mathcal{J}(m,n) = \big\{\bj = (j_1, \ldots, j_m) \in \mathcal{M}(m, n)\colon \, j_1 \leq \ldots \leq j_m\big\}\,.
\end{align*}
Then there obviously
is a~canonical bijection between $\mathcal{J}(m,n)$ and $\Lambda(m,n)$. Indeed,
assign to $\bj \in \mathcal{J}(m,n) $ the multi index $\alpha \in \Lambda(m,n)$
given by $\alpha_r = |\{k \colon \bj_k = r\}|),\, 1 \leq r \leq n$, and conversely for each $\alpha \in \Lambda(m,n)$
the index  $\bj \in \mathcal{J}(m,n)$, where $j_1 =\ldots =j_{\alpha_1} =1$, $j_{\alpha_1 +1} =\ldots =j_{\alpha_1+\alpha_2} =2,\, \dots$

On $\mathcal{M}(m,n)$ we consider the equivalence relation: $\bi  \sim \bj$ if there is a permutation
$\sigma$ on $\{1, \ldots,m\}$ such that $(i_1, \ldots, i_k) = (i_{\sigma(1)}, \ldots, i_{\sigma(m)})$. The equivalence
class of $\bi \in \mathcal{M}(m,n)$ is denoted by $[\bi]$, and its cardinality  by $|[\bi]|$.  We have that
\begin{equation} \label{scholz}
    |[\alpha]| :=  |[\bj]| =\frac{m!}{\alpha!}\,,
\end{equation}
provided that $\bj$ is associated with $\alpha$\,.
Note that, given an index set  $J~\subset~\mathcal{J}(m,n)~=~\Lambda(m,n)$, the reduced set $J^\flat$
 in terms of the $\bj$-mode  reads as follows
\[
J^\flat = \big\{ \bj \in \mathcal{J}(m-1,n)
\colon \exists 1 \leq k \leq n \,\, \text{  such that }\,\,\, (\bj,k)_\ast \in J \big\}\,,
\]
where we associate to each $\bi \in \mathcal{M}(m,n)$  the unique element $\bi_\ast\in \mathcal{J}(m,n)$ for which $\bi \in [\bi_\ast]$.

Finally, we use the fundamental theorem
of arithmetic in order to define index sets in  $\mathbb{N}_0^{(\mathbb{N})}$ generated by subsets of $\mathbb{N}$. To do so,
recall  that each $n \in \mathbb{N}$ has a unique prime number decomposition, i.e., there is a~unique $\alpha \in  \mathbb{N}_0^{(\mathbb{N})}$ such that $n = \mathfrak{p}^\alpha$, where $\mathfrak{p} = (\mathfrak{p}_1, \mathfrak{p}_2, \ldots)$ stands for the sequence of primes. This shows that each set  $N$ of natural numbers defines an  index set $J(N) = \big\{\alpha \in \mathbb{N}_0^{(\mathbb{N})}\colon \mathfrak{p}^\alpha \in N\big\}$ of multi indices, and vice versa. If we, for  $x \geq 1$, consider the set $N(x)~=~\big\{n~\in~\mathbb{N}\colon~1~\leq~n ~\leq~x\big\}$, then the index set 
\[
\Delta(x) = \big\{\alpha \in \mathbb{N}_0^{\pi(x)}\colon 1 \leq \mathfrak{p}^\alpha \leq x\,\big\}
\]
will turn out to be of particular interests; here  as usual $\pi(x)$  counts all primes $\mathfrak{p}_k \leq x$. Given $m \in \mathbb{N}$, the $m$-homogeneous part of this set is denoted by
\[
\Delta(x,m) = \big\{\alpha \in \mathbb{N}_0^{\pi(x)}\colon 1 \leq \mathfrak{p}^\alpha \leq x,\, |\alpha| = m\,\big\}\,,
\]
so the special case $\Delta(x,1)$ is simply given by the set of all primes $\mathfrak{p}_1, \cdots, \mathfrak{p}_{\pi(x)}$.

\smallskip

\subsection{Banach spaces and (quasi-)Banach lattices}
\label{Banach spaces and (Quasi-)Banach lattices}
We use standard notation from (local) Banach space theory as e.g., used in the monographs \cite{ diestel1995absolutely, LT1, pisier1986factorization, tomczak1989banach, wojtaszczyk1996banach}. If not indicated differently, we consider complex Banach
spaces. Given two Banach spaces $X$ and $Y$, we denote by $\mathcal{L}(X, Y)$ the space of all (bounded) linear
operators $T\colon X\to Y$ with the usual operator norm. The dual Banach space of $X$ is denoted by $X^\ast$, and
the canonical embedding of $X$ into its bidual $X^{\ast \ast}$ by $\kappa_X$. The symbol $B_X$ (resp.  $\overline{B}_X$) denotes the open (resp. closed) unit
ball of $X$.

 If $X$ and $Y$ are isomorphic spaces, that
is, there is an invertible operator from $X$ onto $Y$, we write $X\simeq Y$. We use the notion   $X\equiv Y$ whenever $X$ and $Y$
may be identified isometrically. To indicate injective bounded mappings $T: X \to Y$ we sometimes write
$T\colon X \hookrightarrow Y$.

For isomorphic Banach spaces $X$ and $Y$, the Banach-Mazur  distance between $X$ and $Y$ is defined to be
\[
d(X, Y):=\inf \big\{\|T\|\,\|T^{-1}\|: \,\, \text{$T$ an isomorphism of $X$ onto $Y$} \big\}\,.
\]
If $X$ and $Y$ are not isomorphic, we  let
$d(X, Y)= +\infty$.

\smallskip

For the theory of Banach lattices we refer to  \cite{LT1,LT}.
We mainly focus on (quasi)-Banach (function) lattices over measure spaces. To recall the definition, let $(\Omega, \mathcal{A}, \mu)$ be a~complete $\sigma$-finite measure space,  and the space
$L^0(\mu):= L^0(\Omega, \mathcal{A}, \mu)$ of all (equivalence classes of) $\mathcal{A}$-measurable real-valued functions defined on $\Omega$,
equipped with the topology of convergence in measure on sets of the finite $\mu$-measure. For every $f\in L^0(\mu)$, the distribution function $d_{|f|} \colon [0, \infty) \to [0,\infty]$ of $|f|$ is given by 
$d_{|f|}(t) := \mu(\{x \in \Omega: \, |f(\omega)| > t\})$ for all $t\in [0, \infty)$.

A (quasi)-Banach space $X$ which is a~monotone ideal of $L^0(\mu)$ is said to be a (quasi)-Banach (function) lattice over  $(\Omega, \mathcal{A}, \mu)$ (over $\Omega$ for short).
Recall that $X$ is an ideal of $L^0(\mu)$ if for any $y \in X$ and  $x \in L^0(\mu)$ such that $|x(t)| \leq |y(t)|$ for $\mu$-almost all $t\in \Omega$,
we have  $x\in X$ and $\|x\|_X \leq \|y\|_X$. By a~complex (quasi)-Banach lattice over $(\Omega, \mathcal{A}, \mu)$, we mean a~natural
complexification of a~real (quasi)-Banach lattice $X$ equipped with the norm $X\ni x \mapsto \||x|\|_X$. 
A (quasi)-Banach lattice $X$ over $(\Omega, \mathcal{A}, \mu)$ is called symmetric (or, rearrangement invariant), 
if [$f\in L^0(\mu)$, $g\in X$ and $d_{|f|}(t)= d_{|g|}(t)$ for all $t\geq 0$] implies that [$f \in X$ and $\|f\|_X = \|g\|_X$]. For the 
theory of symmetric Banach lattices, we refer to the books \cite{BS} and \cite{LT}.

For any Banach function lattice $X$ over $\Omega$, we define the K\"othe dual space $X'$ of $X$ by
\[
X' := \Big\{x\in L^0(\mu): \, \int_\Omega |x y|\,d\mu < \infty  \,\,\, \mbox{for all \,\, $y \in X$} \Big\}\,.
\]
Note also $X'$ becomes a~maximal Banach function lattice if we equip it with the norm
\[
\|x\|_{X'} := \sup\bigg\{\int_\Omega |x y|\,d\mu: \,\|y\|_X \leq 1\bigg\}\,.
\]
A~Banach lattice $X$ is said to be  maximal (or $X$ has the Fatou property) whenever the closed unit ball of $X$ is closed in $L^0(\mu)$ (equivalently, $X \equiv X''$).

Let $X$ and $Y$ be any pair of (quasi)-Banach function lattices over a~measure space $(\Omega, \mathcal{A}, \mu)$. We define the pointwise product $X\circ Y$ of $X$ and $Y$ by 
\[
X\circ Y : =\big\{z\in L^0(\mu): \,\, z = x y, \,\, x\in X, \,\, y\in Y\big\}\,,
\]
equipped with the (quasi)-norm
\[
\|z\|_{X\circ Y}:= \inf \big\{\|x\|_X \|y\|_{Y}: \, z = x y,\,\, x\in X,\,\, y\in Y\big\}\,.
\]
Moreover, we define the space $M(X, Y)$ of all pointwise multipliers to be the~(quasi)-Banach  function lattice of all $x\in~L^0(\mu)$ such that $xy\in Y$ for all $y\in X$, equipped with the (quasi)-norm
\[
\|x\|_{M(X, Y)} := \sup \big\{\|x y\|_{Y}:\, \|y\|_X \leq 1\big\}\,.
\]
Note that $M(X, Y)$ is a~Banach function lattice whenever $Y$ is a~Banach space, and that $M(X, Y)$ can be trivial. 

Important geometrical notions from the theory of Banach lattices will be used.  A~Banach function lattice $X$ over a~measure space $\Omega$ is said to be $p$-convex, $1 \leq p \leq \infty$, respectively $q$-concave,
$1 \leq q \leq \infty$, if there is a~constant $C>0$ such that for every choice of finitely many $x_1,\ldots, x_N \in~X$
\begin{align*}
\Big\|\Big(\sum_{k=1}^{N} |x_k|^p \Big)^{1/p}\Big\|_X \le C \, \Big (\sum^{N}_{k=1}\|x_k\|_X^p \Big)^{1/p}\,,
\end{align*}
resp.,
\begin{align*}
\Big (\sum^{N}_{k=1}\|x_k\|_X^q \Big) ^{1/q}
\le C \, \Big\|\Big(\sum_{k=1}^{N}|x_k|^q \Big)^{1/q}\Big\|_{X}
\end{align*}
(with the usual modification whenever $p= \infty$ or $q= \infty$). We define the $p$-convexity constant $M^{(p)}(X)$ (resp.
$q$-concavity constant $M_{(p)}(X)$) to be the least constant $C>0$ satisfying the above inequality. In case that $X$ is not $p$-convex (resp. not $q$-concave), then we write $M^{(p)}(X) = \infty$ (resp. $M_{(q)}(X)~=~\infty$).
Clearly, every Banach function lattice $X$ is $1$-convex and $\infty$-concave (with constants $1$). We also note that, if $r < p < s$ and  $X$ is $p$-convex (resp. $p$-concave), then $X$ is $r$-convex (resp. $s$-concave) with $M^{(r)}(X) \le M^{(p)}(X)$ (resp. $M_{(s)}(X) \le M_{(p)}(X)$). For details we refer to \cite[Proposition 1.d.5]{LT1}. 

We point out that if in the definition  of $p$-convexity, respectively $q$-concavity, of a~Banach lattice  $X$ the corresponding inequality holds only for vectors $x_1, \ldots, x_n$ with pairwise disjoint supports,  then $X$ is said to satisfy an upper $p$-estimate, respectively a~lower $q$-estimate. 

It is obvious that a~Banach lattice satisfies an upper, respectively lower $p$-estimate, whenever it is 
$p$-convex, respectively $p$-concave. The converse is not true in general. However the following result 
is true:  If a~Banach lattice satisfies an upper, respectively lower $r$-estimate, for some $1<r<\infty$, 
then it is $p$-convex, respectively $q$-concave, for every $1 < p < r < q <\infty$ (see 
\cite[Theorem 1.f.7]{LT1}). 

\smallskip

(Quasi)-Banach sequence lattices $X$ are of special interest for our purposes. We recall that a (quasi-)Banach function lattice over
the counting measure space $(\Omega, 2^{\Omega}, \mu)$ is called a~(quasi-)Banach sequence lattice over $\Omega$. We are mainly interested in the case $\Omega:=\{1, \ldots, n\}$ and $\Omega:=\mathbb{N}$.
The fundamental function of a~(quasi)-Banach sequence lattice $X$ is defined by
\[
\varphi_X(k):= \Big\|\sum_{j=1}^k e_j\Big\|_X, \quad\, k\in \mathbb{\Omega}\,.
\]
A (quasi)-Banach sequence lattice $X$ is said to be symmetric whenever we have that $(x_{\sigma(n)}) \in X$ for every $x = (x_n)\in X$ and every permutation $\sigma: \mathbb{N} \to \mathbb{N}$.

Notice that if $X$ is a~separable Banach sequence lattice over $\mathbb{N}$, then the Banach dual $X^{*}$ can be identified in a~natural way with the K\"othe dual $X'$ of $X$
\begin{align*}
X' = \bigg\{(x_k): \, \sum_{k=1}^{\infty}\, |x_k y_k| < \infty \,\, \mbox{for all \,\, $(y_k) \in X$} \bigg\}
\end{align*}
equipped with the norm
\[
\|(x_k)\|_{X'} = \sup\bigg\{\sum_{k=1}^{\infty}\,|x_k y_k|: \,\, \|(y_k)\|_X \leq 1 \bigg\}\,.
\]
If $X$ is a~complex (quasi-)Banach sequence lattice over $\mathbb{N}$ and $n \in \mathbb{N}$, then
\[
\|z\|:= \Big\|\sum_{k=1}^n z_k e_k\Big\|_{X}, \quad\, z=(z_1,\ldots, z_n)\in \mathbb{C}^n
\]
defines  a~lattice (quasi-)norm on $\mathbb{C}^n$, where $e_k$ stands for $k$th standard unit vectors of $c_0$.
This defines the (quasi-)Banach sequence lattice
\[
X_n = (\mathbb{C}^n, \|\cdot\|)\,,
\]
which we call the $n$th section of $X$.
Note that the collection $\{e_k\}_{k=1}^n$ of the unit basis vectors forms  a~$1$-unconditional basis of $X_n$.

We refer to  Section~\ref{bounds} for all needed information on the Banach sequence lattices of Lorentz and Marcinkiewicz spaces, Nakano spaces, and mixed $\ell_p$-spaces -- and in particular to 
Chapter~\ref{Part: Polynomials on Lorentz sequences spaces} were we collect information on the   convexity/concavity of these spaces.

\smallskip

\subsection{Polynomials}
Let $J \subset \mathbb{N}_0^n$ be finite subset. By $\mathcal{P}(\mathbb{C}^n)$ we denote the vector space of all polynomials given by 
\[
\sum_{\alpha\in J}c_\alpha(P)\,z^\alpha, \quad\, z=(z_1, \ldots, z_n)\in \mathbb{C}^n\,,
\]
where $c_\alpha(P)\in \mathbb{C}^n$ for each $\alpha \in J$. In what follows, for each $m \in \mathbb{N}$, we let $\mathcal{P}_{\leq m}(\mathbb{C}^n)$ to be the subspaces of all such  polynomials $P$ of degree $\text{deg}\,P = \{|\alpha|\colon \alpha\in J \} \leq m$,
and $\mathcal{P}_{m }(\mathbb{C}^n)$ for all $m$-homoge\-neous polynomials. More generally, for any fixed (nonempty)  finite index set $J \subset \mathbb{N}_0^n$, we define $\mathcal{P}_J(\mathbb{C}^n)$
to be the subspace of all $P \in \mathcal{P}(\mathbb{C}^n)$ for which $c_{\alpha}(P) = 0$ for all $\alpha \notin J$.
Given a finite index set $J \subset \mathbb{N}_0^n$ and a~Banach space $X_n = (\mathbb{C}^n,\|\cdot\|)$,
we equipp $\mathcal{P}(\mathbb{C}^n)$ with the norm 
\[
\|P\|_{B_{X_n}} := \sup_{z\in B_{X_n}} |P(z)|\,,
\]
 and then we denote
the resulting Banach space by
\[
\mathcal{P}_J(X_n)\,.
\]
We will also use this concept within a slightly more general context. A~bounded subset $K\subset \mathbb{C}^n$ is said
to be  admissible, if for each $P\in \mathcal{P}_J(\mathbb{C}^n)$ with $\sup_{z\in K}|P(z)| = 0$
we have  $P=0$.  Clearly,  $K \subset \mathbb{C}^n$ is  admissible  whenever the functional $\|\cdot\|_K$
on $\mathcal{P}_J(\mathbb{C}^n)$ given by
$
\|P\|_K  =\text{sup}_{z\in K} |P(z)|
$
is a norm on $\mathcal{P}_J(\mathbb{C}^n)$. Note that for any non-trivial
$P\in \mathcal{P}_J(\mathbb{C}^n)$, the function $z\mapsto |P(z)|$ is  subharmonic, and hence
by the maximum principle $\text{sup}_{z\in K}|P(z)|$ is not attained at any  interior
point of $K$.

Let $X_n:= (\mathbb{C}^n, \|\cdot\|)$ be a~quasi-Banach space. Then it is easy to check that the  closed unit ball $\overline{B}_{X_n}$
as well as the unit sphere $S_{X_n}$ are admissible, and moreover
\[
\sup_{z\in \overline{B}_{X_n}} |P(z)| = \sup_{z\in S_{X_n}}|P(z)|, \quad\, P\in \mathcal{P}_{J}(\mathbb{C}^n)\,.
\]
If, in addition, the quasi-norm $\|\cdot\| \colon X_n \to [0, \infty)$ is continuous, then the open unit ball
$B_{X_n}$  is an admissible set, and the preceding two suprema coincides with $\sup_{z\in B_{X_n}} |P(z)|$
(note that for the space $\mathcal{P}_m(X_n)$ of   $m$-homogeneous polynomials  the coincidence of these 
suprema  easily follow  without using the maximum principle for subharmonic functions).

We  also need a few facts on polynomials defined on arbitrary Banach spaces $X$ (and not only $X_n:= (\mathbb{C}^n, \|\cdot\|)$).
For all relevant information we refer to  \cite{defant1992tensor,defant2019libro,dineen1999complex,floret1997natural}.

Let $X$ be a   Banach space $X$ over the field $\mathbb{K}$,  where $\mathbb{K} = \mathbb{R}$ or  $\mathbb{C}$,
and $m \in \mathbb{N}$. A mapping $P\colon X\to \mathbb{K}$ is a (bounded) $m$-homogeneous  polynomial if there exists 
a~(bounded)  $m$-linear form $T\colon X\times\cdots \times X \to \mathbb{K}$  such that $P(z)=T(z,\dots,z)$ for all $z\in X$. 
In fact, given a  (bounded) $m$-homogeneous  polynomial $P$ on $X$, by the so-called polarization formula

\begin{equation}\label{formula de polarizacion}
T(z^{(1)}, \dots, z^{(m)}) = \frac{1}{m!2^m} \sum_{\varepsilon_i = \pm 1} \varepsilon_1 \dots \varepsilon_m P\left(\sum_{i=1}^m \varepsilon_i z^{(i)}\right),
\end{equation}
there is a unique symmetric (bounded) $m$-linear form $T$ with the property that $P(z)=T(z,\dots,z)$ for all $z\in X$. We as usual denote this unique form by $\overset{\vee}{P}=T$. The space of all bounded $m$-linear forms over $X$ is denoted by $\mathcal L_m(X)$, whereas we write $\mathcal{P}_m(X)$ for all bounded $m$-homogeneous polynomials on $X$. Endowed with the norms
\[
\Vert T \Vert = \sup_{z^{(i)} \in B_{X}} \vert T(z^{(1)}, \dots, z^{(m)}) \vert \,\,\,\,\,\,\,\,\,\,\text{ resp.} \,\,\,\,\,\,\,\,\,\, \Vert P \Vert = \sup_{z \in B_{X}} \vert P(z) \vert \,,
\]
each gets a  Banach space. 
Moreover, we need the notion of a polynomial of degree $\leq m$ on a given Banach spaces $X$. By  $\mathcal{P}_{\leq m}(X)$ we denote the linear space of all
functions $P: X \to \mathbb{K}$ having the  form
$P = P_0 + \sum_{k=1}^m P_k$, where $P_k \in \mathcal{P}_k(X), \, 1 \leq k \leq m$ and $P_0 \in \mathbb{C}$. As before, $\mathcal{P}_{\leq m}(X)$ together with the supremum norm (of $P$ on $B_X$) leads to a  
Banach space.

We recall the notion of so-called polarization constants, which appear naturally  when relating homogeneous polynomials with their associated symmetric multilinear forms.
From  \eqref{formula de polarizacion} we easily obtain the  following polarization inequality
\begin{equation}\label{ec polarizacion}
\Vert \overset{\vee}{P} \Vert \leq \frac{m^m}{m!} \Vert P \Vert\,,\,\,\,\,\,\,\,\,P \in \mathcal P_m(X)\,.
\end{equation}
In this context it is natural to define, given  $m\in \mathbb{N}$ and an arbitrary  Banach space $X$, the  $m$-polarization constant
\begin{equation} \label{def k-constante}
\ccc(m,X) := \inf\Big\{C>0: \,\, \Vert \overset{\vee}{P} \Vert \leq C \Vert P \Vert, \; \mbox{for all } P \in \mathcal P_m(X) \Big\}.
\end{equation}
Obviously,
\begin{equation}
\label{projl1}
 1 \leq \ccc(m,X) \leq \frac{m^m}{m!} < e^m   \,.
\end{equation}
Since  $\ccc(m,\ell_1) = \frac{m^m}{m!} $,
here the  upper bound  $\frac{m^m}{m!}$ is  best possible; more precisely, 
given $\varepsilon>0$ there is $n\in \mathbb{N}$ and an $m$-homogeneous polynomial $P \in \mathcal P_m(\ell_1^n)$ satisfying  $\Vert \overset{\vee}{P} \Vert \geq (1-\varepsilon) \frac{m^m}{m!}$ (see for example \cite{H}).
On the other hand, it is known that $\ccc(m,H)=1$ for each Hilbert space $H$ and $\ccc(m,\ell_\infty) < \frac{m^m}{m!}$. 

Given  a finite dimensional Banach space $X$, we at a few occasions  need to represent  $\mathcal L_m(X)$ and $\mathcal P_m(X)$ as injective tensor products.
Recall that in this case the canonical identities
\[
\mathcal{L}_m(X) = \otimes_\varepsilon^m X^\ast 
\,\,\,\,\,\,\,\,\,\,\, \text{ and } \,\,\,\,\,\,\,\,\,\,\,
\mathcal{P}_m(X) = \otimes_{\varepsilon_s}^m X^\ast
\]
hold algebraically as well as isometrically, where $\otimes_\varepsilon^m$ stands for the $m$th full injective tensor product
and $\otimes_{\varepsilon_s}^m$ for the $m$th symmetric  injective tensor product (see, e.g., \cite{floret1997natural}). 
We also   need the  $m$th full projective  tensor product $\otimes_\pi^m$ (resp., the $m$th symmetric  projective tensor product $\otimes_{\pi_s}^m$),  which is dual to  $\otimes_\varepsilon^m$ (resp. $\otimes_{\varepsilon_s}^m$).

\smallskip

\chapter{Trigonometric polynomials on abelian groups
\label{trigo}}

The basic aim of this chapter is to further develop a cycle of ideas that will be  applied 
to prove asymptotically best possible estimates for the projection constant of several spaces of
multivariate
polynomials:  trigonometric polynomials in Section \ref{The Lozinski and Kharishiladze revisited},  polynomials 
on Euclidean spaces
in Chapter~\ref{eukli},
 traces class   
operators  in 
Chapter~\ref{Trace class operators},
Dirichlet polynomials in Chapter \ref{Part: Dirichlet polynomials and polynomials on the Boolean cube}, and polynomials on 
Boolean cubes in Chapter \ref{Polynomials on the Boolean cube}.

It is well-known, that if $X$
is finite dimensional Banach space,
then finding the projection constant $\boldsymbol{\lambda}(X)$ is equivalent to the problem of calculating the norm of a projection which is minimal among all projections  from a
$C(K)$ onto (a subspace isometrically isomorphic to) $X$.

There is a large amount of literature, mainly coming from approximation theory, which
studies classical problems on the convergence of Fourier series or of trigonometric interpolation in different situations. In this context, a classical 
approach  to show  non-convergence is to prove that a certain  projection (which is usually  orthogonal on an appropriate Hilbert space)
  has minimal norm,  and then to estimate its norm. 
  
  A tool, whose origin can be traced back to   Faber's article  \cite{faber1914interpolatorische} from 1914 (and then developed in \cite{berman1952class,fejer1930,marcinkiewicz1937quelques,
lozinski1948class}, followed by many others) consists in obtaining  a projection $P$ on $C(\mathbb{T})$ onto the subspace of all 
trigonometric
polynomials of degree $\le m$
with minimal norm, by averaging
 compositions of $P$ with translation  operators
 on $C(\mathbb{T})$. That is, if 
 \begin{equation*}
     \text{$T_s(f)=f(\cdot+s)$
 for $ s \in \mathbb{T}$ and  $f \in C(\mathbb{T})$,}
 \end{equation*}
  then the formula
 \begin{equation*}
     Qf=\frac1{2\pi}\int_{-\pi}^\pi T_{-s}PT_sf\,ds, \quad\, f \in C(\mathbb{T})
 \end{equation*}
defines a~projection 
on
$C(\mathbb{T})$ onto the subspace of all 
trigonometric polynomials of degree less or equal than $m$
with minimal norm, which turns out to  coincide with the 
so-called Fourier projection. This average property of the Fourier projection is sometimes called the Berman-Marcinkiewicz identity.

This averaging technique was further developed by Rudin~\cite{rudin1962projections}, who presented it  in a very general context.
In the setting of finite dimensional spaces of multivariate polynomials, Rudin's result, which we recall in Theorem~\ref{rudy}, has become
one of the main tools used to get estimates for projection constants. 

Based on this theorem, we prove an integral formula for finite dimensional spaces of trigonometric polynomials on a compact topological group
(Theorem~\ref{C(G)proj}), which are supported on a finite set of characters of the
dual group - thus recovering and extending some of the classic results. Later in this chapter 
(Section~\ref{The Lozinski and Kharishiladze revisited}) we apply this formula to obtain multivariate variants of the Lozinski-Kharishiladze theorem from \eqref{LoKa}. 
As mentioned above, we leave further applications of this formula to  subsequent chapters.

\section{Basics on topological groups}
\label{basicsbasics}

 As usual a~group $G$ equipped with a~topology $\tau$ is said to be
a~topological group whenever the mapping $(G, \tau) \times (G, \tau) \ni (x, y) \to x y^{-1} \in (G, \tau)$ is continuous.
From here on $G$ is assumed to be a compact group, that is, its topology is compact. In this case, $G$ defines a~natural set of maps
$\{L_g\}_{g\in G}$ and $\{R_g\}_{g\in G}$ on  $C(G)$, the complex-valued continuous functions on $G$, given
for all $g, h\in G$ by
\[
L_gf(h):= f(gh), \quad\, \, \, \,\text{and\, $R_gf(h) = f(hg)$}, \quad\, f\in C(G)\,.
\]
It is well-known that for every compact group $G$ there exists a~unique Borel~probability measure $\mathrm{m}$ which is left
invariant, that is,
\[
\int_G f\,d\mathrm{m} = \int_G (L_gf)\,d\mathrm{m}, \quad\, g\in G, \, \, \, f\in C(G)\,.
\]
This $\mathrm{m}$ is called the Haar measure of $G$. If in addition $\mathrm{m}$ is also right invariant:
\[
\int_G f\,d\mathrm{m} = \int_G (R_gf)\,d\mathrm{m}, \quad\, g\in G, \, \, \, f\in C(G)\,,
\]
then the compact group $G$ is called unimodular. Examples of unimodular groups are compact groups in which every one point set is closed.

Let $\mathrm{m}$ be the normalized Haar measure on $G$, and $\widehat{G}$ as usual the dual group of $G$ (i.e., the set of all continuous characters on $G$). For any $f\in L^1(G):=L^1(G, \mathrm{m})$, the Fourier transform of $f$
is given by
\[
\widehat{f}(\gamma):= \int_G\, f(x)\,\overline{\gamma(x)} \,d\mathrm{m}(x), \quad\, \gamma \in \widehat{G}\,.
\]
Recall that $L_1(G)$ forms a~commutative Banach algebra, whenever it carries  the convolution $f_1\ast f_2$ as its multiplication, that is, for $\mathrm{m}$-almost every $x\in G$
\[
(f_1\ast f_2)(x) := \int_G f_1(xy^{-1})f_2(y)\,d\mathrm{m}(y)\,.
\]
We need  two  fundamental consequences of the Peter-Weyl theorem (see, e.g., \cite[Theorem 1.3.3]{queffelec2013diophantine}: For
any compact abelian group $G$, the space $\text{Trig}(G)$ is dense in the space $C(G)$ of complex-valued functions
(see \cite[Theorem 1.3.4]{queffelec2013diophantine}), and  the dual group $\widehat{G}$ is an orthonormal basis of the Hilbert space
$L^2(G, \mathrm{m})$ (see \cite[Theorem 1.3.6]{queffelec2013diophantine}).

We  repeat that all concrete  groups we are interested  in  fact are compact, and we are going to explain their definitions locally. In particular, we consider the  $n$-dimensional torus $\mathbb{T}^n$ (Section~\ref{Polynomials on products of groups}), the $n$-dimensional Boolean cube $\{1,-1\}^n$ (Chapter~\ref{Polynomials on the Boolean cube}),  and the (non)-abelian unitary group $\mathcal{U}_n$ of all unitaries
on the $n$-dimensional Hilbert space (Chapter~\ref{Trace class operators}).

For any compact group $G$ and any  nonempty finite set $E\subset \widehat{G}$ we denote $\text{span}\,E$  by
$$\text{Trig}_E(G)\,,$$
 the
space of all trigonometric polynomials supported on $E$. Note that each such polynomial has the form
\[
P(g) \,= \,\sum_{\gamma \in E} \widehat{f}(\gamma)\, \gamma(g)\,, \quad g \in G\,.
\]
In the case $E=\widehat{G}$, the space $\text{Trig}_{\widehat{G}}(G)$ of all trigonometric polynomials
on $G$ is denoted in short by $\text{Trig}(G)$.

\smallskip

\section{Rudin's averaging technique} \label{Rudin's averaging technique}

As mentioned, one of the main tools we intend to use is a method due to Rudin (see the forthcoming Theorem~\ref{rudy}). 
Roughly speaking, under certain assumptions, there is a somewhat universal  averaging technique to construct new projections 
with additional properties from an a~priori given projection. 

This method is intimately  connected with  topological groups  acting  as bounded  linear operators on a~given Banach space $X$. 
More precisely, let $X$ be a Banach space and  $G$  a~topological group, and suppose that there is a mapping
\[
T\colon G \to \mathcal{L}(X)\,, \, \,\, \,g \mapsto T_g
\]
such that
\[
T_e = I_X, \quad\, T_{gh} = T_g T_h, \quad\,  g, \, \, h\in G
\]
and all mappings
\begin{equation}\label{con(i)}
 g\ni G \mapsto T_g(x) \in X,  \quad\,   \, \, x\in X
\end{equation}
are continuous. Then $G$ is said to act on $X$ through $T$ (or simply, $G$ acts on $X$). If in addition all operators $T_g, g\in G$ are isometries, then we say that $G$ acts isometrically on $X$. We say that $S \in \mathcal{L}(X)$ commutes with the action of $G$ on $X$ through $T$ whenever $S$ commutes with  all $T_h,\, h\in G$.

\smallskip

\begin{theorem} \label{rudy}
Let $X$ be a~Banach space, $Y$ a~complemented subspace of $X$, and \,$\mathbf{Q} \colon X \to X$ a~projection onto $Y$. Suppose that  $G$ is a compact group with Haar measure $m$, which   acts on $X$ through $T$ such that $Y$ is invariant under the action of $G$, that is,
\begin{equation} \label{invariant}
T_g(Y) \subset Y, \quad\,g\in G\,.
\end{equation}
Then $\mathbf{P}\colon X \to X$ given by
\begin{equation}\label{equation rudy}
\mathbf{P}(x):= \int_{G} T_{g^{-1}}\mathbf{Q}T_g(x)\,d\mathrm{m}(g), \quad\, x\in X\,,
\end{equation}
is a~projection onto $Y$ which commutes with the action  of $G$ on $X$, i.e., $T_g\mathbf{P} = \mathbf{P}T_g$ for all
$g\in G$, and satisfies
\[
\|\mathbf{P}\| \,\leq\, \|\mathbf{Q}\|\,\,\,\sup_{g\in G}\|T_g\|^2\,.
\]
Moreover, if there is a unique projection  on $X$ onto $Y$  that commutes
with the  action of $G$ on $X$, then $\mathbf{P}$ given in \eqref{equation rudy} is minimal, i.e.,
\[
\boldsymbol{\lambda}(Y;X) = \|\mathbf{P}\|\,.
\]
\end{theorem}

This result (in one way or the other) found various applications in the literature (see, e.g., \cite{konig1995projections,
konig1999projection, light1986minimal,  schuettkwapien,rieffel2006lipschitz, rudin1962projections,  rudin1986new}.
In the case of the circle group, it can be traced back to Faber's   article  \cite{faber1914interpolatorische}. For the sake of completeness 
we include a~proof, which very much inspired by \cite{rudin1962projections} (and  also \cite[Theorem III.B.13]{wojtaszczyk1996banach}).

\begin{proof}
Note first, that, by the Banach-Steinhaus theorem,  $\sup_{g\in G} \|T_g\|<\infty$. Then, given $x \in X$, the mapping
$g\ni G \mapsto T_{g^{-1}}\mathbf{Q}T_g \in \mathcal{L}(X)$ is bounded and by $\eqref{con(i)}$ measurable
(being almost everywhere separably valued and weakly measurable), and hence Bochner integrable.
Consequently, $\mathbf{P}$ defines an operator on $X$.
 Moreover, from  \eqref{invariant} we deduce
for all $x\in X$ and for all $y\in Y$ that
\[
T_{g^{-1}}\mathbf{Q}T_g(x) \in Y, \quad\, T_{g^{-1}}\mathbf{Q}T_g(y) = y\,,
\]
implying that  $\mathbf{P}$ is a~projection from $X$ onto $Y$. The hypothesis that $G$ acts on $X$ (through $T$) yields 
for all $h\in G$
\[
T_{h^{-1}}\mathbf{P}T_h = \int_G T_{h^{-1}}T_{g^{-1}}\mathbf{Q}T_g T_h\,d\mathrm{m}(g)
= \int_G T_{(gh)^{-1}}\mathbf{Q}T_{gh}\,d\mathrm{m}(g) = \mathbf{P}\,,
\]
so $\mathbf{P}$ commutes with the action of $G$ on $X$.
Since for all $x\in X$
\[
\|\mathbf{P}x\|_X \leq \int_G \|T_{g^{-1}}\|\,\|\mathbf{Q}\,\|\|T_g\|\,\|x\|_X\,d\mathrm{m}(g)
\leq \|\mathbf{Q}\|\,\sup_{g\in G}\|T_g\|^{2}\,\|x\|_X\,,
\]
the required estimate for  $\|\mathbf{P}\|$ follows. Finally, if we assume that if there is a unique projection  on $X$ onto $Y$  that commutes
with the  action of $G$ on $X$, every projection $\mathbf{Q}$ on $X$ onto $Y$ induces the same projection $\mathbf{P}$.
If moreover
$G$ acts isometrically on $X$,
then $\|\mathbf{P}\| \,\leq\, \|\mathbf{Q}\|$ for all projections $\mathbf{Q}$ on $X$ onto $Y$, i.e.,
$\|\mathbf{P}\| \leq  \boldsymbol{\lambda}(Y;X)$. This completes the proof.
\end{proof}

\smallskip

\section{Integral formula} \label{An integral formula}

We start with a general result, which we are going to apply  to multivariate trigonometric polynomials (Section~\ref{Polynomials 
on products of groups}), Dirichlet polynomials (Chapter~\ref{Part: Dirichlet polynomials and polynomials on the Boolean cube}), 
and polynomials on the Boolean cube (Chapter~\ref{Polynomials on the Boolean cube}). In particular, we give variants of the  well-known result due to Lozinski and Kharishiladze from \eqref{LoKa} for the circle group.

As mentioned in  the introduction, the proof of the following theorem is based on Rudin's Theorem~\ref{rudy}
on the averaging of projections (it should also be compared with \cite{lambert1969minimum,lambert1970minimum, dreseler1974charshiladze,dreseler1975convergence}).

\medskip

\begin{theorem}\label{C(G)proj}
Let $G$ be a compact abelian group and  $E:=\{\gamma_1,\ldots, \gamma_N\}\subset \widehat{G}$  a~finite set of
characters. Then $P\colon C(G) \to C(G)$, given by $Pf = \sum_{j=1}^N \widehat{f}(\gamma_j) \gamma_j$ for
all $f\in C(G)$, is a~projection onto $\text{Trig}_E(G)$ such that
\[
\boldsymbol{\lambda}\big(\text{Trig}_E(G)\big) = \big\|P\colon C(G)\to C(G)\big\|
= \int_G \Big|\sum_{j=1}^N \gamma_j(x)\Big|\,d\mathrm{m}(x)\,.
\]
\end{theorem}

\begin{proof}Note first  that $G$ in a natural way acts on $C(G)$ (in the sense of Section~\ref{Rudin's averaging technique}), 
where the action is given by the mapping $T: G \to \mathcal{L}(C(G)), \,g \mapsto T_g$ with
\[
T_gf(h):= f(gh), \quad\, f\in C(G), \,\, h\in G\,.
\]
We claim that $P \colon C(G) \to C(G)$ is the unique projection onto $\text{Trig}_E(G)$ that commutes with all translation operators
$T_g$, $g\in G$. To see this, assume that $Q\colon C(G) \to  C(G)$ onto $\text{Trig}_E(G)$ is a~projection that commutes with all
translation operators. Then for all $\gamma, \gamma'\in \widehat{G}$ one has
\[
\widehat{T_gQ\gamma}(\gamma') = \widehat{QT_g\gamma}(\gamma')\,.
\]
It is easy to check that $\widehat{T_gQ\gamma}(\gamma') = \gamma'(g)\, \widehat{Q\gamma}(\gamma')$ and
$\widehat{QT_g\gamma}(\gamma')= \gamma(g)\, \widehat{Q\gamma}(\gamma')$. In consequence, we get
\[
\gamma'(g) \widehat{Q\gamma}(\gamma') = \gamma(g) \widehat{Q\gamma}(\gamma'), \quad\, g\in G\,.
\]
This implies that, for all $\gamma, \gamma'\in \widehat{G}$ with $\gamma \neq \gamma'$, we have
$\widehat{Q\gamma}(\gamma') = 0$. Combining with the Peter-Weyl theorem that $\widehat{G}$ forms an orthonormal basis in the Hilbert space $L^2(G)$, we conclude that
\[
Q\gamma = \sum_{\gamma' \in \widehat{G}} \widehat{Q\gamma}(\gamma')\,\gamma', \quad\, \gamma\in \widehat{G}\,,
\]
and consequently, that for every character $\gamma \in \widehat{G}$ there is a~scalar $c_\gamma$ such that $Q\gamma = c_{\gamma}\gamma$.

Since $Q$ is a~projection onto $\text{Trig}(G)$, $c_\gamma = 0$ for all $\gamma \in\widehat{G}\setminus E$, and $c_{\gamma} = 1$
for all $\gamma \in E$. In consequence, $Q\gamma = \gamma$ for all $\gamma \in E$, and $Q\gamma = 0$ for all
$\gamma  \in \widehat{G} \setminus E$. Hence the projection $Q$, restricted to the algebra $\text{Trig}(G)$
of all trigonometric polynomials on $G$, has the representation
\begin{align}\label{representation of proj onto Trig}
Qf = \sum_{j=1}^N \widehat{f}(\gamma_j) \gamma_j, \quad\, f\in \text{Trig}(G)\,.
\end{align}
Consequently, we conclude from the density of $\text{Trig}(G)$ in  $C(G)$ that the above formula holds
for all $f\in C(G)$. Hence $Q=P$. This proves the claim.

Now observe that $G$ acts isometrically on $X$, i.e., all mappings  $T_g\colon C(G) \to C(G)$ are isometries on  $C(G)$. Since for all $g\in G$
\[
T_g\gamma  = \gamma(g)\gamma, \quad\, \gamma \in \widehat{G}\,,
\]
it follows that $T_g(\text{Trig}_E(G)) \subset \text{Trig}_E(G)$ for all $g\in G$. But then we are ready  to apply Rudin's theorem~\ref{rudy} showing  that
$P$ is a minimal projection, that is,
\[
\boldsymbol{\lambda}\big(\text{Trig}_E(G)\big) = \big\|P\colon C(G)\to C(G)\big\|\,.
\]
Finally, it remains to prove the integral formula for the norm of  $P$. Since $f \ast \gamma=\widehat{f}(\gamma)\gamma$
for all $f\in L^1(G)$ and $\gamma \in \widehat{G}$, we get by  \eqref{representation of proj onto Trig},
\[
Pf= f \ast \Big(\sum_{j=1}^N \gamma_j\Big), \quad\, f\in C(G)\,.
\]
Clearly, $\sum_{j=1}^N \gamma_j \in C(G)$, so it can readily be shown by direct computation that
\[
\|P \colon C(G) \to C(G)\| = \int_G \, \Big|\sum_{j=1}^N \gamma_j(x)\Big|\,d\mathrm{m}(x)\,,
\]
and this completes the proof.
\end{proof}

\medskip

In order to show a very first application of Theorem \ref{C(G)proj}, we need some further notation and preliminaries.
Recall that Rudin in his  classical paper \cite{rudin1960trigonometric} from  1960 introduced the notion of $\Lambda(p)$-sets 
within the setting of Fourier analysis on the circle group $\mathbb{T}$. In modern language, if $G$ is a~compact abelian group 
(with Haar measure $\mathrm{m}$) and $p\in (1, \infty)$, then the~subset ${E} \subset \widehat{G}$ is said to be a~$\Lambda(p)$-set whenever there exists a~constant $C>0$ such that, for every trigonometric polynomial $P \in \text{Trig}_E(G)$, one has
\begin{equation}\label{lambdap}
\|P\|_{L_p(G)} \leq C \|P\|_{L_1(G)}\,.
\end{equation}
In this case, the least such constant  is called the $\Lambda(p)$-constant of $E$, and denoted by $C_p=C_p(E)$. Let us here 
remark that for  $p>2$ the validity of \eqref{lambdap}  is equivalent to the existence of a~constant $A_p \ge 0$ such that
\[
\|P\|_{L_p(G)} \leq A_p \|P\|_{L_2(G)}, \quad\, P\in \text{Trig}_E(G)\,.
\]
The following almost immediate consequence of Theorem \ref{C(G)proj} shows that $\Lambda(2)$-sets are of particular
importance for our purposes -- see, e.g., Corollary~\ref{manydimensionsB},
 Corollary~\ref{b2dirichelet}, or Corollary~\ref{kiel}.

\begin{corollary} \label{corB2}
Let $G$ be a compact abelian group. Then, for any set finite set $E= \{\gamma_1,\ldots, \gamma_N\}\subset \widehat{G}$
of different characters with $\Lambda(2)$ constant $C_2$, we have
\[
C_2^{-1} \sqrt{N} \leq \boldsymbol{\lambda}\big(\text{Trig}_E(G)\big) \leq \sqrt{N}\,.
\]
\end{corollary}

\begin{proof}
Let $\mathrm{m}$ be the Haar measure on $G$. Then from Theorem \ref{C(G)proj}, we get 
\[
\boldsymbol{\lambda}\big(\text{Trig}_E(G)\big)= \int_G |\lambda_1(x) + \ldots + \lambda_N(x)|\,d\mathrm{m}(x)\,.
\]
Since $\widehat{G}$ is an orthonormal basis in $L_2(G, \mathrm{m})$,
\[
\bigg(\int_G |\lambda_1(x) + \ldots + \lambda_N(x)|^2\,d\mathrm{m}(x)\bigg)^{\frac{1}{2}} = \sqrt{N}\,,
\]
and this completes the proof.
\end{proof}

In combination with the preceding corollary we need the following example.

\begin{example}
Let $G$ be a compact abelian group.
  Following \cite{grahamhare2013}, a~set $E \subset \widehat{G}$ is said to be a~$B_2$-set, whenever $\gamma_1 \gamma_2
= \gamma_3 \gamma_4$ for all $\gamma_1, \ldots, \gamma_4 \in E$ if and only if $\{\gamma_3, \gamma_4\}$
is a permutation of $\{\gamma_1, \gamma_2\}$. It is worth noting that $B_2$-sets are $\Lambda(2)$-sets (see
\cite[Proposition 6.3.11]{grahamhare2013}).
\end{example}

 Since we  need this result in Chapter~\ref{Part: Dirichlet polynomials and polynomials on the Boolean cube}, we include 
 its standard proof.

\begin{lemma} \label{b2}
Let $G$ be a compact abelian group  and  $E\subset \widehat{G}$ a $B_2$-set.
Then the following estimate holds
\[
\|P\|_{L_4(G)} \leq 2^{\frac{1}{4}}\,\|P\|_{L_2(G)}, \quad\, P\in \text{Trig}_E(G)\,.
\]
As a~consequence, $E$ is $\Lambda(2)$-set with $C_2\leq \sqrt{2}$.
\end{lemma}

\begin{proof}
Fix some $P= \sum_{j=1}^N c_j \gamma_j\in \text{Trig}_E(G)$. Then
\[
\int_G\bigg|\sum_{j=1}^N c_j \gamma_j(x)\bigg|^4 \,d\mathrm{m}(x) = \sum_{i,j, k, \ell=1}^N c_i \overline{c_j} c_k
\overline{c_{\ell}}\, \int_G \gamma_i(x)\,\overline{\gamma_j(x)}\,\gamma_k(x)\,\overline{\gamma_{\ell}(x)}\,d\mathrm{m}(x)\,,
\]
where $\mathrm{m}$ again denotes the Haar measure on $G$.
Now we use the  obvious fact that  for every  $\gamma\in \widehat{G}$
\[
\text{$\int_G \gamma(x)\,d\mathrm{m}(x)= 0$ \quad if and only if \quad
$\gamma \neq 1$.}
\]
 Thus, we conclude that
\[
\int_G \gamma_i(x) \overline{\gamma_j(x)}\,\gamma_k(x)\,\overline{\gamma_{\ell}(x)}\,d\mathrm{m}(x) \neq 0
\]
if and only if $\gamma_i\,\overline{\gamma_j}\,\gamma_k\,\overline{\gamma_{\ell}} = 1$, or equivalently  $\gamma_i \gamma_k
= \gamma_j \gamma_{\ell}$. Since $E$ is $B_2$-set, it follows that $i=j$ and $k=\ell$, or $i=\ell$ and $k=j$ and hence
\[
\int_G \gamma_i(x)\,\overline{\gamma_j(x)}\,\gamma_k(x)\, \overline{\gamma_{\ell}(x)}\,d\mathrm{m}(x)
= \int_G |\gamma_i(x)\gamma_j(x)|^2\,d\mathrm{m}(x) =1\,.
\]
Combining the the above facts, we obtain
\begin{align*}
\int_G\bigg|\sum_{j=1}^N c_j \gamma_j(x)\bigg|^4 \,d\mathrm{m}(x) & = 2 \sum_{j \neq k}|c_i c_j|^2
+ \sum_{i=1}^N |c_j|^4 \leq 2 \sum_{j, k=1}|c_j c_k|^2 \\
&= 2 \bigg(\int_G \bigg|\sum_{j=1}^N c_j \gamma_j(x)\bigg|^2 \,d\mathrm{m}(x)\bigg)^2\,,
\end{align*}
which yields the required first estimate. For the second claim note that  $|P|^2 = |P|^{\frac{2}{3}} |P|^{\frac{4}{3}}$. 
Hence by H\"older's inequality  
\[
\|P\|_2 \leq \|P\|_1^{\frac{1}{3}}\,\|P\|_4^{\frac{2}{3}}
\leq 2^{\frac{1}{6}}\,\|P\|_1^{\frac{1}{3}}\,\|P\|_2^{\frac{2}{3}}\,,
\]
and so $\|P\|_2 \leq \sqrt{2}\,\|P\|_1$. This completes the proof.
\end{proof}

\smallskip

\begin{remark} \label{remark graham}
We observe that \cite[Proposition 6.3.11]{grahamhare2013}, also implies the following: If $E\subset \widehat{G}$ is
such that there exists $N\ge 2$ such that for every $\gamma\in \widehat{G}$ there at most $N$ pairs of the form
$(\gamma_{i_1},\gamma_{i_2})\subset E\times E$ with $\gamma_{i_1}\gamma_{i_2}=\gamma$, then  $E$ is $\Lambda(2)$-set
with $C_2\leq \sqrt{N}$.
\end{remark}

\medskip

\section{The Lozinski-Kharishiladze result revisited} \label{The Lozinski and Kharishiladze revisited}

The aim is to prove various multivariate  variants of the Lozinski-Kharshiladze result from \eqref{LoKa}. 
 This is done by concrete applications of \,Theorem \ref{rudy}, which are mostly routine. We point out
that these results are the key for further results to obtain formulas and estimates for the projection constants 
of spaces of Dirichlet polynomials and Boolean functions (see Chapter \ref{Part: Dirichlet polynomials and polynomials on the Boolean cube} and Chapter \ref{Polynomials on the Boolean cube}).

More precisely, in this section we derive formulas and asymptotically optimal estimates for the projection constants 
of spaces of trigonometric polynomials on products of compact abelian groups, which  have Fourier expansions supported 
on an a~priori fixed  set of characters. Coming back to our original motivation, the Lozinski-Kharshiladze result 
\eqref{LoKa}, we finally specify our results to trigonometric and analytic trigonometric  polynomials on the 
multidimensional circle group.

\subsection{Polynomials on products of groups} \label{Polynomials on products of groups}
We need to introduce further notation.
 Given compact abelian groups $G_1, \ldots, G_n$, each with the Haar measure $\mathrm{m}_j, 1\leq j\leq n$, we denote  by
$G:= G_1 \times \cdots \times G_n$ the product of these groups endowed  its natural  product operation and product topology . Given an $n$-tuple
of characters $(\gamma_1, \ldots, \gamma_n) \in \widehat{G_1} \times \cdots \times \widehat{G_n}$ and
$\alpha = (\alpha_1, \ldots, \alpha_n)\in \mathbb{Z}^n$, we write $\gamma^\alpha$ for the~character in $\widehat{G}$
given by
\[
\gamma^\alpha(x_1,\ldots, x_n) := \gamma_1(x_1)^{\alpha_1}\cdots \gamma_n(x_n)^{\alpha_n}, \quad\, (x_1, \ldots, x_n)\in G\,.
\]
Based on this notation, we formulate our first result.

\smallskip

\begin{proposition}
\label{product}
Given  $m, n \in \mathbb{N} $, let $E:= \big\{\gamma^{\alpha} \colon \alpha = (\alpha_1, \ldots, \alpha_n) \in \mathbb{Z}^n;\, |\alpha_j| \leq m,
\,\,1\leq j \leq n\big\}$. Then
\[
\boldsymbol{\lambda}\big(\text{Trig}_{{E}}(G)\big) =
\prod_{j=1}^n \int_{G_j} \Big|\sum_{k \in \mathbb{Z}:|k|\leq m} \gamma_j(x_j)^k\Big|\,d{\text{\emph{m}}}_j(x_j)\,.
\]
\end{proposition}

\begin{proof}
Let $\mu$ be the Haar measure on $G:=G_1 \times \cdots \times G_n$, which is nothing else then the
the product of the Haar measures $\mathrm{m}_j$, and $I_m:= \big\{ \alpha = (\alpha_1, \ldots, \alpha_n) \in \mathbb{Z}^n \colon |\alpha_j| \leq m,
\,\,1\leq j \leq n\big\}$. Then Theorem \ref{C(G)proj} yields
\[
\boldsymbol{\lambda}\big(\text{Trig}_{{E}}(G)\big) = \int_G \Big|\sum_{\alpha\in I_m} \gamma^{\alpha}(x) \Big|\,d\mu(x)\,.
\]
Clearly, for all $x=(x_1,\ldots, x_n) \in G_1 \times \cdots \times G_n$ one has
\[
\sum_{\alpha \in I_m} \gamma^\alpha(x)  = \prod_{j=1}^n \sum_{|k|\leq m} \gamma_j(x_j)^k\,,
\]
and hence
\[
\boldsymbol{\lambda}\big(\text{Trig}_{{E}}(G)\big) = \int_G \prod_{j=1}^n \Big|\sum_{|k|\leq m} \gamma_j(x_j)^k\Big|\,d\mu(x)\,.
\]
Fubini's theorem finishes the argument.
\end{proof}

A similar proof yields to the following.

\begin{corollary}
\label{product corollary}
Let $I:=I_1\times \dots\times I_n\subset \mathbb{Z}^n$, with $J_l$ finite subsets of \, $\mathbb Z$ for $l=1,\dots,n$.
Then, for ${E} := \{\gamma^{\alpha}\}_{\alpha \in I}$ one has
\[
\boldsymbol{\lambda}\big(\text{Trig}_{{E}}(G)\big) = \prod_{l=1}^n \int_{G_l} \Big|\sum_{k\in I_l}
\gamma_l(x_l)^k\Big|\,d{\text{\emph{m}}}_l(x_l)\,.
\]
\end{corollary}

\smallskip

\subsection{Polynomials on products of the circle group} \label{Polynomials on polytorii}
As a special case of the preceding results we consider polynomials on on products of the circle group, that is, polynomials on the compact abelian group
$G:= \mathbb{T}^n$.
The Haar measure $\mathrm{m}=:dz$ on $\mathbb{T}^n$ acts on a~Borel function $f\colon \mathbb{T}^n \to \mathbb{C} $ by the formula
\[
\int_{\mathbb{T}^n} f(z)\,dz = \frac{1}{(2\pi)^n} \int_0^{2\pi}\cdots \int_0^{2\pi} f(e^{it_1}, \ldots, e^{it_n})\,dt_1\ldots dt_n\,.
\]
Recall that $\widehat{\mathbb{T}^n} = \mathbb{Z}^n$, where the identification is given by the fact that for every
character $\gamma \in \widehat{\mathbb{T}^n} $ there is a unique multi index
$(\alpha_1, \ldots, \alpha_n) \in \mathbb{Z}^n$ for which  $\gamma(z) = z^{\alpha}$ for every $z= (z_1, \ldots, z_n)\in \mathbb{T}^n$.
Given a finite subset $I \subset \mathbb{Z}^n$, we write  $$\text{Trig}_{I}(\mathbb{T}^n)$$
for the space of all
trigonometric polynomials
$
P(z) = \sum_{\alpha \in I} c_\alpha z^{\alpha}, \quad\, z\in \mathbb{T}^n\,,
$
supported on $I$. Together with the sup norm $\|\cdot\|_{\mathbb{T}^n}$ (also denoted by
$\|\cdot\|_{\infty}$) this space clearly forms a Banach space. In the analytic case $J \subset \mathbb{N}_0^n$
the maximum modulus theorem implies that 
\begin{equation}\label{maxmod}
  \text{Trig}_{J}(\mathbb{T}^n) = \mathcal{P}_{J}(\ell_\infty^n)\,,
\end{equation}
as Banach spaces, where the Fourier and monomial coefficients obviously are preserved.

We start with the  following  immediate consequence of Theorem~\ref{C(G)proj}.
\smallskip

\begin{corollary}  \label{manydimensionsA}
Let $I \subset \mathbb{Z}^n$ be a finite set. Then
\[
\boldsymbol{\lambda}\big(\text{Trig}_{I}(\mathbb{T}^n)\big) = \int_{\mathbb{T}^n} \Big|\sum_{\alpha\in I } z^\alpha\Big|\,dz\,.
\]
\end{corollary}

\smallskip

In the analytic case $J \subset \mathbb{N}_0^n$ we are able to collect more information.  To see this, note first  that by
an inequality due to  Weissler \cite{weissler1980logarithmic}
(see also \cite[Theorem~8.10]{defant2019libro}) for all $P \in \mathcal{P}_{\leq m}(\ell_\infty^n)$
\[
\frac{1}{\sqrt{2}^m}
\Big(\int_{\mathbb{T}^\infty} |P|^2\,dz\Big)^\frac{1}{2}
\leq\int_{\mathbb{T}^\infty} |P|\,dz\,.
\]
In other terms, the set $\{ z^\alpha \in \mathbb{N}_0^n \colon |\alpha| \leq m \}$ (of characters in $\mathbb{Z}^n = \widehat{\mathbb{T}^n}$) forms a $\Lambda(2)$-set with constant $C_2 \leq \sqrt{2^m}$. Then the following result is an immediate consequence of Corollary~\ref{corB2} and  Corollary~\ref{product}.

\begin{corollary}  \label{manydimensionsB}
Let $J \subset \mathbb{N}_0^n$ be a finite set. Then
\[
\boldsymbol{\lambda}\big(\mathcal{P}_{J}(\ell_\infty^n)\big) = \int_{\mathbb{T}^n} \Big|\sum_{\alpha\in J } z^\alpha\Big|\,dz\,,
\]
and if $J$ has degree $m$, then
\[
\frac{1}{\sqrt{2^m}}\sqrt{|J|}  \leq \boldsymbol{\lambda}\big(\mathcal{P}_{J}(\ell_\infty^n)\big) \leq  \sqrt{|J|} .
\]
\end{corollary}

As announced, we finally come to multivariate variants of the
the Lozinski-Kharshiladze result mentioned in \eqref{LoKa}.

\smallskip

For each $m\in \mathbb{N}_0$,  let $D_m$  be the $m$th Dirichlet kernel $D_m:= \sum_{k=-m}^{m} e_k$, and $L_m$ the $m$th
Lebesgue constant given by
\[
L_m:= \frac{1}{2\pi} \int_0^{2\pi} |D_m(t)|\,dt
= \frac{1}{2\pi} \int_0^{2\pi} \bigg|\frac{\sin (m + \frac{1}{2})t}{\sin\frac{t}{2}}\bigg|\,dt\,.
\]
We recall the well-known standard estimates 
\begin{equation}\label{Lebconst}
\frac{4}{\pi^2} \log(m +1) \,\,< \,\,L_m \,\,<\,\, 
3 + \log m, \quad\, m\in \mathbb{N}\,.
\end{equation}

\smallskip

\begin{corollary}  \label{Lozinski-Kharshiladze}
For $\mathbf{d}=(d_1,\dots,d_n)\in \mathbb{N}^n$ let $I_\mathbf{d}:= \big\{\alpha = (\alpha_1, \ldots, \alpha_n)
\in \mathbb{Z}^n \colon \, |\alpha_j| \leq d_j, \,\,1\leq j \leq n\big\}$. Then
\[
\boldsymbol{\lambda}\big(\text{Trig}_{I_{\mathbf{d}}}(\mathbb{T}^n)\big) = \prod_{j=1}^n L_{d_j}\,.
\]
\end{corollary}

\begin{proof}
Applying Corollary \ref{product corollary}, we get
\begin{align*}
\boldsymbol{\lambda}\big(\text{Trig}_{I_{\mathbf{d}}}(\mathbb{T}^n)\big) = \prod_{j=1}^n \int_{\mathbb{T}} \Big|\sum_{|\alpha_j|\le d_j } z^{\alpha_j}\Big|\,dz
= \prod_{j=1}^n \int_{\mathbb{T}} |D_{d_j}(z_j)|\,dz_j = \prod_{j=1}^n L_{d_j}\,,
\end{align*}
as required.
\end{proof}

In order to state the 'analytic counterpart' of Corollary~\ref{Lozinski-Kharshiladze}, we define for each $m\in \mathbb{N}_0$,
\[
\text{$D_m^{+}:= \sum_{k=0}^{m} e_k$ \,\,\,\,and \,\,\,\, $L_m^{+}:= \frac{1}{2\pi} \int_0^{2\pi} |D_m^{+}(t)|\,dt$\,.}
\]

In a similar fashion we prove the following result for analytic polynomials.

\medskip

\begin{corollary} \label{LoKha}
For each $m\in \mathbb{N}$ let $\Lambda_m:= \big\{(\alpha_1, \ldots, \alpha_n) \in \mathbb{N}_0^n ; \,\, \alpha_j \leq m,\, \,\,
1\leq j\leq n\big\}$. Then
\begin{itemize}
\item[\rm{(i)}] $\boldsymbol{\lambda}\big(\mathcal{P}_{\Lambda_m}(\ell_\infty^n)\big) = (L_m^{+})^n$
\item[\rm{(ii)}] $\boldsymbol{\lambda}\big(\mathcal{P}_{\Lambda_{2m}}(\ell_\infty^n)\big) = (L_m)^n$
\item[\rm{(iii)}] $(L_m - 1)^n \leq \boldsymbol{\lambda}\big(\mathcal{P}_{\Lambda_{2m+ 1}}(\ell_\infty^n)\big) \leq (L_m + 1)^n$
\item[\rm{(iv)}] $\boldsymbol{\lambda}\big(\mathcal{P}_{\Lambda_m}(\ell_\infty^n)\big) \asymp (1 + \log m)^n$
\end{itemize}
\end{corollary}

\begin{proof}
By Proposition \ref{product}, it follows that
\begin{align*}
\boldsymbol{\lambda}\big(\text{Trig}_{\Lambda_m}(\mathbb{T}^n)\big)
=  \prod_{j=1}^n \int_{\mathbb{T}} \Big|\sum_{0\le \alpha_j \le m} z^{\alpha_j}\Big|\,dz
= \prod_{j=1}^n \int_{\mathbb{T}} |D_m^{+}(z_j)|\,dz_j = (L_m^{+})^n\,.
\end{align*}
Clearly, $\mathcal{P}_{\Lambda_m}(\ell_\infty^n)$ is isometrically isomorphic to $\text{Trig}_{\Lambda_m}(\mathbb{T}^n)$
(see again \eqref{maxmod}), so
the proof of (i) is complete.

(ii) It is easy to check that
\[
D_m^{+}(t) = e^{i\frac{mt}{2}}\,\frac{\sin (m + \frac{1}{2})t}{\sin\frac{t}{2}}, \quad\, t \in (0, 2\pi)\,.
\]
This implies that $L_m = \|D_m\|_{L_1(\mathbb{T})} = \|D_{2m}^{+}\|_{L_1(\mathbb{T})}=L_{2m}^+$, so the statement follows
from (i).

(iii) The statement follows by (ii) combined with the obvious estimates
\[
\|D_{2m}^{+}\|_{L_1(\mathbb{T})} - 1 \leq \|D_{2m+1}^{+}\|_{L_1(\mathbb{T})} \leq \|D_{2m}^{+}\|_{L_1(\mathbb{T})} + 1,
\quad\, m\in \mathbb{N}\,.
\]

(iv) Combining the estimates from \eqref{Lebconst} with those from  (ii) and (iii), we get the required
equivalence.
\end{proof}

We conclude with a limit formula.

\begin{corollary}
For each $m\in \mathbb{N}$ let
\[
J_m:= \big\{(\alpha_1, \ldots, \alpha_n) \in \mathbb{Z}_0^n ; \,\, \alpha_j \leq m,\, \,\,
1\leq j\leq n\big\}
\quad \text{and} \quad
\Lambda_m:= \big\{(\alpha_1, \ldots, \alpha_n) \in \mathbb{N}_0^n ; \,\, \alpha_j \leq m,\, \,\,
1\leq j\leq n\big\}\,.
\]
Then \[
\lim_{m\to \infty} \frac{\boldsymbol{\lambda}\big(\text{Trig}_{J_m}(\mathbb{T}^n)\big)}{\log^n m}
=
\lim_{m\to \infty} \frac{\boldsymbol{\lambda}\big(\mathcal{P}_{\Lambda_m}(\ell_\infty^n)\big)}{\log^n m} = \Big(\frac{4}{\pi^2}\Big)^n\,.
\]
\end{corollary}

\begin{proof}
Recall the well-known formula
\begin{equation}\label{lebformula}
                               \lim_{m \to \infty} \frac{L_m}{\log m} = \frac{4}{\pi^2}\,.
                             \end{equation}
Then the first limit follows from Corollary~\ref{Lozinski-Kharshiladze}.
For the proof of the second formula we distinguish the two sequences 
$\big(\boldsymbol{\lambda}\big(\mathcal{P}_{\Lambda_{2m}}(\ell_\infty^n)\big)/ \log (2m)\big)_m$  
and $\big(\boldsymbol{\lambda}\big(\mathcal{P}_{\Lambda_{2m+1}}(\ell_\infty^n)\big)/ \log (2m+1)\big)_m$. 
Observe that both sequences converge to$\frac{4}{\pi^2}$. Indeed, by  Corollary~\ref{LoKha}(ii) we have that
$
\boldsymbol{\lambda}\big(\mathcal{P}_{\Lambda_{2m}}(\ell_\infty^n)\big) = (L_{m})^n\,,
$
which, using \eqref{lebformula}, gives the claim for the first sequence. The second sequence  is handled the same way using Corollary~\ref{LoKha}(iii).
\end{proof}

\bigskip

\chapter{Polynomials on Hilbert spaces}
\label{eukli}

This chapter is motivated by the~remarkable work \cite{ryll1983homogeneous} of Ryll and Wojtaszczyk. Solving a~problem of S.~Winger, they proved that the inclusion  $H_\infty(B_{\ell_2^n}) \hookrightarrow H_1(B_{\ell_2^n})$ is not compact. As  noted in \cite{ryll1983homogeneous} this result is intimately connected with a~problem posed by Rudin in his monograph \cite{rudin1980}: Does there exist an inner function on the open unit ball of the Hilbert space $\ell_2^n$, $n>1$. For detailed information on all this see \cite{rudin1985} and \cite{wojtaszczyk1996banach}.

Much of this deep cycle of ideas (from complex analysis on $B_{\ell_2^n}$) relies on a~concrete formula for the projection
constant of $\mathcal{P}_{m}(\ell_2^n)$  proved  in \cite{ryll1983homogeneous} (see also \cite[III.B.15]{wojtaszczyk1996banach}
and \cite{rudin1985}), namely that for  each $m,n\in \mathbb{N}$ with $n >1$  one has
\begin{align} \label{fascinating}
\boldsymbol{\lambda}\big(\mathcal{P}_m(\ell_2^n)\big) = \frac{\Gamma(n +m) \Gamma(1 + \frac{m}{2})}{\Gamma(1 + m)\Gamma(n + \frac{m}{2})}\,.
\end{align}
A simple calculus yields
\[
\boldsymbol{\lambda}\big(\mathcal{P}_{m}(\ell_2^n)\big) \leq 2^{n-1}, \quad\, m\in \mathbb{N}\,.
\]
The  case $n=2$ is of particular interest. Indeed in the mentioned paper \cite{ryll1983homogeneous} the authors noticed
the~surprising fact that the sequence $\{X_m\}_{m\geq 1}$ with $X_m:= \mathcal{P}_{m}(\ell_2^2)$ forms the first
known example of finite-dimensional Banach spaces for which $\lim_{m \to \infty} \dim X_m =\infty$ although $\sup_m \boldsymbol{\lambda}(X_m) <~\infty$.
 It is worth noting here that Bourgain \cite{bourgain1989} gave an affirmative solution to a~problem considered in \cite{ryll1983homogeneous},
showing that
\[
\sup_m d\big(X_m, \ell_\infty^{\dim X_m}\big)<~\infty\,.
\]

Motivated by the fascinating formula from \eqref{fascinating}, our main goal in this chapter is to prove, for a certain class of index sets $J \subset \mathbb{N}_0^n$, variants of this formula  for the projection constant of  $\mathcal{P}_{J}(\ell_2^n)$.
As a~by-product, we again recover \eqref{fascinating}.

\smallskip

\section{Orthogonal projection} \label{Orthogonal projection}
We need some more  notation. Given $J \subset \mathbb{N}_0^n$ and $1\leq r <\infty$, we define for every
$Q\in \mathcal{P}_{J}(\mathbb{C}^n)$, the norm
\[
\|Q\|_r := \Big(\int_{\mathbb{S}_n} |Q(\xi)|^r\,d\sigma(\xi)\Big)^{\frac{1}{r}}\,,
\]
where $\sigma$ stands for the rotation-invariant normalized surface measure on $\mathbb{S}_n$.
We abbreviate
\[
\mathcal{P}_{J}^{r}(\mathbb{S}_n):= \big(\mathcal{P}_{J} (\ell_2^n),\|\cdot\|_r\big)\,,
\]
thus
\[
\mathcal{P}_{J}^{\infty}(\mathbb{S}_n) := \mathcal{P}_{J} (\ell_2^n)\,.
\]
Recall that the measure $\sigma$ is  invariant  under the action of the unitary group $\mathcal{U}_n$, that is,
for every Borel function $f\colon \mathbb{S}_n \to \mathbb{C}$ and every $U\in \mathcal{U}_n$ one has
\[
\int_{\mathbb{S}_n} f(U\xi)\,d\sigma(\xi) = \int_{\mathbb{S}_n} f(\xi)\,d\sigma(\xi)\,.
\]
By the rotation invariance under $U\xi = \xi e^{i\theta}$, we get
\[
\int_{\mathbb{S}_n} \xi^{\alpha}\, \overline{\xi^\beta}\,d\sigma(\xi) = 0,
\quad\, \alpha, \beta \in \mathbb{N}_0^n,\,\, \alpha \neq \beta\,.
\]
This in particular shows that the collection $(e_\alpha)_{\alpha \in \mathbb{N}_0^n}$ of all monomials
\[
e_\alpha (z) := z^\alpha, \quad\, z \in \mathbb{S}_n
\]
is an orthogonal family in $L_2(\mathbb{S}_n)$, and consequently the  collection $(f_\alpha)_{\alpha \in \mathbb{N}_0^n}$
of all normalized monomials given by
\[
f_\alpha : = \frac{e_\alpha}{\|e_\alpha\|_2}, \quad\, \alpha \in \mathbb{N}_0^n\,,
\]
forms  an orthonormal system in $L_2(\mathbb{S}_n)$. The following concrete formula for $\|e_\alpha\|_2$ is central
(see, e.g., \cite[Proposition~1.4.9]{rudin1980}):
\begin{equation} \label{normalized}
\|e_\alpha\|_2 = \bigg(\int_{\mathbb{S}_n} |\xi^\alpha|^2\,d\sigma(\xi)\bigg)^{\frac{1}{2}}
= \sqrt{\frac{(n-1)!\,\alpha!}{(n-1 +|\alpha|)!}}, \quad\, \alpha\in \mathbb{N}_0^n\,.
\end{equation}
The orthogonal projection $\textbf{P}_J$ of $L_2(\mathbb{S}_n)$ onto $\mathcal{P}_{J}^{2}(\mathbb{S}_n)$ is given by
\[
\textbf{P}_J(f):=\sum_{\alpha \in J} \langle f, f_\alpha \rangle f_\alpha, \quad\, f\in L_2(\mathbb{S}_n)\,.
\]
\smallskip
As a consequence we see that
\[
\boldsymbol{\lambda} \big(\mathcal{P}_{J} (\ell_2^n)\big) \leq \big\|\textbf{P}_J\colon C(\mathbb{S}_n)  \to \mathcal{P}_{J}(\ell_2^n)\big\|
\leq \big\|\textbf{P}_J\colon L_\infty(\mathbb{S}_n)  \to \mathcal{P}_{J}(\ell_2^n)\big\|\,.
\]
In what follows we are interested in  obtaining more precise information on concrete formulas for $\|\textbf{P}_J\|$,
at least for certain classes of index sets $J$.

We claim that the following integral formula for the projection $\textbf{P}_J$ holds:
\begin{align}
\label{integral}
\textbf{P}_J f(z) =\int_{\mathbb{S}_n} f(\xi)\,\sum_{k=0}^{m} c_k(n) [z, \xi ]_{J,k}\,d\sigma(\xi), \quad\, f \in L_2(\mathbb{S}_n),
\, \,z \in \ell_2^n\,,
\end{align}
where
\[
c_k(n):= \frac{(n-1 + k)!}{(n-1)!\,k!}, \quad\, k \in \N_0\,,
\]
\[
[z, \xi ]_{J,k}  : = \sum_{\substack{\alpha \in J\\|\alpha|=k }} \frac{k!}{\alpha !} z^{\alpha} \overline{\xi}^{\alpha}, \quad\,
z,\,\,\xi \in \mathbb{C}^n\,.
\]
Indeed, by \eqref{normalized} we have
\begin{align*}
\textbf{P}_J f(z) & = \sum_{\alpha \in J} \bigg(\int_{\mathbb{S}_n} f(\xi) \frac{\overline{\xi}^\alpha}{\|e_\alpha\|_2}\, d\sigma(\xi)\bigg)\frac{z^\alpha}{\|e_\alpha\|_2}\, \\
& =\int_{\mathbb{S}_n} f(\xi) \sum_{k=0}^{m} c_k(n) \sum_{\substack{\alpha \in J\\|\alpha|=k }} \frac{k!}{\alpha !} z^{\alpha}
\overline{\xi}^{\alpha}\,\, d\sigma(\xi) = \int_{\mathbb{S}_n} f(\xi)\,\sum_{k=0}^{m} c_k(n) [z, \xi ]_{J,k}\,d\sigma(\xi)\,.
\end{align*}

\smallskip
\noindent
The following simple integral description for the norm of the orthogonal projection $\textbf{P}_J$ as an operator on $L_\infty(\mathbb{S}_n) $
(resp. $C(\mathbb{S}_n)$) is fundamental.

\begin{proposition} \label{formel}
Let $J \subset \N_0^n$ be a finite index set,  and $m= \max_{\alpha \in J}|\alpha|$. Then
\begin{align*}
\big\|\textbf{P}_J\colon L_\infty(\mathbb{S}_n)  \to \mathcal{P}_{J} (\ell_2^n)\big\| & =
\big\|\textbf{P}_J\colon C(\mathbb{S}_n)  \to \mathcal{P}_{J} (\ell_2^n)\big\|\\
& = \sup_{z \in \mathbb{S}_n}
\int_{\mathbb{S}_n} \Big| \sum_{k=1}^{m} c_k(n)\,[z, \xi ]_{J,k} \Big|\,d\sigma(\xi)\,.
\end{align*}
In particular,
\[
\boldsymbol{\lambda}\big(\mathcal{P}_{J} (\ell_2^n)\big)
\leq \sup_{z \in \mathbb{S}_n}
\int_{\mathbb{S}_n} \Big| \sum_{k=1}^{m} c_k(n)\,\, [z, \xi ]_{J,k} \Big|\,d\sigma(\xi)\,.
\]
\end{proposition}

\begin{proof}
Consider the continuous function  $g\colon \mathbb{S}_n \times \mathbb{S}_n \to \mathbb{C}$
given by
\[
g(z,\xi) : = \sum_{k=0}^{m} c_k(n) [z, \xi ]_{J,k}, \quad\, z, \xi \in \mathbb{C}^n\,.
\]
From \eqref{integral}  we deduce that
\[
\big\|\textbf{P}_J\colon  L_\infty(\mathbb{S}_n)  \to \mathcal{P}_{J} (\ell_2^n)\big\|
\leq \sup_{z \in \mathbb{S}_n} \int_{\mathbb{S}_n} |g(z,\xi)|\,d\sigma(\xi)\,.
\]
To see that this estimate is in fact an equality, we for every $z\in \mathbb{S}_n$ take $f_z\in L_\infty(\mathbb{S}_n)$ given by
\[
f_z(\xi):= \sign g(z,\xi)\,.
\]
Note that $\|f_z\|_{\infty}= 1$. In the case of $C(\mathbb{S}_n)$ instead of $L_\infty(\mathbb{S}_n)$, we also have
\begin{align*}
\big\|\textbf{P}_J\colon  C(\mathbb{S}_n)  \to \mathcal{P}_{J} (\ell_2^n)\big\| & =
\sup_{\|f\|_\infty \leq 1}\sup_{z \in \mathbb{S}_n} \Big|\int_{\mathbb{S}_n} f(\xi) g(z, \xi)\,d\sigma(\xi)\Big| \\
& \ge \sup_{z \in \mathbb{S}_n} \Big|\int_{\mathbb{S}_n}  f_z(\xi)g(z, \xi)\,d\sigma(\xi)\Big|
=  \sup_{z \in \mathbb{S}_n}  \int_{\mathbb{S}_n}  |g(z, \xi)|\,d\sigma(\xi)\,. \qedhere
\end{align*}
\end{proof}

\smallskip

\section{Integral formula} \label{Hpq-C^n}

We consider the natural representation of the unitary group $\mathcal{U}_n$ in $L_\infty(\mathbb{S}_n)$ (and so also
in $C(\mathbb{S}_n)$) defined by
\[
T_{U}f : = f\circ U, \quad\, (U, f) \in \mathcal{U}_n \times L_\infty(\mathbb{S}_n)\,.
\]
Note that all these operators $T_U$ define isometries on $L_\infty(\mathbb{S}_n)$ (resp. $C(\mathbb{S}_n)$). In what
follows, we say that an operator $R$ on $L_\infty(\mathbb{S}_n)$ (resp. $C(\mathbb{S}_n)$) commutes with the unitary
group $\mathcal{U}_n$ whenever  $R$ commutes with $T_U$, for all $U \in \mathcal{U}_n$. A~subset $X$ of
$L_\infty(\mathbb{S}_n)$ (resp. $C(\mathbb{S}_n)$) is said to $\mathcal{U}_n$-invariant whenever
$T_{U}f =  f\circ U \in X$ for all $f \in X$ and $U \in \mathcal{U}_n$.

The  main result of this section  shows that for each finite index set $J \subset \N_0^n$ for which
$\mathcal{P}_{J} (\ell_2^n)$ is $\mathcal{U}_n$-invariant, the  orthogonal projection $\textbf{P}_{J}$ in fact is
the minimal projection from $L_\infty(\mathbb{S}_n)$ (resp. $C(\mathbb{S}_n)$) onto $\mathcal{P}_{J}(\ell_2^n)$.

\begin{theorem}
\label{polleqm}
Let $J \subset \N_0^n$  be a finite  index set such that $\mathcal{P}_{J} (\ell_2^n)$ is $\mathcal{U}_n$-invariant.
 Then
\[
\boldsymbol{\lambda}\big(\mathcal{P}_{J}(\ell_2^n), C(\mathbb{S}_n)\big) = \big\|\textbf{P}_{J}:C(\mathbb{S}_n) \to \mathcal{P}_{J}(\ell_2^n)\big\| \,.
\]
In particular,
\[
\boldsymbol{\lambda}\big( \mathcal{P}_{J} (\ell_2^n)\big)
= \sup_{z \in \mathbb{S}_n}
\int_{\mathbb{S}_n} \Big| \sum_{k=1}^{m} c_k(n)\,\, [z, \xi ]_{J,k} \Big|\,d\sigma(\xi)\,,
\vspace{2mm}
\]
where $m= \max_{\alpha \in J}|\alpha|$, and $c_k(n)$ and $[z, \xi ]_{J,k}$ are defined as for   the formula \eqref{integral}.
\end{theorem}

The structure of the proof is  the same as the one we followed  in the previous section: In a first step we prove that if $\textbf{Q}$
is a projection on $C(\mathbb{S}_n)$ onto $\mathcal{P}_{J}(\ell_2^n)$ that commutes with the unitary group $\mathcal{U}_n$, then
$\textbf{P}_{J} =\textbf{Q}$, and then in a second step, we apply Rudin's technique from  Theorem \ref{rudy}.

The first step towards the proof of Theorem~\ref{polleqm} is isolated in the following proposition.

\begin{proposition} \label{unique}
Let $J \subset \N_0^n $ be a  finite index set, and  $\textbf{Q}$  a~projection on $C(\mathbb{S}_n)$ onto $\mathcal{P}_{J}(\ell_2^n)$,
which commutes with the unitary group $\mathcal{U}_n$. Then $\textbf{P}_{J} = \textbf{Q}$ \, on\,  $C(\mathbb{S}_n)$.
\end{proposition}

The proof  is   based on two results (see Lemma~\ref{RudA} and Lemma~\ref{RudB}) carried out carefully in Rudin's monograph \cite{rudin1980};
for the sake of completeness we sketch some  essential details.

We write $\mathfrak P (\mathbb{C}^n)$  for all polynomials $f\colon \mathbb{C}^n \to \mathbb{C}$ of the form
 \begin{align*}
f\left(z\right)=\sum_{(\alpha,\beta) \in I} c_{(\alpha,\beta)}(f)\,\,z^\alpha\overline{z}^\beta,
\end{align*}
where  $I$  is a finite subset of pairs $(\alpha, \beta) \in\mathbb{N}_0^n \times\mathbb{N}_0^n$, and
 $\big(c_{(\alpha,\beta)}(f)\big)_{(\alpha, \beta) \in I}$ are complex coefficients. Observe that for every $(\alpha,\beta) \in I$
\[
c_{(\alpha,\beta)}(f) =  \frac{1}{\alpha! \beta!} \frac{\partial^kf}{\partial z^\alpha \partial\overline{z}^\beta}(0)\,,
\]
which in particular shows that the monomial coefficients $c_{(\alpha,\beta)}(f)$ of each $f \in \mathfrak P (\mathbb{C}^n)$
are unique. A~polynomial $f \in  \mathfrak P (\mathbb{C}^n)$ is said to be harmonic, whenever
\[
\Delta f  := \sum_{i}\frac{\partial^2 f}{\partial z_{i}\partial\overline{z_{i}}}
=
\frac{1}{4}\sum_{i}\,\,\Big(\frac{\partial^2 f}{\partial x_{i}^2  } + \frac{\partial^2 f}{\partial y_{i}^2  }\Big) =0\,,
\]
and we define $\mathfrak{H} \big( \mathbb{C}^n\big)$
to be the subspace of all harmonic polynomials in $\mathfrak P (\mathbb{C}^n)$.

By $\mathfrak P (\mathbb{S}_n)$ we denote the linear space of all restrictions $f|_{\mathbb{S}_n}$ of polynomials$f \in \mathfrak P (\mathbb{S}_n)$ to the unitary  group. A~restrictions of a~harmonic polynomial to $\mathbb{S}_n$ is called a~spherical harmonic, and together with the supremum norm on $\mathbb{S}_n$ the linear space spherical harmonics forms the Banach space  $\mathfrak H (\mathbb{S}_n)$. Moreover, we need  to introduce a~scale of subspaces of
$\mathfrak H (\mathbb{S}_n)$. Given $p,q\in\mathbb N_0$, let
\[
\mathfrak H_{(p,q)}(\mathbb{S}_n)\subset \mathfrak{H}(\mathbb{S}_n)
\]
denote the space of all spherical harmonics which  are $p$-homogeneous in $z=(z_{i})$ and $q$-homo\-geneous in $\overline{z}
=(\overline{z_{i}}),$ i.e., all $f \in  \mathfrak{H}(\mathbb{S}_n)$ such that $c_{(\alpha,\beta)}(f) \neq 0$
only if $|\alpha|= p$ and $|\beta|= q$.  It is clear that
\begin{equation}\label{denso3}
\mathfrak H (\mathbb{S}_n)=\text{span}\big\{\mathfrak H_{(p,q)}(\mathbb{S}_n) \,:\, p,q\in\mathbb N_0\big\}\,,
\end{equation} and  for each $k\ge 1$ the mapping
\[
\mathcal P_k(\ell_2^n) =  \mathfrak H_{(k,0)}(\mathbb{S}_n)\,,\quad  f \mapsto f|_{\mathbb{S}_n}\,,
\]
is an isometric linear bijection.We need the following two important results from \cite{rudin1980}. The first one is taken from  \cite[Corollary of Theorem 12.1.3]{rudin1980} and shows that the collection of all spherical harmonics is dense in $C(\mathbb{S}_n)$.

\begin{lemma}\label{RudA}
$\mathfrak H (\mathbb{S}_n)$ is dense in $C(\mathbb{S}_n)$, and hence also in $L^2(\mathbb{S}_n).$
\end{lemma}

The second tool is a special case of a more general result of Nagel and Rudin -- presented in detail in \cite[Theorem 12.3.6 and Theorem 12.3.8]{rudin1980}.

\begin{lemma} \label{RudB}
Suppose that $T\colon C(\mathbb{S}_n) \to C(\mathbb{S}_n)$ is a~continuous linear operator, which commutes with the unitary group $\mathcal{U}_n$. Then for every choice of $p,q \in \mathbb{N}_0$ there exist $c(p,q) \in \mathbb{C}$ such that $Tf = c(p,q) f$ for all $f \in \mathfrak H_{(p,q)} (\mathbb{S}_n)$.
\end{lemma}

Finally, we present the proof of Proposition~\ref{unique} first, and then second the proof of
Theorem~\ref{polleqm}.

\begin{proof}[Proof of Proposition~$\ref{unique}$]
Indeed,  by Lemma~\ref{RudB}  for  each $p,q\ge0$ there is a~constant $c(p,q)$ such that $\textbf{Q}f=c(p,q)f$ for all each $f \in \mathfrak H_{(p,q)}(\mathbb{S}_n)$. The fact that $\textbf{Q}$ is a~projection implies that $c(p,q)=1$ if $(p,q)=(|\alpha|,0)$ for some $\alpha\in J$, and $c(p,q)=0$ otherwise. Now, by Lemma~\ref{RudA}, we conclude that
$\textbf{Q} = \textbf{P}_J$  on $C(\mathbb{S}_n)$.
\end{proof}

\begin{proof}[Proof of Theorem~$\ref{polleqm}$]
Let $\textbf{P}$ be any projection on $C(\mathbb{S}_n)$ onto $\mathcal{P}_{J}(\ell_2^n)$. We apply Theorem~\ref{rudy}
with $X=C(\mathbb{S}_n),$ $G=\mathcal{U}_n$ and $Y=\mathcal{P}_{J}(\ell_2^n)$ to see that
\begin{equation} \label{tag}
\textbf{Q}f = \int_{\mathcal{U}_n} T_{U} \textbf{P} T_{U^{-1}}(f)\,dU, \quad\, f\in C(\mathbb{S}_n)
\end{equation}
defines  a~projection on $C(\mathbb{S}_n)$ onto $\mathcal{P}_{J}(\ell_2^n)$, which commutes with  the action of $\mathcal{U}_n$
on $C(\mathbb{S}_n)$ (clearly  $dU$ denotes the normalized Haar measure on $\mathcal{U}_n$). But then $\textbf{Q} =\textbf{P}_{J}$
by Proposition~\ref{unique}. Now recall that all operators $T_U$ define isometries on $C(\mathbb{S}_n)$, thus Theorem~\ref{rudy}
also   yields  that $\|\textbf{P}_{J}\| = \|\textbf{Q}\| \leq  \|\textbf{P}\|$. Finally, the
integral formula we claim, follows from Proposition~\ref{formel}.
\end{proof}

We finish with a~slight reformulation of Theorem~\ref{polleqm}.

\begin{corollary}
Let $J \subset \N_0^n $ be a finite index set. Then $\mathcal{P}_{J} (\ell_2^n)$ is $\mathcal{U}_n$-invariant if and only if
$\textbf{P}_J$ commutes with  $\mathcal{U}_n$. Moreover, in this case
\[
\boldsymbol{\lambda}\big(\mathcal{P}_{J}(\ell_2^n), C(\mathbb{S}_n)\big)
= \big\|\textbf{P}_{J}\colon C(\mathbb{S}_n) \to \mathcal{P}_{J}(\ell_2^n)\big\| \,.
\]
\end{corollary}

\begin{proof}
Assume that $\mathcal{P}_{J} (\ell_2^n)$ is $\mathcal{U}_n$-invariant. Since the finite dimensional Banach space
$\mathcal{P}_{J} (\ell_2^n)$ is complemented in $C(\mathcal{U}_n)$, by \eqref{tag} there is a projection $\textbf{Q}$
on $C(\mathcal{U}_n)$ onto $\mathcal{P}_{J} (\ell_2^n)$ which commutes with $\mathcal{U}_n$. But then $\textbf{Q} = \textbf{P}_J$
by Theorem~\ref{polleqm}, so $\textbf{P}_J$ commutes with  $\mathcal{U}_n$. Conversely, if $f \in \mathcal{P}_{J} (\ell_2^n)$,
then  it follows by the commutativity assumption that
\[
 f \circ U = \textbf{P}_J(f) \circ U =   \textbf{P}_J(f \circ U) \in \mathcal{P}_{J} (\ell_2^n), \quad\, U \in \mathcal{U}_{n}\,.
\]
This completes the proof.  
\end{proof}

\smallskip

\section{Unitarily invariant index sets}

A  finite index set
$J \subset \N_0^n $ is called $\mathcal{U}_n$-invariant, whenever for each $k  \in \N_0 $, we have that
\[
J_k=\big\{\alpha \in J \colon |\alpha| =k\big\}
\]
either equals $\Lambda(k,n)$ or is the empty set. In other terms,
$J \subset \N_0^n $ is  $\mathcal{U}_n$-invariant if and only if there are $k_1, \ldots, k_m \in~\N_0$ with
$0 \leq k_1 < \ldots < k_m = \max_{\alpha \in J  } |\alpha|$ such that 
\begin{align}\label{decompo}
J = \bigcup_{\ell =1}^m \Lambda(k_\ell,n)\,.
\end{align}
Moreover, observe that in this case every $P \in \mathcal{P}_J(\ell_2^n)$ has a~unique decomposition
\begin{equation*}
  \text{
$P = \sum_{\ell =1}^m  P_\ell$ \quad
with \quad $P_\ell \in  \mathcal{P}_{k_\ell} (\ell_2^n)$.}
\end{equation*}
Basic examples of $\mathcal{U}_n$-invariant index sets $J \subset \N_0^n $ are $J = \Lambda(m,n)$ and
$J = \Lambda(\leq m,n)$. The following immediate consequence motivated the name
'$\mathcal{U}_n$-invariant'.

\begin{remark}
If $J \subset \N_0^n $ is  $\mathcal{U}_n$-invariant, then $\mathcal{P}_{J} (\ell_2^n)$ is $\mathcal{U}_n$-invariant.
\end{remark}

We come to the main contribution of this section - an integral formula for the projection constant of
$\mathcal{P}_{J}(\ell_2^n)$, whenever $J \subset \N_0^n$ is  $\mathcal{U}_n$-invariant. This  is a~refinement of the
formula given in Theorem~\ref{polleqm}.

\begin{theorem} \label{main}
Let $J \subset \N_0^n, \, n >1 $ be a $\mathcal{U}_n$-invariant index set. Then
\begin{align*}
\boldsymbol{\lambda}\big(\mathcal{P}_{J}(\ell_2^n)\big) & =
(n-1) \int_{\mathbb{D}}\bigg|\sum_{k:J(k)\neq \emptyset}  \frac{(n-1 + k)!}{(n-1)!\,k!}\,z^k\bigg|\,(1- |z|^2)^{n-2}\,dA(z)
\\[2ex]
& = \frac{(n-1)}{\pi} \int_0^1 \Big(\int_0^{2\pi} \Big|\sum_{k:J(k)\neq \emptyset} \frac{(n-1 + k)!}{(n-1)!\,k!}\,r^k e^{ik\theta}\Big|\,d\theta\Big)(1- r^2)^{n-2}\,r\,dr\,,
\end{align*}
where  $A$  denotes the normalized Lebesgue measure on the open unit disc $\mathbb{D}$.
\end{theorem}

\begin{proof}
From Theorem~\ref{polleqm} we know that
\[
\boldsymbol{\lambda} \big(\mathcal{P}_{J} (\ell_2^n)\big)
= \sup_{z \in \mathbb{S}_n}
\int_{\mathbb{S}_n} \Big| \sum_{k=1}^{m} c_k(n)\,\, [z, \xi ]_{J,k} \Big|\,d\sigma(\xi)\,.
\]
Assuming that $J \subset \N_0^n$ is $\mathcal{U}_n$-invariant, the binomial formula implies that, for every
$z, \xi \in \mathbb{C}^n$ and every $ k \in \N_0$ for which  $J_k \neq \emptyset$, one has
\[
[z, \xi ]_{J,k}  = \sum_{\alpha \in \Lambda(k,n)} \frac{k!}{\alpha !} z^{\alpha} \overline{\xi}^{\alpha}
= \langle z, \xi \rangle^k\,.
\]
We now may apply the standard integral formula (see, e.g., \cite[1.4.4]{rudin1980})
\[
\int_{\mathbb{S}_n} F(\langle z,\xi \rangle)\,d\sigma(\xi) = (n-1)\int_{\mathbb{D}} F(w)\,(1 - |w|^2)^{n-2}\,dA(w),
\quad\, z \in \mathbb{C}^n\,,
\]
which holds for every  continuous function $F\colon \mathbb{D} \to \mathbb{C}$. Indeed, using this result, the
first required equality   follows taking $F(w):= \big|\sum_{k=0}^m c_k(n)\,w^k\big|$ for all $w\in \mathbb{D}$.
Passing to polar coordinates the second one then is an immediate consequence.
\end{proof}

For any $\mathcal{U}_n$-invariant  index set $J \subset \N_0^n,\, n >1 $ and $m= \max_{\alpha \in J}|\alpha|$ we know
by Cauchy inequalities that $\mathcal{P}_{m}(\ell_2^n)$ is a $1$-complemented subspace of $\mathcal{P}_{J}(\ell_2^n)$,
so in particular
\[
\boldsymbol{\lambda}\big(\mathcal{P}_{m}(\ell_2^n)\big) \leq \boldsymbol{\lambda}\big(\mathcal{P}_{J}(\ell_2^n)\big)
\]
(see the forthcoming Proposition~\ref{Cauchy}). The following simple consequence of Theorem~\ref{main} is a~far reaching
converse.

\begin{corollary} \label{limsup}
Let $J \subset \N_0^n,\, n >1 $ be a~$\mathcal{U}_n$-invariant index set and $m= \max_{\alpha \in J}|\alpha|$. Then
\[
\lim_{m\to \infty  } \sup_{n>1} \sqrt[m]{\frac{ \boldsymbol{\lambda}\big(\mathcal{P}_{J}(\ell_2^n)\big)}{\boldsymbol{\lambda}\big(\mathcal{P}_{m}(\ell_2^n)\big)}} =1\,.
\]
\end{corollary}

\begin{proof}
The triangle inequality applied to the formula from Theorem~\ref{main} shows that
\[
\boldsymbol{\lambda}\big(\mathcal{P}_{J}(\ell_2^n)\big) \leq
\sum_{k:J(k)\neq \emptyset}  \boldsymbol{\lambda}\big(\mathcal{P}_{k}(\ell_2^n)\big) \leq (m+1)\, \boldsymbol{\lambda}\big(\mathcal{P}_{J}(\ell_2^n)\big)\,,
\]
where we for the last estimate (as before) anticipate Proposition \ref{Cauchy}, which shows that $\mathcal{P}_{k}(\ell_2^n)$
may be viewed as a $1$-complemented subspace of $\mathcal{P}_{J}(\ell_2^n)$.
\end{proof}

\smallskip

\section{Estimates}
Our next aim is to evaluate the  integral from Theorem~\ref{main} in more concrete situations. The first result is due to \cite{ryll1983homogeneous},
a~result which inspired everything done in this section. We as before use  the notation  $$c_k(n):=\frac{(n-1 + k)!}{(n-1)!\,k!}\,,
\quad k,n \in \mathbb{N}$$
from~\eqref{integral} .

\begin{corollary} \label{RW}
For $m,n \in \N$  with $n > 1$ we have
\[
\boldsymbol{\lambda}\big(\mathcal{P}_{m}(\ell_2^n)\big) \,\,=\,\, \frac{\Gamma(n+m) \Gamma(1 + \frac{m}{2})}{\Gamma(1 + m) \Gamma(n + \frac{m}{2})} \,\,\leq \,\,2^{n-1}\,.
\]
\end{corollary}

\begin{proof}
Using Theorem~\ref{main}, yields
\begin{align*}
\boldsymbol{\lambda}\big(\mathcal{P}_{m}(\ell_2^n)\big)  = (n-1) c_m(n) \int_0^1 (1 - t)^{n-2} t^{\frac{m}{2}}\,dt\,.
\end{align*}
Then by the basic properties of Beta and Gamma functions, we get
\begin{equation} \label{beta}
 (n-1) c_m(n) \int_0^1 (1 - t)^{n-2} t^{\frac{m}{2}}\,dt
= \frac{\Gamma(n+m) \Gamma(1 + \frac{m}{2})}{\Gamma(1 + m) \Gamma(n + \frac{m}{2})}\,,
\end{equation}
and the required formula follows.
\end{proof}

We go on  with further  consequences  of Theorem~\ref{main}, which provide lower and upper estimates of the projection
constant of $\mathcal{P}_{\leq m}(\ell_2^n)$, and more generally $\mathcal{P}_{J}(\ell_2^n)$, where
$J \subset \N_0^n,\, n >1 $ is a $\mathcal{U}_n$-invariant index set.

For the upper bound we need the following simple consequence of the well-known Fej\'er-Riesz inequality: Given
a~polynomial $P(z)= \sum_{k=0}^m a_k z^k, \, z \in \mathbb{C}$ with non-negative coefficients $a_k$, for all
$r\in [0, 1)$
\begin{align} \label{RF}
\frac{1}{2} \int_0^{2\pi} \Big|\sum_{k=0}^m a_k r^k e^{ik\theta}\Big|\,d\theta \geq \sum_{k=0}^m \frac{a_k}{k+1} r^k\,.
\end{align}
Indeed, recall the Fej\'er-Riesz
inequality \cite{fejer1921} (see also \cite[Theorem 3.13]{duren1970}): For every $f \in H_p(\mathbb{D})$ with
$0<p<\infty$
\[
\int_{-1}^1 |f(t)|^p\,dt \leq \frac{1}{2} \int_0^{2\pi} |f(e^{i\theta})|^p\,d\theta\,,
\]
where $f(e^{it})$ is the radial limit of $f$ which exists for almost all $t\in [0, 2\pi)$; moreover,
the constant $1/2$ is best possible.

In order to prove \eqref{RF}, fix  $r\in (0, 1)$, and define $f_r(z):= P(rz), \,z\in \mathbb{D}$. Then
$f_r\in H^{\infty}(\mathbb{D}) \subset H^1(\mathbb{D})$, and we obtain as desired
\begin{equation} \label{polyFR}
\frac{1}{2}\, \int_0^{2\pi} \Big|\sum_{k=0}^m a_k r^k e^{ik\theta}\Big|\,d\theta
\geq \int_0^1 \Big(\sum_{k=0}^m a_k r^k t^k\Big)\,dt = \sum_{k=0}^m \frac{a_k}{k+1} r^{k}\,.
\end{equation}

\begin{corollary}  \label{low}Let $J \subset \N_0^n,\, n >1 $ be a~$\mathcal{U}_n$-invariant index set and
$m= \max_{\alpha \in J}|\alpha|$. Then
\[
\boldsymbol{\lambda}\big(\mathcal{P}_{J}(\ell_2^n)\big) \,\geq \,
\sum_{k:J(k)\neq \emptyset}\frac{1}{k+1} \,\,\frac{\Gamma(n+k) \Gamma(1 + \frac{k}{2})}{\Gamma(1 + k) \Gamma(n  + \frac{k}{2})}\,.
\]
\end{corollary}

\begin{proof}
By Theorem~\ref{main} one has
\begin{align*}
\boldsymbol{\lambda}\big(\mathcal{P}_{J}(\ell_2^n)\big) & =
(n-1) \int_0^1 \Big(\int_0^{2\pi} \Big|\sum_{k:J(k)\neq \emptyset} c_k(n)\,r^k e^{ik\theta}\Big|\,d\theta\Big)(1- r^2)^{n-2}\,r\,dr \\
& \geq 2(n-1)\sum_{k:J(k)\neq \emptyset}\,\frac{c_k(n)}{k+1} \int_0^1 r^{k} (1 - r^2)^{n-2} \,r\,dr \\
& = \sum_{k:J(k)\neq \emptyset}\,\frac{1}{k+1} (n-1)c_k(n) \int_0^1 (1 - t)^{n-2}\,t^{\frac{k}{2}}\,dt\,.
\end{align*}
To finish we use formula \eqref{beta}.
\end{proof}

In this context it  seems interesting to recall the following  inequality due to  Konyagin \cite{konyagin2001translation}
and independently  McGehee, Pigno, and Smith in \cite{mcgehee1981}, which confirms  a~famous  conjecture of Littlewood:
There exists an absolute constant $C>0$ such that, for any collection $n_1, \ldots, n_m \in \mathbb{N}$ with
$n_1 < \ldots < n_m$, we have
\[
\int_0^{2\pi} \bigg |\sum_{k=1}^m e^{in_kt}\Big|\,dt  
\geq C\,\log m\,.
\]
More generally, in \cite{mcgehee1981} the following deep fact was proved, which may replace~\eqref{RF} in the proof of
Corollary~\ref{low}: There exists a~constant $C>0$ such that for any finite choice  $a_1, \ldots, a_m \in \mathbb{C}$ and
$n_1, \ldots, n_m \in \mathbb{N}$ with
$n_1 < \ldots < n_m$, we have
\[
\int_0^{2\pi} \bigg |\sum_{k=1}^m a_k e^{in_kt}\Big|\,dt \geq C \sum_{k=1}^m \frac{|a_k|}{k}\,.
\]

Finally, we turn to the upper bound for $\boldsymbol{\lambda}(\mathcal{P}_{J}(\ell_2^n))$, where $J \subset \N_0^n,\, n >1 $ is again
$\mathcal{U}_n$-invariant.

\begin{corollary}  \label{upper}Let $J \subset \N_0^n,\, n >1 $ be a $\mathcal{U}_n$-invariant index set and
$m= \max_{\alpha \in J}|\alpha|$. Then
\begin{align*}
\boldsymbol{\lambda}\big(\mathcal{P}_{J}(\ell_2^n)\big) \leq C\,\Big(c_m(n) \log(m+2) +  \, \sum_{\substack{0 \leq k <m\\ J(k) \neq \emptyset }}
|c_k(n) - c_{k+1}(n)| \log(k+2)\Big),
\end{align*}
where $C >0$ is an absolute constant. For the particular case $J = \Lambda^{\leq} (m,n)$, we obtain
\[
\boldsymbol{\lambda}\big(\mathcal{P}_{\leq m}(\ell_2^n)\big)
\leq C\,\bigg(\frac{(n-1 + m)!}{(n-1)!\,m!} \log(m+2) \, + \, \sum_{k=0}^{m-1} \frac{(n-1 + k)!}{(n-2)!\,(k+1)!} \log(k+2)\bigg)\,.
\]
\end{corollary}

\begin{proof}
According to \eqref{decompo}, we choose a decomposition $J = \bigcup_{\ell =1}^m \Lambda(k_\ell,n)$.
Then by Abel summation, for all $\theta \in [0, 2\pi)$ and for all $r\in (0, 1)$, we get
\begin{align*}
\sum_{\ell =0}^m c_{k_\ell}(n)\,r^{k_\ell} e^{ik_\ell\theta}
& =   c_m(n)\,\Big(\sum_{\ell =0}^m r^{k_\ell} e^{ik_\ell \theta}\Big) \,
+ \, \sum_{\ell =0}^{m-1} \big(c_{k_\ell}(n)- c_{k_{\ell+1}}(n)\big)\,\Big(\sum_{j=0}^{k_\ell} r^j e^{ij\theta}\Big)\,.
\end{align*}
Since for any holomorphic function $f$ on $\mathbb{D}$, the function $[0, 1)
\ni r \mapsto \int_0^{2\pi} |f(re^{i\theta})|\,d\theta$ is non-decreasing, we get
\begin{align*}
\int_0^{2\pi} & \Big| \sum_{\ell =0}^m c_{k_\ell}(n)\,r^{k_\ell} e^{ik_\ell\theta} \Big|\,d\theta  \leq
c_m(n) \int_0^{2\pi} \Big|\sum_{\ell =0}^m  e^{ik_\ell \theta}\Big|\,d\theta\\
& + \sum_{\ell =0}^{m-1} \big|c_{k_\ell}(n)- c_{k_{\ell+1}}(n)\big| \int_0^{2\pi} \Big|\sum_{j=0}^{k_\ell}
e^{ij\theta}\Big|\,d\theta \\
& \leq \gamma \Big(c_m(n) \log(m+2) + \sum_{\ell=0}^{m-1} \big|c_{k_\ell}(n)- c_{k_{\ell+1}}(n)\big| \log(k_\ell + 2)\Big)\,,
\end{align*}
where $\gamma>0$ is a~universal constant. This estimate combined with Theorem~\ref{main} and~\eqref{beta} yields
\begin{align*}
& \boldsymbol{\lambda}\big(\mathcal{P}_{J}(\ell_2^n)\big)  =
\frac{(n-1)}{\pi}\int_0^1 \Big(\int_0^{2\pi} \Big|\sum_{\ell =0}^m c_{k_\ell}(n)\,r^k e^{ik_\ell\theta}\Big|\,d\theta\Big) (1- r^2)^{n-2}\,r\,dr\\
& \leq 2 \gamma \frac{(n-1)}{\pi}\Big(c_m(n) \log(m+2) + \sum_{\ell =0}^{m-1} \big|c_{k_\ell}(n)- c_{k_{\ell+1}}(n)\big| \log(k_\ell + 2)\Big)
\int_0^1 (1-t)^{n-2}\,dt\\
& = \frac{2 \gamma}{\pi} \Big(c_m(n) \log(m+2) + \sum_{k=0}^{m-1} \big|c_{k_\ell}(n)- c_{k_\ell +1}(n)\big| \log(k_\ell + 2)\Big)\,,
\end{align*}
so the desired estimate follows.
\end{proof}

Finishing, we consider the  special case $n=2$.

\begin{corollary}\label{cor: 2 dim proy constant }
There exits a universal positive constant $\gamma$ such that one has
\[
1 + \log m \leq \boldsymbol{\lambda}\big(\mathcal{P}_{\leq m}(\mathbb{S}_2)\big) \leq \gamma m\,(1 + \log m), \quad\, m \in \mathbb{N}\,.
\]
\end{corollary}

\begin{proof}
Since $\Gamma(x + 1) = x \Gamma(x)$ for all $x>0$, Corollary \ref{low} shows that
\[
\boldsymbol{\lambda}\big(\mathcal{P}_{\leq m}(\mathbb{S}_2)\big) \geq \sum_{k=0}^m \frac{1}{k + 1} \geq 1 + \log m, \quad\, m\in \mathbb{N}\,.
\]
The upper estimate follows from Corollary \ref{upper}
\[
\boldsymbol{\lambda}\big(\mathcal{P}_{\leq m}(\mathbb{S}_2)\big) \prec \bigg (\log (m+2) \,+ \,\sum_{k=0}^{m-1} \log (k+2)\bigg)
\prec  m\,(1 + \log m).\,\,\,\,\qedhere
\]
\end{proof}

\chapter{Trace class operators}\label{Trace class operators}
As usual $\mathcal L(\ell_2^n)$ stands for the $C^\ast$-algebra of all linear operators $u\colon \ell_2^n \to \ell_2^n$ endowed with the operator norm. The space of all  linear operators $u\colon \ell_2^n \to \ell_2^n$ equipped with the Hilbert-Schmidt
norm $\|u\|_2 := \big(\sum_{i,j} |\langle ue_i,e_j\rangle|^2\big)^{\frac{1}{2}}$ is denoted by  $\mathcal{S}_2(n)$, and if we
consider the so-called trace class norm $\|u\|_1 := \||u|^{\frac{1}{2}}\|_2$\,, then we write $\mathcal{S}_1(n)$. Recall that
$\mathcal L(\ell_2^n)$ and $\mathcal{S}_1(n)$ are in trace duality, that is, the mapping
\begin{equation}\label{dualitytr}
\mathcal{S}_1(n) \to \mathcal L(\ell_2^n)^\ast\,,\,\,\,\,\,\,\,\,\,\,\,\, u \mapsto [v \mapsto \textrm{tr}(uv)]
\end{equation}
defines a~linear isometrical isomorphism between Banach spaces $\mathcal{S}_1(n)$ and $L(\ell_2^n)^\ast$.

By \cite[Theorem 5.6]{gordon1974absolutely} it is known that
\begin{equation}\label{asy1}
\boldsymbol{\lambda}\big(\mathcal L(\ell_2^n)\big)=\frac{\pi}{4}n
\,\,\,\,\,\,\,\,\,\,\,\text{and}\,\,\,\,\,\,\,\,\,\,\,
\frac{n}{3}\le\boldsymbol{\lambda}\big(\mathcal S_1(\ell_2^n)\big)\le n\,,
\end{equation}
whereas we may deduce from \eqref{grunbuschC-A} (see also Corollary~\ref{RW}) that
\[
\boldsymbol{\lambda}\big(\mathcal{S}_2(n)\big) \,\,=\,\, \frac{\sqrt{\pi}}{4}\frac{\Gamma(n^2+1) }{ \Gamma(n^2 + \frac{1}{2})}\,\,.
\]
In particular, we  obtain the limits
\begin{equation} \label{missing}
\lim_{n \to \infty} \frac{\boldsymbol{\lambda}\big(\mathcal{S}_2(n)\big)}{n} = \frac{\sqrt{\pi}}{2} \,\,\,\,\,\,\,\,
\text{and}\,\,\,\,\,\,\,\, \lim_{n \to \infty} \frac{\boldsymbol{\lambda}\big(\mathcal L(\ell_2^n)\big)}{n} = \frac{\pi}{4}\,.
\end{equation}

Our  aim in this section is to  show how the ideas of the previous sections can be applied to estimate the projection constant
$\boldsymbol{\lambda}\big(\mathcal{S}_1(n)\big)$ more precisely. The main contribution (from  Section~\ref{trace-uni}) states that
\begin{equation*}
\boldsymbol{\lambda}\big(\mathcal S_1(n)\big)=n\int_{\mathcal{U}_n} |\mathrm{tr}(V)|\,dU\,,
\end{equation*}
where the integration is taken with respect to the Haar measure $dU$ on the unitary group $\mathcal{U}_n$, and using probabilistic
arguments we then may deduce the missing limit from \eqref{missing}:
\begin{equation*}
\lim_n \frac{\boldsymbol{\lambda}\big(\mathcal S_1(n)\big)}{n}=\frac{\sqrt \pi}{2} \,.
\end{equation*}

How is all this related with our main leitmotif of studying projection constants of spaces of polynomials
on finite dimensional Banach spaces supported on certain index sets?

Note first that in view of \eqref{dualitytr} the trace class $\mathcal{S}_1(n)$ may be interpreted as the  space of all
one homogeneous polynomials on $\mathcal L(\ell_2^n)$, that is all  one homogeneous polynomials on  all $n \times n$ matrices.

The idea is to use Rudin's strategy of averaging projections (see Theorem~\ref{rudy}). So we need to define so-called
'unitary harmonics', which (roughly spoken) are harmonic  polynomials in finitely many 'matrix variables' $z$ and $\overline{z}$
from the unitary group $\mathcal{U}_n$. Everything we do is deeply inspired by the theory of spherical harmonics as presented in detail
in \cite{rudin1980} (and sketched in Section~\ref{Hpq-C^n}), that is, the study of harmonic polynomials in finitely many complex variables $z$ and $\overline{z}$ on
the $n$-dimensional sphere $\mathbb{S}_n$.

\smallskip

\section{Unitary harmonics and their density}

In this  section we start to extend parts of the setting of spherical  harmonics on the sphere $\mathbb{S}_n$ (as briefly sketched in  Section~\ref{Hpq-C^n}) to unitary harmonics on the unitary group $\mathcal{U}_n$.

Denote by $M_n(\mathbb{C})$  the space of all $n \times n$-matrices $z = (z_{ij})$
with entries from $\mathbb{C}$. The subset of such matrices $\alpha  = (\alpha_{ij})$ with entries from  $\mathbb{N}_0$ is
denoted by $M_n(\mathbb{N}_0)$. For  $z \in M_n(\mathbb{C})$ and $\alpha  = (\alpha_{ij}) \in  M_n(\mathbb{N}_0)$  we define
\[
z^\alpha = \prod_{i,j=1}^{n} z_{ij}^{\alpha_{ij}}\,.
\]
The  symbol $\mathfrak P (M_n(\mathbb{C}))$ stands for all  polynomials $f: M_n(\mathbb{C}) \to \mathbb{C}$ of the form
 \begin{align*}
f\left(z\right)=\sum_{(\alpha,\beta) \in J} c_{(\alpha,\beta)}(f)\,\,z^\alpha\overline{z}^\beta,
\end{align*}
where  $J$  is a finite subset of pairs
$(\alpha, \beta) \in M_n(\mathbb{N}_0) \times M_n(\mathbb{N}_0)$,
and
$(c_{(\alpha,\beta)}(f))_{(\alpha, \beta) \in J}$ are  complex coefficients. Observe that for every $(\alpha,\beta) \in J$
\[
c_{(\alpha,\beta)}(f) =  \frac{1}{\alpha! \beta!} \frac{\partial^kf}{\partial z^\alpha \partial\overline{z}^\beta}(0)\,,
\]
which in particular shows that the monomial coefficients $c_{(\alpha,\beta)}(f)$ of each $f \in \mathfrak P (M_n(\mathbb{C}))$
are unique. Note that the linear space $\mathfrak P (M_n(\mathbb{C}))$ carries a natural inner product which
for $f,g \in \mathfrak P (M_n(\mathbb{C}))$ is given by
\begin{equation}\label{inner}
\big\langle f,g\big\rangle_{\mathfrak P}:=\sum_{(\alpha,\beta)}\alpha!\,\beta!\,  c_{(\alpha,\beta)}(f)\overline{ c_{(\alpha,\beta)}(g)}\,.
\end{equation}
Every polynomial  $f \in \mathfrak P \big(M_n(\mathbb{C})\big) $ defines the differential operator $f(D)\colon \mathfrak P (M_n(\mathbb{C})) \to \mathfrak P(M_n(\mathbb{C}))$ given by 
\[
f\left(D\right):=\sum_{(\alpha,\beta)}  c_{(\alpha,\beta)}(f)\frac{\partial^k}{\partial z^\alpha \partial\overline{z}^\beta}\,,
\]
and then it is not difficult to see that
\begin{equation}\label{diff}
\big\langle f,g\big\rangle_{\mathfrak P}=\big[f(D)\overline g\big](0)\,.
\end{equation}
Given $k \in \mathbb{N}$, we denote by  
\[
\mathfrak P_k\big(M_n(\mathbb{C})\big)
\]
the subspace of all polynomials $f\in \mathfrak P (M_n(\mathbb{C}))$, which are supported on the index set of all
pairs $(\alpha, \beta)$ for which  $|\alpha|+|\beta|=k$, that is,  $c_{(\alpha,\beta)}(f) = 0$ whenever
$|\alpha|+|\beta|\neq k$. Clearly, all   $ f \in \mathfrak P_k (M_n(\mathbb{C}))$ satisfy that
$f(\lambda z) = \lambda^kf( z)$ for all $\lambda \in \mathbb{R}$ and $z \in M_n(\mathbb{C})$, which is the reason why  we call them  $k$-homogeneous. Moreover,
\begin{equation}\label{denso1}
\mathfrak P ( M_n(\mathbb{C}))=\text{span}\big\{\mathfrak P_k ( M_n(\mathbb{C})) \colon \,k\ge 0\big\}\,.
\end{equation}
The  polynomial $\mathbf t\in \mathfrak P_2\big(M_n(\mathbb{C})\big)$ given by
\[
{\mathbf t}(A):=\mathrm{tr}(A^*A), \quad\, 
A\in M_n(\mathbb{C})\,,
\]
where $\mathrm{tr}:M_n(\mathbb{C}) \to \mathbb{C}$ denotes the trace, is of special importance, since then
\[
\Delta : =\mathbf t (D) = \sum_{i,j}\frac{\partial^2}{\partial z_{ij}\partial\overline{z_{ij}}} :
\,\, \mathfrak P \big(M_n(\mathbb{C})\big) \to \mathfrak P \big(M_n(\mathbb{C})\big)\,.
\]
 is the Laplace operator.
A  polynomial $f \in  \mathfrak P \big(M_n(\mathbb{C})\big)$ is said to be harmonic, whenever  $\Delta f=0$, and we write
\[
\mathfrak{ H} \big( M_n(\mathbb{C})\big)
\]
for the subspace of all harmonic polynomials in $\mathfrak P (M_n(\mathbb{C}))$, and
$\mathfrak H_k ( M_n(\mathbb{C})\big)$
for all  $k$-homogeneous, harmonic polynomials. Obviously,
\begin{equation}\label{denso2}
\mathfrak H (M_n(\mathbb{C}))=\text{span}\big\{\mathfrak H_k (M_n(\mathbb{C})) \,:\, k\in \mathbb N\big\}\,.
\end{equation}

 The group $\mathcal{U}_n$ of all unitary $n\times n$ matrices $U = (u_{ij})_{1 \leq i,j \leq n}$ endowed with the topology induced by  $\mathcal{L}(\ell_2^n)$   forms a~non-abelian
compact group. It  is unimodular,
and we denote the integral of a function $f \in L_2(\mathcal{U}_n)$  mostly by $$\int_{\mathcal{U}_n} f(U) dU\,.$$
Integrals of this type form the so-called  Weingarten calculus,  which is  of outstanding importance in random matrix
theory, mathematical physics, and the theory of quantum information (see e,g. \cite{collins2006integration,kostenberger2021weingarten}).

Basically, we in this chapter only need   the precise values  for two concrete integrals
from the Weingarten calculus. The first one is
\begin{equation}\label{intformA}
\int_{\mathcal{U}_n} u_{i,j} \overline{u_{k,\ell}}dU = \frac{1}{n} \delta_{i,k}  \delta_{j,\ell}
\end{equation}
for all possible $1\leq i,j,k,\ell \leq n$, and the second one
\begin{equation}\label{intformB}
  \int_{\mathcal{U}_n} |\textrm{tr}(AU)|^2 dU = \frac{1}{n} \textrm{tr}(AA^\ast)
\end{equation}
for every $A \in  M_n(\mathbb{C})$ (see, e.g., \cite[p.~16]{cerezo2021cost}, \cite{collins2006integration}, or\cite[Corollary 3.6]{zhang2014matrix}).

By $\mathfrak P (\mathcal U_n)$ we denote the linear space of all restrictions  $f|_{\mathcal{U}_n}$ of polynomials
$f \in \mathfrak P( M_n(\mathbb{C}))$  to the unitary  group. All restrictions of harmonic polynomials on $M_n(\mathbb{C})$
to the unitary group  are denoted~by 
\[
\mathfrak H (\mathcal U_n)\,,
\]
and  such  polynomials we call unitary harmonics.
Similarly, we denote by $\mathfrak P_k (\mathcal U_n)$ and  $\mathfrak H_k(\mathcal U_n)$ the corresponding
spaces of $k$-homogeneous polynomials restricted to $\mathcal U_n$.

\begin{proposition} \label{harmlemA}
Given $k \in \mathbb{N}$,
\[
\mathfrak P_k (M_n(\mathbb{C}))=\mathfrak H_k (M_n(\mathbb{C}))\,\oplus\,\mathbf t\cdot \mathfrak H_{k-2} (M_n(\mathbb{C}))\,\oplus\,\mathbf\,t^2\cdot \mathfrak H_{k-4} (M_n(\mathbb{C}))\,\oplus\,\,\dots\,\,,
\]
where the last term of the sum is the span of $\mathbf t^{k/2}$ for even $k$ and $\mathbf t^{(k-1)/2}\cdot \mathfrak H_1(M_n(\mathbb{C}))   $ for odd $k$; the symbol $\oplus$ denotes the orthogonal sum with respect to the inner product $\big\langle \cdot, \cdot\big\rangle_{\mathfrak P}$ (defined in \eqref{inner}).
\end{proposition}
\begin{proof}
Given $g\in \mathfrak P_{k-2} (M_n(\mathbb{C}))$, let us for $A\in M_n(\mathbb{C})$ define
\[
h(A)=\mathbf t(A)g(A).
\]
Note that, since $\mathbf t(D)=\Delta,$  we  have $h(D)=\Delta \circ g(D)=g(D)\circ \Delta.$
Let $f\in\mathfrak P_k\big(M_n(\mathbb{C})\big)   $, then by \eqref{diff}
\[
\big\langle h,f\big\rangle_{\mathfrak{P}} =\big[h(D)\overline f\big](0) =\big[g(D)\big(\Delta \overline f\big)\big](0) = \big\langle g,\Delta f\big\rangle_{\mathfrak{P}}\,.
\]
Thus, $f\perp \mathbf t g$ for every $g\in \mathfrak P_{k-2}(M_n(\mathbb{C}))   $ if and only if $\Delta f\perp  g$ for every $g\in \mathfrak P_{k-2}(M_n(\mathbb{C}))$ if and only $\Delta f=0$, that is, $f \in \mathfrak H_k(M_n(\mathbb{C}))$, and consequently
\[
\mathfrak P_k(M_n(\mathbb{C}))    =\mathfrak H_k(M_n(\mathbb{C}))    \oplus \mathbf t\cdot\mathfrak P_{k-2}(M_n(\mathbb{C}))\,.
\]
The proof finishes repeating this procedure for $\mathfrak P_{k-2}(M_n(\mathbb{C})),\mathfrak P_{k-4}(M_n(\mathbb{C})), \ldots$
\end{proof}

\begin{corollary} \label{harmlemB}
For each $k,$
\begin{align}\label{eq:span of harmonic pols U_n}
\mathfrak P_k (\mathcal{U}_n)= \text{span}\,\big\{ \mathfrak H_\ell(\mathcal{U}_n) \colon  \ell \leq k \big\}\,.
\end{align}
Consequently, for every $f \in \mathfrak P(M_n(\mathbb{C}))$ there is $g \in \mathfrak H(M_n(\mathbb{C}))$ such that
both functions coincide on $\mathcal{U}_n$.
\end{corollary}
\begin{proof}
Since  $\mathbf t(U)=n$ for every $U\in \mathcal U_n$,
the assertion in \eqref{eq:span of harmonic pols U_n} is a consequence of Proposition~\ref{harmlemA}.
Then the last claim follows from \eqref{denso1}\,.
 \end{proof}

In Lemma~\ref{RudA} we recalled from  \cite{rudin1980} that all  spherical harmonics on $\mathbb{S}_n$ are dense in
$C(\mathbb{S}_n)$. For unitary harmonics the following density theorem  is an analog, which is crucial for our
coming purposes.

\begin{theorem}\label{teo: harmonic dense in C(U)}
$\mathfrak H (\mathcal U_n)$ is dense in $C(\mathcal U_n)$, and hence also in $L^2(\mathcal U_n).$
\end{theorem}

\begin{proof}
Note  first that $\mathfrak P (\mathcal U_n)$ is a~subalgebra of $C(\mathcal U_n)$ closed under conjugation, and that the collection of all coordinate functions $z_{ij}$ separates the points of $\mathcal U_n$. Thus, by the Stone-Weierstrass Theorem, $\mathfrak P (\mathcal U_n)$ is dense in $C (\mathcal U_n)$. Then the claim  is an immediate consequence of Corollary~\ref{harmlemB}.
\end{proof}

We need  to introduce another notation. For $p,q\in\mathbb N_0$ let
\[
\mathfrak H_{(p,q)}(\mathcal U_n)\subset \mathfrak{H}(\mathcal U_n)
\]
denote the space of all harmonic polynomials which are $p$-homogeneous in $z=(z_{ij})$ and $q$-homogeneous in $\overline{z}=(\overline{z_{ij}}).$ Clearly, we have
\begin{equation}\label{denso4}
\mathfrak H (\mathcal U_n)=\text{span}\,\big\{\mathfrak H_{(p,q)}(\mathcal U_n) \,:\, p,q\in\mathbb N_0\big\}\,.
\end{equation}

Recall that a polynomial  $f\colon \mathbb{S}_n \to \mathbb{C}$ of the form
\begin{align*}
f\left(z\right)=\sum_{\alpha \in J} c_{\alpha}(f)\,\,z^\alpha\overline{z}^\beta,
\end{align*}
where  $J \subset \mathbb{N}_o^n$  is any finite index set, is called a spherical harmonic, whenever  $\triangle f = 0$.
An important difference between spherical harmonics and unitary harmonics  is that for the case of the sphere the corresponding spaces $\mathfrak{H}_{(p,q)}(\mathbb{S}_n)$ are mutually orthogonal in $L_2(\mathbb{S}_n)$ (see \cite[Theorem 12.2.3]{rudin1980}). But for the subspaces $\mathfrak{H}_{(p,q)}(\mathcal{U}_n)$ of $L^2(\mathcal{U}_n)$ this is no longer true. To see an example take
$f\in \mathfrak{H}_{(1,0)}(\mathcal{U}_n)$ and
$g\in \mathfrak{H}_{(2,1)}(\mathcal{U}_n)$
defined by
$f(U)=u_{1,1}$ and $g(U)=\overline{u_{2,2}}u_{1,2}u_{2,1}$. Then
\begin{equation}\label{no-ortho}
  \big\langle f,g\big\rangle_{L_2}=\int_{\mathcal U_n}u_{1,1}u_{2,2}\overline{u_{1,2}u_{2,1}}dU=-\frac1{(n-1)n(n+1)}
\end{equation}
(see, e.g., \cite[Section 4.2]{hiai2000semicircle}).
On the other hand, using basic properties of the Haar measure  on $\mathcal U_n$, it is not difficult to prove that
\begin{equation*}
\text{ $\mathfrak{H}_{(p,q)}(\mathcal{U}_n)\perp \mathfrak{H}_{(p',q')}(\mathcal{U}_n)$ \,\,\,\, whenever \,\,\,\, $p+q=p'+q'$ }
\end{equation*}
(see \cite[\S 29]{hewitt2013abstract}, or \cite{kostenberger2021weingarten}).
Thus we have  the following interesting remark (not needed for our further purposes).
\begin{remark}\label{prop k-homog rudin 12.2.2}
$\mathfrak H_k(\mathcal{U}_n)= \mathfrak H_{(k,0)}(\mathcal{U}_n)\, \oplus\, \mathfrak H_{(k-1,1)}(\mathcal{U}_n)
\oplus\dots \, \oplus\, \mathfrak H_{(0,k)}(\mathcal{U}_n) $, where $\oplus$ indicates the orthogonal sum in $L_2(\mathcal U_n)$.
\end{remark}
Note also that in contrast to \eqref{no-ortho} we have  $\big\langle f,g\big\rangle_{\mathfrak P}= 0$ , so the Euclidean structure, which $\mathfrak{H}_{(p,q)}(\mathcal{U}_n)$ inherits from $L_2(\mathcal{U}_n)$, is different from that induced by the inner product from \eqref{inner}.

The next proposition  links the theory of homogeneous polynomials on the Banach space $\mathcal L(\ell_2^n)$ with
unitary harmonics.

\begin{proposition}\label{prop: pols in B(H) seen as functions on U(H)}
For any $k\ge1,$ $\mathfrak H_{(k,0)}(\mathcal U_n)$, as a subspace of $C(\mathcal U_n)$, is isometrically isomorphic to $\mathcal P_k(\mathcal L(\ell_2^n))$. More precisely, the restriction map
\[
\mathcal P_k(\mathcal L(\ell_2^n)) \to \mathfrak H_{(k,0)}(\mathcal U_n)\,, \,\, \,\, \,\, \,\,f \mapsto  f|_{\mathcal U_n}
\]
is a surjective isometry.
\end{proposition}

\begin{proof}
The restriction map is obviously well-defined, linear and surjective. To prove  that it is isometric we use a result
of Nelson \cite{nelson1961distinguished} (see also \cite[Theorem 1]{harris1997holomorphic}) showing that for any
complex-valued continuous function $f$ on the closed unit ball of ${\mathcal L(\ell_2^n)}$, which is   holomorphic
on the open unit ball $B_{\mathcal L(\ell_2^n)}$, we have
\begin{equation*}
    \sup_{T \in B_{\mathcal L(\ell_2^n)} }|f(T)|=\sup_{U\in\mathcal U_n}|f(U)| \,. \qedhere
\end{equation*}
\end{proof}

As a by-product, we obtain the following consequence.

\begin{corollary}
For any $k\ge1$
\[
\boldsymbol{\lambda}\big( \mathfrak H_{(k,0)}(\mathcal U_n)  \big) \,\,\,\sim_{C(k)} \,\,n^k\,,
\]
 where and $\mathfrak H_{(k,0)}(\mathcal U_n)$ is considered as a subspace of $C(\mathcal U_n)$.
\end{corollary}

\begin{proof}
As mentioned in \eqref{dualitytr} and  \eqref{asy1} we have
\[
\boldsymbol{\lambda}\big(\mathcal{L}(\ell_2^n)^\ast\big) = \boldsymbol{\lambda}\big (\mathcal{S}_1(n)\big) \sim_C  n\,.
\]
On the other hand, we know from Corollary~\ref{immediateABC} that
\[
\boldsymbol{\lambda}\big(\mathfrak H_{(k,0)}(\mathcal U_n)\big)
=\boldsymbol{\lambda}\big( \mathcal{P}_k(\mathcal{L}(\ell_2^n))\big)
\,\,\,\sim_{C(k)} \,\, \boldsymbol{\lambda}\big(\mathcal{L}(\ell_2^n)^\ast\big)^k\,,
\]
and this completes the proof.
\end{proof}

\smallskip

\section{Reproducing kernels}
For any  subspace $S \subset L_2(\mathcal U_n)$, we denote by $\pi_S$
the orthogonal projection on $L_2(\mathcal U_n)$ onto  $S$. Moreover, every operator  $T: M_n(\mathbb{C}^n) \to M_n(\mathbb{C}^n)$
that leaves $\mathcal U_n$ invariant  (i.e., $T\mathcal U_n\subset \mathcal U_n$), defines a so-called composition operator
\[
C_{T}: L_2(\mathcal U_n) \to L_2(\mathcal U_n)\,,\,\,\,\,\, f \mapsto f \circ T\,.
\]
There are in fact two such operators $T$, leaving $\mathcal U_n$ invariant,  of special interest -- the left and right multiplication
operators $L_V$ and $R_V$ with respect to  $V \in \mathcal U_n $, which for $U \in M_n(\mathbb{C}^n)$ are given by
\[
\text{$L_V (U) = VU$ \, \,\, \, and \, \,\, \, $R_V (U) = UV$ }.
\]
A subspace $S \subset L_2(\mathcal U_n)$ is said to be $\mathcal U_n$-invariant whenever it is invariant under all possible
composition operators $C_{L_V}$ and $C_{R_V}$ with $V \in \mathcal U_n $\,.

In the following we need a technical lemma,  which  mimics the well-known concept of 'reproducing kernels' from Hilbert spaces theory.

\begin{lemma}\label{theorem kernel 12.2.5}
Given a $\mathcal U_n$-invariant subspace $S$ of $ L_2(\mathcal U_n)$ with $\{0\}\ne S\subset C(\mathcal{U}_n)$ and $U\in\mathcal{U}_n$, there exists a unique function $K_U^S\in S\subset L_2(\mathcal U_n)$ such that
for every $f\in L_2(\mathcal{U}_n)$
\begin{itemize}
\item[(i)]\label{12.2.5 (1,3)}  $(\pi_Sf)(U)=\big\langle f, K_U^S\big\rangle_{L_2}
=\int_{\mathcal{U}_n}f(V)\overline{K_U^S(V)}dV=\int_{\mathcal{U}_n}f(V)K_V^S(U)dV \,,$
\end{itemize}
and moreover for every $V\in\mathcal U_n$
\begin{itemize}
\item[(ii)]
\label{12.2.5 (2)} $K_U^S(V)= \big\langle K_U^S, K_V^S\big\rangle_{L_2}=\overline{K_V^S(U)}$\,,
\item[{(iii)}]\label{12.2.5 (4)}
$
K_U^S\circ L_{V^{-1}}=K_{VU}^S=K_V^S\circ R_{U^{-1}}\,,
$
\item[(iv)]\label{12.2.5 (6)} $K_V^S(V)=K_{Id}^S(Id)>0\,.$
\end{itemize}
\end{lemma}

We note that if $S$ is an arbitrary $\mathcal U_n$-invariant  subspace of $ L_2(\mathcal U_n)$, then the equations in the above
items are still true, but they must be stated almost everywhere.

\begin{proof}
The claim from $(i)$ is an immediate consequence of the Riesz representation theorem applied to the continuous linear functional
$
L_2(\mathcal U_n)  \to  \mathbb{C}\,,\,\,\, f \mapsto (\pi_Sf)(U)\,.
$

\noindent $(ii)$ $K_U^S(V)=\pi_S(K_U^S)(V)=\big\langle K_U^S, K_V^S\big\rangle_{L_2}=\overline{\big\langle K_V^S, K_U^S\big\rangle_{L_2}}=\overline{K_V^S(U)}$
for all $V\in\mathcal U_n$.

\noindent $(iii)$
Fix some $ V\in \mathcal U_n$ and
$f\in L_2(\mathcal U_n)$, and note
first that  $S^\perp$ is also $\mathcal U_n$-invariant. Then
\[
\text{$(Id-\pi_S)(f)\circ L_V \in S^\perp$
\,\,\,\, and \,\,\,\,$f\circ L_V=\pi_S(f)\circ L_V+(Id-\pi_S)(f)\circ L_V$\,,}
\]
and hence
\begin{align}\label{eq: pi_S commutes with L_V}
\pi_S(f\circ L_V)=\pi_S(\pi_S(f)\circ L_V)+\pi_S((Id-\pi_S)(f)\circ L_V)=\pi_S(f)\circ L_V\,.
\end{align}
Then
\[
\big\langle f, K_{VU}^S\big\rangle_{L_2}=\pi_S(f)(VU)=\pi_S(f)\circ L_V(U)=\pi_s(f\circ L_V)(U)\,,
\]
and thus  by $(i)$
\[
\big\langle f, K_{VU}^S\big\rangle_{L_2}= \big\langle f\circ L_V, K_{U}^S\big\rangle_{L_2} = \big\langle C_{L_V}f , K_{U}^S\big\rangle_{L_2}
= \big\langle  f , C_{L_{V^{-1}}} K_{U}^S\big\rangle_{L_2} = \big\langle  f ,  K_{U}^S \circ L_{V^{-1}}\big\rangle_{L_2}\,.
\]
Since $f\in L_2(\mathcal U_n)$ was chosen arbitrarily, we obtain that  $K_{VU}^S=K_{U}^S \circ L_{V^{-1}}.$ The other identity follows similarly.

\noindent
$(iv)$ Let $V\in\mathcal U_n, $ then
\begin{eqnarray*}
K_V^S(V)=\big\langle K_V^S, K_V^S\big\rangle_{L_2}=\big\langle K_{Id}^S\circ L_{V^{-1}}, K_V^S\big\rangle_{L_2}=\big\langle K_{Id}^S, K_V^S\circ L_V\big\rangle_{L_2}=\big\langle K_{Id}^S, K_{Id}^S\big\rangle_{L_2}=K_{Id}^S(Id)>0\,,
\end{eqnarray*}
this concludes the proof.
\end{proof}
\smallskip

\begin{remark}
For each $(p,q)$,  the function
$
f=\|K_{Id}^{\mathfrak H_{(p,q)}}\|_{L^2}^{-2}\,\,K_{Id}^{\mathfrak H_{(p,q)}}  \in \mathfrak H_{(p,q)}(\mathcal U_n)
$
satisfies the properties
\begin{equation}\label{twoconditions}
  \text{$f(Id)=1$ \,\,\,\,\,\,and \,\,\,\,\,\,$ (C_{L_{V^{-1}}}\circ C_{R_V}) f =f$ \quad for every $V\in\mathcal U_n$\,.}
\end{equation}
Indeed, by Theorem \ref{theorem kernel 12.2.5}, $f(Id)=\big\|K_{Id}^{\mathfrak H_{(p,q)}}\big\|_{L^2}^{-2} \,\big\langle K_{Id}^{\mathfrak H_{(p,q)}},K_{Id}^{\mathfrak H_{(p,q)}}\big\rangle_{L_2}=1.$ Also,
\begin{align*}
(C_{L_{V^{-1}}}\circ C_{R_V})f & =\big \|K_{Id}^{\mathfrak H_{(p,q)}}\big\|_{L^2}^{-2} \, C_{L_{V^{-1}}}(K_{Id}^{\mathfrak H_{(p,q)}}\circ {R_V})  \\
& = \big\|K_{Id}^{\mathfrak H_{(p,q)}}\big\|_{L^2}^{-2} \, C_{L_{V^{-1}}}(K_{V^{-1}}^{\mathfrak H_{(p,q)}}) = \big\|K_{Id}^{\mathfrak H_{(p,q)}}\big\|_{L^2}^{-2} \, (K_{V^{-1}}^{\mathfrak H_{(p,q)}}\circ L_{V^{-1}})
=f\,.
\end{align*}
\end{remark}

\smallskip

\section{Minimal projection onto the trace class}
We now  turn our attention to the projection constant of the  trace class $\mathcal S_1(n)$, and start with the following result, which obviously is a special case of Proposition~\ref{prop: pols in B(H) seen as functions on U(H)}.

\begin{proposition}\label{linco}
 $\mathfrak H_{(1,0)}(\mathcal U_n)$, as  a subspace of $C(\mathcal{U}_n)$, is isometrically isomorphic to $\mathcal S_1(n)$. More precisely,
 \[
 \mathcal S_1(n) \to \mathfrak H_{(1,0)}(\mathcal U_n)\,, \,\,\,\,\, A \mapsto [U \mapsto \mathrm{tr}(AU)]
 \]
 is an isometry onto.
\end{proposition}

We intend to
follow Rudin's strategy from Theorem~\ref{rudy} of averaging projections, in order to prove that the restriction
\begin{equation*}
  \pi_{1,0}|_{C(\mathcal{U}_n)}: C(\mathcal{U}_n)  \to \mathfrak H_{(1,0)}(\mathcal U_n)=\mathcal S_1(n)
\end{equation*}
actually is a minimal projection. To put all of this into more concrete terms we  describe (although this is not needed later) the orthogonal projection
$\pi_{1,0}: L_2(\mathcal{U}_n)  \to L_2(\mathcal{U}_n)$ onto $\mathfrak H_{(1,0)}(\mathcal U_n)=\mathcal S_1(n) $ in terms of an orthonormal system.
For every  choice of $1 \leq i,j \leq n$  define the  functions
\[
e_{ij} \in \mathfrak H_{(1,0)}(\mathcal U_n)\,,\,\,\,\,\,\, e_{ij}(U) = u_{i,j}\,.
\]
Then, by \eqref{intformA},  the collection  of all  normalized functions $\sqrt{n}\,e_{ij}, \, 1 \leq i,j \leq n$ forms an orthonormal
system in $L_2(\mathcal{U}_n)$, and hence an orthonormal basis of  $\mathfrak H_{(1,0)}(\mathcal U_n)=\mathcal S_1(n)$
considered as a subspace of $L_2(\mathcal U_n)$.
Consequently,   for each $f\in L_2(\mathcal U_n)$
\[
\pi_{1,0}(f)=\sum_{1 \leq i,j \leq n}\big\langle f, \sqrt n e_{ij} \big\rangle_{L_2}\, \sqrt n e_{ij}= n \sum_{1 \leq i,j \leq n}\big\langle f,  e_{ij} \big\rangle_{L_2}\,\, e_{ij}\,.
\]

The following result is the main contribution of this section (recall the much stronger result for
spherical harmonics from Theorem~\ref{polleqm}).

\begin{theorem}
\label{mainU}
For each $n \in \mathbb{N}$,
\[
\boldsymbol{\lambda}\big(\mathcal S_1(n), C(\mathcal U_n)\big) = \big\|\pi_{1,0}|_{C(\mathcal{U}_n)}:C(\mathcal U_n) \to \mathcal S_1(n)\big\| \,,
\]
that is, if $\textbf{Q}$ is any  projection on $C(\mathcal U_n)$ onto $\mathcal S_1(n)$, then $\|\textbf{Q}\| \geq
\|\pi_{1,0}\|$.
\end{theorem}

We prepare the proof of Theorem~\ref{mainU} with three independently interesting lemmas.

\begin{lemma}\label{prop 12.2.6}
The collection of all functions  $f \in \mathfrak H_{(1,0)}(\mathcal U_n)$, such that $ (C_{L_{V^{-1}}}\circ C_{R_V}) f =f$ for all
$V\in\mathcal{U}_n$, coincides with the set of all  scalar  multiples of the trace functional $\,\mathrm{tr}$. In particular,
if $f(Id)=1$ and $ (C_{L_{V^{-1}}}\circ C_{R_V}) f =f$ for all $V\in\mathcal U_n$, then $f=\frac{1}{n} \mathrm{tr}$.
\end{lemma}

\begin{proof}
 Obviously, every multiple of the trace  is invariant under all operators  $C_{L_{V^{-1}}}\circ C_{R_V}$. Assume conversely, that
$f \in \mathfrak H_{(1,0)}(\mathcal U_n)$ has this property. By Proposition~\ref{linco}  $f$ may be seen as a linear functional
on $M_n(\mathbb{C})$. This implies that there exists $A\in M_n(\mathbb{C})$ such that $f(U)=\mathrm{tr}(A^*U)$ for all $U\in M_n(\mathbb{C})$.
Since by assumption $f(V^{-1}UV)=f(U)$ for every $V\in\mathcal U_n,$  and $U\in M_n(\mathbb{C})$, we conclude that
\[
\mathrm{tr}(A^*U)=\mathrm{tr}(A^*V^{-1}UV)=\mathrm{tr}(VA^*V^{-1}U).
\]
Consequently, $A^*=VA^*V^{-1}$ for every $V\in\mathcal U_n$, implying that $A$ commutes with all matrices in $M_n(\mathbb{C})$
(recall that every matrix in $M_n(\mathbb{C})$ is a finite sum of unitaries). This  implies that $A=\lambda Id$ for some
$\lambda \in \mathbb C$, and hence  as desired   $f=\lambda\,\mathrm{tr}$.  The second claim of the lemma is an immediate consequence.
\end{proof}

In order to be able to apply  Rudin's Theorem~\ref{rudy}, we need to  endow  $\mathcal U_n \times \mathcal U_n$
with a special group structure, which allows to represent the resulting group in   $\mathcal{L}(C(\mathcal{U}_n))$.
To do so, consider on $\mathcal U_n \times \mathcal U_n$  the  multiplication
\[
(U_0,V_0)\cdot (U_1,V_1):=(U_1U_0,V_0V_1)\,.
\]
With this multiplication and endowed with the product topology, $\mathcal U_n \times \mathcal U_n$ turns into a compact topological group,
and it may be seen easily  that the Haar measure on $\mathcal U_n \times \mathcal U_n$ is given by the product measure of the Haar measure on $\mathcal U_n $
with itself.

Further, for  any $(U,V) \in \mathcal U_n \times \mathcal U_n $ and  any $f \in L_2(\mathcal U_n)$ we define
\begin{align*}\label{action of U^2}
    \rho_{(U,V)}f : = (C_{L_U}\circ C_{R_V})f = f\circ L_U\circ R_V\,,
\end{align*}
which leads to an action of $\mathcal U_n \times \mathcal U_n$ on $C(\mathcal{U}_n)$
given by
\begin{equation} \label{actU}
  \mathcal U_n \times \mathcal U_n  \to \mathcal{L}\big(C(\mathcal{U}_n)\big)\,,\,\,\, \,\,\,
  (U,V) \mapsto \big[\rho_{(U,V)}: f \mapsto f\circ L_U\circ R_V\big] \,.
\end{equation}
We say that a mapping  $T: S_1 \to S_2 $, where
$S_1$ and $S_2$ both  are $\mathcal{U}_n$-invariant subspaces of $L_2(\mathcal{U}_n)$, commutes with the action of
$\mathcal{U}_n \times \mathcal{U}_n$
on $C(\mathcal{U}_n)$, whenever
\[
\text{$\big(C_{L_U}\circ C_{R_{V}}\big)(Tf)=T\big(( C_{L_U}\circ C_{R_V})f\big)$
\quad for every \quad $(U,V)\in \mathcal U_n \times \mathcal U_n$ and $f\in S_1$\,.}
\]

\begin{lemma}\label{prop 12.2.7 for L2(U)}
Let $S\subset L_2(\mathcal U_n)$ be a $\mathcal U_n$-invariant subspace such that  $S\cap \mathfrak H_{(1,0)}(\mathcal U_n)=\{0\}$,
and let $T~:~S\to \mathfrak H_{(1,0)}(\mathcal U_n)$  be an operator that commutes with the action of $\mathcal U_n \times \mathcal U_n$
on $C(\mathcal{U}_n)$.
Then $T$ is a scalar multiple of   $\pi_{1,0}\restrict{S}$.

Moreover, if $S$ is  orthogonal to $\mathfrak H_{(1,0)}(\mathcal U_n)$ and $T$ is  a projection on $S+\mathfrak H_{(1,0)}(\mathcal U_n)$ onto $\mathfrak H_{(1,0)}(\mathcal U_n)$ that commutes with the action of $\mathcal U_n \times \mathcal U_n$
on $C(\mathcal{U}_n)$,  then $T=\pi_{1,0}\restrict{S+\mathfrak H_{(1,0)}}.$
\end{lemma}

\begin{proof}
We write $H = \mathfrak H_{(1,0)}(\mathcal U_n)$.
 Given $U\in\mathcal U_n$, let us denote by $K_U^H$ the kernel in $H$ and by $K_U^S$  the kernel in $S$. By the assumption on $T$
 and Lemma~\ref{theorem kernel 12.2.5}, (iii) for every $V\in\mathcal U_n$,
\[
\big(C_{L_V}\circ C_{R_{V^{-1}}}\big) (T K_{Id}^S)=T\big( (C_{L_V}\circ C_{R_{V^{-1}}})K_{Id}^S\big) =T( K_{Id}^S)\,.
\]
Then we deduce from Lemma~\ref{prop 12.2.6}  that  $T( K_{Id}^S)=c K_{Id}^H$ for some constant $c \in \mathbb{C}$.
 On the other hand, for every  $h\in S$ and  $V\in \mathcal U_n$  by Lemma~\ref{theorem kernel 12.2.5}, (i)
\[
h(V)=(\pi_Sh)(V)=\int_{\mathcal U_n}h(U)\, \overline { K_{V}^S(U)}dU=\int_{\mathcal U_n}h(U)\,K_{U}^S(V)dU.
\]
Applying  $T$ to this equality, by Lemma~\ref{theorem kernel 12.2.5}, (iii)
and another use of the  assumption on  $T$, for every  $h\in S$ and  $V\in \mathcal U_n$
\begin{align*}
Th(V)  &=\int_{\mathcal U_n}h(U)\, T(K_{U}^S)(V)dU =\int_{\mathcal U_n}h(U)\, T(K_{Id}^S\circ L_{U^{-1}})(V)\,dU \\
&=\int_{\mathcal U_n}h(U)\,T( K_{Id}^S)\circ L_{U^{-1}}(V)dU =c \int_{\mathcal U_n}h(U)\,K_{Id}^H\circ L_{U^{-1}}(V)dU\\
&=c \int_{\mathcal U_n}h(U)\, {K_{U}^H(V)}\,dU=c\,\pi_{1,0}(h)(V)\,,
\end{align*}
which is our first claim. To see the second assertion, note first that by  the first part of the lemma  we have $T\restrict{S}=c\,\pi_{1,0}\restrict {S}$ for some
$c \in \mathbb{C}$. But since by assumption $S\subset H^\perp$, this implies $T\restrict{S}= 0=\pi_{1,0}\restrict {S}$.
 On the other hand,  since $T$ is a projection onto $H$, we see that  $T\restrict{H}=Id_{H}=\pi_{1,0}\restrict {H}$, which finishes
 the proof.
\end{proof}

The third lemma we provide, is in fact the crucial step towards the proof of Theorem~\ref{mainU}
(for spherical harmonics we refer to Proposition~\ref{unique}).

\begin{lemma}\label{prop 12.2.7 for C(U)}
If $T$ is a projection on $C(\mathcal U_n)$ onto $\mathfrak H_{(1,0)}(\mathcal U_n)$ that commutes with the action of
$\mathcal U_n \times \mathcal U_n $
on $C(\mathcal U_n)$,  then $T=\pi_{1,0}\restrict{C(\mathcal U_n)}.$
\end{lemma}

\begin{proof}
By  Theorem~\ref{teo: harmonic dense in C(U)}
(the density theorem), it suffices to show that for each
$(p,q) \in \mathbb{N}_0 \times \mathbb{N}_0 $
\[
T\restrict{\mathfrak H_{(p,q)}} = (\pi_{1,0})\restrict{\mathfrak H_{(p,q)}}\,.
\]
 Given  such pair $(p,q)$, we define the subspace
\[
 S := \big\{f-\pi_{1,0}f\,:\,\,f\in \mathfrak H_{(p,q)}\big\} \subset C(\mathcal U_n).
\]
Then $S$ is  $\mathcal U_n$-invariant; indeed, for each   $f\in{\mathfrak H_{(p,q)}}$
and $U \in \mathcal U_n$,  equation \eqref{eq: pi_S commutes with L_V} implies
\[
(f-\pi_{1,0}f)\circ L_U=f\circ L_U -\pi_{1,0}f\circ L_U= f\circ L_U -\pi_{1,0}(f\circ L_U)\in S\,,
\]
and the  invariance under right multiplication follows similarly.
 Since $S\perp \mathfrak H_{(1,0)}$ and $T$ commutes with the action of $\mathcal U_n \times \mathcal U_n$
 on $C(\mathcal U_n)$,
 Lemma~\ref{prop 12.2.7 for L2(U)} (the second part applied to the restriction  of $T$ to $S + \mathfrak H_{(1,0)}$) shows  that
 $$T\restrict{S + \mathfrak H_{(1,0)}}=  \pi_{1,0}\restrict{S + \mathfrak H_{(1,0)}}\,,$$
  so  in particular
 $
     T\restrict{S}=  \pi_{1,0}\restrict{S} = 0\,.
     $
But then for every $f \in \mathfrak H_{(p,q)}$
\begin{equation*}
   T(f) = T(f-\pi_{1,0}f)+ T(\pi_{1,0}f) =   \pi_{1,0}f\,,
  \end{equation*}
  which completes the argument.
  \end{proof}

Finally, we are ready  to prove the main goal of this section,  Theorem~\ref{mainU}.

\begin{proof}[Proof of Theorem~\ref{mainU}]
Let $\textbf{P}$ be any projection on $C(\mathcal{U}_n)$ onto
$\mathfrak H_{(1,0)}(\mathcal U_n)$. We apply Theorem~\ref{rudy}, with $X=C(\mathcal{U}_n),$ $Y=\mathfrak H_{(1,0)}(\mathcal U_n)$, $G=\mathcal{U}_n \times \mathcal{U}_n$, and  the action   $\rho: G \to \mathcal{L}(X)$ as defined in \eqref{actU}.
Then we have that
\begin{equation*}
\textbf{Q}f = \int_{\mathcal{U}_n \times \mathcal{U}_n} \rho_{(U,V)}\, \textbf{P} \rho_{(U^{-1},V^{-1})}(f)\,dU\,dV, \quad\, f\in C(\mathcal{U}_n)
\end{equation*}
defines  a~projection on $C(\mathcal{U}_n)$ onto $\mathfrak H_{(1,0)}(\mathcal U_n)$,
which commutes with  the action of $\mathcal{U}_n  \times \mathcal{U}_n $
on $C(\mathcal{U}_n)$. Consequently,  we obtain from  Lemma~\ref{prop 12.2.7 for C(U)} that $\textbf{Q} =\pi_{1,0}|_{C(\mathcal{U}_n)}$. Finally, since the operators
$\rho_{(U,V)}$ define isometries on $C(\mathcal U_n)$,  we immediately see that
$\|\pi_{1,0}|_{C(\mathcal{U}_n)}\| = \|\textbf{Q}\| \leq  \|\textbf{P}\|$ (see again Theorem~\ref{rudy}).
Since by Proposition~\ref{linco}, $\mathfrak H_{(1,0)}(\mathcal U_n)=\mathcal S_1(n),$ the proof is finished.
\end{proof}

\smallskip

\section{Integral formula } \label{trace-uni}

We now provide an integral formula for the projection constant of the trace class.

\smallskip

\begin{theorem} \label{traceform}
For every $n,$
\[
\boldsymbol{\lambda}\big(\mathcal S_1(n)\big)=n\int_{\mathcal{U}_n} |\mathrm{tr}(V)|\,dV\,.
\]
\end{theorem}
\begin{proof}
Again we abbreviate $H = \mathfrak H_{(1,0)}(\mathcal U_n)$.
Given $U\in\mathcal U_n$, let us denote by $K_U^H$ the kernel in $\mathfrak H_{(1,0)}(\mathcal U_n)$. Let us first see that
\begin{align}\label{K_Id formula}
K_{Id}^H=n\cdot  \mathrm{tr}.
\end{align}
Indeed, since
$C_{L_{V^{-1}}}\circ C_{R_V} (K_{Id}^H)=K_{Id}^H$, for every $V\in\mathcal U_n$, we know from Lemma~\ref{prop 12.2.6} that $K_{Id}^H=c\,\mathrm{tr}$ for some constant $c \in \mathbb{C}$. Thus,
\begin{align*}
n&=\mathrm{tr}(Id)=\pi_{1,0}(\mathrm {tr})(Id)  = \big\langle \mathrm {tr},K_{Id}^H\big\rangle_{L_2} = \overline c\int_{\mathcal{U}_n} \mathrm{tr}(V)
\, \overline{\mathrm{tr}(V)}\,dV =\overline{c}\,,
\end{align*}
where the last equality  is a consequence of \eqref{intformB} (put $A= Id$). This proves \eqref{K_Id formula}.
Then, for $f\in L_2(\mathcal U_n)$, by Lemma~\ref{theorem kernel 12.2.5}, (i) and (iii)
\begin{align*}
\pi_{1,0}(f)(U) & = \big\langle f,K_U^H\big\rangle_{L_2} =\big\langle f,K_{Id}^H\circ L_{U^{-1}}\big\rangle_{L_2} \\
& = \int_{\mathcal{U}_n} f(V)\,  \overline{K_{Id}^H(U^{-1}V)}\,dV  =n\int_{\mathcal{U}_n} f(V)\,  \mathrm{tr}(V^*U)\,dV\,.
\end{align*}
As a consequence we by  Theorem~\ref{mainU} obtain
\[
\boldsymbol{\lambda}(\mathcal U_n)=\|\pi_{1,0}\colon L_2(\mathcal{U}_n) \to \mathfrak{H}_{(1.0)}(\mathcal{U}_n)\|\,.
\]
On the other hand, since
\[
\pi_{1,0}(f)(U) =n\int_{\mathcal{U}_n} f(V)\,  \mathrm{tr}(V^*U)\,dV\,,
\]
taking $f_U(V)=\sign(\mathrm{tr}(U^*V))$,  it follows that
\[
\|\pi_{(1,0)}\|=n\sup_{U}\int_{\mathcal{U}_n} |\mathrm{tr}(V^*U)|\,dV\,.
\]
The proof ends by observing that the invariance of the Haar measure implies that for each $U\in\mathcal U_n,$
\begin{equation*}
  \int_{\mathcal{U}_n} |\mathrm{tr}(V^*U)|\,dV=
\int_{\mathcal{U}_n} |\mathrm{tr}(V^*UU^*)|\,dV=\int_{\mathcal{U}_n} |\mathrm{tr}(V)|\,dV.\qedhere
\end{equation*}
\end{proof}

\smallskip

The following result show the precise asymptotic order of $\boldsymbol{\lambda}(\mathcal S_1(n))$ as $n$ goes to infinity.

\begin{theorem}\label{cor: lim proj constant trace}
\[
\lim_{n\to \infty} \, \frac{\boldsymbol{\lambda}\big(\mathcal S_1(n)\big)}{n}=\frac{\sqrt \pi}{2}\,.
\]
\end{theorem}
For the proof of the Theorem \ref{cor: lim proj constant trace}, we  need to recall  some well-known results from probability theory, for more on this see \cite{billingsley2013convergence}. We
are going to  use that, given any sequence of random variables $(Y_n)$ which converges in distribution to another random variable $Y$ (denoted by $Y_n \overset{D}{\longrightarrow} Y$) and any continuous real valued function $f$, we have $f(Y_n) \overset{D}{\longrightarrow} f(Y)$.
Recall also that a sequence of random variables $(Y_n)_{n }$ is called uniformly integrable whenever $$\lim_{a \to  \infty} \sup_{n \ge 1} \int_{|Y_n| \ge a } |Y_n| d P = 0\,.$$
\begin{remark}\label{rem: bounded moment implies unif int}
Notice that whenever $\sup_{n} \mathbb{E}( |Y_n|^{1+\varepsilon}) \le C$ for some $\varepsilon, C>0$, we have that $(Y_n)_{n}$ is uniformly integrable. Indeed,
\[
\lim_{a \to  \infty} \sup_{n \ge 1} \int_{|Y_n| \ge a } |Y_n| d P \le \lim_{a \to  \infty} \frac{1}{a^\varepsilon} C\,.
\]
\end{remark}
Uniform integrability will be useful for us due to the fact (see for example \cite[Theorem 3.5]{billingsley2013convergence})
that  if $(Y_n)_{n}$ is a uniformly integrable sequence of random variables and $Y_n \overset{D}{\longrightarrow} Y$, then $Y$ is integrable and \begin{align}\label{thm: unif int + conv in dist implies conv in mean}
\mathbb{E} (Y_n) \to \mathbb{E}(Y)\,.
\end{align}
It is known that if $(U(n))_n$ is  a sequence  of random unitary matrices that are uniformly Haar distributed, then $(\mathrm{tr}(U(n)))_n$ converges in distribution to the standard Gaussian complex random variable $\gamma$ (that is the vectors $(\sqrt{2}Re[\mathrm{tr}(U(n))]$ and $\sqrt{2}Im[\mathrm{tr}(U(n))])$ converge in distribution to a standard Gaussian random vector), see \cite[Corollary 2.4]{johansson1997random}, \cite{diaconis1994eigenvalues} or \cite[Problem 8.5.5]{pastur2011eigenvalue}. Moreover,  \cite[Theorem 2.6]{johansson1997random}  estimates the speed of this convergence (see also \cite[p.164]{diaconis2003patterns},\cite[Section 4]{meckes2008linear}).

\begin{proof}[Proof of Theorem \ref{cor: lim proj constant trace}]
By the above comments, the sequence $(\sqrt{2}|\,\mathrm{tr}(U(n))|)_n$ of
Random variables on $\mathcal{U}_n$ converges in distribution to a Rayleigh random variable.
Moreover, since it is known that
\[
\mathbb{E}\big(|\mathrm{tr}(U(n))|^2\big)=
\int_{\mathcal U_n} |\mathrm{tr}(V)|^2\,dV  =  1
\]
(see, e.g., \cite[Corollary 3.6]{zhang2014matrix}),  the random variables $|\mathrm{tr}(U(n))|$ are uniformly integrable, and hence
$\big(\mathbb{E}(\sqrt{2}| \,\mathrm{tr}(U(n))|)\big)_n$ converges to the expectation of a Rayleigh random variable (see Remark \ref{rem: bounded moment
implies unif int} and \eqref{thm: unif int + conv in dist implies conv in mean}). Thus 
\[
\lim_{n\to \infty}\, \mathbb{E}\big(\sqrt{2}\,|\mathrm{tr}(U(n))|\big)\to\sqrt{\frac{\pi}{2}}
\quad
\]
yields
\[
\lim_{n\to \infty} \frac{1}{n}\boldsymbol{\lambda}(\mathcal S_1(n)))=
\frac{1}{\sqrt 2}\,\lim_{n\to \infty} \mathbb E\big(\sqrt{2}\,|\mathrm{tr}(U(n)\big)|) = \frac{\sqrt{\pi}}{2}\,.
\]
This completes the proof.
\end{proof}

\bigskip

\chapter{Tensor product methods} \label{tensor-pro-me}

In the Chapters \ref{trigo}, \ref{eukli} and \ref{Trace class operators} we studied the projection constant $\boldsymbol{\lambda}\big(\mathcal{P}_{J}(X_n)\big)$ of spaces of polynomials   supported on some fixed index set~$J$, which allow an additional structure. 
They are all, in a natural way, embedded in some space $C(K)$ of continuous functions on a compact set $K$, for which there is a compact  group $G$ with elements  acting as operators on  $C(K)$ (in short, $G$ acts on $C(K)$).

In this way we considered spaces of
trigonometric polynomials on compact abelian groups $G$ as subspaces of $C(G)$ with the action induced by $G$
itself (in particular for the $n$-dimensional torus $G=\mathbb{T}^n$),  polynomials on finite dimensional Hilbert spaces $\ell_2^n$ as subspaces of $C(\mathbb{S}_n)$
with the action coming from the unitary group $\mathcal{U}_n$, and (certain) spaces of polynomials
on $\mathcal{L}(\ell_2^n)$  as subspaces of $C(\mathcal{U}_n)$, where the action is induced by  the product 
$\mathcal{U}_n\times \mathcal{U}_n $.

In several cases this led to concrete integral formulas for $\boldsymbol{\lambda}\big(\mathcal{P}_{J}(X_n)\big)$, and given an index set of degree $\leq m$ we in turn got quite accurate estimates for the asymptotic increase of these constants in terms
of the dimension $n$ and the degree~$m$. 

In this chapter our focus  is different. Given an arbitrary Banach space $X_n=(\mathbb{C}^n,\|\cdot\|)$, we study 
the projection constant $\boldsymbol{\lambda}\big(\mathcal{P}_{J}(X_n)\big)$ from the point of view of local Banach space theory.

Fixing the degree $m$, we show that for a large variety of index sets $J\subset  \mathbb{N}_0^n $ of degree at most $m$ and for a rather 
large class of Banach spaces $X_n = (\mathbb C^n, \|\cdot\|)$, the projection constant $\boldsymbol{\lambda}\big(\mathcal{P}_{J}(X_n)\big)$, equals the $m$th power of
$\boldsymbol{\lambda}(X_n^\ast)$,
up to a constant only depending on the degree  $m$ and not on the dimension $n$. On this way, we collect independently interesting information on $\boldsymbol{\lambda}\big(\mathcal{P}_{J}(X_n)\big)$. The main proofs are based on the theory of tensor products and operator ideal norms.

\section{Projection constants  and operator ideal norms}
The theory of Banach operator ideals  found significant applications within the study of projection constants. In this section, we gather some
 results which are of particular importance for   our study. 
For further information we refer to
\cite{defant1992tensor, diestel2001operator, diestel1995absolutely, jameson1987summing, pisier1986factorization, tomczak1989banach, wojtaszczyk1996banach}.

\subsection{Preliminaries}
We start recalling the definitions of a couple of classical operator ideal norms together with a few of their basic properties.

Given Banach spaces $X$, $Y$ and $1 \leq p \leq \infty$, an  operator $u\in \mathcal{L}(X,Y)$ is said to be $p$-factorable
whenever  there exist a~measure space $(\Omega, \Sigma, \mu)$ and operators $v \in  \mathcal{L}(X, L_p(\mu))$,
$w \in \mathcal{L}(L_p(\mu), Y^{\ast\ast})$, satisfying the following factorization
\begin{equation*}
  \kappa_Yu\colon X \stackrel{v} \longrightarrow L_p(\mu) \stackrel{w} \longrightarrow Y^{\ast\ast}\,;
\end{equation*}
here, as usual, $\kappa_{Y}\colon Y \to Y^{\ast\ast}$ is the canonical embedding. In this case the $\gamma_p$-norm of the $p$-factorable
operator $u$ is given by
\begin{equation}\label{fac-p}
 \gamma_p(u) = \inf \|v\| \|w\|\,,
\end{equation}
where the infimum is taken over all possible factorizations.

We are  mainly interested  in the  norms  $\gamma_p$ for operators acting between  finite dimensional Banach spaces $X$ and $Y$. In
this case, the infimum in \eqref{fac-p} is realized  considering all possible factorizations of the more simple form
\begin{equation} 
\begin{tikzcd}
X  \arrow[rd, "v"']  \arrow[rr, "u"] &  & Y \,\,,\\
& \ell_p^n   \arrow[ru, "w"'] &
\end{tikzcd}
\end{equation}
where $n$ is arbitrary.

An operator $u\in \mathcal{L}(X,Y)$ is said to be $(q, p)$-summing $(1\leq p\leq q<\infty)$ if there is a~constant $C>0$ such that
for each choice of finitely many $x_1, \ldots, x_N\in X$ one has
\[
\Big(\sum_{j=1}^N \|ux_j\|_Y^q \Big)^{\frac{1}{q}} \leq C
\sup\Big\{\Big(\sum_{j=1}^{N} |x^\ast(x_j)|^p\Big)^{\frac{1}{p}}\,:\,\, \|x^\ast\|_{X^\ast} \leq 1\Big\}\,.
\]
By $\pi_{q, p}(u\colon X\to Y)$ (and $\pi_{q, p}(u)$ for short) we denote the least such $C$ satisfying this inequality. If $q=p$, then
$u$ is called $p$-summing, and we write $\pi_p(u)$ instead of $\pi_{p, p}(u)$.

The class of all $p$-factorable operators as well as the class of all $(q, p)$-summing operators form so-called Banach operator ideals
in the sense of Pietsch. This  in particular means that both norms  $\pmb{A}=\gamma_p$ and $\pmb{A}=\pi_{q, p}$  satisfy the so-called ideal property,
that is $\pmb{A}(uvw) \leq \|u\| \pmb{A}(v) \|w\|$ for every choice of appropriate operators $u$, $v$, $w$ acting between
Banach spaces. Moreover, given two Banach spaces $X,Y$, the linear space $\pmb{\mathcal{A}}(X,Y)$ of all operators 
$u: X \to Y$ with finite $\pmb{A}$-norm, together with the norm $\pmb{A}$ forms a Banach space, which contains all finite rank operators.
In what follows, if $X$ is any finite-dimensional Banach space and $\pmb{A}$ is an ideal norm,  we write
$\pmb{A}(X) := \pmb{A}(\mbox{id}_X\colon X \to X)$.

We recall that $\pi_1$ and $\gamma_\infty$ are in trace duality in the sense that for each operator $u: X \to Y$ between finite
dimensional Banach spaces we have
\begin{equation}
\label{trace}
\gamma_\infty(u) = \sup\big\{|\text{tr}(uv)|\,: \, \, \, \pi_1(v\colon Y\to X) \leq 1\big\}\,.
\end{equation}
The following keystone result of the theory of $2$-summing operators has numerous applications. It states that for every $n$-dimensional
Banach space $X$
\begin{equation}
\label{pi2formula}
\pi_2(X) = \sqrt{n}\,.
\end{equation}
Combined  with Pietsch's famous factorization for $p$-summing operators (for the special case $p=2$) equation \eqref{pi2formula}
in fact yields the  fundamental Kadets-Snobar theorem from~\eqref{kadets1}, as a simple consequence.

A Banach space $X$ has the Gordon-Lewis property if every 1-summing operator $u\colon X \to \ell_2$ is $1$-factorable. In this case, there is a~constant $C>0$ such that for all  1-summing operators $u \colon X \to \ell_2$
\begin{equation}
\label{eq: def G-L}
\gamma_1(u)\leq C\,\pi_1(u)\,,
\end{equation}
and the best such $C$ is called the Gordon-Lewis constant of $X$ and denoted by $\mbox{gl} (X)$ (again we put $\mbox{gl}(X)=\infty$
whenever $X$ does not have the Gordon-Lewis property).

\smallskip

\subsection{Some classical estimates}\label{section: classical estimates}
In this subsection we collect    some  estimates of projection constants in terms of
operator ideal norms and geometric properties in  Banach spaces which will be of relevance later.

The projection constant of a~Banach space $X$ can be formulated in terms of the  $\infty$-factorization norm of the identity
operator $\id_X$.  More precisely, if $X$ is a~Banach space and $X_0$
is any subspace of $L_{\infty}(\mu)$ isometric to $X$, then
\begin{align} \label{gammainfty}
\boldsymbol{\lambda}(X)= \gamma_\infty(\id_X)= \boldsymbol{\lambda}(X_0, L_{\infty}(\mu)).
\end{align}
Thus, for any finite-dimensional Banach space $X$, we have
\begin{equation} \label{formulaX}
\boldsymbol{\lambda}(X) \leq d\big(X, \ell_\infty^{\text{dim}\,X}\big)\,.
\end{equation}

Let us mention that there is an open problem related to the above estimate (see \cite{rutovitz}): Does
there exist a~positive function $\phi$ on $[1, \infty)$ such that, for every finite-dimensional
Banach space $X$ one has
\[
d\big(X, \ell_\infty^{\text{dim}\,X}\big) \leq \phi(\boldsymbol{\lambda}(X))\,?
\]
We note that according to Sch\"utt's result from  \cite{schutt1978projection} there is a~universal constant $c>0$ such that
for every real finite-dimensional symmetric Banach space $X$,
\[
d\big(X, \ell_\infty^{\text{dim}\,X}\big)
\leq c \phi(\lambda(X)) = c\,\big(1 + \log \boldsymbol{\lambda}(X)\big)^{\frac{7}{2}}\boldsymbol{\lambda}(X)\,.
\]

It is worth noting here another estimate due to Lewis \cite{Lewis1975}, and some notation is required to formulate this result. Given a~basis 
$B=\{b_1, \ldots, b_n\}$ for an  $n$-dimensional Banach space $X$ and  $\pi$ a~permutation of $\{1, \ldots, n\}$, let $g_{\pi}$ be the operator on $X$ defined by $g_{\pi}(b_i)= b_{\pi(i)}$ for each $1\leq i \leq n$. The  diagonal symmetry constant of $B$ is given by $\delta(B)=\sup_{\pi}\|g_{\pi}\|$. This constant defines a~corresponding  symmetry parameter for the space $X$ itself given by \[\delta(X)= \inf_B \delta(B)\,,
\]
where the infimum is taken over all bases $B$ for $X$. 
In the remarkable paper \cite{Lewis1975}, Lewis proved that there is  a constant $\gamma >0$ such that for all $n$-dimensional Banach spaces~$X$,
\[
d(X, \ell_{1}^n) \leq \gamma\,\boldsymbol{\lambda}(X^{*})^{2}\,\delta(X)^3\,.
\]
Concerning the above mentioned open problem,
the result of Lewis implies that, given a  class of all finite dimensional spaces $X$, for which uniformly  $\delta(X^{*}) \leq C$ for some $C\ \geq 1$, one has 
\[
d\big(X, \ell_\infty^{\text{dim}\,X}\big)\ \leq \gamma C^{3}\,\boldsymbol{\lambda}(X)^{2}\,.
\]
Note that the proof of the estimate of Lewis is based on Grothendieck's theorem - showing   the constant $\gamma=16 K_G$. For completeness we recall  Grothendieck's theorem: Every operator $u\colon L_1(\mu) \to L_2(\mu)$
is $1$-summing with $\pi_1(u) \leq K_G \|u\|$ (see e.g. \cite[Theorem 1.13]{diestel1995absolutely}), where 
\[
K_G \leq \frac{\pi}{2 \log(1 + \sqrt{2})}
\]
is the Grothendieck constant (a bound proved by  Krivine in  \cite{krivine1978constantes}).

 By \eqref{trace} and \eqref{gammainfty}
for any finite-dimensional Banach space $X$
\[
\boldsymbol{\lambda}(X) = \sup\big\{|\text{tr}(v)|\,:\, \, \pi_1(v\colon X\to X) \leq 1\big\}.
\]
This shows that bounds for the projection constant of a Banach space $X$ are intimately connected to 
bounds for the $1$-summing norm of the identity operator on $X$.
In particular, for any $n$-dimensional Banach space $X$, we have
\begin{equation} \label{formulaB}
\boldsymbol{\lambda}(X)\,\pi_1(X) \geq n\,,
\end{equation}
and if $X$ has enough symmetries (i.e., any operator $u\colon X \to X$ which commutes with all isometries on $X$, is a~multiple of the identity map), then
\begin{equation} \label{formulaA}
\boldsymbol{\lambda}(X)\pi_1(X)= n\,.
\end{equation}

We also note that Grothendieck's theorem in combination with the ideal property of the norm $\pi_1$ gives that, for every $n$-dimensional Banach space $X$ one has (see \cite{gordon1974absolutely} or \cite[10.15, p.~115]{jameson1987summing})

\begin{equation*}
\pi_1(X) \leq  K_{G}\, d(X, \ell_1^n)\,d(X, \ell_2^n)\,.
\end{equation*}
In particular, \eqref{formulaB} leads to
\begin{equation} \label{formulaC}
\frac{1}{K_G}\,\frac{n}{d(X, \ell_1^n)\,d(X, \ell_2^n)} \leq \boldsymbol{\lambda}(X)\,.
\end{equation}

In many concrete situations the projection constant $\boldsymbol{\lambda}(X_n)$ of a finite dimensional Banach
lattice $X_n = (\mathbb{C}^n, \|\cdot\|)$ is closely related with the  fundamental function of $X_n$ given by
\[
\varphi_{X_n}(k) := \Big\| \sum_{j=1}^k e_j\Big\|_{X_n}, \quad 1 \leq k \leq n\,.
\]
Note first that, given a~Banach lattice $X_n = (\mathbb{C}^n, \|\cdot\|)$, we deduce from \eqref{formulaX} that
\begin{equation} \label{Carsten1}
\boldsymbol{\lambda}(X_n) \leq d(\ell^n_\infty, X_n) \leq \varphi_{X_n}(n), \quad\, n\in \mathbb{N}\,.
\end{equation}
provided that
 $\|\id\colon X_n \to \ell_\infty\| \leq 1$, or equivalently  $\|e_k\|_{X_n}\leq 1\,, \, 1 \leq k \leq n$.
Conversely, Sch\"{u}tt proved in \cite{schutt1978projection} that
\begin{equation} \label{schuett}
\varphi_{X_n}(n) \leq \sqrt{2} \, \|\id\colon \ell_2 \to  X_n\| \, \boldsymbol{\lambda}(X_n)\,.
\end{equation}
If the Banach lattice $X_n = (\mathbb{C}^n, \|\cdot\|)$  is symmetric with normalized standard
unit vector basis, then
\[
d(X_n, \ell_1^n) = d(X_{n}', \ell_\infty^n)\, \leq \, \frac{n}{\varphi_{X_n}(n)}\,,
\]
and hence by \eqref{formulaC}, we deduce that
\begin{equation} \label{formulaD}
\frac{1}{K_G}\,\frac{\varphi_{X_n}(n)}{d(X_n, \ell_2^n)} \leq \boldsymbol{\lambda}(X_n)\,.
\end{equation}

Under convexity and concavity assumptions more can be said. Given a Banach sequence lattice $X$ and denoting its $n$th section by
$X_n$, then the following equivalences
\begin{equation} \label{conv-conc}
\boldsymbol{\lambda}(X_n)\sim
\begin{cases}   \
\varphi_X(n)  &  \text{if $X$ is $2$-convex}\\[2mm]
n^{\frac{1}{2}}  &  \text{if $X$ is $2$-concave}\,.\\[2mm]
\end{cases}
\end{equation}
hold with constants only depending on $X$ and not on the dimension~$n$.

Finally, we briefly discuss the importance of the norms  $\pi_2$ and $\pi_{2, 1}$ for the estimation of  projection constants.~Assume that $X$ is an $n$-dimensional Banach space. Since 
any finite rank operator is $2$-summing, $\mathcal{L}(\ell_\infty, X) = \Pi_2(\ell_\infty, X)$. 
Defining 
\[
\boldsymbol{\Delta}_n(X) := \big\|\id\colon \mathcal{L}(\ell_\infty, X) \to \Pi_2(\ell_\infty, X)\big\|\,,
\]
we get
\begin{equation}
\label{pi2lambda}
\frac{\sqrt{n}}{\boldsymbol{\Delta}_n(X)} \, \leq \,\boldsymbol{\lambda}(X)\,.
\end{equation}
Indeed, given $\varepsilon>0$, we  find an isometric embedding $I\colon X \to \ell_\infty$, and a~projection
$P\colon \ell_\infty \to Y:= I(X)$ such that $\|P\| \leq (1 + \varepsilon)\boldsymbol{\lambda}(X)$. Combining  \eqref{pi2formula}
with the ideal properties of all $2$-summing operators, yields
\[
\sqrt{n} = \pi_2(Y)= \pi_2(P\,\id_Y) \leq {\boldsymbol{\Delta}}_n(X)\,\|P\|
\leq (1 + \varepsilon) {\boldsymbol{\Delta}}_n(X)  \, \boldsymbol{\lambda}(X)\,.
\]

We conclude with an application of a~deep theorem proved by Pisier \cite{pisier1986factorization} (see also \cite[Theorem 10.9]{diestel1995absolutely}).
A~special case of this result states (see \cite[Theorem 14.1]{jameson1987summing}): If $K$ is a~compact Hausdorff
space, $X$ a Banach space and $u\colon C(K) \to X$ is $(2, 1)$-summing, then there is a~positive functional $\phi \in C(K)^{\ast}$ with $\|\phi\|=1$ such that
\[
\|uf\|_{X}^2 \leq 2 \,\pi_{2,1}(T) \,\phi(|f|)\,\|f\|_{C(K)}, \quad\, f\in C(K)\,.
\]
This immediately leads to the  following conclusion (see \cite[14.2]{jameson1987summing}) that for every
finite dimensional Banach space $X$,
\[
\pi_1(X) \leq 2 \,\pi_{2,1}(X)^2 \,\boldsymbol{\lambda}(X)^2\,.
\]
In combination with \eqref{formulaB}, we obtain
\[
\frac{n}{2\,\pi_{2,1}(X)^2} \leq \boldsymbol{\lambda}(X)^3\,.
\]

\bigskip

\section{Annihilating coefficients}
\label{Annihilating coefficients}

Let $X_n = (\mathbb C^n, \|\cdot\|)$ be a Banach space.
Then for each pair of  (finite) index sets $I,J \subset  \mathbb{N}_0^n$ with $I \subset J$ we define the projection
\begin{equation} \label{anni}
  \mathbf{Q}_{J,I}:\mathcal{P}_J(X_n) \to \mathcal{P}_I(X_n)\,,\,\,\,
P \mapsto \sum_{\alpha \in  I} c_\alpha(P) z^\alpha\,.  
\end{equation}
This is the projection which  annihilates those coefficients of a polynomial in $\mathcal{P}_J(X_n)$, which have indices in the complement of  $I$.

\begin{remark} \label{simple}
Given $I,J \subset  \mathbb{N}_0^n$ with $I \subset J$ and a  Banach space $X_n= (\mathbb{C}^n, \|\cdot\|)$, we have
\[
\boldsymbol{\lambda}\big(\mathcal{P}_{I}(X_n)
\big) \leq \|\mathbf{Q}_{J,I}\, \| \boldsymbol{\lambda}\big(\mathcal{P}_J(X_n)\big)\,.
\]
Indeed, factorize $\id_{\mathcal{P}_I(X_n)} = \mathbf{Q}_{J,I}\circ j_{I,J}$ \, through the\,canonical embedding $j_{I,J}: \mathcal{P}_I(X_n)\hookrightarrow \mathcal{P}_J(X_n)$ and \,the
projection $\mathbf{Q}_{J,I}: \mathcal{P}_J(X_n)\hookrightarrow \mathcal{P}_I(X_n)$, and \, use \eqref{gammainfty}.
\end{remark}

Let us see a first simple but important example.

\begin{proposition} \label{Cauchy}
  Let $J \subset \mathbb{N}_0^n$ be an index set of degree $m$. Then, for each $0 \leq k \leq m$ and each Banach space $X_n = (\mathbb{C}^n, \|\cdot\|)$,
  \[
  \big\|\mathbf{Q}_{J,J_k}\colon \mathcal{P}_J(X_n) \to \mathcal{P}_{J_k}(X_n)\big\| =1\,,
  \]
  so in particular
  \[
  \boldsymbol{\lambda}\big(\mathcal{P}_{J_k}(X_n)\big) \leq  \boldsymbol{\lambda}\big(\mathcal{P}_J(X_n)\big)\,.
  \]
\end{proposition}

\begin{proof}
  Given  $P \in \mathcal{P}_J(X_n)$,
  \[
   P = \sum_{k=0}^{m} \mathbf{Q}_{J,J_k}P 
  \]
   is the unique Taylor expansion of $P$, and by Cauchy's inequality we have that
  $\|\mathbf{Q}_{J,J_k}P\| \leq \|P\|$ for all $0 \leq k \leq m$ (see, e.g., \cite[Proposition 15.33]{defant2019libro}).
\end{proof}

Less standard is what we now prove for subsets of tetrahedral indices $\alpha$, i.e., $\alpha \in \Lambda_T( \leq m,n)$. The following tool is crucial for our purposes, and it is basically due to Ortega-Cerd\`a, Ouna\"{\i}es and Seip. Since it appears in the unpublished  manuscript \cite{ortega2009sidon}
and we need a slight improvement of this result, we include a detailed proof for completeness.
It will be useful to have the definition of the following constant present,
\begin{equation}\label{kappa}
    \kappa:=\bigg(\prod_{k=1}^\infty\
\sinc{\frac{\pi}{\mathfrak{p}_k}}\bigg)^{-1}=2.209\ldots\,,
\end{equation}
where  $\mathfrak{p}_1=2, \mathfrak{p}_2=3,\ldots$ stands for the  sequence of prime numbers and $\sinc x:=(\sin x)/x$.

\begin{theorem} \label{OrOuSe}
Let $X = (\CC^n, \| \cdot \|)$ be a Banach lattice  and  $ J \subset \mathbb{N}_0^n$ an  index set of degree $m$. Then
\[
\big\|\mathbf{Q}_{J,J_T}: \mathcal{P}_J(X_n) \to \mathcal{P}_{J_T}(X_n)\big\|
\leq \kappa^m\,,
\]
where $J_T = J \cap \Lambda_T(\le\! m,n)$. In particular,
\[
\boldsymbol{\lambda}\big(\mathcal{P}_{J_T }(X_n)\big) \leq \kappa^m \boldsymbol{\lambda}\big(\mathcal{P}_J(X_n)\big)\,.
\]
\end{theorem}

\begin{proof}

As usual, we write $\pi(x)$ for the counting function of the prime numbers.  Now, given
$$t=(t_1,\ldots,t_{\pi(m)})  \in  Q :=[0,1]^{\pi(m)}\,,$$
  define
\[
r_m(t)=c_m \exp \left(2\pi i\Bigl(\frac {t_1}2 + \frac {t_2}3+\cdots+ \frac
{t_{\pi(m)}}{\mathfrak{p}_{\pi(m)}}\Bigr)\right),
\]
where
\[
c_m=\prod_{k=1}^{\pi(m)} \left(\frac {\mathfrak{p}_k}{2\pi i}
\Bigl(e^{\frac {2\pi i}{\mathfrak p_k}} -1\Bigr)\right)^{-1}.
\]
Note that the function $r_m:Q\to \mathbb{C}$ has the following properties:
\begin{enumerate}
\item[(i)] $\int_Q r_m(t)\,d\mu(t)=1$,
\item[(ii)] $\int_Q r_m^k(t)\, d\mu(t)=0$\ \ for each \, $2 \leq k \leq m$,
\item[(iii)] $|r_m(t)|\le \kappa$\ \ for all $t \in Q$,
\end{enumerate}
here $d\mu$ denotes the  Lebesgue measure on $Q$.
Indeed,  (i) and (ii)
are trivial and follow by the definition of the function, and (iii) holds because
$|r_m(t)| = |c_m|$ and
\[
|c_m|^{-2}= \prod_{k=1}^{\pi(m)} \frac{\mathfrak{p}_k^2}{(2\pi)^2}
\Bigl|e^{\frac{2\pi i}{\mathfrak{p}_k}}-1\Bigr|^2
 = \prod_{k=1}^{\pi(m)} \sinc^2\frac{\pi}{\mathfrak{p}_k}.
\]
Given a polynomial $P \in \mathcal P_J(X_n)$, note that by the properties (i) and (ii) we have the representation
\[
\mathbf{Q}_{J,J_T}P(z)=\int_{Q^n} P(z_1r_m(t^1),\ldots, z_n r_m(t^n))\,
d\mu(t^1)\cdots d\mu(t^n)\,,\,\,\, z \in X_n\,.
\]
Since $X_n$ is a Banach lattice   we by (iii) deduce that  $|P(z_1 r_m(t^1),\ldots, z_n
r_m(t^n))| \le \kappa^m \|P\|_{\mathcal{P}_{J}(X_n)}$ for every $z~\in~B_{X_n}$, and therefore
\[
\Vert \mathbf{Q}_{J,J_T}(P) \Vert_{\mathcal{P}_{\Lambda_T}(X_n)} \leq \kappa^m \,\Vert P \Vert_{\mathcal{P}_{J}(X_n)}.
\]
This proves the first statement, the second one is an immediate consequence of the observation from
Remark~\ref{simple}.
\end{proof}

\bigskip

\section{Homogeneous building blocks} \label{Homogeneous building blocks}
By the Kadets-Snobar theorem from \eqref{kadets1} we know that for any
 Banach space $X_n = (\mathbb{C}^n, \|\cdot\|) $ and any
finite index set
$J \subset \mathbb{N}_0^n $
we have
\begin{equation}\label{tincho}
    \boldsymbol{\lambda}\big(\mathcal{P}_{J}(X_n)\big)
 \leq \sqrt{|J|}
 \leq \sqrt{ m+1} \max_{0 \leq k \leq m}
 \sqrt{|J_k|}\,.
 \end{equation}
The following theorem improves this considerably -- it shows  that estimating the projection constant of $\mathcal{P}_{J}(X_n)$ from above or below reduces to estimating
    the projection constants of its 'homogeneous building blocks' $\mathcal{P}_{J_k}(X_n), \, 0 \leq k \leq m$.
\smallskip

    \begin{theorem} \label{degree-homo}
 Let $X_n = (\mathbb{C}^n, \|\cdot\|) $ be a Banach space, and
$J \subset \mathbb{N}_0^n$ an  index set of degree   $m$. Then
 \[
 \max_{0 \leq k \leq m} \boldsymbol{\lambda}\big(\mathcal{P}_{J_k}(X_n)\big) \,\,\leq \,\,
 \boldsymbol{\lambda}\big(\mathcal{P}_{J}(X_n)\big)
 \,\,\leq \,\,(m+1) \max_{0 \leq k \leq m} \boldsymbol{\lambda}\big(\mathcal{P}_{J_k}(X_n)\big)
 \,.
 \]
 In particular, for any  Banach sequence lattice $X$
 \[
  \lim_{m \to \infty}  \sup_{n\in \mathbb{\mathbb{N}}}
 \frac{\sqrt[m]{\boldsymbol{\lambda}\big(\mathcal{P}_{\leq m}(X_n)\big)}}{\sqrt[m]{\max_{0 \leq k \leq m} \boldsymbol{\lambda}\big(\mathcal{P}_{k}(X_n)\big)}}
 = 1
 \,.
 \]
\end{theorem}

It should be noted that, in general, there is a dependence on $m$ when we compare  $\mathcal{P}_{\leq m}(X_n)$ with $\mathcal{P}_{k}(X_n)$ if $ 0 \leq k \leq m$. For example, $\boldsymbol{\lambda}(\mathcal{P}_{k}(\mathbb{C}))=1$ for all $0 \leq k \leq m$ (since $\mathcal{P}_{k}(\mathbb{C})$  is one dimensional), and on the other hand by Corollary \ref{LoKha}, we have $\boldsymbol{\lambda}(\mathcal{P}_{\leq m}(\mathbb{C})) \asymp 1 + \log m$. The preceding theorem shows that this dependence in fact is subexponencial for any space $X_n$.

\begin{proof}
Note first that by Proposition~\ref{Cauchy}
for all $0 \leq k \leq m$
\[
 \boldsymbol{\lambda}\big(\mathcal{P}_{J_k}(X_n)\big) \,\,\leq \,\,
 \boldsymbol{\lambda}\big(\mathcal{P}_{J}(X_n)\big)\,,
 \]
so that it remains  to  check the second estimate.
We (as in the proof of Proposition~\ref{Cauchy}) use  that each $P \in \mathcal{P}_{J}(X_n)$ has a unique Taylor series expansion
$P = \sum_{k=0}^m P_k$ with $P_k \in \mathcal{P}_{J_k}(X_n)$, and from
the Cauchy inequality we  know that
$\|P_k\| \leq \|P\| $ for all $0 \leq k \leq m$. Consequently,  the two operators
\begin{align*}
&
U: \mathcal{P}_{J}(X_n) \to  \bigoplus_\infty \mathcal{P}_{J_k}(X_n)\,,\,\,\,\, P \mapsto (P_k)_{k=1}^m
\\&
V:  \bigoplus_1 \mathcal{P}_{J_k}(X_n) \to \mathcal{P}_{J}(X_n) \,,\,\,\,\, (Q_k)_{k=1}^m  \mapsto \sum_{k=1}^m Q_k\,,
\end{align*}
 both have norms $\leq 1$.
 Now fix some $\varepsilon >0$, and choose for each $1 \leq k \leq m$ an appropriate  factorization
\[
\xymatrix
{
  \mathcal{P}_{J_k}(X_n) \ar[r]^{\id} \ar[d]_{u_k}
 & \mathcal{P}_{J_k}(X_n)
 \\
 \ell_\infty^{M_k} \ar[ur]_{v_k}
 }
\]
such that $\|u_k\|\leq 1$ and $\|v_k\|\leq (1 + \varepsilon) \boldsymbol{\lambda} \big( \mathcal{P}_{J_k}(X_n) \big)$.
Then we arrive at  the following commutative diagram
\[
\xymatrix
{
  \mathcal{P}_{J}(X_n) \ar[r]^{\id} \ar[d]_{U}
 & \mathcal{P}_{J}(X_n)
 \\
\bigoplus_\infty \mathcal{P}_{J_k}(X_n)
\ar[d]_{\bigoplus u_k}
&
\bigoplus_1 \mathcal{P}_{J_k}(X_n) \ar[u]_{V}
\\
\bigoplus_\infty \ell_\infty^{M_k} \ar[r]^{\Phi}
&
\bigoplus_1 \ell_\infty^{M_k}  \ar[u]_{\bigoplus v_k}\,,
}
\]
where $\Phi$ stands for the identity map which here obviously has norm $\leq m+1$.
But
\[
\big\|\bigoplus u_k\big\| \leq \max{\|u_k\|} \leq 1\,,
\]
as well as
\[
\big\|\bigoplus v_k\big\| \leq \max{\|v_k\|} \leq (1+\varepsilon)  \max_{0 \leq k \leq m} \boldsymbol{\lambda}\big(\mathcal{P}_{J_k}(X_n)\big)\,.
\]
This finally gives
\begin{align*}
  \boldsymbol{\lambda}\big(\mathcal{P}_{J}(X_n)\big)
  \leq \|U\| \,\big\|\bigoplus u_k\big\| \,\,\boldsymbol{\lambda}\big(\bigoplus_\infty \ell_\infty^{M_k}\big)\,\,
\|\Phi\|\,\big\|\bigoplus v_k\big\| \, \|V\|
\leq (m+1)   \max_{0 \leq k \leq m} \boldsymbol{\lambda}\big(\mathcal{P}_{J_k}(X_n)\big)\,,
\end{align*}
the conclusion.
\end{proof}

For certain index sets of multi indices, as e.g., $J =  \Lambda(\leq\!\! m,n)$, the preceding inequality
 simplifies -- for the price of weaker constants.

\smallskip

\begin{corollary} \label{degree-homoA}
 Let $X_n = (\mathbb{C}^n, \|\cdot\|) $ be a Banach space. Then
 \[
  \boldsymbol{\lambda}\big(\mathcal{P}_{m}(X_n)\big) \,\,\leq \,\,
 \boldsymbol{\lambda}\big(\mathcal{P}_{\leq m}(X_n)\big)
 \,\,\leq \,\,  e^{m(m+2)} \boldsymbol{\lambda}\big(\mathcal{P}_{m}(X_n)\big)
 \,.
 \]
  \end{corollary}
  
The proof of this corollary  
follows immediately from Theorem~\ref{degree-homo} combined with the following independently interesting lemma,
which indicates  that  the projection constants $\boldsymbol{\lambda}\big(\mathcal{P}_{k}(X_n)\big)$
are   monotonous in $k$ (up to constants independent on $n$).
This  result is due to Aron and Schottenloher  from \cite[Proposition 5.3]{aron1976compact}, and we here for our
purpose isolate the constants obtained from their proof.

\begin{lemma} \label{old}
 Let $X_n = (\mathbb{C}^n, \|\cdot\|)$
be a  Banach space and $k,\ell \in \mathbb{N}$ with $k \leq \ell$. Then
\[
 \boldsymbol{\lambda}\big(\mathcal{P}_{k}(X_n), \mathcal{P}_{l}(X_n)\big)
\,\leq \, 2^{\frac{1}{2}(\ell-k)(\ell+k +1)} \prod_{j = k+1}^{\ell} \ccc(j,X_n) 
\,\leq e^{(\ell-k)(\ell+k +1)}\, \,.
\]
In particular,
\[
\boldsymbol{\lambda}\big(\mathcal{P}_{k}(X_n)\big) \,\leq \, e^{(\ell-k)(\ell+k +1)} \,\boldsymbol{\lambda}\big(\mathcal{P}_{\ell}(X_n)\big)\,.
\]
\end{lemma}

\begin{proof}
The proof follows by induction, and we start with  $\ell = k +1$. Choose $e \in X_n$ and
  $\gamma \in X_n^\ast$ such that $\gamma(e) = 1 = \| \gamma    \|$, and define
     \begin{align*}
  &
    \rho: \mathcal{P}_{k}(X_n) \to \mathcal{P}_{k+1}(X_n)\,, \quad \rho(Q)(x): = \gamma(x) Q(x)
    \\&
    \pi: \mathcal{P}_{k+1}(X_n) \to \mathcal{P}_{k}(X_n)\,, \quad \pi(P)(x): = \,\sum_{j=1}^{k+1} \binom{k+1}{j} (-1)^{j+1} \gamma(x)^{j-1} \overset{\vee}{P}(e^{(j)},x^{(k+1-j)})\,.
  \end{align*}
  Using that for  $P \in \mathcal{P}_{k+1}(X_n)$ and $x \in X_n$
  \[
  P(x) - P (x - \gamma(x) e) = \,\sum_{j=1}^{k+1} \binom{k+1}{j} (-1)^{j+1} \gamma(x)^{j} \overset{\vee}{P} (e^{(j)},x^{(k+1-j)})=\gamma(x)\pi(P)(x),
  \]
and that if $P=\rho(Q)$ then $P (x - \gamma(x) e)=0$ for every $x \in X_n$,  a simple calculation gives $\gamma Q=\rho(Q)=\gamma\pi(\rho(Q))$ and thus
  \[
 \pi \circ \rho = \text{id}_{\mathcal{P}_{k}(X_n)}\,.
  \]
  Hence
  \[
\boldsymbol{\lambda}\big(\mathcal{P}_{k}(X_n), \mathcal{P}_{k+1}(X_n)\big) \leq \|\pi\|  \|\rho \| = \|\pi\|\,.
  \]
  But  by \eqref{projl1} we obtain  $\|\pi\| \leq 2^{k+1} \ccc(k+1,X_n) \leq  2^{k+1} e^{k+1}$, which settles the case $\ell = k+1$. The general case
  follows by induction:
  \begin{equation*}
   \boldsymbol{\lambda}\big(\mathcal{P}_{k}(X_n), \mathcal{P}_{l}(X_n)\big)
    \,\leq \, 2^{(k+1)+ \ldots +   \ell}
    \prod_{j = k+1}^{\ell} \ccc(j,X_n)
   \,.  \qedhere
 \end{equation*}
      \end{proof}

    \smallskip  
      
\section{Fixing the degree}
\label{comparing}

Given a set $J \subset \Lambda(m,n)$ of $m$-homogeneous indices, the following theorem relates the projection constant of $\mathcal{P}_{J}(X_n)$ with the $m$th power of the projection constant of $X^\ast_n$ and the norm of the projection $\mathbf{Q}_{\Lambda(m,n),J}$ on $\mathcal{P}_{m}(X_n)$ annihilating the coefficients with indices in the complement of~$J$ (see again~\eqref{anni}).

\begin{theorem} \label{tensor}
Let $X_n = (\mathbb{C}^n, \|\cdot\|) $ be any Banach space, and
$J \subset \Lambda(m,n)$. Then
\[
\boldsymbol{\lambda}\big(\mathcal{P}_{J}(X_n)\big) \,\,\leq \,\,\ccc(m,X_n) \|\mathbf{Q}_{\Lambda(m,n),J}\| \, \boldsymbol{\lambda}(X^\ast_n)^{m}\,.
\]
Moreover, whenever $X_n$ has enough symmetries, it holds
\[
\frac{|J|}{ n^m\ccc(m,X_n)\|\mathbf{Q}_{\Lambda(m,n),J}\|}\,\,\boldsymbol{\lambda}(X^\ast_n)^{m}
\,\,\leq \,\,
\boldsymbol{\lambda}\big(\mathcal{P}_{J}(X_n)\big)\,.
\]
 \end{theorem}

\begin{proof}
By Remark~\ref{simple} we have that
\[
\boldsymbol{\lambda}\big(\mathcal{P}_{J}(X_n)\big) \,\,\leq \,\,\|\mathbf{Q}_{\Lambda(m,n),J}\| \, \boldsymbol{\lambda}\big(\mathcal{P}_{m}(X_n)\big)\,,
\]
and hence it remains to show that $\boldsymbol{\lambda}\big(\mathcal{P}_{m}(X_n)\big) \,\,\leq \,\,\ccc(m,X_n) \boldsymbol{\lambda}(X^\ast_n)^{m}$. Recall that by polarization
\[
\big\|i_1: \mathcal{P}_{m}(X_n) \to \mathcal{L}_m(X_n)\,, \,\,\, P \mapsto \check{P} \big\| \leq \ccc(m,X_n) \,.
\]
Moreover, we clearly have that
\[
\big\|i_2:  \mathcal{L}_m(X_n)  \to \mathcal{P}_{m}(X_n)\,, \,\,\, L \mapsto [P:x \mapsto L(x, \cdots, x)] \big\| \leq 1\,.
\]
On the other hand,  we may  identify  $\mathcal{L}_m(X_n)$ isometrically with the
$m$th  injective tensor product of $X_n^\ast$ with itself,
\begin{equation} \label{eisa}
 \mathcal{L}_m(X_n) = \otimes_\varepsilon^m X_n^\ast\,,
\end{equation}
and it is well-known that
$\gamma_\infty\big(\id_{\otimes_{\varepsilon}^{m}X_n^\ast}\big)  =
 \gamma_\infty(\id_{X_n^\ast})^m$  (see, e.g., \cite[Section 34.6]{defant1992tensor})\,.  
But then by \eqref{gammainfty} we obtain that
\begin{align*}
\boldsymbol{\lambda}\big(\mathcal{P}_{m}(X_n)\big) = \gamma_\infty\big(\id_{\mathcal{P}_{m}(X_n)}\big) & \leq \|i_1\| \|i_2\| \gamma_\infty(\id_{\mathcal{L}_{m}(X_n)})
\\&
\leq \ccc(m,X_n) \gamma_\infty\big(\id_{\otimes_{\varepsilon}^{m}X_n^\ast}\big) = \ccc(m,X_n) \boldsymbol{\lambda}(X_n^\ast)^m\,.
\end{align*}
It remains to check the lower bound, so assume that $X_n = (\mathbb{C}^n, \|\cdot\|) $ has enough symmetries.
By \eqref{formulaB} we have
\[
\boldsymbol{\lambda}\big(\mathcal{P}_{J}(X_n)\big) \geq \frac{|J|}{ \pi_1(\id_{\mathcal{P}_J(X_n)})}\,.
\]
Now we follow a similar way as in the first part of the proof, and define $i_1$ and $i_2$ as above. Moreover, let
\[
j_1  \mathcal{P}_{J}(X_n)\hookrightarrow \mathcal{P}_m(X_n)
\]
be the canonical embedding. Then
\[
 \pi_1(\id_{\mathcal{P}_J(X_n)}) =
 \pi_1(\mathbf{Q}_{\Lambda(m,n), J}\circ i_2 \circ \id_{\mathcal{L}_m(X_n)} \circ i_1 \circ j_1)
 \leq \|\mathbf{Q}_{\Lambda(m,n), J}\|\, \ccc(m,X_n) \, \pi_1(\id_{\mathcal{L}_m(X_n)})\,.
\]
But the $1$-summing norm $\pi_1$  is $\varepsilon$-tensor stable (see, e.g.,
\cite[Section 34.9]{defant1992tensor})-- so, using \eqref{eisa}, this shows that
\[
\pi_1(\id_{\mathcal{L}_m(X_n)}) = \pi_1(\id_{X_n^\ast})^m\,.
\]
Finally, since $X_n$ has enough symmetries, we have that   $\pi_1(\id_{X_n^\ast}) = n/ \boldsymbol{\lambda}(\id_{X_n^\ast})$  by  \eqref{formulaA},
which all together gives the conclusion.
\end{proof}

\smallskip

\begin{corollary}\label{immediate}
Let $X_n = (\mathbb{C}^n, \|\cdot\|) $ be a
Banach space, and
$J \subset \Lambda(m,n)$.  Then
 \[
 \boldsymbol{\lambda}\big(\mathcal{P}_{J}(X_n)\big)
\,\,\leq \,\, e^m \, \|\mathbf{Q}_{\Lambda(m,n),J}\| \,\boldsymbol{\lambda}(X^\ast_n)^{m}
\,.
\]
Moreover, assume that  $X_n = (\mathbb{C}^n, \|\cdot\|) $  is a Banach lattice which has  enough symmetries. Then
 \[
 \frac{1}{e^m m^m}   \boldsymbol{\lambda}(X^\ast_n)^{m}
\,\,\leq \,\,
\boldsymbol{\lambda}\big(\mathcal{P}_{m}(X_n)\big)
\,\,\leq \,\, e^m \boldsymbol{\lambda}(X^\ast_n)^{m}\,,
 \]
 as well as
\[
 \frac{1}{(e\kappa)^m m^m}   \boldsymbol{\lambda}(X^\ast_n)^{m}
\,\,\leq \,\,
\boldsymbol{\lambda}\big(\Pp_{\Lambda_T(m,n)}(X_n)\big)
\,\,\leq \,\, (e\kappa)^m \boldsymbol{\lambda}(X^\ast_n)^{m}\,.
 \]
    \end{corollary}

\begin{proof}
For the three upper estimates we apply the first statement of Theorem~\ref{tensor}; indeed, we know that  $\ccc(m,X_n) \leq e^m$
as well as  $\|\mathbf{Q}_{J,\Lambda_T(m,n)}\| \leq~\kappa^m$ by Theorem~\ref{OrOuSe}. The proofs of the lower bounds are similar:
Assume first that $J = \Lambda(m,n)$,
then for each $m,n$
\[
|\Lambda(m,n)| = \binom{n+m-1}{m} \ge \Big(\frac{n+m-1}{m}\Big)^m \ge \Big(\frac{n}{m}\Big)^m\,.
\]
Using again that  $\ccc(m,X_n) \leq e^m$ and combining it with the second statement of Theorem~\ref{tensor}, gives the first lower bound.  Another application of Theorem~\ref{OrOuSe} and by the fact that  for each $m,n$
\[
|\Lambda_T(m,n)| = \binom{n}{m} \ge \Big(\frac{n}{m}\Big)^m\,,
\]
we in the same way  obtain the second one.
\end{proof}

\smallskip

We now use the results from Section~\ref{Homogeneous building blocks} to extend the previous corollary
to a large variety of (eventually non-homogeneous) index sets $J \subset \mathbb{N}_0^n$ of degree $m$. 

\smallskip

\begin{corollary}\label{immediateABC}
For every Banach space  $X_n = (\mathbb{C}^n, \|\cdot\|) $, and index set $J \subset \mathbb{N}_0^n$  of  degree $m$ the following statements hold{\rm:}
\begin{itemize}
\item[(i)]    
$\boldsymbol{\lambda}\big(\mathcal{P}_{J}(X_n)\big)
    \,\,\leq \,\,
    \left[ (m+1)e^m \dis\max_{0\leq k \leq m}\|\mathbf{Q}_{\Lambda(k,n),J_k}\| \right] \, \boldsymbol{\lambda}(X^\ast_n)^{m}\,.
     $
     \vspace{3mm}
     \item[(ii)] Assume that $X_n = (\mathbb{C}^n, \|\cdot\|) $ is a Banach lattice with   enough symmetries. Then
     \[
     \frac{1}{e^m m^m}   \boldsymbol{\lambda}(X^\ast_n)^{m}
\,\,\leq \,\,
    \boldsymbol{\lambda}\big(\mathcal{P}_{J}(X_n)\big)
    \,\,\,\,\,\,\,
    \text{provided  $\Lambda(m,n) \subset J$}\,,
    \,\,\,\,\,\,\,
     \]
           and
     \[
     \frac{1}{e^m\kappa^{2m} m^m}   \boldsymbol{\lambda}(X^\ast_n)^{m}
\,\,\leq \,\,
    \boldsymbol{\lambda}\big(\mathcal{P}_{J}(X_n)\big) 
    \,\,\,\,\,\,\,
    \text{provided $m \leq n$ and $\Lambda_T(m,n) \subset J$}\,.
    \,\,\,\,\,\,\,
     \]       
      \end{itemize}
         In particular, if $X_n = (\mathbb{C}^n, \|\cdot\|) $ is a Banach lattice with   enough symmetries, and $J$ is one of the index sets  $\Lambda_T (m,n)$,
          $\Lambda (m,n)$,
  $\Lambda_T (\leq m,n)$,  or  $\Lambda (\leq m,n)$, then
 \[
 \boldsymbol{\lambda}\big(\mathcal{P}_{J}(X_n)\big)
 \sim_{C(m)}
  \boldsymbol{\lambda}(X^\ast_n)^{m}\,.\]
  \end{corollary}

   \begin{proof} 
   For the proof of $(i)$ note that by Theorem~\ref{degree-homo} and Corollary~\ref{immediate} (first statement)
     \[
  \boldsymbol{\lambda}\big(\mathcal{P}_{J}(X_n)\big)
 \,\,\leq \,\,(m+1) \max_{0 \leq k \leq m} \boldsymbol{\lambda}\big(\mathcal{P}_{J_k}(X_n)\big)
 \,\,\leq \,\,
    \big[(m+1)e^m \max_{0\leq k \leq m}\|\mathbf{Q}_{\Lambda(k,n),J_k}\|\big] \, \boldsymbol{\lambda}(X^\ast_n)^{m}\,.
 \]
 To see the first claim of $(ii)$
 we use  Proposition~\ref{Cauchy} to get
 \[
 \boldsymbol{\lambda}\big(\mathcal{P}_{\Lambda(m,n)}(X_n)\big)
 =
    \boldsymbol{\lambda}\big(\mathcal{P}_{J_m}(X_n)\big)
    \,\,\leq \,\,
    \boldsymbol{\lambda}\big(\mathcal{P}_{J}(X_n)\big)\,,
   \]
  and then  (the second statement of)  Corollary~\ref{immediate}  applies.
   To prove the  second  statement of $(ii)$ we similarly deduce from Theorem~\ref{OrOuSe} and Proposition~\ref{Cauchy}   that
 \[
  \boldsymbol{\lambda}\big(\mathcal{P}_{\Lambda_T(m,n)}(X_n)\big)
  \,\,\leq \,\,
  \kappa^m
  \boldsymbol{\lambda}\big(\mathcal{P}_{J_m}(X_n)\big)
  \,\,\leq \,\,
  \kappa^m
  \boldsymbol{\lambda}\big(\mathcal{P}_{J}(X_n)\big)\,,
   \]
so here the conclusion is a consequence of Corollary~\ref{immediate} (third statement). Finally, in order to check  the very last claim,
 we again apply Theorem~\ref{OrOuSe} (for $\Lambda_T (\leq m,n)$) and Proposition~\ref{Cauchy} (for $\Lambda (\leq m,n)$)
 to the estimate proved in  $(i)$.
       \end{proof}

\bigskip

\chapter{Unconditionality}
\label{Unconditionality}

Given a    Banach sequence lattice $X_n=(\mathbb{C}^n, \|\cdot\|)$
and a finite  index set $J \subset \mathbb{N}_0^{(\mathbb{N})}$, our strategy here is   to relate the study of the projection constant of  the Banach space $\mathcal{P}_J(X_n)$  with various important invariants
of local Banach space theory -- among others the Gordon-Lewis constant,
the unconditional basis constant and the convexity/concavity constant.
In this way it is possible to use the deep knowledge on these constants for our purposes.

As discussed earlier, we are particularly interested in analyzing the projection  constant
of $\mathcal{P}_{J}(X_n)$
 viewed as a function of
the dimension $n$ of the Banach space $X_n$ and the degree~$m = \max_J |\alpha|$ of the index set $J$.

To achieve this goal we establish a
   deep link between the constant  $\boldsymbol{\lambda}\big(\mathcal{P}_{J}(X_n)\big)$
and the unconditional
basis constant  $\boldsymbol{\chimon}\big(\mathcal{P}_{J}(X_n)\big)$ (see the definition below). Among others, this    allows
to involve  probabilistic techniques in our analysis of $\boldsymbol{\lambda}\big(\mathcal{P}_{J}(X_n)\big)$. But, vice versa, this link also leads to  new independently interesting knowledge  on $\boldsymbol{\chimon}\big(\mathcal{P}_{J}(X_n)\big)$.

To make all this a bit more precise, we first recall a few more preliminaries, and second, we collect some  concrete examples to illustrate differences as well as similarities of unconditional basis and projection constants of spaces of multivariate polynomials.

\noindent {\bf A few more preliminaries.}
Recall that a family $(e_i)_{i \in I}$ of nonzero elements in a Banach space $X$ is said to be an unconditional basis
for $X$ if the span of all $e_i, \,i\in I$ is dense in $X$ and there exists a constant $K > 0$ such that for any choice of finitely supported families
$(\alpha_i)_{i \in I}$ and $(\beta_i)_{i \in I}$  of scalars with $|\beta_i| \leq |\alpha_i|$ for all $i\in I$, one has
\[
\Big\|\sum_{i \in I} \beta_i e_i\Big\|_X \leq K \Big\|\sum_{i \in I} \alpha_i e_i\Big\|_X\,.
\]
In this case, the unconditional basis constant
$$\boldsymbol{\chi}((e_i)_{i\in I})= \boldsymbol{\chi}((e_i)_{i\in I}; X)$$
of $(e_i)_{i\in I}$ is defined to be the smallest constant $K$ which satisfies the above inequality.
  We write  $\boldsymbol{\chi}((e_i)_{i \in I})= +\infty$, whenever $(e_i)_{i\in I}$ is not unconditional, and say
that $(e_i)_{i\in I}$ is a~$1$-unconditional basis, whenever $\boldsymbol{\chi}((e_i)_{i\in I}) =1$.
The unconditional basis constant
$$\boldsymbol{\chi}(X)$$
of $X$ is defined to be the infimum of $\boldsymbol{\chi}((e_i)_{i\in I})$ taken over all possible unconditional bases
$(e_i)_{i\in I}$ of~$X$.

Occasionally we need equivalent reformulations of  $\boldsymbol{\chi}((e_i)_{i\in I})$. It is easily proved that this constant in fact equals the infimum over all $K > 0$ such that for any finitely supported family $(\alpha_i)_{i \in I}$ of scalars and for any finitely supported family $(\varepsilon_i)_{i \in I}$ with  $|\varepsilon_i| =1, \,  i \in I$ we have
\begin{equation}\label{unconditionality}
\Big\Vert  \sum_{i \in I} \varepsilon_i \alpha_i e_i \Big\Vert \leq K \Big\Vert \sum_{i \in I} \alpha_i e_i \Big\Vert\,.
\end{equation}

For any finite index set $J \subset \mathbb{N}_0^n$ and  any Banach space $X_n = (\mathbb{C}^n,\|\cdot\|)$, the collection of all monomials
$z^\alpha,\, \alpha  \in J$ of course forms an unconditional basis of $\mathcal{P}_J(X_n)$, and the unconditional basis constant of this
so-called monomial basis is denoted by
\[
\boldsymbol{\chimon}\big(\mathcal{P}_{J}(X_n)\big)\,.
\]
A fundamental tool for the study of unconditionality in Banach spaces is the  Gordon-Lewis inequality from \cite{gordon1974absolutely}
(see also \cite[17.7]{diestel1995absolutely} or \cite[Proposition~21.13]{defant2019libro}): For every Banach space $X$ with an unconditional basis~$(e_i)_{i \in I}$ one has
\begin{equation} \label{gl-inequality}
{\mbox{gl}}(X)\leq  \boldsymbol{\chi}(X) \leq \boldsymbol{\chi}( (e_i)_{i \in I})\,,
\end{equation}
where we for the definition of ${\mbox{gl}}(X)$ again refer to \eqref{eq: def G-L}.
In contrast to the unconditional basis constant, the Gordon-Lewis constant has the nice property that
\begin{equation} \label{niceprop}
{\mbox{gl}}(X_0)\leq \boldsymbol{\lambda}(X_0,X)\,\, {\mbox{gl}}(X)\,,
\end{equation}
whenever $X_0$ is an isometric  subspace of $X$; this is a straight forward  consequence of the ideal properties
of the norms $\gamma_1$ and $\pi_1$.

\noindent {\bf Differences and similarities.} Given a finite dimensional Banach space $X$, the unconditional basis constant
with respect to a basis of this space and its projection constant are two quite different objects (compare for  example 
${\chi} (\ell_2^n) =1$ with $\boldsymbol{\lambda} (\ell_2^n) \sim \sqrt{n}$).
But in the case of Banach spaces of multivariate polynomials, we want to convince our reader that a better understanding of one of the two constants in many concrete situations  leads to a better understanding of the other constant.

To illustrate  this point of view, we start considering  analytic trigonometric polynomials in one variable. Recall that  $\text{Trig}_{ \{k\colon 1 \leq k\leq d\}}(\mathbb{T})$ stands for all analytic trigonometric polynomials of the form $$P(z) = \sum_{k=1}^{d} c_k z^k,\, \quad z~\in~\mathbb{T}\,,$$ so polynomials of degree $\leq d$ without a constant term $c_0$
(following our notation from Section~\ref{basicsbasics}).  Then
 Rudin~\cite{rudin1959some} and
Shapiro \cite{shapiro1952extremal} 
(see also \cite[Proposition~9.7]{defant2019libro}) proved that
\begin{equation}\label{exo1}
   \frac{1}{\sqrt{2}} \sqrt{d}\leq \boldsymbol{\chimon} \big(\text{Trig}_{ \{k\colon 1 \leq k\leq d\}}(\mathbb{T})\big) \leq \sqrt{d}\,.
\end{equation}
If we allow  constant terms, then the best known estimate is
\begin{equation}\label{exo2}
   \sqrt{d}  - O( \log d)^{\frac{2}{3}+ \varepsilon} \leq \boldsymbol{\chimon} \big(\text{Trig}_{ \{k\colon 0 \leq k\leq d\}}(\mathbb{T})\big)
\leq \sqrt{d}\,.
\end{equation}
This  is a~deep fact  proved by Bombieri and Bourgain \cite{bombieri2004remark},
and it shows that at least  from the technical point of view
a~seemingly small perturbation of the index set may change the situation drastically.

Let us compare these results with what  we in the preceding sections proved for projection constants:
\[
\boldsymbol{\lambda}\big(\text{Trig}_{ \{k\colon 1 \leq k\leq d\}}(\mathbb{T})\big) \sim
\boldsymbol{\lambda}\big(\text{Trig}_{ \{k\colon 0 \leq k\leq d\}}(\mathbb{T})\big) \sim 1 + \log d\,;
\]
this is a consequence of  Corollary~\ref{manydimensionsB} and  Corollary~\ref{LoKha}.  

If we turn to multivariate polynomials, in particular to  polynomials on the  $n$-dimensional torus $\mathbb{T}^n$,
the situation gets more complicated. Understanding $\boldsymbol{\chimon} (\text{Trig}_{\Lambda(m,n)}(\mathbb{T}^n))$,
so the unconditional basis constant of the monomial basis in the Banach space $\text{Trig}_{\Lambda(m,n)}(\mathbb{T}^n)$ of all $m$-homogeneous
analytic polynomials,  is rather involved since in this case we face a function in the degree of homogeneity $m$ and in the dimension $n$.

In \cite{defant2011bohnenblust} (improving earlier results from \cite{defant2006logarithmic} and \cite{defant2003bohr} )
it was proved that
\[
\boldsymbol{\chimon}\left( \text{Trig}_{\Lambda(m,n)}(\mathbb{T}^n)\right) \sim_{C^m} |\Lambda(m-1,n)|^{1/2},
\]
and from Corollary~\ref{manydimensionsB} we know that
\[
\boldsymbol{\lambda}\left( \text{Trig}_{\Lambda(m,n)}(\mathbb{T}^n)\right) \sim_{C^m} |\Lambda(m,n)|^{1/2}.
\]
Using \eqref{maxmod}, these results may be reformulated in terms of $\mathcal{P}_{m}(\ell_\infty^n)$ instead of
$\text{Trig}_{\Lambda(m,n)}(\mathbb{T}^n)$.

 Looking at the scale   $X_n = \ell_r^n, 1 \leq r \leq \infty$ instead of $X_n=\ell_\infty^n$ only, again increases the
 technical requirements. Extending the previous two results for $\text{Trig}_{\Lambda(m,n)}(\mathbb{T}^n)$, it was shown in
\cite[Theorem~1.2 and Proposition~4.4]{defant2011bohr} (see also \cite[Corollary~19.8]{defant2019libro}) that
\begin{equation} \label{leo}
\boldsymbol{\chimon}\big(\mathcal{P}_{m+1}(\ell_r^n)\big) \sim_{C^m} \boldsymbol{\lambda}\big(\mathcal{P}_m(\ell_r^n)\big)
\sim_{C^m} |\Lambda(m,n)|^{\min\big\{\frac{1}{r'},\frac{1}{2}\big\}}\,;
\end{equation}
for an extension to more general index sets than $\Lambda(m,n)$  see Theorem~\ref{start-poly2}.

The main intention of this chapter is to involve  techniques from local Banach space theory (Section~\ref{Unconditional basis vs projection constant}),
probability theory (Section~\ref{Probabilistic estimates}), and the theory of 
convexity and  concavity in  Banach function lattices (Section~\ref{conv/conc}) 
  to enlarge the scope of the preceding results
considerably.

To see a very first sample we mention Theorem~\ref{main3}, which is one of our major tools: For any index sets $J \subset \Lambda(m,n)$ and any Banach sequence lattice  $X_n = (\mathbb{C}^n,\|\cdot\|)$,  we have that
\begin{equation}\label{strategy}
  \frac{1}{e 2^{m+1} \, \|\mathbf{Q}_{\Lambda(m+1,n),I}\|}\,\boldsymbol{\chimon}\big( \mathcal{P}_{I}(X_n)\big)
 \,\le \,
 \boldsymbol{\lambda}\big( \mathcal{P}_{J}(X_n)\big)\,,
\end{equation}
where
 $$I = \big\{\alpha \in \Lambda(m+1,n)\colon \exists \beta \in \Lambda(1,n) \textrm{ such that }
\alpha-\beta\in J\big\}\,,$$
  and $$\mathbf{Q}_{\Lambda(m+1,n),I}: \mathcal{P}_{\Lambda(m+1,n)}(X_n) \to \mathcal{P}_{I}(X_n)$$
is the canonical projection annihilating coefficients with indices outside of $I$ (see  Section~\ref{Annihilating coefficients}). Note that, in the notation introduced in 
Section~\ref{index-sets}, $I^\flat =J$.
In the special case  $I = \Lambda(m+1,n)$ we have $J = \Lambda(m,n)$, and hence we deduce that
$\boldsymbol{\chimon}\big(\mathcal{P}_{m+1}(\ell_r^n)\big) \prec_{C^m} \boldsymbol{\lambda}\big(\mathcal{P}_m(\ell_r^n)\big)$
which is one of the major steps for the proof of~\eqref{leo}.

\bigskip

\section{Finite degree  vs  homogeneous case}

We  prove  in this section some analogues of Theorem~\ref{degree-homo}. 
Given a Banach lattice $X_n = (\mathbb{C}^n, \|\cdot\|) $ and a finite index set
$J \subset \mathbb{N}_0^n$  of degree  $m$, the task of estimating the unconditional basis constant of the monomial basis in 
$\mathcal{P}_{J}(X_n)$ may be translated into estimating the unconditional basis  constant  of the monomial basis for each of its `homogeneous building blocks' $\mathcal{P}_{J_k}(X_n), \, 0 \leq k \leq m$.

\begin{proposition} \label{degreelessm}
Let $X_n = (\mathbb{C}^n, \|\cdot\|) $ be a Banach lattice, and
$J \subset \mathbb{N}_0^n$ a finite index set of degree  $m$. Then
 \[
 \max_{0 \leq k \leq m} \boldsymbol{\chimon}\big(\mathcal{P}_{J_k}(X_n)\big) \,\,\leq \,\,
 \boldsymbol{\chimon}\big(\mathcal{P}_{J}(X_n)\big)
 \,\,\leq  \sum_{k=0}^m \boldsymbol{\chimon}\big(\mathcal{P}_{J_k}(X_n)\big) \leq \,\,(m+1)\,\,
 \max_{0 \leq k \leq m} \boldsymbol{\chimon}\big(\mathcal{P}_{J_k}(X_n)\big)
 \,.
 \]
 \end{proposition}

\begin{proof}
The first  estimate is immediate (by the definitions).
For the second fix some polynomial $P(z) = \sum_{\alpha \in J} c_\alpha z^\alpha \in \mathcal{P}_{J}(X_n)$ and complex signs $(\varepsilon_{\alpha})_{\alpha \in J}$.
Then
\[
\Big\| \sum_{\alpha \in J} \varepsilon_{\alpha} c_\alpha  z^\alpha  \Big\|
=
\Big\|  \sum_{k=0}^m \sum_{\alpha \in J_k}\varepsilon_{\alpha} c_\alpha  z^\alpha  \Big\|
\leq
\sum_{k=0}^m \Big\|\sum_{\alpha \in J_k}\varepsilon_{\alpha} c_\alpha   z^\alpha  \Big\|
    \leq
  \sum_{k=0}^m \boldsymbol{\chimon}\big(\mathcal{P}_{J_k}(X_n)\big) \Big\|\sum_{\alpha \in J_k}c_\alpha  z^\alpha  \Big\|\,,
\]
and  hence the conclusion follows using Cauchy's inequality (see Proposition~\ref{Cauchy} or \cite[Proposition 15.33]{defant2019libro}).
\end{proof}

The following proposition compares the unconditional basis constant of the monomial basis of spaces of homogeneous polynomials of different degrees. See Proposition~\ref{Cauchy} for an analog result for projection constants. 

\begin{proposition}\label{comparision unconditional basis}
 Let $X_n = (\mathbb{C}^n, \|\cdot\|) $ be a Banach lattice. For each $1 \leq k \leq m-1$, we have
 \[
 \boldsymbol{\chimon}\big(\mathcal{P}_{k}(X_n)\big)
 \,\,\leq \frac{m^m}{k^k (m-k)^{m-k}} \,\,
 \boldsymbol{\chimon}\big(\mathcal{P}_{ m}(X_n)\big)
 \,.
 \]
In particular, \[
 \boldsymbol{\chimon}\big(\mathcal{P}_{k}(X_n)\big)
 \,\,\leq 2^m \,\,
 \boldsymbol{\chimon}\big(\mathcal{P}_{ m}(X_n)\big).
 \]
\end{proposition}

\begin{proof}
Fix polynomial $P(z) = \sum_{\alpha \in \Lambda(k,n)} c_\alpha z^\alpha \in \mathcal{P}_{k}(X_n)$ and signs $(\varepsilon_{\alpha})_{\alpha \in \Lambda(k,n)}$.
Then
\begin{align*}
\Big\| \sum_{\alpha \in \Lambda(k,n)} \varepsilon_{\alpha} c_\alpha  z^\alpha  \Big\|
& = \Big\| \sum_{\alpha \in \Lambda(k,n)} \varepsilon_{\alpha} c_\alpha  z^\alpha  \Big\| \cdot \frac{\big\| z_1^{m-k} \big\|}{\big\| z_1^{m-k} \big\|} \leq \frac{m^m}{k^k (m-k)^{m-k}} \Big\| \sum_{\alpha \in \Lambda(k,n)} \varepsilon_{\alpha} c_\alpha  z^\alpha z_1^{m-k} \Big\| \frac{1}{\big\| z_1^{m-k} \big\|} \\
& \leq \frac{m^m}{k^k (m-k)^{m-k}}   \boldsymbol{\chimon}\big(\mathcal{P}_{m}(X_n) \big) \Big\| \sum_{\alpha \in \Lambda(k,n)}  c_\alpha  z^\alpha z_1^{m-k} \Big\| \frac{1}{\Big\| z_1^{m-k} \Big\|} \\ & \leq \frac{m^m}{k^k (m-k)^{m-k}}   \boldsymbol{\chimon}\big(\mathcal{P}_{m}(X_n) \big) \Big\| \sum_{\alpha \in \Lambda(k,n)}  c_\alpha  z^\alpha \Big\|,
\end{align*}
where in the first inequality we  use  \cite[Theorem 3]{benitez1998lower}, a classical result  comparing  norms of products of polynomials.
Then, by definition of the constant $\boldsymbol{\chimon}\big(\mathcal{P}_{k}(X_n) \big)$, the result follows.
The general bound $2^m$ follows by a standard calculus argument: the function $k\mapsto \frac{m^m}{k^k (m-k)^{m-k}}$ attains its maximum at $k=\frac{m}{2}$, if we consider its domain in the real numbers between~$1$ and $m-1$,  and its value is $2^m$. 
\end{proof}

\smallskip

As a consequence of the previous two propositions we get the following corollary.

\smallskip

\begin{corollary} \label{degreelessmB}
For each $m \in \mathbb{N}$ and any Banach lattice $X_n = (\mathbb{C}^n, \|\cdot\|)$, we have
 \[
 \boldsymbol{\chimon}\big(\mathcal{P}_{m}(X_n)\big) \,\,\leq \,\,
 \boldsymbol{\chimon}\big(\mathcal{P}_{\leq m}(X_n)\big)
 \,\,\leq \,\,(m+1)\,  2^m \,\,
 \boldsymbol{\chimon}\big(\mathcal{P}_{ m}(X_n)\big)
 \,.
 \]
 \end{corollary}

\smallskip

\section{Unconditional basis vs projection constant}
\label{Unconditional basis vs projection constant}
We prove  that for spaces $ \mathcal{P}_{J}(X_n)$ the projection constant and the unconditional basis constant of the monomial basis are intimately related.

Recall 
from Section~\ref{index-sets}
the notion of  reduced index sets $J^\flat$,
and from Section~ \ref{Annihilating coefficients} the  definition 
of the projection
$
\mathbf{Q}_{J,I}:\mathcal{P}_J(X_n) \to \mathcal{P}_I(X_n)\,
$
annihilating coefficients,
whenever  $I \subset J$.

\begin{theorem} \label{main3}
Let $X_n = (\mathbb{C}^n,\|\cdot\|)$  be a Banach  lattice,
and $J\subset~\mathbb{N}_0^n$ a finite index set
of degree  $m$\,. Then
\[
   \boldsymbol{\chimon}\big( \mathcal{P}_{J}(X_n)\big)
 \,\le \,e (m+1) 2^m \,\max_{1 \leq  k \leq m}\|\mathbf{Q}_{\Lambda(k,n),J_k}\|\,\,
\max_{1 \leq  k \leq m} \boldsymbol{\lambda}\big( \mathcal{P}_{J_k^\flat}(X_n)\big)\,.
\]
In addition, 
\[
   \boldsymbol{\chimon}\big( \mathcal{P}_{J}(X_n)\big)
 \,\le \,e 2^m \, \|\mathbf{Q}_{\Lambda(m,n),J}\|\,\,
 \boldsymbol{\lambda}\big( \mathcal{P}_{J^\flat}(X_n)\big)\,,
\]
whenever  $J \subset \Lambda(m,n)$.
\end{theorem}

The proof  follows  immediately from
independently interesting results presented in the following two subsections, namely    Theorem~\ref{gl-versus-unc} and Theorem~\ref{gl_versus_proj}.

\smallskip
Before we start let us mention an immediate  corollary  of
Theorem~\ref{main3} -- for the first estimate use Proposition~\ref{Cauchy} and for the second 
Theorem~\ref{OrOuSe}.

\smallskip

\begin{corollary} \label{main3A}
Let $m \in \mathbb{N}$ and $X_n = (\mathbb{C}^n,\|\cdot\|)$  be a Banach  lattice. Then
\[
\boldsymbol{\chimon}\big( \mathcal{P}_{ m}(X_n)\big)
 \,\le \,
   \boldsymbol{\chimon}\big( \mathcal{P}_{\leq m}(X_n)\big)
 \,\le \,e (m+1) 2^m  \,\,\max_{1 \leq  k \leq m-1} \boldsymbol{\lambda}\big( \mathcal{P}_{k}(X_n)\big)\,,
\]
and
\[
\boldsymbol{\chimon}\big( \mathcal{T}_{ m}(X_n)\big)
 \,\le \,
   \boldsymbol{\chimon}\big( \mathcal{T}_{\leq m}(X_n)\big)
 \,\le \,e (m+1) 2^m \kappa^m  \,\,\,\,\max_{1 \leq  k \leq m-1} \boldsymbol{\lambda}\big( \mathcal{T}_{k}(X_n)\big)\,.
\]
\end{corollary}

\smallskip

\subsection{Unconditional basis  vs Gordon-Lewis constant} \label{GL}

We start the proof  of
Theorem~\ref{main3} with the following result, which relates the Gordon-Lewis constant of $\mathcal{P}_{J}(X_n)$ with the unconditional basis constant of the monomial basis $(z^\alpha)_{\alpha \in J}$.

\begin{theorem} \label{gl-versus-unc}
Let $X_n = (\mathbb{C}^n,\|\cdot\|)$  be a Banach  lattice,
and $J\subset~\mathbb{N}_0^n$ a finite index set of degree $m$\,.
Then
\[
 {\mbox{gl}}\big( \mathcal{P}_{J}(X_n)\big) \,\leq\,
    \boldsymbol{\chi}\big( \mathcal{P}_{J}(X_n) \big) 
    \,\leq\,
        \boldsymbol{\chimon}\big( \mathcal{P}_{J}(X_n) \big)
           \le 2^m  {\mbox{gl}}\big( \mathcal{P}_{J}(X_n) \big) \,.
\]
\end{theorem}

For $J= \Lambda(m,n)$ the result is proved in \cite[Proposition~3.1]{defant2011bohr}. Our extension to arbitrary index sets $J$
of degree at most $ m$ follows the proof given in \cite[Theorem~21.11]{defant2019libro}. For the sake of completeness, 
we concentrate on elaborating those steps which have to be improved.

Observe  that the first  estimate in Theorem~\ref{gl-versus-unc} is immediate  from the Gordon-Lewis inequality as formulated in  \eqref{gl-inequality}, and the second one is trivial.
The proof of the third estimate is  more involved and needs preparation. The first  lemma we need is taken from \cite[Proposition 21.14]{defant2019libro} (its roots have to  be traced back to  the works \cite{pisier1978some} and \cite{schutt1978projection}).

\begin{lemma} \label{tool2}
Let $X_n$  be an $n$-dimensional Banach space with a~basis  $(x_k)_{k=1}^n$, and  denote the coefficient functionals of this  basis by $(x_k^\ast)$. Suppose that there exist constants $K_1,K_2\ge 1 $ such that for every choice of $\lambda, \mu \in \mathbb{C}^n$ the two diagonal 
operators
\begin{align*}
 D_\lambda: & X_n \rightarrow \ell_2^n
\hspace{2cm}
D_\mu: X_n^{*} \rightarrow \ell_2^n
\\&
x_k \mapsto \lambda_k e_k
\hspace{2.3cm}
x^{*}_k \mapsto \mu_k e_k
\end{align*}
satisfy
\[
\pi_1(D_\lambda)  \le K_1 \Big\| \sum_{k=1}^n \lambda_k x_k^{*}  \Big\|_{X_n^{*}}
\,\,\,\,
\text{ and  }
\,\,\,\,
\pi_1(D_\mu)  \le K_2 \Big\| \sum_{k=1}^n \mu_k x_k  \Big\|_{X_n}\,.
\]
Then
\[
 \boldsymbol{\chi}\big((x_k),X_n\big) \le K_1  K_2  \,{\mbox{gl}}(X_n)\,.
\]
\end{lemma}

The second  crucial tool for the proof of Theorem~\ref{gl-versus-unc} is as follows: 
For any index set   $J\subset~\mathbb{N}_0^n$  of degree  $m$
\begin{equation}\label{toolA}
   \pi_1 \big( \id: \mathcal{P}_{J}(\ell_\infty^n ) \longrightarrow \ell_2(J)) \big)  \leq \sqrt{2}^m\,,
\end{equation}
where
$\,\,\,
\id \big(\sum_{ \alpha \in J} c_\alpha z^\alpha\big)  =
\big( c_\alpha \big)_{\alpha \in J}\,.$

The special case $J =\Lambda(m,n)$  was proved in 
\cite[Lemma~3.6]{defant2011bohr} (see also \cite[Proposition 21.15]{defant2019libro}).
As there, the proof under  our less restrictive assumption on $J$ is an easy consequence of the definition of the 
$1$-summing norm $\pi_1$ and the fact (see e.g \cite[Proposition 8.10]{defant2019libro}) that
for every polynomial $P \in \mathcal{P}_{\leq m}(\mathbb{C}^n)$ one has 
\begin{equation}\label{weissler:analyitic}
\Big( \int_{\mathbb{T}^n} |P(z)|^2 dz \Big)^{\frac{1}{2}} \leq \sqrt{2}^m  \int_{\mathbb{T}^n} |P(z)|\,dz\,.
\end{equation}

\smallskip

\begin{proof}[Proof of Theorem~\ref{gl-versus-unc}]
As explained we may concentrate on the third  estimate.
 The coefficient functionals of the canonical basis $(e_k)$ of $\mathbb{C}^n$ are denoted by $(e_k^\ast)$. Further, we write
 $$
 f_\alpha: \mathcal{P}_{J}(\mathbb{C}^n) \rightarrow \mathbb{C}, \,\,\,\alpha \in J
 $$  
 for the coefficient functionals of the  monomial basis  $(e^{*}_{\alpha})_{\alpha \in J}$  of $\mathcal{P}_{J}(\mathbb{C}^n)$
 defined  by 
 $$
 e^{*}_{\alpha}(z) =   z^\alpha,\,\,\,z\in X_n\,.
 $$ They
 form the orthogonal  basis of the (algebraic) dual $\mathcal{P}_{m}(\mathbb{C}^n)^{*}$, in  the sense that 
 $$\langle e^{*}_{\alpha}, f_\beta\rangle_{\mathcal{P}_{J}(\mathbb{C}^n),\mathcal{P}_{J}(\mathbb{C}^n)^{*}} = \delta_{\alpha,\beta}\,.
 $$
Given  two sequences  $\lambda= (\lambda_{\alpha})_{\alpha \in J}$ and $\mu=(\mu_{\alpha})_{\alpha \in J}$ of scalars,  we consider the two  diagonal operators
\begin{align*}
&
D_\lambda: \mathcal{P}_{J}(X_n)  \longrightarrow \mathcal{P}_{J}(\ell_\infty^n)\,,
\,\, D_\lambda(e^{*}_{\alpha}) =  \lambda_\alpha e^{*}_{\alpha}
\\&
D_\mu: \mathcal{P}_{J}(X_n)^{*}  \longrightarrow \mathcal{P}_{J}( \ell_\infty^n)\,,
\,\, D_\mu(f_\alpha) =  \mu_\alpha e^{*}_{\alpha}\,,
\end{align*}
and show that
\begin{gather}
\|D_\lambda \| \leq \Big\| \sum_{\alpha \in J} \lambda_\alpha f_\alpha \Big\|_{\mathcal{P}_{J}(X_n )^{*}} \label{tonelli} \\
\|D_\mu \| \leq \Big\| \sum_{\alpha \in J} \mu_\alpha e^{*}_{\alpha} \Big\|_{\mathcal{P}_{J}(X_n ) }  \,. \label{hobson}
\end{gather}
If we combine these estimates  with equation \eqref{toolA}, then Lemma~\ref{tool2} gives the conclusion.
Let us prove~\eqref{tonelli}. Note first that for $z \in B_{\ell_\infty^n}$ we have that the norm of the
diagonal operator 
$$
D_z: X_n \rightarrow X_n,\,\,\, e_k \mapsto z_k e_k
$$ is less or equal than one  (the $e_k$'s form a 1-unconditional basis), 
and moreover for every $\alpha \in J$
\[
e^{*}_{\alpha} \circ D_z = z^\alpha e^{*}_{\alpha}\,.
\]
Then for each  $z \in B_{\ell_\infty^n}$ we get
\begin{align*}
\Big| \Big[D_\lambda \big( \sum_{J}  c_\alpha  e^{*}_{\alpha} \big)\Big] (z) \Big| 
& = \Big| \sum_{J} \lambda_\alpha  c_\alpha  e^{*}_{\alpha}(z)\Big| = \Big| \sum_{J} \lambda_\alpha  c_\alpha  z^\alpha\Big| \\
& = \Big|  \langle \sum_{J}  c_\alpha z^\alpha e^{*}_{\alpha},  \sum_{J} \lambda_\alpha f_\alpha\rangle _{\mathcal{P}_{J}(X_n),\mathcal{P}_{J}(X_n)^{*}} \Big|\\
& = \Big|  \langle  \sum_{J}  c_\alpha e^{*}_{\alpha} \circ D_z ,  \sum_{J} \lambda_\alpha f_\alpha\rangle _{\mathcal{P}_{J}(X_n),\mathcal{P}_{J}(X_n)^{*}} \Big|
\\&
\leq
\Big\| \sum_{J} c_\alpha e^{*}_{\alpha} \circ D_z     \Big\|_{\mathcal{P}_{J}(X_n)}  \Big\| \sum_{J} \lambda_\alpha f_\alpha   \Big\|_{\mathcal{P}_{J}(X_n)^{*}}
\\&
\leq
\Big\| \sum_{J} c_\alpha e^{*}_{\alpha}    \Big\|_{\mathcal{P}_{J}(X_n)}  \Big\| \sum_{J} \lambda_\alpha f_\alpha   \Big\|_{\mathcal{P}_{J}(X_n)^{*}}\,.
\end{align*}
Obviously, this leads to the estimate from~\eqref{tonelli}. To see~\eqref{hobson} we repeat the argument:
\begin{align*}
\Big| \Big[D_\mu \big( \sum_{J}  c_\alpha  f_\alpha \big)\Big] (z)   \Big|
&
= \Big| \sum_{J} \mu_\alpha  c_\alpha  e^{*}_{\alpha} (z)\Big|
= \Big| \sum_{J} \mu_\alpha  c_\alpha  z^\alpha \Big|
\\&
=
\Big|  \langle  \sum_{J}  c_\alpha f_\alpha ,  \sum_{J} \mu_\alpha z^\alpha e^{*}_{\alpha} \rangle _{\mathcal{P}_{J}(X_n)^{*},\mathcal{P}_{J}(X_n)} \Big|
\\&
=
\Big|  \langle  \sum_{J}  c_\alpha f_\alpha,  \sum_{J} \mu_\alpha e^{*}_{\alpha} \circ D_z\rangle _{\mathcal{P}_{J}(X_n)^{*},\mathcal{P}_{J}(X_n)} \Big|
\\&
\leq
\Big\| \sum_{J} c_\alpha f_\alpha     \Big\|_{\mathcal{P}_{m}(X_n)^{*}}  \Big\| \sum_{J} \mu_\alpha  e^{*}_{\alpha} \circ D_z   \Big\|_{\mathcal{P}_{J}(X_n)}
\\&
\leq
\Big\| \sum_{J} c_\alpha f_\alpha    \Big\|_{\mathcal{P}_{J}(X_n)^{*}}  \Big\| \sum_{J} \mu_\alpha  e^{*}_{\alpha}   \Big\|_{\mathcal{P}_{J}(X_n)}\,.
\end{align*}
This completes the proof.
\end{proof}

\smallskip

\smallskip

As a very first application of the preceding results we obtain a basis-free version of
Proposition \ref{comparision unconditional basis}. 
It should be noted that the constant that appears in Proposition \ref{comparision unconditional basis} is better than the one we obtain here.

\begin{proposition} \label{ba-free}
Let $X_n = (\mathbb{C}^n, \|\cdot\|) $ be a Banach lattice. For each $1 \leq k \leq m$, we have
\[
\boldsymbol{\chi}\big(\mathcal{P}_{k}(X_n)\big)\,\,\le 2^m   e^{m(2m+1)}\,\,\boldsymbol{\chi}\big(\mathcal{P}_{m}(X_n)\big)\,.
\]
\end{proposition}

\begin{proof}
Applying Theorem~\ref{gl-versus-unc}, a combination of Lemma~\ref{old} with  \eqref{niceprop}, and  finally the
Gordon-Lewis inequality from \eqref{gl-inequality}, we see that for each $1 \leq k \leq m$
 \begin{align*}
\boldsymbol{\chi}\big( \mathcal{P}_{k}(X_n) \big) & \le 2^k  {\mbox{gl}}\big( \mathcal{P}_{ k}(X_n) \big) \\
& \le 2^m  e^{m(2m+1)} {\mbox{gl}}\big( \mathcal{P}_{ m}(X_n) \big)
\le 2^m  e^{m(2m+1)}
\boldsymbol{\chi}\big( \mathcal{P}_{m}(X_n)\big)\,.
\end{align*}
This concludes the proof.
  \end{proof}
  
As an immediate consequence of
Theorem~\ref{gl-versus-unc}
and Proposition~\ref{ba-free} we obtain also a  basis-free variant  of Corollary~\ref{degreelessmB}. 

\smallskip

\begin{corollary}
For each $m \in \mathbb{N}$ \,and any Banach lattice $X_n = (\mathbb{C}^n, \|\cdot\|)$, we have
\[
\boldsymbol{\chi}\big(\mathcal{P}_{m}(X_n)\big) \,\,\leq \,\,
2^m \boldsymbol{\chi}\big(\mathcal{P}_{\leq m}(X_n)\big) \,\,\leq \,\,(m+1)\,2^{2m}\,e^{m(2m+1)} \,\, \boldsymbol{\chi}\big(\mathcal{P}_{ m}(X_n)\big)\,.
\]
\end{corollary}

\smallskip

\subsection{Gordon-Lewis vs projection constant} \label{Gordon-Lewis vs projection constant}

The following result estimates the Gordon-Lewis constant of $\mathcal{P}_{J}(X_n)$, where $J$ is a finite  index set of degree  
$m$, in terms of  the projection constants of the spaces $\mathcal{P}_{J_k^\flat}(X_n),\, 1 \leq  k \leq m$.

\begin{theorem} \label{gl_versus_proj}
Let $X_n = (\mathbb{C}^n,\|\cdot\|)$  be a Banach  lattice and $J \subset \Lambda(m,n)$. Then 
\[
  {\mbox{gl}}\big( \mathcal{P}_{J}(X_n)\big)
 \,\le \,e  \|\mathbf{Q}_{\Lambda(m,n),J}\|\,\,
 \boldsymbol{\lambda}\big( \mathcal{P}_{J^\flat}(X_n)\big)\,.
\]
Moreover, if $J\subset~\mathbb{N}_0^n$ is an index set of degree  $m$, then 
\[
   {\mbox{gl}}\big( \mathcal{P}_{J}(X_n)\big)
 \le e (m+1) \,\max_{1 \leq  k \leq m}\|\mathbf{Q}_{\Lambda(k,n),J_k}\|\,\,
\max_{1 \leq  k \leq m} \boldsymbol{\lambda}\big( \mathcal{P}_{J_k^\flat}(X_n)\big)\,.
\]
\end{theorem}
\smallskip

Again, the case $J= \Lambda(m,n)$ was  proved earlier in \cite[Proposition~ 4.2]{defant2011bohr}, and an alternative 
detailed presentation of the proof of this homogeneous case was given in \cite[Proposition~22.1]{defant2019libro}. As above, we here only focus 
on the modification of those arguments, which are needed to cover the general situation.

We start with an  elementary  observation  taken from \cite[Lemma 22.2]{defant2019libro}, which
 `up to polarization' covers the case $J=\Lambda(2,n)$ of Theorem~\ref{gl_versus_proj} (take $Y= X^\ast$):
  For every finite dimensional Banach  lattice  $X$, and every finite dimensional Banach space $Y$ one has
  \begin{equation}\label{supo}
    {\mbox{gl}}\big( \mathcal{L}(X,Y)\big) \leq \boldsymbol{\lambda}(Y) \,.
  \end{equation}

\begin{proof}[Proof of Theorem~\ref{gl_versus_proj}]
To see the first statement, we  consider the following commutative diagram:
\begin{equation}\label{diagram}
  \xymatrix
{
  \mathcal{P}_{J}(X_n) \ar[d]^{\text{$U_m$}}
 & \mathcal{P}_{m}(X_n) \ar[l]_{\mathbf{Q}_{\Lambda(m,n),J}}
 \\
\mathcal{L}\big(X_n,\mathcal{P}_{J^\flat}(X_n)\big)
\ar[r]^{I_m}
&
\mathcal{L}\big(X_n,\mathcal{P}_{m-1}(X_n)\big) \ar[u]^{V_m}\,,
}
\end{equation}
where $I_m$ is the canonical inclusion map and
\begin{align*}
&
\big(U_m(P)x\big)(u) := \check{P}( u, \ldots, u
,x)
\,\,\, \text{ for  } \,\,\, P \in  \mathcal{P}_{J}(X_n) \,\,\, \text{ and  } \,\,\, x,u \in X_n\,,
\\[1ex]&
V_m(T)(y) := (Ty)y \,\,\, \text{ for  } \,\,\,  T \in \mathcal{L}\big(X_n,\mathcal{P}_{m-1}(X_n)\big)
\,\,\, \text{ and  } \,\,\, y \in X_n.
\end{align*}
We show that $U_m$, as an operator from $\mathcal{P}_{J}(X_n)$ into $\mathcal{L}\big(X_n, \mathcal{P}_{J^\flat}(X_n)\big)$,
is well-defined. Indeed,  define  
\[
a_\bi(\check{P})= \frac{c_\bj(P)}{|[\bj]|}\,\,\,\, \text{ for $\bj \in \mathcal{J}(m,n)$ and $\bi \in [\bj]$} \,.
\]
Then, given $P \in  \mathcal{P}_{J}(X)$ and $x,u \in X_n$, 
\begin{align*}
\check{P}( u, \ldots, u,x)
&
=\sum_{\bi \in \mathcal{M}(m,n)} a_\bi(\check{P})
u_{i_1}\ldots u_{i_{m-1}} x_{i_m}
\\&
=\sum_{\bi \in \mathcal{M}(m-1,n)} \sum_{\ell=1}^n a_{(\bi,\ell)}(\check{P})
u_\bi x_{\ell}
=
\sum_{\bj \in \mathcal{J}(m-1,n)} \sum_{\bi \in  [\bj]} \,\, \sum_{\ell=1}^n a_{(\bi,\ell)}(\check{P})
u_\bi x_{\ell}
\\&
=
\sum_{\bj \in \mathcal{J}(m-1,n)}  
\sum_{\ell=1}^n
\,\,
\sum_{\bi \in  [\bj]}
 a_{(\bi,\ell)}(\check{P})
 u_\bi x_{\ell}
   =
\sum_{\bj \in \mathcal{J}(m-1,n)}
\sum_{\ell=1}^n
\,\,
\sum_{\bi \in  [\bj]}
 \frac{c_{(\bi,\ell)_\ast}(P)}{|[(\bi,\ell)_\ast]|}
  u_\bi x_{\ell}
   \\&
  =
\sum_{\bj \in \mathcal{J}(m-1,n)}
\sum_{\ell=1}^n
\bigg[
 \frac{c_{(\bj,\ell)_\ast}(P)}{|[(\bj,\ell)_\ast]|}  u_\bj \bigg] |[\bj]|
 x_{\ell}
     =
\sum_{\bj \in \mathcal{J}(m-1,n)}
\bigg[
\sum_{\ell=1}^n
  \frac{c_{(\bj,\ell)_\ast}(P)}{|[(\bj,\ell)_\ast]|} |[\bj]|
  x_{\ell} \bigg]  u_\bj
  \\&
  =
\sum_{\bj \in \mathcal{J}(m-1,n)}
\bigg[\sum_{\substack{1 \leq \ell \leq n \\ (\bj,\ell)_\ast \in J}}
 \frac{c_{(\bj,\ell)_\ast}(P)}{|[(\bj,\ell)_\ast]|} |[\bj]|
  x_{\ell} \bigg]  u_\bj
    =
\sum_{\bj \in J^\flat}
\bigg[\sum_{\substack{1 \leq \ell \leq n \\ (\bj,\ell)_\ast \in J}}
 \frac{c_{(\bj,\ell)_\ast}(P)}{|[(\bj\grave{},\ell)_\ast]|} |[\bj]|
  x_{\ell} \bigg]  u_\bj\,,
  \end{align*}
  which shows that $U_m(P)x \in \mathcal{P}_{ J^\flat}(X_n)$ for every $x \in X_n$.
By the Harris polarization formula (see, e.g., \cite[Proposition 2.34]{defant2019libro}) we have $\|U_m\| \leq e$, and moreover trivially  $\|V_m\| \leq~1$. Hence by the ideal properties of the involved ideal norms we see that
\[
{\mbox{gl}}\big( \mathcal{P}_{J}(X_n)\big) \leq e \, \|\mathbf{Q}_{\Lambda(m,n),J}\| \,{\mbox{gl}}\big(\mathcal{L}\big(X_n,\mathcal{P}_{J^\flat}(X_n)\big)\big) \leq  e \, \|\mathbf{Q}_{\Lambda(m,n),J}\| \, \boldsymbol{\lambda}\big(\mathcal{P}_{J^\flat}(X_n)\big)\,,
\]
where for the last estimate we use  \eqref{supo}. This proves  the first claim.

For the second claim we have   to handle an  index set $J$ of degree $m$, and consider the following commutative diagram
\begin{equation} \label{greatpic}
    \xymatrix
{
  \mathcal{P}_{J}(X_n)
  \ar[d]^{\text{$\mathbf{O}\oplus\bigoplus\mathbf{Q}_{J,J_k}$}}
  \ar[r]^{\text{$ \id_{ \mathcal{P}_{J}(X_n)}$}}
 & \mathcal{P}_{J}(X_n)
 \\
  \mathbb{C} \oplus_\infty \bigoplus_\infty \mathcal{P}_{J_k}(X_n)
  \ar[d]^{\text{$ \id_\mathbb{C}\oplus
  \bigoplus U_k$}}
  \ar[r]^{\text{$\id_\mathbb{C}\oplus\bigoplus \id_{ \mathcal{P}_{J_k}(X_n)}$}}
 &
 \mathbb{C} \oplus_1\bigoplus_1 \mathcal{P}_{J_k}(X_n) \ar[u]^{\sum }  & \mathbb{C} \oplus_1\bigoplus_1 \mathcal{P}_{k}(X_n) \ar[l]_{\text{$\id_\mathbb{C}\oplus\bigoplus\mathbf{Q}_{\Lambda(k,n),J_k}$}}
 \\
 \mathbb{C} \oplus_\infty\bigoplus_\infty  \mathcal{L}\big(X_n,\mathcal{P}_{J_k^\flat}(X_n)\big)
\ar[r]^{\id_\mathbb{C}\oplus\bigoplus I_k}
&
 \mathbb{C} \oplus_\infty\bigoplus_\infty \mathcal{L}\big(X_n,\mathcal{P}_{k-1}(X_n)\big) \ar[r]^{\Phi}  & \mathbb{C} \oplus_1\bigoplus_1 \mathcal{L}\big(X_n,\mathcal{P}_{k-1}(X_n)\big)
 \ar[u]^{\text{$\id_\mathbb{C}\oplus\bigoplus V_k$}}\,.
}
\end{equation}
Let us explain our notation in this diagram:
Here  $U_k$, $V_k$ and $I_k$ for $1 \leq k \leq m$ are the operators from \eqref{diagram}. If $P = a_0 +\sum_{k=1}^m P_k$
is the Taylor decomposition of $P \in \mathcal{P}_{J}(X_n)$, then
$\mathbf{O}(P) = a_0$, and hence  $\big(\mathbf{O}\oplus\bigoplus\mathbf{Q}_{J,J_k}\big)(P) = \big(a_0, (P_k)_{k=1}^m\big)$.
Additionally,
$\sum$ is the mapping which assigns to every $\big(a_0, (P_k)_{k=1}^m\big)$ the polynomial
$P = a_0 +\sum_{k=1}^m P_k$,  and $\Phi$ stands for the identity map -- whereas the notation for the rest of the maps is self-explaining.  Obviously, this
gives that
\[
   {\mbox{gl}}\big( \mathcal{P}_{J}(X_n)\big)
 \le e (m+1) \,\max_{1 \leq  k \leq m}\|\mathbf{Q}_{\Lambda(k,n),J_k}\|\,\,
\max_{1 \leq  k \leq m} {\mbox{gl}}\big( \bigoplus_\infty  \mathcal{L}\big(X_n,\mathcal{P}_{J_k^\flat}(X_n)\big)\big)\,,
\]
and it remains to prove the following claim
\begin{equation}\label{claimA}
 {\mbox{gl}}\big( \bigoplus_\infty  \mathcal{L}\big(X_n,\mathcal{P}_{J_k^\flat}(X_n)\big)\big)
 \leq
 \max_{1 \leq  k \leq m} \boldsymbol{\lambda}\big( \mathcal{P}_{J_k^\flat}(X_n)\big)\,.
\end{equation}
Indeed, using standard properties of $\varepsilon$- and $\pi$-tensor products, we have
\begin{align*}
  \bigoplus_\infty  \mathcal{L}\big(X_n,\mathcal{P}_{J_k^\flat}(X_n)\big)
  &
  \hookrightarrow
  \bigoplus_\infty  \mathcal{L}\big(X_n,\bigoplus_\infty \mathcal{P}_{J_k^\flat}(X_n)\big)
  \\&
  = \ell_\infty^m \otimes_\varepsilon \big[ X_n^\ast \otimes_\varepsilon \bigoplus_\infty \mathcal{P}_{J_k^\flat}(X_n)\big]
  \\&
   = \big[\ell_\infty^m \otimes_\varepsilon  X_n^\ast\big] \otimes_\varepsilon \bigoplus_\infty \mathcal{P}_{J_k^\flat}(X_n)
    \\&
    =
    \big(\ell_1^m \otimes_\pi X_n\big)^\ast   \otimes_\varepsilon \bigoplus_\infty \mathcal{P}_{J_k^\flat}(X_n)
        =
    \mathcal{L}\big(  \ell_1^m(X_n), \bigoplus_\infty \mathcal{P}_{J_k^\flat}(X_n)\big)\,,
\end{align*}
where the first space in fact is $1$-complemented in the second one, and all other identifications are isometries. Then we deduce from \eqref{supo} that
\[
{\mbox{gl}}\big( \bigoplus_\infty  \mathcal{L}\big(X_n,\mathcal{P}_{J_k^\flat}(X_n)\big)\big)
 \leq
 {\mbox{gl}}\big(\mathcal{L}\big(  \ell_1^m(X_n), \bigoplus_\infty \mathcal{P}_{J_k^\flat}(X_n)\big)\big)
 \leq
 \boldsymbol{\lambda} \big(\bigoplus_\infty \mathcal{P}_{J_k^\flat}(X_n) \big)\,.
\]
Since obviously
\[
\boldsymbol{\lambda} \big(\bigoplus_\infty \mathcal{P}_{J_k^\flat}(X_n) \big)
=
\gamma_\infty \big(\id_{\bigoplus_\infty \mathcal{P}_{J_k^\flat}(X_n)} \big)
\leq \max_{1 \leq  k \leq m} \gamma_\infty \big(\id_{\mathcal{P}_{J_k^\flat}(X_n)} \big)
= \max_{1 \leq  k \leq m} \boldsymbol{\lambda} \big(\mathcal{P}_{J_k^\flat}(X_n) \big)\,,
\]
the proof is complete.
\end{proof}

\section{Probabilistic estimates}
\label{Probabilistic estimates}
For later use, we  isolate a couple of  probabilistic lower bounds for the unconditional basis constant of the monomial basis
$(z^\alpha)_{\alpha\in J}$ in spaces $\mathcal{P}_{J}(X_n)$ of multivariate polynomials supported on $J$,  under certain restrictions on the underlying Banach sequence
lattice $X_n$ and index set $J \subset \mathbb{N}_0^{n}$.

We closely follow methods from
\cite{bayart2012maximum}, 
\cite{boas2000majorant},
\cite{defant2004maximum}, 
\cite{defant2020subgaussian},
 and \cite{mastylo2017kahane},
 which were mainly invented to  cover the homogeneous  case $J = \Lambda(m,n)$ .

All  results presented are  consequences  of the following two theorems due to Bayart \cite{bayart2012maximum}. The proofs of both results use
quite different methods -- the 'covering method' and the 'entropy method'. Together with all preliminary tools needed,  these proofs are elaborated in \cite[Corollary~17.5, Corollary~17.22]{defant2019libro}.

\begin{theorem} \label{Anquetil alfa}
Given  $1 \leq r \leq 2$, there is a constant $C >0$ such that for  each $m \geq 2$, for every Banach space $X_n = (\mathbb{C}^n,\|\cdot\|)$,
and for every choice of scalars $(c_{\alpha})_{\alpha \in \Lambda(m,n)}$ there exists a choice of signs $\varepsilon_{\alpha} = \pm1, \, \alpha \in \Lambda(m,n)$ such that
\[
\sup_{z \in B_{X_n}} \Big\vert \sum_{\alpha \in \Lambda(m,n)} \varepsilon_{\alpha} c_{\alpha} z^{\alpha} \Big\vert
\leq C
(n \log m)^{\frac{1}{r'}}  \sup_{\alpha}  \Big(\vert c_{\alpha} \vert \Big( \frac{\alpha !}{m!}\Big)^{\frac{1}{r}} \Big)  \sup_{z \in B_{X_n}} \Big( \sum_{k=1}^{n} \vert z_{k} \vert^{r} \Big)^{\frac{m}{r}} \,.
\]
\end{theorem}

\smallskip

\begin{theorem} \label{Hinault alfa}
Given  $1 \leq r \leq 2$, there is a constant $C >0$ such that for each  $m \geq 2$, for every Banach space $X_n = (\mathbb{C}^n,\|\cdot\|)$,
and for every choice of scalars $(c_{\alpha})_{\alpha \in \Lambda(m,n)}$ there exists a choice of signs $\varepsilon_{\alpha} = \pm1, \, \alpha \in \Lambda(m,n)$ such that
\begin{align*}
\sup_{z \in B_{X_n}} \Big\vert \sum_{\alpha \in \Lambda(m,n)} \varepsilon_{\alpha} c_{\alpha} z^{\alpha} \Big\vert
\leq C m (\log n)^{1+\frac{1}{r'}}   \sup_{\alpha}  \bigg( \vert c_{\alpha} \vert \Big( \frac{\alpha !}{m!} \Big)^{\frac{1}{r}} \bigg)
  \sup_{z \in B_{X_n}} \Big( \sum_{k=1}^{n} \vert z_{k} \vert^{r} \Big)^{\frac{m-1}{r}}  \sup_{z \in B_{X_n}} \sum_{k=1}^{n} \vert z_{k} \vert \,.
\end{align*}
\end{theorem}

\smallskip

\subsection{Consequences of the covering method}
We present two lower bounds for $\boldsymbol{\chimon}\big(\mathcal{P}_J(X_n)\big)$, where $X_n$ is the $n$th section of a Banach sequence lattice $X$ and $J \subset \mathbb{N}_0^{(\mathbb{N})}$
contains all $m$-homogeneous, tetrahedral multi indices  of length $m\leq n$. The proofs of both estimates are based on
Theorem~\ref{Anquetil alfa} and  the following lemma.

\begin{lemma} \label{innichenA}
For each  $1 \leq r \leq 2$, there is a constant $C >0$ such that for every Banach sequence lattice $X_n = (\mathbb{C}^n, \|\cdot \|)$
and for each $m \le n$ one has
\[
\frac{1}{\|\id\colon X_n\to \ell_r^n\|^m}\frac{|\Lambda_T(m,n)|}{\varphi_{X_n}(n)^m n^{\frac{1}{r'}} m ^{-\frac{m}{r}}}
\,\,\leq\,\, C e^{\frac{m}{r}} (\log m)^{1/r'}
 \boldsymbol{\chimon}\big(\mathcal{P}_{\Lambda_T(m,n)}(X_n)\big)
 \,.
\]
\end{lemma}

\begin{proof}
Clearly, taking $z = (\varphi_{X_n}(n)^{-1}, \ldots,\varphi_{X_n}(n)^{-1}) \in B_{X_n}$, we get
\[
\frac{|\Lambda_T(m,n)|}{\varphi_{X_n}(n)^m} \leq \sup_{z \in B_{X_n}} \Big| \sum_{\alpha \in \Lambda_T(m,n)} z^\alpha\Big|\,.
\]
Then, it follows from Theorem~\ref{Anquetil alfa} that we  find signs $\varepsilon_\alpha = \pm 1, \alpha \in \Lambda_T(m,n)$
for which for all  $m \leq n$
\begin{align} \label{poly} 
\sup_{z \in B_{X_n}} \Big| \sum_{\alpha \in \Lambda_T(m,n)} \varepsilon_\alpha  z^\alpha\Big|
& \leq C \,\,  (n \log m)^{1/r'} \sup_{\alpha \in \Lambda_T(m,n)} \Big(\frac{\alpha!}{m!}\Big)^{1/r}
\sup_{z \in B_X} \Big(\sum_{k=1}^n |z_k|^r\Big)^{m/r} \\
& \leq C \,\, (\log m)^{1/r'} n ^{1/r'}m!^{-1/r} \sup_{z \in B_X} \Big(\sum_{k=1}^n |z_k|^r\Big)^{m/r}   \nonumber\\
& =  C \,\, (\log m)^{1/r'} n ^{1/r'}m!^{-1/r} \|\id \colon X_n\to \ell_r^n\|^m\,,  \nonumber
\end{align}
where $C >0$ is a constant only depending on $r$\,. Using  that $m^m \leq e^m m!$, we finally arrive at

\begin{align*}
\frac{|\Lambda_T(m,n)|}{\varphi_{X_n}(n)^m} & \leq
\sup_{z \in B_{X_n}} \Big| \sum_{\alpha \in  \Lambda_T(m,n)} \varepsilon_\alpha \varepsilon_\alpha z^\alpha \Big| \\
& \leq  \boldsymbol{\chimon} \big(\mathcal{P}_{ \Lambda_T(m,n)}\big)
\sup_{z \in B_{X_n}} \Big| \sum_{\alpha \in \Lambda_T(m,n)} \varepsilon_\alpha z^\alpha\Big| \\
& \leq \boldsymbol{\chimon} \big(\mathcal{P}_{\Lambda_T(m,n)}\big)
\,C\,(\log m)^{1/r'}n^{1/r'} (m!)^{-1/r}\|\id : X_n\to \ell_r^n\|^m \\
& \leq \boldsymbol{\chimon} \big(\mathcal{P}_{ \Lambda_T(m,n)}\big)
\,C\, e^{\frac{m}{r}}(\log m)^{1/r'} n^{1/r'} (m!)^{-m/r} \|\id \colon X_n \to \ell_r^n\|^{m} \,.
\end{align*}
\end{proof}

\smallskip

In the following we are going to need two consequences.

\smallskip

\begin{proposition} \label{innichen1}
For  $1 \leq r \leq 2$, let $X$ be a   Banach sequence lattice  such that
$\varphi_X(n)\prec n^{1/r}$ (up to a uniform constant
depending only on $X$).
Then for every index set $J \subset \mathbb{N}_0^{(\mathbb{N})}$ for which  $\Lambda_T(m,n) \subset J  $
for some $m\le n$, we have
\[
\frac{1}{\|\mathrm{id}:X_n\to \ell_r^n\|^m} \Big( \frac{n}{m} \Big)^{\frac{m-1}{r'}}
\prec_{C^m}
\boldsymbol{\chimon}\big(\mathcal{P}_J(X_n)
\big)\,,
\]
where $C \ge 1$ is a constant depending only on $X$.
\end{proposition}

\begin{proof}
We only have to consider the case $J = \Lambda_T(m,n)$. Note first that
\begin{equation} \label{binom}
\Big(\frac{n}{m}\Big)^m \leq \binom{n}{m} = |\Lambda_T(m,n)|\,.
\end{equation}
Combining this  with the assumption  $\varphi_X(n) \leq C n^{1/r}$, we get
\[
\frac{|\Lambda_T(m,n)|}{\varphi_{X_n}(n)^m n^{\frac{1}{r'}} m ^{-\frac{m}{r}}} \ge
\frac{n^m}{C^m n^{\frac{m}{r}} n^{\frac{1}{r'}}  m^{\frac{m}{r'}}} 
=  \Big( \frac{n}{m} \Big)^{\frac{m-1}{r'}}  \frac{1}{C^m m^\frac{1}{r'}},
\]
and hence  the claim is an immediate consequence of Lemma~\ref{innichenA}.
\end{proof}

\smallskip

\begin{proposition} \label{toblach}
Let $X$ be a  Banach sequence lattice  such that \,$\varphi_{X_n}(n)\,\varphi_{X_{n}'}(n) \prec n$ and 
\begin{equation} \label{M}
\|\id\colon X_n \to \ell_2^{n}\| \prec \frac{1}{\sqrt{n}}\,\|\id\colon X_n \to \ell_1^{n}\|\,,
\end{equation}
up to    constants  depending only on $X$.
Then for every index set $J \subset \mathbb{N}_0^{(\mathbb{N})}$ for which   $\Lambda_T(m,n) \subset~J  $ for some $m\le n$, we have
\[
\Big( \frac{n}{m} \Big)^{\frac{m-1}{2}}
\prec_{C^m}
\boldsymbol{\chimon}\big(\mathcal{P}_{J}(X_n)\big)\,,
\]
where $C \ge 1$ is a constant depending only on $X$.
\end{proposition}

\begin{proof}
Since for $x =\frac{e_1}{\|e_1\|_X}$ one has $\|x\|_X=1$, it follows that 
\[
\gamma:=\frac{1}{{\|e_1\|_X}} \leq \|\id\colon X_n \to \ell_2^{n}\| \prec \frac{1}{\sqrt{n}}\,\|\id\colon X_n \to \ell_1^{n}\|\,.
\]
Using the estimate $\varphi_{X_n}(n)\,\varphi_{X_{n}'}(n) \prec n$, we get 
\[
\gamma \prec \frac{1}{\sqrt{n}}\,\|\id\colon X_n \to \ell_1^{n}\| = \frac{1}{\sqrt{n}}\,\varphi_{X_{n}'}(n)
\prec \frac{\sqrt{n}}{\varphi_{X_n}(n)}\,,
\]
and hence $\varphi_{X_n}(n)\prec \sqrt{n}$. This combined with the above estimate \eqref{binom} yields 
\begin{align*}
\frac{1}{\|\id \colon X_n\to \ell_2^n\|^m}\frac{|\Lambda_T(m,n)|}{\varphi_{X_n}(n)^m n^{\frac{1}{2}} m ^{-\frac{m}{2}}}
& \succ \frac{n^\frac{m}{2}}{\varphi_{X_{n}'}(n)^m} \Big(\frac{n}{m}\Big)^m \frac{1}{n^\frac{1}{2}
\varphi_{X_n}(n)^m  m^{-\frac{m}{2}}} 
\sim \frac{n^{\frac{m-1}{2}}}{m^{\frac{m}{2}}} = \frac{1}{m^{\frac{1}{2}}}\Big( \frac{n}{m} \Big)^{\frac{m-1}{2}}\,,
\end{align*}
and hence  the conclusion follows (as in the preceding proof) from Lemma~\ref{innichenA}.
\end{proof}

\smallskip

\subsection{Consequences of the entropy method}
We need one more estimate -- this time based on Theorem~\ref{Hinault alfa}. See Theorem~\ref{limits} and
Theorem~\ref{limits++} for the applications we aim for.

\begin{proposition} \label{innichen++}
Let $1 \leq r \leq 2$ and $X$ a Banach sequence lattice. Then there is a constant
$C = C(r,X)$ such for all~$m,n$
\[
\frac{1}{Cm^{1+1/2r'} \,  (\log n)^{1+1/r'} e^{-\frac{m}{r'}}m^{\frac{m}{r'}}   }
\left(  \frac{\|\mathrm{id}:X_n\to \ell_1^n\|}{\|\id \colon X_n\to \ell_r^n\|}\right)^{m-1}
\,\leq\,
 \boldsymbol{\chimon}\big(\mathcal{P}_m(X_n)\big)\,.
\]
\end{proposition}

This is a  result  (slightly modified for our purposes) taken from \cite[Proposition~5.2]{bayart2012maximum}, and for the sake of completeness
we add the short proof (see also \cite[Lemma~20.28]{defant2019libro}).

\begin{proof}
By the multinomial formula  for all $m,n$ and all  choices of signs $(\varepsilon_\alpha)_{\alpha \in \Lambda(m,n)}$
\begin{align*}
\bigg( \sup_{\|z\|_{X_n}\leq 1} \sum_{k=1}^n  |z_k| \bigg)^m
 =\sup_{\|z\|_{X_n}\leq 1}  \Big| \sum_{\alpha \in \Lambda(m,n)} \frac{m!}{\alpha!} z^\alpha\Big|
 &=\sup_{\|z\|_{X_n}\leq 1} \Big| \sum_{\alpha \in \Lambda(m,n)}   \frac{m!}{\alpha!} \varepsilon_\alpha \varepsilon_\alpha  z^\alpha\Big| \\
& \leq \boldsymbol{\chimon}\big(\mathcal{P}_{m}(X_n)\big) \sup_{\|z\|_{X_n}\leq 1} \Big| \sum_{\alpha \in \Lambda(m,n)} \varepsilon_\alpha \frac{m!}{\alpha!} z^\alpha\Big|\,.
\end{align*}
  Choosing  signs $\varepsilon_\alpha$, according to Theorem~\ref{Hinault alfa}, we get for all $m,n$
 \begin{align*}
   \bigg( \sup_{\|z\|_{X_n}\leq 1} \sum_{k=1}^n  |z_k| \bigg)^m
   \leq
   C_1 \boldsymbol{\chimon}\big(\mathcal{P}_{m}(X_n)\big) m\, (\log n)^{1+\frac{1}{r'}}   m!^{\frac{1}{r'}}
  \sup_{z \in B_{X_n}} \Big( \sum_{k=1}^{n} \vert z_{k} \vert^{r} \Big)^{\frac{m-1}{r}}  \sup_{z \in B_{X_n}} \sum_{k=1}^{n} \vert z_{k} \vert \,,
 \end{align*}
 where the constant $C_1 = C_1(r)$ only depends on $r$. Since  by Stirling's formula for each positive integer~$m$
 \[
 m! \,\,< \,\,\frac{12}{11}\sqrt{2 \pi} m^{m + \frac{1}{2}} e^{-m}\,,
 \]
 the argument completes.
    \end{proof}

\smallskip
Under an additional symmetry assumption on $X$, we obtain a  tetrahedral variant of the preceding result.
\smallskip

\begin{proposition} \label{innichen}
For every $1 \leq r \leq 2$ and every Banach sequence lattice $X$  such that $\varphi_{X_n}(n) \varphi_{X_{n}'}(n)~\prec~n$,
there is a~constant $C = C(r,X)$ such that for each $m\le n$
\[
\frac{
   \Big( \frac{n}{n-m} \Big)^{n-m} \sqrt{\frac{n}{m(n-m)}}
}{Cm\,  (\log n)^{1+1/r'} e^{\frac{m}{r}}     m^{m/r'}}
\,\,\,\,\left(  \frac{\|\id\colon X_n\to \ell_1^n\|}{\|\id:X_n\to \ell_r^n\|}\right)^{m-1}
\,\leq\,
 \boldsymbol{\chimon}\big(\mathcal{P}_{\Lambda_T(m,n)}(X_n)\big)\,.
\]
\end{proposition}

\begin{proof}
Note first that for all $m,n$
\begin{equation} \label{puchner}
  \frac{|\Lambda_T(m,n)|}{C_1m \,  (\log n)^{1+1/r'} m!^{-1/r}\|\id \colon X_n\to \ell_r^n\|^{m-1}\varphi_{X_n}(n)^m \varphi_{X_{n}'}(n)  }
\,\leq\,
 \boldsymbol{\chimon}\big(\mathcal{P}_{\Lambda_T(m,n)}(X_n)\big)\,,
\end{equation}
where $C_1 = C_1(r)$ only depends on $r$. Indeed, replacing  the probabilistic tool used in \eqref{poly} by Theorem~\ref{Hinault alfa}, the proof
of this estimate is  absolutely parallel to that of Lemma~\ref{innichenA}:  Choose signs $\varepsilon_\alpha = \pm 1, \alpha \in \Lambda_T(m,n)$
such that
\begin{align*} \label{poly2}
\begin{split}
   \sup_{z \in B_{X_n}} \Big| \sum_{\alpha \in \Lambda_T(m,n)} \varepsilon_\alpha  z^\alpha\Big|
\leq
\,C_1 m \,  (\log n)^{1+1/r'} m!^{-1/r}  \|\id\colon X_n\to \ell_r^n\|^{m-1}  \|\id \colon X_n\to \ell_1^n\|\,,
\end{split}
\end{align*}
where $C_1 = C_1(r)$ is the constant from Theorem~\ref{Hinault alfa}. Since $\|\id\colon X_n\to \ell_1^n\| = \varphi_{X_{n}'}(n)$,
we proceed exactly as in the proof of Lemma~\ref{innichenA} to get \eqref{puchner}. Continuing, we in \eqref{puchner} estimate
$|\Lambda_T(m,n)| = \binom{n}{m}$ from below. Using Stirling's formula, 
\begin{equation*}
\sqrt{2 \pi}  n^{n + \frac{1}{2}} e^{-n} \leq n! \leq   \frac{12}{11}\sqrt{2 \pi} n^{n + \frac{1}{2}} e^{-n}, \quad\, n\in \mathbb{N}\,,
\end{equation*}
we conclude that there is a uniform constant $D \ge 1$ such that for each $m \leq n$
\begin{equation*}
  \binom{n}{m}  \ge D  \Big( \frac{n}{m} \Big)^{m}  \Big( \frac{n}{n-m} \Big)^{n-m} \sqrt{\frac{n}{m(n-m)}}\,.
\end{equation*}
Estimating now \eqref{puchner}, we get with $C = C_1/D$ that for all $m \leq n$
 \begin{multline*}
   \frac{|\Lambda_T(m,n)|}{C_1m \,  (\log n)^{1+1/r'}  m!^{-1/r}\|\id\colon X_n\to \ell_r^n\|^{m-1}\varphi_{X_n}(n)^m \varphi_{X_{n}'}(n)  }
 \\
\,\geq\,
\frac{   \Big( \frac{n}{n-m} \Big)^{n-m} \sqrt{\frac{n}{m(n-m)}}}{Cm   (\log n)^{1+1/r'}  m!^{-1/r}}
\, \, \,  \frac{n^m}{m^m\|\id
\colon X_n\to \ell_r^n\|^{m-1}\varphi_{X_n}(n)^m \varphi_{X_{n}'}(n)}\,,
\end{multline*}
and hence using the   symmetry assumption on  $X$
 \begin{multline*}
     \frac{|\Lambda_T(m,n)|}{Cm \,  (\log n)^{1+1/r'} m!^{-1/r}\|\id \colon X_n\to \ell_r^n\|^{m-1}\varphi_{X_n}(n)^m \varphi_{X_{n}'}(n)  }
     \\
\,\geq\,
\frac{   \Big( \frac{n}{n-m} \Big)^{n-m} \sqrt{\frac{n}{m(n-m)}}}{Cm   (\log n)^{1+1/r'} m^m m!^{-1/r}}
\, \, \,  \left(  \frac{\|\id\colon X_n\to \ell_1^n\|}{\|\id \colon X_n\to \ell_r^n\|}\right)^{m-1}\,,
   \end{multline*}
    where now $C \ge 1$ depends on $r$ and $X$. Since $m^m \leq e^m m!$, we have that
   $m!^{-\frac{1}{r}} \leq e^{\frac{m}{r}} m^{-\frac{m}{r}} $, and so the proof completes.
      \end{proof}

\smallskip

\section{Convexity and concavity}
\label{conv/conc}
Given two Banach sequence lattices  $X,Y$ and an index set $J \subset \mathbb{N}_0^{(\mathbb{N})}$, we study how the unconditional basis constants $\boldsymbol{\chimon}(\mathcal{P}_{J}(X_n))$
and $\boldsymbol{\chimon}(\mathcal{P}_{J}(Y_n))$ are related to each other -- provided $X$ and $Y$ satisfy certain convexity and concavity conditions.

For all here  needed notions 
on (quasi-) Banach function lattices over measure spaces we again refer to Section~
\ref{Banach spaces and (Quasi-)Banach lattices}.

Much of what we intend to do
in this section, is   based on  a~deep factorization  theorem of
Lozanovskii \cite{Loz} (see also~\cite{schutt1978projection}). Since we later on  need it again, we prefer to formulate it explicitly.

\begin{theorem} \label{Lozanovskii}
Let $X$ be  Banach function lattice
over $(\Omega, \mathcal{A}, \mu)$ with the Fatou property. Then   $$X\circ X' \equiv L_1(\mu)\,.$$ More precisely,
for every $f\in L_1(\mu)$, there exist $g\in X$ and $h\in X'$ such that
\[
\text{$f= gh$ \,\,\,and\,\,\, $\|f\|_{L_1(\mu)} = \|g\|_{X}\,\|h\|_{X'}$\,,}
\]
Additionally,  if $f$ is positive, then also $g$ and $h$ may be chosen positive.
\end{theorem}

\smallskip
The following lemma is our starting point.

\smallskip

\begin{lemma} \label{YM(Y,X)}
Let $X_n = (\mathbb{C}^n, \|\cdot\|)$ and $Y_n = (\mathbb{C}^n, \|\cdot\|)$ be two Banach lattices  such that 
\[
\|\id \colon X_n \to Y_n\circ M(Y_n, X_n)\| \leq 1\,.
\]
Then
\begin{equation*} 
\boldsymbol{\chimon}\big(\mathcal{P}_{J}(X_n)\big) \leq \boldsymbol{\chimon}\big(\mathcal{P}_{J}(Y_n)\big)\,.
\end{equation*}
\end{lemma}

\begin{proof}
Fix some $z\in B_{X_n}$. Then by assumption there exist $\xi, v\in \mathbb{C}^n$ such that  $z= v \xi$, $\|v\|_{Y_n} \leq 1$,
and  
\[
\|D_\xi \colon Y_n \to X_n\|\leq1\,.
\]
Hence, for every $P\in \mathcal{P}_J(X_n)$, we have
\begin{align*}
\Big| \sum_{\alpha \in J} |c_\alpha(P)|\,z^\alpha\Big|   =
\Big|\sum_{\alpha \in J} |c_\alpha(P\circ D_\xi)|\,v^\alpha\Big|
&
\leq \boldsymbol{\chimon}\big(\mathcal{P}_J(Y_n)\big)\,\|P\circ D_\xi\|_{\mathcal{P}_J(Y_n)} \\
& \leq \boldsymbol{\chimon}\big(\mathcal{P}_J(Y_n)\big)\,\|P\|_{\mathcal{P}_J(X_n)}\,,
\end{align*}
which completes the proof.
\end{proof}

\smallskip

\begin{corollary} \label{C1C2}
Let $X_n = (\mathbb{C}^n, \|\cdot\|)$  be a Banach lattice and $J \subset \mathbb{N}_0^n$.
Then
\[
\boldsymbol{\chimon}\big(\mathcal{P}_{J}(\ell^n_1)\big) \le
\boldsymbol{\chimon}\big(\mathcal{P}_{J}(X_n)\big) \le \boldsymbol{\chimon}\big(\mathcal{P}_{J}(\ell^n_\infty)\big)\,.
\]
\end{corollary}

\begin{proof}
We apply Lemma~\ref{YM(Y,X)}.
  Obviously,
  \[
  \ell^n_\infty \circ M(\ell^n_\infty, X_n ) \equiv X_n \,,
  \]
  which  gives the second inequality. On the other hand by Lozanovskii's Theorem~\ref{Lozanovskii} we get
  \[
   X_n  \circ M( X_n ,\ell^n_1 ) \equiv X_n  \circ  X_n'  \equiv\ell^n_1\,,
  \]
  which leads to  the first estimate.
\end{proof}

\smallskip
Under convexity and concavity assumptions of the underlying Banach lattices, the preceding ideas may be extended.

\smallskip

\begin{theorem}
\label{PropAppl}
Let $X_n = (\mathbb{C}^n, \|\cdot\|)$ and $Y_n = (\mathbb{C}^n, \|\cdot\|)$ be two Banach lattices
such that $M_{(r)}(X_n ) = M^{(r)}(Y_n) = 1$, where $1 <r <\infty$. Then, for every $J \subset \mathbb{N}_0^n$, we have
\begin{equation*} 
\boldsymbol{\chimon}\big(\mathcal{P}_{J}(X_n)\big) \leq \boldsymbol{\chimon}\big(\mathcal{P}_{J}(Y_n)\big)\,.
\end{equation*}
\end{theorem}

In view of Lemma~\ref{YM(Y,X)}, the proof is an immediate consequence of the following independently interesting
result, which was also   proved in \cite[Theorem 3.8,(i)]{schep2010products} under the additional assumption that $Y$ has the Fatou property. We here give an alternative approach.

\begin{lemma} \label{app5}
Let $X$ and $Y$ be Banach sequence lattices  such that  $X$ is $r$-concave and
$Y$ is $r$-convex for  some $1 <r <\infty$ with concavity and convexity constants equal to $1$. Then
\[
X \equiv Y\circ M(Y, X)\,.
\]
\end{lemma}

\begin{proof}
We start collecting  three facts:

\noindent Fact 1.
 Assume that the product $E\circ F$ of two Banach sequences lattices $E$ and $F$
is a~Banach space. Then the equality
\[
E'\equiv F\circ(E\circ F)'
\]
holds isometrically.
Indeed, by Lozanovskii's factorization Theorem~\ref{Lozanovskii} (recall that the K\"othe duals of Banach function lattices have the Fatou property)
\[
E\circ E'\equiv \ell_1 \equiv (E\circ F)\circ (E\circ F)' \equiv E\circ (F\circ (E\circ F)')\,,
\]
and since there is some $x \in E$ with $\text{supp}\,x = \mathbb{N}$, this yields the required statement.

\noindent Fact 2.
If $X$ has the Fatou property (i.e., $X\equiv X''$)
and $X' \circ Y$ is a~Banach space, then by 
Fact 1 we have
\[
X\equiv (X')' \equiv  Y\circ (X'\circ Y)' \equiv Y \circ M(X', Y') \equiv Y \circ M(Y, X)\,.
\]

\noindent
Fact 3. Assume that $X$ be is $p$-convex and $Y$ is $q$-convex with convexity constants equal to $1$.
If $\frac{1}{p} + \frac{1}{q} \leq 1$, then $X \circ Y$ is a~Banach space.

\begin{proof}[Proof of Fact 3.]
To see this we let $Z:= X\circ Y$, and observe that if $\frac{1}{p} + \frac{1}{q} <1$, then there exist
$r \in (1, p)$ and $s\in (1, q)$ such that $\frac{1}{r} + \frac{1}{s} =1$. Since $r<p$ and $s<q$, $X$ is $r$-convex and
$Y$ is $s$-convex  with convexity constants equal to $1$. Thus to prove the statement we may assume without loss of
generality that $\frac{1}{p} + \frac{1}{q} =1$.
  We only need to show the triangle inequality in $Z$.  Given $z_1,z_2 \in Z$ and $\varepsilon >0$, we find  $x_j \in S_X$, $y_j\in S_Y$
and $c_j>0$ such that
\[
|z_j| = c_j |x_jy_j| \quad\, \text{and \,\, $c_j \leq (1+ \varepsilon) \|z_j\|_Z$}, \quad\, j=1, 2\,.
\]
Since $\frac{1}{p} + \frac{1}{q}=1$, it follows by H\"older's inequality that 
\[
|z_1+ z_2| \leq \big(c_1^{1/p} |x_1|\big)\,\big(c_1^{1/q}|y_1|\big) +  \big(c_2^{1/p}|x_2|\big)\,\big(c_2^{1/q} |y_2|\big)
\leq   x\cdot y\,,
\]
where $x:=\big(c_1 |x_1|^p + c_2 |x_2|^p\big)^{\frac{1}{p}}$ and $y:=  \big(c_1 |y_1|^q + c_2 |y_2|^q\big)^{\frac{1}{q}}$.  Due to the $p$-convexity of $X$ and the $q$-convexity of $Y$, we get
\[
\|x\|_X \leq (c_1 + c_2)^{\frac{1}{p}}\quad\, \text{and} \quad\, \|y\|_Y \leq (c_1 + c_2)^{\frac{1}{q}}\,.
\]
Combining, we deduce that $z_1 + z_2 \in Z$ and
\[
\|z_1 + z_2\|_Z \leq \|x\|_X\,\|y\|_Y \leq c_1 + c_2 \leq (1 + \varepsilon)(\|z_1\|_Z  + \|z_2\|_Z)\,.
\]
Since $\varepsilon>0$ is arbitrary, $Z$ is a Banach space.
\end{proof}

Finally, we turn to the proof of Lemma~\ref{app5}: Since $X$ is $r$-concave with $1<r<\infty$, $X$ does not contain an
order isomorphic copy of $c_0$ (clearly such an order isomorphic copy would preserve the concavity of $X$, but
$c_0$ has trivial concavity). In consequence, $X$ has the Fatou property, otherwise it would contain an order
copy $c_0$ by the well-known characterization of Banach function lattices without copies of $c_0$ (see, e.g.,
\cite[Theorem 14.13]{aliprantis2006positive}). Then by duality $X'$ is $r'$-convex with $1/r' = 1- 1/r$ and with
convexity constant  equal to $1$. Hence it  follows from Fact 3.  that $X'\circ Y$ is a~Banach space. In consequence
by Fact $2$, we get the required statement.
\end{proof}

\smallskip

\begin{corollary} \label{C1}
Let $X_n = (\mathbb{C}^n, \|\cdot\|)$ be a  Banach lattice
such that $M_{(r)}(X_n ) =  1$, where $1 <r <\infty$. Then, for every $J \subset \mathbb{N}_0^n$, we have
\[
\boldsymbol{\chimon}\big(\mathcal{P}_{J}(X_n)) \leq \boldsymbol{\chimon}\big(\mathcal{P}_{J}(\ell_r^n)\big)\,.
\]
\end{corollary}

As already shown in Corollary~\ref{C1C2}, this result also holds for $r = \infty$ (note that $M_{(\infty)}(X_n ) =  1$ for
any~$X_n$).

\begin{corollary} \label{C2}
Let $Y_n = (\mathbb{C}^n, \|\cdot\|)$ be a  Banach lattice
such that $M^{(r)}(Y_n ) =  1$, where $1 <r <\infty$. Then, for every $J \subset \mathbb{N}_0^n$, we have
\[
\boldsymbol{\chimon}\big(\mathcal{P}_{J}(\ell^n_r)\big) \leq \boldsymbol{\chimon}\big(\mathcal{P}_{J}(Y_n)\big)\,.
\]
\end{corollary}

Since  $M^{(1)}(Y_n ) =  1$ holds for any $Y_n$, the case $r=1$ is again covered by Corollary~\ref{C1C2}.

\smallskip

Finally, we show that,
if in the above corrolaries  the concavity constant of $X_n$ and the  convexity constant of $Y_n$ differ from  $1$,
then at least in the  homogeneous case (i.e, $J \subset \Lambda(m,n)$ for some $m$) the estimates hold
up to a~constant $C^m$, where $C$ depends on the $r$-concavity and the $r$-convexity constants.

To see this we
recall  the so-called renorming theorem of Figiel and Johnson (see, e.g., \cite[Proposition~1.d.8]{LT}): If
a~Banach function lattice $X$ on $(\Omega, \Sigma, \mu)$ is $p$-convex and $q$-concave for some $1\leq p\leq q\leq \infty$,
then on $X$ there is an equivalent lattice norm $\|\cdot\|_0$ whose $p$-convexity and $q$-concavity constants are
both equal to $1$, and
\[
\|x\|_X \leq \|x\|_0 \leq M^{(p)}(X)M_{(q)}(X) \|x\|_{X}\,, \quad\, x\in X\,.
\]

\begin{corollary}
\label{appl}
Let $X$ and $Y$ be Banach sequence lattices such that  $X$ is $r$-concave and
$Y$ is $r$-convex for some $1 <r <\infty$. Then for any subset $J \subset \Lambda(m,n)$, we have
\[
\boldsymbol{\chimon}\big(\mathcal{P}_{J}(X_n)\big) \leq (M_{(r)}(X) M^{(r)}(Y))^m\,\,\boldsymbol{\chimon}\big(\mathcal{P}_{J}(Y_n)\big)\,.
\]
\end{corollary}

\smallskip

\begin{proof}
Observe first that if $E:= (\mathbb{C}^n, \|\cdot\|_E)$ and $F:=(\mathbb{C}^n, \|\cdot\|_F)$ are
 $n$-dimensional Banach lattices such that for some $\gamma \geq 1$
\[
\|z\|_E \leq \|z\|_F \leq \gamma \|z\|_E, \quad\, z\in \mathbb{C}^n\,,
\]
then for any polynomial $P\in \mathcal{P}_J(\mathbb{C}^n)$, we have
$\gamma^{-m} \|P\|_{\mathcal{P}_J(E)} \leq \|P\|_{\mathcal{P}_J(F)} \leq \|P\|_{\mathcal{P}_J(E)}$.
Clearly, this yields
\[
\gamma^{-m}\,\boldsymbol{\chimon}\big(\mathcal{P}_J(E)\big) \leq  \boldsymbol{\chimon}\big(\mathcal{P}_J(F)\big) \leq \gamma^m\,\boldsymbol{\chimon}
\big(\mathcal{P}_J(E)\big)\,.
\]
Applying the renorming theorem mentioned above, we conclude that there exist equivalent lattice norms
$\|\cdot\|_{\tilde X}$ and $\|\cdot\|_{\tilde Y}$ on $X$ and $Y$, respectively, such that
the Banach lattices $\widetilde{X}:= (X, \|\cdot\|_{\widetilde{X}})$ and $\widetilde{Y} :=
(Y, \|\cdot\|_{\widetilde{Y}})$ satisfy $M_{(r)}(\widetilde{X}) = M^{(r)}(\widetilde{Y})=1$ as well as
\begin{equation} \label{equi}
\|x\|_X \leq \|x\|_{\widetilde{X}} \leq M_{(r)}(X) \|x\|_X, \quad\, x\in X\,,\,\,\,\,\,\,
\text{and}\,\,\,\,\,\,\,
\|y\|_Y \leq \|y\|_{\widetilde{Y}} \leq M^{(r)}(Y) \|y\|_Y, \quad\, y\in Y\,.
\end{equation}
The above facts combined with Proposition \ref{PropAppl} yields
\begin{align*}
M_{(r)}(X)^{-m}\,\boldsymbol{\chimon}
\big(\mathcal{P}_J(X_n)\big) & \leq  \boldsymbol{\chimon}\big(\mathcal{P}_J(\widetilde{X}_n)\big)
\leq \boldsymbol{\chimon}\big(\mathcal{P}_J(\widetilde{Y}_n)\big) \leq M^{(r)}(Y)^{m}\,\boldsymbol{\chimon}\big (\mathcal{P}_J(Y_n)\big)
\end{align*}
and this completes the proof of the  claim.
\end{proof}

\smallskip

\section{Kadets-Snobar case} \label{Kadets-Snobar case}
 We start with two  general upper estimates for the  invariants
  $\boldsymbol{\lambda}\big(\mathcal{P}_{J}(X_n)\big)$ and $\boldsymbol{\chimon}\big(\mathcal{P}_{J}(X_n)\big)$.
  The first one (see \eqref{AAA}) is an immediate consequence of the Kadets-Snobar inequality from \eqref{kadets1}
  (already remarked in \eqref{tincho}), and the second 
  follows from recent improvements of the  Bohnenblust-Hille inequality. Analyzing these estimates further,
  we show  that these bounds are asymptotically optimal, whenever  we consider sections $X_n$ of  $2$-convex Banach sequence  lattices $X$ and index sets~$J$, which are  not too sparse.

\begin{theorem} \label{conny2}
Let $X_n= (\mathbb{C}^n,\|\cdot\|)$ be a Banach space. Then for any index set $J \subset \mathbb{N}_0^n$
\begin{align} \label{AAA}
\boldsymbol{\lambda}\big(\mathcal{P}_{J}(X_n)\big)  \,\,\leq \,\, |J|^{\frac{1}{2}}\,,
\end{align}
and if $X_n= (\mathbb{C}^n,\|\cdot\|)$ is a Banach lattice and  $J \subset \mathbb{N}_0^n$ has degree $m$, then 
\begin{align} \label{BBB}
\boldsymbol{\chimon}\big(\mathcal{P}_{J}(X_n)\big) \,\, \leq \, \, C^{\sqrt{m \log m}}
|J|^{\frac{m-1}{2m}}
\,,
\end{align}
where $C \ge 1 $ is a universal constant.
\end{theorem}

\smallskip
As already mentioned, the  proof
of \eqref{BBB} is based on the so-called hypercontractive Bohnenblust-Hille inequality: For every $P \in \mathcal{P}_{\leq m }(\ell_\infty^n)$
with coefficients $(c_\alpha(P)_{\alpha \in \Lambda(\leq m,n)}$ we have 
\begin{equation}\label{equa:BH}
\Big(\sum_{\alpha \in \Lambda(\leq m,n)}|c_\alpha(P)|^{\frac{2m}{m+1}}\Big)^{\frac{m+1}{2m}}
\leq C^{\sqrt{m \log m}} \, \| f\|_{\infty}\,,
\end{equation}
where $C \ge 1$ is universal.

This result has a long history starting with Littlewood \cite{littlewood1930bounded} and Bohnenblust-Hille \cite{bohnenblust1931absolute}, who proved that the best constant
$\mathbf{BH}_m$ in \eqref{equa:BH} only depends on $m$ and not on $n$. Hypercontractivity of the Bohnenblust-Hille
inequality (i.e., $\mathbf{BH}_m \leq C^m$ for some universal $C \ge 1$) was first proved in
\cite{defant2011bohnenblust}, and the considerably  better  subexponential estimate $\mathbf{BH}_m \leq C^{\sqrt{m \log(m)}}$ is  due to
\cite{bayart2014bohr}. For more information we again refer to \cite{defant2019libro}.

\begin{proof}[Proof of Theorem~\ref{conny2}, \eqref{BBB}]

By  Corollary~\ref{C1C2} 
\[
 \boldsymbol{\chimon}\big(\mathcal{P}_{J}(X_n)\big)
\leq \boldsymbol {\chimon} \big(\mathcal{P}_{J}(\ell_{\infty}^n)\big)
\,,
\]
and hence we assume without loss of generality that $X_n=\ell^n_\infty$.
Fix some $P = \sum_{\alpha \in J} c_\alpha(P) z^\alpha \in \mathcal{P}_J(\ell_\infty^n)$, signs $(\varepsilon_{\alpha})_{\alpha \in J}$
and $z\in B_{\ell_\infty^n}$. Then by H\"older's inequality and the hypercontractive Bohnenblust-Hille inequality
we obtain 
\[
  \Big|\sum_{\alpha \in J} \varepsilon_{\alpha} c_\alpha(P) z^\alpha \Big|
  \leq
    \bigg(  \sum_{\alpha \in J}  |c_\alpha(P)|^{\frac{2m}{m+1}}\bigg)^{\frac{m+1}{2m}}   |J|^{\frac{m-1}{2m}}
  \leq
  C^{\sqrt{m\log m }} |J|^{\frac{m-1}{2m}} \|P\|_\infty\,,
  \]
  which immediately implies the desired estimate.
  \end{proof}

  \smallskip
  We  now turn our attention to the following question: When are the two estimates from  Theorem~\ref{conny2} asymptotically optimal?

  \smallskip

  \begin{remark} \label{cardo}
    Using \eqref{cardi} and  the  elementary estimate $ \Big(\frac{N}{\ell}\Big)^\ell \le \binom{N}{\ell} \le e^\ell\Big(\frac{N}{\ell}\Big)^\ell$ for each $1 \leq \ell \leq N$, a simple calculation shows that
    \begin{equation}
    |J| \sim_{C^m}\Big( 1+\frac{n}{m}\Big)^{m}\,,
  \end{equation}
  whenever $J$ is one of the following four index sets: 
    $\Lambda_T(m,n), \, \Lambda_T( \leq m,n), \,\Lambda(m,n),\,\Lambda(\leq m,n)$.
  \end{remark}

  \smallskip

  \begin{theorem} \label{conny3}
Let  $X$ be a Banach sequence lattice and  $J \subset \mathbb{N}_0^{(\mathbb{N})}$
an index set. Then
\begin{align} \label{mmm}
\boldsymbol{\lambda}\big(\mathcal{P}_{J_{\leq m}}(X_n)\big)  \prec_{C^m}\Big( 1+\frac{n}{m}\Big)^{\frac{m}{2}}
\end{align}
and
\begin{align} \label{MMM}
\boldsymbol{\chimon}\big(\mathcal{P}_{J_{\leq m}}(X_n)\big)  \prec_{C^m}\Big(1+ \frac{n}{m}\Big)^{\frac{m-1}{2}}\,,
\end{align}
where $C > 0$ is universal.

Moreover, if $X$ is $2$-convex,  then the preceding estimates are optimal for all $m$ for which   $\Lambda_T(m) \subset J_{\leq m}$, in the sense that under this assumption
$\prec_{C^m}$ may be replaced by $\sim_{C^m}$, where $C >0$ is  only depending on the $2$-convexity constant of $X$.
\end{theorem}

\begin{proof}[Proof]
Both  upper estimates are immediate consequences of Theorem~\ref{conny2} and Remark~\ref{cardo}. To prove the lower bounds, we start with that of \eqref{MMM}.
By Corollary~\ref{appl}  we know that for $m\leq n$ with $\Lambda_T(m) \subset J_{\leq m}$
\[
\boldsymbol{\chimon}\big(\mathcal{P}_{\Lambda_T(m,n)}
(\ell_2^n)\big) \leq M^{(2)}(X_n)^m
\boldsymbol{\chimon}\big(\mathcal{P}_{\Lambda_T(m,n)}(X_n)
\big) \leq  M^{(2)}(X)^m\boldsymbol{\chimon}
\big(\mathcal{P}_{J_{\leq m}}(X_n)\big)\,.
\]
Then we deduce from Proposition~\ref{toblach} (in the Hilbert space case) that for $m\le n$
\[
\Big( 1+ \frac{n}{m}\Big)^{\frac{m-1}{2}}
\leq
2^{\frac{m-1}{2}}
\Big(\frac{n}{m}\Big)^{\frac{m-1}{2}}
\prec_{C^m}
\boldsymbol{\chimon}
\big(\mathcal{P}_{\Lambda_T(m,n)}(\ell_2^n)\big) \prec_{C^m} \boldsymbol{\chimon}\big(\mathcal{P}_{J_{\leq m}}(X_n)\big)\,,
\]
 and since we on the other hand for $m \ge n$ obviously have
 \begin{align} \label{obvious}
   \Big( 1+ \frac{n}{m}\Big)^{\frac{m-1}{2}} \leq 2^{\frac{m-1}{2}} \prec_{C^m} \boldsymbol{\chimon}\big(\mathcal{P}_{J_{\leq m}}(X_n)\big)\,,
 \end{align}
the  lower bound in  \eqref{MMM} is proved. Finally, it remains to check
the  lower bound in \eqref{mmm}.
We apply  Theorem~\ref{main3}, Theorem~\ref{OrOuSe} (twice), and  Proposition~\ref{Cauchy} which all together prove that for $m \leq n$
with $\Lambda_T(m) \subset J_{\leq m}$
\begin{align} \label{verflucht}
\begin{split}
  \boldsymbol{\chimon}\big( \mathcal{P}_{\Lambda_T(m+1,n)}(X_n)\big)
  &
  \, \leq \,e2^{m+1}
  \|\mathbf{Q}_{\Lambda(m+1,n),\Lambda_T(m+1,n)}\|\,\,
 \boldsymbol{\lambda}\big( \mathcal{P}_{\Lambda_T(m,n)}(X_n)\big)
 \\&
 \, \leq \,e2^{m+1}  \kappa^{m+1} \boldsymbol{\lambda}\big( \mathcal{P}_{\Lambda_T(m,n)}(X_n)\big)
 \\&
 \, \leq \,e2^{m+1}  \kappa^{2m+1} \boldsymbol{\lambda}\big( \mathcal{P}_{J_m}(X_n)\big)
 \, \leq \,e2^{m+1} \kappa^{2m+1} \boldsymbol{\lambda}\big( \mathcal{P}_{J_{\leq m}}(X_n)\big)\,.
 \end{split}
\end{align}
Then for $m \leq n$ with $\Lambda_T(m) \subset J_{\leq m}$ we obtain the lower bound in \eqref{mmm} from \eqref{MMM}, and for $m \ge n$ we use a simple modification of the straightforward argument in~\eqref{obvious}.
\end{proof}

We finish this section indicating an alternative approach to the upper bound in \eqref{MMM}. This approach has the advantage that it avoids the hypercontractivity of the
Bohnenblust-Hille inequality, but the disadvantage that it leads to  weaker constants. Note first that by assumption
\[
\boldsymbol{\chimon}\big(\mathcal{P}_{J}(X_n)\big)  \leq \boldsymbol{\chimon}\big(\mathcal{P}_{\leq m}(X_n)\big)\,.
\]
 By Proposition~\ref{degreelessm}
\[
\boldsymbol{\chimon}\big(\mathcal{P}_{\leq m}(X_n)\big)
 \,\,\leq \,\,(m+1) \max_{1 \leq k \leq m} \boldsymbol{\chimon}\big(\mathcal{P}_{k}(X_n)\big)\,,
\]
and  by Theorem~\ref{main3} for all $1 \leq k \leq m$
\[
   \boldsymbol{\chimon}\big( \mathcal{P}_{k}(X_n)\big)
 \,\le \,e 2^k \,
 \boldsymbol{\lambda}\big( \mathcal{P}_{k-1}(X_n)\big)\,.
\]
Using  \eqref{AAA}
and Remark~\ref{cardo}, we all together get
\begin{align*}
 \boldsymbol{\chimon}\big(\mathcal{P}_{J}(X_n)\big)
 &
\,\leq \,(m+1) \max_{1 \leq k \leq m} e 2^k \sqrt{\Lambda(k-1,n)}
\,= \,(m+1)e 2^m \sqrt{\Lambda(m-1,n)}
\\
&
\,\leq \, (m+1)2^m e^{m} \Big( 1+\frac{n}{m-1}\Big)^{\frac{m-1}{2}}
\,\leq \,
(m+1)2^m e^{m} 2^{^{\frac{m-1}{2}}}\Big( 1+\frac{n}{m}\Big)^{\frac{m-1}{2}}
\,,
\end{align*}
 which is the upper bound in \eqref{MMM}.

 \smallskip

\section{Miscellanea}
\label{Miscellanea}
We answer a couple of specific questions on projection and unconditional basis constants of spaces of multivariate polynomials -- all results in one way or the other are related with  polynomials on
$\ell_1^n$ or $\ell_2^n$.

\subsection{Remarks for polynomials on $\ell_1^n$}
Recall from Theorem~\ref{tensor} that for every $m,n$
\begin{equation} \label{q0}
    \boldsymbol{\lambda}\big(\mathcal{P}_{ m}(\ell_1^n)\big)
    \leq c(m,\ell_1^n )\,,
\end{equation}
and so in particular
by \eqref{projl1} and Theorem~\ref{degree-homo}
\begin{equation}\label{start}
  \boldsymbol{\lambda}\big(\mathcal{P}_{ m}(\ell_1^n)\big) \leq  e^m
    \quad
    \text{ and }
  \quad
    \boldsymbol{\lambda}\big(\mathcal{P}_{ \leq m}(\ell_1^n)\big) \leq  (m+1) e^{m}\,.
\end{equation}
What about lower bounds? A particular question in this direction is as follows:  Is it true that for each fixed $n$
 \begin{equation} \label{q1}
     \lim_{m \to \infty} \boldsymbol{\lambda}(\mathcal{P}_{\leq m}(\ell_1^n)) = \infty\,\,\, \pmb{?}
 \end{equation}
  The answer in fact is yes due to the result of  Lozinski-Kharshiladze from
  Corollary ~\ref{LoKha}, and 
  even more, we may   replace  $\ell_1$ by any  other Banach sequence lattice $X$.
  
  To see this we need the  following simple lemma.
\begin{lemma}
  Let $Y$ be a Banach space, and $X$ a $1$-complemented subspace of $Y$. Then
  \[
  \boldsymbol{\lambda}\big(\mathcal{P}_{\leq m}(X)\big) \leq  \boldsymbol{\lambda}\big(\mathcal{P}_{\leq m}(Y)\big).
  \]
  Moreover, we may replace the space of all polynomials of degree $\leq m$
  by all $m$-homogeneous polynomials.
\end{lemma}

\begin{proof}
Take a factorization $\id_X = v\circ u$, where $u: X \to Y$ and $v: Y \to X$, both with norms $\leq 1$. Then
\[
\id_{\mathcal{P}_{\leq m}(X)} = V\circ U\,,
\]
where
\begin{align*}
&
  U:\mathcal{P}_{\leq m}(X) \to \mathcal{P}_{\leq m}(Y)\,,\,\,\,\, P \mapsto P\circ v
  \quad \text{ and } \quad
  V:\mathcal{P}_{\leq m}(Y) \to \mathcal{P}_{\leq m}(X)\,,\,\,\,\, Q \mapsto Q\circ u\,.
\end{align*}
Since $\|U\|\leq \|u\|$ and $\|V\|\leq \|v\|$, the claim follows.
  \end{proof}
  
Here is the announced result, which in the  particular case $X = \ell_1$ gives a positive answers to the  question from~\eqref{q1}.

 \begin{proposition}
 Let $X$ be a Banach sequence lattice. Then for every $n$
 \[
  \lim_{m \to \infty} \boldsymbol{\lambda}\big(\mathcal{P}_{\leq m}(X_n)\big) = \infty\,.
  \]
 More precisely,
 there is some $C >0$ such that for all $n$ and $m$
     \[
   \log(1+ m) \,\, \leq  \,\,C\,\boldsymbol{\lambda}\big(\mathcal{P}_{\leq m}(X_n)\big)\,.
  \]
    \end{proposition}

    \begin{proof}
      From Corollary ~\ref{LoKha}  we know that
      $$ \log(1+ m) \,\, \prec  \,\,\boldsymbol{\lambda}\big(\mathcal{P}_{\leq m}(\mathbb{C})\big)\,.$$
      Since $\mathbb{C}$ is $1$-complemented in $X_n$, the proof follows from the preceding lemma.
    \end{proof}

    \smallskip

 Note that for the unconditional basis constant of $\mathcal{P}_{\leq m}(\ell_1^n)$  a question similar to \eqref{q1} might be asked, namely
   \[
  \lim_{m \to \infty} \boldsymbol{\chimon}\big(\mathcal{P}_{\leq m}(\ell_1^n)\big) = \infty\,\,\, \pmb{?}
  \]
Due to a non-trivial result of Gordon and Reisner from 
\cite{gordon1982some}
 the answer again is affirmative. In fact,   $\ell_1$ again 
 may be replaced by any Banach sequence lattice
 $X$.

\smallskip
 \begin{proposition}
 Let $X$ be a Banach sequence lattice. Then for every $n$
 \[
  \lim_{m \to \infty} \boldsymbol{\chimon}\big(\mathcal{P}_{\leq m}(X_n)\big) = \infty\,.
  \]
 More precisely,
 there is some $C >0$ such that for all $n$ and $m$
     \[
   \,\,C \sqrt{\log(1+ m)} \,\, \leq  \,\boldsymbol{\chimon}\big(\mathcal{P}_{\leq m}(X_n)\big)\,.
  \]
    \end{proposition}

    \begin{proof}
    By \cite[Proposition 2.1]{gordon1982some} there is $C >0$ such that for all $m$
  \[
  C \sqrt{\log m} \leq {\mbox{gl}}\big(\mathcal{P}_{\leq m}(\mathbb{C})\big)\,.
   \]
   Hence, we deduce from  the Gordon-Lewis inequality 
   \eqref{gl-inequality}
   that indeed
   \[
   C \sqrt{\log m} \leq  \boldsymbol{\chimon}\big(\mathcal{P}_{\leq m}(\mathbb{C})\big)
   =   \boldsymbol{\chimon}\big(\mathcal{P}_{\{ ke_1 \colon 0 \leq k \leq m\}}(X_n)\big)
   \leq  \boldsymbol{\chimon}\big(\mathcal{P}_{ \leq m}(X_n)\big)\,. \qedhere
   \]
  \end{proof}

Let us come to another question, which as the question in \eqref{q1}, is  motivated by~\eqref{start}:
Let $X$ be a Banach sequence lattice such that
\begin{equation*} \label{q2}
    \sup_{m,n} \boldsymbol{\lambda}\big(\mathcal{P}_{m}(X_n)\big)^\frac{1}{m} < \infty.
\end{equation*}
By~\eqref{start} this definitely holds for $X = \ell_1$, but does this condition conversely imply that $X = \ell_1$? Indeed, this is actually the case as the following theorem shows.

  \begin{theorem}
  \label{Xl1}
    Let $X$ be a Banach sequence lattice. Then the following are equivalent:

    \begin{itemize}
      \item[(a)]
      $\sup_{m,n} \boldsymbol{\lambda}\big(\mathcal{P}_{ \leq m}(X_n)\big)^\frac{1}{m} < \infty$,
      \item[(b)]
      $\sup_{m,n} \boldsymbol{\lambda}\big(\mathcal{P}_{ m}(X_n)\big)^\frac{1}{m} < \infty$,
      \item[(c)]
      $X_n = \ell_1^n$ uniformly,
      \item[(d)]
      $X=\ell_1$ whenever $X$ is separable.
    \end{itemize}
     \end{theorem}

  \begin{proof}
 By Theorem~\ref{degree-homo} the first two statements are equivalent.
  According to \eqref{start}, we only have to show that $(b)$ implies $(c)$. Under the assumption of $(b)$ we have that
  \[
  \sup_{n} \gamma_1(X_n)  =
  \sup_{n} \gamma_\infty(X_n^\ast) =\sup_{n} \boldsymbol{\lambda}(X_n^\ast) \leq C:= \sup_{m,n} \boldsymbol{\lambda}\big(\mathcal{P}_{  m}(X_n)\big)^\frac{1}{m} < \infty\,.
  \]
 We show that there is $K >0$ such that for any diagonal operator $D_\lambda: X_n  \to \ell_n^2, \,(e_i) \mapsto (\lambda_i e_i) $ we have
 \[
 \sup_{n} \pi_1(D_\lambda: X_n\to \ell^n_2) \leq  K \|D_\lambda\|\,.
 \]
  Indeed, take a factorization $\id = uv$, where $u: X_n \to L_1(\mu)$ and $v:  L_1(\mu) \to X_n $ for some measure $\mu$ and
 $\|u\| \|v\| \leq  2  \gamma_1(X_n)$. From Grothendieck's theorem we know that
 \[
 \pi_1 (D_\lambda \circ v) \leq K_G  \|D_\lambda \circ v\|\,,
 \]
 and hence
 \begin{align*}
   \pi_1(D_\lambda: X_n\to \ell^n_2)
   &
   =  \pi_1(D_\lambda\circ v \circ u)
 \leq \pi_1(D_\lambda\circ v) \|u\|
 \\&
 \leq K_G  \|D_\lambda\| \|v\| \|u\|\leq  K_G  \|D_\lambda\| 2  \gamma_1(X_n)
 \leq  2 K_G C  \|D_\lambda\|\,.
 \end{align*}
 Then we deduce from \cite[Proposition 33.6]{tomczak1989banach} (a result which essentially goes  back to Lindenstrauss and
 Pelczynski) that
 \[
 \|\id\colon  X_n \to \ell_1^n \| \leq K=(2 K_G C)^2\,,
 \]
 the conclusion.
 \end{proof}

\smallskip

In what follows we assume that $(\Omega, \mathcal{A}, \mu)$ is a~$\sigma$-finite separable and atomless measure space,
or it is purely atomic with the counting measure $\mu$ on $\mathcal{A} := 2^{\mathbb{N}}$. See  \cite{KM2000}
for a~complete description of all symmetric  Banach lattices over $(\Omega, \mathcal{A}, \mu)$ with the Dunford-Pettis
property.  As an application of this result it is shown in \cite{KM2000} that, the only symmetric Banach lattices with
both the Dunford-Pettis and the Grothendieck property are the $L_\infty(\mu)$-spaces. In consequence, we may conclude that
for every symmetric Banach lattice  $X\neq L_\infty(\mu)$, we have $\lambda(X)=\infty$.
We again refer to  the discussion in
the introduction  (in particular to the part 'Some history').

The following result characterizes those symmetric Banach lattices over $(\Omega, \mathcal{A}, \mu)$ for which
the projection constant of $\mathcal{P}_m(X)$ or $\mathcal{P}_{\leq m}(X)$ is finite. Note that this an extended
version of Theorem \ref{Xl1} with a~completely different proof.

\begin{theorem}
Let $X$ be  a~real or complex symmetric Banach lattice over $(\Omega, \mathcal{A}, \mu)$ and $m \in \mathbb{N}$. Then the following
statements are equivalent.
\begin{itemize}
\item[{\rm(i)}] $\boldsymbol{\lambda}\big(\mathcal{P}_{\leq m}(X)\big)<\infty$,
\item[{\rm(ii)}] $\boldsymbol{\lambda}
\big(\mathcal{P}_m(X)\big)<\infty$,
\item[{\rm(iii)}] $X=L_1(\mu)$,
\item[{\rm(iv)}] $\boldsymbol{\lambda} \big(\mathcal{P}_k(X)\big)<\infty$, for each $k\in \mathbb{N}$.
\end{itemize}
\end{theorem}

\begin{proof} Without loss of generality we may assume that $X$ is a~real space. (i) $\Rightarrow$ (ii). Since for any Banach space $E$, we have the natural topological decomposition
\[
\mathcal{P}_{\leq m}(E) =  \mathbb{R} \oplus \mathcal{P}_1(E) \oplus \cdots \oplus \mathcal{P}_m(E)\,,
\]
 we conclude that $\boldsymbol{\lambda} \big(\mathcal{P}_k(X)\big)<\infty$ for each $1\leq k \leq m$.

(ii) $\Rightarrow$ (iii). Recall that there is a~canonical isomorphism between $\mathcal{P}_m(X)$ and
$\mathcal{P}_{\leq m}\big(X\oplus \mathbb{R}\big)$ given by homogenisation (see, e.g.,  \cite{HR2022}). This implies that $\boldsymbol{\lambda}
\big(\mathcal{P}_1(X\oplus \mathbb{R})\big) <\infty$, and so $\lambda(X^{*})<\infty$. Since the K\"othe dual space $X'$ is isometricaly isomorphic to a~complemented subspace of $X^{*}$, $\lambda(X') <\infty$.
Hence the symmetric Banach lattice $X'$ over $(\Omega, \mathcal{A}, \mu)$ is isomorphic to a~complemented
subspace of $\ell_\infty(S)$ for some nonempty set $S$. In consequence, $X'$ has both the Dunford--Pettis
and the Grothendieck property.  We can then apply the result from \cite{KM2000} mentioned above to deduce that
$X'=L_\infty(\mu)$, which clearly yields $X=L_1(\mu)$ as required.

(iii) $\Rightarrow$ (iv). Note that based on the previously mentioned
 Lindenstrauss characterization of complemented subspaces in $\ell_{\infty}$ (see again the introduction), in \cite{HR2022} it was proved that if $X$ is a~separable 
$\mathcal{L}_1$-space, then for each $k\in \mathbb{N}$ the space $\mathcal{P}_k(X)$, so also $\mathcal{P}_{\leq m}(X)$, is isomorphic to $\ell_\infty$. Clearly, this yields $\boldsymbol{\lambda}\big(\mathcal{P}_{\leq m}(X)\big)<\infty$. 

The implication (iv) $\Rightarrow $ (i) is obvious, and this completes the proof. 
\end{proof}

Whether we in Theorem \ref{Xl1} may replace the projection constant of $\mathcal{P}_{ \leq m}(X_n)$ by the  unconditional basis constant of the monomial basis, is an open  problem  (which in  \cite{defant2019libro} is called 'Lempert's problem'). We recall the following far reaching partial solution
from \cite{bayart2012maximum} (see also \cite[Theorem 20.26]{defant2019libro}).

\begin{theorem}
Let $X$ be a Banach sequence lattice such that
\[
\sup_{m,n} \boldsymbol{\chimon}\big(\mathcal{P}_{ m}(X_n)\big)^\frac{1}{m} < \infty.
\]
Then  for every $\varepsilon >0$ there is $D >0$ such for $n$ large enough
\[
\|\id\colon X_n \to \ell_1^n \| \,\,\leq D \,\,(\log \log n)^\varepsilon\,.
\]
 Moreover, here  $\mathcal{P}_{ m}(X_n)$ may be replaced by  $\mathcal{P}_{\leq  m}(X_n)$.
  \end{theorem}

We  finish this discussion for polynomials on $\ell_1^n$ with the following remark, which may be seen as a~polynomial version of \eqref{formulaX}.
\begin{remark}\label{rem:dist to P(l_1)}
Let $X_n = (\mathbb{C}^n, \|\cdot\|)$ be a Banach space. Then for any index set 
$\Lambda_T(m,n) \subset J\subset \Lambda(m,n)$,
\begin{align}\label{eq:dist to P(l_1)}
\boldsymbol{\lambda}\big(\mathcal{P}_{J}(X_n)\big)\prec_{C^m} d\big(\mathcal{P}_{J}(X_n),\mathcal{P}_{J}(\ell_1^n)\big)\,.
\end{align}
\end{remark}

\begin{proof}
By either \eqref{start}  in the case $J= \Lambda(m,n)$ or by Theorem \ref{start-poly2} for general $J$, we know that $\boldsymbol{\lambda}(\mathcal{P}_{J}(\ell_1^n))\prec_{C^m}1.$
Thus from \eqref{gammainfty} we deduce that  there is a factorization 
$$
 \mathcal{P}_{J}(\ell_1^n)\rightarrow\ell_\infty\rightarrow\mathcal{P}_{J}(\ell_1^n)
$$
of the identity  on $\mathcal{P}_{J}(\ell_1^n)$ with norm $\prec_{C^m}1.$ 
This implies, given a factorization 
$X_n \to \ell^n_1 \to X_n$ of the identity on $X_n$,   the existence of a factorization of the identity on $\mathcal{P}_{J}(\ell_1^n)$ of the form 
$$
\mathcal{P}_{J}(X)\rightarrow \mathcal{P}_{J}(\ell_1^n)\rightarrow\ell_\infty\rightarrow\mathcal{P}_{J}(\ell_1^n)\rightarrow \mathcal{P}_{J}(X).
$$
Another application of \eqref{gammainfty} allows us to conclude \eqref{eq:dist to P(l_1)}.
\end{proof}

\smallskip

\subsection{Remarks for polynomials on $\ell_2^n$}
From Theorem~\ref{conny2}, \eqref{BBB} we see that
\begin{equation}\label{glzero}
\boldsymbol{\chimon}\big(\mathcal{P}_m(\ell_2^n)\big) \leq e 2^m  \sqrt{\Lambda(m-1,n)} = e 2^m \,\sqrt{ {n-2 + m}\choose {m-1}}\,.
\end{equation}
The following proposition in fact  complements the  preceding estimate.
Let us point out that the proof of \eqref{glzero},  based on the Gordon--Lewis approach, is rather involved;
in particular, the Kadets-Snobar theorem and the following two non-trivial estimates are used
\begin{equation*}  \label{glA}
\boldsymbol{\chimon}\big(\mathcal{P}_m(\ell_2^n)\big) \leq 2^m {\mbox{gl}}\big(\mathcal{P}_m(\ell_2^n)\big)
\,\,\,\,\,\,\,\, \text{and} \,\,\,\,\,\,\,\,
{\mbox{gl}}\big(\mathcal{P}_m\big) \leq e \boldsymbol{\lambda}\big(\mathcal{P}_{m-1}(\ell_2^n)\big)
\end{equation*}
(see Theorem~\ref{gl-versus-unc} and Theorem~\ref{gl_versus_proj} for the particular case in which $J = \Lambda(m,n)$).
We show next  an alternative bound by a more direct approach, which is also better for some values of $m$. 
\begin{proposition}
\label{monhilbert}
For each $m, n\in \mathbb{N}$ one has
\[
\chi_{\mon}\big(\mathcal{P}_m(\ell_2^n)
\big) \leq
\begin{cases}
\sqrt{{n-1 + m}\choose{m}}\, \quad\, & n \leq m\,(e^2 4^m + 1) + 1 \\
e\,2^m\, \sqrt{{n-2 + m}\choose{ m-1}}\, \quad\, & n > m\,(e^2 4^m + 1) + 1\,. \\
\end{cases}
\]
In particular, it follows that
\[
\chi_{\mon}\big(\mathcal{P}_m(\ell_2^2)\big)\leq \sqrt{m+1}\,.
\]
\end{proposition}

The proof needs the following lemma.

\begin{lemma} \label{estL2}
For each $m, n\in \mathbb{N}$ and for every polynomial $P\in \mathcal{P}_m(\mathbb{C}^n)$ one has
\[
\|P\|_{\mathcal{P}_m(\ell_2^{n})} \leq \sqrt{\Lambda(m,n)}\,\|P\|_{L_2(\mathbb{S}_n)}=\sqrt{n-1 + m \choose m}\,\|P\|_{L_2(\mathbb{S}_n)}\,.
\]
\end{lemma}

\begin{proof}
As indicated in Section~\ref{Orthogonal projection}
the  collection $(e_\alpha)_{\alpha \in \mathbb{N}_0^n}$ of all monomials
$
e_\alpha (z) := z^\alpha
$
forms an orthogonal family in $L_2(\mathbb{S}_n)$, and by \eqref{normalized} we have that $\|e_\alpha\|_2^2= \frac{(n-1)!\,\alpha!}{(n-1 + m)!}$.
Thus for every polynomial
$P = \sum_{|\alpha|=m} c_\alpha e_\alpha \in \mathcal{P}_m(\mathbb{S}_n)$ one has
\begin{equation} \label{L2}
\|P\|_{L_2(\mathbb{S}_n)} = \bigg(\sum_{|\alpha| =m} |c_\alpha|^2 \,\|e_\alpha\|_2^2\bigg)^{\frac{1}{2}}
= \bigg(\sum_{|\alpha| =m} |c_\alpha|^2\,\frac{(n-1)!\,\alpha!}{(n-1 + m)!}\bigg)^{\frac{1}{2}}\,.
\end{equation}
Fix now $z= (z_1, \ldots, z_n) \in S_{\ell_2^n}$ and observe that the multinomial formula gives
\[
\bigg(\sum_{|\alpha|=m} \frac{m!}{\alpha !}\,|z_\alpha|^2\bigg)^{\frac{1}{2}} =
\big(|z_1|^2 + \cdots + |z_n|^2\big)^{\frac{m}{2}} = 1\,.
\]
Since $|\Lambda(m, n)|= \frac{(n-1 + m)!}{(n-1)!\,m!}$,  Cauchy's inequality yields as desired
\begin{align*}
\bigg |\sum_{|\alpha|=m} c_\alpha z^\alpha\bigg| & = \sqrt{|\Lambda(m, n)|}\,\bigg|\sum_{|\alpha|=m} c_\alpha\,
\Big(\frac{(n-1)!\,\alpha!}{(n - 1 + m)!}\Big)^{\frac{1}{2}}\,\Big(\frac{m!}{\alpha!}\Big)^{\frac{1}{2}} z^\alpha \bigg| \\
& \leq \sqrt{|\Lambda(m, n)|}\,\bigg(\sum_{|\alpha|=m} |c_\alpha|^2\,\frac{(n-1)!\,\alpha!}{(n - 1 + m)!}\bigg)^{\frac{1}{2}}\,
\bigg(\sum_{|\alpha|=m} \frac{m!}{\alpha!}\,|z_\alpha|^2\bigg)^{\frac{1}{2}} \\
& = \sqrt{|\Lambda(m, n)|}\,\|P\|_{L_2(\mathbb{S}_n)}\,. \qedhere
\end{align*}
\end{proof}

\smallskip

\begin{proof}[Proof of Proposition~\ref{monhilbert}]
We fix some   $P\in \mathcal{P}_m(\mathbb{C}^n)$ with coefficients
$c_\alpha,  \alpha \in \Lambda(m,n)$.
Hence  by  Proposition \ref{estL2} and \eqref{L2}, for all $(\varepsilon_\alpha)_{\alpha \in \Lambda(m, n)}$
with $|\varepsilon_\alpha| \leq 1$ for each  $\alpha \in \Lambda(m, n)$, we get
\begin{align*}
\Big\|\sum_{|\alpha|=m} \varepsilon_\alpha c_\alpha z^\alpha \Big\|_{\mathcal{P}_m(\ell_2^{n})}
& \leq \sqrt{|\Lambda(m,n)|}\, \bigg(\sum_{|\alpha| =m} |\varepsilon_\alpha|^2\,|c_\alpha|^2\,
\,\frac{(n-1)!\,\alpha!}{(n-1 + m)!}\bigg)^{\frac{1}{2}} \\
& = \sqrt{|\Lambda(m,n)|}\, \|P\|_{L_2(\mathbb{S}_n)}
\leq \sqrt{|\Lambda(m,n)|}\, \Big\|\sum_{|\alpha|=m} c_\alpha z^\alpha \Big\|_{\mathcal{P}_m(\ell_2^{n})}\,.
\end{align*}
Recall that by the estimate mentioned in \eqref{glzero}, we have
\[
\boldsymbol{\chimon}(\mathcal{P}_m(\ell_2^n)) \leq e\,2^m \,|\Lambda(m-1, n)|^{\frac{1}{2}}\,.
\]
Thus to finish, it is enough to observe that
\[
{n-1 + m \choose m}^{\frac{1}{2}} \leq e\,2^m\, {n-2 + m \choose m-1}^{\frac{1}{2}}
\]
if and only if $n \leq m\,(e^2 4^m + 1) + 1$.
\end{proof}

\bigskip

\chapter{Dirichlet polynomials} \label{Part: Dirichlet polynomials and polynomials on the Boolean cube}

A frequency  $\lambda = (\lambda_n)$  is a strictly increasing, non-negative real sequence tending to $+\infty$. Given a finite index set  $J \subset \mathbb{N}$
and    complex numbers $(a_n)_{n\in J}$,  we say that
\[
D(s) := \sum_{n \in J} a_n e^{-\lambda_n s}, \quad\, s \in \mathbb{C}
\]
is a $\lambda$-Dirichlet polynomial supported on  $J$.
 For the frequency $\lambda = (n)_{n \in \mathbb{N}}$ we obtain (after the substitution $z = e^{-s}$)
polynomials $\sum_{n \in J} a_n z^n$ in one complex variable (supported on $J$), and in the case $\lambda = (\log n)_{n \in \mathbb{N}}$ all ordinary Dirichlet
polynomials $\sum_{n \in J} a_n n^{-s}$ (supported on $J$).

Denote by $\mathcal{H}_\infty^{J}(\lambda)$ the finite dimensional Banach space of all
$\lambda$-Dirichlet polynomials supported
on the finite index subset  $J \subset \mathbb{N}$, endowed with the norm
\begin{equation}\label{norm-di}
\|D\|_\infty :=  \sup_{t \in \mathbb{R}} \Big|\sum_{n \in J} a_n e^{-i\lambda_n t}\Big|
=
   \sup_{\re s >0 } \big|\sum_{n \in J} a_n e^{-\lambda_n s}\big|\,,
\end{equation}
where the last equality is a simple consequence of the maximum modulus principle.
The main goal of this chapter is to study the projection constant $$\boldsymbol{\lambda}\big(\mathcal{H}_\infty^{J}(\lambda)\big)$$
for various 'natural' frequencies $\lambda$ and 
 various 'natural'
finite index sets  $J$ of $\mathbb{N}$. Given $x \in \mathbb{N}$, we are particularly interested  in the projection constant of the  Banach space
\[
\mathcal{H}_\infty^{\leq x}(\lambda) = \mathcal{H}_\infty^{\{n\in \mathbb{N}\colon n \leq x\}}(\lambda)\,,
\]
so  all $\lambda$-Dirichlet polynomials $D(s) = \sum_{n \leq x} a_n e^{-\lambda_n s}$ of length $x$.

Let us pause for a moment to say a few words about the modern theory of Dirichlet series.
Within the last two decades, the theory of ordinary Dirichlet series $\sum a_{n} n^{-s}$ experienced a kind of renaissance.
The study of these series in fact  was   one of the hot topics in mathematics at the beginning of the 20th century. Among others,
H. Bohr, Besicovitch, Bohnenblust, Hardy,  Hille, Landau, Perron, and M.~Riesz
were  the leading mathematicians in this issue.

However, this research took place before the modern interplay
between function theory and functional analysis, as well as the
advent of the field of several complex variables, and the area was
in many ways dormant until the late 1990s. One of the main goals of
the 1997 paper of Hedenmalm, Lindqvist, and Seip \cite{hedenmalm1997hilbert} was to
initiate a systematic study of Dirichlet series from the point of
view of modern operator-related function theory and harmonic
analysis. Independently, at the same time, a paper of Boas and
Khavinson \cite{bohnenblust1931absolute} attracted renewed attention, in the context of
several complex variables, to the original work of Bohr.

A new field emerged intertwining the classical work in novel ways with modern functional analysis, infinite dimensional holomorphy, probability theory as well as analytic number theory. As a consequence, a number of  challenging research problems crystallized and were solved over the last decades. We refer to the monographs\cite{defant2019libro}, \cite{HelsonBook},  and \cite{queffelec2013diophantine},
where many   of the   key elements of this new developments
 for ordinary Dirichlet series are described in detail.

Contemporary research in this field owes much to the following
fundamental observation of H. Bohr \cite{bohr1913ueber}: By the
transformation $z_j=\mathfrak{p}_j^{-s}$  and the fundamental theorem of arithmetics, an ordinary
Dirichlet series $\sum a_{n} n^{-s}$  may be thought of as a function
$\sum_{\alpha \in \mathbb{N}_0^{(\mathbb{N})}} a_{\mathfrak{p}^\alpha} z^{\alpha}$ of infinitely many
complex variables $z_1, z_2, ...$, where $\mathfrak{p} =(\mathfrak{p}_j)$ stands for the sequence of prime
numbers. By a classical approximation
theorem of Kronecker, this is more than just a formal
transformation: If, say, only a finite number of the coefficients
$a_n$ are nonzero (so that questions about convergence of the series
are avoided), the supremum of the Dirichlet polynomial $\sum a_n
n^{-s}$ in the half-plane $\operatorname{Re} s>0$ equals the
supremum of the corresponding polynomial on the infinite-dimensional
polydisc $\mathbb{D}^\infty$.

In \cite{bohr1913ueber} Bohr  considered the number $S$
that gives the maximal width of the band on which a Dirichlet series  converges uniformly but not absolutely.
Bohr showed that $S \leq 1/2$ . The problem of whether or not this was the correct value remained open
for some 15 years, until in their groundbreaking work from \cite{bohnenblust1931absolute}  Bohnenblust and
Hille showed that homogeneous
polynomials - the basic building blocks of functions analytic on
polydiscs - may, via the method of polarization, be transformed into
symmetric multilinear forms. They  used this
revolutionary insight  to prove that Bohr's upper estimate indeed is  optimal: $S= 1/2$.
In retrospect, one may in the work of Bohr and Bohnenblust-Hille
see the seeds of the modern  theory of  Dirichlet series.

Mainly inspired by the classical  monograph \cite{HardyRiesz} of Hardy and  M. Riesz from 1915,  recent research
  suggest a modern study of general Dirichlet series
$$D=\sum a_{n}e^{-\lambda_{n}s}\,$$
for arbitrary frequency $\lambda=(\lambda_{n})$
(see, e.g.,
\cite{bayart2021convergence}, 
\cite{DefantSchoolmann2018},
\cite{defantschoolmann2019Hptheory}, \cite{defant2020riesz}, \cite{defant2020variants} and 
\cite{HelsonBook}).
 Making the jump from the ordinary case $\lambda=(\log n)$ to arbitrary frequencies reveals serious difficulties.
  Whereas modern Fourier analysis and infinite dimensional holomorphy enrich the theory of ordinary Dirichlet series considerably, facing  general Dirichlet series   many of these powerful bridges
  unfortunately seem to collapse. New questions arise which make the theory of general Dirichlet series quite  challenging.
 Whereas the dominant  tool of the early days of this theory was complex analysis,
the idea now is to implement  modern techniques like  functional analysis,  probability theory, Fourier analysis, and harmonic analysis
on compact abelian groups.

\smallskip

\section{Integral formula}
We start with two important examples following easily from what we already know.
 For $x \in \mathbb{N}$ and the frequency $\lambda_0 = 0, \lambda_1 = 1,  \lambda_2 = 2, \ldots $ the identification
\[
\mathcal{H}_\infty^{\leq x}(\lambda) \,\, = \,\, \text{Trig}_{\{n\in \mathbb{N}_0\colon n \leq x\}}(\mathbb{T})
\,, \,\,\,\,\,\, \sum_{n=0}^x a_{n} e^{-\lambda_n s} \mapsto  \sum_{n=0}^{x} a_{n} z^{n}
\]
is obviously isometric, and hence the  integral formula
\begin{align} \label{inte1}
\boldsymbol{\lambda}\big(\mathcal{H}_\infty^{\leq x}(\lambda)\big) =
\int_{\mathbb{T}} \Big|\sum_{n=0}^{x}  z^k\Big|\,dz =  \frac{4}{\pi^2} \log(x+1) + o(1)
\end{align}
is an immediate consequence of the Lozinski-Kharshiladze
theorem from \eqref{LoKa} (see also  Corollary~\ref{LoKha}).

The  second   example considers the  frequency $\lambda = (\log \mathfrak{p}_n)$,
where $\mathfrak{p}_n$ is the $n$th prime number. This is the  canonical example of a $\mathbb{Q}$-linearly
independent frequency.
 It is well-known (see, e.g., \cite[Theorem 9.2]{katznelson2004introduction})
that, if
$\lambda_1, \ldots, \lambda_n$ are real numbers,  which are linearly independent over $\mathbb{Q}$, and $\lambda_0=0$, then for any choice of complex numbers $a_0, a_1, \ldots , a_x$ we have
\begin{equation*}\label{bohr-in}
\sum_{n=0}^x |a_n|  =  \sup_{ t \in \mathbb{R}} \Big|\sum_{n=0}^x a_n e^{-i \lambda_n t}\Big|\,.
\end{equation*}
In other terms, for any frequency $\lambda$, which is linearly independent over $\mathbb{Q}$ and for any finite subset $J \subset \mathbb{N}$ the identification
\[
\mathcal{H}_\infty^{J}(\lambda) \,\, = \,\, \ell_1^{|J|}\,, \,\,\quad \sum_{n \in J} a_n e^{-\lambda_n s} \mapsto (a_n)_{n \in J}
\]
is isometric. In this case, the following integral and limit formula are then  immediate consequences of \eqref{grunbuschC-B}
and \eqref{koenigschuetttomczak}.

\begin{remark} \label{logp}
  Let $\lambda $ be a $\mathbb{Q}$-linearly independent frequency and  $J \subset \mathbb{N}$ a finite subset. Then
\begin{align*}
\boldsymbol{\lambda}\big(\mathcal{H}_\infty^{J}(\lambda)\big) =    \int_{\mathbb{T}^{|J|}} \Big|\sum_{n \in J}  z_k\Big| dz\,.
\end{align*}
Moreover,
\begin{equation*}
\lim_{x \to \infty}  \frac{\boldsymbol{\lambda}\big(\mathcal{H}_\infty^{\leq x}(\lambda)\big)}{\sqrt{x}}
= \frac{\sqrt{\pi}}{2}\,.
\end{equation*}
\end{remark}

As already mentioned, $\lambda = (\log \mathfrak{p}_n)$ is the canonical example.
As a sort of highlight of this chapter,   we in  Theorem~\ref{harpo}  are going to see an analog of the preceding
remark for the  ordinary frequency $\lambda = (\log n)$. One of the  main ingredient of its proof is a recent deep result of Harper~\cite{Harper}
from probability related analytic number theory.

The first step towards all of this  is the following counterpart of Theorem~\ref{C(G)proj}
for $\lambda$-Dirichlet polynomials, which extends  the integral formulas from \eqref{inte1} and Remark~\ref{logp}
(to understand that this is in fact an extension see  Theorem~\ref{main-dirB} and combine it with  Example~\ref{examples2}).

\smallskip

\begin{theorem} \label{main-dir}
Let  $J \subset \mathbb{N}$ be a finite subset and   $\lambda= (\lambda_n)$  a~frequency. Then
\begin{align*}
\boldsymbol{\lambda}\big(\mathcal{H}_\infty^{J}(\lambda)\big) =
\lim_{T \to \infty} \frac{1}{2T} \int_{-T}^T \Big|\sum_{n \in J} e^{-i\lambda_n t}\Big|\,dt\,.
\end{align*}
\end{theorem}

\smallskip

For the proof (see the end of this section) we need to adapt a few ideas from  \cite{defantschoolmann2019Hptheory}.
Our first aim is to describe the finite dimensional Banach space $\mathcal{H}_\infty^{J}(\lambda)$ in terms of a~Hardy
type space of functions on certain compact abelian groups.

In what follows, a~pair $(G, \beta)$ is said to be a~Dirichlet group if $G$ is a~compact abelian group and
$\beta \colon \mathbb{R} \to G$ is a~continuous homomorphism with dense range; as before we in the following denote the Haar measure on $G$ by $m$.
Then, by density,  the  dual map
\[
\widehat{\beta}\colon \widehat{G} \to \widehat{\mathbb{R}}
\]
is injective. Moreover, if we write  $\mathbb{R}$ for the group $(\mathbb{R},+)$ endowed with its natural topology, then the homomorphism $ \mathbb{R}  = \widehat{\mathbb{R}} \,, \,\,\,x \mapsto e^{ix\,\bullet}$
identifies  $\mathbb{R}$ and $\widehat{\mathbb{R}}$ as topological groups, and hence $\widehat{\beta}(\widehat{G} )$
may be interpreted as a  subset of $\mathbb{R}$.

We in particular observe  that if $(G, \beta)$ is a~Dirichlet
group and $\gamma \colon G\to \mathbb{T}$  a~character, then $\gamma \circ \beta \colon \mathbb{R} \to \mathbb{T}$ is
a~character on  $\mathbb{R}$. Thus there exists a~unique $x\in \mathbb{R}$ such that $\gamma \circ \beta(t)
=e^{ixt}$ for all $t\in \mathbb{R}$.
The following  result from \cite[Proposion~3.10]{defantschoolmann2019Hptheory} is crucial - we include a~proof for
the sake of completeness.

\begin{proposition} \label{basic}
For any Dirichlet group $(G, \beta)$ one has
\[
\int_G f(g)\,d\mathrm{m}(g) = \lim_{T \to \infty}\, \frac{1}{2T} \int_{-T}^T f\circ \beta(t)\,dt, \quad\, f\in C(G)\,.
\]
\end{proposition}

\begin{proof} Note first that for all  $\gamma\in \widehat{G}$ we have
\[
\int_G \gamma(g)\,d\mathrm{m}(g) = \lim_{T \to \infty}\, \frac{1}{2T} \int_{-T}^T \gamma\circ \beta(t)\,dt
\,.
\]
Indeed, as explained above  by the injectivity of the dual map $\widehat{\beta}$, we know that for every  $\gamma \in \widehat{G}$
there is a unique $x = x(\gamma) \in \mathbb{R}$ such that $\gamma \circ \beta(t) = e^{ixt}$ for all $t\in \mathbb{R}$. Now recall
that $\int_G \gamma\,d\mathrm{m}= 0$ \big(resp., $\int_G \gamma\,d\mathrm{m}= 1$\big) whenever $\gamma \neq 1$ (resp., $\gamma =1$). Clearly,
$\gamma \neq 1$ (resp., $\gamma = 1$) yields $x \neq 0$ (resp., $x =0$), so the above equality is obvious in this case. As a~consequence,
the claim also holds for all trigonometric polynomials  $f \in C(G)$. But since all trigonometric polynomials are dense in
$C(G)$, and moreover the set $\{\Phi_T \colon T >0\}$ of all linear functionals
\[
\Phi_T\colon  C(G) \to \mathbb{C}\,, \quad\, \Phi_T(f):= \frac{1}{2T} \int_{-T}^T f \circ \beta(t)dt\,
\]
is uniformly bounded, the conclusion is a~consequence of the Banach-Steinhaus theorem.
\end{proof}

Given a~frequency $\lambda=(\lambda_n)$,  we need Dirichlet groups which are adapted to $\lambda$. We call the pair $(G,\beta)$
a~$\lambda$-Dirichlet group, whenever $\{\lambda_n \colon n \in \mathbb{N}\}\subset \widehat{\beta}(\widehat{G}) $, i.e., for every
character $e^{-i\lambda_n \bullet}: \mathbb{R}\to \mathbb{T}$ there is a~character $h_{\lambda_n} \in \widehat{G}$ (which then is unique)
such  that the following diagram commutes:
\begin{equation*}
\begin{tikzpicture}[scale = 0.8]
        \node (G) at (0,0) {$G$};
        \node (T) at (3,0) {$\mathbb{T}$};
        \node (R) at (0,-2) {$\mathbb{R}$};

        \draw[-latex] (G) -- node[above] {$h_{\lambda_n}$} (T);
        \draw[-latex] (R) -- node[below right] {$e^{-i\lambda_n \bullet}$} (T);
        \draw[-latex] (R) -- node[left] {$\beta$} (G);
    \end{tikzpicture}
    \end{equation*}
For all needed information on $\lambda$-Dirichlet groups see \cite{defantschoolmann2019Hptheory}.

Let us collect some examples.
The very first is the  frequency $\lambda = (n)$ for which the pair $(\mathbb{T},\beta_{\mathbb{T}} )$ with
\[
\beta_{\mathbb{T}} : \mathbb{R} \rightarrow \mathbb{T} \,, \,\, t \mapsto e^{-it}
\]
obviously forms a $\lambda$-Dirichlet group. Identifying  $\widehat{\mathbb{T}} = \mathbb{Z}$ we get
that $h_n(z) = z^n$ for $z \in \mathbb{T}, n \in \mathbb{Z}$.

The second example  shows that $\lambda$-Dirichlet groups for any possible frequency $\lambda$ always exist.
In fact, given a frequency $\lambda$, the so-called Bohr compactification of the reals  forms a $\lambda$-Dirichlet group.

\smallskip

\begin{example}
\label{examples1}
Denote by $\overline{\mathbb{R}}:=\widehat{(\mathbb{R},+,d)}$  the   Bohr compactification of $\mathbb{R}$, where $d$ stands
for the discrete topology. This is a compact abelian group,  which together with the mapping
\begin{equation*}
\beta_{\overline{\mathbb{R}}}\colon \mathbb{R} \hookrightarrow \overline{\mathbb{R}},\,\, x \mapsto \left[ t \mapsto e^{-ixt}\right]
\end{equation*}
forms a $\lambda$-Dirichlet group for every frequency $\lambda$.
\end{example}

But for  concrete  frequencies $\lambda$, there often are $\lambda$-Dirichlet groups which in a~sense reflect their structure more
naturally than the Bohr compactification.

\begin{example} \label{examples2}
Let $\lambda = (\log n)$. Then the infinite dimensional torus
$\mathbb{T}^\infty$ together with the so-called Kronecker flow 
\[
\beta_{\mathbb{T}^\infty}\colon  \mathbb{R}  \rightarrow \mathbb{T}^\infty\,, \, \quad t \mapsto (\mathfrak{p}_k^{-it})\,,
\]
where $\mathfrak{p}_k$ again denotes the $k$th prime, forms a $\lambda$-Dirichlet group.  This is basically a~consequence of  Kronecker's approximation
theorem (for an alternative 'harmonic analysis' argument see \cite[Example 3.7]{defantschoolmann2019Hptheory}). Note that, identifying 
$$\mathbb{Z}^{(\mathbb{N})}
=
\widehat{\mathbb{T}^\infty}\,,
\quad \alpha \mapsto [z \mapsto z^\alpha] $$ 
($\mathbb{Z}^{(\mathbb{N})}$ all finite sequences of integers), we for every $z \in \mathbb{T}^\infty$
have
\begin{equation}\label{bohrtrafo}
h_{ \log n}(z)=h_{\sum \alpha_j \log \mathfrak{p}_j}(z) = z^\alpha\,,
\end{equation}
where $n=\mathfrak{p}^\alpha$ with $\alpha \in \mathbb{Z}^{(\mathbb{N})}$.
\end{example}

As announced, we now reformulate the Banach space $\mathcal{H}_\infty^{J}(\lambda)$  of all $\lambda$-Dirichlet polynomials
supported on the finite index set $J \subset \mathbb{N}$ as a Hardy space of functions on a $\lambda$-Dirichlet group.
Given a~finite set $J \subset \mathbb{N}$ and a~$\lambda$-Dirichlet group $(G,\beta)$, we
write
\[
H_\infty^{\lambda, J}(G)
\]
 for the subspace of all $\lambda$-Dirichlet polynomials
$f= \sum_{n \in J} \hat{f}(h_{\lambda_n}) h_{\lambda_n}$ in $L_\infty(G)$ (with respect to the Haar measure $\mathrm{m}$ on $G$).

The following equivalent description of $\mathcal{H}_\infty^{J}(\lambda)$ 
is now obvious. 

\begin{proposition} \label{H=H}
Let  $J \subset \mathbb{N}$ be a~finite subset, and $(G,\beta)$ a~$\lambda$-Dirichlet group. Then the Bohr map
\[
\mathcal{B}: \mathcal{H}_\infty^{J}(\lambda) \rightarrow  H_\infty^{\lambda, J}(G), \quad\,
\sum_{n \in J} a_n e^{-\lambda_n s} \mapsto \sum_{n \in J} a_n h_{\lambda_n}
\]
defines an isometric linear bijection which preserves Bohr and Fourier coefficients.
\end{proposition}

\begin{proof}
The collections $(e^{-i\lambda_n \bullet})_{n\in J}$ and $(h_{\lambda_n})_{n\in J}$ are both linearly independent. Hence it  suffices to check that for every collection $(a_n)_{n\in J}$ of complex scalars, we have 
\[
\sup_{t \in \mathbb{R}} \Big|\sum_{n \in J} a_n e^{-i\lambda_n t}\Big|
= \sup_{g \in G} \Big|\sum_{n \in J} a_n h_{\lambda_n}(g)\Big|
\]
(see again \eqref{norm-di}). But this is clear by $e^{-i\lambda_n \bullet} =
h_{\lambda_n} \circ \beta$ for each $n \in J$ and $\beta$ has dense range.
\end{proof}

Finally, we are in the position to prove Theorem~\ref{main-dir}.

\begin{proof}[Proof of Theorem~\ref{main-dir}]
Choose some $\lambda$-Dirichlet group $(G, \beta)$ (this is possible by  Example~\ref{examples1}). Then by
Proposition~\ref{H=H} and Theorem~\ref{C(G)proj} one has
\[
\boldsymbol{\lambda}\big(\mathcal{H}_\infty^{J}(\lambda)\big) = \boldsymbol{\lambda}\big( H_\infty^{\lambda, J}(G)\big)
= \int_G \Big|  \sum_{n \in J} h_{\lambda_n}\Big|\,d\mathrm{m}\,.
\]
Applying now Proposition~\ref{basic}, we get
\[
\int_G \Big|  \sum_{n \in J} h_{\lambda_n}\Big|\,d\mathrm{m}
= \lim_{T \to \infty} \frac{1}{2T} \int_{-T}^T \Big|\sum_{n \in J} e^{-i\lambda_n t}\Big|\,dt\,,
\]
 so this completes the proof.
\end{proof}

\smallskip

We add a 'group version' of Theorem~\ref{main-dir}, which in view of Proposition~\ref{basic} is immediate.

\begin{theorem} \label{main-dirB}
Let  $J \subset \mathbb{N}$ be a finite subset, $\lambda$ a~frequency, and $(G,\beta)$ a~$\lambda$-Dirichlet group
with Haar measure $m$. Then
\begin{align*}
\boldsymbol{\lambda}\big(\mathcal{H}_\infty^{J}(\lambda)\big) =  \int_G \Big|  \sum_{n \in J} h_{\lambda_n}\Big|\,d\mathrm{m}\,.
\end{align*}
\end{theorem}

\smallskip
Within the setting of ordinary Dirichlet series $\sum a_n n^{-s}$ the coming two sections  collect applications  applying
some analytic number theory.

Here we finish with an application to $\mathbb{Q}$-linearly independent frequencies, or more generally to
frequencies which satisfy the so-called $\textbf{B}_2$-condition.
A frequency  $\lambda$ satisfies the  $\textbf{B}_2$-condition, whenever  the distances between two elements of $\lambda$
are always different, i.e., if $\lambda_j - \lambda_r = \lambda_s - \lambda_k$, then $j=r$ and $s=k$ or $j=s$ and $r=k$.
An example is $\lambda=(\log \mathfrak{p}_n)$ or more generally every $\mathbb{Q}$-linearly independent frequency.

In Section~\ref{An integral formula} we defined when a~set of characters on  a compact abelian group satisfies the
$\textbf{B}_2$-condition, and how this notion is related with the $\pmb{\Lambda}(2)$-condition.

\begin{remark}
Let $\lambda$ be a frequency and $(G,\beta)$ a~$\lambda$-Dirichlet group. Then $\lambda$  satisfies the $\textbf{B}_2$-condition
if and only if  $\big\{  h_{\lambda_n} \colon n \in \mathbb{N} \big\} \subset \widehat{G}$ satisfies the $\textbf{B}_2$-condition.
Indeed, this is immediate from the facts that $h_{\lambda_n} = \widehat{\beta}^{-1} \lambda_n$ and $h_{\lambda_j} h_{\lambda_k}
= h_{\lambda_j +\lambda_k}$ for all $j,k,n \in \mathbb{N}$.
\end{remark}

Recall from Lemma~\ref{b2} that every $\textbf{B}_2$-set of characters on a compact abelian group  is a $\pmb{\Lambda}(2)$-set
with $\pmb{\Lambda}(2)$-constant $\leq \sqrt{2}$. Hence Theorem~\ref{main-dirB} and Corollary~\ref{corB2} imply the following
corollary.

\begin{corollary}\label{b2dirichelet}
Let $\lambda$ be a frequency and $(G,\beta)$ a $\lambda$-Dirichlet group with the property that all characters
$h_{\lambda_n} \in \widehat{G}$ form a  $\pmb{\Lambda}(2)$-set with constant $C_2(\lambda)$.  Then
\begin{equation*}
\frac{1}{C_2(\lambda)}\sqrt{x} \,\, \leq \,\,
\boldsymbol{\lambda}\big(\mathcal{H}_\infty^{\leq x}(\lambda)\big)   \,\, \leq \,\,  \sqrt{x} \,.
\end{equation*}
In particular, if $\lambda$  satisfies the $B_2$-condition, then
\begin{equation*}
\frac{1}{\sqrt{2}} \sqrt{x}
\,\, \leq \,\,
\boldsymbol{\lambda}\big(\mathcal{H}_\infty^{\leq x}(\lambda)\big) \,\, \leq \,\,  \sqrt{x} \,.
\end{equation*}
\end{corollary}

As mentioned the frequency $\lambda = (\log \mathfrak{p}_n)$ is an important example -- already treated
in Remark~\ref{logp}.

\smallskip

\section{Applying number theory}

Let us turn to the ordinary frequency $\lambda = (\log n)$, which generates ordinary Dirichlet series
$\sum_n a_n n^{-s}$, and fix some finite set   $J \subset \mathbb{N}$ of indices. If there is no risk for confusion
with other frequencies, then  we use the simplified notation
\[
\mathcal{H}_\infty^{J} = \mathcal{H}_\infty^{J}\big((\log n) \big)\,.
\]
Moreover, we define the two numbers
\begin{equation*}
\text{ $\pi(J) = \max_{n \in J} \pi(n)$ \,\,\,\,\,\,and \,\,\,\,\,\, $\Omega(J) = \max_{n \in J} \Omega(n)$\,,}
\end{equation*}
and the index set
\begin{equation*}
\Delta(J) = \big\{  \alpha \in \mathbb{N}_0^{\pi(\mathbb{N})}\colon  n \in J\,, \,\, \mathfrak{p}^\alpha = n\big\}\,,
\end{equation*}
where $\pi(n)$ counts all primes $\leq n$ and $\Omega(n)$ is the number of prime divisors of $n$
counted according to their multiplicities.

\begin{proposition} \label{H=Hord}
Let $J \subset \mathbb{N}$ be a finite subset. Then the so-called Bohr map
\[
\mathcal{B} \colon  \mathcal{H}_\infty^{J} \to \mathcal{P}_{\Delta(J)}\big(\ell_{\infty}^{\pi(J)}\big)\,,\,\,\,\,\,\,
\sum_{n \in J} a_n e^{-\lambda_n s} \mapsto \sum_{\alpha  \in \Delta(J)} a_{\mathfrak{p}^\alpha} z^\alpha
\]
defines an isometric linear bijection which preserves Bohr and Fourier coefficients.
\end{proposition}

\begin{proof} This is an  immediate consequence of Example~\ref{examples2} together with  \eqref{bohrtrafo},
Proposition~\ref{H=H}, and the maximum modulus principle (more precisely, the fact that the modulus of a~polynomial on
$\ell_{\infty}^{\pi(J)}$
attains its maximum on the distinguished boundary ~$\mathbb{T}^{\pi(J)}$).
\end{proof}

In view of the preceding result, the study of the projection constant $\boldsymbol{\lambda}\big(\mathcal{H}_\infty^{J}\big)$
reduces to the study of the projection constant of the Banach space of all polynomials on   $\ell_{\infty}^{\pi(J)}$ (or equivalently, all analytic trigonometric polynomials on $\mathbb{T}^{\pi(J)}$) supported on the index set $\Delta(J)$.
The following theorem is a first example illustrating this general strategy.

\begin{theorem} \label{abstract}
Let  $J \subset \mathbb{N}$ be a finite subset. Then
\begin{align*}
\boldsymbol{\lambda}\big(\mathcal{H}_\infty^{J}\big)
= \int_{\mathbb{T}^{\pi(J)}} \Big| \sum_{\alpha \in \Delta(J)}z^\alpha\Big|\,d\mathrm{m}\,.
\end{align*}
Moreover,
\begin{equation*}
\frac{1}{\sqrt{2^{\Omega(J)}}}\sqrt{|J|} \,\,\, \leq \,\,\,
\boldsymbol{\lambda}\big(\mathcal{H}_\infty^{J}\big) \leq  \sqrt{|J|} .
\end{equation*}
\end{theorem}

\noindent
Note that we here may also integrate over the infinite dimensional $\mathbb{T}^\infty$ instead of $\mathbb{T}^{\pi(J)}$.

\begin{proof}
The proof is an immediate consequence of Proposition~\ref{H=Hord} and Corollary~\ref{manydimensionsB}.
Indeed, the degree of the index set $\Delta(J)$ is at most $\Omega(J)$, and the cardinalities of $J$ and $\Delta(J)$ are the same.
\end{proof}

\smallskip

We illustrate Theorem~\ref{abstract} with a few more concrete examples. Using  for $m,x \in \mathbb{N}$ the  notation
\[
N_m(x) = \Big\{ 1 \leq  n \leq x \colon n = \mathfrak{p}^\alpha \,\, \,\,\, \text{with} \, \, \,\,\,
|\alpha| = m\Big\}\,,
\]
we write
\[
\mathcal{H}_\infty^{\leq x, m} := \mathcal{H}_\infty^{N_m(x)}\,,
\]
for the Banach space  of all so-called $m$-homogeneous Dirichet polynomials $\sum_{n=1}^{x} a_n\frac{1}{n^{s}}$ of length $x$, i.e.,
$a_n \neq 0$ implies that $n = \mathfrak{p}^\alpha$ with $|\alpha| =m$ (or equivalently, the index $n$ of every non-zero coefficient $a_n$
has up to multiplicities $m$ prime factors).

Again by Proposition~\ref{H=Hord},
\[
\mathcal{H}_\infty^{\leq x,m} = \mathcal{P}_{\Delta(N_m(x))}\big(\ell_{\infty}^{\pi(N_m(x))}\big)\,,
\]
where the identification, given by the Bohr map, is isometric and coefficient preserving.
For $m=1$ we immediately deduce from Remark~\ref{logp} and the prime number theorem that
\begin{align} \label{m=1}
\lim_{x \to \infty}  \frac{\boldsymbol{\lambda}\big(\mathcal{H}_\infty^{\leq x,1}\big)}{\sqrt{\pi{(x)}}}
=\lim_{x \to \infty}  \frac{\boldsymbol{\lambda}\big(\mathcal{H}_\infty^{\leq x,1}\big)}{\sqrt{\frac{x}{\log x}}}
= \frac{\sqrt{\pi}}{2}\,.
\end{align}
The following theorem partly extends the preceding  formula.

\smallskip

\begin{theorem} \label{projhom}
For $m \ge 1$
\begin{equation*}\label{m>1}
\boldsymbol{\lambda}\big(\mathcal{H}_\infty^{\leq x,m}\big) \,\,\,\sim_{C(m)}\,\,\,
\sqrt{\frac{x}{\log x}}\,\,\,\sqrt{\frac{(\log\log x)^{m-1}}{(m-1)!}}\,.
\end{equation*}
Moreover,  the constant $C(m)$  is independent of $m$, whenever $m\le \frac{\log\log x}{2e}$.
\end{theorem}

\begin{proof}
Since $\pi(N_m(x)) \leq \pi(x)$ and $\Omega(N_m(x) )) = m$,  by Theorem~\ref{abstract}
\begin{equation*}
\frac{1}{\sqrt{2^m}}\sqrt{|N_m(x) |} \,\,\, \leq \,\,\,
\boldsymbol{\lambda}\big(\mathcal{H}_\infty^{\leq x, m } \big)\leq  \sqrt{|N_m(x)|}\,.
\end{equation*}
Then the   first claim follows from a well-known result of Landau (see, e.g., \cite[p.~200]{tenenbaum1995introduction})
showing that
\begin{equation*}\label{landau}
|N_m(x)| \,\,\,\sim_{C(m)}\,\,\, \frac{x}{\log x} \frac{(\log \log x)^{m-1}}{(m-1)!} \,\,\,\,\,\,\,\,\,\,\,\,\,\,
\text{as $x \to \infty$}\,.
\end{equation*}
 For the proof of the second claim we use again that by Theorem~\ref{abstract}
\[
\boldsymbol{\lambda}\big(\mathcal{H}_\infty^{\leq x,m}\big) = \int_{\mathbb{T}^\infty} \Big| \sum_{\alpha \in \Delta(N_m(x))} z^\alpha\Big|\,dz\,.
\]
But in \cite[Equation (13),  p.107]{bondarenkoseip2016} it is shown that there is
a~universal constant $C \ge 1$ such that for any $x, m$ with $m\le \frac{\log\log x}{2e}$
\[
\int_{\mathbb{T}^\infty} \Big| \sum_{\alpha \in \Delta(N_m(x))} z^\alpha\Big|\,dz \,\,\,\sim_C\,\,\,
\Bigg(\int_{\mathbb{T}^\infty} \Big| \sum_{\alpha \in \Delta(N_m(x))} z^\alpha\Big|^2 dz\Bigg)^{\frac12} \, = \, \sqrt{|N_m(x)|}\,.
\]
On the other hand, the so-called  Sathe-Selberg formula (see, e.g., \cite{ErdoesSarkozy})
asserts that for every $\delta>0$ there is $c(\delta) > 0$ such that for $m\le (2-\delta)\log\log x$
\[
|N_m(x)| \,\,\,\sim_C\,\,\, \frac{x}{\log x} \frac{(\log \log x)^{m-1}}{(m-1)!} \,\,\,\,\,\,\,\, \text{as $x \to \infty$}\,.
\]
This completes the argument.
\end{proof}

\smallskip

Recall that, given $x \in \mathbb{N}$, we by
$$
\mathcal{H}_\infty^{\leq x}= \mathcal{H}_\infty^{\leq x}\big((\log n) \big)
$$
denote  the Banach space of all ordinary Dirichlet polynomials $D=\sum_{n=1}^{x} a_n n^{-s}$ of length $x$  endowed with
the sup norm on the imaginary line -- in other terms,
\[
\mathcal{H}_\infty^{\leq x} =\mathcal{H}_\infty^{N(x)}\big((\log n) \big)\,,
\]
where $ N(x) = \{n \colon 1 \leq n \leq x\}$.

A~recent deep theorem of Harper from \cite{Harper} shows that
\begin{equation}\label{harperA}
\lim_{T \to \infty }\frac{1}{2T} \int_{-T}^T \Big|\sum_{n=1}^x \frac{1}{n^{it}}\Big| dt \,\,\,\,\,\,\sim_C\,\,\,\,\,\,
\frac{\sqrt{x}}{(\log \log x)^{\frac{1}{4}}}\,,
\end{equation}
which in  view of Proposition~\ref{basic} and Example \ref{examples2} is equivalent to
\begin{equation}\label{harperB}
  \int_{\mathbb{T}^\infty} \Big| \sum_{\alpha \in \mathbb{N}^{(\mathbb{N})}_0: 1 \leq n \leq x} z^\alpha\Bigg|\,dz
\,\,\,\,\,\,\sim_C\,\,\,\,\,\,
\frac{\sqrt{x}}{(\log \log x)^{\frac{1}{4}}}\,.
\end{equation}
Combining Theorem~\ref{abstract} and Harper's result from \eqref{harperA}, we obtain the precise asymptotic order of
the projection constant of the ($x$-dimensional) Banach space $\mathcal{H}_\infty^{\leq x}$.

\smallskip

\begin{theorem}\label{harpo}
\begin{align*}
\boldsymbol{\lambda}\big(\mathcal{H}_\infty^{\leq x}\big) \,\,\,\sim_C\,\,\,
\frac{\sqrt{x}}{(\log \log x)^{\frac{1}{4}}}\,
\end{align*}
\end{theorem}

\smallskip

It seems interesting to rephrase this result again in terms of multivariate polynomials.
Recall from Section~\ref{index-sets} that $\Delta(x) = \{  \alpha \in \mathbb{N}_0^{\pi(x)}\colon 1 \leq \mathfrak{p}^\alpha \leq x\}$,
where $x \geq 1$ and $\pi(x)$  as usual counts the primes $\mathfrak{p} \leq x$. Then by Proposition~\ref{H=Hord} we have that
\[
\mathcal{H}_\infty^{\leq x} = \mathcal{P}_{\Delta(x)}\big(\ell_{\infty}^{\pi(x)}\big)\,,
\]
where the identification is given by the Bohr transform -- being isometric and coefficient preserving. Hence the following result
is a~consequence of Theorem~\ref{harpo} and Theorem~\ref{C(G)proj}.

\begin{corollary}
\[
\boldsymbol{\lambda}\big(\mathcal{P}_{\Delta(x)}\big(\ell_{\infty}^{\pi(x)}\big)\big) \,\,\,\sim_C\,\,\, \frac{\sqrt{x}}{(\log\log x)^{1/4}}\,.
\]
Moreover,
\begin{align*}
\boldsymbol{\lambda}\big(\mathcal{P}_{\Delta(x)}(\ell_{\infty}^{\pi(x)})\big)
\,= \,\big\|\mathbf{P}_{\Delta(x)}: C(\mathbb{T}^{\pi(x)}) \rightarrow \mathcal{P}_{\Delta(x)} (\ell_\infty^{\pi(x)})\big\|
\,= \,\int_{\mathbb{T}^{\pi(x)}} \big| \sum_{\alpha \in \Delta(x)} z^\alpha \big| d z\,,
\end{align*}
where $\mathbf{P}_{\Delta(x)}$ stands for the restriction of the orthogonal projection on $L_2(\mathbb{T}^{\pi(x)})$
onto $\mathcal{P}_{\Delta(x)} (\ell_\infty^{\pi(x)})$\,.
\end{corollary}

\smallskip

\section{Comparing  
Sidon constants
}
From Chapter~\ref{Unconditionality} we know  that the unconditional basis constant  and the projection  constant for spaces
$\mathcal{P}_J(X_n)$ of polynomials on Banach spaces $X_n = (\mathbb{C}^n, \|\cdot\|)$, which are supported on finite index sets
$J \subset \mathbb{N}_0^{(\mathbb{N})}$, are intimately linked. Here we want to compare some of the preceding estimates on projection
constants for spaces of ordinary Dirichlet polynomials with their (recently published) analogs for unconditional basis constants.

Given a finite subset $J$ of $\mathbb{N}$, it is obvious that the 'monomials' $n^{-s}, \, n \in J$ form a basis of  the Banach space
$\mathcal{H}_\infty^{J}$, and so we denote by
\[
\boldsymbol{\chimon}  \big(\mathcal{H}_\infty^{J} \big)
\]
the unconditional basis constant of all $n^{-s}, \, n \in J$. Since by Proposition~\ref{H=Hord} there is a~coefficient preserving
isometric  identification
\[
\mathcal{H}_\infty^{J} = \mathcal{P}_{\Delta(J)}\big(\ell_{\infty}^{\pi(J)}\big)\,,
\]
we have
\[
\boldsymbol{\chimon}  \big(\mathcal{H}_\infty^{J} \big)\,=\,\boldsymbol{\chimon}
\big( \mathcal{P}_{\Delta(J)}\big(\ell_{\infty}^{\pi(J)}\big)\big)\,.
\]
On the other hand, we also have the identity
\[
\mathcal{P}_{\Delta(J)}\big(\ell_{\infty}^{\pi(J)}\big) = \text{Trig}_{\Delta(J)}\big(\mathbb{T}^{\pi(J)}\big)\,,
\]
which shows that $\boldsymbol{\chimon}  \big(\mathcal{H}_\infty^{J} \big)$ may be interpreted as the Sidon constant of the
characters $z^\alpha, \, \alpha \in \Delta(J)$  on the compact abelian group  $\mathbb{T}^{\pi(J)}$.

Recall that given
a~topological group $G$, the Sidon constant of a finite set $\Gamma$ of characters in the dual group $\widehat{G}$ is the best
constant $c\ge 0$ such that for every trigonometric polynomial $f = \sum_{\gamma \in \Gamma} c_\gamma \gamma$ on $G$
\begin{equation}\label{sidony}
\sum_{\gamma \in \Gamma} |c_\gamma| \leq c \|f\|_\infty\,.
\end{equation}

As defined above  $\mathcal{H}_\infty^{\leq x} $ stands for the Banach space of all Dirichlet polynomials
$D(s) = \sum_{n=1}^{x}a_n n^{-s}$ of length $x$ and $\Delta(x) = \big\{  \alpha \in \mathbb{N}_0^{\pi(\mathbb{N})}\colon  n \leq x\,,
\,\, \mathfrak{p}^\alpha = n\big\}\,.$
Then
\begin{equation} \label{annals}
\boldsymbol{\chimon} \big( \mathcal{H}_\infty^{\leq x} \big) \sim_C \frac{\sqrt{x}}{e^{\big(\frac{1}{\sqrt{2}} + o(1)\big)
\sqrt{\log x \log \log x}}}\,,
\end{equation}
and equivalently
\begin{equation*}
\boldsymbol{\chimon} \big(\mathcal{P}_{\Delta(x)}(\ell_\infty^{\pi(x)}) \big) \sim_C \frac{\sqrt{x}}{e^{\big(\frac{1}{\sqrt{2}}
+ o(1)\big) \sqrt{\log x \log \log x}}}\,\,;
\vspace{2mm}
\end{equation*}
this result is taken from \cite[Theorem~3]{defant2011bohnenblust} (see also \cite[Theorem~9.1]{defant2019libro}), and it is
the final outcome of a~long lasting  research project started by  Queff\'elec, Konyagin, and de la Bret\`eche
(see \cite[Section~9.3]{defant2019libro} for more details on its history).

In the preceding section we proved  the  counterpart for projection constants. From Theorem~\ref{harpo} (based among others
on Harper's asymptotic from \eqref{harperA}) we know that
\[
\boldsymbol{\lambda}\big( \mathcal{H}_\infty^{\leq x} \big) = \boldsymbol{\lambda} \big(\mathcal{P}_{\Delta(x)}(\ell_\infty^{\pi(x)})\big)
\sim_C \frac{\sqrt{x}}{(\log \log x)^{\frac{1}{4}}}\,.
\]

Analog results are known in the $m$-homogeneous case. By Balasubramanian-Calado-Queff\'elec \cite{balasubramanian2006bohr} (a~result
elaborated in \cite[Theorem~9.4]{defant2019libro}) one has
\[
\boldsymbol{{\chimon}} \big(\mathcal{H}_\infty^{\leq x,m} \big) = \boldsymbol{\chimon}\big(\mathcal{P}_{\Delta(x,m) }(\ell_\infty^{\pi(x)}) \big)
\sim_{C(m)} \frac{x^{\frac{m-1}{2m}}}{( \log x)^{\frac{m-1}{2}}}\,,
\]
whereas the corresponding result for projection constants from Theorem~\ref{projhom} reads
\[
\boldsymbol{\lambda}\big(\mathcal{H}_\infty^{\leq x,m}\big) =
\boldsymbol{\lambda} \big(\mathcal{P}_{\Delta(x,m) }(\ell_\infty^{\pi(x)}) \big) \,\,\,\sim_{C(m)}\,\,\,
\sqrt{\frac{x}{\log x}}\sqrt{\frac{(\log\log x)^{m-1}}{(m-1)!}}\,.
\]

\bigskip

\chapter{Polynomials on the Boolean cube}\label{Polynomials on the Boolean cube}

In this chapter we  focus on some  local Banach space theory 
of
spaces of real valued functions  defined on  Boolean cubes. For $N \in \mathbb{N}$ we call $\{\pm 1\}^{N}:=\{-1, +1\}^N$ the $N$-dimensional Boolean cube. A Boolean function is any function $f : \{ \pm 1\}^{N} \rightarrow \{ \pm 1\}$, and more generally,  functions  $f : \{ \pm 1\}^{N} \rightarrow \mathbb{R}$ are said to be  'real valued functions on the  $N$-dimensional Boolean cube'. This kind of functions are essential in  theoretical computer science, and they also play a key role in combinatorics,
random graph theory, statistical physics, Gaussian geometry, the theories of metric spaces/ Banach spaces, or  social choice theory (see \cite{o2008some} and \cite{o2014analysis}). Moreover, the last decades show  growing interest for applications of Boolean functions in the context of quantum algorithms complexity and  quantum information (see \cite{beals2001quantum, defant2019fourier}).  

The study of  real valued functions on Boolean cubes is deeply influenced by Fourier analysis. Considering the Boolean cube $\{\pm 1\}^{N}$ as a~compact abelian group endowed with the coordinatewise product and the  discrete topology (so the Haar measure is given by the normalized counting measure), we may apply the machinery given by abstract harmonic analysis. In particular, the integral or expectation of  each function $f:\{ \pm 1\}^{N} \rightarrow \mathbb{R}$ is given by
\[
\mathbb{E}\big[ f \big] := \frac{1}{2^{N}} \sum_{x \in \{ \pm 1\}^{N}}{f(x)}\,.
\]
The characters on $\{ \pm 1\}^{N}$  are the so-called Walsh functions defined as
\[
\chi_{S}:\{ \pm 1\}^{N} \rightarrow \{ \pm 1\} \, ,
\hspace{3mm} \chi_{S}(x) = x^{S}:= \prod_{k \in S}{x_{k}} \,\, \,\,\,\, \text{for} \,\,\,\,\,\,x \in \{\pm 1\}^{N},
\]
where $S \subset [N]:=\{ 1, \ldots, N\}$ and  $\chi_{\emptyset}(x):= 1$ for each $x\in \{\pm 1\}^N$.
This allows to associate to each function $f:\{\pm 1\}^{N} \rightarrow \R$ its Fourier-Walsh expansion
\begin{equation*}\label{equa:FourierExpansionShort}
f(x) = \sum_{S \subset [N]}{\widehat{f}(S) \, x^S}\, , \hspace{3mm} x \in \{\pm 1\}^{N}\,,
\end{equation*}
where  $\widehat{f}(S) = \mathbb{E}\big[ f \cdot \chi_{S} \big]$ are the Fourier coefficients.
Given $d \in \mathbb{N}$, we say that  $f$ has degree $d = \deg f$ whenever  $\widehat{f}(S) = 0$ for all $|S| > d$, and $f$ is $d$-homogeneous
whenever additionally  $\widehat{f}(S) = 0$ provided  $|S| \neq d$.

For  every `index set'  $\mathcal{S}$ of subsets $S \subset [N] $, we define
the linear space
$$\mathcal{B}^{N}_{\mathcal{S}}$$
of all  functions $f:\{ \pm 1\}^{N} \rightarrow \mathbb{R}$, which have Fourier-Walsh expansions supported on $\mathcal{S}$. Endowed  with the supremum norm $\|\cdot\|_\infty$ on the $N$-dimensional Boolean cube, this space  turns into a  Banach space.

Our aim then is to estimate  the projection constant $\boldsymbol{\lambda}(\mathcal{B}_{\mathcal{S}}^{N})$,
as well as the unconditional basis constant $\boldsymbol{\chimon}(\mathcal{B}_{\mathcal{S}}^{N})$ of the Fourier-Walsh basis $(\chi_{S})_{S \in \mathcal{S}}$
of $\mathcal{B}_{\mathcal{S}}^{N}$.

Let us indicate how  this study  embeds, in a natural way, into  our setting of
spaces $\mathcal{P}_J(X_N)$ of polynomials in $N$ variables supported on certain index sets $J$.

Obviously, for each  $f:\{ \pm 1\}^{N} \rightarrow \mathbb{R}$
  there is a unique  tetrahedral real polynomial $P_f:\mathbb{R}^N\rightarrow \mathbb{R}$ for which the  following diagram commutes:
    \begin{equation*} \label{Bild}
\begin{tikzcd}[column sep=small]
\{ \pm 1\}^{N} \arrow[hookrightarrow]{rr}{} \arrow[swap]{dr}{f}& & \,\,\mathbb{R}^{N} \,\,, \arrow{dl}{P_{f}}\\
& \mathbb{R} &
\end{tikzcd}
\end{equation*}
and in this case
\begin{equation*} \label{affine}
\| f\|_{\infty} := \sup_{x \in \{ \pm 1\}^{N}}{|f(x)|} = \sup_{x \in \{ \pm 1\}^{N}}{|P_{f}(x)|} = \sup_{x \in [-1,1]^{N}}{|P_{f}(x)|} =: \| P_{f}\|_{\infty}.
\end{equation*}
Moreover,  each  $S \subset [N]$  may be identified with a tetrahedral multi index
$\alpha^S \in  \mathbb{N}_0^N$ given by $\alpha^S(k) = 1, k \in S$ and $\alpha^S(k) = 0, k \notin S$, and conversely every
tetrahedral multi index
$\alpha \in  \mathbb{N}_0^N$ defines the subset $S = \text{supp} \,\alpha \subset [N] $.
 All together this shows that, given an  `index set' $\mathcal{S}$ of subsets $S \subset [N]$,   we have   the following isometric identity
 \begin{equation}\label{identities}
   \mathcal{B}^{N}_{\mathcal{S}} = \mathcal{P}_{J(\mathcal{S})}(\ell_\infty^N (\mathbb{R}))\,, \,\,\,\,\,\,\,\,f \mapsto P_f\,\,,
 \end{equation}
where   $J(\mathcal{S})$ collects all tetrahedral indices in $\mathbb{N}_0^N$ associated to sets $S \in \mathcal{S}$, and $\ell_\infty^N (\mathbb{R})= (\mathbb{R}^N,\|\cdot\|_\infty)$
(in contrast to the complex $\ell^N_\infty = \ell_\infty^N (\mathbb{C})$ we usually look at).

We mainly concentrate  on the following special classes of functions on the $N$-dimensional Boolean cube:
\begin{itemize}
\setlength\itemsep{0.8em}
\item $\mathcal{B}^{N}$ :=  all functions $f:\{ \pm 1\}^{N} \rightarrow \mathbb{R}$,
\item $\mathcal{B}^{N}_{=d}$  :=  all  $d$-homogeneous $f:\{ \pm 1\}^{N} \rightarrow \mathbb{R}$,
\item $\mathcal{B}^{N}_{\leq d}$ :=  all $f:\{ \pm 1\}^{N} \rightarrow \mathbb{R}$ with   $\deg f \leq d$\,.
\end{itemize}
Obviously, we have the isometric identity
\begin{align} \label{identitiesI}
\mathcal{B}^N = \ell_\infty\big(\{ \pm 1 \}^N\big)\,, \,\,\,\,\,\,\,\,f \mapsto (f(x))_{x \in \{ \pm 1 \}^N}\,,
  \end{align}
and by \eqref{identities}  the isometric identities
\begin{align*}
\mathcal{B}^{N}_{=d} = \mathcal{P}_{\Lambda_T(d,N)}(\ell_\infty^N(\mathbb{R}))
\,\,\, \,\,\,\,\,\,\text{and}\,\,\,\,\,\,\,\,\,
\mathcal{B}^{N}_{\leq d} = \mathcal{P}_{\Lambda_T(\leq d,N)}(\ell_\infty^N(\mathbb{R}))=\mathcal{T}_{\leq d}(\ell_\infty^N(\mathbb{R}))\,.
  \end{align*}

\smallskip

\section{Projection constants}

The identity from \eqref{identitiesI} gives the following trivial fact:
\[
\boldsymbol{\lambda}\big(\mathcal{B}^{N}
\big) = 1\,.
\]
More generally, from Theorem \ref{C(G)proj} we immediately obtain the following  integral formula
for projection constants of  families of functions on the $N$-dimensional Boolean cubes.

\begin{theorem} \label{bool-int}
For each $\mathcal{S}  \subset [N]$
\[
\boldsymbol{\lambda}\big(\mathcal{B}_\mathcal{S} ^{N}\big) = \mathbb{E}\big[    \big| \sum_{S \in \mathcal{S} } \chi_S \big| \big]\,.
\]
\end{theorem}

\noindent
Clearly, by the Kadets-Snobar theorem (see, e.g., \eqref{kadets1}) one has
\begin{equation}\label{KS}
\boldsymbol{\lambda}\big(\mathcal{B}^N_\mathcal{S} \big) \leq \sqrt{|\mathcal{S} |}\,.
\end{equation}
Note that this estimate also follows easily from Proposition \ref{bool-int} and  the orthogonality of the Fourier-Walsh functions $\chi_S$
implying
\[
\boldsymbol{\lambda}\big(\mathcal{B}_\mathcal{S} ^{N}\big) =\mathbb{E}\big[    \big| \sum_{S \in \mathcal{S} } \chi_S \big| \big]\le
\Big( \mathbb{E}\big[\big| \sum_{S \in \mathcal{S} } \chi_S \big|^2 \big]\Big)^{1/2} =  \sqrt{|\mathcal{S} |}\,.
\]

We collect a couple of simple corollaries of Theorem~\ref{bool-int}. For the first recall from
\cite[Theorem 9.22]{o2014analysis}
the following important hypercontractivity  estimate for functions on Boolen cubes: for every $f \in \mathcal{B}_{\leq d}^{N}$
    \begin{equation}\label{weissler}
    \|f\|_2 \leq e^d \|f\|_1\,.
  \end{equation}
  In other terms (see Section~\ref{An integral formula}), for all $1 \leq d \leq N$ the set of all Walsh functions $\chi_S$ with $S \subset [N]$ and $|S|\leq d$ forms a $\Lambda(2)$-set with constant $\leq e^d$. Then the following result is a Boolean counterpart of
  Corollary~\ref{corB2}.

\begin{corollary}\label{kiel}
For each $\mathcal{S}  \subset [N]$ with  $|S|\leq d$ for all $S \in \mathcal{S}  $
\[
\frac{1}{e^d} \sqrt{|\mathcal{S} |}\,\,\leq\,\,\boldsymbol{\lambda}\big(\mathcal{B}_{\mathcal{S}}^{N}\big) \,\,\leq\,\, \,\sqrt{|\mathcal{S}|}.
\]
\end{corollary}

\smallskip

Calculating cardinalities, yields another immediate consequence.

\smallskip

\begin{corollary}
\label{peskow}
For all integers $1 \leq d \leq N$ one has
\[
\frac{1}{e^d}  \,  \binom{N}{d}^{\frac{1}{2}}
\, \, \leq \,\, \boldsymbol{\lambda}\big(\mathcal{B}_{=d}^{N}\big) \,\,\leq\,\,
\binom{N}{d}^{\frac{1}{2}}
\]
and
\[
\frac{1}{e^d}\,  \left( \sum_{k=0}^d \binom{N}{k}\right)^{\frac{1}{2}}
\, \, \leq \,\, \boldsymbol{\lambda}\big(\mathcal{B}_{\le d}^N\big) \,\,\leq\,\, \left( \sum_{k=0}^d \binom{N}{k}\right)^{\frac{1}{2}}\,.
\]
\end{corollary}

\smallskip

Comparing with what we  proved for tetrahedral $d$-homogeneous or degree $d$ polynomials on $\ell_\infty^n(\mathbb{C})$ (see, e.g., Theorem~\ref{conny3}), we also add the following result.

\smallskip

\begin{corollary} \label{mariupol}
For all integers $1 \leq d \leq N$ one has
\[
\frac{1}{e^{d}}
\bigg(\frac{N}{d}\bigg)^{\frac{d}{2}}\,\leq\,\boldsymbol{\lambda}\big(\mathcal{B}^N_{= d}\big) \,\leq\,
e^{\frac{d}{2}}\bigg(\frac{N}{d}\bigg)^{\frac{d}{2}}\,,
\]
and we here may replace $\boldsymbol{\lambda}(\mathcal{B}^N_{= d})$ by $\boldsymbol{\lambda}(\mathcal{B}^N_{\leq d})$.
\end{corollary}

\begin{proof}[Proof]
  The claim follows immediately from Corollary~\ref{peskow} combined with the two  elementary estimates
\begin{align} \label{ukraineA}
\Big(\frac{N}{d}\Big)^d \leq \binom{N}{d}\,,
\end{align}
and
\begin{equation} \label{ukraineAA}
\binom{N}{d}
\leq
\sum_{k=0}^d \binom{N}{k} \leq \sum_{k=0}^d  \frac{N^k}{k!} = \sum_{k=0}^d  \frac{d^k}{k!} \Big(\frac{N}{d}\Big)^k
\leq e^d \Big(\frac{N}{d}\Big)^d\,.\qedhere
\end{equation}
\end{proof}

\smallskip

In this context the following remark seems of interest  (compare also with Proposition~\ref{Cauchy}).

\begin{remark}
We have 
\[
\boldsymbol{\lambda}
\big(\mathcal{B}_{=d}^{N}\big) \le (1+\sqrt{2})^d  \boldsymbol{\lambda}
\big(\mathcal{B}_{\le d}^{N}\big)\,.
\]
This follows   from an  important  fact proved by Klimek in the article  \cite{klimek1995metrics} (see also  \cite[Lemma1,(iv)]{defant2019fourier}): If $f \in \mathcal{B}_{\leq d}^{N} $ and
$f_k = \sum_{|S|=k} \hat{f}(S) \chi_S$ is the $k$-homogeneous part of $f$  for $0 \leq k \leq d$, then
\begin{equation}\label{klimek}
\|f_k\|_\infty \leq (1+\sqrt{2})^d \|f\|_\infty \,.
\end{equation}
\end{remark}

\smallskip

Note that applying a~remarkable formula due to ~McKay \cite{mckay1989littlewood} we have (see also \cite[Lemma 5.7]{defant2018bohrBoolean}):
For each $N \in \N$ and each $0 \leq \alpha <N$
with $N - \alpha$ being an odd integer, there exists a~real number $0 \leq c_{\alpha,N} \leq \sqrt{\pi/2}$ such that
\[
\sum_{k \leq \frac{N - \alpha - 1}{2}}{\binom{N}{k}}
= \sqrt{N}\,\binom{N-1}{\frac{N - \alpha -1}{2}}\,Y\left( \frac{\alpha + 1}{\sqrt{N}} \right)
\,\exp{\left(  \frac{c_{\alpha, N}}{\sqrt{N}} \right)}\,,
\]
where $Y$ is given by
\[
Y(x) = e^{\frac{x^{2}}{2}} \int_{x}^{\infty}e^{-\frac{t^{2}}{2}}\,dt, \quad\, x\geq 0\,.
\]
In particular,  taking $\alpha =0$, we obtain a~nice asymptotic formula for $\sum_{k=0}^{d}\binom{N}{k}$, whenever $N$  is odd and $d = \frac{N-1}{2}$.

It worth noting that the following estimates hold (see \cite[Proposition 3]{szarek1999nonsymmetric})
\[
\frac{2}{x + (x^{2} + 4)^{1/2}} \leq Y(x) \leq \frac{4}{3x + (x^{2} + 8)^{1/2}}, \quad\, x\geq 0\,.
\]
We conclude with an observation showing  that the upper bound in \eqref{ukraineAA} can be improved for
$2d-1 < N$ as follows:
\begin{equation} \label{Desigforo}
\sum_{k=0}^d \binom{N}{k} \le {N \choose d} \frac{N-(d-1)}{N-(2d-1)}\,.
\end{equation}
Indeed, we follow  the argument that appears in \cite{17236}. Note that
\[
\frac{{N \choose d} + {N \choose d-1} +  {N \choose d-2} +\,\cdots\, + 1} {{N \choose d}}
= 1 + \frac{d}{N-d+1} + \frac{d(d-1)}{ (N-d+1)(N-d+2)} + \,\cdots\,\,.
\]
Bounding  the right-hand side from above by the geometric series 
\[1 + \frac{d}{N-d+1} + \left( \frac{d} {N-d+1} \right)^2
+ \, \cdots\,=\, \frac{N-(d-1)}{N - (2d-1)}\,,
\] 
we get \eqref{Desigforo}.

\section{The limit case}
It is easy to see that $\mathcal{B}_{=1}^N$ and $\ell_1^N(\mathbb{R})$ identify as Banach spaces
whenever we interpret the $N$-tuple $\sum_{k=1}^N x_k e_k$ as the function $\sum_{k=1}^N x_k \chi_{\{k\}}$
on the $N$-dimensional Boolean cube.
Then by the result of Gr\"unbaum mentioned in~\eqref{grunbuschR} we know that
\[
\lim_{N \to \infty} \frac{\boldsymbol{\lambda}\big(\mathcal{B}_{=1}^N\big)}{\sqrt{N}}  =
\lim_{N \to \infty} \frac{\boldsymbol{\lambda}\big(\ell_1^N(\mathbb{R})\big)}{\sqrt{N}}
= \sqrt{\frac{2}{\pi }}\,.
 \]
  In the following we show a procedure that allows to extend this result to  $\mathcal{B}_{=d}^N$ and $\mathcal{B}_{\leq d}^N$,
  where the degree $d$ is arbitrary.
  
  Note that our results are going to improve the outcome of  Corollary~\ref{mariupol} considerably. In fact this corollary shows that  for each $d$
\[
\frac{1}{e^{d} d^{d/2}}
\,\,\leq\,\,
\liminf_{d \leq N \to \infty} \frac{\boldsymbol{\lambda}(\mathcal{B}^N_{= d})}{N^{d/2}}
\,\,\leq\,\,
\limsup_{d \leq N \to \infty} \frac{\boldsymbol{\lambda}(\mathcal{B}^N_{= d})}{N^{d/2}}
\,\,\leq\,\,
\frac{e^{d/2}}{d^{d/2}}\,,
\]
where we may also replace $\boldsymbol{\lambda}(\mathcal{B}^N_{= d})$ by $\boldsymbol{\lambda}(\mathcal{B}^N_{\leq d})$.

 To do all this, we  use a probabilistic point of view, treating  the coordinate functions $(\chi_{\{i\}})_{1 \le i \le N}$ on the Boolean cube  as independent Bernoulli random variables (taking the values $\pm 1$ with equal probability $1/2$). From this perspective, any Walsh function $\chi_S$  is itself a random variable, being a certain  product of coordinate functions. Consequently any function on the Boolean cube may  be seen as a random variable.

 By Theorem~\ref{bool-int} the projection constant of  $\mathcal{B}_{=d}^N$ can be computed as the expectation
 \begin{equation*}
   \mathbb{E} \dis \bigg| \sum_{ |S| = d, S \subset [N]  } \chi_S \bigg|\,.
 \end{equation*}
     Based on the Central Limit Theorem  and convergence in distribution,
the main idea of our procedure is to rewrite the random variable
$  \sum_{ |S| = d, S \subset [N]  } \chi_S $
in a suitable way into another
random variable, for which we manage to control its mean.

Recall that we use the notation $Y_n \overset{D}{\longrightarrow}Y$, whenever a sequence $(Y_n)$ converges in distribution to a random variable $Y$. Additionally to the notion of convergence in distribution, it will be necessary to consider convergence in probability.
We  write $Y_n \overset{P}{\longrightarrow}Y$  if the sequence $(Y_n)$ converges in probability to a random variable $Y$.  Of course, convergence in probability  implies convergence in distribution  but, in general, the converse is not true. We recall that these two notions of convergence coincide, provided  the limit is a constant.

Moreover, we frequently  need a classical theorem of Slutsky. It states that, given two sequences  $(X_n)_{n }$ and $(Y_n)_{n}$ of random variables  such that $X_n \overset{D}{\longrightarrow} X$ and $Y_n \overset{P}{\longrightarrow} c$ (where $X$ is a random variable and $c \in \mathbb{R}$ a constant), then
\begin{itemize}
    \item[a.] $X_n + Y_n \overset{D}{\longrightarrow} X + c$\,,
    \item[b.] $X_n  Y_n \overset{D}{\longrightarrow} c X$\,.
\end{itemize}

To keep our later  calculations more transparent, we deal in detail first with the 2-homogeneous case.    
\smallskip

\begin{theorem}\label{thm: cte proy boole 2 hom}
$\dis\lim_{N \to \infty} \frac{\boldsymbol{\lambda}(\Bb_{=2}^N) }{N} = \sqrt{\frac{2}{\pi e}}$
\end{theorem}

\begin{proof}
If we expand the square of the Boolean function $x \mapsto \sum_{i=1}^N x_i$ and rearrange the terms using $x_i^2 = 1$, then  we get the equality
\[
\left|\dis\sum_{1 \le i < j \le N}  x_i x_j \right| = \left|\dis\sum_{1 \le i < j \le N}  x_{\{i,j\}} \right| = \frac{1}{2} \left|\left( \sum_{i=1}^N x_i \right)^2 - N \right|\,.
\]
By the Central Limit Theorem the sequence of random variables $(Z_N)$ given by
$$
Z_N := \frac{1}{\sqrt{N}}\dis\sum_{i =1}^N x_i $$
converges in distribution to a normal random variable $Z$ with mean $0$ and variance $1$. Since the function $f(x) = \frac{|x^2-1|}{2}$ is continuous, we have
\[
\frac{1}{N} \left|
\sum_{1 \le i < j \le N}  x_i x_j \right| = \frac{|Z_N^2 -1 |}{2} \overset{D}{\longrightarrow} \frac{|Z^2 -1 |}{2}.
\]
Now note that by the orthogonality of the Fourier-Walsh basis we have
\[
\mathbb{E} \left[    \left| \frac{1}{N} \sum_{1 \le i < j \le N}  x_i x_j  \right|^2 \right] = \frac{|\Lambda_T(2,N)|}{N^2} \le 1\,.
\]
This gives the uniform integrability of the random variable sequence $\left(\left|\sum_{1 \le i < j \le N}  x_i x_j \right| \right)_{N \ge 1}$ (see  Remark~\ref{rem: bounded moment implies unif int}).
Finally, thanks to Theorem~\ref{bool-int} and \cite[Theorem 3.5]{billingsley2013convergence}, we have
\[
\lim_{N \to \infty} \frac{\boldsymbol{\lambda}(\Bb_{=2}^N) }{N} = \lim_{N \to \infty} \frac{1}{N} \; \mathbb{E} \left( \left| \sum_{1 \le i < j \le N}  x_i x_j \right| \right) = \lim_{N \to \infty} \mathbb E \left( \frac{|Z_N^2 -1 |}{2} \right) =  \mathbb E \left( \frac{|Z^2 -1 |}{2} \right)\,.
\]
Doing the computations, we finally arrive at
\[
\lim_{N \to \infty} \frac{\boldsymbol{\lambda}(\Bb_{=2}^N) }{N} = \mathbb E \left( \frac{|Z^2 -1 |}{2} \right) = \frac{1}{\sqrt{2 \pi}} \int_{-\infty}^{\infty} \frac{|t^2-1|}{2}e^{\frac{-t^2}{2}} dt = \sqrt{\frac{2}{\pi e}}\,. \qedhere
\]
\end{proof}

\smallskip

The general case for arbitrary  degrees $d \in \NN$ is  technically more involved.
In the previous proof for $d=2$, the key step is to rewrite $$\frac{1}{N} \dis\sum_{1 \le i < j \le N}  x_i x_j $$ as a polynomial in one variable.
For  arbitrary $d$ we  require an adequate decomposition of the random variable
\begin{equation}\label{AvariableA}
  Y_N(x) = \frac{1}{N^{d/2}}\dis  \sum_{ \alpha \in \Lambda_T(d,N)} x^\alpha\,.
\end{equation}
To manage the technical difficulties that arise, it will be necessary to use a certain decomposition of the indices in the set $\Lambda(d,N)$: the so-called factorization decomposition. This particular decomposition was introduced in \cite{galicer2021monomial} and \cite{mansilla2019thesis}, where it was successfully used to understand sets of monomial convergence of spaces of holomorphic functions.

Given $k,N \in \NN$, an index $\alpha \in \Lambda(2k,N)$ is even, whenever $\alpha = 2 \beta$ for some $\beta \in \Lambda(k,N) $ (i.e., all its coordinates are even).

Note that for any index $\alpha~\in~\Lambda(d,N)$ and for any  $k \leq d/2$, there is a unique  decomposition of $\alpha$
into the   sum of a  tetrahedral index $\alpha_T \in \Lambda_T(d-2k,N)$ and an even index $\alpha_E \in \Lambda_E(2k,N)$, i.e., 
$$
\alpha = \alpha_T+ \alpha_E\,.
$$
This in particular implies that
\[
\text{
$x^\alpha = x^{\alpha_T} \cdot \underbrace{x^{\alpha_E}}_{=1} = x^{\alpha_T} $ \quad
for every \quad  $x \in \{ \pm 1 \}^N$.}
\]
 The way to find such a decomposition of a  given index $\alpha \in \Lambda(d,N)$ is rather simple: Given $1 \leq j \leq N$, the $j$-th coordinate of the tetrahedral part $\alpha_T$ is equal to 1,
whenever $\alpha_j$ is odd, and 0 else. The even part is defined as $\alpha_E:= \alpha - \alpha_T$. It is clear that all the coordinates of $\alpha_E$ are even, and thus there exists $\beta$ such that $\alpha_E = 2 \beta$.
This decomposition allows the following set identifications:
\begin{equation} \label{trick}
    \Lambda(d,N)	\leftrightsquigarrow \bigcup_{k=0}^{\lfloor d/2 \rfloor} \Lambda_T(d-2k,N) \times \Lambda_E(2k,N) \leftrightsquigarrow \bigcup_{k=0}^{\lfloor d/2 \rfloor} \Lambda_T(d-2k,N) \times \Lambda(k,N)\,,
\end{equation}
where the second identification comes from the fact that there is a canonical correspondence between  $\Lambda_E(2k,N)$ and $\Lambda(k,N)$.

To avoid miss-understandings, if not indicated differently, an element $x$ will always be in the Boolean cube $\{-1,1\}^N$.
Moreover, in view of our purpose it is  convenient  to use the usual monomial notation, that is, for $S\subset [N]$ we identify the Boolean function $\chi_S$ with $x^{\alpha}$, where $\alpha \in \Lambda_T(|S|,N)$ with  $\alpha_i = 1$, whenever $i \in S$, and $0$ otherwise.

We say that two indices $\alpha \in \Lambda(d_1,N)$ and $\beta \in \Lambda(d_2,N)$ do not share variables whenever  they have disjoint support.

\begin{lemma}\label{rem: separate variables}
Let  $\alpha \in \Lambda(d,N) $ and  $k \leq d/2$. Assume that the tetrahedral part $\alpha_T~\in~\Lambda_T(d-2k,N)$ and the  even part $\alpha_E \in \Lambda(2k,N)$ of $\alpha$ do not share variables, and that $ \alpha_E = 2 \beta $ with $\beta \in \Lambda_T(k,N)$. Then 
\[
|[\alpha]| =  \frac{d!}{2^k}\,.
\]
\end{lemma}

\begin{proof}
  We deduce  from~\eqref{scholz} that
  \begin{equation*}
    |[\alpha]| = \frac{d!}{\alpha!} = \frac{d!}{(\alpha_T+ \alpha_E)!}=\frac{d!}{\alpha_T! \alpha_E!} =\frac{d!}{(2\beta)!} = \frac{d!}{2^k}. \qedhere
  \end{equation*}
\end{proof}

\smallskip

Let us begin analyzing the idea to rewrite the random variable from~\eqref{AvariableA} for general degrees $d$.
Taking $ \sum_{i = 1}^N x_i $ for $x \in \{ \pm 1 \}^N$ to the power $d$
and using the fact that $x_i^2 = 1$, we get
\[
\left( \sum_{i = 1}^N x_i \right)^d = \sum_{\alpha \in \Lambda(d,N)} |[\alpha]| x^{\alpha}\,,
\]
and  decomposing each $\alpha \in \Lambda(d,N)$
according to \eqref{trick},
we have
\begin{equation}\label{eq: sum power d}
\left( \sum_{i = 1}^N x_i \right)^d = \sum_{k = 0}^{\floor{d/2}} \left(  \sum_{ \alpha_E \in \Lambda_E(2k,N)} \; \; \sum_{ \alpha_T \in \Lambda_T(d - 2k,N)} |[\alpha_T + \alpha_E]| x^{\alpha_T} \right)\,.
\end{equation}
Rearranging terms,
gives
\begin{equation}\label{eq: sum power d (2)}
\sum_{ \alpha \in \Lambda_T(d ,N)} x^{\alpha} = \frac{1}{d!} \left[ \left( \sum_{i = 1}^N x_i \right)^d - \sum_{k = 1}^{\floor{d/2}} \left(  \sum_{ \alpha_T \in \Lambda_T(d - 2k,N)}  \; \;
\sum_{ \alpha_E \in \Lambda_E(2k,N)} \; \;  |[\alpha_T + \alpha_E]|x^{\alpha_T} \right) \right].
\end{equation}
To illustrate this, note that for $d = 2,3$  we get
\begin{itemize}
\item $
\dis\sum_{1 \le i < j \le N}  x_i x_j = \frac{1}{2} \left[ \left( \sum_{i=1}^N x_i \right)^2 - N \right]$\,,
\item $
\dis\sum_{1 \le i < j < k \le N}  x_i x_j x_k = \frac{1}{6} \left[ \left( \sum_{i=1}^N x_i \right)^3 - (3N-2)\left( \sum_{i=1}^N x_i \right) \right]$\,.
\end{itemize}

\bigskip

The following technical lemma is crucial  for our purposes.

\begin{lemma}\label{lem: lim coef homog bool}
Given $d,N \in \NN$, we for every $x \in \{ \pm 1 \}^N$ have
\[
\sum_{ \alpha \in \Lambda_T(d ,N)} x^{\alpha} = \frac{1}{d!} \left[ \left( \sum_{i = 1}^N x_i \right)^d - \sum_{k = 1}^{\floor{d/2}} \pmb{C_{d,k,N}} \sum_{ \alpha_T \in \Lambda_T(d - 2k,N)} x^{\alpha_T}  \right]\,,
\]
where for $1 \le k \le \floor{d/2}$
\[
\text{$\pmb{C_{d,k,N}} = \binom{N-d+2k}{k} \frac{d!}{2^k} + \pmb{D_{d,k,N}} $ \,\,\,\,and\,\,\,\, $0 \leq \pmb{D_{d,k,N}} \le N^{k-1} 2 d d!$}\,.
\]
 In particular,
\[
\lim_{N \to \infty} \frac{\pmb{C_{d,k,N}}}{N^k} = \frac{d!}{k! \;2^k}.
\]
\end{lemma}

\begin{proof}
We fix some  $1 \le k \le \floor{d/2}$, and note that in view of
 equation \eqref{eq: sum power d (2)}  we need to study
\[
\sum_{ \alpha_T \in \Lambda_T(d - 2k,N)}
 \; \;
\sum_{ \alpha_E \in \Lambda_E(2k,N)}|[\alpha_T + \alpha_E]| x^{\alpha_T}\,,
\]
in order to be able to define and to control the factor $\pmb{C_{d,k,N}}$.
We fix some $\alpha_T \in \Lambda_T(d-2k,N)$, and start to decompose
\[
\sum_{ \alpha_E \in \Lambda_E(2k,N)}|[\alpha_T + \alpha_E]| x^{\alpha_T}\,.
\]
Let us denote the set of even indices which do not share variables with $\alpha_T$, by $\Lambda_E(\alpha_T) \subset \Lambda_E(2k,N)$, and use $\Lambda_E(\alpha_T)^c \subset \Lambda_E(2k,N)$ for its complement in $\Lambda_E(2k,N)$. Then
\[
\sum_{ \alpha_E \in \Lambda_E(2k,N)}|[\alpha_T + \alpha_E]| x^{\alpha_T}
=
\bigg[
\underbrace{\sum_{ \alpha_E \in \Lambda_E(\alpha_T)^c}|[\alpha_T + \alpha_E]|}_{= :\pmb{A}} + \underbrace{\sum_{ \alpha_E \in \Lambda_E(\alpha_T)}|[\alpha_T + \alpha_E]| }_{= :\pmb{B}}\bigg]  x^{\alpha_T}\,,
\]
and we handle both sums separately.

Before the estimations, note that $\pmb{A + B}$, does not depend on $\alpha_T$, i.e., for every $\alpha_T$ the sum 
$
\sum_{ \alpha_E \in \Lambda_E(2k,N)}|[\alpha_T + \alpha_E]|$ is the same.
Indeed, given two different tetrahedral multi indices  $\alpha_T, \alpha_T' \in \Lambda_T(d-2k,N)$, the natural bijection between index sets that maps $\alpha_T$ to $\alpha_T'$, given by a suitable permutation of coordinates, also lets the sums invariant. This allows us to define $\pmb{C_{d,k,N}} := \pmb{A} + \pmb{B}$.

\noindent {\bf Estimating $\pmb{A}:$} In order to estimate the cardinality of $\Lambda_E(\alpha_T)^c$,
observe  that any multi index in $\Lambda_E(\alpha_T)^c$ needs to share at least one of the $d-2k$ possible variables of $\alpha_T$, and therefore 
\begin{equation} \label{bound-a}
  |\Lambda_E(\alpha_T)^c| \le (d-2k) |\Lambda_E(2k-2,N)| = (d-2k) |\Lambda(k-1,N)| \le (d-2k) N^{k-1}\,.
\end{equation}
Clearly, $\alpha_T + \alpha_E \in \Lambda(d,N)$ for any $\alpha_E \in \Lambda_E(\alpha_T)^c$, and hence
by~\eqref{scholz}
\[
\textbf{A} = \sum_{ \alpha_E \in \Lambda_E(\alpha_T)^c}  |[\alpha_T + \alpha_E]|
= \sum_{ \alpha_E \in \Lambda_E(\alpha_T)^c} \frac{d!}{(\alpha_T + \alpha_E)!} \leq (d-2k) N^{k-1} d!\,.
\]

\noindent {\bf Estimating $\pmb{B}:$}
We have
\begin{align*}
\textbf{B}=\sum_{ \alpha_E \in \Lambda_E(\alpha_T)}  |[\alpha_T +\alpha_E]|
 =  \sum_{ \alpha_E \in \Lambda_E(2k,N-d+2k)}  |[\alpha_T +\alpha_E]|.
\end{align*}
We then may  decompose the index set $\Lambda_E(2k,N-d+2k)$ into the set of those indices which use $k$ variables, denoted by $\Lambda_E(2k,N-d+2k)^k$, and the set that contains all even indices with less than $k$ variables, denoted by $\Lambda_E(2k,N-d+2k)^{<k}$,
so
\begin{align*}
\textbf{B}
 =  \underbrace{\sum_{ \alpha_E \in \Lambda_E(2k,N-d+2k)^{k}}  |[\alpha_T +\alpha_E]|}_{= :\pmb{B^{=k}}}  \,+\,
 \underbrace{\sum_{ \alpha_E \in \Lambda_E(2k,N-d+2k)^{< k}}  |[\alpha_T +\alpha_E]|}_{= :\pmb{B^{< k}}}\,.
\end{align*}
 Observe that given a multi index in $\Lambda_E(2k,N-d+2k)^{<k}$, it is mandatory that some variable appears to at least the $4$th power (since all the indices in the set $\Lambda_E(2k,N-d+2k)^{<k}$ are even). Going through all the possible $N-d +2k$ variables, we get the bound
\begin{align*}
|\Lambda_E(2k,N-d+2k)^{<k}| & \le (N-d+2k) |\Lambda_E(2k-4,N-d+2k)| \\
& = (N-d+2k) |\Lambda(k-2,N-d+2k)| \le (N-d+2k)^{k-1}\,,
\end{align*}
and  then (as above with \eqref{scholz})
\[
\pmb{B^{<k}} = \sum_{ \alpha_E \in \Lambda_E(2k,N-d+2k)^{<k}} |[\alpha_T + \alpha_E]| \le \left| \Lambda_E(2k,N-d+2k)^{<k} \right| \; d! \le N^{k-1} d!\,.
\]
Moreover,
$$|\Lambda_E(2k,N-d+2k)^{<k}| = |\Lambda_T(k,N-d+2k)| = \binom{N-d+2k}{k}\,,$$
as for every $ \alpha \in \Lambda_E(2k,N-d+2k)^{k}$  there is $\beta \in \Lambda_T(k,N-d+2k)$ such that $\alpha = 2 \beta$ and this defines a one to one mapping. Then by   Lemma~\ref{rem: separate variables} 
\begin{align*}
\pmb{B^{=k}} = \sum_{ \alpha_E \in \Lambda_E(2k,N-d+2k)^{k}}  |[\alpha_T+\alpha_E] |
=  \binom{N-d+2k}{k} \frac{d!}{2^k}\,.
\end{align*}

\noindent {\bf Combining:}
We define
$\pmb{D} := \pmb{A} + \pmb{B^{<k}} $
(note that $\pmb{D} = \pmb{D_{d,k,N}}$ in fact depends on $d,k$ and $N$).
Then $ \pmb{D}\leq  N^{k-1}  2d d!$, and  all in all we obtain
\begin{align*}
 \pmb{C_{d,k,N}} = \sum_{ \alpha_E \in \Lambda_E(2k,N)}   |[\alpha_T + \alpha_E]|
    =  
    \pmb{A} + \pmb{B}
       =
    \pmb{B^{=k}} + \pmb{D} =  \binom{N-d+2k}{k}\frac{d!}{2^k}+ \pmb{D}\,.
\end{align*}
Finally, using Stirling's formula we get that $\dis\lim_{N \to \infty} \frac{\pmb{C_{d,k,N}}}{N^k} = \frac{d!}{k! \;2^k}.$
\end{proof}

The following lemma goes one  step further -   namely, rewriting the  random variable from~\eqref{AvariableA}
in a way that later allows us to calculate the limit of its mean.
Notice that for $d = 0$ this random variable equals the constant function of value $1$.

\begin{lemma}\label{lem: boolean polynomial rewriting}
Given $d \in \NN_0$ and $N \in \NN$, there are $P_d \in \Pp_{\le d}(\RR)$ and $\varphi_{d,N} \in C(\RR)$ such that for all
$x \in \{ \pm 1 \}^N$
\[
\frac{1}{N^{d/2}}\dis  \sum_{ \alpha \in \Lambda_T(d,N)} x^\alpha = P_d \left( \frac{\sum_{i=1}^N x_i}{\sqrt{N}} \right) + \varphi_{d,N} \left(\frac{\sum_{i=1}^N x_i}{\sqrt{N}} \right)\,,
\]
where  $P_d(t) = \frac{t^d}{d!} - \dis\sum_{k = 1}^{\floor{d/2}} \frac{1}{k!2^k} P_{d-2k}(t) $
with $P_0(t) = 1$, $P_1(t) = t$,
and
$\varphi_{0,N} = \varphi_{1,N} =0$.
Moreover,
\begin{equation} \label{limit1}
    \varphi_{d,N} \left(\frac{\sum_{i=1}^N x_i}{\sqrt{N}} \right) 
    {\overset{P}{\longrightarrow}}  0
\end{equation}
and
\begin{equation} \label{limit2}
    \frac{1}{N^{d/2}}\dis  \sum_{ \alpha \in \Lambda_T(d,N)} x^\alpha  
    {\overset{D}{\longrightarrow}} P_d(Z) \quad\, \text{as \,\, $N\to \infty$\,,}
\end{equation}
where $Z$ is a normal distribution with mean $0$ and variance $1$.
\end{lemma}

\begin{proof}
The proof will be by induction on $d$. The base cases $d=0,1$ are tautological. Let us assume that the result is true for $ 1 \le k \le d -1  $, and prove that  it is true for $d$. Using Lemma \ref{lem: lim coef homog bool} and our inductive hypothesis, we have
\begin{align*}
    \frac{1}{N^{d/2}} \sum_{ \alpha \in \Lambda_T(d ,N)} x^{\alpha} & = \frac{1}{d!} \left( \frac{\sum_{i = 1}^N x_i}{\sqrt{N}} \right)^d - \sum_{k = 1}^{\floor{d/2}}  \frac{C_{d,k,n} }{d! N^k} \left(\frac{1}{N^{(d-2k)/2}} \sum_{ \alpha_T \in \Lambda_T(d - 2k,N)} x^{\alpha_T} \right)\\[1ex]&
     =  \frac{1}{d!} \left( \frac{\sum_{i = 1}^N x_i}{\sqrt{N}} \right)^d  - \sum_{k = 1}^{\floor{d/2}}  \frac{C_{d,k,n} }{d! N^k} \left(P_{d-2k}\left( \frac{\sum_{i = 1}^N x_i}{\sqrt{N}} \right) + \varphi_{d-2k,N}\left( \frac{\sum_{i = 1}^N x_i}{\sqrt{N}} \right) \right).
\end{align*}
Defining for $t \in \mathbb{R}$
\[
\varphi_{d,N}(t) := \sum_{k = 1}^{\floor{d/2}}  \left(\frac{C_{d,k,n} }{d! N^k} - \frac{1}{k! 2^k} \right)\left(P_{d-2k}(t) + \varphi_{d-2k,N}\left( t \right) \right)
\]
and
\[
P_d(t) = \frac{t^d}{d!} - \dis\sum_{k = 1}^{\floor{d/2}} \frac{1}{k!2^k} P_{d-2k}(t)\,,
\]
we see  that the desired equality holds.
It remains to show the two limit
formulas from~\eqref{limit1} and~\eqref{limit2}.
By the Central Limit Theorem we know that $ \frac{\sum_{i = 1}^N x_i}{\sqrt{N}} \overset{D}{\longrightarrow} Z$, where $Z$ is a normal random variable with mean  $0$ and variance~$1$.
Moreover, by
 Lemma \ref{lem: lim coef homog bool} we know that
  \[
 \dis\lim_{N \to \infty} \frac{C_{d,k,n} }{d! N^k} - \frac{1}{k! 2^k} = 0\,.
 \]
 Since convergence in distribution
 is preserved under continuous
 functions, and
 convergence in distribution and convergence in probability are equivalent when the limit is a constant, we see that the limit
 from~\eqref{limit1}
 holds true.
 With similar arguments we first
conclude that
\[ P_{d-2k}\left(\frac{\sum_{i=1}^N x_i}{\sqrt{N}} \right) 
{\overset{D}{\longrightarrow}} P_{d-2k}(Z) \quad\, \text{as \,\, $N\to \infty$\,,}
\]
and then we obtain~\eqref{limit2} combining  Slutsky's Theorem with~\eqref{limit1}.
\end{proof}

Finally, we come to the main contribution of this section,
which extends  Theorem~\ref{thm: cte proy boole 2 hom} to  all possible degrees.

\begin{theorem}\label{thm: limit bool homog}
For every $d \in \NN$
\[
\dis\lim_{N \to \infty} \frac{\boldsymbol{\lambda}(\mathcal{B}_{=d}^N)}{N^{d/2}} = \frac{1}{\sqrt{2 \pi}} \int_{-\infty}^{\infty} |P_d(t)|e^{-\frac{t^2}{2}} dt\,,
\]
where $P_d$ is the polynomial described in Lemma \ref{lem: boolean polynomial rewriting}. Moreover,
$$
\dis\lim_{N \to \infty} \frac{\boldsymbol{\lambda}(\mathcal{B}_{=d}^N)}{N^{d/2}} = \dis\lim_{N \to \infty} \frac{\boldsymbol{\lambda}(\mathcal{B}_{\le d}^N)}{N^{d/2}}\,.
$$
\end{theorem}

\begin{proof} As in~\eqref{AvariableA} we for $N \in \mathbb{N}$ define the random variable  
\begin{equation*}
  Y_N(x) = \frac{1}{N^{d/2}}\dis  \sum_{ \alpha \in \Lambda_T(d,N)} x^\alpha\,,
\end{equation*}
and  start, using that for $d=1$ by the Central Limit Theorem
for some
  normal random variable $Z$ with mean~$0$ and variance~$1$, we have
$$ \frac{\sum_{i = 1}^N x_i}{\sqrt{N}} \overset{D}{\longrightarrow} Z\,.
$$
Applying Lemma \ref{lem: boolean polynomial rewriting}, the fact that convergence in distribution is preserved under  continuous functions, and Slutky's Theorem, we get
\[
|Y_N(x)|  = \left| P_d \left( \frac{\sum_{i=1}^N x_i}{\sqrt{N}} \right) + \varphi_{d,N} \left(\frac{\sum_{i=1}^N x_i}{\sqrt{N}} \right) \right| 
{\overset{D}{\longrightarrow}} |P_d(Z)| \quad\, \text{as \,\, $N\to \infty$}\,.
\]

Now orthogonality of the Fourier-Walsh basis assures that for all $N$
\[
 \mathbb{E}|Y_N|^2 = \frac{|\Lambda_T(d,N)|}{N^d} \le 1 \,,
 \]
 which gives the uniform integrability of all
 $
  Y_N
    $
  (see also Remark~\ref{rem: bounded moment implies unif int}). Using  \cite[Theorem 3.5]{billingsley2013convergence}, this implies
\[
\lim_{N \to \infty}   \mathbb E |Y_N| = \mathbb E [|P_d(Z)|] = \frac{1}{\sqrt{2 \pi}} \int_{-\infty}^{\infty} |P_d(t)|e^{-\frac{t^2}{2}} dt\,,
\]
which by  Theorem~\ref{bool-int}  finishes the proof of the first integral formula.

The second claim follows with similar arguments. Observe first that by another application of Theorem~\ref{bool-int} we have
\[
\boldsymbol{\lambda}(\mathcal{B}_{\le d}^N)  = \mathbb E \left[ \left| \dis\sum_{  \alpha \in \Lambda_{T}(\le d,N) } x^\alpha \right| \right].
\]
Also by Lemma \ref{lem: boolean polynomial rewriting},
\eqref{limit2}, for every $0 \le k \le d$
\[
\frac{1}{N^{k/2}} \dis\sum_{\alpha \in \Lambda_{T}(k,N) } x^\alpha 
{\overset{D}{\longrightarrow}} P_k(Z) \quad\, \text{as \,\, $N\to \infty$\,,} 
\]
where $Z$ is as above. Hence we as before  use  Slutsky's Theorem and the fact that
 a sequence of random variables converges  in probability
whenever it converges in distribution to a constant, to see that for every $0 \le k < d $
\[
\frac{1}{N^{d/2}} \dis\sum_{  \alpha \in \Lambda_{T}(k,N) } x^\alpha 
{\overset{P}{\longrightarrow}} 0, \quad\, \text{as \,\, $N\to \infty$\,.}
\]
Now one more  application of  Slutsky's Theorem shows that
\[
\lim_{N \to \infty} \frac{1}{N^{d/2}} \dis\sum_{  \alpha \in \Lambda_{T}(\le d,N) } x^\alpha
=
\lim_{N \to \infty} \frac{1}{N^{d/2}} \dis\sum_{k=1 }^d\dis\sum_{  \alpha \in \Lambda_{T}(k,N) } x^\alpha
=\lim_{N \to \infty} \frac{1}{N^{d/2}} \dis\sum_{  \alpha \in \Lambda_{T}( d,N) } x^\alpha
=
P_d(Z),
\]
where the limit is taken in the sense of  distribution. Again orthogonality gives
\[
 \mathbb{E} \left[    \left| \frac{1}{N^{d/2}} \sum_{ \alpha \in \Lambda_T(\leq d,N)} x^\alpha \right|^2 \right] = \frac{|\Lambda_T(\leq d,N)|}{N^d} < \infty \,,
 \]
implying that the random variables  $\frac{1}{N^{d/2}} \dis\sum_{  \alpha \in \Lambda_{T}(\leq d,N) } x^\alpha,\,\, N \in \mathbb{N}$
are uniformly integrable, and this is enough to conclude  that
\[
\dis\lim_{N \to \infty} \frac{\boldsymbol{\lambda}(\mathcal{B}_{\le d}^N)}{N^{d/2}} = \EE \left(|P_d(Z)| \right)
\]
(see again \cite[Theorem 3.5]{billingsley2013convergence}), which together with the first claim finishes the proof.
\end{proof}

\smallskip

With the explicit expressions for $P_d$ (which are easily obtained recursively) and using 
a computational platform we get the following limits:
\begin{align*}
    \dis\lim_{N \to \infty} \frac{\boldsymbol{\lambda}(\Bb_{=2}^N) }{N} & = \dis\lim_{N \to \infty} \frac{\boldsymbol{\lambda}(\Bb_{\le 2}^N) }{N} = \frac{1}{\sqrt{2 \pi}} \int_{-\infty}^{\infty} \left| \frac{t^2-1}{2} \right| e^{-\frac{t^2}{2}} dt = \sqrt{\frac{2}{\pi e}} \\[1ex]
    \dis\lim_{N \to \infty} \frac{\boldsymbol{\lambda}(\Bb_{=3}^N) }{N^{3/2}} & = \dis\lim_{N \to \infty} \frac{\boldsymbol{\lambda}(\Bb_{\le 3}^N) }{N^{3/2}} = \frac{1}{\sqrt{2 \pi}} \int_{-\infty}^{\infty} \left| \frac{t^3-3t}{6} \right| e^{-\frac{t^2}{2}} dt =\frac{1}{3\sqrt{2 \pi} } \left(1 + \frac{4}{e^{3/2}}\right)  \\[1ex]
        \dis\lim_{N \to \infty} \frac{\boldsymbol{\lambda}(\Bb_{=4}^N) }{N^2} & = \dis\lim_{N \to \infty} \frac{\boldsymbol{\lambda}(\Bb_{\le 4}^N) }{N^2} = \frac{1}{\sqrt{2 \pi}} \int_{-\infty}^{\infty} \left| \frac{t^4-6t^2+6}{24} \right| e^{-\frac{t^2}{2}} dt \approx\frac{0.400228}{\sqrt{2 \pi} }  \\[1ex]
        \dis\lim_{N \to \infty} \frac{\boldsymbol{\lambda}(\Bb_{=5}^N) }{N^{5/2}} & = \dis\lim_{N \to \infty} \frac{\boldsymbol{\lambda}(\Bb_{\le 5}^N) }{N^{5/2}} = \frac{1}{\sqrt{2 \pi}} \int_{-\infty}^{\infty} \left| \frac{t^5-10t^3+30t}{120} \right| e^{-\frac{t^2}{2}} dt =\frac{3}{10\sqrt{2 \pi} }   \\[1ex]
        \dis\lim_{N \to \infty} \frac{\boldsymbol{\lambda}(\Bb_{=6}^N) }{N^{3}} & = \dis\lim_{N \to \infty} \frac{\boldsymbol{\lambda}(\Bb_{\le 6}^N) }{N^{3}} = \frac{1}{\sqrt{2 \pi}} \int_{-\infty}^{\infty} \left| \frac{t^6-15t^4-15}{720} \right| e^{-\frac{t^2}{2}} dt \approx\frac{0.157166}{\sqrt{2 \pi} }
\end{align*}

\smallskip

\section{Sidon constants}
We start by observing that given $\mathcal{S} \subset [N]$, the constant $\boldsymbol{\chimon}(\mathcal{B}^N_{\mathcal{S} })$
is nothing else then the Sidon constant of the characters $(\chi_S)_{S \in \mathcal{S} }$ in the group $\{ \pm 1\}^{N}$, so the best constant $C >0$ such that for all $f \in \mathcal{B}^N_{\mathcal{S} }$, 
\begin{equation}\label{sidon}
\sum_{S \in \mathcal{S} } |\hat{f}(S)| \le C \| f \|_\infty\,.
\end{equation}

\begin{proposition} \label{sid}
Given integers $1 \leq d \leq N$, let $\mathcal{S} \subset [N]$ be such  that $|S| \leq d$
for all $S \in \mathcal{S}$ and $(N/d)^{d/2} \leq  |\mathcal{S}|$. Then there are constants
$C_1$, $C_2 \ge 1$ (independent of $N,d,\mathcal{S}$) such that
\[
C_1 \, \, \frac{1}{\sqrt{N}}|\mathcal{S}|^{\frac{1}{2}}
\, \, \leq \,\, \boldsymbol{\chimon}\big(\mathcal{B}^N_{\mathcal{S}}\big) \,\,\leq\,\,
C_2^{\sqrt{d \log d}} \frac{1}{\sqrt{N}}|\mathcal{S}|^{\frac{1}{2}}\,.
\]
\end{proposition}

For the proof of the upper bound we will need  the so-called  subexponential Bohnen\-blust-Hille
inequality from \cite[Theorem~1]{defant2018bohr}. That is, there is a~constant $C \ge 1$ such that for each $1 \leq d \leq N$
and every $f \in \mathcal{B}_{\leq d}^{N}$ one has
\begin{equation}\label{equa:BHBoolean}
\Big(\sum_{\substack{S \subset [N]\\ |S|\leq d}}{|\widehat{f}(S)|^{\frac{2d}{d+1}}}\Big)^{\frac{d+1}{2d}}
\leq C^{\sqrt{d \log d}} \, \| f\|_{\infty}.
\end{equation}

\begin{proof}[Proof of Proposition~$\ref{sid}$]
The upper bound  follows from H\"older's inequality and \eqref{equa:BHBoolean}. That is,  for all functions
$f \in \mathcal{B}^N_{\mathcal{S}}$,
\[
\sum_{|S|\leq d} |\hat{f}(S)| \leq \bigg(  \sum_{|S|\leq d} |\hat{f}(S)|^{\frac{2d}{d+1}} \bigg)^{\frac{2d}{d+1}}
|\mathcal{S}|^{\frac{d-1}{2d}} \leq C^{\sqrt{d\log d }} \frac{1}{|\mathcal{S}|^{\frac{1}{d}}} |\mathcal{S}|^{\frac{1}{2}}\|f\|_\infty
\,.
\]
But by assumption
$
\sqrt{N} \leq \sqrt{d} |\mathcal{S}|^{\frac{1}{d}}\,,
$
and hence the claim follows from \eqref{sidon}. The proof of the lower estimate is  probabilistic. Indeed, by the Kahane-Salem-Zygmund
inequality for the Boolean cube (see, e.g., \cite[Lemma~3.1]{defant2018bohr}) there is a family $(\varepsilon_S)_{S \in \mathcal{S}}$
of signs such that for $f =  \sum_{S \in \mathcal{S}}  \varepsilon_S \chi_S$ we have
\[
\|f\|_\infty \leq 6 \sqrt{\log 2} \,\,  \sqrt{N} \Big( \sum_{S \in \mathcal{S}}  |\varepsilon_S|^2  \Big)^{\frac{1}{2}}\,,
\]
and hence
\[
|\mathcal{S}| = \sum_{S \in \mathcal{S}}|\widehat{f}(S)|\leq  \boldsymbol{\chimon}\big(\mathcal{B}^N_{\mathcal{S}}\big) \,\,  6 \sqrt{\log 2} \,\,\sqrt{N} |\mathcal{S}|^{\frac{1}{2}}\,.
\]
This completes the argument.
\end{proof}

\smallskip

\begin{corollary}
There are constants $C_1, C_2 > 0$ such that for each integer $1 \leq d \leq N$ one has
\[
C_1 \, \, \frac{1}{\sqrt{N}} \binom{N}{d}^{\frac{1}{2}}
\, \, \leq \,\, \boldsymbol{\chimon}\big(\mathcal{B}^N_{=d}\big) \,\,\leq\,\,
C_2^{\sqrt{d \log d}} \frac{1}{\sqrt{N}} \binom{N}{d}^{\frac{1}{2}}
\]
and
\[
C_1 \, \, \frac{1}{\sqrt{N}} \left( \sum_{k=0}^d \binom{N}{k}\right)^{\frac{1}{2}}
\, \, \leq \,\, \boldsymbol{\chimon}\big(\mathcal{B}^N_{\leq d}\big) \,\,\leq\,\,
C_2^{\sqrt{d \log d}} \frac{1}{\sqrt{N}} \left( \sum_{k=0}^d \binom{N}{k}\right)^{\frac{1}{2}}\,.
\]
\end{corollary}

\begin{proof}
Since
$
\Big(\frac{N}{d}\Big)^d \leq \binom{N}{d} \leq  \sum_{k=0}^d \binom{N}{k}
$
(see again  \eqref{ukraineA}), both  sets $\mathcal{S} = \{ S  \colon |S| =d \}$  and $\mathcal{S} = \{ S  \colon |S| \leq d \}$
satisfy the assumptions of Proposition~\ref{sid}.
\end{proof}

\smallskip

\begin{corollary}
There are constants $C_1, C_2 >0$ such that for all $1 \leq d \leq N$
\[
C_1 \frac{1}{\sqrt{d}}\,\bigg( \frac{N}{d}  \bigg)^{\frac{d-1}{2}} \leq
\boldsymbol{\chimon}\big(\mathcal{B}^N_{= d}\big) \leq \boldsymbol{\chimon}\big(\mathcal{B}^N_{\leq d}\big) \leq
C_2^{\sqrt{d \log d}} e^{\frac{d}{2}} \,\frac{1}{\sqrt{d}}\,\bigg( \frac{N}{d}  \bigg)^{\frac{d-1}{2}}\,.
\]
In particular, 
\[
\boldsymbol{\chimon}\big(\mathcal{B}^N_{= d}\big) \sim_{C^d } 
\bigg( \frac{N}{d}  \bigg)^{\frac{d-1}{2}}
\quad \text{ and } \quad
 \boldsymbol{\chimon}\big(\mathcal{B}^N_{\leq d}\big) \sim_{C^d }
\bigg( \frac{N}{d}  \bigg)^{\frac{d-1}{2}}
\,.
\]
\end{corollary}

\begin{proof}
Both estimates follow from the  preceding corollary. For the lower one  use again \eqref{ukraineA}, and for the upper 
note that it suffices to check that
\[
\frac{1}{\sqrt{N}} \bigg( \sum_{k=0}^d \binom{N}{k}\bigg)^{\frac{1}{2}}
\leq e^\frac{d}{2} \frac{1}{\sqrt{d}}\,\bigg( \frac{N}{d}  \bigg)^{\frac{d-1}{2}}\,;
\]
indeed, this is another consequence of  \eqref{ukraineAA}.
\end{proof}

\smallskip

\section{Gordon-Lewis cycle}
We sketch that the Gordon-Lewis cycle of ideas from Section~\ref{GL} and Section~\ref{Gordon-Lewis vs projection constant}
may be repeated for functions on the $N$-dimensional  Boolean cube.

The first  result is a Boolean  analog of Theorem~\ref{gl-versus-unc}.

\begin{theorem} \label{BGL1}
Let $0 \leq d \leq N$ and $\mathcal{S} \subset \big\{ S \subset [N] \colon |S| \leq d \big\}$. Then
  \[
  {\mbox{gl}}\big(\mathcal{B}^N_{\mathcal{S}}\big)\,\,
  \leq\,\,
   \boldsymbol{\chi}\big(\mathcal{B}^N_{\mathcal{S}}\big)
    \,\,\leq\,\,
   \boldsymbol{\chimon}\big(\mathcal{B}^N_{\mathcal{S}}\big) \,\,\leq\,\,
  e^{2d} {\mbox{gl}}\big(\mathcal{B}^N_{\mathcal{S}}\big)\,.
  \]
\end{theorem}

\begin{proof}
As in Section~\ref{GL} we only have to check the last estimate.
   The proof follows from an analysis of the arguments given in Theorem~\ref{gl-versus-unc} -- but one has  to be careful since we
    in Section~\ref{GL}  always worked within  a complex setting.
 Identifying\begin{equation}\label{asinA}
              \mathcal{B}^{N}_{\mathcal{S}} \,= \,\mathcal{P}_{J(\mathcal{S})}(\ell_\infty^N(\mathbb{R}))
            \end{equation}
 as in \eqref{identities}, we  in a first step  show the analogs of \eqref{tonelli} and \eqref{hobson} (which now is even simpler
since $X_n = \ell_\infty^N(\mathbb{R})$). Then, in a first step, we check that
\begin{equation}\label{toolA Boole}
   \pi_1 \big( \id: \mathcal{P}_{J(\mathcal{S})}(\ell_\infty^N(\mathbb{R})) \to  \ell_2(J(\mathcal{S})) \big)  \leq e^d\,;
\end{equation}
this follows exactly as in the proof for \eqref{toolA Boole}, whenever we replace the 'hypercontractivity of   analytic polynomials' from \eqref{weissler:analyitic}
by the 'hypercontractivity of functions on Boolean cubes' as described in \eqref{weissler}.
Finally, using Lemma~\ref{tool2} (a result independent from choosing a real or complex
setting), the proof of Theorem~\ref{BGL1}   completes (exactly like in Section~\ref{GL}).
\end{proof}

\smallskip
We go on   with the analog of Theorem~\ref{gl_versus_proj}, which again follows from a careful inspection of the proofs given in
Section~\ref{Gordon-Lewis vs projection constant}. For any $\mathcal{S} \subset \big\{ S \subset [N] \colon |S| = d \big\}$
we define
$$\mathcal{S}^\flat = \big\{S \setminus \{i\} \colon S \in \mathcal{S}, \, i \in S  \big\}\,.$$
If $J(\mathcal{S}) \subset \Lambda_T(d, N)$ is the index set of multi indices associated to $\mathcal{S}$, then
we have that $\mathcal{S}^\flat$ is associated to  $J(\mathcal{S})^\flat \subset \Lambda_T(d-1, N)$ (see Section~\ref{index-sets} for the definition of reduced index sets).

\smallskip

\begin{theorem} \label{BGL2}
Let $0 \leq d \leq N$ and $\mathcal{S} \subset \big\{ S \subset [N] \colon |S| = d \big\}$. Then
\[
{\mbox{gl}}\big(\mathcal{B}^N_{\mathcal{S}}\big) \,\,\leq\,\,
C(d)  \,\,\|\mathbf{Q}: \mathcal{B}^N_{=d} \to \mathcal{B}^N_{\mathcal{S}}\|
\,\, \boldsymbol{\lambda}\big(\mathcal{B}^N_{\mathcal{S}^\flat}\big)\,,
\]
where $C(d) = \kappa^d e2^{3(d-1)}$, $\mathbf{Q}$ denotes the projection annihilating Fourier coefficients with indices $S$ not in~$\mathcal{S}$, and $\kappa >0$ is as defined in \eqref{kappa}.

Moreover, if $\mathcal{S} \subset \big\{ S \subset [N] \colon |S| \leq d \big\}$, then
\[
{\mbox{gl}}\big(\mathcal{B}^N_{\mathcal{S}}\big) \,\,\leq\,\,
C(d) \,\,\max_{1 \leq  k \leq m}\|\mathbf{Q}: \mathcal{B}^N_{=k} \to \mathcal{B}^N_{\mathcal{S}_{=k}}\|
\max_{1 \leq  k \leq d}  \boldsymbol{\lambda}(\mathcal{B}^N_{(\mathcal{S}_{=k})^\flat})\,,
\]
where $C(d) =  \kappa^d e\,2^{3(d-1)}(1+\sqrt{2})^\kappa$.
\end{theorem}

\begin{proof}
   As in  \eqref{asinA} we identify $\mathcal{B}^N_{\mathcal{S}}$ isometrically with $\mathcal{P}_{J(\mathcal{S})}(\ell_\infty^N(\mathbb{R}))$, and obtain
as in the proof  for Theorem~\ref{gl_versus_proj} that
\begin{equation}\label{back}
  {\mbox{gl}}\big( \mathcal{P}_{J(\mathcal{S})}(\ell_\infty^N(\mathbb{R}))\big)
 \,\le \, e2^{d-1} \big\|\mathbf{Q}_
 {\Lambda(d,N),J(\mathcal{S}))}:\mathcal{P}_{d}(\ell_\infty^N(\mathbb{R})) \to \mathcal{P}_{J(\mathcal{S})}(\ell_\infty^N(\mathbb{R}))\big\|
 \,\,
 \boldsymbol{\lambda}\big(\mathcal{T}_{d-1}(\ell_\infty^N(\mathbb{R}))\big)\,,
\end{equation}
with the only difference that  we have to modify  the use of the  Harris' polarization formula (which only holds in the complex case).
Indeed, using the notation from the proof of  Theorem~\ref{gl_versus_proj}, we only derive that   $\|U_d\|\le  e2^{d-1} $ instead of $\|U_d\|\le e$: For $P \in \mathcal{P}_{J(\mathcal{S})}(\ell_\infty^N(\mathbb{R}))$ and $x,u \in \ell_\infty^N(\mathbb{R})$ we have that
\begin{align*}
  |\check{P}(u, \ldots, u,x)| = |\check{P}_{\mathbb{C}}(u, \ldots, u,x)|\leq e \|P_{\mathbb{C}}\|
  \leq e2^{d-1} \|P\|\,.
  \end{align*}
  where the second estimate follows by   Harris' polarization and the  last estimate  is a consequence  of  a classical result of Visser \cite{visser1946generalization} showing that
  \begin{equation}\label{visser}
   \text{$\|Q_{\mathbb{C}}\|
  \leq 2^{d-1} \|Q\|$  \,\,\, for every  $d$-homogeneous polynomial $Q :\ell_\infty^N(\mathbb{R}) \to \mathbb{R} $\,.}
  \end{equation}
  On the other hand,  by Theorem~\ref{OrOuSe}  and another application of this result of Visser we get
  \begin{align*}
  &
  \|\mathbf{Q}_
 {\Lambda(d,N),J(\mathcal{S})}:\mathcal{P}_{d}(\ell_\infty^N(\mathbb{R})) \to \mathcal{P}_{J(\mathcal{S})}(\ell_\infty^N(\mathbb{R}))\|
 \\
 &
 \leq
  2^{d-1}
 \|\mathbf{Q}_
 {\Lambda(d,N),J(\mathcal{S})}:\mathcal{P}_{d}(\ell_\infty^N(\mathbb{C})) \to \mathcal{P}_{J(\mathcal{S})}(\ell_\infty^N(\mathbb{C}))\|
 \\&
 \leq
 \kappa^d 2^{d-1}
  \|\mathbf{Q}_
 {\Lambda_T(d,N), J(\mathcal{S})}:\mathcal{P}_{\Lambda_T(d,N)}(\ell_\infty^N(\mathbb{C})) \to \mathcal{P}_{J(\mathcal{S})}(\ell_\infty^N(\mathbb{C}))\|
 \\&
  \leq
 \kappa^d 2^{2(d-1)}
 \|\mathbf{Q}_
 {\Lambda(d,N),\Lambda_T(d,N)}:\mathcal{P}_{\Lambda_T(d,N)}(\ell_\infty^N(\mathbb{R})) \to \mathcal{P}_{J(\mathcal{S})}(\ell_\infty^N(\mathbb{R}))\|
 \\&
  =
 \kappa^d 2^{2(d-1)}
 \|\mathbf{Q}: \mathcal{B}^N_{=d} \to \mathcal{B}^N_{\mathcal{S}}\|
    \end{align*}
Coming back to \eqref{back}, we see that  the argument for the  first part of Theorem~\ref{BGL2} is complete. The second part then follows through a slight modification of  the proof of the second part of Theorem~\ref{gl_versus_proj}. We for
   $\mathcal{P}_{J(\mathcal{S})}(\ell_\infty^N(\mathbb{R}))$ (instead of $\mathcal{P}_J(X_n)$) consider a diagram like in  \eqref{greatpic}.
   Basically the only change lies in a different estimation of the norm of
   \[
   \text{$\mathbf{O}\oplus\bigoplus\mathbf{Q}_{J(\mathcal{S}),J(\mathcal{S})_k}$}: \mathcal{P}_{J(\mathcal{S})}(\ell_\infty^N(\mathbb{R}))
   \to \mathbb{C} \oplus_1\bigoplus_1 \mathcal{P}_{J(\mathcal{S})_k}(\ell_\infty^N(\mathbb{R}))\,,
   \,\,\,\,\,\, P \mapsto   \big(a_0, (P_k)_{k=1}^m\big)\,,
      \,
   \]
   where $P = a_0 +\sum_{k=1}^m P_k$ is the unique decomposition of $P$ into the sum of its constant term and  $k$-homogeneous parts.
   Using again \eqref{klimek}, we have that
   \[
 \|\text{$\mathbf{O}\oplus\bigoplus\mathbf{Q}_{J(\mathcal{S}),J(\mathcal{S})_k}$}\| \,\,\leq \,\, (1+\sqrt{2})^d\,.
   \]
   Then, finally, an application of the first part of the theorem and copying the arguments used for (the second part of)  Theorem~\ref{BGL2}, leads to the claim.
   \end{proof}

   We are mainly interested in the following special cases -- both being immediate consequences of  Theorem \ref{BGL2}.

\begin{corollary} \label{BGL3}
  Let $0 \leq d \leq N$. Then
  \[
     {\mbox{gl}}(\mathcal{B}^N_{=d}) \,\,\leq\,\,
     C(d)  \,
          \boldsymbol{\lambda}(\mathcal{B}^N_{=d-1})\,,
  \]
  and
  \[
  {\mbox{gl}}(\mathcal{B}^N_{\leq d}) \,\,\leq\,\,
    C(d) \, \max_{1 \leq  k \leq d}  \boldsymbol{\lambda}(\mathcal{B}^N_{= k-1})\,.
  \]
  where $C(d)=\kappa^d e2^{3(d-1)}$ in the first  and $C(d)=  \kappa^d e2^{3(d-1)}(1+\sqrt{2})^\kappa$ in the second estimate.
  \end{corollary}

\smallskip

Finishing, we mention that Theorem~\ref{BGL1} and Corollary~\ref{BGL3} may be used to indicate an alternative approach to Proposition~\ref{sid}
(with weaker constants),
which does not need  the hypercontractive Bohnenblust-Hille inequality~\eqref{equa:BHBoolean}.

Indeed, both results show that
  \[
        \boldsymbol{\chimon}(\mathcal{B}^N_{= d})
   \,\,\prec_{C^d}\,\,
   {\mbox{gl}}(\mathcal{B}^N_{= d}) \,\,\prec_{C^d}\,\,
   \boldsymbol{\lambda}(\mathcal{B}^N_{= d-1}) \leq \binom{N}{d-1}^{\frac{1}{2}} \,.
  \]
But since $ \Big(\frac{N}{\ell}\Big)^\ell \le \binom{N}{\ell} \le e^\ell\Big(\frac{N}{\ell}\Big)^\ell$ for all $1 \leq \ell \leq N$
(see again \eqref{ukraineA} and \eqref{ukraineAA}),
an elementary computation gives
\[
\binom{N}{d-1}^{\frac{1}{2}}
\prec_{C^d}
\frac{1}{\sqrt{N}} \binom{N}{d}^{\frac{1}{2}} \leq \binom{N}{d-1}^{\frac{1}{2}}\,\,\, \,\,\,\,\,\,\text{for all $1 \leq d \leq N$}\,.
\]
As announced this reproves the upper bound in Proposition~\ref{sid}.

\bigskip

\chapter{Polynomial projection constants} \label{Part: Polynomial projection constants}

In this chapter we present a  useful new technique to estimate  the projection constant of  spaces $\mathcal P_J(X_n)$ of multivariate polynomials on finite dimensional Banach lattices $X_n$ supported on a finite index set $J$. The results complement those studied in the previous sections.

Let us first sketch the main idea. Given a~Banach lattice $X_n = (\mathbb{C}^n,\|\cdot\|)$ with canonical  basis vectors $e_k$ of norm $\leq 1$, we first look  at the $1$-homogeneous case $\mathcal P_1(X_n) = X_n^\ast$, and then,  at the general case $\mathcal P_J(X_n)$, where $J \subset \mathbb{N}_0^{(\mathbb{N})}$ is an arbitrary index set.

Consider the restriction map
\begin{equation}\label{asin}
X_n^\ast \hookrightarrow \ell_\infty(B_{X_n})\,, \,\,\,\,\,\, x_n^\ast \mapsto  x_n^\ast|_{B_{X_n}} \,,
\end{equation}
which of course is an isometric embedding. Then, using the Hahn-Banach theorem, each $e_k \in X_n = X_n^{\ast \ast}$ has a~norm-preserving extension $\widetilde{e_k}\in \ell_\infty(B_{X_n})^\ast$. Consequently, denoting by $e_k^\ast \in X_n^\ast$ the
$k$-th coefficient functional of $e_k$, we get that
\begin{equation}\label{proQ}
q\colon  \ell_\infty(B_{X_n}) \to \ell_\infty(B_{X_n})\,,  \,\,\, \,\,\,f \mapsto \sum_{k=1}^n \widetilde{e_k}(f)e_k^\ast|_{B_{X_n}}
\end{equation}
is a projection onto $\mathcal P_1(X_n) = X_n^\ast$, and therefore
\[
\boldsymbol{\lambda}\big(\mathcal{P}_1(X_n)\big) \leq \big\|q\colon \ell_\infty(B_{X_n}) \to \ell_\infty(B_{X_n})\big\|\,.
\]
 But
\[
\big\|q\colon \ell_\infty(B_{X_n}) \to \ell_\infty(B_{X_n})\big\|
 = \sup_{\|f\|_\infty \leq 1} \sup_{\|z\|_{X_n}
\leq 1}\big| \sum_{k=1}^n \widetilde{e_k}(f)e_k^\ast(z) \big|
\leq \sup_{\|z\|_{X_n} \leq 1} \sum_{k=1}^n |z_k |
 = \varphi_{X_n'}(n)\,,
\]
implying the estimate
\begin{equation}\label{mimic}
\boldsymbol(\mathcal{P}_1(X_n)) \leq  \varphi_{X_{n}'}(n)
\end{equation}
(which has been already mentioned in \eqref{Carsten1}).
Vice versa,  Sch\"{u}tt (see again \eqref{schuett}) proved  that
\begin{equation} \label{scholz2}
\varphi_{X_n'}(n) \leq \sqrt{2} \boldsymbol{\lambda}(X_n')\,,
\end{equation}
provided  $\|\id\colon  X_n \to \ell^n_2\| \leq 1$. Then, under this  restriction  on $X_n$, we
 get
\begin{equation}
\label{schuetty}
\frac{1}{\sqrt{2}} \varphi_{X_{n}'}(n)  \,\leq \,  \boldsymbol{\lambda}\big(\mathcal P_1(X_n)\big)  \,\leq \, \varphi_{X_{n}'}(n)\,.
\end{equation}
Let us see a very first example. Applying Sch\"{u}tt's  estimate to $X_n = \ell_r^n, \, 1 \le r \le 2$, we
immediately get
\[
\boldsymbol{\lambda}\big(\mathcal P_1(\ell_r^n)\big) = \boldsymbol{\lambda}(\ell_{r'}^n) \sim_C n^{\frac{1}{r'}}, \quad\, 1 \leq r \leq 2\,.
\]
Unfortunately,  this result is not covered by the outcome of the 'polynomial' Kadets-Snobar Theorem~\ref{conny3}. There we for arbitrary $m$ get
\begin{align}
\label{intro-poly}
\boldsymbol{\lambda}\big(\mathcal{P}_{m}(\ell_r^n) \big)
\,\,\sim_{C^m} \,\,
\Big( 1+\frac{n}{m}\Big)^{\frac{m}{2} }\,,
\end{align}
 but only for  the  restricted range $\,2 \leq r \leq \infty$
 (recall that 
 only in this case
 $\ell_r$ is   $2$-convex).

\smallskip

In the following section we  mimic the above 'one homogeneous idea' proving \eqref{mimic}, in order to show, among others,  an extension of
\eqref{intro-poly} that  holds for the full scale $\,1 \leq r \leq \infty$.
The main idea follows the $1$-homogeneous case just explained.
As in \eqref{asin}, given an index set  $J \subset \mathbb{N}_0^n$ and a (quasi) Banach space $X_n = (\mathbb{C}^n,\|\cdot\|)$,
we  identify the space $\mathcal{P}_{J}(X_n)$ with a~subspace of $\ell_\infty(B_{X_n})$ via the natural isometrical embedding
\[
I\colon \mathcal{P}_{J}(X_n) \to \ell_\infty(B_{X_n})\,,\,\,\,\,\,\,I(P) := P|_{B_{X_n}}\,.
\]
For every multi index  $\alpha \in J$ let 
$$c_{\alpha} \colon \mathcal{P}_J(X_n) \to \mathbb{C}
\,, \quad  P \mapsto c_{\alpha}(P)
$$
 be the  linear
functional which assigns to every polynomial  $P$ its unique monomial
coefficient $c_{\alpha}(P)$. Again using the Hahn-Banach extension theorem, we know that  each of these coefficient functionals
$c_\alpha$ has a norm-preserving extension $\widetilde{c_{\alpha}}\in \ell_\infty(B_{X_n})^{*}$, that is,
$\|\widetilde{c_{\alpha}}\| = \|c_{\alpha}\|$ and
$
\widetilde{c_{\alpha}}(P)= c_\alpha(P)$ for each $P\in \mathcal{P}_{J}(X_n)\,.
$
If we now (as in \eqref{proQ}) consider the~projection $$Q\colon \ell_\infty(B_{X_n}) \to \ell_\infty(B_{X_n})$$
onto $\mathcal{P}_{J}(X_n)$, which for each  $f \in \ell_\infty(B_{X_n})$ is given by the formula
\begin{equation}\label{theprojection}
Qf(z) := \sum_{\alpha \in J} \widetilde{c_{\alpha}}(f) z^\alpha, \quad\, z\in B_{X_n}\,,
\end{equation}
then clearly
\[
\boldsymbol{\lambda}\big(\mathcal{P}_J(X_n)\big) \leq \big\|Q\colon \ell_\infty(B_{X_n}) \to\ell_\infty(B_{X_n})\big\|\,,
\]
an analog of \eqref{mimic}.
Although the determination  of the norm of this projection $Q$   in general is a difficult task, we find upper estimates which under
certain restrictive assumptions on the index set  $J$ and the Banach space $X_n$ turn out to lead to   asymptotically  optimal bounds for $\boldsymbol{\lambda}(\mathcal{P}_J(X_n))$.
More precisely, we  introduce  what we call the polynomial projection constant
\[
\widehat{\boldsymbol{\lambda}}\big(\mathcal P_J(X_n)\big):= \sup_{z \in B_{X_n}} \sum_{\alpha \in J} c_{X_n}(\alpha) \vert z^{\alpha}\vert\,,
\]
where $c_{X_n}(\alpha)$, the characteristic of the multi index $\alpha$,
is the reciprocal of the norm of the monomial $z^{\alpha}$ in $\mathcal P_J(X_n)$ .
The interesting fact is that (based on \eqref{theprojection}) we can actually bound  the projection constant of the space $\mathcal P_{J}(X_n)$ by this quantity, that is,
\begin{equation} \label{key}
\boldsymbol{\lambda}\big(\mathcal P_{J}(X_n)\big) \leq \widehat{\boldsymbol{\lambda}}\big(\mathcal P_J(X_n)\big)
\end{equation}
(see Theorem~\ref{lambda-dash}). 
In many concrete situations,  this estimate is good enough -- the gain here is that this new parameter is more manageable compared with $\boldsymbol{\lambda}(\mathcal P_J(X_n))$.
Therefore having accurate upper bounds for $c_{X_n}(\alpha)$ will  be crucial for our purposes.
In fact it turns out that the polynomial projection constant $\widehat{\boldsymbol{\lambda}}(\mathcal P_J(X_n))$
as a~substitute of the projection constant of $\mathcal{P}_J(X_n)$  allows to extend estimates like \eqref{intro-poly}
in a~systematic and comfortable way to wider ranges of index sets $J$ and spaces $X_n$ -- but at the same time it satisfies
a~reasonable abstract theory which in many ways mimics the role  the value $\varphi_{X'_n}(n)$ plays in the 'one homogeneous
case'~\eqref{schuetty}.

In Section~\ref{Characteristics}
we define and prove the basic properties of characteristics $c_{X_n}(\alpha)$ together  with the key estimate from
\eqref{key}, and in Section~\ref{hermi} we use all this  to study the polynomial projection 
constant of spaces of polynomials on $\ell_r$.
Multiplying  Banach sequence lattices we  in 
Section~\ref{characteristic}
get interpolation estimates for the polynomial projection constant,
and in 
Section~\ref{Characteristics via Lozanovskii's theorem}
we use a deep result of Lozanovskii  on the K\"othe duality of  Calder\'on product spaces is used to prove the  relation
$$
c_{X_n}(\alpha)\,c_{X_{n}'}(\alpha) =
\frac{\vert \alpha \vert^{\vert \alpha \vert}}{\alpha^{\alpha}}\,,
$$
which plays a key role in our study.
This in particular gives a quite complete picture 
for tetrahedral indices 
in Section~\ref{Tetrahedral polynomials}.
 Section~\ref{bounds}   provides asymptotically optimal  bounds  for $c_{X_n}(\alpha)$ in the case when $X_n$ is the classical space $\ell_r^n$, a  Nakano space $\ell_{\vec{r}}^n$, a mixed  space $\ell_r^n(\ell_s^m)$, or the Marcinkiewicz/Lorentz sequence space $\ell^n_{r,s}$.
 Section~\ref{Polynomial projection constants vs  fundamental functions} 
relates the polynomial projection constant $\widehat{\boldsymbol{\lambda}}(\mathcal P_J(X_n))$ with the fundamental function
$\varphi_{X_n'}$
of the K\"othe dual of the underlying sequence lattice $X_n$, and we point out that in many cases, both quantities share similar
properties.
Finally, 
replacing 
 the unconditional basis
  by
 the 
 polynomial projection constant,
Section~\ref{compa}
repeats the program of Section~\ref{conv/conc}.

\smallskip

\section{Characteristics} \label{Characteristics}
Given a (quasi-)Banach space $X_n = (\mathbb{C}^n,\|\cdot\|)$, an index set  $J \subset \mathbb{N}_0^n$, and
$\alpha \in J$. Let $c_\alpha^{*} \colon \mathcal{P}_J(X_n) \to \mathbb{C}$ be the coefficient
functional given by
\[
c_\alpha^{*}(P): = c_\alpha(P), \quad\, P = \sum_{\beta\in J} c_{\beta}(P) z^\beta\,.
\]
In Lemma \ref{sup} a~simple argument using the Cauchy integral representation of the monomial coefficient $c_{\alpha}(P)$,
shows that the norm
\[
\|c_\alpha^{*}\|:= \sup_{\|P\|_{\mathcal{P}_J(X_n) \leq 1}} |c_\alpha(P)|
\]
is bounded by the reciprocal of the norm of the monomial $z^{\alpha}$
in $\mathcal P_m(X_n)$, i.e., by
\[
c_{X_n}(\alpha):= \frac{1}{\sup_{z \in B_{X_n}}|z^\alpha|}\,.
\]
We call this number the characteristic of the multi index $\alpha$ (with respect to the Banach space $X_n$). Moreover,  we  define the number
\begin{equation} \label{lambda-dash-def}
\widehat{\boldsymbol{\lambda}}\big(\mathcal P_J(X_n)\big):=\sup_{z\in B_{X_n}}\,\,\sum_{\alpha \in J}  c_{X_n}(\alpha) |z^\alpha |\,,
\end{equation}
and refer to this quantity as the  polynomial projection constant of $\mathcal P_J(X_n)$.

\smallskip

\begin{remark} \label{m1}
If $X_n=(\mathbb{C}^n, \|\cdot\|)$ is a~Banach space, then
$
\mathcal P_1(X_n) = \mathcal P_{J}(X_n)\,,
$
where $J = \{e_k\colon 1 \leq k \leq n\}$, and hence
\begin{equation*}
\widehat{\boldsymbol{\lambda}}\big(\mathcal P_1(X_n)\big)  = \sup_{z\in B_{X_n}} \sum_{k=1}^n |z_k| = \varphi_{X_n'}(n)\,.
\end{equation*}
\end{remark}

As mentioned above, we are going to see that the polynomial projection constant of $\mathcal P_J(X_n)$ bounds the projection constant of
$\mathcal P_J(X_n)$. Consequently, accurate estimates for the polynomial projection constant automatically lead to good
estimates for the projection constant. In fact it will turn out that many of   our concrete upper estimates for the  projection
constants of the spaces  $\mathcal{P}_{J}(X_n)$ are either derived from the Kadets-Snobar estimate \eqref{kadets1} or from estimates
for $\widehat{\boldsymbol{\lambda}}(\mathcal P_J(X_n))$.

The following lemma shows that the norm of the functional $c_{\alpha}^*$ can be bounded by the characteristic of the index
$\alpha$.

\begin{lemma} \label{sup}
Let $X_n = (\mathbb{C}^n, \|\cdot\|)$ be a Banach lattice. Then for every $P \in \mathcal{P}_J(X_n)$ and $\alpha \in J \subset
\mathbb{N}_0^{(\mathbb{N})}$ we have that
\[
|c_\alpha(P)| \leq c_{X_n}(\alpha) \Vert P \Vert_{\mathcal{P}_J(X_n)}\,;
\]
in other terms, $\,\big\|c_\alpha^{*}\big\|_{\mathcal{P}_J(X_n)^\ast} \leq  c_{X_n}(\alpha)$\,.
\end{lemma}

\begin{proof}
Note that (using an approximation argument) it is enough to prove the inequality
\[
\vert c_{\alpha}(P) \vert \leq  \frac{1}{\vert z^{\alpha} \vert}\,\Vert P \Vert_{\mathcal{P}_J(X_n)}
\]
for all $z \in B_{X_n}$ for which  $z_i\neq 0$ for all  $1 \leq i \leq n$. To see this inequality, fix such $z \in B_{X_n}$  and consider the set
$$\mathbb{D}(0,z):=\big\{w \in \mathbb{C}^n : \vert w_i \vert  \leq \vert z_i \vert \,\, \text{for each\, $1 \leq i \leq n$}\big\}\,.$$
Then by the Cauchy integral formula and
using that $X_n$ is a Banach lattice, we get
\[\vert c_{\alpha}(P) \vert  
= \Big|\frac{1}{(2\pi i)^n} \int_{\partial \mathbb{D}(0, z)} \frac{P(w)}{w_1^{\alpha_1+1}
 \cdots w_n^{\alpha_n+1}} dw_1 \cdots dw_n \Big|
\leq 
\frac{1}{\vert z^{\alpha} \vert}\,\Vert P \Vert_{\mathcal{P}_J(X_n)}\,,
\]
which proves the inequality.
\end{proof}

Coming back to the discussion from the introduction of  this chapter, we get an estimate of the projection constant by the polynomial projection
constant.

\begin{theorem} \label{lambda-dash}
Let $X_n = (\mathbb{C}^n,\|\cdot\|)$ be a~Banach lattice and  $J \subset \mathbb{N}_0^n$. Then
\[
\boldsymbol{\lambda}\big(\mathcal{P}_{J}
(X_n)\big) \leq \widehat{\boldsymbol{\lambda}}\big(\mathcal{P}_{J}(X_n)\big)\,.
\]
\end{theorem}

\begin{proof}
As in \eqref{theprojection} we define the projection $Q: \ell_\infty(B_{X_n}) \to \mathcal{P}_{J}(X_n)$ by
\[
Qf(z) := \sum_{\alpha \in J} \widetilde{c_{\alpha}^{*}}(f) z^\alpha, \quad\, f \in \ell_\infty(B_{X_n}), \, \,\, z\in B_{X_n}\,.
\]
Then by Lemma \ref{sup}
\begin{align*}
\|Q\| &=  \sup_{\|f\|_{\ell_{\infty}(B_{X_n})} \leq 1} \Big\|
\sum_{\alpha  \in J} \widetilde{c^\ast_{\alpha}} (f) z^\alpha\Big\|_{\mathcal{P}_J(X_n)} \\
& \leq \sup_{z \in B_{X_n}} \sum_{\alpha\in J} \big\|\widetilde{c^\ast_{\alpha}} \big\|_{\ell_{\infty}(B_{X_n})^{*}}\,|z^\alpha|
\\& 
= \sup_{z \in B_{X_n}} \sum_{\alpha\in J}  \big\|c^\ast_\alpha\big\|_{\mathcal P_J(X_n)^\ast}|z^\alpha| \leq \sup_{z \in B_{X_n}} \sum_{\alpha\in J}  c_{X_n}(\alpha)|z^\alpha|
 = \widehat{\boldsymbol{\lambda}}\big(\mathcal{P}_J(X_n)\big)\,,
\end{align*}
and since  $\boldsymbol{\lambda}\big(\mathcal{P}_J
(X_n)\big) \leq \|Q\|$, the desired estimate follows.
\end{proof}

\smallskip

\section{Polynomials on $\ell_r$}
\label{hermi}

It is well-known that for every $\alpha \in \Lambda(m,n)$ and $1 \leq r \leq \infty$ we have
\begin{equation} \label{dineen}
c_{\ell_{r}^n}(\alpha) = \Big(\frac{m^m}{\alpha^\alpha}\Big)^{1/r}\,;
\end{equation}
for an  elementary proof see \cite[Lemma 1.38]{dineen1999complex}, and for a proof within a~much more general setting  Theorem \ref{cE} and 
Remark~\ref{caso ell_r}. Observe as an immediate consequence that for  tetrahedral indices $\alpha \in \Lambda_T(m,n)$ with $m \leq n$, we get
\begin{equation} \label{dineen2}
c_{\ell_{r}^n}(\alpha) = m^{\frac{m}{r}}\,.
\end{equation}
By Theorem~\ref{lambda-dash} accurate upper bounds for $c_{X_n}(\alpha)$  provide good upper bounds for the
projection constant of $\mathcal{P}_J(X_n)$. The following two results illustrate this in the case of $X =\ell_r$.

\begin{proposition} \label{start-poly1}
Let $1 \leq r \leq \infty$ and $J \subset \mathbb{N}_0^{(\mathbb{N})}$ an index set. Then
\begin{align} \label{AAA4}
\widehat{\boldsymbol{\lambda}}\big(\mathcal{P}_{J_{\leq m}}(\ell_r^n)\big) \,\,\prec_{C^m} \,\, \Big( 1+\frac{n}{m}\Big)^{\frac{m}{r'}}\,,
\end{align}
where $C >0$ is universal.

Moreover,  the preceding estimates are optimal for all $m$ for which   $\Lambda_T(m) \subset J_{\leq m}$, in the sense that under this assumption
$\prec_{C^m}$ may be replaced by $\sim_{C^m}$, where $C >0$ is  only depending on $r$.

\end{proposition}

\begin{proof}  From \eqref{dineen}, H\"older's inequality,
and the binomial formula we deduce
\begin{align}\label{eq: proof of lambda hat lr}
\begin{split}
\widehat{\boldsymbol{\lambda}}
\big(\mathcal{P}_{J_{\leq m}}(\ell_r^n)\big)
& \leq \widehat{\boldsymbol{\lambda}}\big(\mathcal{P}_{\Lambda(\leq m,n)}(\ell_r^n )\big)
\\
&
\leq
\sum_{k=0}^m \widehat{\boldsymbol{\lambda}}\big(\mathcal{P}_{\Lambda(k,n)}(\ell_r^n )\big)
\\
&
\leq \sum_{k=0}^m \sup_{z \in B_{\ell_r^n}}  \sum_{\alpha \in \Lambda(k,n)}
\bigg(\frac{k^k}{\alpha^\alpha} \bigg)^\frac{1}{r} |z^\alpha|
\\
&
\leq  \sum_{k=0}^m  e^{\frac{k}{r}}\sup_{z \in B_{\ell_r^n}} \sum_{\alpha \in\Lambda(k,n)} \bigg(  \frac{k!}{\alpha!} \bigg)^\frac{1}{r} |z^\alpha| \\
& \leq e^{\frac{m}{r}} \sum_{k=0}^m  |\Lambda(k,n)|^\frac{1}{r'}\sup_{z \in B_{\ell_r^n}} \bigg(\sum_{\alpha \in\Lambda(k,n)}   \frac{k!}{\alpha!} |z^\alpha|^r\bigg)^\frac{1}{r} \\
& \leq e^{\frac{m}{r}}  \sum_{k=0}^m|\Lambda(k,n)|^\frac{1}{r'}\sup_{z \in B_{\ell_r^n}} \bigg(\sum_{k=1}^n   |z_k|^r\bigg)^\frac{k}{r}
\leq e^{\frac{m}{r}}(m+1) |\Lambda(m,n)|^\frac{1}{r'}\,.
\end{split}
\end{align}
To prove the lower bound it suffices to check that for every $m \leq n$ with $\Lambda_T(m) \subset J_{\leq m}$
\[
\Big( 1+\frac{n}{m}\Big)^{\frac{m}{r'}} \prec_{C^m} \widehat{\boldsymbol{\lambda}}\big(\mathcal{P}_{\Lambda_T(m,n)}\big)
\,.
\]
Taking $u:=\big(n^{-\frac{1}{r}}, \ldots, n^{-\frac{1}{r}}\big) \in \mathbb{C}^n$, we get by \eqref{dineen2}
\begin{align*}
\frac{m^{\frac{m}{r}}}{n^{\frac{m}{r}}} |\Lambda_T(m,n)|  & = m^{\frac{m}{r}}  \sum_{\alpha \in \Lambda_T(m,n)} |u^\alpha|
\leq \sup_{z \in B_{\ell_r^n}} \sum_{\alpha \in \Lambda_T(m,n)} c_{X_n}(\alpha) |z^\alpha|
 = \widehat{\boldsymbol{\lambda}}\big(\mathcal{P}_{\Lambda_T(m,n)}(\ell_r^n )\big)\,.
\end{align*}
Then
\[
\Big(\frac{n}{m}\Big)^{\frac{m}{r'}} =  \frac{m^{\frac{m}{r}}}{n^{\frac{m}{r}}} \Big(\frac{n}{m}\Big)^m\leq
\frac{m^{\frac{m}{r}}}{n^{\frac{m}{r}}} 
{n \choose m}
=  \frac{m^{\frac{m}{r}}}{n^{\frac{m}{r}}} |\Lambda_T(m,n)|
\leq \widehat{\boldsymbol{\lambda}}
\big(\mathcal{P}_{\Lambda_T(m,n)} (\ell_r^n )\big)\,,
\]
and this easily yields the required estimate.
\end{proof}

For the Banach sequence lattice $X=\ell_r$ the following result complements Theorem~\ref{conny3}
in the case $1\leq r \leq 2$. For $J\Lambda(m,n)$ this result was implicitly proved in \cite{defant2011bohr}(see also \cite{defant2019libro}).

\begin{theorem} \label{start-poly2}
Let $1 \leq r \leq \infty$ and $J \subset \mathbb{N}_0^{(\mathbb{N})}$ an index set. Then
\begin{align} \label{AAA3}
\boldsymbol{\lambda}\big(\mathcal{P}_{J_{\leq m}}(\ell_r^n) \big)
\,\,\prec_{C^m} \,\,
\Big( 1+\frac{n}{m}\Big)^{m \min \big\{\frac{1}{2}, \frac{1}{r'}\big\}}
\end{align}
and
\begin{align} \label{BBB3}
\boldsymbol{\chimon}\big(\mathcal{P}_{J_{\leq m}}(\ell_r^n) \big)
\,\,\prec_{C^m} \,\,
\Big( 1+\frac{n}{m}\Big)^{(m-1) \min \big\{\frac{1}{2}, \frac{1}{r'}\big\}}\,,
\end{align}
where $C >0$ only depends on $r$.

Moreover,  the preceding estimates are optimal for all $m$ for which   $\Lambda_T(m) \subset J_{\leq m}$, in the sense that under this assumption
$\prec_{C^m}$ may be replaced by $\sim_{C^m}$, where $C >0$ is  only depending on $r$.

\end{theorem}

\begin{proof}
Assume first that  $2 \leq r \leq \infty $. Then $\ell_r$  is $2$-convex, and the claims in \eqref{AAA3} and \eqref{BBB3} are consequences
of Theorem~\ref{conny3}. So let us  assume that $1 \leq r \leq 2$, and start with the upper bound for
\eqref{AAA3}. By Theorem~\ref{lambda-dash} and Proposition~\ref{start-poly1} we have
\[
\boldsymbol{\lambda}\big(\mathcal{P}_{J_{\leq m}}(\ell_r^n) \big)
\leq
\widehat{\boldsymbol{\lambda}}\big(\mathcal{P}_{J_{\leq m}}(\ell_r^n )\big) \,\,\prec_{C^m} \,\, \Big( 1+\frac{n}{m}\Big)^{\frac{m}{r'}}\,.
\]
 For the rest of
the proof we follow a similar strategy like for Theorem~\ref{conny3}. To prove the upper bound in \eqref{BBB3}, we conclude from
Corollary~\ref{main3A} and 
 the estimate from \eqref{AAA3}  that
\begin{align*}
\boldsymbol{\chimon}\big(\mathcal{P}_{J_{\leq m}}(\ell_r^n) \big) & \,\le \, \boldsymbol{\chimon}\big( \mathcal{P}_{\leq m}(\ell_r^n)\big)
\\& \,\prec_{C^m} \, \max_{1 \leq  k \leq m-1} \boldsymbol{\lambda}\big( \mathcal{P}_{k}(\ell_r^n)\big)
\,\prec_{C^m} \, \Big( 1+\frac{n}{m}\Big)^{(m-1) \min \big\{\frac{1}{2}, \frac{1}{r'}\big\}}\,.
\end{align*}
As concerning the proofs of the lower bounds, we begin with that
of \eqref{BBB3}. Note that for $m\leq n$ with $\Lambda_T(m) \subset J_{\leq m}$
\[
\boldsymbol{\chimon}\big(\mathcal{P}_{\Lambda_T(m,n)}(\ell_r^n) \big) \leq \boldsymbol{\chimon}\big(\mathcal{P}_{J_{\leq m}}(\ell_r^n) \big)\,,
\]
so that the conclusion in this case is an immediate consequence of  Proposition~\ref{innichen1}, whereas for
$m\ge n$ it is trivial (see again \eqref{obvious}). Finally, for the lower bound in \eqref{AAA3} observe  first that by Theorem~\ref{OrOuSe}
and Proposition~\ref{Cauchy} for $m \leq n$ with $\Lambda_T(m) \subset J_{\leq m}$
\[
\boldsymbol{\lambda}\big(\mathcal{P}_{\Lambda_T(m,n)}(\ell_r^n) \big) \leq
\kappa^m \boldsymbol{\lambda}\big(\mathcal{P}_{J_{m}}(\ell_r^n )\big) \leq
\kappa^m \boldsymbol{\lambda}\big(\mathcal{P}_{J_{\leq m}}(\ell_r^n) \big)\,,
\]
and by Theorem~\ref{OrOuSe} and Theorem~\ref{main3}
\[
\boldsymbol{\chimon}\big(\mathcal{P}_{\Lambda_T(m+1,n)}(\ell_r^n )\big)
\leq
e2^{m+1}\boldsymbol{\lambda}\big(\mathcal{P}_{\Lambda_T(m,n)}(\ell_r^n) \big)\,.
\]
Then the conclusion for $m \leq n$ with $\Lambda_T(m) \subset J_{\leq m }$ follows from the lower bound in  \eqref{BBB3}. Since this estimate for $m \ge n$
anyway is trivial, the proof completes.
\end{proof}

\section{Multiplication and interpolation}  \label{characteristic}
In order to calculate or estimate the polynomial projection constant we need to find precise formulas or at least   asymptotic estimates
for characteristics $c_\alpha(X_n)$. Actually this seems to be a quite  subtle problem -- even in the case of specific multi 
indices and Banach lattices. This section provides a~ few preliminary tools which later help to achieve this goal  (see, e.g., Section~\ref{bounds}).

We start with the following result.

\begin{lemma}
\label{cproduct}
Let $X_n := (\mathbb{C}^n, \|\cdot\|_X)$ and $Y_n:=(\mathbb{C}^n, \|\cdot\|_Y)$ be (quasi-)Banach lattices.
Then, for each $\alpha\in \mathbb{N}_0^n$, we have
\[
c_{X_n\circ Y_n}(\alpha) = c_{X_n}(\alpha)\,c_{Y_n}(\alpha)\,.
\]
\end{lemma}

\begin{proof}
Fix $\alpha \in \mathbb{N}_0^n$. By a compactness argument, there exists $z\in \overline{B}_{X_n\circ Y_n}$ such that
\[
c_{X_n\circ Y_n}(\alpha)^{-1} = |z^\alpha|\,.
\]
Given $\varepsilon >0$, we can find $x\in X$ and $y\in Y$ such that $z = x\cdot y$ with $\|x\|_{X_n}\,\|y\|_{Y_n}
< (1 + \varepsilon)$. Thus, there exists $u\in \overline{B}_{X_n}$ and $v\in \overline{B}_{Y_n}$ ($u=x/\|x\|_{X_n}$
and $v= y/\|y\|_{Y_n}$ work) such that $|z| \leq (1 + \varepsilon) |u|\,|v|$. This yields
\[
|z^\alpha| \leq (1 + \varepsilon)^{|\alpha|}\,|u^\alpha|\,|v^\alpha| \leq (1+ \varepsilon)^{|\alpha|}\,
\frac{1}{c_{X_n}(\alpha)}\,\frac{1}{c_{Y_n}(\alpha)}\,.
\]
Since $\varepsilon>0$ is arbitrary, we get
\[
c_{X_n}(\alpha)\,c_{Y_n}(\alpha) \leq c_{X_n\circ Y_n}(\alpha)\,.
\]
To show the reverse inequality, we choose $x\in \overline{B}_{X_n}$ and $y\in \overline{B}_{Y_n}$ such that
$$\text{$c_{X_n}(\alpha)^{-1} = |x^\alpha|$ and $c_{Y_n}(\alpha)^{-1} = |y^\alpha|$\,.}$$
 Clearly, $z:= x\cdot y\in X_n\circ Y_n$
with $z\in \overline{B}_{X_n\circ Y_n}$, so
\[
c_{X_n\circ Y_n}(\alpha)^{-1} \geq |z^\alpha| = |x^\alpha|\,|y^\alpha| = c_{X_n}(\alpha)^{-1}\,c_{Y_n}(\alpha)^{-1}
\]
as required.
\end{proof}

Further we establish  the formula for the characteristics of multi indices with respect to (quasi-)Banach lattices 
generated by the famous Calder\'on product.

Consider a~couple  $(X^0, X^1)$ of (quasi-)Banach function lattices over a~$\sigma$-finite complete measure space 
$(\Omega, \mathcal{A}, \mu)$ and let $\theta \in (0, 1)$. Following \cite{calderon1964intermediate} (in the setting of Banach 
lattices), we define
\[
X(\theta):=(X^0)^{1-\theta} (X^1)^{\theta}
\]
to be the space of all $x\in L^0(\mu)$ such that
\begin{align*}
|x| \leq \gamma \,|x_0|^{1-\theta}|x_1|^{\theta} \quad \text{$\mu$-a.e.}\,,
\end{align*}
for some $\gamma >0$ and $x_j\in \bar{B}_{X_j}$, $j\in \{0, 1\}$. The space $X(\theta)$ is equipped with the (quasi-)norm
$\|x\|_{X(\theta)} := \inf \,\gamma$, where the  infimum is taken over all such representations. Note that $X(\theta)$ is
a~(quasi-)Banach lattice (resp. Banach function lattice if $(X^0, X^1)$ is a~Banach couple) over $(\Omega, \mathcal{A}, \mu)$.

It is easy to check that for any (quasi-)Banach lattice $X$ over a~measure space $(\Omega, \mathcal{A}, \mu)$
and for all $\theta \in (0, 1)$, we have with $p=1/(1-\theta)$
\[
X^{1-\theta}L_{\infty}^{\theta} \equiv X^{(p)}\,,
\]
where
\[
X^{(p)}:= \big\{x\in L^0(\mu): \, |x|^p \in X \big\}
\]
is the~$p$-convexification of $X$ equipped with the (quasi-)norm
\[
\|x\|_{X^{(p)}}:= \big\||x|^{p}\big\|_{X}^{1/p}\,.
\]
Finally, we note that if $(X^0, X^1)$ is a~couple of complex Banach lattices with $X^0$ or $X^1$ order continuous,
then the isometric formula
\[
[X^0, X^1]_{\theta} \equiv (X^0)^{1-\theta} (X^1)^{\theta}
\]
holds, where $[\,\cdot\,]_{\theta}$ is the lower complex method of interpolation defined by Calder\'on
(see \cite{calderon1964intermediate}).

\smallskip

The following interpolation formula for the characteristics of multi indices with respect to Calder\'on
products of (quasi-) Banach lattices is going to be a crucial tool.

\smallskip
\begin{lemma}
\label{theta}
For $j\in \{0,1\}$ let $X^j_n := (\mathbb{C}^n, \|\cdot\|_{X^j})$ be a~couple of $($quasi$)$-Banach
lattices and $\alpha \in \mathbb{N}_0^n$. Then, for every $\theta \in (0, 1)$, we have
\[
c_{X_n(\theta)}(\alpha) = c_{X^0_n}(\alpha)^{1-\theta} c_{X^1_n}(\alpha)^{\theta}\,.
\]
\end{lemma}

\begin{proof}
It can be easily verified that the norm in $X_n(\theta)$ is given by the formula
\[
\|x\|_{X_n(\theta)} = \inf\big\{\|x_0\|_{X^0_n}^{1-\theta} \|x_1\|_{X^1_n}^{\theta}: \,\,
|x| = |x_0|^{1-\theta}|x_1|^{\theta}, \,\, x_0\in X^0_n,\,\, x_1 \in X^1_n\big\}\,.
\]
Then the result is proved with arguments very similar to those of Lemma \ref{cproduct}, so
we omit details.
\end{proof}

\smallskip

As an almost immediate consequence we get an  estimate for the polynomial projection constant
of spaces of polynomials defined on  Calder\'on products of (quasi-)Banach lattices.

\smallskip

\begin{theorem}
\label{lambdatheta}
For $j\in \{0,1\}$ let $X^j_n := (\mathbb{C}^n, \|\cdot\|_{X^j_n})$ be a couple of (quasi-)Banach lattices. Then for
every $\theta \in (0, 1)$ and every finite index set  $J\subset \mathbb{N}_0^n$
\[
\widehat{\boldsymbol{\lambda}}\big(\mathcal{P}_J(X_n(\theta)\big) \leq  \widehat{\boldsymbol{\lambda}}\big(\mathcal{P}_J(X^0_n)\big)^{1-\theta}
\,\widehat{\boldsymbol{\lambda}}\big(\mathcal{P}_J(X_n^1)\big)^{\theta}\,.
\]
In particular,
\[
\boldsymbol{\lambda}\big(\mathcal{P}_J(X_n(\theta)\big) \leq \widehat{\boldsymbol{\lambda}}\big(\mathcal{P}_J(X^0_n)\big)^{1-\theta}\,\widehat{\boldsymbol{\lambda}}\big(\mathcal{P}_J(X_n^1)\big)^{\theta}\,.
\]
\end{theorem}

\begin{proof}
From Lemma \ref{theta} we know that
\[
c_{X_n(\theta)}(\alpha) = c_{X_n^0}(\alpha)^{1-\theta} c_{X_n^1}(\alpha)^{\theta}\,.
\]
Fix $z\in B_{X_n(\theta)}$. Then, we can find  $u\in B_{X_n^0}$ and $v\in B_{X_n^1}$ such that
\[
|z|\leq |u|^{1-\theta} |v|^{\theta}\,.
\]
Combining with the concavity of the function $\mathbb{R}^{2}_{+} \ni (s, t)\mapsto s^{1-\theta} t^{\theta}$ yields 
\begin{align*}
\sum_{\alpha\in J} c_{X_n(\theta)}(\alpha)|z^\alpha| & =
\sum_{\alpha\in J}  (c_{X_n^0}(\alpha) |u^\alpha|)^{1-\theta}\,(c_{X_n^1}(\alpha) |v^{\alpha}|)^{\theta} \\
& \leq \bigg(\sum_{\alpha\in J} c_{X_n^0}(\alpha) |u^\alpha|\bigg)^{1-\theta}
\bigg(\sum_{\alpha\in J} c_{X_n^1}(\alpha) |v^\alpha|\bigg)^{\theta} \\
& \leq \widehat{\boldsymbol{\lambda}}\big(\mathcal{P}_J(X^0)\big)^{1-\theta}\,
\widehat{\boldsymbol{\lambda}}\big(\mathcal{P}_J(X^1)\big)^{\theta}\,.
\end{align*}
Since $z\in B_{X_n(\theta)}$ is arbitrary, the first estimate follows. The second estimate follows from Theorem~\ref{lambda-dash}.
\end{proof}

\smallskip

\section{Characteristics via Lozanovskii's theorem}\label{Characteristics via Lozanovskii's theorem}
In the following we reprove \eqref{dineen} within a much more general framework. We in particular show that
$\sup_{z \in B_{\ell_r^n}} \vert z^{\alpha} \vert$ is only attained  at the vector (up to signs)
\[
\bigg( \Big(\frac{\alpha_1}{m}\Big)^{1/r}, \dots, \Big(\frac{\alpha_n}{m}\Big)^{1/r}\bigg)\in S_{\ell_r^n} \,,
\]
for which then the formula from  \eqref{dineen} easily follows. Moreover, note that we as an immediate consequence of
\eqref{dineen} get a~sort of duality theorem for characteristics of multi indices  $\alpha \in \mathbb{N}_0^{(\mathbb{N})}$
with respect to  $\ell_r^n$:
\begin{equation*}
c_{\ell_{r}^n}(\alpha) c_{\ell_{r'}^n}(\alpha) = \frac{m^m}{\alpha^\alpha}\,, \,\,\,\,\,\,\,1 \leq r \leq \infty\,.
\end{equation*}
In fact the following result shows that this  is not at a  lucky incidence which only holds in $\ell_r^n$'s.

\smallskip

\begin{theorem}
\label{cE}
The following statements are true for any Banach lattice $X_n=(\mathbb{C}^n, \|\cdot\|_{X_n})$.
\begin{itemize}
\item[{\mbox(i)}] For each $\alpha \in \Lambda(m,n)$ one has
\[
c_{X_n}(\alpha)\,c_{X_{n}'}(\alpha) = \frac{m^m}{\alpha^{\alpha}}\,.
\]
\item[{\mbox(ii)}] For each $\alpha\in \Lambda(m, n)$ there exist $0 \leq u=(u_1,\ldots, u_n)\in X_n$
and \,$0\leq v = (v_1, \ldots, v_n) \in X_{n}'$ with
\begin{align}\label{eq:loza}
\|u\|_{X_n}~=~\|v\|_{X_{n}'}=1 \quad\textrm{and}\quad u_jv_j = \frac{\alpha_j}{m}, \quad\, j\in \{1, \ldots, n\}
\end{align}
and
\[
c_{X_n}(\alpha) = \frac{1}{u^\alpha} = \frac{m^m}{\alpha^{\alpha}} v^\alpha\ ,
\quad\,\, c_{X_{n}'}(\alpha) = \frac{1}{v^{\alpha}} = \frac{m^m}{\alpha^{\alpha}} u^\alpha\,.
\]
Moreover, for every $u,v$ satisfying \eqref{eq:loza}, we have that
\[
c_{X_n}(\alpha) = \frac{1}{u^\alpha} \quad\textrm{and}\quad c_{X_{n}'}(\alpha) = \frac{1}{v^{\alpha}}\,.
\]
\item[{\mbox(iii)}] If $X_n$ $($resp. $X_{n}')$ is a strictly convex Banach space, then $u$ $($resp. $v)$ in (ii) is  unique.
\end{itemize}
\end{theorem}

We note that the proof is heavily based on Lozanovskii's factorization Theorem~\ref{Lozanovskii}.

\begin{proof}
(i) Clearly, $X_n$ has the Fatou property by the fact that it is finite dimensional. Lozanovskii's factorization
Theorem~\ref{Lozanovskii} combined with \eqref{dineen} yields
\[
c_{X_n}(\alpha)\,c_{X_{n}'}(\alpha) = c_{\ell_1^n}(\alpha) = \frac{m^m}{\alpha^\alpha}\,.
\]

(ii) Again by  Lozanovskii's factorization Theorem~\ref{Lozanovskii},  we deduce that, given  $\alpha\in \Lambda(m, n)$,
there exist vectors $0 \leq u=(u_1,\ldots, u_n)\in X_n$ and $0\leq v = (v_1, \ldots, v_n) \in X_{n}'$ with
\[
\|u\|_{X_n} = \|v\|_{X_{n}'}=1 \quad\textrm{and}\quad u_jv_j = \frac{\alpha_j}{m}, \quad\, j\in \{1, \ldots, n\}\,.
\]
This implies by (i) that
\[
\frac{m^m}{\alpha^\alpha} = c_{X_n}(\alpha)\,c_{X_{n}'}(\alpha)
\leq \frac{1}{u^\alpha}\,\frac{1}{v^\alpha} = \frac{m^m}{\alpha^\alpha}\,,
\]
and hence
\[
\frac{1}{u^\alpha} = v^\alpha c_{X_n}(\alpha)\,c_{X_{n}'}(\alpha)
\ge
v^\alpha c_{X_n}(\alpha)\,\frac{1}{v^\alpha} = c_{X_n}(\alpha) \ge \frac{1}{u^\alpha}\,.
\]

(iii) It is easy to verify that by the strict convexity of $X_n$ (resp. $X_{n}')$, for every positive $x\in S_{\ell_1^n}$ there is a~unique positive vector $y \in S_{X_n}$ (resp. $z \in S_{X_{n}'}$) such that $|x| = y\,z$.
\end{proof}

Let us once again come back to the case $X_n = \ell^n_r$, \,$1<r<\infty$. The following remark gives more information on the
equality from \eqref{dineen}.

\smallskip

\begin{remark}\label{caso ell_r} Let $\alpha \in \Lambda(m,n)$ and $1<r<\infty$. Then for
\[
\text{$u =\bigg( \Big(\frac{\alpha_1}{m}\Big)^{1/r}, \dots, \Big(\frac{\alpha_n}{m}\Big)^{1/r}\bigg) \in S_{\ell_r^n}$ \,\,\,and\,\,\,
$v =  \bigg( \Big(\frac{\alpha_1}{m}\Big)^{1/r'}, \dots, \Big(\frac{\alpha_n}{m}\Big)^{1/r}\bigg) \in S_{\ell_{r'}^n}$}
\]
we have $u_jv_j = \frac{\alpha_j}{m}, \,\, j \in \{1, \ldots, m\}$\,. Since $\ell_r^n$ is strictly convex, it follows from
Proposition~\ref{cE} (iii) that $u$ is the~unique element in $S_{\ell_r^n}$
for which
$
c_{\ell_r^n}(\alpha) = \frac{1}{|u^\alpha|} = \frac{m^m}{\alpha^\alpha}\,.
$
\end{remark}

\smallskip

We now list a number of immediate consequences of the above results. Since for $X_n:= \ell_1^n$ we have $\ell_r^n =X_n^{(r)}$, the following corollary extends Proposition~\ref{start-poly1}.

\smallskip

\begin{corollary}
\label{pconv}
Let   $1< r <\infty$,  $J\subset \mathbb{N}_0^n$  be a finite index set and 
 $X_n := (\mathbb{C}^n, \|\cdot\|)$ be a Banach lattice. Then
\[
\widehat{\boldsymbol{\lambda}}\big(\mathcal{P}_J\big) \leq \widehat{\boldsymbol{\lambda}}\big(\mathcal{P}_J(X_n)\big)^{\frac{1}{r}}\,|J|^{\frac{1}{r'}}\,.
\]
\end{corollary}

\begin{proof}
It is easy to see that $c_{\ell_\infty^n}(\alpha) =1$ for each $\alpha\in \mathbb{N}_0^n$ (see also \eqref{dineen}), and clearly this yields
\[
\widehat{\boldsymbol{\lambda}}\big(\mathcal{P}_J(\ell_\infty^n)\big) \leq |J|\,.
\]
Since $X_n^{(r)} \equiv X_n^{1-\theta} (\ell_\infty^n)^{\theta}$ with $\theta = 1- \frac{1}{r}$, the first claim follows by Theorem~\ref{lambdatheta}. 
\end{proof}
Before we state the next consequence of Theorem~\ref{lambdatheta}, we note that in the case $m=1$ by Remark \ref{m1}, it follows that
\[
\widehat{\boldsymbol{\lambda}}\big(\mathcal{P}_1(\ell_2^n)\big)^2 =\varphi_{\ell_2^n}(n)^2 = n \leq \varphi_{X_n}(n) \varphi_{X_{n}'}(n) = \widehat{\boldsymbol{\lambda}}\big(\mathcal{P}_1(X_n)\big)
\widehat{\boldsymbol{\lambda}}\big(\mathcal{P}_1(X_{n}')\big)\,.
\]
The following extension  may be interpreted as  a sort of 'trace duality' for polynomial projection constants.

\begin{corollary} \label{duality}
Let $X_n = (\mathbb{C}^n,\|\cdot\|)$ be a~Banach lattice  and $J \subset \mathbb{N}_0^n$ a finite index set. Then
\[
\widehat{\boldsymbol{\lambda}}\big(\mathcal{P}_J(\ell_2^n)\big)^2 \,\,\leq \,\, \widehat{\boldsymbol{\lambda}}\big(\mathcal{P}_J(X_n)\big)\,\widehat{\boldsymbol{\lambda}}\big(\mathcal{P}_J(X_{n}')
\big)\,,
\]
with equality for $X_n = \ell^n_2$. Moreover,  there is $C>0$ such that for any Banach lattice $X_n = (\mathbb{C}^n,\|\cdot\|)$
and any index set $J \subset \mathbb{N}_0^n$, such that  $\Lambda_T(m,n) \subset J$ whenever $m\leq n$, we have
\[
\frac{1}{C^m}\Big(\frac{n}{m}\Big)^{m} \,\,\leq \,\, \widehat{\boldsymbol{\lambda}}\big(\mathcal{P}_J(X_n)\big)\,\widehat{\boldsymbol{\lambda}}\big(\mathcal{P}_J(X_{n}')
\big)\,.
\]
\end{corollary}

\begin{proof}
By Lozanovskii's factorization Theorem~\ref{Lozanovskii} one has
\[
X_n^{1/2} (X_n')^{1/2} \equiv \ell^n_2\,.
\]
Using the interpolation estimate of $\widehat{\boldsymbol{\lambda}}\big(\mathcal{P}_J(X_n(\theta)\big)$ for the Calder\'on product space
$X(\theta) := (X_0)^{1-\theta} (X_n^1)^{\theta}$ with $\theta \in (0, 1)$, Theorem~\ref{lambdatheta} gives  the first estimate.
The second assertion then is a consequence of Theorem~\ref{lambda-dash} 
and the fact that
$\boldsymbol{\lambda}\big(\mathcal{P}_{J}(\ell_2^n)\big) \sim_{C^m} \big(\frac{n}{m}\big)^{\frac{m}{2}}$ proved in Theorem~\ref{conny3}.
\end{proof}

We conclude with the following observation that Theorem~\ref{lambda-dash} combined with Corollary~\ref{RW} gives interesting estimate
true for any Banach lattice $X_n = (\mathbb{C}^n,~\|\cdot~\|)$ and $J = \Lambda(m,n)$
\[
\left(\frac{\Gamma(n+m) \Gamma(1+\frac{m}{2})}{\Gamma(1+m)\Gamma(n+\frac{m}{2})} \right)^2
\,\,\leq \,\, \widehat{\boldsymbol{\lambda}}\big(\mathcal{P}_m(X_n)\big)\,\widehat{\boldsymbol{\lambda}}\big(\mathcal{P}_m(X_{n}')
\big)\,.
\]

\smallskip

\section{Tetrahedral polynomials} \label{Tetrahedral polynomials}

We begin by estimating the characteristics $c_{X_n}(\alpha)$ for tetrahedral indices $\alpha$, i.e,  $\alpha \in \Lambda_T(m,n)$,
and symmetric Banach lattices $X_n$. Extending \eqref{dineen2}, we show that they are given by the $m$th power of the fundamental
function $\varphi_{X_n}$ in $m$.

\begin{proposition} \label{mietek}
Let $X_n = (\mathbb{C}^n, \|\cdot\|)$ be a symmetric Banach lattice, and  $1\leq m \leq n$. Then for each $\alpha~\in~\Lambda_T(m,n)$,
one has
\begin{equation*}
c_{X_n}(\alpha)= \varphi_{X_n}(m)^m.
\end{equation*}
\end{proposition}

\begin{proof}
Let us show that $\Vert z^{\alpha} \Vert_{\mathcal P_m (X_n)} = \frac{1}{\varphi_{X_n}(m)^m}$.
If we identify $\alpha \in \Lambda_T(m,n)$ with $\jj = (j_1, \ldots, j_m)  \in \mathcal{J}_T(m,n)$, we evaluate
the monomial $z^\alpha =z_\jj$ at
\[
z = \frac{1}{\varphi_{X_n}(m)} (\underbrace{1, \ldots, 1}_{m}, 0, \ldots) \in  B_{X_n}\,,
\]
which clearly gives one of the desired estimates. To see  the other, observe that for
 $(z_1, \ldots, z_n) \in B_{X_n}$
\begin{align*}
&
|z_{j_1}\ldots z_{j_m}|^{1/m} \leq \frac{1}{m} (|z_{j_1}| + \ldots + |z_{j_m}|)
\\&\leq \frac{1}{m} \|(\underbrace{1, \ldots, 1}_{m},0,\ldots)\|_{X_{n}'}
\|(z_{j_1}, \ldots, z_{j_m},0  \ldots)\|_{X_n}
\\&
\leq \frac{1}{m} \|(\underbrace{1, \ldots, 1}_{m},0,\ldots)\|_{X_{n}'} \|(z_{1}, \ldots, z_{n},0  \ldots)\|_{X_n}
\leq \frac{1}{m} \frac{m}{\varphi_{X_n}(m)} = \frac{1}{\varphi_{X_n}(m)}\,.
\end{align*}
Taking powers completes the argument.
\end{proof}

We now exhibit   upper and lower bounds for the polynomial projection constant of spaces of tetrahedral
polynomials acting on symmetric Banach lattices -- for the case  $m=1$ compare with  Remark~\ref{m1}.

\begin{proposition}\label{lambda}
Let $X_n= (\mathbb{C}^n, \|\cdot\|)$~be a symmetric Banach lattice. Then for each $ m \leq n$
\[
\bigg(\frac{\varphi_{X_{n}'}(n)}{\varphi_{X_{n}'}(m)}\bigg)^m
\leq \widehat{\boldsymbol{\lambda}}\big(\mathcal{P}_{\Lambda_T(m,n)}(X_n)\big)
\leq e^m \bigg(\frac{\varphi_{X_{n}'}(n)}{\varphi_{X_{n}'}(m)}\bigg)^m \,.
\]
\end{proposition}

 The following independently interesting lemma (later also used in Proposition~\ref{lambdaII}) prepares the proof.

\smallskip

\begin{lemma}
\label{lambdaP_J}
Let $X_n = (\mathbb{C}^n,\|\cdot\|)$ be a~Banach lattice. Then for  $J\subset \Lambda(m, n)$ with $m \leq n$, we have
\[
\widehat{\boldsymbol{\lambda}}\big(\mathcal{P}_J(X_n)\big) \leq e^{m} \sup_{z\in B_{X_n}} \sum_{\alpha\in J}
\frac{m!}{\alpha!}\,\Delta_{X_n}(\alpha)\,|z^\alpha|\,,
\]
where $\Delta_{X_n}(\alpha) := \big\|\big(\frac{\alpha_1}{m}, \ldots, \frac{\alpha_n}{m}\big)\big\|_{X_n}^{m}$.
\end{lemma}

\begin{proof}
By Theorem \ref{cE}, (ii) for each $\alpha = (\alpha_1, \ldots, \alpha_n) \in \Lambda(m, n)$, we  find $v=v(\alpha) \geq 0$ in $X_{n}'$ such that $\|v\|_{X_{n}'}=1$ and
\[
c_{X_n}(\alpha)  \leq \frac{m^m}{\alpha^{\alpha}} v^{\alpha}\,.
\]
Now observe that 
\[
(v^{\alpha})^{1/m}= \big(v_1^{\alpha_1} \cdots v_n^{\alpha_n}\big)^{1/m}
\leq \frac{1}{m}(\alpha_1 v_1 + \ldots + \alpha_n v_n)\,.
\]
This implies that
\[
(v^{\alpha})^{1/m} \leq \frac{1}{m} \|(\alpha_1, \ldots, \alpha_n)\|_{X_n}\,\|(v_1, \ldots, v_n)\|_{X_{n}'}
= \frac{1}{m}\,\|(\alpha_1, \ldots, \alpha_n)\|_{X_n}\,,
\]
and so
\[
v^{\alpha} \leq  \Big\|\Big(\frac{\alpha_1}{m}, \ldots, \frac{\alpha_n}{m}\Big)\Big\|_{X_n}^{m} = \Delta_{X_n}(\alpha)\,.
\]
This estimate combined with the obvious inequality $m^{m} \leq e^{m} m!$, gives as desired
\begin{align*} \label{tag}
\widehat{\boldsymbol{\lambda}}\big(\mathcal{P}_J(X_n)\big) & = \sup_{z\in B_{X_n}} \sum_{\alpha\in J} c_{X_n}(\alpha) \,|z^\alpha|
\leq \sup_{z\in B_{X_n}} \sum_{\alpha\in J} \frac{m^m}{\alpha^{\alpha}} v^{\alpha} \,|z^\alpha| \\
& \leq e^{m} \sup_{z\in B_{X_n}} \sum_{\alpha\in J} \frac{m!}{\alpha!}\,\Delta_{X_n}(\alpha)\,|z^\alpha|
\,.\qedhere
\end{align*}
\end{proof}

\bigskip

\begin{proof}[Proof of Proposition~\ref{lambda}]
We start with the proof of the lower bound (which is very similar to an argument used for Theorem~\ref{start-poly1}):
By Proposition~\ref{mietek}
\begin{align*}
\varphi_{X_n}(m)^m  \sum_{\alpha \in \Lambda_T(m,n)} \Big| \Big(\frac{1}{\varphi_{X_n}(n)}, \ldots, \frac{1}{\varphi_{X_n}(n)}\Big)^\alpha\Big|
\leq \sup_{z \in B_{X_n}} \sum_{\alpha \in \Lambda_T(m,n)} c_{X_n}(\alpha) |z^\alpha|
= \widehat{\boldsymbol{\lambda}}\big(\mathcal{P}_{\Lambda_T(m,n)}(X_n )\big)\,,
\end{align*}
and hence
\[
\bigg(\frac{\varphi_{X_n}(m)}{\varphi_{X_n}(n)}\bigg)^m \Big(\frac{n}{m}\Big)^m
\leq \bigg(\frac{\varphi_{X_n}(m)}{\varphi_{X_n}(n)}\bigg)^m  \binom{n}{m}
= \bigg(\frac{\varphi_{X_n}(m)}{\varphi_{X_n}(n)}\bigg)^m |\Lambda_T(m,n)|
\leq \widehat{\boldsymbol{\lambda}}\big(\mathcal{P}_{\Lambda_T(m,n)}(\ell_r^n )\big)\,.
\]
Since by the assumed  symmetry of $X_n$ we have $\varphi_{X_n}(k) \varphi_{X_{n}'}(k) = k$ for each $1 \leq k \leq m$, the conclusion follows.

For the proof of the upper bound, note first that for any $\alpha =(\alpha_1, \ldots, \alpha_n) \in \Lambda (m, n)$
with $m\leq n$, we have $\alpha_{j}^{*} = 0$ for $m<j\leq n$. This implies (recall again that $X_n$ is symmetric) that
\begin{align*}
\Delta_{X_n}(\alpha) & = \Big\|\Big(\frac{\alpha_1}{m}, \ldots, \frac{\alpha_n}{m}\Big)\Big\|_{X_n}^m =
\Big\|\Big(\frac{\alpha_1^{*}}{m}, \ldots, \frac{\alpha_m^{*}}{m}\Big)\Big\|_{X_m}^{m}\,.
\end{align*}
Thus for $\alpha \in \Lambda_T(m,n)$ we have $\Delta_{X_n}(\alpha) = (\varphi_X(m)/m)^m = (1/\varphi_{X_{n}'}(m))^m$.
Moreover, by the multinomial formula
\begin{equation}\label{multinomial}
\sup_{z\in B_{X_n}} \sum_{\alpha\in \Lambda(m, n)} \frac{m!}{\alpha!}\,|z^\alpha| =
\sup_{z \in B_{X_n}} (|z_1| + \ldots + |z_n|)^m = \varphi_{X_{n}'}(n)^m\,.
\end{equation}
All together by Lemma~\ref{lambdaP_J}
\begin{align*}
\widehat{\boldsymbol{\lambda}}\big(\mathcal{P}_{\Lambda_T(m,n)}(X_n)\big) & 
\leq e^{m} \sup_{z\in B_{X_n}} \sum_{\alpha\in \Lambda_T(m,n)}
\frac{m!}{\alpha^{\alpha}}\,\Delta_{X_n}(\alpha)\,|z^\alpha| \\
& \leq \frac{ e^{m}}{\varphi_{X_{n}'}(m)^m}\sup_{z\in B_{X_n}} \sum_{\alpha\in \Lambda(m,n)}
\frac{m!}{\alpha!}\,|z^\alpha| = e^m \bigg(\frac{\varphi_{X_{n}'}(n)}{\varphi_{X_{n}'}(m)}\bigg)^m \,.
\end{align*}
This completes the proof.
\end{proof}

\smallskip

\section{Concrete bounds for characteristics}  \label{bounds}

It was mentioned that, given a multi index $\alpha$ and some
finite-dimensional Banach lattices $X_n$,  the task to find a precise formula or only  an asymptotic estimate for $c_{X_n}(\alpha)$  seems a quite
delicate problem. Nevertheless we here  provide a few  concrete formulas in this direction.

\subsection{Nakano spaces}

Given $\vec{p}:=(p_1, \ldots, p_n) \in [1, \infty)^n$. The absolutely convex, closed and bounded subset
\[
\big\{z=(z_1, \ldots, z_n) \in \mathbb{C}^n :\,\, |z_1|^{p_1}+ \ldots + |z_n|^{p_n} \le 1\big\}
\]
defines via its Minkowski functional a~lattice norm on  $\mathbb{C}^n$ . The space $\mathbb{C}^n$ equipped with this norm is denoted by $\ell_{\vec{p}}^n$
and is called the Nakano space. Note that in the case $p_1= \ldots = p_n =p$, we recover the classical $\ell_p^n$-space.

We  apply  Theorem \ref{cE} to prove precise formulas for the characteristics of multi indices with respect to  strictly convex Nakano spaces and their
K\"othe duals. A crucial tool for our proof  is (a special case of) the well-known Young inequality for convex $N$-functions. Recall that a~function
$\varphi\colon [0, \infty) \to [0, \infty)$, which is nondecreasing, continuous,  convex with $\varphi(0)=0$ is called $N$-function, if
\[
\lim_{t\to 0+} \frac{\varphi(t)}{t} = \lim_{t \to \infty}\,\frac{t}{\varphi(t)}=0\,.
\]
A  well-known and easily verified fact is that $\varphi$ is an
$N$-function if and only if it has a~representation
\[
\varphi(t) = \int_0^t p(s)\,ds, \quad\, t\geq 0\,,
\]
where $p\colon [0, \infty) \to [0, \infty)$ is non-decreasing, right-continuous function 
(the right derivative of $\varphi$) and $p(0)=0$, and $\lim_{t\to \infty}p(t)=0$ 
(see \cite{krasnosel1960convex}).

For any $N$-function $\varphi$, given by the above integral formula, we define  the right inverse
\[
q\colon [0, \infty) \to [0, \infty)\,, \quad q(t):= \sup\{s>0: \, p(s) \leq t\}
\]
 of the function $p$.
 The $N$-function $\varphi^{*}$ defined by
\[
\varphi^{*}(t) = \int_0^t q(s)\,ds, \quad\, t\geq 0
\]
is called  complementary to $\varphi$. The mentioned Young inequality states that
\[
uv \leq \varphi(u) + \varphi^{*}(v), \quad\, u, v\geq 0\,,
\]
and here equality  holds if and only if $v=p(u)$ or $u=q(v)$ (see \cite{krasnosel1960convex}).

We are ready to state the following result.

\smallskip

\begin{proposition}
\label{nakano}
Given $\vec{p}:=(p_1, \ldots, p_n) \in (1, \infty)^n$ and $\alpha \in \Lambda(m, n)$, let
$\gamma := m \big(\frac{\alpha_1}{p_1} + \ldots + \frac{\alpha_n}{p_n}\big)^{-1}$. Then
$u =(u_1,\ldots, u_n)$ and $v= (v_1, \ldots, v_n)$, which  for each $j \in \{1, \ldots, n\}$ are defined by
\begin{align*}
u_j := x_j^{1/p_j} \quad\, \text{and \,\,\, $v_j:= \frac{1}{\gamma} p_j \,x_j^{1/p_{j}'}$}
\end{align*}
with $x_j = \frac{\alpha_j}{p_j} \Big(\frac{\alpha_1}{p_1} + \ldots + \frac{\alpha_n}{p_n}\Big)^{-1}$,
satisfy the following two properties
\begin{align*}
uv = \Big(\frac{\alpha_1}{m}, \ldots, \frac{\alpha_n}{m}\Big) \quad\, \text{and \,\,\,
$\|u\|_{\ell_{\vec{p}}^n} = \|v\|_{(\ell_{\vec{p}}^n)'} = 1$}\,,
\end{align*}
where $\vec{p}' :=(p_{1}', \ldots, p_{n}')$. Moreover, $u$ and $v$ are the only non-negative elements with these properties.
\end{proposition}

\begin{proof}
It is obvious that $uv = \big(\frac{\alpha_1}{m}, \ldots, \frac{\alpha_n}{m}\big)$ and $\|u\|_{\ell_{\vec{p}}^n} =1$.
We claim that also $\|v\|_{(\ell_{\vec{p}}^n)'} = 1$. To prove this observe first that
\[
1 = \sum_{j=1}^n u_j v_j \leq \|u\|_{\ell_{\vec{p}}^n} \|v\|_{(\ell_{\vec{p}}^n)'} =  \|v\|_{(\ell_{\vec{p}}^n)'}\,.
\]
Thus it is enough to show that $\|v\|_{(\ell_{\vec{p}}^n)'}
\leq 1$. To do this, we consider the sequence $\{\varphi_j\}_{j = 1}^n$ of $N$-functions given by
$\varphi_{j}(t) := t^{p_j}$ for all $t\geq 0$. Since $\varphi_{j}'(t) = p_jt^{p_{j}-1}$, it follows that in Young's
inequalities the following equalities hold:
\begin{align*}
u_{j} \big(p_ju_{j}^{p_{j}-1}\big) = u_{j}^{p_j} + \varphi_j^{*}\big(p_j u_j^{p_{j} - 1}\big),
\quad\, j\in \{1, \ldots, n\}\,.
\end{align*}
Note that $\|(z_j)\|_{\ell_{\vec{p}}^n} = 1$ is equivalent to $\sum_{j=1}^n |z_j|^{p_j} = 1$. Observe also that
$\sum_{j=1}^n p_j u_j^{p_j} = \gamma$ and
\[
v_j= \frac{1}{\gamma} p_j \,x_j^{1/p_{j}'} = \frac{1}{\gamma} p_j u_{j}^{p_j-1}, \quad\, j\in \{1, \ldots, n\}\,.
\]
Combining these equalities with Young's inequality for $\varphi_j$ and $\varphi_{j}^{*}$ yields
\begin{align*}
\|v\|_{(\ell_{\vec{p}}^n)'} & = \sup_{\|(z_j)\|_{\ell_{\vec{p}}^n} = 1} \,\sum_{j=1}^n |z_j| v_j
\\
&
\leq \frac{1}{\gamma}\,
\sup_{\|(z_j)\|_{\ell_{\vec{p}}^n} = 1} \,\sum_{j=1}^n \big(|z_j|^{p_j} + \varphi_{j}^{*}\big(p_ju_j^{p_{j}-1}\big)\big) \\
& \leq \frac{1}{\gamma} \sum_{j=1}^{n} \big(u_j^{p_j} + \varphi_j^{*}\big(p_j u_j^{p_j - 1}\big)\big)
= \frac{1}{\gamma} \sum_{j=1}^n p_j u_j^{p_j} = 1\,.
\end{align*}
This proves the claim. It is easy to verify that $\ell_{\vec{p}}^n$ is strictly convex, and so this completes
the proof.
\end{proof}

\smallskip

The following formulas for characteristics are  immediate consequence of Theorem~\ref{cE} and Proposition~\ref{nakano}.

\smallskip

\begin{proposition}
\label{nakano-ch}
For every $\vec{p} := (p_1, \ldots, p_n) \in (1, \infty)^n$ and every $\alpha=(\alpha_1, \ldots, \alpha_n)\in \Lambda(m, n)$
one has
\[
c_{\ell_{\vec{p}}^n}(\alpha)
= \frac{\big(\frac{\alpha_1}{p_1} + \ldots + \frac{\alpha_n}{p_n}\big)^{\frac{\alpha_1}{p_1} + \ldots + \frac{\alpha_n}{p_n}}}
{\big(\frac{\alpha_1}{p_1}\big)^{\frac{\alpha_1}{p_1}}\cdots \big(\frac{\alpha_n}{p_n}\big)^{\frac{\alpha_n}{p_n}}}, \quad\, \,\,\,
c_{(\ell_{\vec{p}}^n)'}(\alpha) = \frac{m^m}{\alpha^{\alpha}}\cdot \frac{\big(\frac{\alpha_1}{p_1}\big)^{\frac{\alpha_1}{p_1}}\cdots \big(\frac{\alpha_n}{p_n}\big)^{\frac{\alpha_n}{p_n}}}{\big(\frac{\alpha_1}{p_1} + \ldots + \frac{\alpha_n}{p_n}\big)^{\frac{\alpha_1}{p_1} + \ldots + \frac{\alpha_n}{p_n}}}\,.
\]
\end{proposition}

We conclude with the remark that the first  formula is a consequence of \cite[Lemma~3.5]{defant2004estimates} proved in a~different way.

\smallskip

\subsection{Mixed $\pmb{\ell_p}$-spaces}

We  give precise formulas for the characteristics of multi indices with respect to so-called mixed $ \ell_p $-spaces.
Recall that, given positive integers $n,k>1$ and  $p, q\in [1, \infty]$, the mixed space $\ell_p^{n}(\ell_q^k)$ is
the~complex Banach lattice of all $n\times k$ matrices $x=(x_{ij})$ equipped with the norm
\[
\|x\|_{\ell_p^n(\ell_q^k)} = \Big\|\big(\big\|(x_{ij})_{j=1}^k\big\|_{\ell_q^k}\big)_{i=1}^n\Big\|_{\ell_p^n}
=\Big\|\big(\big\|x_{i\bullet}\big\|_{\ell_q^k}\big)_{i=1}^n\Big\|_{\ell_p^n}
\,;
\]
where $x_{i\bullet}$ denotes the vector determined by the $i$-th row of $x$ (similarly, $x_{\bullet j}$ denotes the $j$-th column of $x$). In what follows, for notational convenience,  we adopt the  convention $\frac{1}{\infty}:=0$.

We start with the following result.

\begin{lemma}\label{mixed}
Let $p, q \in [1, \infty]$ and  $0\leq \gamma =(\gamma_{ij}) \in \ell_1^n(\ell_1^k)$ with $\|\gamma\|_{\ell_1^n(\ell_1^k)}=1$.~Then
$x=(x_{ij})$ and $y=(y_{ij})$ given for each $1\leq i\leq n$ and $1\leq j\leq k$ by
\[
x_{ij}:= \gamma_{ij}^{\frac{1}{q}}\,\Big(\sum_{l=1}^k \gamma_{il}\Big)^{\frac{1}{p}- \frac{1}{q}}  = \gamma_{ij}^{\frac{1}{q}}\|\gamma_{i\bullet}\|_{\ell_1^k}^{\frac{1}{p}- \frac{1}{q}} ,\quad\,
y_{ij}:= \gamma_{ij}^{\frac{1}{q'}}\,\Big(\sum_{l=1}^k \gamma_{il}\Big)^{\frac{1}{p'}- \frac{1}{q'}}
 = \gamma_{ij}^{\frac{1}{q'}}\|\gamma_{i\bullet}\|_{\ell_1^k}^{\frac{1}{p'}- \frac{1}{q'}}
\]
satisfy the following properties{\rm:}
\[
x y = (x_{ij} y_{ij}) = \gamma, \quad\, \|x\|_{\ell_p^n(\ell_q^k)} =  \|y\|_{\ell_{p'}^n(\ell_{q'}^k)} = 1\,.
\]
If $p, q \in (1, \infty)$, then $x$ and $y$ are the only elements with these properties.
\end{lemma}

\begin{proof}
Since $\frac{1}{p} + \frac{1}{p'} = 1 = \frac{1}{q} + \frac{1}{q'}$, so for each $1\leq i\leq n$, $1\leq j\leq k$,
\[
x_{ij} y_{ij} =  \gamma_{ij}^{\frac{1}{q} + \frac{1}{q'}}\Big(\sum_{l=1}^k \gamma_{il}\Big)^{\frac{1}{p}- \frac{1}{q}
+ \frac{1}{p'} - \frac{1}{q'}}
= \gamma_{ij}\,.
\]
Clearly,  $\|x\|_{\ell_p^n(\ell_q^k)}=1$ and $\|y\|_{\ell_{p'}^n(\ell_{q'}^k)}=1$.~The uniqueness follows
by the K\"othe duality formula
\[
\ell_p^n(\ell_q^k)^{\prime} \equiv  \ell_{p'}^n(\ell_{q'}^k)\,,
\]
and the fact that $\ell_p^n(\ell_q^k)$ is a~strictly convex Banach space.
\end{proof}

  A monomial on the mixed space $\ell_p^n (\ell_q^k)$ is determined by a multi-index which is an $n\times k$ matrix of non-negative integers. We denote $\Lambda(m,n\times k)$ to the set of $n\times k$ multi-indices of degree $m$, that is $\alpha \in\Lambda(m,n\times k)$ if $\displaystyle\sum_{i,j}\alpha_{ij}=m$.

\medskip

We are ready to state the following result about characteristics of multi indices  with respect to the  mixed space $\ell_p^n(\ell_q^k)$.

\begin{proposition}
For every $p, q\in [1, \infty]$ and for each $\alpha \in \Lambda(m,n\times k)$ one has
\[
c_{\ell_p^n(\ell_q^k)}(\alpha)
= m^{\frac{m}{p}} \prod_{i=1}^n\,\prod_{j=1}^k \bigg(\frac{\|\alpha_{i\bullet}\|_{\ell_1}^{1/q - 1/p}}{\alpha_{ij}^{1/q}}\bigg)^{\alpha_{ij}}
= \frac{m^{\frac{m}{p}}}{\alpha^{\frac{\alpha}{q}}}  \Bigg(\prod_{i=1}^n\|\alpha_{i\bullet}\|_{\ell_1}^{\|\alpha_{i\bullet}\|_{\ell_1}}\Bigg)^{\frac1{q}-\frac{1}{p}}\,.
\]
\end{proposition}

\begin{proof}
We let $\gamma:= \big(\frac{\alpha_{ij}}{m}\big)$. Since $\gamma \geq 0$ with $\|\gamma\|_{\ell_1^n(\ell_1^k)} =1$,
it follows from Lemma \ref{mixed} that the elements $x=(x_{ij})\geq 0$ and $y=(y_{ij})\geq 0$ given by
\begin{align*}
x_{ij}:= \gamma_{ij}^{\frac{1}{q}}\,\Big(\sum_{l=1}^k
\gamma_{il}\Big)^{\frac{1}{p}- \frac{1}{q}}& =
m^{-\frac{1}{p}}\,\alpha_{ij}^{\frac{1}{q}}\,\Big(\sum_{l=1}^k \alpha_{il}\Big)^{\frac{1}{p}- \frac{1}{q}}
= m^{-\frac{1}{p}}\,\alpha_{ij}^{\frac{1}{q}}\,\|\alpha_{i\bullet}\|_{\ell_1}^{\frac{1}{q}- \frac{1}{p}}\,,
\end{align*}
\[
y_{ij}:= \gamma_{ij}^{\frac{1}{q'}}\,\Big(\sum_{l=1}^k \gamma_{il}\Big)^{\frac{1}{p'}- \frac{1}{q'}}
= m^{-\frac{1}{p'}}\,\alpha_{ij}^{\frac{1}{q'}}\,\|\alpha_{i\bullet}\|_{\ell_1}^{\frac{1}{q'}- \frac{1}{p'}}
\]
satisfy
\[
(x_{ij} y_{ij}) = \gamma, \quad\, \|x\|_{\ell_p^n(\ell_q^k)} =  \|y\|_{\ell_{p'}^n(\ell_{q'}^k)} = 1\,.
\]
Thus applying Theorem \ref{cE} we get
\begin{align*}
c_{\ell_p^n(\ell_q^k)}(\alpha) & :=
\bigg(\sup_{z \in B_{\ell_p^n(\ell_q^k)}}\,\prod_{i=1}^n \prod_{j=1}^k |z_{ij}|^{\alpha_{ij}}\bigg)^{-1} 
\,=\,
\prod_{i=1}^n \prod_{j=1}^k x_{ij}^{-\alpha_{ij}} = m^{\frac{m}{p}} \prod_{i=1}^n\,\prod_{j=1}^k \bigg(\frac{\|\alpha_{i\bullet}\|_{\ell_1}^{1/q - 1/p}}{\alpha_{ij}^{1/q}}\bigg)^{\alpha_{ij}}\,.
\end{align*}
This completes the proof.
\end{proof}

\smallskip

\subsection{Lorentz and Marcinkiewicz \,spaces} 
\label{Marcinkiewicz spaces}

In our research of multivariate polynomials equipped with uniform norms on unit spheres of Banach spaces, we will focus on
an important class of symmetric sequence spaces, namely Lorentz and Marcinkiewicz spaces. We include some essential definitions and properties which we are going to  use freely and without further reference.  For the study of Lorentz sequence spaces we refer to \cite{LT1}, \cite{creekmore1981type},  \cite{Reisner}, \cite{Jameson}.
   
Given $1\leq p<\infty$ and a~positive non-increasing (resp. non-decreasing) sequence $w:=(w_k)_{k\in \mathbb{N}}$.
Following \cite{LT1}, the Lorentz space $d(w, p)$ is defined to be the symmetric Banach (resp. quasi-Banach) sequence 
space on $\mathbb{N}$ of all scalar sequences $x=(x_k)$ equipped with the norm (resp. quasi-norm) given by
\begin{align*} \label{d(wp)}
\|x\|_{d(w,p)} := \Big(\sum_{k=1}^\infty (x_k^{*})^p w_k\Big)^{\frac{1}{p}}\,,
\end{align*}
where $x^{*}:= (x^{*}_k)$ denotes the decreasing rearrangement of $|x|$. In what follows, for each $n\in \mathbb{N}$, we write
$d^n(w, p)$ instead  of $d(w,p)^n$. Note that the fundamental function of $d(w, p)$ is given by
\[
\varphi_{d(w,p)}(n) = (w_1 + \ldots + w_n)^{\frac{1}{p}}, \quad\, n\in \mathbb{N}\,.
\]

In what follows we use the well known fact that when $(w_k)$ is a non-increasing sequence, the K\"othe dual space $d(w, 1)'$ of $d(w,1)$ coincides isometrically with the Marcinkiewicz space $m_w$ of all sequences $(x_k)$ equipped with the norm
\begin{equation*} \label{mw}
\|x\|_{m_w} := \sup_{n\geq 1} \frac{\sum_{k=1}^n x_{k}^{*}}{w_1 + \ldots + w_n}\,.
\end{equation*}
Observe that if $1<p<\infty$\, and $p'$ denotes the conjugate exponent of $p$, then the Marcinkiewicz 
space $m_{w}$ generated by a~decreasing sequence $w$ given by $w:=\big(k^{1/p'} - (k-1)^{1/p'}\big)_{k\in \mathbb{N}}$ is usually 
denoted by $\ell_{p, \infty}$, and it is equipped with the norm
\[
\|x\|_{\ell_{p, \infty}}: = \sup_{n\geq 1} \frac{\sum_{k=1}^n\,x_k^{*}}{n^{1-1/p}}\,.
\]

We place special emphasis on the case of Lorentz spaces $d(w, p)$ with $w=\big(k^{\frac{q}{p} -1}\big)_{k\in \mathbb{N}}$,
where $p\in (1, \infty)$ and $q\in [1, \infty)$. As usual, this space is denoted by $\ell_{p,q}$ and its fundamental function 
satisfies the equivalence
\[
\varphi_{\ell_{p,q}}(n) \sim n^{\frac{1}{p}}\,.
\]
Note that if $1\leq q<p$, then $\ell_{p, q}$ is a~symmetric Banach  space. It is worth noting that $\ell_{p, q} = d(v, q)$ up 
to equivalence of norms, where $v:= \big(k^{\frac{q}{p}} - (k-1)^{\frac{q}{p}}\big)_{k\in \mathbb{N}}$. This renorming of
$\ell_{p, q}$ is useful because the fundamental function $d(v, q)$ is expressed by the precise formula
\[
\varphi_{d(v, q)}(n) = n^{\frac{1}{p}}, \quad\, n\in \mathbb{N}\,.
\]
Note also that in the case when $q>p$ the Lorentz space
$\ell_{p, q}$ forms a~symmetric Banach sequence space  whenever it is equipped with the norm
\[
\|x\|_{\ell_{p,q}}^{*}:= \Big(\sum_{n=1}^\infty n^{\frac{q}{p} -1} \Big( \frac{1}{n} \sum_{k=1}^n x_k^* \Big)^q \Big)^{1/q}\,,
\]
which satisfies
\begin{align*}
\|x\|_{\ell_{p,q}} \leq \Vert x \Vert_{\ell_{p,q}}^* \leq p'\, \|x\|_{\ell_{p,q}}, \quad\, x\in \ell_{p,q}\,.
\end{align*}

We heavily use the fact that the spaces $\ell_{p,q}$ are ordered lexicographically:
\begin{align*}
&
\ell_{p,q} \hookrightarrow \ell_{r, s},  \quad\,\,\,\,\, \text{for \, $p < r$} \\
&
\ell_{p, q} \hookrightarrow \ell_{r, s}, \quad\,\,\,\,\, \text{for \, $r = s$ \, and \, $q < s$\,.}
&
\end{align*}

After these preliminaries let us start providing estimates for the characteristics of finite dimensional Marcinkiewicz $m_w$ and Lorentz spaces $d(w,1)$. Below we use a~well known inequality from the classical book by Hardy, Littlewood and P\'olya \cite{hardy1952inequalities} involving finite sequences with terms rearranged. This
inequality states that for any $(a_1,\ldots, a_n)  \in \mathbb{R}^{n}_{+}$ and $(b_1, \ldots, b_n)\in \mathbb{R}^{n}_{+}$ one has
\begin{equation}\label{HLP}
a_1 b_1 + \ldots + a_n b_n \geq a_{1}^{*}\,(b_1)_{*} + \ldots + a_{n}^{*}\,(b_n)_{*}\,,
\end{equation}
where $(a_{j}^{*})_{j=1}^n$ is the decreasing rearrangement of $(a_j)_{j=1}^n$ and $((b_j)_{*})_{j=1}^n$ is a~non-decreasing rearrangement of $(b_j)_{j=1}^n$.

\smallskip

\begin{proposition}
\label{cEE}
Let $X_n=(\mathbb{C}^n, \|\cdot\|)$ be a symmetric space with normalized standard basis. Then, for each
$\alpha \in \Lambda(m, n)$, there exists $u\in S_{X_n}^{+}:= \{u\in \mathbb{C}^{n}\,:\, \|u\|_{X_n}
= 1,\,\, u\in \mathbb{R}^{n}_{+}\,\}$ such that
\[
c_{X_n}(\alpha) = c_{X_n}(\alpha^{*}) = \frac{1}{(u^{*})^{\alpha^{*}}}\,.
\]
In particular, this yields that for each $m\leq n$,
\[
c_{X_n}(\alpha) \geq \varphi_{X_n}(1)^{\alpha^{*}_1} \cdots \varphi_{X_n}(m)^{\alpha^{*}_m}\,.
\]
\end{proposition}
\begin{proof}
Since $\bar{B}_{X_n}$ is compact, there exists $u = (u_1, \ldots, u_n) \in S_{X_n}^{+}$ such that
\[
c_{X_n}(\alpha) = \frac{1}{u_{1}^{\alpha_1} \cdots u_{n}^{\alpha_n}}\,.
\]
Note that $u_j=0$ for some $j\in [n] = \{i\colon 1 \leq i \leq n\}$ implies that $\alpha_j = 0$, and so without loss of generality
we may assume $u_j>0$ for each $j\in [n]$ (otherwise we will consider in the above formula only $\alpha \in \text{supp} \alpha:= \{j\in [n];\, \alpha_j \neq 0\}$). Since $\|\id\colon X_n \to \ell_{\infty}^n\|\leq 1$, it follows that for each $j\in [n]$, we have $u_j\leq 1$ and hence  $\log\frac{1}{u_j}\geq 0$. Thus applying the above inequality \eqref{HLP}, we get that (where the last equality is by
$\text{supp}(\alpha^{*}) \subset [m]$)
\begin{align*}
\log c_{X_n}(\alpha) & = \alpha_1 \log \frac{1}{u_1} + \ldots +  \alpha_{n} \log \frac{1}{u_n}
 \\
&
\geq \alpha^{*}_{1} \Big(\log \frac{1}{u_1}\Big)_{*} + \ldots +  \alpha^{*}_{n} \Big(\log \frac{1}{u_n}\Big)_{*}
 \\
&
=  \alpha^{*}_{1} \log \frac{1}{u^{*}_1} + \ldots +  \alpha^{*}_{n} \log \frac{1}{u^{*}_n}
= \log \frac{1}{(u^{*})^{\alpha^{*}}}
\,,
\end{align*}
and so
\[
c_{X_n}(\alpha) \geq \frac{1}{(u^{*})^{\alpha^{*}}}\,.
\]
Since $u^{*} \in S_{X_n}$, it follows that
\[
c_{X_n}(\alpha^{*}) \leq \frac{1}{(u^{*})^{\alpha^{*}}} \leq c_{X_n}(\alpha)\,.
\]
Conversely, we claim  that $c_{X_n}(\alpha) \leq c_{X_n}(\alpha^{*})$. Indeed, let $v\in S_{X_n}^{+}$  be such that
\[
c_{X_n}(\alpha^{*})= \frac{1}{v_{1}^{\alpha_{1}^{*}} \cdots  v_{n}^{\alpha_n^{*}}}\,.
\]
Now let $\sigma \colon [n] \to [n]$ be a permutation such that that
$v_{\sigma}^{\alpha} = v^{\alpha^{*}}$, where
 $v_{\sigma} := (v_{\sigma(1)}, \ldots,v_{\sigma(n)})$.
Since $E$ is symmetric,
$\|v_{\sigma}\|_{X_n} = \|v\|_{X_n} = 1$. Thus
\[
c_{X_n}(\alpha) \leq \frac{1}{v_{\sigma}^{\alpha}} = \frac{1}{v^{\alpha^{*}}} = c_{X_n}(\alpha^{*})
\]
shows the claim and proves the first statement.
To finish it is enough to observe that for any $x=(x_1, \ldots, x_n) \in X_n$ we have
\begin{equation*}
x_{k}^{*} \leq \frac{1}{\varphi_{X_n}(k)}\, \|x\|_{X_n}, \quad\, k\in [n]\,. \qedhere
\end{equation*}
\end{proof}

\smallskip

Under some mild assumptions on the weight $w$, the following result provides a concrete expression of the
characteristic $c_{m^n_w}(\alpha)$ of $\alpha \in \Lambda(m,n)$.
\begin{corollary}
\label{marc}
Let $m_w$ be a Marcinkiewcz space with a positive non-increasing weight $w=(w_k)~\in~\omega(\mathbb{N})$ such that $w_1 =1$
and for some $\gamma>1$, we have
\[
w_1 + \ldots + w_k \leq k\gamma w_k, \quad\, k\in \mathbb{N}\,.
\]
Then, for each $\alpha\in \Lambda(m, n)$ with $m\leq n$, we have
\[
\prod_{k=1}^m \varphi(k)^{\alpha_{k}^{*}}\,\,
\leq \,\,c_{m_{w}^n}(\alpha)= c_{m_{w}^n}(\alpha^{*}) \,\,\leq\,\,  \gamma^m\,\prod_{k=1}^m \varphi(k)^{\alpha_{k}^{*}}\,,
\]
where $\varphi(k):= \frac{k}{w_1 + \ldots + w_k}\,,\, \,k\in \mathbb{N}$ is the fundamental function
of $m_w$.
\end{corollary}

\begin{proof} Given $n\in \mathbb{N}$,  let $E:= m_{w}^n$. We have
\[
\varphi(k)= \frac{k}{w_1 + \ldots + w_k} \geq  \frac{1}{\gamma w_k}, \quad\, k\in [n]\,.
\]
Since $w_1=1$, we have that $\|e_k\|_E= 1$ for each $k\in [n]$. Now we are in position to apply Proposition~\ref{cEE}
to get that
\[
\prod_{k=1}^n \varphi(k)^{\alpha_{k}^{*}} \leq c_E(\alpha)= c_E(\alpha^{*})\,.
\]
Clearly, $\|(w_1, \ldots, w_n)\|_E = 1$ and whence
\[
c_E(\alpha^{*})\leq \prod_{k=1}^n \frac{1}{w_{k}^{\alpha_{k}^{*}}} \leq \gamma^{m}\,
\prod_{k=1}^n \varphi(k)^{\alpha_{k}^{*}}\,.
\]
This completes the proof.
\end{proof}

\noindent
In the next corollary we apply the preceding result in the special case of Marcinkiwicz spaces
$\ell^n_{r, \infty}$.

\smallskip

\begin{corollary}
\label{cMar}
Let  $1<r<\infty$. Then for each $\alpha\in \Lambda(m, n)$
with $m\leq n$, we have
\[
\prod_{k=1}^m k^{\frac{\alpha_{k}^{*}}{r}}
\leq c_{\ell_{r,\infty}^n}(\alpha) \leq  \Big(\frac{r}{r-1}\Big)^m \prod_{k=1}^m k^{\frac{\alpha_{k}^{*}}{r}}\,.
\]
\end{corollary}

\begin{proof}
We have $\ell_{r, \infty} \equiv m_{w}$, where $w:=(w_k)= \big(k^{1/r'} - (k-1)^{1/r'}\big)$ and
$1/r' = 1- 1/r$. Since $w_k = \frac{1}{r'} \int_{k-1}^k t^{-1/r}\,dt$, it follows that $k^{-1/r} \leq r' w_k$
for each $k\geq 1$. In consequence,
\[
w_1 + \ldots + w_k = k^{1/r'} \leq  \frac{r}{r-1} k w_{k}\,, \quad\, k\in \mathbb{N}\,.
\]
Thus Corollary \ref{marc} applies with $\gamma:=\frac{r}{r-1}$ and $\varphi(k)= k^{1/r}$ for each $k\in \mathbb{N}$.
\end{proof}

\smallskip

\subsubsection{Lorentz spaces}
Finally, we are in position to handle Lorentz spaces $d(w, 1)$.

\begin{corollary}
\label{cLor}
Let $d(w, 1)$ be the Lorentz space, where $w:= (k^{1/r} - (k-1)^{1/r})$. Then for each
$\alpha\in \Lambda(m, n)$ with $m\leq n$ one has
\[
\frac{1}{r^m} \Big(\frac{m^m}{\alpha^\alpha}\Big)\,\prod_{k=1}^m k^{-\frac{\alpha_{k}^{*}}{r'}}
\leq c_{d(w,1)^n}(\alpha) \leq \Big(\frac{m^m}{\alpha^\alpha}\Big) \prod_{k=1}^m k^{-\frac{\alpha_{k}^{*}}{r'}}\,.
\]
\end{corollary}

\begin{proof}
We let $E:=d(w, 1)$ and apply  the formula proved in Proposition~\ref{cE}, (i):
\[
c_{E_n}(\alpha)\,c_{(E')_n}(\alpha) = \frac{m^m}{\alpha^\alpha}\,.
\]
Combining the K\"othe duality formula $E' \equiv  \ell_{r', \infty}$ with Corollary \ref{cMar}, then yields the required
estimates.
\end{proof}

Note that the Lorentz space $d(w,1)$  shown in the above corollary, coincides up to equivalence of norm (depending
only on $r$) with the classical Lorentz space $\ell_{r, 1}$. Thus, we conclude that for $\alpha \in \Lambda(m,n)$
with $m\leq n$,
\begin{equation} \label{car-lor}
c_{\ell_{r,1}^n}(\alpha) \sim_{C^m} \Big(\frac{m^m}{\alpha^\alpha}\Big) \prod_{k=1}^m k^{-\frac{\alpha_{k}^{*}}{r'}}\,.
\end{equation}
Let us finish this section by stressing out that, while $c_{\ell_{r,1}^n}(\alpha)=c_{\ell_{r}^n}(\alpha)$ for any tetrahedral multi index $\alpha$ (see Proposition \ref{mietek}), 
for some $\alpha \in \Lambda(m,n)$ the characteristic $c_{\ell_{r,1}^n}(\alpha)$ differs substantially from $c_{\ell_{r}^n}(\alpha)$, as the following example shows.
This difference is one of the key reasons why some of the known techniques to bound the projection constant of the space $\mathcal P_m(\ell_r^n)$ are no longer useful to provide something tight about $\boldsymbol{\lambda}(\mathcal P_m(\ell_{r,s}^n))$. Using subtler techniques, we in Chapter \ref{Part: Polynomials on Lorentz sequences spaces} will obtain  bounds for $\boldsymbol{\lambda}(\mathcal P_m(\ell_{r,s}^n))$ which in some cases very much resemble to the $\ell_r^n$ case.

\smallskip

\begin{example}\label{ejemplo}
Let $1 < r < \infty$ and $k\in\mathbb N$ be fixed, and define
$\alpha=\big(\frac{k!}{1},\frac{k!}{2},\frac{k!}{3},\dots,\frac{k!}{k}\big) \in \Lambda(m,n)$, where
$m =|\alpha|= \dis\sum_{j=1}^k \frac{k!}{j}$. Then $$c_{\ell^n_{r,1}}(\alpha) \geq \left(\frac{r}{r-1} \right)^m (\log\log(m+1))^\frac{m}{r'}c_{\ell_{r}^n}(\alpha).$$
\end{example}

\begin{proof}
By \eqref{car-lor} we have
\begin{align*}
c_{\ell_{r,1}^n}(\alpha) & \sim_{C^m} \frac{m^m}{\alpha^\alpha}  \frac{1}{\left(1^{k!} \cdot 2^{k!/2} \cdots k^{k!/k}\right)^{1/r'}}, \\
\end{align*}
hence,
\begin{align*}
\left(\frac{\alpha^\alpha}{m^m}\right)^{1/r}c_{\ell_{r,1}}(\alpha) & \sim_{C^m} \left(\frac{m^m}{\alpha^\alpha}  \frac{1}{1^{k!} \cdot 2^{k!/2}
\cdots k^{k!/k}}\right)^{1/r'} \\
& = \left(\frac{m^m}{k!^{k!} \cdot (k!/2)^{k!/2} \cdots (k!/k)^{k!/k} } \cdot \frac{1}{1^{k!} \cdot 2^{k!/2} \cdots k^{k!/k}}\right)^{1/r'} \\
& = \left(\frac{m^m}{k!^{k! \sum_{j=1}^k 1/j}  }\right)^{1/r'}
= \left( \frac{m}{k!} \right)^{m/r'} \succ_{C^m} \left( \log(k+1) \right)^{m/r'} \succ_{C^m} \left(\log \log(m+1)\right)^{m/r'}.  \\
\end{align*}
Then by  \eqref{dineen} this concludes the proof.
\end{proof}

\smallskip

\section{Polynomial projection constants vs  fundamental functions}
\label{Polynomial projection constants vs  fundamental functions}
Our main motivation for this section is to show that for some class of norms on $\mathbb{C}^n$ the introduced polynomial projection constant
is  intimately connected with the notion of fundamental function from the setting of Banach sequence lattices. Indeed, we
establish
variants of  \eqref{schuetty} and of Proposition~\ref{lambda}.

\smallskip

\begin{proposition}
\label{lambdaII}
Let $X_n = (\mathbb{C}^n,\|\cdot\|)$  be a Banach lattice and $J \subset \Lambda(m,n)$. Then
\[
\widehat{\boldsymbol{\lambda}}\big(\mathcal{P}_{J}(X_n)\big) \leq e^{m} \varphi_{X_n'}(n)^{m}\,.
\]
\end{proposition}

\begin{proof}
It suffices to show the result for $J = \Lambda(m, n)$.  Recall once again from \eqref{multinomial} that
\[
\sup_{z\in B_{X_n}} \sum_{\alpha\in \Lambda(m, n)} \frac{m!}{\alpha!}\,|z^\alpha| = \varphi_{X_n'}(n)^m\,.
\]
Since for each $\alpha \in \Lambda(m, n)$,
\begin{align*}
\Delta_{X_n}(\alpha) & = \Big\|\Big(\frac{\alpha_1}{m}, \ldots, \frac{\alpha_n}{m}\Big)\Big\|_{X_n}^{m}
\leq \Big(\frac{1}{m}\big(\alpha_1 \|e_1\|_X + \ldots + \alpha_n \|e_n\|_X\big) \Big)^{m} \leq 1\,,
\end{align*}
Lemma~\ref{lambdaP_J} yields
\begin{equation*}
\widehat{\boldsymbol{\lambda}}\big(\mathcal{P}_m(X_n)\big) \leq e^{m} \sup_{z\in B_{X_n}} \sum_{\alpha\in \Lambda(m, n)} \frac{m!}{\alpha!}\,|z^\alpha|
=  e^m \varphi_{X_n'}(n)^m\,. \qedhere
\end{equation*}
\end{proof}

\smallskip
If $X_n$ is the $n$th section of a symmetric Banach sequence lattice and
  $J$ contains the tetrahedral $m$-homogeneous indices, then the preceding proposition is asymptotically
correct in $n$ whenever $m$ is fixed.
More generally, we consider  Banach sequence lattices $X$  for which
\[
n \sim \varphi_{X}(n) \varphi_{X'}(n)
\]
(for $X$ symmetric this holds with = instead of $\sim$). For this class of  $X$, which is not necessarily symmetric, we  have the  following variant of Theorem~\ref{tensor}; for the trivial
case $m=1$ recall Remark~\ref{m1}.

\begin{proposition} \label{equivalence}
Let $X$ be a Banach sequence lattice such that
$
n \sim \varphi_{X_n}(n) \varphi_{X_{n}'}(n)\,.
$
Then there is $C>0$, only depending on $X$, such that for each $m\leq n$ and any index set $J \subset \mathbb{N}_0^n$ with  $\Lambda_T(m,n) \subset J$ we have
\[
\frac{1}{C^{m}m^m} \varphi_{X'}(n)^m \,\,\leq \,\,\widehat{\boldsymbol{\lambda}}\big(\mathcal{P}_J(X_n)\big)\,.
\]
\end{proposition}

\begin{proof}
Without loss of generality we assume that $J = \Lambda_T(m,n)$.
Then we know from Corollary~\ref{duality} that
\begin{equation*} \label{star}
\Big(\frac{n}{m}\Big)^m \prec_{C^m} \widehat{\boldsymbol{\lambda}}\big(\mathcal{P}_{\Lambda_T(m,n)}(X_n)\big) \widehat{\boldsymbol{\lambda}}\big(\mathcal{P}_{\Lambda_T(m,n)}(X_{n}')\big).
\end{equation*}
Moreover by  Proposition~\ref{lambdaII} (for $X_{n}'$ instead of $X_n$) we have
\[
\widehat{\boldsymbol{\lambda}}(\mathcal{P}_{\Lambda_T(m,n)}(X_{n}')) \leq e^m \varphi_{X_{n}''}(n)^m = e^m \varphi_{X_n}(n)^m.
\]
The the conclusion follows, combining both estimates and using again Proposition~\ref{lambda} together with the assumption.
\end{proof}

\smallskip

Under symmetry assumptions we may unify Corollary~\ref{immediate} and Proposition~\ref{equivalence}.

\begin{proposition} \label{mistery}
Let $X_n= (\mathbb{C}^n, \|\cdot\|)$ be a~Banach lattice with enough symmetries. Then for each $m$ and $n$ with $m \leq n$
\[
\frac{1}{e^{m} m^m} \,\,\boldsymbol{\lambda}(X_{n}')^m \leq \,\, \boldsymbol{\lambda}\big(\mathcal{P}_m(X_n)
\big) \leq  \widehat{\boldsymbol{\lambda}}
\big(\mathcal{P}_m(X_n)\big) \leq e^m\,\varphi_{X'}(n)^m\,.
\]
If additionally $\|\id\colon X_n \hookrightarrow  \ell^n_2 \| \leq 1$, then
\[
\frac{1}{e^{m}\sqrt{2}^m m^m} \,\,\varphi_{X_{n}'}(n)^m \,\,\leq\,\, \boldsymbol{\lambda}\big(\mathcal{P}_m(X_n)
\big)
\,\,\leq\,\, \widehat{\boldsymbol{\lambda}}\big(\mathcal{P}_m(X_n)\big) \,\,\leq \,\,e^m\,\varphi_{X_{n}'}(n)^m\,.
\]
\end{proposition}

\begin{proof}
For the first estimate see Corollary~\ref{immediate}, the second one follows from Theorem~\ref{lambda-dash},
and the third one from Proposition~\ref{lambdaII}. The last asymptotic then is a consequence of \eqref{schuetty}.
\end{proof}

\smallskip

\section{Comparison}
\label{compa}
Our aim here is to compare $\widehat{\boldsymbol{\lambda}}\big(\mathcal{P}_J(X_n)\big)$ with  $\widehat{\boldsymbol{\lambda}}\big(\mathcal{P}_J(Y_n)\big)$,
whenever $J\subset \mathbb{N}_0^n$, and $X_n := (\mathbb{C}^n, \|\cdot\|_{X_n})$ and $Y_n:=(\mathbb{C}^n, \|\cdot\|_{Y_n})$
are both (quasi-)Banach lattices. Much of the material is based on what we explained in Section~\ref{conv/conc}.

\smallskip
The main result here is as follows.

\begin{theorem} \label{min}
Let $X_n := (\mathbb{C}^n, \|\cdot\|)$ and $Y_n:=(\mathbb{C}^n, \|\cdot\|)$ be (quasi-)Banach lattices.
Then, for every finite index set  $J\subset \mathbb{N}_0^n$ the following estimate holds{\mbox:}
\[
\widehat{\boldsymbol{\lambda}}\big(\mathcal{P}_J(X_n\circ Y_n)\big) \leq \min \big\{\widehat{\boldsymbol{\lambda}}\big(\mathcal{P}_J(X_n)),\,
\widehat{\boldsymbol{\lambda}}\big(\mathcal{P}_J(Y_n)\big)\big\}\,.
\]
\end{theorem}

\begin{proof}
Combining  Lemma \ref{cproduct} and a compactness argument, we conclude that there exists $z\in \bar{B}_{X_n \circ Y_n}$ such that
\[
\widehat{\boldsymbol{\lambda}}(\mathcal{P}_J(X_n\circ Y_n)) = \sum_{\alpha \in J} c_{X_n}(\alpha)\,c_{Y_n}(\alpha)\,|z^\alpha|\,.
\]
Given $\varepsilon >0$, we can find $u\in \bar{B}_{X_n}$ and $v\in \bar{B}_{Y_n}$ such that $|z| \leq (1 + \varepsilon) |u|\,|v|$.
Clearly, for each $\alpha\in J$, we have $c_{Y_n}(\alpha)\,|v^\alpha| \leq 1$. Hence we get with
$m:= \max_{\alpha \in J} |\alpha|$,
\begin{align*}
\widehat{\boldsymbol{\lambda}}\big(\mathcal{P}_J(X_n\circ Y_n)\big)  \leq (1 + \varepsilon)^m \sum_{\alpha \in J} (c_{X_n}(\alpha)\,|u^\alpha|)
\,(c_{Y_n}(\alpha)\,|v^\alpha|)
\leq (1 + \varepsilon)^{m}\,\widehat{\boldsymbol{\lambda}}\big(\mathcal{P}_J(X_n)\big)\,.
\end{align*}
In consequence,
\[
\widehat{\boldsymbol{\lambda}}\big(\mathcal{P}_J(X_n\circ Y_n)\big) \leq (1 + \varepsilon)^m\, \text{min} \big\{\widehat{\boldsymbol{\lambda}}\big(\mathcal{P}_J(X_n)\big),\,
\widehat{\boldsymbol{\lambda}}\big(\mathcal{P}_J(Y)\big)\big\}\,.
\]
Letting $\varepsilon \to 0$, we arrive to the desired inequality.
\end{proof}

\smallskip

Obviously, we have that $X_n \equiv X_n \circ \ell_\infty^n$ for every (quasi-)Banach lattice $X_n$. This implies the following
consequence of the preceding theorem.

\begin{corollary}
Let $X_n := (\mathbb{C}^n, \|\cdot\|)$ be a~$($quasi-$)$ Banach lattices. Then, for every finite index set $J\subset \mathbb{N}_0^n$ we have
\[
\widehat{\boldsymbol{\lambda}}\big(\mathcal{P}_J(X_n)\big) \leq \widehat{\boldsymbol{\lambda}}\big(\mathcal{P}_J(\ell_\infty^n)\big)\,.
\]
\end{corollary}

Now observe that by Lozanovskii's factorization Theorem~\ref{Lozanovskii} we know that $\ell_1^n \equiv X_n \circ X_{n}'$, whenever $X_n := (\mathbb{C}^n, \|\cdot\|)$ is a~Banach lattice. This gives the following dual version of the preceding corollary.

\begin{corollary}
Let $X_n := (\mathbb{C}^n, \|\cdot\|)$  be a Banach lattice. Then, for every finite index set $J\subset \mathbb{N}_0^n$ we have
\[
\widehat{\boldsymbol{\lambda}}\big(\mathcal{P}_J(\ell_1^n)\big) \leq \widehat{\boldsymbol{\lambda}}\big(\mathcal{P}_J(X_n)\big)\,.
\]
\end{corollary}

\smallskip

The following results are now straight forward  consequences of the work done in Section~\ref{conv/conc}. The first
and most important one is an immediate consequence of Lemma~\ref{app5} and Theorem~\ref{min}.

\smallskip

\begin{theorem}
\label{PropApplnew}
Let $X_n = (\mathbb{C}^n, \|\cdot\|)$ and $Y_n = (\mathbb{C}^n, \|\cdot\|)$ be two Banach lattices
such that $M_{(r)}(X_n ) = M^{(r)}(Y_n) = 1$, where $1 <r <\infty$. Then, for every finite $J \subset \mathbb{N}_0^n$, we have
\begin{equation*} 
\widehat{\boldsymbol{\lambda}}\big(\mathcal{P}_{J}(X_n)\big) \leq \widehat{\boldsymbol{\lambda}}\big(\mathcal{P}_{J}(Y_n)\big)\,.
\end{equation*}
\end{theorem}

\smallskip

As in Section~\ref{conv/conc} we obtain the following  interesting corollary - analogs of Corollary~\ref{C1} and
Corollary~\ref{C2}.

\smallskip
\begin{corollary}
Let $X_n = (\mathbb{C}^n, \|\cdot\|)$ be a  Banach lattice, $1 <r <\infty$ and $J \subset \mathbb{N}_0^n$ a finite index set.
If $M_{(r)}(X_n ) =  1$ (resp. $M^{(r)}(X_n ) =  1$), then
$$
\widehat{\boldsymbol{\lambda}}\big(\mathcal{P}_{J}(X_n)\big) \leq \widehat{\boldsymbol{\lambda}}\big(\mathcal{P}_{J}(\ell_r^n)\big)
\;\;\;\;\;\;
\big(\mbox{resp. \;} \widehat{\boldsymbol{\lambda}}\big(\mathcal{P}_{J}(\ell^n_r)\big) \leq \widehat{\boldsymbol{\lambda}}\big(\mathcal{P}_{J}(X_n)\big)\,\big).
$$
\end{corollary}

\smallskip

\smallskip
Similarly to Corollary~\ref{appl}, we also have the following for the homogeneous case.
\begin{corollary}
Let $X$ and $Y$ be Banach sequence lattices such that $X$ is $r$-concave and
$Y$ is $r$-convex  for some $1 <r <\infty$. Then for any index set $J \subset \Lambda(m,n)$, we have
\[
\widehat{\boldsymbol{\lambda}}\big(\mathcal{P}_{J}(X_n)\big) \leq \big(M_{(r)}(X) M^{(r)}(Y)\big)^m\,\,\widehat{\boldsymbol{\lambda}}\big(\mathcal{P}_{J}(Y_n)\big)\,.
\]
\end{corollary}

\medskip

\section{Unconditionalizing the norm}

Finally, we show that for a wide class of Banach lattices $X_n = (\mathbb{C}^n,\|\cdot~\|)$ it is possible to 'unconditionalize'
the norm in $\mathcal{P}_m(X_n)$ such that for the new space $\mathcal{P}^{uc}_m(X_n)$ the classical projection
constant $\boldsymbol{\lambda}\big(\mathcal{P}^{uc}_m(X_n)\big)$ and the polynomial projection constant
$\widehat{\boldsymbol{\lambda}}\big(\mathcal{P}_m(X_n)\big)$ are equivalent up to  constants just depending on $m$.

\smallskip

Fix some finite index set $J \subset \mathbb{N}_0^n$ and a~Banach space $X_n = (\mathbb{C}^n,\|\cdot\|)$. We denote by
$\mathcal{P}^{uc}_J(X_n)$ the Banach space of all polynomials $P(z) = \sum_{\alpha \in J} c_\alpha z^\alpha\,,\,\, z \in X_n$
endowed with the norm
\[
\|P\|_{\text{uc}}:= \sup_{(\xi_{\alpha})_{\alpha \in J} \subset \mathbb{T}}
\Big\|\sum_{\alpha \in J} \xi_\alpha c_\alpha e_\alpha \Big\|_{\mathcal{P}_J(X_n)}
=
\sup_{z \in B_{X_n}}\,\sum_{\alpha \in J}   |c_\alpha| |z^\alpha|\,;
\]
here  again  $e_\alpha(z) := z^\alpha,\,\,z \in X_n$. Clearly, the family $(\tilde{e}_\alpha)$ given by
\[
\tilde{e}_\alpha := c_{X_n}(\alpha) e_\alpha\,,\,\, \,\,\alpha \in J
\]
forms a~normalized $1$-unconditional basis of the space $\mathcal{P}^{uc}_J(X_n)$.

The following result relates  the projection  constant of $\mathcal{P}^{uc}_J(X_n)$ and the fundamental function of
$\mathcal{P}^{uc}_J(X_n)$ with  the polynomial projection constant of $\mathcal{P}_J(X_n)$. 

\begin{proposition}
\label{lPuc}
Let $J \subset \mathbb{N}_0^n$ be a finite index set and $X_n = (\mathbb{C}^n,\|\cdot\|)$ a~Banach space. Then
\begin{align*}
\widehat{\boldsymbol{\lambda}}\big(\mathcal{P}_J(X_n)\big)  \,\,=\,\,
\varphi_{\mathcal{P}^{uc}_J(X_n)}(|J|)\,,
\end{align*}
and
\begin{align*}
\boldsymbol{\lambda}\big(\mathcal{P}^{uc}_J(X_n)\big)
\,\,\leq \,\, d\big(\mathcal{P}^{uc}_J(X_n),  \ell_\infty(J)\big)
\,\,\leq \,\,
\varphi_{\mathcal{P}^{uc}_J(X_n)}(|J|)
 \,\,\leq \,\,
\sqrt{2}\,\gamma(X_n)\,  \boldsymbol{\lambda}\big(\mathcal{P}^{uc}_J(X_n)\big)\,,
\end{align*}
where \,$\gamma(X_n):= \sup_{z\in B_{X_n}} \Big( \sum_{\alpha \in J} |c_{X_n}(\alpha) z^\alpha|^2\Big)^{\frac{1}{2}}$.
\end{proposition}

Note that, by the 1-unconditionality and the fact that $J$ is finite, we can take any order of the basis of the space $\mathcal{P}^{uc}_J(X_n)$, since we are only dealing with  $\varphi_{\mathcal{P}^{uc}_J(X_n)}(|J|)$.

\begin{proof}

Since the $\tilde{e}_\alpha$ form a normalized $1$-unconditional basis in $\mathcal{P}^{uc}_J(X_n)$, we have
\begin{align*}
  \widehat{\boldsymbol{\lambda}}\big(\mathcal{P}_J(X_n)\big)
  =
\sup_{z \in B_{X_n}}\,\sum_{\alpha \in J} c_{X_n}(\alpha)|z^\alpha|
=
\sup_{(\xi_{\alpha})_{\alpha \in J} \subset \mathbb T}
\Big\|\sum_{\alpha \in J} \xi_\alpha \tilde{e}_\alpha  \Big\|_{\mathcal{P}_J(X_n)}
=
 \varphi_{\mathcal{P}^{uc}_J(X_n)}(|J|) \,.
\end{align*}
Moreover, we  conclude  from  \eqref{Carsten1} that
\begin{align*}
\boldsymbol{\lambda}\big(\mathcal{P}^{uc}_J(X_n)\big)
  \leq \boldsymbol{\lambda}\big(\mathcal{P}^{uc}_J(X_n), & \ell_\infty(J)\big)
\leq
\varphi_{\mathcal{P}^{uc}_J(X_n)}(|J|) \,.
\end{align*}
Finally, consider the  linear map
\[
I:  \ell_2(J) \to \mathcal{P}^{uc}_J(X_n) \,, \,\, \,\, \,\,
(c_\alpha)  \mapsto \sum_{\alpha \in J} c_\alpha \tilde{e}_\alpha \,,
\]
and observe  that
$
\|I\| \leq \gamma(X_n)\,,
$
since by the Cauchy-Schwarz inequality  for every $(c_\alpha) \in \ell_2(J)$
\[
\|I(c_\alpha)\|_{\text{uc}} =
\Big\|\sum_{\alpha \in J} c_\alpha \tilde{e}_\alpha\Big\|_{\text{uc}}
= \sup_{z \in B_{X_n}}\,\sum_{\alpha \in J} |c_\alpha | |c_{X_n}(\alpha)z^\alpha|
\leq \|(c_\alpha)\|_{\ell_2(J)} \,\gamma(X_n)\,.
\]
Then by Sch\"utt's result from \eqref{schuett}
\begin{align*}
   \varphi_{\mathcal{P}^{uc}_J(X_n)}(|J|) \leq \sqrt{2}\,\, \gamma(X_n)\,\,
\boldsymbol{\lambda}\big(\mathcal{P}^{uc}_J(X_n)\big)\,,
\end{align*}
which completes the proof\,.
\end{proof}

We will show that in the homogeneous  case $J=\Lambda(m, n)$ for a wide class of Banach lattices
$X_n = (\mathbb{C}^n, \|\cdot\|)$  the constant
$\gamma(X_n) $ does not depend on $n$.

\begin{corollary} \label{unccor}
Let $X_n= (\mathbb{C}^n, \|\cdot\|)$ be a~Banach lattice such that $X_n\equiv \ell_q^{n}\circ Y_n$
for some $1\leq q\leq 2$ and some Banach lattice $Y_n= (\mathbb{C}^n, \|\cdot\|)$. Then
\[
\gamma(X_n)  \,\,\leq  \,\,e^{\frac{m}{q}}\,.
\]
In particular,
for each $m\in \mathbb{N}$
\[
\boldsymbol{\lambda}\big(\mathcal{P}^{uc}_m(X_n)\big)  \,\,\leq \,\,\widehat{\boldsymbol{\lambda}}\big(\mathcal{P}_m(X_n)\big) \,\,\leq \,\,
\sqrt{2}\,e^{\frac{m}{q}}\, \boldsymbol{\lambda}\big(\mathcal{P}^{uc}_m(X_n)\big)\,.
\]
\end{corollary}

\begin{proof}
Recall that for each $\alpha \in \Lambda(m, n)$, we have 
\[
c_{\ell_q^n}(\alpha) = \big(\frac{m^m}{\alpha^\alpha}\big)^{1/q}
\leq e^{\frac{m}{q}} \big(\frac{m!}{\alpha!}\big)^{1/q}
\]
(see again \eqref{dineen}). Then by Lemma~\ref{cproduct}, we get
\begin{align*}
\gamma(X_n)
&
\leq \Big(\sup_{z\in B_{X_n}} \sum_{\alpha \in \Lambda(m, n)} |c_{X_n}(\alpha) z^\alpha|^q\Big)^{\frac{1}{q}}
 \leq \Big(\sup_{z\in B_{\ell_q^n}}\sup_{y\in Y_n} \sum_{\alpha \in \Lambda(m, n)} |c_{\ell_q^n}(\alpha)c_{Y_n}(\alpha)y^\alpha z^\alpha|^q\Big)^{\frac{1}{q}}
 \\
&
\leq e^{\frac{m}{q}} \Big(\sup_{z\in B_{\ell_q^n}} \sum_{\alpha \in \Lambda(m, n)} \frac{m!}{\alpha!} |z^{q\alpha}|\Big)^{\frac{1}{q}} = e^{\frac{m}{q}} \sup_{z\in B_{\ell_q^n}} (|z_1|^q + \ldots + |z_n|^q)^{\frac{m}{q}}= e^{\frac{m}{q}}\,.
\qedhere
\end{align*}
\end{proof}

\smallskip

Finally, we apply this result to Lorentz spaces. We  use a well-known  multiplication formula for such spaces, 
which states that for all $p_0$, $p_1$, $s_0$, $s_1\in (0, \infty)$ and $r$, $s$ given by $1/r_0 + 1/r_1 = 1/r$ 
and $1/s_0+ 1/s_1 = 1/s$ one has (see, e.g., \cite[Proposition 2.1.13]{pietsch1986eigenvalues}):
\[
\ell_{r, s}^n= \ell_{r_0, s_0}^{n}\circ \ell_{r_1, s_1}^{n}, \quad\, n\in \mathbb{N}\,.
\]
Let us note the obvious but essential fact that equivalence constants of (quasi-) 
norms in the above formula do not depend on $n$. 

\smallskip

\begin{corollary}
\label{special}
For every $1<r<s <q\leq 2$ there is a~constant $C>0$ such that, for each $m, n\in \mathbb{N}$, we have
\[
\boldsymbol{\lambda}\big(\mathcal{P}^{uc}_m(\ell_{r, s}^n)\big)  \,\,\leq  \,\,C^m e^{\frac{m}{q}} \,
\widehat{\boldsymbol{\lambda}}\big(\mathcal{P}_m(\ell_{r, s}^n)\big)  \,\,\leq  \,\,\sqrt{2}\,C^m \,e^{\frac{m}{q}}\,
\boldsymbol{\lambda}\big(\mathcal{P}^{uc}_m((\ell_{r, s}^n)\big)\,.
\]
\end{corollary}

\chapter{Polynomials on Lorentz and mixed spaces} 
\label{Part: Polynomials on Lorentz sequences spaces}

In this chapter we mainly focus on  polynomials defined on $n$-dimensional Lorentz sequence spaces $\ell_{r,s}^n $, and in an appendix on similar results for polynomials on mixed spaces $\ell_r^n (\ell_s^k)$ (with a lower claim to completeness). Our aim is to obtain estimates  of  the projection constant and unconditional basis constant of  spaces of polynomials on these spaces, which depend both, on the dimension of the spaces and on the degree of the polynomials.

In our study, we  use  well-known results on $p$-convexity and $q$-concavity of these spaces. For the sake completeness, we hence include the following result for Lorentz spaces, which contains  also  information about  upper and lower lattice estimates for these spaces. As a consequence of the  relationships  between the discussed geometrical concepts (see Section~\ref{Banach spaces and (Quasi-)Banach lattices} and \cite[Theorem 3.4]{creekmore1981type}) we have:

\noindent
(a)\, \emph{If \,$1\leq q<p$, then} 
\par {\rm(i)} \emph{$\ell_{p,q}$ is not $p$-concave, but satisfies a lower $p$-estimate and so it is $r$-concave for all $r>p$}.
\par {\rm(ii)} \emph{$\ell_{p, q}$ is $q$-convex}.

\noindent
(b)\, \emph{If \,$1<p<q$, then}
\par {\rm(i)}  \emph{$\ell_{p,q}$ is not $p$-convex, but satisfies an upper $p$-estimate and so it is $r$-convex for all $r<p$}.
\par {\rm(ii)} \emph{$\ell_{p,q}$ is $q$-concave}.

Of special interest for our purposes  is that $\ell_{p, q}$ is $2$-convex whenever $2 <p < \infty$ \,and \, $2 \leq q < \infty$, and  $\ell_{p, q}$ is $2$-concave whenever 
$1 \leq p < 2$ \,and \,$1 \leq q \leq 2$.

We refer to \cite{Jameson}, where explicit formulae for the the $q$-concavity constant of the Lorentz space $d(w, p)$ as well as $d^n(w, p)$ are given in terms of the sequence $w$.

Recall that, for a~given positive integer $m$, we in Section \ref{comparing}  related  the asymptotic growth of  $\boldsymbol{\lambda}\big(\mathcal{P}_{J_{\leq m}}(X_n)\big)$  in the dimension $n$, where  $X_n = (\mathbb{C}^n, \|\cdot\|) $ is the $n$th section of  a Banach sequence lattice  and $J \subset \mathbb{N}_0^{(\mathbb{N})}$, with the asymptotic growth  
of  the $m$th power of $\boldsymbol{\lambda}(X^\ast_n)$ in the dimension $n$.

Since we here concentrate on the  Banach sequence spaces $\ell_{r,s}$\,, it is in a first step  of course crucial to control the case $m=1$, that is,  the projection constant of the scale of finite dimensional  Lorentz spaces itself. The following result collects our knowledge.

\begin{proposition} \label{lor}
For all $1 < r < \infty$ \,and \, $1 \leq s \leq \infty$ one has 
\[
\boldsymbol{\lambda}(\ell^n_{r,s})\sim
\begin{cases}   
n^{\min\big\{\frac{1}{2}, \frac{1}{r} \big\}},  & r \neq 2\\[2mm]
n^{\frac{1}{2}},&  r = 2, \, \, \, 2 \leq s \leq \infty \\[2mm]
\Big({\frac{n}{\log \,(e + \log n)}}\Big)^{\frac12},&  r = 2, \, \,\,s=1.
\end{cases}
\]
\end{proposition}

\begin{proof}
As explained above,  $\ell_{r,s}$ 
is $2$-concave for $r < 2$ and $1 \leq s \leq 2$. Thus the statement follows from
\eqref{conv-conc}. For the proof of the case  $r < 2$ and $2 \leq s \leq \infty$ see Theorem~\ref{t-final}. The case $r=2$, $s\ge2$ and the case $r>2$ follow from the equations \eqref{Carsten1} and \eqref{schuett}. The last equivalence is  a~remarkable result due to Kwapie\'n and Sch\"{u}tt \cite{schuettkwapien}, which states that 
$
\boldsymbol{\lambda}(\ell^n_{2,1})\sim
\big(n/\log\,\log n\big)^{\frac12}, \, n\geq 5\,. 
$
\end{proof}

As far as we know, for the case $r = 2$ and $1 <s <2$ (not present in Proposition \ref{lor}) there is no known asymptotic 
of $\boldsymbol{\lambda}(\ell_{2,s}^n)$ as a~function of the dimension $n$. By combining standard estimates that take into account the Banach-Mazur distances, we get the following.

\begin{lemma} For every $1<s<2$
\[
\boldsymbol{\lambda}(\ell^n_{2,s})\succ
\begin{cases} \
\frac{1}{(1 + \log n)^{1 - \frac{1}{s}}}\,
\Big({\frac{n}{\log \,(e + \log n)}}\Big)^{\frac12},& 1<s <\frac{4}{3} \\[2mm]
\frac{\sqrt{n}}{(1+\log n)^{\frac{1}{s}-\frac{1}{2}}}, & \frac{4}{3} \leq s <2\,. 
\end{cases}
\]
\end{lemma}

\begin{proof}
Recall that for any pair of  isomorphic Banach spaces $E$ and $F$ one has
\[
\boldsymbol{\lambda}(E) \leq d(E, F)\, \boldsymbol{\lambda}(F)\,.
\]
Thus, for each $n\in \mathbb{N}$, we have the estimates
\[
\boldsymbol{\lambda}({\ell_2^n)} \leq d(\ell_2^n,\, \ell_{2, s}^n)\,\boldsymbol{\lambda}(\ell_{2, s}^n), \quad \,\,\,\,    
\boldsymbol{\lambda}{(\ell_{2, 1}^n)} \leq d(\ell_{2, 1}^n, \, \ell_{2, s}^n)\, \boldsymbol{\lambda}(\ell_{2, s}^n)\,,
\]
and hence
\[
\boldsymbol{\lambda}(\ell^n_{2,s}) \geq 
\max
\Bigg\{
\frac{\boldsymbol{\lambda}(\ell_2^n)}{d(\ell_2^n,\, \ell_{2, s}^n)}, \,\,\, \frac{\boldsymbol{\lambda}(\ell_{2, 1}^n)}{d(\ell_{2, 1}^n, \, \ell_{2, s}^n)}
\Bigg\}\,.
\]

\smallskip

\noindent
It is easy to check that there is a~positive constant $C=C(s)>0$ such that 
\begin{equation*} 
\big\|{\id} \colon \ell_{2}^n \to \ell_{2,s}^n\big\| \leq  C\, (1 + \log n)^{\frac{1}{s}-\frac{1}{2}}, \quad\,\,\,\, \big\|{\id} \colon \ell_{2,s}^n \to \ell_{2, 1}^n\big\| \leq  C\, (1 + \log n)^{1-\frac{1}{s}}\,.
\end{equation*}
Since $\sup_{n\geq 1} \big\|{\id}\colon \ell_{2, s}^n \to \ell_{2}^n\big\|<\infty$ \, and \,   
$\sup_{n\geq 1} \big\|{\id}\colon \ell_{2, 1}^n \to \ell_{2, s}^n\big\|<\infty$, it follows that 
\[
d(\ell_2^n,\,\ell_{2, s}^n) \prec (1 + \log n)^{\frac{1}{s}-\frac{1}{2}}, \quad\, 
d(\ell_{2, 1}^n,\,\ell_{2, s}^n )\prec (1 + \log n)^{1-\frac{1}{s}}\,.
\]
The above inequalities combined with $\boldsymbol{\lambda} {(\ell_2^n)} \sim \sqrt{n}$ \, and
the asymptotic for $\boldsymbol{\lambda}{(\ell_{2, 1}^n)}$ shown in Proposition \ref{lor} give the required estimates.
\end{proof}

\smallskip
Combining  Corollary~\ref{immediateABC} and  Proposition~\ref{lor}, immediately leads to the following result.

\begin{proposition} \label{lor2}
For $1 < r < \infty$ \, and \, $1 \leq s \leq \infty$
\[
\boldsymbol{\lambda}\big(\mathcal{P}_{m}(\ell^n_{r,s})\big)
\sim_{C(m)}
\begin{cases}   \
n^{ m\min\big\{\frac{m}{2}, \frac{1}{r'} \big\}},  & r \neq 2\\[2mm]
n^{\frac{m}{2}},&  2 = r,\, \,1 \leq s \leq 2 \\[2mm]
\Big(\frac{n}{\log \,(e + \log n)}\Big)^{\frac{m}{2}},&  2 = r,\, \,\,s=\infty\,,  \\[2mm]
\end{cases}
\]
where $C(m)>0$ is a constant which only depends on $m$. Moreover, the same asymptotic holds true for
$\boldsymbol{\lambda}(\mathcal{P}_{\leq  m}(\ell^n_{r,s}))$, $\boldsymbol{\lambda}(\mathcal{P}_{\Lambda_T( m)}(\ell^n_{r,s}))$,
and $\boldsymbol{\lambda}(\mathcal{P}_{\Lambda_T( \leq m)}(\ell^n_{r,s}))$.
\end{proposition}

Recall from Theorem~\ref{start-poly2} that, given   $1 \leq r \leq \infty$ and an index set $J \subset \mathbb{N}_0^{(\mathbb{N})}$, we have  for all $m$ for which   $\Lambda_T(m) \subset J_{\leq m}$ that
\[
\boldsymbol{\lambda}\big(\mathcal{P}_{J_{\leq m}}(\ell_r^n) \big)
\,\,\sim_{C^m} \,\,
\Big( 1+\frac{n}{m}\Big)^{m \min \big\{\frac{1}{2}, \frac{1}{r'}\big\}}
\quad 
and
\quad 
\boldsymbol{\chimon}\big(\mathcal{P}_{J_{\leq m}}(\ell_r^n) \big)
\,\,\sim_{C^m} \,\,
\Big( 1+\frac{n}{m}\Big)^{(m-1) \min \big\{\frac{1}{2}, \frac{1}{r'}\big\}}\,,
\]
where $C >0$ is  only depending on $r$. We see that this result in  case $r=s$ shows a much more precise description
of $\boldsymbol{\lambda}(\mathcal{P}_{J_{\leq m}}(\ell^n_{r,s}))$
than Proposition~\ref{lor2}. In fact, it describes 
$\boldsymbol{\lambda}(\mathcal{P}_{J_{\leq m}}(\ell^n_{r,s}))$
 as a function in the two variables $n$ and in $m$,
 and it moreover allows a much larger variety of index sets $J$.
See also  Theorem~\ref{conny3} or  Corollary~\ref{RW}, where even  a precise formula in the Hilbert space case $r=s=2$ is proved.

 The task, which asks for  similar results for the full scale of finite dimensional Lorentz spaces, is laborious.
We start concentrating  on the tetrahedral case, which is easier than the general case, since  the characteristics of  tetrahedral multi indices with respect to Lorentz spaces $\ell_{r,s}^n$ behave exactly as those  of arbitrary multi indices  in the $\ell_r^n$-case (Lemma \ref{mietek}). For  tetrahedral indices  we in fact control  all cases  with the exception of  $r\le2$, $r<s$.
In a second attempt  we then deal with the (more involved) case of  polynomials supported on arbitrary index sets of degree at most $m$.

\bigskip
\section{The tetrahedral  case}
  The understanding of  unconditional basis constants and projection constants
for spaces $\mathcal{P}_{J}(\ell_{r,s}^n)$, where $J$ is a tetrahedral index set,  is in fact crucial in order to reach similar  results for more general index sets -- in particular,  index sets like $\Lambda(m,n)$  or $\Lambda(\leq m,n)$. 

\smallskip

We start with the upper estimates in the tetrahedral case, which are mainly based on our results from
Section~\ref{Kadets-Snobar case} (in particular, the 'polynomial' consequences of the  Kadets-Snobar theorem and the subexponential Bohnenblust-Hille inequality) as well as  Section~\ref{Characteristics} and Section~\ref{Tetrahedral polynomials} (dealing with the polynomial projection constant for spaces of polynomials supported on tetrahedral  indices).

\begin{theorem} \label{t-final}
Let $1 < r < \infty$,  $1 \leq s \leq \infty$,
and  let $J \subset \mathbb{N}_0^{(\mathbb{N})}$
be a~tetrahedral index set.
Then
\begin{align*} \label{putin}
 \boldsymbol{\chimon}\big(\mathcal{P}_{J_{\leq m}}(\ell_{r,s}^n)\big)  \prec_{C^m}\Big( \frac{n}{m}\Big)^{(m-1)\min\big\{\frac{1}{2},\frac{1}{r'}\big\}}
\,\,\,\, \text{ and } \,\,\,\,\,\,\,
\boldsymbol{\lambda}\big(\mathcal{P}_{J_{\leq m}}(\ell_{r,s}^n)\big)  \prec_{C^m}\Big( \frac{n}{m}\Big)^{m\min\big\{\frac{1}{2},\frac{1}{r'}\big\}}\,.
\end{align*}
\end{theorem}
\begin{proof}

\noindent  We distinguish the two cases $\mathbf{A}:$ $2 \leq r < \infty,  \,1 \leq s \leq \infty$
and $\mathbf{B}:$ $1 < r \leq 2 ,\, 1 \leq s \leq \infty$. In case $\mathbf{A}$ both claims  follow from Theorem \ref{conny3}.
Let us turn to the case $\mathbf{B}$, and start with the upper bound for the projection constant. By  Theorem~\ref{lambda-dash} and Proposition~\ref{lambda} we have
\begin{align*}
  \boldsymbol{\lambda}\big(\mathcal{P}_{\leq m}(\ell_{r,s}^n)\big) & \leq \boldsymbol{\widehat{\lambda}}\big(\mathcal{P}_{\leq m}(\ell_{r,s}^n)\big)
  \\
  &
      \leq  \boldsymbol{\widehat{\lambda}}\big(\mathcal{P}_{\Lambda_T(\leq m,n)}(\ell_{r,s}^n)\big)
    \leq
  \sum_{k=0}^m \boldsymbol{\widehat{\lambda}}\big(\mathcal{P}_{\Lambda_T( k,n)}(\ell_{r,s}^n)\big)
    \leq
  \sum_{k=0}^m e^k
\left(  \frac{\varphi_{\ell_{r',s'}}(n)}{\varphi_{\ell_{r',s'}}(k)}  \right)^{k}\,.
 \end{align*}
 But since $\varphi_{\ell_{r',s'}}(n) \sim n^{1/r'}$ and
 \begin{equation} \label{wuerzburg}
 \sum_{0\leq k \leq m} \Big( \frac{n}{k}\Big)^{\frac{k}{r'}}
    \leq
    (m+1)
  \left(\max_{0\leq k \leq m}  \Big(\frac{n}{k}\Big)^k\right)^{\frac{1}{r'}}
    \leq    (m+1) \left(\sum_{0\leq k \leq m} \binom{n}{k}\right)^{\frac{1}{r'}}
    \leq   (m+1)e^{\frac{m}{r'}} \Big( \frac{n}{m}\Big)^{\frac{m}{r'}}\,,
\end{equation}
the claim follows.
In order to prove the upper bound for the unconditional basis constant, observe that by Corollary~\ref{main3A} (and by what we have already proved), it holds
\begin{align*}
  \boldsymbol{\chimon}\big(\mathcal{P}_{J_{\leq m}}(\ell_{r,s}^n)\big)
  &
  \leq
       \boldsymbol{\chimon}\big(\mathcal{P}_{\Lambda_T(\leq m,n)}(\ell_{r,s}^n)\big)
         \\
  &
    \prec_{C^m}  \max_{1 \leq  k \leq m-1} \boldsymbol{\lambda}\big( \mathcal{P}_{\Lambda_T(k,n)}(X_n)\big)
  \prec_{C^m}       \max_{1 \leq  k \leq m-1} \Big( \frac{n}{k}\Big)^{\frac{k}{r'}}\,,
 \end{align*}
hence the claim is again a consequence of the elementary estimate from \eqref{wuerzburg}.
\end{proof}

\smallskip

Let us turn to the lower estimates. Proposition~\ref{lor2} shows that these estimates, at least for $r=2$ and $s= \infty$,
do not give the correct asymptotics. See Corollary~\ref{thm: proy + uncon l2,s} for the collection 
of all results we know in the case $r=2$.

\smallskip

\begin{theorem} \label{t-final2}
Assume that  $\mathbf{A}: 2 < r < \infty,  \,2\leq s \leq \infty$ or 
$\mathbf{B}:  2 \leq r < \infty, \, 1 \leq s \leq 2$ or $\mathbf{C}: 1 < r < 2 ,\, r \leq s$.
Then for every tetrahedral index set $J \subset \mathbb{N}_0^{(\mathbb{N})}$
\begin{align*} \label{putin}
\Big( \frac{n}{m}\Big)^{(m-1)\min\big\{\frac{1}{2},\frac{1}{r'}\big\}}
\prec_{C^m}
 \boldsymbol{\chimon}\big(\mathcal{P}_{J_{\leq m}}(\ell_{r,s}^n)\big)
\,\,\,\, \text{ and } \,\,\,\,\,\,\,
 \Big( \frac{n}{m}\Big)^{m\min\big\{\frac{1}{2},\frac{1}{r'}\big\}}
 \prec_{C^m}
\boldsymbol{\lambda}\big(\mathcal{P}_{J_{\leq m}}(\ell_{r,s}^n)\big) \,.
\end{align*}
Moreover, in the missing case $2 \leq  r < \infty$ and $r \leq  s$ we only know that 
\[
\frac{1}{(\log n)^{\frac{m}{r}-\frac{m}{s} }}\Big( \frac{n}{m} \Big)^{\frac{m-1}{r'}}
\prec_{C^m}
\boldsymbol{\chimon}\big(\mathcal{P}_{J_{\leq m}}(\ell_{r,s}^n)\big)
\]
and
\[
\frac{1}{(\log n)^{\frac{m}{r}-\frac{m}{s} }}\Big( \frac{n}{m} \Big)^{\frac{m}{r'}}
 \prec_{C^m}
 \boldsymbol{\lambda}\big(\mathcal{P}_{J_{\leq m}}(\ell_{r,s}^n)\big)\,.
 \]
\end{theorem}

\begin{proof} 
In case $\mathbf{A}$ we know that $\ell_{r,s}$ is $2$-convex (see the introduction of this section),
and hence both claims are consequences of Theorem~\ref{conny3}.

In  case $\mathbf{B}$ we first deal with the 
unconditional basis constant:
Recall that  $\varphi_{\ell_{r,s}^n}(n) \sim n^{\frac{1}{r}}$
and $\varphi_{(\ell_{r,s}^n)'}(n) = \varphi_{\ell_{r',s'}^n}(n) \sim n^{\frac{1}{r'}}$, hence the first assumptions of Proposition~\ref{toblach} is satisfied. For the second assumption of this proposition note that
\begin{equation*}
\|\id: \ell_{r,s}^n \to \ell_2^n\| \leq \|\id: \ell_{r}^n \to \ell_2^n\| = n^{\frac{1}{2}-\frac{1}{r} } = \frac{n^{1-\frac{1}{r} } }{n^{\frac{1}{2} } } =\frac{\|\id: \ell_{r,s}^n \to \ell_1^n\|}{\sqrt{n}}\,.
\end{equation*}

Consequently, we deduce from Proposition~\ref{toblach} that for $m \leq n$
\[
\Big(\frac{n}{m}\Big)^{\frac{m-1}{2}}
\prec_{C^m}
\boldsymbol{\chimon}\big( \mathcal{P}_{\Lambda_T(m,n)}(\ell_{r,s}^n)\big)
\,.
\]
Since the case $n \leq m$ is anyway obvious, the claim for the lower bound of the unconditional basis constant is proved.

Let us consider  case $\mathbf{C}$:
For the unconditional basis constant this is an immediate  consequence of Proposition~\ref{innichen1}:
\[
\Big( \frac{n}{m} \Big)^{\frac{m-1}{r'}}
=
\frac{1}{\|\id:\ell_{r,s}^n\to \ell_r^n\|^m} \Big( \frac{n}{m} \Big)^{\frac{m-1}{r'}}
\prec_{C^m}
 \boldsymbol{\chimon}\big(\mathcal{P}_{\Lambda_T(m,n)}(\ell_{r,s}^n)\big) \,.
 \]
 It remains to show the claims for the projection constant in the cases $\mathbf{B}$ and $\mathbf{C}$: 
We use  Theorem~\ref{main3} together with Theorem~\ref{OrOuSe} which show
\begin{align}\label{end}
\boldsymbol{\chimon}\big(\mathcal{P}_{\Lambda_T(m+1,n)}(\ell_{r,s}^n)\big)  \prec_{C^m} \boldsymbol{\lambda}\big(\mathcal{P}_{\Lambda_T(m,n)}(\ell_{r,s}^n)\big) \,,
\end{align}
and this gives the lower bound for the projection constant, using the bounds we found for the unconditional basis constant (of one degree more).

Finally, it remains to prove the very last statement of the theorem. The  estimate for the unconditional basis constant
follows from    Proposition~\ref{innichen1}, and for the estimate on the projection constant we again repeat the argument from~\eqref{end}.
\end{proof}

\medskip

\section{The general case}

After the tetrahedral case we turn to more general index sets.
As indicated, a natural question one may ask  is whether the estimates from Theorem~\ref{start-poly2} for the  spaces $\ell_r^n$ remain valid for
$\ell_{r,s}^n$ regardless of the parameter $s$. We see in Theorem~\ref{t-final} and Theorem~\ref{t-final2} that if we restrict ourselves to tetrahedral indices, then this is in fact true --  at least for $r\ne 2$.
But in more complex  situations  subtle differences occur. It turns that the case $r \ge2$ essentially may be handled  with the 
methods which worked in the tetrahedral case. But different and technically much more involved is the  case  when $1 \leq r \leq 2$.
We later divide this study  into two subcases $s\leq r$ and $s \geq r$.

\subsection{The general case for  $2 \leq r < \infty$ and $1 \leq s \leq \infty$}
Theorem~\ref{t-final} and Theorem~\ref{t-final2} are concerned with the  tetrahedral index set.
The following theorem  extends these results (under further restrictions on $r$ and $s$) to more general situations.

\begin{theorem} \label{t-finalII}

Let $2 \leq r < \infty$  and $1 \leq s \leq \infty$,
and  $J \subset \mathbb{N}_0^{(\mathbb{N})}$
an index set.
Then
\begin{align*} \label{putin}
 \boldsymbol{\chimon}\big(\mathcal{P}_{J_{\leq m}}(\ell_{r,s}^n) \big)  \prec_{C^m}\Big( 1+\frac{n}{m}\Big)^{\frac{m-1}{2}}
\,\,\,\, \text{ and } \,\,\,\,\,\,\,
\boldsymbol{\lambda}\big(\mathcal{P}_{J_{\leq m}}(\ell_{r,s}^n) \big)  \prec_{C^m}
\Big(1+ \frac{n}{m}\Big)^{\frac{m}{2}}\,.
\end{align*}
Moreover, in each of the cases 
$\mathbf{A}: 2 < r < \infty, \,1 \leq s \leq \infty $  
or 
$\mathbf{B}: r= 2, \,1 \leq s \leq  2 $
the preceding estimates are optimal for all $m$ for which   $\Lambda_T(m) \subset J_{\leq m}$, in the sense that under this assumption
$\prec_{C^m}$ may be replaced by $\sim_{C^m}$, where $C >0$ is  only depending on $r,s$.

 \end{theorem}

\begin{proof}
The upper estimates are special cases  from Theorem~\ref{conny3}. The lower one for the unconditional basis constants  follow from
Theorem~\ref{t-final2}, since for all $m \leq n$ with $\Lambda_T(m) \subset J_{\leq m}$ clearly
\[
\boldsymbol{\chimon}\big( \mathcal{P}_{\Lambda_T(m,n)}(\ell_{r,s}^n)\big)
\leq
  \boldsymbol{\chimon} \big(\mathcal{P}_{J_{\leq m}} (\ell_{r,s}^n) \big)\,,
\]
and the claim for $m \ge n$ is anyway clear.
It remains to show the lower bound for the projection constant: By what was proved in \eqref{verflucht} we 
then for all $m \leq n$ with $\Lambda_T(m) \subset J_{\leq m}$ have 
\begin{align*}
  \boldsymbol{\chimon}\big( \mathcal{P}_{\Lambda_T(m+1,n)}(\ell_{r,s}^n)\big)
 \, \prec_{C^m}  \boldsymbol{\lambda}\big( \mathcal{P}_{J_{\leq m}}(\ell_{r,s}^n)\big)\,,
\end{align*}
hence the claim  follows from the result on the unconditional basis constant.
\end{proof}

\smallskip
We already indicated that the case $r=2$ is somewhat special -- at least for $s = \infty$ (see again Proposition~\ref{lor2}).
For a better overview, we collect our knowledge for $r=2$ in the following corollary.
\smallskip

\begin{corollary}\label{thm: proy + uncon l2,s}
Let  $1 \leq s \leq \infty$ and  $J \subset \mathbb{N}_0^{(\mathbb{N})}$
be an index set. Then
\[
\boldsymbol{\chimon}\big(\mathcal{P}_{J_{\leq m}}(\ell_{2,s}^n)\big)  \prec_{C^m}\Big(1+ \frac{n}{m}\Big)^{\frac{m-1}{2}}
\quad and \quad
\boldsymbol{\lambda}\big(\mathcal{P}_{J_{\leq m}}(\ell_{2,s}^n)\big)  \prec_{C^m}\Big( 1+\frac{n}{m}\Big)^{\frac{m}{2}}\,,
\]
and for $1 \leq s \leq 2$ these estimates  are optimal, provided $\Lambda_T(m) \subset J_{\leq m}$, whereas for 
$2 \leq s \leq \infty$ and $\Lambda_T(m) \subset J_{\leq m}$ we only know that 
\[
\frac{1}{(\log n)^{\frac{m-1}{2}-\frac{m-1}{s} }}
\Big( 1+\frac{n}{m} \Big)^{\frac{m-1}{2}}
\prec_{C^m}
\boldsymbol{\chimon}\big(\mathcal{P}_{J_{\leq m}}(\ell^n_{2,s})\big)
\]
and
\begin{equation} \label{april7} \frac{1}{(\log n)^{\frac{m}{2}-\frac{m}{s} }} \Big( 1+\frac{n}{m} \Big)^{\frac{m}{2}}
\prec_{C^m} \boldsymbol{\lambda}\big(\mathcal{P}_{J_{\leq m}}(\ell^n_{2,s})\big)\,.
  \end{equation}
  Moreover, if   $J$ is one of the index sets  $\Lambda_T (m,n)$,
          $\Lambda (m,n)$,
  $\Lambda_T (\leq m,n)$,  or  $\Lambda (\leq m,n)$, then
\[
\frac{1}{m^{\frac{m-1}{2}}} \Big(\frac{n}{\log n} \Big)^{\frac{m-1}{2}}
\prec_{C^m}
\boldsymbol{\chimon}\big(
\mathcal{P}_{J}(\ell^n_{2,\infty}) \big) \prec_{C^m}
\Big(\frac{n}{\log \log n}\Big)^{\frac{m-1}{2}}
 \]
 and
 \[
\frac{1}{m^{m}}  \Big(\frac{n}{\log \log n}\Big)^{\frac{m}{2}}
 \prec_{C^m}
\boldsymbol{\lambda}\big(
\mathcal{P}_{J}(\ell^n_{2,\infty}) \big) \prec_{C^m}
\Big(\frac{n}{\log \log n}\Big)^{\frac{m}{2}}\,.
\]
\end{corollary}

\begin{proof}The first four statements are all covered by Theorem~\ref{t-finalII}  and Theorem~\ref{t-final2}. In order  to prove the last two statements, we for the projection constant use 
Proposition~\ref{lor} and 
Corollary~\ref{immediateABC} (last statement).
For the upper estimate of  the unconditional basis 
constant we use 
\eqref{verflucht} together with the upper estimate for the projection constant just proved, whereas the lower estimate
for the unconditional basis 
constant is nothing else than  the lower estimate from \eqref{april7}
for $s= \infty$.
\end{proof}

\medskip

 \subsection{The general case for $1 < r  \le 2$ and $s\leq r$}
  For $1 \leq s \leq r \leq 2$ and an index set $J \subset  \mathbb{N}_0^{(\mathbb{N})}$, we take advantage of how permeable the definition of the polynomial projection constant $\widehat{\boldsymbol{\lambda}}(\mathcal P_{J_{\leq m}}(\ell_{r,s}^n))$ is. That is, we break the sum $\sum_{\alpha \in J} c_{\ell_{r,s}^n}(\alpha) \vert z^{\alpha}\vert$
into smaller pieces.

This splitting is related to the number of variables involved in each of the different monomials.
The idea behind all this is that, if an index $\alpha$ concerns a large amount of variables, then $c_{\ell_{r,s}^n}(\alpha)$ is essentially similar to $c_{\ell_{r}^n}(\alpha)$, so we are in the classical case.
If the number of variables involved in an index $\alpha$ is not so large, then we have  rather annoying logarithmic term  in $c_{\ell_{r,s}^n}(\alpha)$ (compared to $c_{\ell_{r}^n}(\alpha)$).

To deal with this we  carefully analyze the number of monomials which are related with a fixed amount of variables. The philosophy is that those indices that produce bad estimates are not so many. So there is a kind of compensation. The problem is that dealing with all of this simultaneously requires a lot of subtle technicalities. The main difficulty lies in properly handle the balance between the number of indices and the corresponding estimates one obtains,  and in the fact that finally  all the pieces  have to be glued together.

We believe that this new technique is interesting in its own right and could be very useful for other types of problems, even not related with the projection constant.
As a consequence we show that, in many cases, for $1 \leq s \leq r \leq 2$ the bound for the projection constant is the same as in the classical $\ell_r^n$-case.

\begin{theorem}\label{bound_similar_ell_r}
	Let  $1\le s \le r\le 2$ and  $J\subset \mathbb N_0^n$ an index set  of degree $m$ such that $\Lambda_T(m,n)\subset J$. Then 
	$$
	\boldsymbol{\chimon}\big(\mathcal P_J( \ell_{r,s}^n)\big) \sim_{C^m}  \left( \frac{n}{m} \right)^{\frac{m-1}{r'}}
\,\,\quad \text{and}\,\,\quad 
 \boldsymbol{\lambda}\big(\mathcal P_J(\ell_{r,s}^n)\big)  \sim_{C^m}   \left( \frac{n}{m} \right)^{\frac{m}{r'}}\,,
	$$
	provided
	\[
		\log m + \frac{r'}{r}(\log m)^{r(\frac1{s}-\frac1{r})}\prec_c  \log n\,.
	\]
In particular, these estimates hold if $m^{1+\delta}\prec_c n$ for some $\delta>0$ and $(r,s)\ne(2,1)$, and if $m^{2}\prec_c n$ for $(r,s)=(2,1)$.
\end{theorem}

We prepare the proof with some preliminary lemmas.
To do so, we need to consider some special sets of multi indices, which determine subspaces of the polynomials defined in terms of how many variables are involved in each monomial.

We define, for $1\le L\le m,$ the set of indices
$$\Lambda^L(m,n):=\left\{ \alpha \in \Lambda(m,n) : \;\; \vert\{ i : \alpha_i \neq 0 \} \vert = L \right\}. $$
In other words,  $\alpha\in \Lambda^L(m,n)$ whenever the monomial $z^\alpha$ involves exactly $L$ variables.
	Given a Banach space $X_n=(\mathbb C^n,\|\cdot\|)$, we denote
	$$\mathcal{P}_{m,L}(X_n):=  \mathcal{ P}_{\Lambda^L(m,n)}(X_n)\,.$$
		Observe that with this notation the set of tetrahedral $m$-homogenous polynomials may be denoted by $
	\mathcal P_{m,m}(X_n).$
	Our plan is to give upper  estimates for  the projection constant of   $\mathcal P_{m,L}(\ell_{r,s}^n)$, and for this we  need an upper bound for the cardinality of $\Lambda^{L}(m,n).$

\begin{lemma}\label{cardinal J^L} For  $n,m \in \NN$ and $ 1 \le L \le m$ 
$$\displaystyle
|\Lambda^{L}(m,n)|\sim_{C^m} \Big(\frac{n}{L}\Big)^{L} \sim_{C^m} \binom{n}{L}^{L}\,.
$$
\end{lemma}
\begin{proof}
Any multi index $\alpha\in\Lambda^L(m,n)$ can be written as a sum of a tetrahedral index $\beta\in \Lambda_T(L,n)$ and another index whose support is contained in the support of $\beta$. Moreover, if $\alpha$ has $L-k$ coordinates with value $1$,  the remaining $k$ non-zero coordinates of $\alpha$ are $\ge2$, and then
\[
L-k+2k\le m, \quad or,\quad k\le m-L\,.
\]
Thus the decreasing reordering of $\alpha$ may be written as 
\[
\alpha^*=\beta^* \, + \,(\alpha^*-\beta^*)
=(1,\dots,1,0,\dots)+(\alpha_1^*-1,
\dots,\alpha_k^*-1,0\dots)\,.
\]
Therefore, as any $\alpha^*$ can be decomposed as a sum of $\beta^* \in \Lambda_T(L,n)$ and $\alpha^* - \beta^* \in \Lambda(m-L,m-L)$, we have
\begin{align}
|\Lambda^{L}(m,n)|\le |\Lambda_T(L,n)|\cdot|\Lambda(m-L,m-L)|\le \binom{n}{L}\cdot\Big(1+\frac{m-L}{m-L}\Big)^{m-L}
\prec_{C^m} \Big(\frac{n}{L}\Big)^{L}\,. 
\end{align}
The lower bound follows from the fact that we may define an injection from $\Lambda_T(L,n)$ to $\Lambda^{L}(m,n)$, given by
$
\alpha\, \mapsto \, (\alpha_1+m-L,\alpha_2,\dots,\alpha_n).
$
Therefore 
\[
|\Lambda^{L}(m,n)|\ge |\Lambda_T(L,n)|
\sim_{C^m} \Big(\frac{n}{L}\Big)^{L}\,,
\]
which completes the argument.
\end{proof}

\begin{remark}
From the above lemma we obtain a simple bound for the polynomial projection constant:
\[
\widehat{\boldsymbol{\lambda}}\big(\mathcal P_{m,L}(\ell_{r,s}^n)\big) \prec_{C^m} \Big(\frac{n}{L}\Big)^{\frac{L}{r'}}(\log m)^{m(\frac1{s}-\frac1{r})}\,.
\]
Indeed,
\[
{c_{\ell_{r,s}^n} (\alpha)}= \frac1{\sup_{v \in B_{\ell{r,s}^n}} |v_\alpha|}\le {\|m^{-1/r}(\alpha_1^{1/r},\dots,\alpha_n^{1/r})\|_{r,s}^{m}} \left( \frac{\alpha^\alpha}{m^{m}} \right)^{1/r}\le  (\log m)^{m(\frac1{s}-\frac1{r})}|[\alpha]|^{1/r},
\]
and thus
\begin{align*}
\sum_{\alpha \in \Lambda^{L}(m,n)}  c_{\ell_{r,s}^n}(\alpha) |z^\alpha| & \le \log(m)^{m(\frac1{s}-\frac1{r})} \sum_{\alpha \in \Lambda^{L}(m,n)} |[\alpha]|^{1/r}|z^\alpha|
 \nonumber\\
& \le
\log(m)^{m(\frac1{s}-\frac1{r})} \left(\sum_{\alpha \in \Lambda^{L}(m,n)} |[\alpha]||z^{\alpha r}| \right)^\frac{1}{r}\left| \Lambda^{L}(m,n) \right|^{1/r'}  \nonumber\\
& \prec_{C^m} \Big(\frac{n}{L}\Big)^{\frac{L}{r'}}(\log m)^{m(\frac1{s}-\frac1{r})} \|z\|_{\ell_r}^{m}.
\end{align*}
\end{remark}
Unfortunately the  bound of the preceding remark is not enough to prove Theorem \ref{bound_similar_ell_r}. We  need to obtain  finer estimates that are provided by  the following lemma.
\begin{lemma} \label{proj lorentz s<r - ell_r bound}
Let  $1\le s \le r\le 2$. Then, for
\begin{itemize}
\item[(i)] $L\ge m\left(1-\frac{1}{(\log m)^{r(\frac1{s}-\frac1{r})}}\right)$ and $1<m\le n,$
\item[(ii)] $\frac{m}{2}<L<m\left(1-\frac{1}{(\log m)^{r(\frac1{s}-\frac1{r})}}\right)$ and $\log m + \frac{r'}{r}(\log m)^{r(\frac1{s}-\frac1{r})}\le  \log n$,
\item[(iii)] $L\le\frac{m}{2}$ and $ m \log(m)^{2r'(\frac{1}{s}-\frac{1}{r})}\le n$,
\end{itemize}
we have 
\[
\widehat{\boldsymbol{\lambda}}\big(\mathcal P_{m,L}(\ell_{r,s}^n)\big) \prec_{C^m} \Big(\frac{n}{m}\Big)^{\frac{m}{r'}}.
\]
\end{lemma}

\medskip

Before we prove this lemma we show that it indeed allows to prove as desired Theorem \ref{bound_similar_ell_r}.

\begin{proof}[Proof of Theorem \ref{bound_similar_ell_r}]
The lower bounds hold for any $n\ge m$:
 For the unconditional basis constant this  follows directly from Proposition \ref{innichen1}.
The lower estimate for the projection constant uses that by Theorem~\ref{OrOuSe}, Theorem \ref{degree-homo}, and Corollary~\ref{main3A}  we have
\begin{align*}
\Big(\frac{n}{m}\Big)^{\frac{m}{r'}} &
\prec_{C^m} \boldsymbol{\chimon}\big( \mathcal{P}_{\Lambda_T(m+1,n)}(X_n)\big) \\
& \,\prec_{C^m}\max_{0 \leq k \leq m}\boldsymbol{\lambda}\big( \mathcal{P}_{\Lambda_T(k,n)}(X_n)\big)
\,\le{\kappa^m} \,\max_{0 \leq k \leq m} \boldsymbol{\lambda}\big(\mathcal{P}_{J_k}(X_n)\big) \,\le{\kappa^m} \,\boldsymbol{\lambda}\big(\mathcal{P}_{J}(X_n)\big).
\end{align*}
The upper bound for the projection constant is 
a~consequence of  Lemma \ref{proj lorentz s<r - ell_r bound}, the fact that 
\[
\widehat{\boldsymbol{\lambda}}\big(\mathcal P_{J}(\ell_{r,s}^n))\le \sum_{L=0}^m\widehat{\boldsymbol{\lambda}}(\mathcal P_{m,L}(\ell_{r,s}^n)\big)\,,
\]
and Theorem \ref{lambda-dash}.
For the upper bound of the unconditional basis constant we may apply Corollary~\ref{main3A} to conclude that
 \begin{equation*}
     \boldsymbol{\chimon}\big(\mathcal{P}_{J}(X_n)\big)\prec_{C^m}
 \boldsymbol{\chimon}\big(\mathcal{P}_{\le m}(X_n)\big)\prec_{C^m}\max_{k\le m-1} \boldsymbol{\lambda}\big(\mathcal{P}_{k}(X_n)\big)\prec_{C^m}  \Big(\frac{n}{m}\Big)^{\frac{m-1}{r'}}\,. \qedhere
 \end{equation*}
 \end{proof}
 
 We turn to the proof of Lemma \ref{proj lorentz s<r - ell_r bound}, which needs further preparation. In fact, this lemma is an immediate consequence of
Lemma~\ref{lemma1 L<m/2}, Lemma~\ref{lemma3 L<m/2} and Lemma~\ref{lemma L>m/2} below.
We start with the case when $L$ is big -- so  the case when the polynomials are 'almost tetrahedral'.

\begin{lemma}\label{lemma1 L<m/2}
Let $1 \leq s \leq r \leq 2$ and $L\ge  m\left(1-\frac{1}{(\log m)^{r(\frac1{s}-\frac1{r})}}\right)$. Then
\begin{equation}\label{eq: lema1 s<r}
\widehat{\boldsymbol{\lambda}}\big(\mathcal P_{m,L}(\ell_{r,s}^n)\big)\prec_{C^m}  \Big(\frac{n}{L}\Big)^{\frac{L}{r'}} \prec_{C^m} \Big(\frac{n}{m}\Big)^{\frac{m}{r'}}\,,
\end{equation}
for every $3\le m\le n.$ Moreover, equation \eqref{eq: lema1 s<r} also holds for $1\le L\le m\le2.$
\end{lemma}

\begin{proof}
	We start with some remarks on the norm of the vector $m^{-1/r}\alpha^{1/r}$.
    Note that if $\alpha$ is tetrahedral, then
 	\begin{equation}\label{norm for alpha tetrahedral}
  \|m^{-1/r}\alpha^{1/r}\|_{r,s} = \|m^{-1/r}(\underbrace{1,\dots,1}_{m},0,\dots)\|_{r,s}\sim 1.
	\end{equation}
We show next that for $\alpha$ almost tetrahedral a similar bound holds.	
		Suppose that  $\alpha  \in \Lambda_{L}(m,n)$, and suppose also that $2L\ge m$. Then there are at most $m-L$ coordinates greater than 1 (in particular there are at least $2L-m$ coordinates which are $1$), thus
$$
\alpha^*=(\alpha_1^*,\dots,\alpha_{m-L}^*,\underbrace{1,\dots,1}_{2L-m},0,\dots).
$$
    Recall that  $\|M^{-1/r}\beta^{1/r}\|_{\ell_r}=1$ for any $\beta\in\Lambda(M,N)$. Then, since the degree of $(\alpha_1^*,\dots,\alpha_{m-L}^*)$ equals $2(m-L)$, using \eqref{norm for alpha tetrahedral},
\begin{align}\label{norma}
\nonumber \|m^{-1/r}(\alpha_1^{1/r},\dots,\alpha_n^{1/r})\|_{r,s}
 &\le  \|m^{-1/r}((\alpha_1^*)^{1/r},\dots,(\alpha_{m-L}^*)^{1/r})\|_{r,s}+
 \|m^{-1/r}(\underbrace{1,\dots,1}_{2L-m},0,\dots)\|_{r,s}\\
 \nonumber
 &\le \|(2(m-L))^{-1/r}((\alpha_1^*)^{1/r},\dots,(\alpha_{m-L}^*)^{1/r})\|_{r}\Big(\frac{2(m-L)}{m}\Big)^{1/r}\log(m-L)^{\frac1{s}-\frac1{r}}\\
 &+\|(2L-m)^{-1/r}(\underbrace{1,\dots,1}_{2L-m},0,\dots)\|_{r,s} \nonumber \\
 &\le \Big(\frac{2(m-L)}{m}\Big)^{1/r}\log(m-L)^{\frac1{s}-\frac1{r}}+1.
\end{align}
Thus, if
$m-L\le \frac{m}{(\log m)^{r(\frac1{s}-\frac1{r})}}$ for $3\le m\le n$, then 
we have
\begin{align}\label{bound norm in lrs many variables}
    \|m^{-1/r}\alpha^{1/r}\|_{r,s}\le 2^{1/r}+1.
\end{align}
Moreover, note that the estimate from  \eqref{bound norm in lrs many variables} also holds for $1\le L\le m\le2.$
Then for $\alpha\in\Lambda^L(m,n)$ such that
\[
\text{
$\displaystyle L\ge m\left(1-\frac{1}{(\log m)^{r(\frac1{s}-\frac1{r})}}\right)$, with either  $3\le m\le n$
 or  $1\le L\le m\le2$\,,}
\]
we have
$$
{c_{\ell_{r,s}^n} (\alpha)}= \sup_{v \in B_{\ell_{r,s}^n}}\frac1{|v_\alpha|}\le {\|m^{-1/r}(\alpha_1^{1/r},\dots,\alpha_n^{1/r})\|_{r,s}^{m}} \left( \frac{\alpha^\alpha}{m^{m}} \right)^{1/r}
\le  \big((2^{1/r}+1)e^{1/r}\big)^{m}|[\alpha]|^{1/r}\,,
$$
and therefore
\begin{align}
\sum_{\alpha \in \Lambda^{L}(m,n)} c_{\ell_{r,s}^n}(\alpha) |z^\alpha| & \le \big((2^{1/r}+1)e^{1/r}\big)^{m}
 \sum_{\alpha \in \Lambda^L  (m,n)} |z^\alpha| |[\alpha]|^{1/r}  \prec_{C^m}  \Big(\frac{n}{L}\Big)^{\frac{L}{r'}} \prec_{C^m}  \Big(\frac{n}{m}\Big)^{\frac{m}{r'}}.\label{final bound for L>m-sqrt(m)}
 \qedhere
\end{align}
\end{proof}

\smallskip

\begin{lemma}\label{lemma2 L<m/2}
Let $1 \leq s \leq r \leq 2$ and $1<\frac{m}{2}<L\le m\left(1-\frac{1}{(\log m)^{r(\frac1{s}-\frac1{r})}}\right)$.
 Then
\begin{equation}
\widehat{\boldsymbol{\lambda}} (\mathcal P_{m,L}(\ell_{r,s}^n))\prec_{C^m} 
\, \|z\|_{\ell_r}^{m}\Big(\frac{n}{m}\Big)^{m/r'}
   \log(m-L)^{m(\frac1{s}-\frac1{r})}
   \Big(\frac{m}{n}\Big)^{\frac{m-L}{r'}}\left(\frac{m-L}{m}\right)^{\frac{m}{r}}\,.
\end{equation}
\end{lemma}
\begin{proof}
Note that, by \eqref{norma}, for any $\alpha\in\Lambda^L(m,n)$
\begin{align}\label{bound norm in l_rs few variables}
    \|m^{-1/r}\alpha^{1/r}\|_{r,s}\le 3\Big(\frac{2(m-L)}{m}\Big)^{1/r}\log(m-L)^{\frac1{s}-\frac1{r}}.
\end{align}
Thus, by Lemma \ref{cardinal J^L},
\begin{align*}
\sum_{\alpha \in \Lambda^{L}(m,n)}  c_{\ell_{r,s}^n}(\alpha) |z^\alpha| & \prec_{C^m} 3^m\Big(\frac{2(m-L)}{m}\Big)^{m/r}\log(m-L)^{m(\frac1{s}-\frac1{r})} \sum_{\alpha \in \Lambda^{L}(m,n)} |[\alpha]|^{1/r}|z^\alpha|
 \nonumber\\
& \prec_{C^m}
 \left(\Big(\frac{m-L}{m}\Big)^{m/r}\log(m-L)^{m(\frac1{s}-\frac1{r})} \left| \Lambda^{L}(m,n) \right|^{1/r'} \right) \|z\|_{\ell_r}^{m}\nonumber\\
& \prec_{C^m}  \|z\|_{\ell_r}^{m}\Big(\frac{m-L}{m}\Big)^{m/r}\log(m-L)^{m(\frac1{s}-\frac1{r})}  \left( \frac{n}{L} \right)^{L/r'}  \nonumber\\
& \prec_{C^m}  \|z\|_{\ell_r}^{m}\Big(\frac{n}{m}\Big)^{m/r'}\log(m-L)^{m(\frac1{s}-\frac1{r})}
\Big(\frac{m^{\frac{m-L}{r'}-\frac{m}{r}}}{n^{\frac{m-L}{r'}}}\Big)(m-L)^{\frac{m}{r}}\,,
\end{align*}
as desired.
\end{proof}

\smallskip

The next lemma bounds the polynomial projection constant of $\mathcal P_{m,L}$ for the cases $\frac{m}{2}<L$ not included in Lemma~\ref{lemma1 L<m/2}.

\smallskip

\begin{lemma}\label{lemma3 L<m/2}
Let $1 \leq s \leq r \leq 2$ and $1<\frac{m}{2}<L<m\left(1-\frac{1}{(\log m)^{r(\frac1{s}-\frac1{r})}}\right)$. Then
\begin{itemize}
\item[(i)] for $e^{-\frac{2r}{r'}}n\le m\le n$,
\begin{align*}
\widehat{\boldsymbol{\lambda}}\big(\mathcal P_{m,L}(\ell_{r,s}^n)\big) &\prec_{C^m} (\log m)^{m(\frac1{s}-\frac1{r})}.
\end{align*}

\item[(ii)] for $m\le e^{-\frac{2r}{r'}}n$ and $ \log n \le \log m + \frac{r'}{r}(\log m)^{r(\frac1{s}-\frac1{r})}$,
\[
\widehat{\boldsymbol{\lambda}}\big(\mathcal P_{m,L}(\ell_{r,s}^n)\big)\prec_{C^m} \Big(\frac{n}{m}\Big)^{\frac{m}{r'}} \frac{\left(\log m\right)^{m(\frac1{s}-\frac1{r})}}{\left(\log \frac{n}{m}\right)^{\frac{m}{r}}}.
\]

\item[(iii)] for $ \log n \ge \log m + \frac{r'}{r}(\log m)^{r(\frac1{s}-\frac1{r})}$,
\begin{align*}
\widehat{\boldsymbol{\lambda}}\big(\mathcal P_{m,L}(\ell_{r,s}^n)\big) &\prec_{C^m} \Big(\frac{n}{m}\Big)^{\frac{m}{r'}\left(1-\frac1{(\log m)^{r(\frac1{s}-\frac1{r})}}\right)}.
\end{align*}
In particular, this holds for $n\ge m^{1+\delta},$
provided $\delta\ge \frac{r'}{r(\log m)^{2-\frac{r}{s}}}$
\end{itemize}
\end{lemma}

\begin{proof}
By Lemma~\ref{lemma2 L<m/2},
\begin{align*}
\widehat{\boldsymbol{\lambda}}\big(\mathcal P_{m,L}(\ell_{r,s}^n)\big) & \prec_{C^m} \Big(\frac{n}{m}\Big)^{m/r'}\log(m-L)^{m(\frac1{s}-\frac1{r})}\Big(\frac{m^{\frac{m-L}{r'}-\frac{m}{r}}}{n^{\frac{m-L}{r'}}}\Big)(m-L)^{\frac{m}{r}}.
\end{align*}
Assume now that $m-L=tm$, with $(\log m)^{-r(\frac1{s}-\frac1{r})}\le t\le \frac12.$ Then, taking the $m$-th root and rearranging,
\begin{align}
\widehat{\boldsymbol{\lambda}}\big(\mathcal  P_{m,L}(\ell_{r,s}^n)\big)^{1/m}\Big(\frac{m}{n}\Big)^{1/r'}\log(m)^{\frac1{r}-\frac1{s}}
\prec_C t^{\frac1{r}}\left(\frac{m}{n}\right)^{t/r'}.
\label{L as a function of t}
\end{align}
Let $f(t):=\frac1{r}\log(t)-\frac{t}{r'}\log\frac{n}{m}$ for all $t>0$. Then, $f$ has a global maximum at $t_0=\left(\frac{r}{r'}\log\frac{n}{m}\right)^{-1}
$
Now, since we are looking for the maximum of $f$ for  $(\log m)^{-r(\frac1{s}-\frac1{r})}\le t\le \frac12$,  we have three possibilities:
\begin{enumerate}
	\item[$(i)$] \label{case1} $t_0>\frac12$, in which case we  consider $t^*:=0$. This is the case when $n<e^{\frac{2r'}{r}}m$.
	\item[$(ii)$] \label{case3} $0\le t_0\le r(\frac1{s}-\frac1{r})$, in which case we  consider $t^*:=t_0$. This is the case when $e^{\frac{2r'}{r}}\le n$  and $\log n\le \log m +\frac{r'}{r}(\log m)^{r(\frac1{s}-\frac1{r})}$.
	\item[$(iii)$] \label{case2} $t_0<(\log m)^{-r(\frac1{s}-\frac1{r})}$, in which case we  consider $t^*:=(\log m)^{-r(\frac1{s}-\frac1{r})}$. This is the case when $\log n\le \log m +\frac{r'}{r}(\log m)^{r(\frac1{s}-\frac1{r})}$.
\end{enumerate}
The proof finishes replacing $t$ by $t^*$ in \eqref{L as a function of t} in each case.
\end{proof}

It remains to prove the case $L<\frac{m}{2}.$ The key point in this case is that although we don't have a good bound for $c_{\ell_{r,s}^n}$, the cardinality of $\Lambda^L(m,n)$ is small enough.
\begin{lemma}\label{lemma L>m/2}
Let $1 \leq s \leq r \leq 2$,  $L\le \frac{m}{2}$. Then for every
$n \geq m \log(m)^{2r'(\frac{1}{s}-\frac{1}{r})}$,
\begin{equation} \label{good bound for L>m/2}
\widehat{\boldsymbol{\lambda}}\big(\mathcal P_{m,L}(\ell_{r,s}^n)\big)\prec_{C^m} \left(\frac{n}{m} \right)^{\frac{m}{r'}}.
\end{equation}
\end{lemma}

\begin{proof}
Let $\alpha \in \Lambda^{L}(m,n)$.
Since $\alpha$ involves exactly $L$ variables, we have 
\[ 
\Big\Vert \frac{\alpha^{1/r}}{m^{1/r}}\Big\Vert_{\ell_{r,s}} \prec_C \Big\Vert \frac{\alpha^{1/r}}{m^{1/r}}\Big\Vert_{\ell_{r}} \log(L)^{\frac{1}{r}-\frac{1}{s}}\prec_C  \log(L)^{\frac{1}{r}-\frac{1}{s}}\,,
\]
Thus,
\begin{align}\label{eq: bound c_alpha J^L}
c_{\ell_{r,s}}(\alpha) \prec_{C^m} \log(L)^{m(\frac{1}{r}-\frac{1}{s})} |[\alpha]|^{\frac1{r}}.
\end{align}
Therefore, using Lemma \ref{cardinal J^L},
\begin{align}
\sum_{\alpha \in \Lambda^{L}(m,n)}  c_{\ell_{r,s}^n}(\alpha) |z^\alpha| & \prec_{C^m} \log(L)^{m(\frac1{s}-\frac1{r})} \sum_{\alpha \in \Lambda^{L}(m,n)} |[\alpha]|^{1/r}|z^\alpha|
 \nonumber\\
& \prec_{C^m}
 \left(\log(L)^{m(\frac1{s}-\frac1{r})} \left| \Lambda^{L}(m,n) \right|^{1/r'} \right) \|z\|_{\ell_r}^{m}\nonumber\\
& \prec_{C^m}  \|z\|_{\ell_r}^{m}\log(L)^{m(\frac1{s}-\frac1{r})}  \left( \frac{n}{L} \right)^{L/r'} \label{bound depending on L for L>m/2} \\
&= \|z\|_{\ell_r}^{m}\,\Big(\frac{n}{m}\Big)^{m/r'}\log(L)^{m(\frac1{s}-\frac1{r})}\Big(\frac{m^m}{n^{m-L}L^L}\Big)^{\frac{1}{r'}}.\nonumber
\end{align}
For $n \geq m \log(m)^{2r'(\frac{1}{s}-\frac{1}{r})}$, and $L<\frac{m}{2},$ (and hence $2(m-L)>m$),
\begin{align*}
\widehat{\boldsymbol{\lambda}}\big(\mathcal P_{m,L}(\ell_{r,s}^n)\big)
&
\le \Big(\frac{n}{m}\Big)^{m/r'}\log(m)^{m(\frac1{s}-\frac1{r})}\Big(\frac{m^m}{m^{m-L}L^L}\Big)^{\frac{1}{r'}}\log(m)^{-2(m-L)(\frac1{s}-\frac1{r})}\\
&
\le \Big(\frac{n}{m}\Big)^{m/r'}\Big(\frac{m}{L}\Big)^{\frac{L}{r'}}\prec_{C^m} \Big(\frac{n}{m}\Big)^{m/r'}.
\end{align*}
This proves \eqref{good bound for L>m/2}.
\end{proof}

As explained above, this completes the proof of  Theorem~\ref{bound_similar_ell_r}.
We finish this subsection stating an upper bound which holds for any $n,m$, which we in Section \ref{Section: bohr on lorentz} need to obtain asymptotically optimal bounds for the Bohr radius on Lorentz spaces.

\begin{theorem} \label{proj lorentz s<r}
	Let  $1\le s \le r\le 2$, $J\subset \mathbb N_0^n$ of degree $m$
	and  $0<\kappa<1$. Then for every $n\ge m$,
\begin{equation}\label{proj lorentz s<r - bound}
   {\boldsymbol{\lambda}}\big(\mathcal P_{J}(\ell_{r,s}^n)\big)\prec_{C^m}\Big(\frac{n}{m}\Big)^{\frac{m}{r'}}\max\left\{\frac{\log(m)^{m(\frac1{s}-\frac1{r})}}{n^{\frac{m^\kappa}{2r'}}};\,1\right\}.
\end{equation}
\end{theorem}

\begin{proof}
Since  by Theorem~ \ref{lambda-dash}
\[
\boldsymbol{\lambda}\big(\mathcal P_{J}(\ell_{r,s}^n)\big)\le \widehat{\boldsymbol{\lambda}}\big(\mathcal P_{J}(\ell_{r,s}^n)\big)\le \sum_{L=0}^m\widehat{\boldsymbol{\lambda}}\big(\mathcal P_{m,L}(\ell_{r,s}^n)\big)\,,
\]
it suffices to show that  $\widehat{\boldsymbol{\lambda}}\big(\mathcal P_{m,L}(\ell_{r,s}^n)\big)$
for each $L\le m$ satisfies the bound in \eqref{proj lorentz s<r - bound}. Lemma~\ref{lemma1 L<m/2} implies the bound for $L\ge m\Big(1-\frac{1}{(\log m)^{r(\frac1{s}-\frac1{r})}}\Big)$.

For $\frac{m}{2}\le L\le m\Big(1-\frac{1}{(\log m)^{r(\frac1{s}-\frac1{r})}}\Big)$ we use Lemma~\ref{lemma3 L<m/2}. Part~$(iii)$ shows the desired bound for $n$ big enough.

In the cases $(i)$ and $(ii)$ of Lemma~\ref{lemma3 L<m/2} we have that $n\prec_C m^2$, and then $n^{\frac{m^\kappa}{2r'}}\prec_{C^m} 1$. Therefore Lemma~\ref{lemma3 L<m/2} $(i)$ and $(ii)$ also imply the bound in \eqref{proj lorentz s<r - bound}
for $\widehat{\boldsymbol{\lambda}}(\mathcal P_{m,L}(\ell_{r,s}^n))$.

For  $L\le \frac{m}{2}$ we use \eqref{eq: bound c_alpha J^L} from Lemma~\ref{lemma L>m/2}.
 Note that since $L\le \frac{m}{2}\le m-\frac{m^\kappa}{2}\le n$,
$$
\binom{n}{L} \le \binom{2n}{L} \le \binom{2n}{[m-\frac{m^\kappa}{2}]}.
$$
Thus, by \eqref{eq: bound c_alpha J^L} and Lemma~\ref{cardinal J^L}, for every $z \in \mathbb{C}^n$
\begin{align*}
\sum_{\alpha \in \Lambda^{L}(m,n)}  c_{\ell_{r,s}^n}(\alpha) |z^\alpha| & \le
 \left(\log(L)^{m(\frac1{s}-\frac1{r})} \left| \Lambda^{L}(m,n) \right|^{1/r'} \right) \|z\|_{\ell_r}^{m}\nonumber\\
& {\prec_{C^m}} \log(m)^{m(\frac1{s}-\frac1{r})}  \binom{2n}{[m-m^\kappa/2]}^{1/r'}  \|z\|_{\ell_r}^{m}\nonumber\\
& {\prec_{C^m}} \log(m)^{m(\frac1{s}-\frac1{r})} \left( \Big(\frac{n}{m-m^\kappa/2}\Big)^{m-m^\kappa/2} \right)^{1/r'}  \|z\|_{\ell_r}^{m}\nonumber\\
& = \|z\|_{\ell_r}^{m}\Big(\frac{n}{m}\Big)^{m/r'}\log(m)^{m(\frac1{s}-\frac1{r})}\Big(\frac{m-m^\kappa/2}{n}\Big)^{\frac{cm^\kappa}{r'}}\Big(\frac{m}{m-m^\kappa/2}\Big)^{\frac{m}{r'}}\nonumber\\
& \prec_{C^m} \|z\|_{\ell_r}^{m}\Big(\frac{n}{m}\Big)^{m/r'}\Big(\frac{\log(m)^{m(\frac1{s}-\frac1{r})}}{n^{\frac{m^\kappa}{2r'}}}\Big)\,,\nonumber
\end{align*}
again implying \eqref{proj lorentz s<r - bound}
for $\widehat{\boldsymbol{\lambda}}(\mathcal P_{m,L}(\ell_{r,s}^n))$.
\end{proof}

\subsection{The general case for $1 < r \le 2$ and $r \leq s$}
For $1 \leq r \leq 2 $ and $s \geq r$ we follow a different strategy. We do not partition the Banach spaces $\mathcal P_J(\ell_{r,s}^n)$ into small pieces as before; instead, we use a certain decomposition of multi indices
in order  to factorize the sum  defining the polynomial projection constant $\widehat{\boldsymbol{\lambda}}(\mathcal P_J(\ell_{r,s}^n))$ as a product of certain terms
involving monomials of lower degrees, for which we are able to find proper bounds.

\begin{theorem} \label{proj lorentz s>r}
Let $1 < r \leq 2$ and $r \leq s$ and $J\subset \mathbb N_0^n$ an index set of degree $m$. Then, for  $\frac{n}{e(\log n)^{r'(\frac1{r}-\frac1{s})}}\ge~m$, we have
\[
\widehat{\boldsymbol{\lambda}}\big(\mathcal P_J(\ell_{r,s}^n)\big) \prec_{C^m} \Big(\frac{n}{m} \Big)^{\frac{m}{r'}}.
\]
In particular,
\[
\boldsymbol{\chimon}\big(\mathcal P_{J}(\ell_{r,s}^n)\big) \prec_{C^m}  \Big(\frac{n}{m} \Big)^{\frac{m-1}{r'}} \quad \textrm{and}\quad
{\boldsymbol{\lambda}}\big(\mathcal P_J(\ell_{r,s}^n)\big) \prec_{C^m} \Big(\frac{n}{m} \Big)^{\frac{m}{r'}}.
\]
\end{theorem}
\bigskip

Since for $s\geq r$, we have
$c_{\ell_{r,s}^n}(\alpha) \leq c_{\ell_{r}^n}(\alpha) \sim_{C^m} |[\alpha]|^{1/r},$ to prove Theorem~\ref{proj lorentz s>r} it will be important to obtain good bounds for the sum \begin{equation} \label{pontes}
\sum_{\alpha \in \Lambda (m,n)} \vert z \vert^{\alpha} \vert [\alpha] \vert^{1/r}.
\end{equation}
 The strategy is to analyze smaller pieces of the sum: the \emph{tetrahedral} and the \emph{even} part, and use the bounds obtained for each of these parts to conclude something about sums which involve general monomials. As mentioned, this technique was introduced in \cite{galicer2021monomial} (see also \cite{mansilla2019thesis}) to study sets of monomial convergence. Note also that this decomposition method  was already used in Chapter~\ref{Polynomials on the Boolean cube}.

 Recall that for each $m,n$ the set of even multi indices is given by
\[
\Lambda_E(m,n) = \big\lbrace \alpha \in \Lambda(m,n) :  \alpha_i \text{ is even for every } i=1,\ldots,n \big\rbrace.
\]
Observe that for every $\alpha \in \Lambda_E(m,n)$ there is a unique $\beta \in \Lambda(m/2,n)$ such that $\alpha = 2 \beta$.
Given $\alpha \in \Lambda (M,N)$, the tetrahedral part and the even part are  defined as,
\[
\big(\alpha_{T} \big)_{i} = \begin{cases}
1 & \text{ if } \alpha_{i} \text{ is odd} \\
0 & \text{ if } \alpha_{i} \text{ is even}
\end{cases}
\quad \text{ and } \quad
\alpha_E=\alpha-\alpha_T
\,.
\]
If $0 \leq k \leq M$ is the number of odd entries in $\alpha$, then $\alpha_{T} \in \Lambda_T(k,N)$  and $\alpha_E \in \Lambda_E(M-k,N)$.
As $(\alpha_E)_i \le \alpha_i$ for every $i$ then $\alpha_E! \le \alpha!$. On the other hand since $\alpha_T! = 1$, then $\alpha_T! \alpha_E! \le \alpha!$, and thus
\begin{align}\label{eq: cardinality decomposition even-tetra}
|[\alpha]| = \frac{M!}{\alpha!} \le \frac{M!}{\alpha_T! \alpha_E!} = \frac{M!}{(M-k)! k!} \frac{k!}{\alpha_T!} \frac{(M-k)!}{\alpha_E!} =  \binom{M}{k} |[\alpha_T]||[\alpha_E]|
\le 2^M |[\alpha_T]||[\alpha_E]|.
\end{align}

\smallskip

\begin{lemma}\label{lem: tetra 1 < r < 2}
Fixed $1 < r \le 2$, $r \leq s$ and $m,n \in \NN$, for every $z \in \mathbb{C}^{n}$ we have
\[
\sum_{\alpha \in \Lambda_T(m,n)} |z^\alpha| |[\alpha]|^{\frac{1}{r}} \le n^{m/r'} \| z \|_{\ell_{r,s}}^m \frac{1}{m!^{\frac{1}{r'}}}.
\]
and
\[
\sum_{\alpha \in \Lambda_E(m,n)} |z^\alpha| |[\alpha]|^{\frac{1}{r}}
\le \| z \|_{\ell_r}^m \le \| z \|_{\ell_{r,s}}^m \log(n)^{(\frac{1}{r}-\frac{1}{s})m}.
\]
\end{lemma}

\begin{proof}
We begin with the first inequality, observing that it is obvious if $n=1$. We may then assume that $n \geq 2$. Given $\alpha \in \Lambda_T(m,n)$, note that $\alpha !=1$ and $|[\alpha]|$ is exactly $m!$.
Then,
\begin{align*}
\sum_{\alpha \in \Lambda_T(m,n)} |z^\alpha| |[\alpha]|^{\frac{1}{r}}
	 = \sum_{\alpha \in \Lambda_T(m,n)} | z^\alpha| |[\alpha]| \frac{1}{|[\alpha]|^{\frac{1}{r'}}}
     \le \Big( \sum_{k=1}^n |z_k| \Big)^{m} \frac{1}{m!^{\frac{1}{r'}}}
     \le n^{m/r'
    } \| z \|_{\ell_{r,s}}^m \frac{1}{m!^{\frac{1}{r'}}}.
\end{align*}
For the proof of the second inequality let us recall first that for each $\alpha \in \Lambda_E(m,n)$ there is a unique $\beta \in {\Lambda}(m/2,n)$ such that $\alpha = 2 \beta$ and, moreover,
\[
|[\alpha]| = \frac{m!}{\alpha_1 ! \cdots \alpha_n!} = \Big(\frac{(m/2)!}{\beta_1 ! \cdots \beta_n!}\Big)^2 \frac{m!}{(m/2)!(m/2)!} \prod_{i=1}^n \frac{\beta_i! \beta_i!}{(2\beta_i)!} \le |[\beta]|^2,
\]
where last inequality holds because $2^k \le \frac{(2k)!}{k!^2} \le 2^{2k}$, and then
\[
\frac{m!}{(m/2)!(m/2)!} \prod_{i=1}^n \frac{\beta_i! \beta_i!}{(2\beta_i)!} \le 2^m \prod_{i=1}^n \frac{1}{2^{\beta_i}} = 1.
\]
Consequently (note that, since $2/r \ge 1$, the $\ell_1$-norm bounds the $\ell_{2/r}$-norm),
\begin{multline*}
\sum_{\alpha \in \Lambda_E(m,n)} |z^\alpha| |[\alpha]|^{\frac{1}{r}}
	 \le \dis\sum_{\beta \in {\Lambda}(m/2,n)} |(z^2)^\beta| |[\beta]|^{2/r}
     =  \dis\sum_{\beta \in {\Lambda}(m/2,n)} \Big( |(z^r)^\beta| |[\beta]| \Big)^{2/r}\\
     \le \Big( \dis\sum_{\beta \in {\Lambda}(m/2,n)} |(z^r)^\beta| |[\beta]| \Big)^{2/r}
     =  \Big( \sum_{l=1}^n |z_l|^r \Big)^{m/r}
     \le \| z \|_{\ell_{r,s}}^m \log(n)^{(\frac{1}{r}-\frac{1}{s})m}.
\end{multline*}
This concludes the proof.
\end{proof}
\begin{remark}\label{rem: bound for n/klog(n)}
Fixing $a > e$, the maximum of $f(t) = \left( \frac{a}{t} \right)^t$ for $t \ge 1$ is attained at $t = a/e$ and $f(t) \le 1$. Also $f$ is increasing for $ 1< t < a/e$.
\end{remark}

\begin{lemma}\label{lem: general}
Given $1 < r \le 2$ and $r<s$, for every  $m, n \in \NN$ and every  $z \in \mathbb{C}^{n}$ we have
\[
\sum_{\alpha \in \Lambda(m,n)} |z^\alpha| |[\alpha]|^{\frac{1}{r}}
\prec_{C^m}  \| z \|_{\ell_{r,s}}^m  \begin{cases}
\left(\frac{n}{m}\right)^{m/r'} & \text{ for } m \le \frac{n}{e(\log n)^{r'(\frac1{r}-\frac1{s})}} \\
\log(n)^{m(\frac1{r}-\frac1{s})}  & \text{ else.}
\end{cases}
\]
\end{lemma}

\begin{proof}
Using \eqref{eq: cardinality decomposition even-tetra} and Lemma~\ref{lem: tetra 1 < r < 2}, we get
\begin{align*}
\sum_{\alpha \in \Lambda(m,n)} |z^\alpha| |[\alpha]|^{\frac{1}{r}}
& =  \sum_{k = 0}^m \sum_{\alpha_{T} \in \Lambda_T(k,n)} \sum_{\alpha_{E} \in \Lambda_E(m-k,n)}|z^{(\alpha_{T} + \alpha_{E})}| |[\alpha_{T} + \alpha_{E}]|^\frac{1}{r}\\
& \leq 2^{\frac{m}{r}} \sum_{k = 0}^m
\left( \sum_{\alpha_{T} \in \Lambda_T(k,n)} |z^\alpha_{T} | | [\alpha_{T}]|^{\frac{1}{r}} \right)
\left( \sum_{\alpha_{E} \in \Lambda_E(m-k,n)}|z^\alpha_{E}| |[\alpha_{E}]|^\frac{1}{r} \right) \\
& \leq 2^{\frac{m}{r}} \sum_{k = 0}^m
\left(n^{k/r'} \| z \|_{\ell_{r,s}}^k \frac{1}{k!^{\frac{1}{r'}}}  \right)
\left(  \log(n)^{(m-k)(\frac{1}{r} - \frac{1}{s})} \| z \|_{\ell_{r,s}}^{m-k} \right) \\
& \prec_{C^m} 2^{\frac{m}{r}} \| z \|_{\ell_{r,s}}^m m \dis\max_{k = 1, \ldots, m}
n^{k/r'} \frac{1}{k^{\frac{k}{r'}}}
 \log(n)^{(m-k)(\frac{1}{r} - \frac{1}{s})}.
\end{align*}
Using Remark \ref{rem: bound for n/klog(n)} with $a = \frac{n}{e(\log n)^{r'(\frac1{r}-\frac1{s})}}$, we have,
for $m \le a/e $, that the maximum above is attained at $k=m$, then
\begin{align*}
\sum_{\alpha \in \Lambda(m,n)} |z^\alpha| |[\alpha]|^{\frac{1}{r}}
& \prec_{C^m}  \| z \|_{\ell_{r,s}}^m
n^{m/r'} \frac{1}{m^{\frac{m}{r'}}}.
\end{align*}
For $m >a/e $ we can bound
$
\dis\max_{ k = 1, \ldots, m} \left( \frac{a}{k}  \right)^k \le  \left( \frac{a}{a/e} \right)^{a/e} \le e^{a/e} \le e^m$, then
\begin{align*}
\sum_{\alpha \in \Lambda(m,n)} |z^\alpha| |[\alpha]|^{\frac{1}{r}}
& \prec_{C^m} 2^{\frac{m}{r}} \| z \|_{\ell_{r,s}}^m m \dis\max_{k = 1, \ldots, m}
n^{k/r'} \frac{1}{k^{\frac{k}{r'}}}
 \log(n)^{(m-k)(\frac{1}{r} - \frac{1}{s})}\\
& \prec_{C^m} \| z \|_{\ell_{r,s}^m} \log(n)^{m(\frac{1}{r} - \frac{1}{s})} \left( \dis\max_{ k = 1, \ldots, m} \left( \frac{n}{k \log(n)^{r'(\frac{1}{r} - \frac{1}{s})}}  \right)^k \right)^{1/r'} \\
& \prec_{C^m} \| z \|_{\ell_{r,s}^m} \log(n)^{m(\frac{1}{r} - \frac{1}{s})}.
\end{align*}
 This completes the proof.
\end{proof}

We are now ready to prove Theorem \ref{proj lorentz s>r}.
\begin{proof}[Proof of Theorem \ref{proj lorentz s>r}]
Note that, since $s\geq r$, we have
$c_{\ell_{r,s}^n}(\alpha) \leq c_{\ell_{r}^n}(\alpha) \sim_{C^m} |[\alpha]|^{1/r}.$
Then,
\begin{equation}\label{desig 1er lemma 1 < r < 2}
\begin{split}
\sum_{\alpha \in J}  c_{\ell_{r,s}^n}(\alpha)|z_\alpha |
     \prec_{C^m}	\sum_{\alpha \in \Lambda(m,n)}  c_{\ell_{r,s}^n}(\alpha)|z_\alpha |
     & \prec_{C^m} \sum_{\alpha \in \Lambda(m,n)}  |z_\alpha | |[\alpha]|^{1/r}\\
	& =   \sum_{j_{m}=1}^n |z_{j_{m}}| \sum_{\ii \in \Jj(m-1,j_{m})} |z_\ii| |[(\ii,j_{m})]|^{\frac{1}{r}}  \\
    &  \prec_{C^m}  \sum_{j_{m}=1}^n |z_{j_{m}}|   \sum_{\ii \in \Jj(m-1,j_{m})} |z_\ii| |[\ii]|^\frac{1}{r},
    \end{split}
\end{equation}
where the last inequality is due to the fact that $|[(\ii,j_{m})]| \le m |[\ii]|$ for every $ \ii \in \Jj(m-1, j_{m})$.
By Lemma~\ref{lem: general}, in the range $\frac{n}{e(\log n)^{r'(\frac1{r}-\frac1{s})}}\ge m$,  we have
$$\sum_{j_{m}=1}^n |z_{j_{m}}|   \sum_{\ii \in \Jj(m-1,j_{m})} |z_\ii| |[\ii]|^\frac{1}{r} \prec_{C^m} \Vert z \Vert_{\ell_{r,s}}^{m-1} \sum_{j_{m}=1}^n |z_{j_{m}}|  \left(\frac{n}{m}\right)^{(m-1)/r'} \prec_{C^m} \Vert z \Vert_{\ell_{r,s}}^{m-1} \left(\frac{n}{m}\right)^{m/r'}.$$
The last assertions follow from  Theorem~\ref{lambda-dash} in the case of  the projection constant and  Corollary~\ref{main3A} for  the unconditional basis constant. \end{proof}

In a similar way we have the following.
\begin{corollary}
\label{proj lorentz s>r, big m}
	Let $1 < r \leq 2$ and $r \leq s$ and $J\subset \mathbb N_0^n$ of degree $m$. Then, if  $\frac{n}{e(\log n)^{r'(\frac1{r}-\frac1{s})}}< m$, we have
	$$
	\widehat{\boldsymbol{\lambda}}(\mathcal P_J(\ell_{r,s}^n)) \prec_{C^m} \log (n) ^{m(\frac{1}{r}-\frac1{s})}.
	$$
In particular,
$$
\boldsymbol{\chimon}(\mathcal P_{J}(\ell_{r,s}^n)) \prec_{C^m}  \log (n) ^{(m-1)(\frac{1}{r}-\frac1{s})} \quad \textrm{and}\quad
 {\boldsymbol{\lambda}}(\mathcal P_J(\ell_{r,s}^n)) \prec_{C^m} \log (n) ^{m(\frac{1}{r}-\frac1{s})}.$$	\end{corollary}

\medskip

\section{An attempt for an approach by interpolation}
\noindent Recall that in  the Sections \ref{proj lorentz s<r}
and \ref{proj lorentz s>r} we proved  that for $1 < r <  2$ and $1 \leq s \leq \infty$
\[
\boldsymbol{\lambda}\big( \mathcal{P}_m(\ell_{r,s}^n)\big)
\prec_{C^m} \varphi(n,m) \,|\Lambda(m,n)|^{\frac{1}{r'}}\,,
\]
where the concrete function $\varphi(n,m)$ seems an artifact of our proofs. At least for our applications on Bohr radii in Chapter \ref{Part: Bohr radii}, $\varphi(n,m)$ can be controlled whenever $m$ is small with respect to $n$ .
In fact the results from the previous Section suggest that  for all $1 < r < 2$ and $1 \leq s \leq \infty$
we have
\begin{equation*}
\boldsymbol{\lambda} \big( \mathcal{P}_m(\ell_{r,s}^n)\big)
\prec_{C^m}  \,|\Lambda(m,n)|^{\frac{1}{r'}}\,,
\end{equation*}
or in  view of   Theorem~\ref{lambda-dash} the even stronger estimate
\begin{equation} \label{want}
\widehat{\boldsymbol{\lambda}}\big( \mathcal{P}_m(\ell_{r,s}^n)\big)
\prec_{C^m}  \,|\Lambda(m,n)|^{\frac{1}{r'}}\,.
\end{equation}
Unfortunately, the answer whether this holds or not remains open. 

What we can show, is that a positive solution for two special cases, namely the proper estimates for for  $\widehat{\boldsymbol{\lambda}}\big( \mathcal{P}_m(\ell_{2,s}^n)\big)$, $s<r$
and $\widehat{\boldsymbol{\lambda}}\big( \mathcal{P}_m(\ell_{r,\infty}^n)\big)$, lead to a positive answer for the  full scale  of estimates indicated in~\eqref{want}. Observe that the analogue to the first estimate involving the projection constant instead of the polynomial projection constant $\widehat{\boldsymbol{\lambda}}$ follows directly from an application of the Kadets-Snobar
theorem from~\eqref{kadets1}.

\begin{theorem}
\label{llorentz} 
For each $m, n\in mathbb{N}$ the following statements hold true:
\begin{itemize}
\item[{\mbox(i)}] For every $1\leq s <r\leq 2$ there is a constant $\gamma>0$ such that 
\[
\widehat{\boldsymbol{\lambda}}\big(\mathcal{P}_m(\ell_{r, s}^n)\big) \leq \gamma^m \big(e^{\frac{m}{s}}
|\Lambda(m,n)|^{\frac{1}{s'}}\big)^{1-\theta}\,\widehat{\boldsymbol{\lambda}}\big(\mathcal{P}_m(\ell_{2,s}^n)\big)^{\theta}\,,
\]
where $\theta$ is given by $\frac{1}{r} = \frac{1 -\theta}{s} + \frac{\theta}{2}$.
\item[{\mbox(ii)}] For every $1 <r\leq 2$  and $r < s < \infty$ there is a~constant
$\gamma >0$ such that
\[
\widehat{\boldsymbol{\lambda}}\big(\mathcal{P}_m(\ell_{r, s}^n)\big) \leq \gamma^m \big(e^{{m}{r}}
|\Lambda(m,n)|^{\frac{1}{r'}}\big)^{1-\theta}\,
\widehat{\boldsymbol{\lambda}}\big(\mathcal{P}_m(\ell_{r, \infty}^n)\big)^{\theta}\,,
\]
where $\theta = 1 - \frac{r}{s}$.
\item[{\mbox(iii)}] For every $ 2 < r < \infty$ there is a constant $\gamma >0$ such that  
\[
\widehat{\boldsymbol{\lambda}}\big(\mathcal{P}_m(\ell_{r, \infty}^n)\big) \leq
\widehat{\boldsymbol{\lambda}}\big(\mathcal{P}_m(\ell_{2, \infty}^n)\big)^{\frac{2}{r}}\,|\Lambda(m,n)|^{1 - \frac{2}{r}}\,.
\]
\end{itemize}
\end{theorem}

We once again remark that (i) and (ii) then shows that an  affirmative answer to  the preceding conjecture
automatically leads to the desired estimates in \eqref{want} for all $1 < r < 2$ and $1 \leq s \leq \infty$ .

\smallskip

Our proof of Theorem~\ref{llorentz} is by interpolation. From \cite[Theorem 5.2.1, p.109]{BerLof76} we obtain that
for all $\theta \in (0, 1)$ and $s\in [1, \infty]$, we have
\begin{equation} \label{intfromulaA}
\big(\ell_1^n, \ell_\infty^n\big)_{\theta, s} = \ell_{r, s}^n, \quad\, n\in \mathbb{N}\,,
\end{equation}
where $\frac{1}{r} = 1- \theta$. Note that the (quasi)-norms on both sides of the preceding formula are equivalent with constants that do not depend on $n$.

Now we combine this with the so-called reiteration
formula (see \cite[Theorem 4.7.2, p.~103]{BerLof76}): for any compatible couple of Banach
spaces $(A_0, A_1)$, the formula
\begin{align} \label{Rformula}
\big[(A_0, A_1)_{\theta_0, p_0}, (A_0, A_1)_{\theta_1, p_1}\big]_{\alpha} = (A_0, A_1)_{\beta, p}
\end{align}
holds for all indices $\alpha, \theta_0, \theta_1 \in (0, 1)$ and all $p_0, p_1 \in [1, \infty]$,
except for the cases $p_0=p_1 = \infty$, where $\beta$ and $p$ are given by $\beta= (1-\alpha)\theta_0
+ \alpha \theta_1$ and $\frac{1}{p}= \frac{1-\alpha}{p_0} + \frac{\alpha}{p_1}$. Again it is important to note
that here the constants of equivalence of norms  do not depend on the  Banach
couple $(A_0, A_1)$.

Thus combining \eqref{intfromulaA} and  \eqref{Rformula} for  $(A_0, A_1) = (\ell_1^n, \ell_\infty^n)$, we get,
up to equivalences of (quasi)-norms independent of $n$, the formula
\begin{equation} \label{intformulaB}
(\ell_{r_0, s_0}^n)^{1-\theta}(\ell_{r_1, s_1}^n)^{\theta} = \ell_{r, s}^n, \quad\, n\in \mathbb{N}
\end{equation}
for all indices $\theta \in (0, 1)$, all $r_0, r_1 \in (1, \infty)$ and all $s_0, s_1 \in [1, \infty]$
except $s_0 = s_1=\infty$, where $r$ and $s$ are given by $\frac{1}{r}= \frac{1-\theta}{r_0} +
\frac{\theta}{r_1}$ and $\frac{1}{s}= \frac{1-\theta}{s_0} + \frac{\theta}{s_1}$.

\begin{proof}[Proof of Theorem~\ref{llorentz}]
(i) Given $1\leq s< r\leq 2$, we choose $\theta \in (0, 1)$ such that $\frac{1}{r} =
\frac{1-\theta}{s} + \frac{\theta}{2}$. By \eqref{intformulaB} we have the interpolation formula
\[
(\ell_s^n)^{1-\theta}(\ell_{2, s}^n)^{\theta} = \ell_{r, s}^n\,,
\]
which by Theorem \ref{lambdatheta} implies that there is a constant $C_1>0$ so that
\[
\widehat{\boldsymbol{\lambda}}\big(\mathcal{P}_m(\ell_{r,s}^n)\big) \leq C^m\,\widehat{\boldsymbol{\lambda}}\big(\mathcal{P}_m(\ell_s^n)\big)^{1-\theta}
\,\widehat{\boldsymbol{\lambda}}\big(\mathcal{P}_m(\ell_{2,s}^n)\big)^{\theta}\,.
\]
Since by Corollary \ref{pconv} and \eqref{eq: proof of lambda hat lr} for each $u\in [1, \infty)$,
\[
\widehat{\boldsymbol{\lambda}}\big(\mathcal{P}_m(\ell_u^n)\big) \leq e^{\frac{m}{u}} |\Lambda(m, n)|^{\frac{1}{u'}}
\]
the required estimate follows.

(ii) For  $r \in (1, 2]$ and  $s\in (r, \infty)$ define $\theta := 1 - \frac{r}{s} \in (0, 1)$.
By the interpolation formula \eqref{intformulaB}, we get
\[
(\ell_r^n)^{1-\theta}(\ell_{r, \infty}^n)^{\theta} = \ell_{r, s}^n\,.
\]
Combining all together, we obtain the required estimate similarly as in the case (i).

(iii) It is easy to see that for any $r\in (2, \infty)$, we have that, up to equivalence of norms depending only on $r$,
\[
\ell_{r, \infty}^n = (\ell_{2, \infty}^n)^{(\frac{2}{r})}\,.
\]
Then the desired estimate follows from Corollary \ref{pconv} applied to $X=\ell_{2, \infty}$.
\end{proof}

\smallskip

\section{Appendix: Mixed spaces}
In this section we will apply the techniques developed throughout the article to estimate the projection constants of  mixed $\ell_p$-spaces. That is, for $1 \leq r,s \leq \infty$ and $n,k \in \mathbb{N}$, we  study the asymptotics of
\[
\boldsymbol{\lambda}\big(\mathcal{P}_J(\ell_r^n (\ell_s^k)\big)\,,
\]
where $J \subset \mathbb{N}_0^n$ is an index set of degree at most $m$. Recall that $\ell_r^n (\ell_s^k))$ stands for 
the linear space of all  $n \times k$-matrices $z = (z_{ij})_{1\leq i \leq n,1\leq j \leq k}$
with entries from $\mathbb{C}$, which together with the norm 
\[
\|z\|_{\ell_r^n (\ell_s^k))} = \Big(\sum_{i=1}^{n} \Big(\sum_{j=1}^{k} |z_{ij}|^s\Big)^{\frac{r}{s}}\Big)^{\frac{1}{r}}
\]
forms a non-symmetric  Banach sequence space. We once again remark that in this section our claim on  completeness is much lower.

\subsection{Projection constants on mixed spaces}
 K\"onig, Sch\"utt and Tomczak-Jagermann \cite{konig1999projection}, studied the projection constant for the non-symmetric spaces $\ell_1^n(\ell_2^k)$ and $\ell_2^n(\ell_1^k)$. It is shown there that in the real case
\begin{align}\label{l1(l2) konig}
    \lim_{k,n\to\infty}\frac{\boldsymbol{\lambda}(\ell_1^n(\ell_2^k))}{\sqrt{nk}}=\sqrt{\frac{2}{\pi}}, \quad\textrm{while,}\quad \lim_{k,n\to\infty}\frac{\boldsymbol{\lambda}(\ell_2^n(\ell_1^k))}{\sqrt{nk}}=\frac{2}{\pi}.
\end{align}
For the complex case, Lewicki and  Masty{\l}o \cite[Corollary 5.3]{lewickimastylo} showed that for $1\le p,q\le 2 $  with $pq < 4$
\begin{align}\label{l1(l2) mietek}
\lim_{k,n\to\infty}\frac{\boldsymbol{\lambda}(\ell_p^n(\ell_q^k))}{\sqrt{nk}}=\frac{\sqrt{\pi}}{2}.
\end{align}
For the general case, and up to  absolute constants, the next proposition gives us the correct asymptotic growth of the projection constant on mixed spaces.

\begin{proposition}\label{prop:proj const mixed m=1}
Let $1 \leq p,q \leq \infty$ and $k,n \in \mathbb{N}$. Then
\[
\boldsymbol{\lambda}(\ell_p^n (\ell_q^k)) \,\,\sim_{C(p,q)} \,\, n^{\min(\frac{1}{2},\frac{1}{p})}\,\, k^{\min(\frac{1}{2},\frac{1}{q})},
\]
whenever
\begin{itemize}
\item[(i)]
$1 \leq p,q \leq 2$ and
$2 \leq p,q \leq \infty$
\item[(ii)]
$1 \leq q \leq 2 \leq p \leq \infty$
\end{itemize}
\end{proposition}

We note that the full answer for the scale $1 \leq p \leq 2 \leq q \leq \infty$ remains open.

\begin{proof}[Proof of case $(i)$]
	We use that
	\[
	\varphi_{\ell_p^n (\ell_q^k)}(nk) = n^{\frac{1}{p}}k^{\frac{1}{q}}\,,
	\]
	and that $\boldsymbol{\lambda}(\ell_p^n (\ell_q^k))$ is $\min(p,q)$-convex and $\max(p,q)$-concave. Then the result follows from \eqref{conv-conc}.
\end{proof}

\begin{proof}[Proof of case $(ii)$]
	
	For the lower estimate, note  that
	\begin{align*}
	\|\id: \ell_2^n (\ell_q^k) \to \ell_p^n (\ell_q^k)\| \leq 1
	\,\,\,\,\,\text{ and }  \,\,\,\,\,
	\|\id: \ell_p^n (\ell_q^k) \to \ell_2^n (\ell_q^k)\| \leq  n^{\frac{1}{2}-\frac{1}{p}}\,.
	\end{align*}
	Then by $(i)$ we get
	\begin{equation*}
	n^{\frac{1}{2}}k^{\frac{1}{2}} \sim \boldsymbol{\lambda}(\ell_2^n (\ell_q^k))
	= \gamma_\infty(\id_{\ell_2^n (\ell_q^k)})
	\leq
	\gamma_\infty(\id_{\ell_p^n (\ell_q^k)})
	n^{\frac{1}{2}-\frac{1}{p}}
	= \boldsymbol{\lambda}(\ell_p^n (\ell_q^k))n^{\frac{1}{2}-\frac{1}{p}}\,.
	\end{equation*}
		For the upper estimate note that
	\begin{align*}
 	d(\ell_p^n (\ell_q^k),\ell_\infty^{nk}) \le  n^{\frac{1}{p}} k^{\frac{1}{2}}.
    	\end{align*}
  Indeed, suppose that $T:\ell_q^k\to \ell_\infty^k$ is an isomorphism that attains the distance $(\|T\|\|T^{-1}\|\sim k^{1/2}).$ Define 
  $$\tilde T:\ell_p^n(\ell_q^k)\to \ell_\infty^{nk}$$  applying $T$ to each row, i.e.,
$\tilde T(a)_{i\bullet}:= T(a_{i\bullet})$ (here $x_{i\bullet}$ denotes the vector determined by the $i$-th row of a matrix $(x_{ij})$).
Then,
\begin{align*}
\|\tilde T(a)\|_\infty = \max_i\|T(a_{i\bullet})\|_\infty \le \|T\| \max_i\|a_{i\bullet}\|_q\le \|T\| \Big(\sum_i\|a_{i\bullet}\|_q^p\Big)^{1/p}. 
\end{align*}
and
\begin{align*}
\|\tilde T^{-1}(a)\|_{\ell_p^n(\ell_q^k)} =  \, \Big(\sum_i\|T^{-1}(a_{i\bullet})\|_q^p\Big)^{1/p} \le \|T^{-1}\| \Big(\sum_i\|a_{i\bullet}\|_\infty^p\Big)^{1/p} \le \|T^{-1}\| n^{1/p}\|a\|_\infty.
    	\end{align*}
Thus,
\begin{align*}
 	d(\ell_p^n (\ell_q^k),\ell_\infty^{nk}) \le \|\tilde T^{-1}\|\|\tilde T\| \le \|T^{-1}\|\|T\|  n^{1/p} \le n^{\frac{1}{p}} k^{\frac{1}{2}}.
    	\end{align*}
	Then  by \eqref{formulaX} we get
	\begin{equation*}
	\boldsymbol{\lambda}(\ell_p^n (\ell_q^k))\le d(\ell_p^n (\ell_q^k),\ell_\infty^{nk})
\sim   n^{\frac{1}{p}} k^{\frac{1}{2}}\,.\qedhere
	\end{equation*}
\end{proof}

\subsection{Polynomials on mixed spaces: first estimates}

The multi indices that define monomials on the mixed space $\ell_r^n (\ell_s^k)$ are $n\times k$ matrices of non-negative integers. We denote by $\Lambda(m,n\times k)$  the set of $n\times k$ multi indices of degree $m$, that is $\alpha \in\Lambda(m,n\times k)$ whenever  $\displaystyle\sum_{i,j}\alpha_{ij}=m$. The set of tetrahedral multi indices $\Lambda_T(m,n\times k)$ consists on all $\alpha \in\Lambda(m,n\times k)$ such that $\max_{i,j} \alpha_{ij}=1.$

If we allow constants depending arbitrarily on $m,$ then a combination  of the case $m=1$
(Proposition \ref{prop:proj const mixed m=1}) with the tools developed in the previous chapters gives us the correct asymptotic order of the projection constant for homogeneous polynomials on mixed spaces for a wide range of $1\le r,s\le \infty$.

\begin{proposition}
\label{prop: proj const mixed fixed degree}	Let $m \in \mathbb{N}$ and let $J \subset \mathbb{N}_0^n$ be such that and $\Lambda_T(m,n\times k) \subset J\subset \Lambda( \leq m,n\times k)$. Let $r,s$ such that $1\leq r \le2$ and $1\leq s \le\infty$ or such that $2\le r,s\le \infty$.Then for all  $nk\ge m$, with constants depending on $m$, we have
\[
\boldsymbol{\lambda}\big(\mathcal{P}_J(\ell_r^n (\ell_s^k))\big) \,\,\sim_{C(r,s,m)} \,\,  n^{m\min(\frac{1}{2},\frac{1}{r'})}
\,\, k^{m\min(\frac{1}{2},\frac{1}{s'})}\,.
\]
\end{proposition}

\begin{proof}
The lower bound follows from $(ii)$ of Corollary~\ref{immediateABC}. The upper bound for the case $r<2$ follows from Remark \ref{rem:dist to P(l_1)}, using that 
\[
d\big(\mathcal{P}_{J}(\ell_r^n (\ell_s^k),\mathcal{P}_{J}(\ell_1^{nk})\big)\le \big(d(\ell_r^n (\ell_s^k)),\ell_1^{nk})\big)^m\le (n^{\frac1{r'}}k^{\min\{\frac12,\frac1{s'}\}})^m \,.
\]
The case $r,s\ge 2$ is a contained in the next result.
\end{proof}

\medskip
For $2\leq r,s \leq \infty$ the space $\ell_r(\ell_s)$  is $2$-convex -- therefore we may apply  Theorem \ref{conny3} to obtain hypercontractive constants for the upper bound.

\begin{proposition}\label{proj const pols on mixed case 2-convex}
For $2\leq r,s \le\infty$, $m \in \mathbb{N}$ let $J \subset \mathbb{N}_0^n$ be such that
$\Lambda_T(m,n\times k) \subset J\subset \Lambda( \leq m,n\times k)$. Then for all  $nk\ge m$, we have	
\[
\boldsymbol{\lambda}\big(\mathcal{P}_J(\ell_r^n (\ell_s^k))\big) \,\,\sim_{C(r,s)^m} \,\, \Big(\frac{nk}{m}\Big)^{\frac{m}{2}}\,\, .
\]
\end{proposition}

\medskip
The following result gives us hypercontractive estimates for the lower bound for every $r,s$.
\begin{proposition}\label{lower bound mixed}
Let  $X=\ell_{r}^n(\ell_{s}^k)$ for  $1\le s,r\le \infty$ and $nk\ge m$. Then
\begin{align}\label{eq:lower bound mixed}
\Big(\frac{n^{\min\{\frac{1}{r'};\frac{1}{2}\}}k^{\min\{\frac{1}{s'};\frac{1}{2}\}}}{m^{\max\big\{\min\{\frac{1}{r'};\frac{1}{2}\};\min\{\frac{1}{s'};\frac{1}{2}\}\big\}}}\Big)^m
	\prec_{C^m} {\boldsymbol{\lambda}} \big(\mathcal P_{m}(\ell_{r}^n(\ell_{s}^k))\big).
	\end{align}
\end{proposition}
\begin{proof}[Proof of the case $2\le s,r\le \infty$]
This case is contained in Proposition \ref{proj const pols on mixed case 2-convex}.
\end{proof}
\begin{proof}[Proof of the case  $1 \le s< r  \le 2$]
We know that
\[
\boldsymbol{\lambda}\big(\mathcal P_{m}(\ell_{r}^{nk})\big)\le \|\id\colon \mathcal P_{m}(\ell_{r}^{nk})\to \mathcal P_{m}(X)\|\id: \boldsymbol{\lambda}\big(\mathcal P_{m}(X)\big)\cdot\|\id:\mathcal P_{m}(X)\to \mathcal P_{m}(\ell_{r}^{nk})\|.
\]
For $1\le s<r\le 2$, we have 
\[
\|\id: \mathcal P_{m}(X)\to \mathcal P_{m}(\ell_{r}^{nk})\|\le \|\id: \ell_{r}^{nk}\to \ell_{r}^n(\ell_{s}^k)\|^m\le k^{m(\frac{1}{s}-\frac{1}{r})}
\]
and
\[
\|\id:\mathcal P_{m}(\ell_{r}^{nk})\to \mathcal P_{m}(X)\|\le \|\id:\ell_{r}^n(\ell_{s}^k)\to \ell_{r}^{nk}\|^m\le1.
\]
Thus we conclude,
\begin{equation*}
{\boldsymbol{\lambda}}\big(\mathcal P_{m}(\ell_{r}^n(\ell_{s}^k))\big)\ge
\boldsymbol{\lambda}\big(\mathcal P_{m}(\ell_{r}^{nk})\big)k^{m(\frac1{r}-\frac{1}{s})}\sim_{C^m}
\Big(\frac{n}{m}\Big)^{\frac{m}{r'}}k^{\frac{m}{s'}}. \qedhere
\end{equation*}
\end{proof}	
	
\begin{proof}[Proof of the case $1\le r<s\le 2$]
Follows as the case  $1 \le s< r  \le 2$, but factorizing through  $\mathcal P_{m}(\ell_{s}^{nk})$, that is, using that we know that
\[
\boldsymbol{\lambda}(\mathcal P_{m}(\ell_{s}^{nk}))\le \|\id:\mathcal P_{m}(\ell_{s}^{nk})\to \mathcal P_{m}(X)\|\,
\boldsymbol{\lambda}\big(\mathcal P_{m}(X)\big)\,\|\id:\mathcal P_{m}(X)\to \mathcal P_{m}(\ell_{s}^{nk})\|\,,
\]
we conclude that
\begin{equation*}
n^{\frac{m}{r'}}\Big(\frac{k}{m}\Big)^{\frac{m}{s'}}
\prec_{C^m} {\boldsymbol{\lambda}}\big(\mathcal P_{m}(\ell_{r}^n(\ell_{s}^k))\big)\,. 
\qedhere
\end{equation*} 
\end{proof}

\begin{proof}[Proof of the case  $1 \le s\le 2< r$] For $1\le s\le 2<r$, we have 
\[
\|\id: \mathcal P_{m}(X)\to \mathcal P_{m}(\ell_{r}^n(\ell_{2}^k))\|\le k^{m(\frac1{s}-\frac1{2})}
\]
and
\[
\|\id:\mathcal P_{m}(\ell_{r}^n(\ell_{2}^k))\to \mathcal P_{m}(X)\|\le 1.
\]
Then by Proposition \ref{proj const pols on mixed case 2-convex},
\begin{align*}
\Big(\frac{nk}{m}\Big)^{\frac{m}{2}} &
\sim_{C^m}\boldsymbol{\lambda}\big(\mathcal P_{m}(\ell_{r}^n(\ell_{2}^k))\big) \\
& \le \big\|\id:\mathcal P_{m}\big(\mathcal P_{m}(\ell_{r}^n(\ell_{2}^k))\big)\to \mathcal P_{m}(X)\big\|\,
\boldsymbol{\lambda}\big(\mathcal P_{m}(X)\big)\,\|\id:\mathcal P_{m}(X)\to \mathcal P_{m}(\ell_{r}^n(\ell_{2}^k))\|\\
& \le k^{m(\frac1{s}-\frac1{2})}\boldsymbol{\lambda}\big(\mathcal{P}_m(X)\big)\,.
\end{align*}
Therefore we have,
\begin{equation*}
\Big(\frac{n}{m}\Big)^{\frac{m}{2}}k^{\frac{m}{s'}} \prec_{C^m} {\boldsymbol{\lambda}}\big(\mathcal P_{m}(\ell_{r}^n(\ell_{s}^k))\big). 
\qedhere
\end{equation*}
\end{proof}	

\begin{proof}[Proof of the case  $1 \le r\le 2< s$]
    The proof is similar to the case $1\le s\le 2<r$:  considering that  
\[
\|\id: \mathcal P_{m}(X)\to \mathcal P_{m}(\ell_{2}^n(\ell_{s}^k))\|\le n^{m(\frac1{r}-\frac1{2})}
\]
and
\[
\|\id:\mathcal P_{m}(\ell_{2}^n(\ell_{s}^k))\to \mathcal P_{m}(X)\|\le 1\,,
\]
we use Proposition \ref{proj const pols on mixed case 2-convex} to obtain
\begin{equation*}
\Big(\frac{k}{m}\Big)^{\frac{m}{2}}n^{\frac{m}{r'}} \prec_{C^m}
{\boldsymbol{\lambda}}\big(\mathcal P_{m}(\ell_{r}^n(\ell_{s}^k))\big). \qedhere
\end{equation*}
\end{proof}

\smallskip

\subsection{The $m$-homogeneous case: The upper bound for $1 \le s< r  \le 2$.}

Looking at the results  of the preceding section it is natural to ask whether the  converse to inequality \eqref{eq:lower bound mixed} hold.
	Recall that	Proposition \ref{proj const pols on mixed case 2-convex} provides a positive answer in the case $2\le r,s\le\infty.$ In the present section we obtain a positive answer in the  range  $1\le s\le r\le 2$. This may be seen  as another notable  example of how the techniques involving the polynomial projection constant allow us to obtain satisfying  estimates for
the projection constant of spaces of polynomial (this time on non-symmetric Banach sequence spaces).

\begin{theorem}\label{theo:mixed s<r<2}
	Let  
$1\le s<r\le 2$ and $n\ge m^2$. Then	\begin{align}
\boldsymbol{\lambda}\big(\mathcal{P}_m(\ell_r^{n}(\ell_s^{k}))\big)\sim_{C^m}  \Big(\frac{n}{m}\Big)^{\frac{m}{r'}}k^{\frac{m}{s'}}\,.
\end{align}
\end{theorem}
To prove this theorem we are going to decompose $\Lambda(n\times k,m)$  into certain subsets of  multi indices which we now define. 
Given $(\alpha_{ij})=\alpha\in\Lambda(m,n\times k)$, let $\alpha_{i\bullet}$ denotes its $i$-th row and   $\alpha_{\bullet j}$ its $j$-th column. Thus, in particular,  $|\alpha_{i\bullet }|:=\sum_j \alpha_{ij}$.
Moreover, we define, for $1\le L\le m,$ the index sets
$$\Lambda^{L\times k}(m,n\times k):=\left\{ \alpha \in \Lambda(m,n\times k) : \;\; \vert\{ i : |\alpha_{i\bullet}| > 0 \} \vert = L \right\}; $$
in other terms,  $\alpha\in \Lambda^{L\times k}(m,n\times k)$ whenever it has exactly $L$ non-zero rows.
Given a Banach space $X=(\mathbb C^{n\times k},\|\cdot\|)$, we write
$$\mathcal{P}_{m,L\times k}(X):=  \mathcal{ P}_{\Lambda^{L\times k}(m,n\times k)}(X)\,,$$
 i.e., all  polynomials of the form 
 $$P(z)= \displaystyle\sum_{\alpha\in\Lambda^{L\times k}(m,n\times k)}c_\alpha z^\alpha\,, \quad z \in X.$$
In the following we in several separate lemmas 
 bound the polynomial projection constants of the Banach space   $\mathcal P_{m,L\times k}(\ell_{r,s}^n)$ when either $L$ is similar to $m$ or  $L$ is small. We start estimating the cardinality of $\Lambda^{L\times k}(m,n\times k).$
\begin{lemma}\label{cardinal J^L rows}
	For $nk\ge m$,
	\begin{align*}
	|\Lambda^{L\times k}(m,n\times k)|\prec_{C^m} \Big(\frac{n}{L}\Big)^{L}k^{m}. \end{align*}
\end{lemma}
\begin{proof}
	Any  $\alpha\in\Lambda^{L\times k}(m,n\times k)$ may be written as a sum of a tetrahedral index $\beta\in \Lambda^{L\times k}(L,n\times k)$ (that is, $|\beta|=L$ and for each row $|\beta_{i\bullet}|\le 1$) and another index $\gamma\in\Lambda^{L\times k}(m-L,L\times k) $ whose support is contained in the rows such that $|\beta_{i\bullet}|=1$, i.e., $\gamma$ uses (at most) the same $L$ rows as $\beta$.
		Thus,
	\begin{align*}
	|\Lambda^{L\times k}(m,n\times k)|&
\le |\Lambda^{L\times k}(L,n\times k)|\,|\Lambda^{L\times k}(m-L,L\times k)|
\\
	&
\prec_{C^m} \Big(1+\frac{nk}{L}\Big)^{L}\,\Big(1+\frac{Lk}{m-L}\Big)^{m-L}
\\
	&
\prec_{C^m}\Big(\frac{n}{L}\Big)^{L}k^{m}\left(1+\frac{L}{m-L}\right)^{m-L}\prec_{C^m}\Big(\frac{n}{L}\Big)^{L}k^{m}.
\qedhere
		\end{align*}
	
\end{proof}

\begin{lemma}\label{lemma: not tetra mixed s<r<2}
	Let  $X=\ell_{r}^n(\ell_{s}^k)$ for $1\le s<r\le 2$. Let $J\subset \Lambda^{L\times k}(m,n)$  and suppose that $L\le m-m^a$ 
for some $0<a<1$. Then
	\begin{align}
	\widehat{\lambda}\big(\mathcal P_{J}(\ell_{r}^n(\ell_{s}^k))\big)\prec_{C^m} \Big(\frac{n}{m}\Big)^{\frac{m}{r'}}k^{\frac{m}{s'}}\,\,\,	\max\left\{\Big(\frac{m^2}{n}\Big)^{\frac{m}{r'}}\left(\frac{1}{m^{\frac{m}{s'}}}
\right),\Big(\frac{m^2}{n}\Big)^{m(\frac1{s}-\frac{1}{r})}\left(\frac{1}{n^{\frac{m^a}{s'}}}\right)\right\}.
	\end{align}
	In particular, for $n\ge m^2$,
	$$
	\widehat{\lambda} \big(\mathcal P_{\Lambda^{L\times k}(m,n\times k)}(\ell_{r}^n(\ell_{s}^k))\big)\prec_{C^m} \Big(\frac{n}{m}\Big)^{\frac{m}{r'}}k^{\frac{m}{s'}}.
	$$
\end{lemma}
\begin{proof}
	Note that by \eqref{dineen} we have $c_X(\alpha)\le c_{\ell_s^{nk}}(\alpha)=\left(\frac{m^m}{\alpha^\alpha}\right)^{1/s}$.
		Then, given $z \in \mathbb{C}^{n \times k}$,
		\begin{align*}
\sum_{\alpha \in J}  c_{X}(\alpha) |z^\alpha| & \le	\sum_{\alpha \in \Lambda^{L\times k}(m,n\times k)}  c_{X}(\alpha) |z^\alpha| \\
&\prec_{C^m}
	\sum_{\alpha \in \Lambda^{L\times k}(m,n\times k)}  \frac{m^{m/s}}{\alpha^{\alpha/s}} |z^\alpha| \prec_{C^m} |\Lambda^{L\times k}(m,n\times k)|^{1/s'}
	\left(\sum_{\alpha \in \Lambda^{L\times k}(m,n\times k)}  \frac{m!}{\alpha!} |z^{\alpha s}|\right)^{1/s}.
	\end{align*}
	Note that 
	$$
	\Lambda^{L\times k}(m,n\times k)\subset \bigcup_{A\subset \{1,\dots,n\}, |A|=L} \Lambda(m,A\times \{1,\dots,k\}),
	$$
	thus
	\begin{align*}
	\left(\sum_{\alpha \in \Lambda^{L\times k}(m,n\times k)}  \frac{m!}{\alpha!} |z^{\alpha s}|\right)^{1/s}
&
\le \left(\sum_{|A|=L}\;\sum_{\alpha\in \Lambda(m,A\times \{1,\dots,k\})}\frac{m!}{\alpha!} |z^{\alpha s}|\right)^{1/s} 
\\&
=\left(\sum_{|A|=L}\;\|(z_{ij})_{i\in A,1\le j\le k}\|_{\ell_s^{Lk}}^{sm}\right)^{1/s}
\\
	&\le
	 L^{m(\frac1{s}-\frac1{r})}
	\left(\sum_{|A|=L}\;\|(z_{ij})_{i\in A,1\le j\le k}\|_{\ell_r^L(\ell_s^{k})}^{sm}\right)^{1/s}
\\&
=L^{m(\frac1{s}-\frac1{r})}
\left(\sum_{|A|=L}\left(\sum_{i\in A} \Big(\sum_{j\le k}|z_{ij}|^s\Big)^{r/s}\right)^{\frac{sm}{r}}\right)^{1/s} \\
&\le L^{m(\frac1{s}-\frac1{r})}
\left(\sum_{|A|=L}\sum_{i\in A} \Big(\sum_{j\le k}|z_{ij}|^s\Big)^{r/s}\right)^{m/r}
\le L^{m(\frac1{s}-\frac1{r})}\|z\|_{\ell_r^n(\ell_s^{k})}^m.
	\end{align*}
Therefore, by Lemma \ref{cardinal J^L rows},
	\begin{align*}
	\sum_{\alpha \in \Lambda^{L\times k}(m,n\times k)}  c_{X}(\alpha) |z^\alpha| & \prec_{C^m} \left(\frac{n}{L}\right)^{\frac{L}{s'}}k^{\frac{m}{s'}}L^{m(\frac1{s}-\frac1{r})}\|z\|_{\ell_r^n(\ell_s^{k})}^m.
	\end{align*}
	Let $f(L):=\log\Big(\left(\frac{n}{L}\right)^{\frac{L}{s'}}L^{m(\frac1{s}-\frac1{r})}\Big)$ then either:
	\begin{itemize}
		\item[(i)] $f$ is increasing in $(1,m-m^a)$, or,
		\item[(ii)] $f$ has a global maximum which is in $(1,m-m^a)$.
	\end{itemize}
	In case (i), replacing $L$ by $m-m^a$,
	\begin{align*}
		\sum_{\alpha \in \Lambda^{L\times k}(m,n\times k)}  c_{X}(\alpha) |z^\alpha| \prec_{C^m} \Big(\frac{n}{m}\Big)^{m/r'}k^{m/s'}
		\Big(\frac{m^2}{n}\Big)^{m(\frac1{s}-\frac{1}{r})}\left(\frac{1}{n^{\frac{m^a}{s'}}}\right).
	\end{align*}
In case (ii), the maximum of $f$ is attained at $L_0\in (1,m-m^a)$ satisfying $\frac1{s'}\log(\frac{n}{eL_0})+\frac{1}{L_0}=0$. Note that
$$
\left(\frac{n}{L_0}\right)^{L_0}=e^{L_0}e^{-s'}\prec_{C^m} 1,
$$
and thus, for $L\in[1,m-m^a]$,
\begin{align*}
\left(\frac{n}{L}\right)^{\frac{L}{s'}}k^{\frac{m}{s'}}L^{m(\frac1{s}-\frac1{r})} &\le
\left(\frac{n}{L_0}\right)^{\frac{L_0}{s'}}L_0^{m(\frac1{s}-\frac1{r})}k^{\frac{m}{s'}}
\prec_{C^m}k^{\frac{m}{s'}}L_0^{m(\frac1{s}-\frac1{r})}\\
&\le k^{\frac{m}{s'}}m^{m(\frac1{s}-\frac1{r})} = \Big(\frac{n}{m}\Big)^{\frac{m}{r'}}k^{\frac{m}{s'}}
\Big(\frac{m^2}{n}\Big)^{\frac{m}{r'}}\left(\frac{1}{m^{\frac{m}{s'}}}\right). \qedhere
\end{align*}
\end{proof}

\begin{lemma}\label{lemma: big L, mixed}
Let  $X=\ell_{r}^n(\ell_{s}^k)$ for $1\le s<r\le 2$. Suppose that $J\subset \Lambda^{L\times k}(m,n\times k)$, $n\ge m$  and let $L\ge m-m^{s/r}$. Then
\[
\widehat{\lambda}\big(\mathcal{P}_{J}
(\ell_{r}^n(\ell_{s}^k))\big)\prec_{C^m} \Big(\frac{n}{m}\Big)^{\frac{m}{r'}}\,.
\]
\end{lemma}

\begin{proof}
	 If $\card\{i:|\alpha_{i\bullet}|=1\}=L-l$, then  the remaining rows satisfy that $|\alpha_{i\bullet}|\ge 2$, and then
\begin{align}
\label{number of 1's mixed}
L-l+2l\le m, \quad or,\quad l\le m-L.
\end{align}	
	Thus $\card\{i:|\alpha_{i\bullet}|=1\}\ge 2L-m$.
	
	By $\alpha^*$ we denote the multi index which is obtained by the decreasing rearrangement of the rows of $\alpha$, i.e., there is a permutation $\sigma$ of $\{1,\dots,n\}$ such that $(\alpha^*)_{i\bullet}=\alpha_{\sigma(i)\bullet}$ and $|(\alpha^*)_{i\bullet}|\ge |(\alpha^*)_{(i+1)\bullet}|$ for every $i=1,\dots,n$. Moreover, if $|\alpha_{i\bullet}|=|\alpha_{i'\bullet}|$ and $i<i'$ then $\sigma(i)<\sigma(i')$.
		Note then that
	\begin{itemize}
		\item $|(\alpha^*)_{i\bullet}|=0$ for $i>L$,
		\item $|(\alpha^*)_{i\bullet}|=1$ for $m-L<i\le L$,
		\item if $i<m-L$ then by \eqref{number of 1's mixed},  $|(\alpha^*)_{i\bullet}|\le m- (2L-m)=2(m-L)$.
	\end{itemize}	
		Take now $\tilde z=\frac{\alpha^{1/r}}{m^{1/r}}$. Then
	\begin{align*}
	\|\tilde z\|_{X}&=\frac1{m^{1/r}}\left(\sum_{i=1}^n\|(\alpha_{i\bullet})^{1/r}\|_{\ell_s^k}^r\right)^{1/r} \le \frac1{m^{1/r}}\left(\sum_{i=1}^{m-L}\|((\alpha^*)_{i\bullet})^{1/r}\|_{\ell_s^k}^r + \sum_{i=m-L+1}^L 1\right)^{1/r}, \\
	&	\le \frac1{m^{1/r}}\left(\sum_{i=1}^{m-L}(2(m-L))^{r(\frac1{s}-\frac1{r})}\|((\alpha^*)_{i\bullet})^{1/r}\|_{\ell_r^k}^r + (2L-m)\right)^{1/r} \\
	&= \frac1{m^{1/r}}\left((2(m-L))^{r(\frac1{s}-\frac1{r})}\left(\sum_{i=1}^{m-L}\sum_{j=1}^k(\alpha^*)_{ij}\right) + (2L-m)\right)^{1/r} \\
	& 	= \frac1{m^{1/r}}\left((2(m-L))^{r(\frac1{s}-\frac1{r})}\left(m-(2L-m)\right) + (2L-m)\right)^{1/r} \\
	&= \left(2^{\frac{r}{s}}\frac{(m-L)^{\frac{r}{s}}}{m} + \frac{(2L-m)}{m}\right)^{1/r}   \prec \frac{(m-L)^{\frac1{s}}}{m^{\frac1{r}}}.
	\end{align*}
	Hence, for $m-L\le m^{s/r},$ we have $\|\tilde z\|_X\prec 1$.
	This implies that,
	$$
	c_X(\alpha)\prec_{C^m} \frac1{\tilde z^\alpha} =\left(\frac{m^m}{\alpha^\alpha}\right)^{1/r}\prec_{C^m} |[\alpha]|^{1/r}.
	$$
	Moreover, note that
	 	$$
	 \alpha^\alpha=(\alpha^*)^{\alpha^*}=\prod_{i\le m-L}\prod_{j\le k}((\alpha^*)_{ij})^{(\alpha^*)_{ij}} \le (2(m-L))^{2(m-L)} 
	 \le (2m^{s/r})^{2m^{s/r}} \prec_{C^m} 1.
	 	$$
		Therefore,
	\begin{align*}
	\sum_{\alpha \in J}  c_{X}(\alpha) |z^\alpha|
&
\le
\sum_{\alpha \in \Lambda^{L\times k}(m,n\times k)}  c_{X}(\alpha) |z^\alpha|  
\\&
\prec_{C^m}
\sum_{\alpha \in \Lambda^{L\times k}(m,n\times k)} \frac1{m^{m/r'}} \frac{m^m}{\alpha^\alpha}\alpha^{\alpha/r'} |z^\alpha| 
\\
& 
\prec_{C^m} \frac1{m^{m/r'}}
\sum_{\alpha \in \Lambda^{L\times k}(m,n\times k)}  \frac{m!}{\alpha!} |z^\alpha|  
\\
& 
\prec_{C^m} \frac1{m^{m/r'}} (\|\id:\ell_1^n\to\ell_r^n\|\cdot \|\id:\ell_1^k\to\ell_s^k\|)^m \le \frac{n^{m/r'}k^{m/s'}}{m^{m/r'}}.
\qedhere
	\end{align*}
\end{proof}

\smallskip

\begin{proof}[Proof of Theorem \ref{theo:mixed s<r<2}]
  The lower bound follows from Proposition \ref{lower bound mixed}.
  The upper bound for the polynomial projection constant follows from Lemma \ref{lemma: not tetra mixed s<r<2}, Lemma \ref{lemma: big L, mixed} and the fact that
  $$\widehat{\boldsymbol{\lambda}}\big(\mathcal P_{J}(\ell_{r}^n(\ell_{s}^k))\big)\le \sum_{L=0}^m\widehat{\boldsymbol{\lambda}}
  \big(\mathcal P_{m,L\times k}(\ell_{r}^n(\ell_{s}^k))\big)\,.
  $$
  Finally, by Theorem \ref{lambda-dash}
  the polynomial projection constant bounds the projection constant.
\end{proof}

\medskip


\chapter{Bohr radii} \label{Part: Bohr radii}

In this final chapter we apply our  previous results to the study of multivariate  Bohr radii.
 Given a (complex) $n$-dimensional Banach  lattice  $X_n = (\CC^n, \| \cdot \|)$, the Bohr radius $K(B_{X_n})$ of the open unit ball $B_{X_n}$ of $X_n$ is the  supremum over all possible $0 < r <1$ such that
\[
\sup_{z \in r B_{X_n}}  \sum_{\alpha \in \mathbb{N}_0^{(\mathbb{N})}} \Big|\frac{\partial^\alpha f(0)}{\alpha!} z^\alpha\Big| \le \sup_{z \in B_{X_n}} |f(z)|\,,
\]
for every holomorphic $f: B_{X_n} \to \mathbb{C}$.  As usual, the Banach space of all bounded holomorphic functions $f: B_{X_n} \to \mathbb{C}$ is denoted by $H_\infty(B_{X_n})$, and the unique monomial coefficients
$\partial^\alpha f(0)/\alpha!$ of such functions are often  abbreviated by $c_\alpha(f)$.

Recall from Bohr's famous power series theorem that for the unit ball $\mathbb{D}$ of the Banach space $\mathbb{C}$ we have
\begin{equation}\label{orgB}
  K(\mathbb{D}) = \frac{1}{3}\,.
\end{equation}
In fact, we here are going to focus on a more general setting. Fixing  an $n$-dimensional  Banach lattice  $X_n = (\CC^n, \| \cdot \|)$
together with an index set $J \subset \NN_0^{n}$, we consider
the closed subspace
\[
H_{\infty}^{J}(B_{X_n}) :=  \left\lbrace f \in H^\infty(B_{X_n}) : c_\alpha(f) = 0 \text{ for } \alpha \notin J \right\rbrace
\]
of
$H_\infty(B_{X_n})$, and define
\begin{equation}\label{definitionB}
    K(B_{X_n},J) := \sup \Big\{ 0 < r < 1 \colon \sup_{z \in r  B_{X_n}}  \sum_{\alpha \in J} \Big|\frac{\partial^\alpha f(0)}{\alpha!} z^\alpha\Big| \le \|f\|_\infty \,\,\, \text{for all $f \in H^J_\infty (B_{X_n})$} \Big\}\,,
       \end{equation}
the Bohr radius of $B_{X_n}$ with respect to the index set $J$. Clearly,
\[
K(B_{X_n}) = K(B_{X_n},  \mathbb{N}_0^n) \,.
\]
For two index sets $J,J' \subset \mathbb{N}_0^n$ satisfying $J \subset J'$ we obviously have
  \begin{equation}\label{rem: bohr radii monotony}
   K(B_{X_n},J') \le K(B_{X_n},J)\,.
 \end{equation}
  Note also that, if $X$ is a Banach sequence space and $J \subset \mathbb{N}_0^{(\mathbb{N})}$, then
it is immediate that  for any $n$
\[
 K(B_{X_n},J) =  K(B_{X_n},J^n)\,,
\]
where    $J^n=J \cap \mathbb{N}_0^n$ collects all indices in $J$ of length $n$.

In order to see  the proof of the following result one only has to copy the argument used for the special case $X_n = \ell_\infty^n$ in \cite[Proposition 19.8]{defant2019libro}: For each  Banach lattice  ${X_n} = (\CC^n, \| \cdot \|)$ and index set  $J \subset \NN_0^n$
    \begin{equation}\label{lowinfty}
      K(B_{\ell_\infty^n},J) \leq K(B_{X_n},J)\,.
    \end{equation}
We write
$$
K_m(B_{X_n}):= K(B_{X_n},\Lambda(m,n))\,,
$$
and call this number  the $m$-homogeneous Bohr radius of $B_{X_n}$, since in this case $J=\Lambda(m,n)$, in \eqref{definitionB} we only consider
$m$-homogeneous  polynomials $f \in \mathcal{P}_m( X_n) = H^{\Lambda(m,n)}_\infty (B_{X_n})\,.$

Fixing 
an infinite dimensional  Banach sequence lattice $X$ and an index set    $J \subset \mathbb{N}_0^{(\mathbb{N})}$, 
  we in  this chapter   are mainly  interested in studying the asymptotic behaviour of $K(B_{X_n},J)$,  whenever  the dimension $n$ tends to infinity - as before,   $X_n$ denotes the $n^{\text{th}}$ section of $X$ and we  define
  $K(B_{X_n},J) = K(B_{X_n},J^n) $. For  abstract Banach sequence spaces $X$ with  a certain geometrical structure (e.g., $2$-convex spaces), or concrete spaces $X$ (e.g., Lorentz spaces
$\ell_{r,s}$) we focus on various index sets $J$ in $\mathbb{N}_0^{(\mathbb{N})}$ with a particular  structure (e.g., index sets
which consist of indices  of degree at most   $m$, or index sets  formed by tetrahedral indices, or sets of  multi indices generated by the prime number decompositions $n = \mathfrak{p}^\alpha$ of certain subsets of natural numbers).

In order to illustrate  what we aim for, we recall two results which might be seen as a motivation for what is coming. From \cite{defant2011bohr} (see also \cite[Theorem 19.1]{defant2019libro}) we know that
for each $1  \leq r \leq \infty$
\[
K(B_{\ell_r^n}) \sim \Big(\frac{\log n}{n}\Big)^{\min \big\{ \frac{1}{2},\frac{1}{r'} \big\}}\,.
\]
This result is intimately connected with the central topic of this text -- the asymptotic determination of
$\boldsymbol{\lambda}(\mathcal{P}_J(X_n))$ and $\boldsymbol{\chimon}(\mathcal{P}_J(X_n))$. And in fact, in Theorem~\ref{thm: main bohr radii} we are going to combine
many of the results produced so far, in order to characterize the asymptotic decay of $K(B_{\ell_{r,s}^n}, J)$ (as $n$ tends to $\infty$) for (almost) all possible $r,s$ and a rich class of index sets $J\subset \mathbb{N}_0^{(\mathbb{N})}$.

To see another result, intimately connected with what we have seen in Chapter \ref{Part: Dirichlet polynomials and polynomials on the Boolean cube} (dealing with  the   projection constant
of spaces of Dirichlet polynomials), we mention that
\[
K(B_{\ell_\infty^n}, \Delta(x)) \sim \frac{(\log x)^\frac{1}{4}}{x^\frac{1}{8}}\,,
\]
where  $\Delta(x) = \{  \alpha \in \mathbb{N}_0^{\pi(x)}\colon 1 \leq \mathfrak{p}^\alpha \leq x\}$ for $x \in \mathbb{N}$  and $\pi(x)$  counts the number of  primes $\mathfrak{p} \leq x$
(see again Chapter~\ref{chapter: general preliminaries}). In view of what we intend to explain, this result from
\cite[Theorem 2.1]{carando2014Dirichlet}
should 
be seen as a sort of counterpart of Theorem~\ref{harpo} and
\eqref{annals}.

Let us finally note that our results add to the (really) vast recent literature on various aspects of one dimensional and multidimensional
Bohr radii (for the one dimensional case see, e.g., \cite{alkhaleefah2019bohr,
beneteau2004remarks,
bhowmik2018bohr,
bombieri2004remark,
ismagilov2020sharp,
kayumov2017bohr,
kayumov2018bohr,
paulsen2002bohr,
paulsen2004bohr}
, and for the multivariate case 
\cite{
bayart2012maximum,
bayart2014bohr,
boas2000majorant,
defant2011bohr,
defant2011bohnenblust,
defant2003bohr,
defant2018bohr,
khavinson1997bohr}).

   \smallskip

\section{Bohr radii vs unconditionality}

The following result shows that the study of multidimensional Bohr radii in  Banach spaces is intimately linked with the study of unconditionality of  spaces of multivariate polynomials on Banach spaces.

\begin{proposition}\label{prop: Km vs uncond in Pm}
Let ${X_n} = (\CC^n, \| \cdot \|)$ be a  Banach lattice and  $J \subset \NN_0^n$ an index set. Then for every $m \in \mathbb{N}$
\[
K_m(B_{X_n},J) = \frac{1}{\boldsymbol{\chimon}(\Pp_{J_m}( X_n))^{1/m}}\,.
\]
\end{proposition}

\begin{proof}
    		Let $P \in \Pp_{J_m}( X_n)$ with monomial coefficients $c_\alpha(P), \,\alpha \in J_m$. Then for any set $( \theta_\alpha )_{\alpha \in J_m}$ of signs, we have
		\begin{align*}
		\big\| \dis\sum_{\alpha \in J_m} \theta_\alpha c_\alpha(P) z^\alpha \big\|_{\Pp_m( {X_n})}
		& \le \big\| \dis\sum_{\alpha \in J_m} |c_\alpha(P) z^\alpha| \big\|_{\Pp_m( {X_n})} \\
		& = \frac{1}{K_m(B_{X_n},J)^m} \big\| \dis\sum_{\alpha \in J_m} |c_\alpha(P)  (K_m(B_{X_n},J)z)^\alpha| \big\|_{\Pp_m( {X_n})}   \\
		& \le \frac{1}{K_m(B_{X_n},J)^m} \| P \|_{\PP_m( {X_n})},
		\end{align*}
		which leads to $ \boldsymbol{\chimon}(\Pp_{J_m}( {X_n}))^{1/m} \le \frac{1}{K_m(B_{X_n},J)}$. Conversely, take $\theta_\alpha = \frac{\overline{c_\alpha(P)}}{|c_\alpha(P)|}$. Then we have
		\begin{align*}
		\big\| \dis\sum_{\alpha \in J_m} |c_\alpha(P) z^\alpha| \big\|_{\PP_m( {X_n})}
		 = \big\| \dis\sum_{\alpha \in J_m} \theta_\alpha c_\alpha(P) z^\alpha \big\|_{\PP_m( X)}
		 \le \boldsymbol{\chimon}(\Pp_{J_m}( {X_n})) \| P \|_{\PP_m( {X_n})},
		\end{align*}
		or equivalently,
		\[
		\dis\sup_{z \in B_{X_n}} \dis\sum_{\alpha \in J_m} \Big|c_{\alpha} \left( \frac{1}{\boldsymbol{\chimon}(\Pp_{J_m}( {X_n}))^{1/m}} z\right) ^\alpha\Big|  \le \| P \|_{\PP_{J_m}( {X_n})},
		\]
		which implies $ \frac{1}{ \boldsymbol{\chimon}(\Pp_{J_m}( {X_n}))^{1/m} } \le K_m(B_{X_n},J) $.
	\end{proof}

\smallskip

\smallskip

For the full index set $J \subset \NN_0^n$ the following result  was first observed in \cite{defant2003bohr} (see also \cite[Proposition~19.4]{defant2019libro}). Since this (almost obvious)
extension is crucial for our purposes, we for the sake of completeness include its proof.

	\begin{theorem}\label{thm: Bohr vs unc}
Let ${X_n} = (\CC^n, \| \cdot \|)$ be a  Banach lattice and  $J \subset \NN_0^n$ an index set. Then
		\[
		\frac{1}{3} \inf_{ m \in \mathbb{N} } K_m(B_{X_n},J)  \le K(B_{X_n},J)  \le \inf_{m \in \mathbb{N}} K_m(B_{X_n},J).
		\]
	\end{theorem}

The proof needs a lemma which, for holomorphic functions defined on the unit disc, is due to F.~Wiener (in \cite{bohr1914theorem} Bohr includes a proof of this result for which he credits Wiener, see also \cite[Lemma 8.4]{defant2019libro}).

\begin{lemma}[Wiener's Lemma]\label{lem: Wiener}
Given $f $ an holomorphic function on the open unit disc $\DD$ with $\| f \|_{\DD} \le 1 $ and Taylor expansion $ f(z)= \sum_{n \ge 0} c_n z^n$ it holds $ |c_n | \le 1 - | c_0|^2 $ for every $n \ge 1$.
\end{lemma}

The following extension for holomorphic functions on unit balls of finite dimensional lattices is almost immediate and follows the path given in \cite[Chapter 8]{defant2019libro}.

\begin{lemma}[Generalized Wiener's Lemma]\label{lem: Wiener general}
  Let ${X_n} = (\CC^n, \| \cdot \|)$ be a  Banach lattice, and $f: B_{X_n} \to \mathbb{C}$ a holomorphic function such that $\| f \|_{B_{X_n}} \le 1$ and with Taylor series  $f = c_0 + \dis\sum_{m=1}^\infty P_m $. Then for every $m \in \zN$ it holds
  \begin{equation*}
      	\| P_m \|_{\PP_m( X)} \le 1 - |c_0|^2\,.
	\end{equation*}
		\end{lemma}

\begin{proof}
Fix $z \in B_{X_n}$ and define $ g : \DD \to \CC$ the holomorphic function given by
\[
g(u) = f(u \cdot z) =  c_0 + \sum_{m \ge 1} P_m(z) u^m.
\]
Notice that $\| g \|_\DD \le 1 $ as $\| f \|_{B_{X_n}} \le 1$. Then, given $m \in \NN$, by Wiener's Lemma in \ref{lem: Wiener} it follows
\[
|P_m(z)| \le 1 - |c _0|^2, 
\]
and taking supremum over $z \in B_{X_n}$ the proof is complete.
  \end{proof}

  \smallskip
	
	\begin{proof}[Proof of Theorem~\ref{thm: Bohr vs unc}]
		From inequality \eqref{rem: bohr radii monotony} we know $K(B_{X_n},J) \le \inf_{m \in \mathbb{N} } K_m(B_{X_n},J)$.
		In order to check  the left hand side inequality take some  $f \in H_\infty^J(B_{X_n})$. We assume without loss of generality that $ \| f \|_{B_{X_n}} = \dis\sup_{z \in B_{X_n}}|f(z)| = 1$, and consider its  Taylor series decomposition  $f = \dis\sum_{m=0}^\infty P_m$
with  $m$-homogeneous polynomials $P_m$. For every $m \in \zN_0$ we have $ P_m(z) = \dis\sum_{\alpha \in J_m} c_\alpha(f) z^\alpha$, thus taking $\rho = \inf_{m \in \mathbb{N} } K_m(B_{X_n}, J) $, it follows
		\begin{equation*}\label{desig hom}
		\Big\| \dis\sum_{\alpha \in J_m} |c_\alpha(f)| (\rho z)^\alpha \Big\|_{\Pp_m( {X_n})} \le  \Big\| \dis\sum_{\alpha \in J_m} c_\alpha(f) z^\alpha \Big\|_{\Pp_m( {X_n})}.
		\end{equation*}
				Using Lemma \ref{lem: Wiener general} we have that
		\begin{align*}
		\dis\sum_{m=0}^\infty \dis\sum_{\alpha \in J_m} |c_\alpha(f)| \left( \frac{\rho w}{3 } \right)^\alpha
		& \le |c_0(f)| + \dis\sum_{m=1}^\infty  \frac{1}{3^m} \Big\| \dis\sum_{\alpha \in J_m} c_\alpha(f) z^\alpha \Big\|_{\PP_m( {X_n})} \\
		& \le |c_0(f)| + \dis\sum_{m=1}^\infty  \frac{1}{3^m}(1 - |c_0(f)|^2)
\le |c_0(f)| + \frac{1 - |c_0(f)|^2}{2}
 \le 1 ,
		\end{align*}
		where last inequality holds as $ |c_0(f)| \le \sup_{z \in B_{X_n}} |f(z)| = 1$. This shows that $\frac{\rho}{3} \le K(B_{X_n},J) $,  completing the proof.
	\end{proof}

\smallskip

As an application we add two  interesting examples. For  $X_n = \ell_\infty^n$ the first example  goes back to \cite{aron1989analytic} (see  also \cite[Proposition~9.7]{defant2019libro}), and the second one is an  extension of a result from \cite[Proposition~3.1 and Example~3.4]{carando2014Dirichlet} from $X_n = \ell_\infty^n$ to an arbitrary Banach sequence lattice $X_n$.

\smallskip

\begin{example}
  Let $X$ be a Banach sequence lattice and $J = \big\{ me_k\colon  k,m \in \mathbb{N}\big \}$. Then for all $m,n$
  \[
  \text{$K_m(B_{X_n},J)= 1$\,\,\,\,\,\,and \,\,\,\,\,\,$K(B_{X_n},J)= \frac{1}{3}$\,.}
  \]
  \end{example}

\begin{proof} We start using  that for each $m$
\[
\boldsymbol{\chimon} (\mathcal{P}_{J_m}(\ell_\infty^n)) =1\,,
\]
and hence we obtain from Proposition~\ref{prop: Km vs uncond in Pm} and the monotonicity of Bohr radii from \eqref{rem: bohr radii monotony} that
  \[
  1 = K(B_{\ell_\infty^n},J_m)
  \leq
  K(B_{X_n},J_m) \leq 1\,,
  \]
  the first claim.  Then the  second claim is a simple consequence of Theorem~\ref{thm: Bohr vs unc}:
  \begin{equation*}
    \frac{1}{3}
    =
   \frac{1}{3}  \inf_{ m \in \mathbb{N} } K_m(B_{X_n},J) \leq K(B_{X_n},J) \leq \frac{1}{3} \,. \qedhere
  \end{equation*}
\end{proof}

\smallskip

\section{Homogenizing sequences} The aim of this section is to emphasize the importance of the $m$-homogeneous Bohr radius  for  the understanding of the Bohr radius itself. Given a Banach sequence lattice   $X$ and 
an index set $J \subset \mathbb{N}_0^{(\mathbb{N})}$,
the decay of $K(B_{X_n},J)$ in the dimension $n$ equals asymptotically the decay of  the infimum in $m$ of all $m$-homogeneous Bohr radii $K_m(B_{X_n},J)$.
This is the outcome of  Theorem~\ref{thm: Bohr vs unc}.
Here we intend to go one step further. We show that $K(B_{X_n},J)$,  
when the number of variables $n$ goes to infinity, behaves asymptotically 
like $K_{\pmb{\mathfrak{m}}(n)}(B_{X_n},J)$, where 
$(\pmb{\mathfrak{m}}(n))$ is a sequence in $n$ depending on $X$
and $J$.
Understanding how this homogeneity degree depends on $n$ has been the key to unravel the understating of the Bohr radius in most of the known examples.

Let $X$ be a Banach sequence lattice and $J \subset \mathbb{N}_0^{(\mathbb{N})}$ an index set. We say that a  sequence
$\pmb{\mathfrak{m}} =(\pmb{\mathfrak{m}}(n))_{n \in \mathbb{N}}$ of natural numbers homogenizes
the sequence $(K(B_{X_n},J))_{n \in \mathbb{N}}$ of Bohr radii whenever
\begin{equation}\label{breakpoint}
    K(B_{X_n},J)
\sim K_{\pmb{\mathfrak{m}}(n)}(B_{X_n},J).
  \end{equation}
    Given $X$ and $J$, let us first show existence of such  sequences.

\begin{proposition}\label{cor: exists breaking p.}
Let $X$ be a Banach sequence lattice and $J \subset \mathbb{N}_0^{(\mathbb{N})}$ an index set.
 Then the sequence
  \begin{align*}\label{break}
\pmb{\mathfrak{m}}(n)=\min \big\{ k :  \inf_{m \in \mathbb{N}} K_m(B_{X_n},J) = K_{k}(B_{X_n}, J) \big\}\,,\,\,\, n \in \mathbb{N}
    \end{align*}
   is a homogenizing sequence for $(K(B_{X_n},J))_{n \in \mathbb{N}}$.
\end{proposition}

In view of Theorem~\ref{thm: Bohr vs unc}, all we need to show is the existence of a natural number $k$ for which
$\inf_{m \in \mathbb{N}} K_m(B_{X_n},J) = K_{k}(B_{X_n})$.
To prove this we start with the following simple lemma.

\begin{lemma}\label{prop: m homoge radii}
Let $X_n = (\CC^n, \| \cdot \|)$ be a Banach sequence lattice and $J \subset \NN_0^n$ an index set. Then for every~$n$
\[ \lim_{m \to \infty} K_m(B_{X_n},J) = 1.
\]
\end{lemma}

\begin{proof}
By Proposition \ref{prop: Km vs uncond in Pm} we have to prove that $\lim_{m \to \infty}\boldsymbol{\chimon}(\mathcal P_{J_m^n}(X_n))^{1/m} = 1$.  But using the  fact that all coefficient functionals on $\mathcal{P}_m(X_n)$  are normalized, we have the trivial estimate
\[
\boldsymbol{\chimon}(\mathcal P_{J_m^n}(X_n))  \leq | \Lambda(m,n) | = {n+m-1 \choose m} \leq c(n) m^{\frac{n-1}{2}}\,.
\]
which, taking $m$th roots, leads to the conclusion.
\end{proof}

\begin{proof}[Proof of Proposition~\ref{cor: exists breaking p.}]
Fixing a natural number    $n$, we distinguish two cases,
\[
\text
{
$A=\inf_{m \in \mathbb{N}} K_m(B_{X_n},J) = 1$ \,\,\,and \,\,\,$A = \inf_{m \in \mathbb{N}} K_m(B_{X_n},J)<1$\,.
}
\]
The case $A=1$ is the trivial one, since then all $K_m(B_{X_n},J) =1$.  If $A < 1$, then by Lemma~\ref{prop: m homoge radii},
there is
$m_0=m_0(n)$ such that
\[
A= \inf_{m \leq m_0} K_m(B_{X_n},J)\,,
\]
which completes the argument.
  \end{proof}

We finish with the main result of this section.

\begin{theorem} \label{fini}
Let $X$ be a Banach sequence lattice, $J \subset \mathbb{N}_0^{(\mathbb{N})}$ an index set, and
$(\pmb{\mathfrak{m}}(n))_n$ some  sequence which  homogenizes   $(K(B_{X_n},J))_n$. Then, uniformly in  $n$ and $d$,
    \begin{equation*}
    K(B_{X_n},J_{\leq d})
\sim
\begin{cases}
 K_{\pmb{\mathfrak{m}}(n)}(B_{X_n},J)
 &  \,\,\, \,\pmb{\mathfrak{m}}(n) \leq d\,\\[1mm]
  \inf_{m \leq d} K_{m}(B_{X_n},J)&    \,\,\, \,d <  \pmb{\mathfrak{m}}(n)\,.
 \end{cases}
\end{equation*}
\end{theorem}

\begin{proof}
  By Theorem~\ref{thm: Bohr vs unc} we have that
  \begin{equation}\label{knapp}
     K(B_{X_n},J_{\leq d}) \sim \inf_{m \leq d} K_m(B_{X_n},J_{\leq d})=\inf_{m \leq d} K_m(B_{X_n},J)\,,
  \end{equation}
    uniformly in $d,n$. So the second equivalence is true for any $d$. For $\pmb{\mathfrak{m}}(n) \leq d$ note that by the definition  for homogenizing sequences and  Theorem~\ref{thm: Bohr vs unc} we have
    \[
    K_{\pmb{\mathfrak{m}}(n)}(B_{X_n},J)
    \sim
     \inf_{ m \in \mathbb{N} } K_m(B_{X_n},J)
     \leq
     \inf_{m \leq d} K_m(B_{X_n},J) \leq K_{\pmb{\mathfrak{m}}(n)}(B_{X_n},J),
         \]
   which,  again using \eqref{knapp}, gives what we want.
\end{proof}

\smallskip

\section{Bohr radii vs concavity and convexity}

We  relate the geometric concepts  of concavity and convexity of Banach sequence spaces with the notion of  Bohr radii. For $2$-convex spaces we give an explicit characterization of the asymptotic growth of the Bohr radius for a huge class of index sets.

    \begin{proposition}
   Let $X$ and $Y$ be  Banach sequence lattices and $J \subset \mathbb{N}_0^{(\NN)}$. Assume that  $X$ is $r$-~convex and that $Y$ is $r$-concave, where  $1 \leq r \leq \infty$.  Then for each $m,n \in \NN$
   \[
   K_m(B_{Y_n},J)   \,\,\leq \,\, M_{(r)}(X_n) M^{(r)}(Y_n)\,\, K_m(B_{X_n},J)
   \]
   and
   \[
   K(B_{Y_n},J)   \,\,\leq \,\,3 M_{(r)}(X_n) M^{(r)}(Y_n)\,\, K(B_{X_n},J)  \,.
   \]
    \end{proposition}

    \begin{proof}
     By Corollary \ref{appl}, it follows that, for each $m,n \in \NN$, we have 
     \[
\boldsymbol{\chimon}\big(\mathcal{P}_{J_m}(X_n)\big) \leq \big(M_{(r)}(X) M^{(r)}(Y)\big)^m\,\,\boldsymbol{\chimon}\big(\mathcal{P}_{J_m}(Y_n)\big)\,,
\]
hence, taking the $m$-th root, the first claim follows from Proposition~\ref{prop: Km vs uncond in Pm}.
The second claim is then a consequence of Theorem \ref{thm: Bohr vs unc}.
\end{proof}

    \smallskip

\begin{theorem}\label{thm: 2 convex bohr radiiA}
Let $X$ be a $2$-convex Banach sequence space and  $J \subset \mathbb{N}_0^{(\NN)}$  an index set
such  that $\Lambda_{T}(m,n) \subset~J $
 for any $m\le n$.  Then ,   with a constant $C \geq 1$ only depending on $X$,
\[
K_m(B_{X_n},J) \sim_C \Big(\frac{m}{n+m}\Big)^{\frac{m-1}{2m}}\,.
\]
Moreover,   with constants $C \geq 1$ only depending on $X$,
\[
K(B_{X_n},J) \sim_C \sqrt{\frac{\log n}{n}}
\]
and
    \begin{equation*}
K\big(B_{X_n}, J_{\leq d}\big)
\sim_C
\begin{cases}
\Big(\frac{d}{n}\Big)^{\frac{d-1}{2d}},  &   d \leq  \log n\\[1mm]
\sqrt{\frac{\log n}{n}}  &   \log n \leq d\,.
\end{cases}
\end{equation*}
\end{theorem}

\smallskip

\begin{remark}\label{rem: maximum of f}
In the following we several times need the fact that, fixing a natural number $n$, the maximum of $f_n(t) = \frac{1}{t n^{1/t}}$ for $t > 0$ is attained at $t = \log n$. In particular,
\[
\dis \max_{m \in \mathbb{N}}\,f_n(m) \sim_C \frac{1}{\log n}
\]
\end{remark}

\smallskip

\begin{proof}[Proof of Theorem \ref{thm: 2 convex bohr radiiA}]
 By Theorem \ref{conny3} for all $m, n$
 \[
  \left( 1+\frac{n}{m}\right)^{\frac{m-1}{2}}\sim_{C^m}
\boldsymbol{\chimon}\big(\Pp_{J_m}( X_n)\big) \,,
\]
and hence by Proposition~\ref{prop: Km vs uncond in Pm} the first claim follows. Note that this in particular
means that
\[
\text{
$K_m(B_{X_n},J) \sim_C \Big(\frac{m}{n}\Big)^{\frac{m-1}{2m}}$ \,  for $m \leq n$ \quad and \quad
$K_m(B_{X_n},J) \sim_C 1$  \, for $m \geq n$.
}
\]
As a consequence, we deduce from Theorem~\ref{thm: Bohr vs unc}  that for some constant $C \ge 1$
\begin{align*}
\frac{1}{3} \min \left\{ \frac{1}{C} ,\inf_{m \leq n}\,\Big(\frac{m}{n}\Big)^{\frac{m-1}{2m}} \right\}
&
\leq
\frac{1}{3} \min \left\{  \inf_{n \leq m } K_m(B_{X_n},J),\inf_{m \leq n} K_m(B_{X_n},J) \right\}
\\&
=
\frac{1}{3} \inf_{ m \in \mathbb{N} } K_m(B_{X_n},J)
\\&
\le K(B_{X_n},J)
\le \inf_{m \in \mathbb{N}} K_m(B_{X_n},J) \le C \inf_{m \leq n}\Big(\frac{m}{n}\Big)^{\frac{m-1}{2m}} .
  \end{align*}
Since
\begin{equation}\label{manno}
  \Big(\frac{m}{n}\Big)^{\frac{m-1}{2m}} \sim_C \Big( \frac{m n^{\frac{1}{m}} }{n}\Big)^{\frac{1}{2}} \,,
\end{equation}
we conclude from  Remark \ref{rem: maximum of f}  and the first statement of the theorem ( already proved) that
\[
K(B_{X_n},J) \sim_C \sqrt{\frac{\log n}{n}} \sim_C K_{[\log n]}(B_{X_n},J)\,.
\]
The   first equivalence is our second statement, and together with the second equivalence this   shows that the sequence  $\pmb{\mathfrak{m}}~=~([\log n])_n$
is a homogenizing sequence for $(K(B_{X_n},J))_n$. Then Theorem~\ref{fini} implies 
\begin{equation*}
    K(B_{X_n},J_{\leq d})
\sim_C
\begin{cases}
   \inf_{m \leq d} K_{m}(B_{X_n},J)&    \,\,\, \,d \leq   [\log n]
   \\[1mm]
   K_{[\log n]}(B_{X_n},J)
 &  \,\,\, \,[\log n] \leq d\,\,.
 \end{cases}
\end{equation*}
This with the first claim of our theorem gives the conclusion, since looking at  \eqref{manno} (and Remark~\ref{rem: maximum of f}), we see that
for each fixed $n$ the function $\left( \frac{m}{n}\right)^{\frac{m-1}{2m}}$  is up to constants
equivalent to a decreasing function  on $[0, \log n]$.
\end{proof}

\smallskip

\section{Bohr radii vs projection constants}
Here we present a direct relation between Bohr radii and  projection constants of spaces of homogeneous polynomials.

    \begin{theorem} \label{applyproj}
       Let $X$ be a Banach sequence lattice. Then  for each $n$ and each  index set $J \subset \mathbb{N}_0^{(\mathbb{N})}$
   \[
   K(B_{X_n},J)   \ge  \frac{1}{6} \,\,\,\,\inf_{ m \in \mathbb{N} }
   \bigg( \frac{1}{\sqrt[m]{e\|\mathbf{Q}\|_{\Lambda(m,n), J^n_m}}}
   \frac{1}{\sqrt[m]{\boldsymbol{\lambda}(\mathcal{P}_{(J_m^n)^\flat}(X_n))}}
   \bigg)\,.
   \]
    \end{theorem}

    Observe that on the right side of this inequality we take the $m$-root, although $(J_m^n)^\flat$ by definition is contained in $\Lambda(m-1,n)$.

    \begin{proof}
     From Proposition~\ref{thm: Bohr vs unc} we deduce that for every $n$
     \[
     \frac{1}{3} \inf_{ m \in \mathbb{N} } K_m(B_{X_n},J)  \le K(B_{X_n},J) \,.
     \]
   On the other hand, by Theorem \ref{main3}  for every $m,n$
                 \[
   \boldsymbol{\chimon}\big( \mathcal{P}_{J_m^n}(X_n) \big)\,\,
 \le\,\, e 2^{m}
\|\mathbf{Q}_{\Lambda(m,n),J_{m}^n}\| \boldsymbol{\lambda}\big( \mathcal{P}_{(J_m^n)^\flat}(X_n) \big)\,.
\]
Then the proof completes, taking the $m$-th root and using  Proposition \ref{prop: Km vs uncond in Pm}.
\end{proof}

    \smallskip

 As an application we prove  upper and lower estimates for multivariate Bohr radii in the tetrahedral case.
    \smallskip

    \begin{theorem} \label{applyhedral}
       Let $X$ be a symmetric Banach sequence lattice. Then  for each $n\in \mathbb{N}$
\[
K(B_{X_n},\Lambda_T) 
\prec \inf_{m\leq n} 
\bigg(\frac{\varphi_{X'}(m-1)}{\varphi_{X'}(n)}\bigg)^{\frac{m-1}{m}}\,,
\]
and for  $1 \leq r \leq 2$
\[
K(B_{X_n},\Lambda_T) \prec
\frac{(\log n)^{\frac{1}{r'}}}{\varphi_{X'}(n)}  \|\id: X_n \to \ell_r^n\|\,.
\]
\end{theorem}

\begin{proof}
Let us start with the proof of the lower estimate.
From Theorem~\ref{OrOuSe} we know that for all $m,n$
\[
\big\|\mathbf{Q}_{\Lambda(m,n),\Lambda_T(m,n)}:
\mathcal{P}_{\Lambda(m,n)}(X_n) \to \mathcal{P}_{\Lambda_T(m,n)}(X_n)\big\|
\leq \kappa^m\,.
\]
Moreover, by Theorem \ref{lambda-dash} and Proposition~\ref{lambda} 
we have
\begin{align*}
\boldsymbol{\lambda}(\mathcal{P}_{(\Lambda_T(m,n))^\flat}(X_n))
 \leq e^{m-1}
 \Big(\frac{\varphi_{X_n'}(n)}{\varphi_{X_n'}(m-1)}\Big)^{m-1} = e^{m-1}
 \Big(\frac{\varphi_{X'}(n)}{\varphi_{X'}(m-1)}\Big)^{m-1}
\,,
\end{align*}
where last equality holds as $m \le n$ (Recall that $\Lambda_T(m,n) = \emptyset$ when $m > n$).
Implementing the preceding two inequalities into Theorem~\ref{applyproj}, yields the assertion.

For the proof of the upper estimate, note first that by Theorem \ref{thm: Bohr vs unc}
     \[
      K(B_{X_n},\Lambda_T) \leq  \inf_{ m \in \mathbb{N} } K(B_{X_n},\Lambda_T(m,n))
      =
      \frac{1}{\sup_{ m \in \mathbb{N} } \boldsymbol{\chimon}(\mathcal{P}_{\Lambda_T(m,n)}(X_n))^{1/m} }
       \,.
     \]
     But by Lemma~\ref{innichenA}, the symmetry of $X$, and \eqref{binom} we see that
     \begin{align*}
      \sup_{ m \in \mathbb{N} } \boldsymbol{\chimon}(\mathcal{P}_{\Lambda_T(m,n)}(X_n))^{1/m}
      &
    \leq \frac{1}{\|\id:X_n\to \ell_r^n\|}\sup_{ m \in \mathbb{N} }\frac{|\Lambda_T(m,n)|^{1/m}}{\varphi_X(n) n^{\frac{1}{mr'}} m ^{-\frac{1}{r}}}
        \\&
        \leq \frac{\varphi_{X'}(n) }{ n \|\id:X_n\to \ell_r^n\|}\sup_{ m \in \mathbb{N} }\frac{  n}{n^{\frac{1}{mr'}} m ^{\frac{1}{r'}}}
                \leq \frac{\varphi_{X'}(n) }{ \|\id:X_n\to \ell_r^n\|}\sup_{ m \in \mathbb{N} }\frac{  1}{n^{\frac{1}{mr'}} m ^{\frac{1}{r'}}}\,,
     \end{align*}
     so that the conclusion again follows using Remark~\ref{rem: maximum of f}.
               \end{proof}
               
               If we in the preceding lemma replace the 'covering approach' from Lemma~\ref{innichenA} by the 'entropy approach'
               from Lemma~\ref{innichen}, the upper bound for $K(B_{X_n},\Lambda_T)$ does not improve.

\smallskip

\section{Lorentz spaces}\label{Section: bohr on lorentz}
Now we focus on  Lorentz spaces, i.e., $X_n = \ell_{r,s}^n$. The following theorem is the main result.

\begin{theorem}\label{thm: main bohr radii}
Let $1<r<\infty$, $1 \le s \le \infty $, and  $J \subset \NN_{0}^{(\NN)}$ some index set for which  $\Lambda_{T}(m,n) \subset J$
for all $m\leq n$. Then
\begin{itemize}
    \item[(i)]  $2 < r < \infty$ and $1 \le s \le \infty$: \;\;
    $K(B_{\ell_{r,s}^n}, J) \sim \sqrt{\frac{\log n}{n}}$.
    \item[(ii)] $1 < r \le 2$ and  $1 \le s \le r$: \;\;
    $
   K(B_{\ell_{r,s}^n},J)\sim \left( \frac{\log n}{n} \right)^{1-\frac{1}{r}}.
    $
    \item[(iii)] $1<r\le2$ and $r < s $: \;\;
        $\left( \frac{\log n}{n} \right)^{1-\frac{1}{r}} \prec K(B_{\ell_{r,s}^n},J) \prec\frac{(\log n)^{1-\frac{1}{s}}}{n^{1-\frac{1}{r}}}$.
\end{itemize}
Moreover, with a constant only depending on $r$ and $s$,
 \begin{equation*}
K\big(B_{\ell_{r,s}^n}, J_{\leq d}\big)
\succ
\begin{cases}
\Big(\frac{d}{n}\Big)^{\frac{d-1}{d}\min\big\{\frac{1}{r'},\frac{1}{2}\big\}}  &   d \leq  \log n\\[1mm]
\Big(\frac{\log n}{n}\Big)^{\min\big\{\frac{1}{r'},\frac{1}{2}\big\}}  &   \log n \leq d\,;
\end{cases}
\end{equation*}
 in the cases (i) and (ii) we even have equivalence, whereas in  (iii) we only know the weaker upper bound $(\log n)^{1-\frac{1}{s}}/n^{1-\frac{1}{r}}$.
\end{theorem}

 By monotonicity as in \eqref{rem: bohr radii monotony} it suffices to prove the upper bounds  in the 'tetrahedral case',
and the lower bounds in the  'full case'. But the  lower estimate in the tetrahedral case seems technically  less involved than  the lower estimate in the full case. This motivates us to give (up to some point) two separate proof for both situations.

\begin{proof}
\text{}\\
\noindent {\bf A}:
Note first that $\ell_{r,s}$ in  case $(i)$ is $2$-convex whenever $2 < r < \infty, \,  2 \leq s \leq \infty$  (see the introduction of 
Chapter~\ref{Part: Polynomials on Lorentz sequences spaces}). Hence in this
situation the result is a special case of Theorem~\ref{thm: 2 convex bohr radiiA}. But in the case $2 < r < \infty,\, 1 \leq s \leq 2$
we also know from Theorem \ref{conny3} that
 $
  \big( 1+\frac{n}{m}\big)^{\frac{m-1}{2}}\sim_{C^m}
\boldsymbol{\chimon}(\Pp_{J_m}( X_n)) \,,
$
for all  $m, n$, and so exactly the same arguments as in the proof of  Theorem~\ref{thm: 2 convex bohr radiiA} work.

\noindent {\bf B}:
Let us check the upper bounds in the both cases $(ii)$ and $(iii)$ in the tetrahedral case only, since then the general case follows by monotonicity from \eqref{rem: bohr radii monotony}.  Both cases  then are immediate from Theorem~\ref{applyhedral},
since under the assumption of $(ii)$ we have
\[
 \|\id: \ell_{r,s}^n \to \ell_r^n\| = 1\,,
\]
and under assumption of $(iii)$
\[
 \|\id: \ell_{r,s}^n \to \ell_r^n\| \prec (\log n)^{\frac{1}{r}-\frac{1}{s}} \,.
\]

\noindent {\bf C}: Let us check the lower bounds in the tetrahedral case under the assumptions of $(ii)$ and $(iii)$.
 By Theorem~\ref{applyhedral} we have that
\begin{align*}
K\big(B_{\ell_{r,s}^n},\Lambda_T\big) &
\prec \inf_{m\leq n} \bigg(
\frac{ \varphi_{\ell_{r',s'}^n(m-1)}}{\varphi_{\ell_{r',s'}^n}(n)}\bigg)^{\frac{m-1}{m}}\\
& =\inf_{m\leq n}
\frac{ (m-1)^{\frac{1}{r'}\frac{m-1}{m}}}{n^{\frac{1}{r'}\frac{m-1}{m}}}
\prec \inf_{m\leq n}
\Big(\frac{m n^{\frac{1}{m}} }{n}\Big) ^{\frac{1}{r'}} =\frac{1}{n^{\frac{1}{r'}}}\inf_{m\leq n}
\Big(m n^{\frac{1}{m}}\Big)^{\frac{1}{r'}}
\,,
\end{align*}
hence minimizing the last infimum (again using Remark~\ref{rem: maximum of f}) leads to the  claim.

\noindent {\bf D}: Finally, given $r,s$ like in $(ii)$ or $(iii)$, it remains  to prove the lower bounds for $ K(B_{\ell_{r,s}^n}, J)$,
for any index set
$J \subset \NN_{0}^{(\NN)}$  for which  $\Lambda_{T}(m,n) \subset J$
for all $m\leq n$. But, again looking at  \eqref{rem: bohr radii monotony}, we may assume that  $J = \NN_0^{(\NN)}$. Let us repeat that this in fact  covers what we did in {\bf C}. There we give an independent proof for the tetrahedral case, since  the one for the full case seems technically  more involved.

To achieve the bounds we are looking for, it is central to use Theorem \ref{thm: Bohr vs unc} and good bounds for the unconditional basis constant for the space of $m$-homogeneous polynomials in each of these cases.

\noindent
$1 < r \le 2$ and $1 \le s \le r$: By  Theorem \ref{proj lorentz s<r} with $\kappa = 1/2$ and Theorem \ref{main3} we know, since $\frac{1}{s}-\frac{1}{r}\leq\frac{1}{r'}$, that uniformly in $n,m$
\begin{align*}
\boldsymbol{\chimon}\big(\mathcal P_m( \ell_{r,s}^n)\big)^{1/m} & \prec \left( \Big(\frac{n}{m} \Big)^{(m-1)/r'} \max \left\{ \Big( \frac{  \log(m)^{m}}{n^{\sqrt{m}-1}}\Big)^{1/r'},1 \right\} \right)^{1/m}\\
&\prec n^{1/r'}\Big(\frac{1}{mn^{1/m}} \Big)^{1/r'}\max \left\{ \Big(\frac{\log(m)}{n^{(\sqrt{m}-1)/m}}\Big)^{1/r'},1 \right\} \\
&=c n^{1/r'}\max \left\{\Big(\frac{\log(m)}{mn^{1/\sqrt{m}}} \Big)^{1/r'},\Big(\frac{1}{mn^{1/m}} \Big)^{1/r'}\right\}\\
&\prec n^{1/r'}\max\left\{\Big(\frac{1}{\sqrt{m}n^{1/\sqrt{m}}} \Big)^{1/r'},\Big(\frac{1}{mn^{1/m}} \Big)^{1/r'}\right\}.
\end{align*}
Optimizing with  Remark \ref{rem: maximum of f}  shows
$
\sup_{m}  \boldsymbol{\chimon}(\mathcal P_m( \ell_{r,s}^n))^{1/m} \prec \Big( \frac{n}{\log n} \Big)^{1/r'}\,,
$
and hence Theorem \ref{thm: Bohr vs unc} as desired gives
\[
\Big( \frac{n}{\log n} \Big)^{1/r'} \prec K(B_{X_n})\,.
\]

\smallskip

\noindent $1 < r \le 2$ and $ r<s$: \,By Theorem \ref{proj lorentz s>r} and Corollary \ref{proj lorentz s>r, big m} we have, uniformly for all  $m,n$ with $m \leq n$,
\[
\boldsymbol{\chimon}\big(\Pp_m(\ell_{r,s})\big)^{1/m} \prec \begin{cases}
n^{1/r'} \Big( \frac{n}{m n^{1/m}} \Big)^{1/r'} & m\le  \frac{n}{e (\log n)^{r'(\frac{1}{r} - \frac{1}{s})}},\\
(\log n)^{\frac1{r}-\frac1{s}} & \textrm{ otherwise.}
\end{cases}
\]
Choose now some  $n_0 = n_0(r,s)$  such that $\log n \le \frac{n}{e (\log n)^{r'(\frac{1}{r} - \frac{1}{s})}}
$ for all $n \ge n_0$.
Then by Remark \ref{rem: maximum of f} for every
$n \ge n_0$ and $ m \le \frac{n}{e (\log n)^{r'(\frac{1}{r} - \frac{1}{s})}}$ we have
\[
  \sup_{m \le \frac{n}{e (\log n)^{r'(\frac{1}{r} - \frac{1}{s})}}} \boldsymbol{\chimon}\big(\Pp_m(\ell_{r,s})\big)^{1/m} \prec \left( \frac{n}{\log n} \right)^{1 - \frac{1}{r}}\,.
\]
On the other hand, for all $n$
\[
 \sup_{m \ge \frac{n}{e (\log n)^{r'(\frac{1}{r} - \frac{1}{s})}}} \boldsymbol{\chimon}\big(\Pp_m(\ell_{r,s})\big)^{1/m} \prec (\log n)^{\frac{1}{r} - \frac{1}{s}}.
\]
Combining both estimates, Proposition~\ref{prop: Km vs uncond in Pm} and Theorem \ref{thm: Bohr vs unc} complete the proof.

\smallskip

\noindent Finally, we remark that the 'moreover-part' is another consequence of  Theorem~\ref{fini} (compare with the proof
of Theorem~\ref{thm: 2 convex bohr radiiA}).
\end{proof}

For the full index set $J= \N_{0}^{(\N)}$ statement $(i)$  of Theorem \ref{thm: main bohr radii} was first given in  \cite[Corollary 10, (i)]{defant2018bohr}. The other two cases of  Theorem \ref{thm: main bohr radii} improve both those of \cite[Corollary 10, (ii),(iii)]{defant2018bohr} (in particular, $(ii)$  of Theorem \ref{thm: main bohr radii} is asymptotically optimal
in contrast to \cite[Corollary 10, (ii)]{defant2018bohr}).

\smallskip

\section{Limits}
In \cite{bayart2014bohr} Bayart, Pellegrino, and Seoane  proved the following remarkable result:
\begin{equation}\label{BPS}
  \lim_{n \to \infty} \frac{K(B_{\ell_\infty^n})}{\sqrt{\frac{\log n}{n}}}  = 1\,.
\end{equation}
Based on this limit, we establish a far reaching extension, which in many concrete situations leads to analogs of~\eqref{BPS}
in more general situations. More precisely, the aim is to consider here instead of  $\ell_\infty$ some other classes of
 Banach sequence lattices $X$, and in instead of the  full index set $J = \mathbb{N}_0^{(\mathbb{N})}$
 more general index sets -- including the index set of all tetrahedral indices. In particular, we will show that an analog of \eqref{BPS} holds for the Lorentz sequence spaces appearing in $(i)$ of Theorem \ref{thm: main bohr radii}.

\smallskip

\begin{theorem}\label{limits}
Let $X$ be a  Banach sequence lattice such that for each $n$ 
 \[
\|\id\colon X_n \to \ell_2^{n}\| \leq \frac{1}{\sqrt{n}}\,\|\id\colon X_n \to \ell_1^{n}\|\,.
\]
 Then
\[
\lim_{n \to \infty} \frac{K(B_{X_n})}{\sqrt{\frac{\log n}{n}}}  = 1\,.
\]
\end{theorem}

\begin{proof}
  Note first that by \eqref{lowinfty} and \eqref{BPS} we have
  \begin{align*}
    1 \leq \lim_{n \to \infty}  \frac{K(B_{\ell_\infty^n})}{\sqrt{\frac{\log n}{n}}}
        \leq
     \liminf_{n \to \infty} \frac{K(B_{X_n})}{\sqrt{\frac{\log n}{n}}}
     \leq
     \limsup_{n \to \infty} \frac{K(B_{X_n})}{\sqrt{\frac{\log n}{n}}}
        \,,
\end{align*}
hence it remains to show that  the last term is less than or equal to $1$. To do so we apply Proposition~\ref{innichen++} with $r=2$, and get for each $m\leq n$
  \begin{align*}
    \frac{1}{\boldsymbol{\chimon}\big(\mathcal{P}_m(X_n)\big)^\frac{1}{m}}
    &
    \leq  \left(  \frac{\|\id\colon X_n\to \ell_2^n\|}{\|\id:X_n\to \ell_1^n\|}\right)^{\frac{m-1}{m}}
    \left(Cm^{\frac{5}{4}} \,  (\log n)^{\frac{3}{2}} e^{-\frac{m}{2}}m^{\frac{m}{2}}   \right)^{\frac{1}{m}}
    \\&
    \leq
    n^{-\frac{m-1}{2m}} \left(Cm^{\frac{5}{4}} \,  (\log n)^{\frac{3}{2}} e^{-\frac{m}{2}}m^{\frac{m}{2}}   \right)^{\frac{1}{m}}
    \\&
            =
    \Big(  \frac{m}{n} \Big)^\frac{1}{2} n^\frac{1}{2m} m^{-\frac{1}{2}}
    \left(Cm^{\frac{5}{4}} \,  (\log n)^{\frac{3}{2}} e^{-\frac{m}{2}}m^{\frac{m}{2}}   \right)^{\frac{1}{m}}
  \\&
 =
    \Big(  \frac{m}{n} \Big)^\frac{1}{2} n^\frac{1}{2m} \frac{1}{\sqrt{e}}
    \left(
 Cm^{\frac{5}{4}}(\log n)^{\frac{3}{2}}      \right)^{\frac{1}{m}}
 \,.
  \end{align*}
  From Theorem~\ref{thm: Bohr vs unc} and Proposition~\ref{prop: Km vs uncond in Pm} we then deduce
  \begin{align*}
     \limsup_{n \to \infty}\frac{K(B_{X_n})}{\sqrt{\frac{\log n}{n}}}
     &
    =   \limsup_{n \to \infty} \frac{\sqrt{\frac{n}{\log n}}}{\sup_{m\in \mathbb{N}}  \boldsymbol{\chimon}\big(\mathcal{P}_m(X_n)\big)^\frac{1}{m}}
    \\&
    \leq   \limsup_{n \to \infty} \frac{\sqrt{\frac{n}{\log n}}}{ \boldsymbol{\chimon}\big(\mathcal{P}_{\lceil\log n\rceil}(X_n)\big)^\frac{1}{\lceil\log n\rceil}}
        \\&
    \leq   \limsup_{n \to \infty} \Big(  \frac{n}{\log n} \Big)^\frac{1}{2}
        \Big(  \frac{\lceil\log n\rceil}{n} \Big)^\frac{1}{2} n^\frac{1}{2\lceil\log n\rceil} \frac{1}{\sqrt{e}}
    \left(
    C \lceil\log n\rceil^{\frac{5}{4}} (\log n)^{\frac{3}{2}}      \right)^{\frac{1}{\lceil\log n\rceil}} =1\,,
  \end{align*}
  which completes the proof.
\end{proof}

\smallskip

Under an additional symmetry condition on the underlying Banach sequence lattice $X$, it is possible to incorporate
a much larger class of  index sets -- in particular, the index set of tetrahedral indices.

\smallskip

\begin{theorem}\label{limits++}
Let $X$ be a  Banach sequence lattice such that
\[
\varphi_X(n)\,\varphi_{X'}(n) \prec n\,,
\]
and 
\[
\|\id\colon X_n \to \ell_2^{n}\| \leq \frac{1}{\sqrt{n}}\,\|\id\colon X_n \to \ell_1^{n}\|, \quad\, n\in \mathbb{N}\,.
\]
Then for every index set  $J \subset \mathbb{N}_0^{(\NN)}$ for which $\Lambda_{T}(m,n) \subset~J$ 
for all $m \leq n$ one has
\[
\lim_{n \to \infty} \frac{K(B_{X_n},J)}{\sqrt{\frac{\log n}{n}}}  = 1\,.
\]
\end{theorem}

\begin{proof}
Observe  again first that by \eqref{lowinfty}, the monotonicity of Bohr radii in the index set, and \eqref{BPS} one has
\begin{align*}
1 = \lim_{n \to \infty}  \frac{K(B_{\ell_\infty^n})}{\sqrt{\frac{\log n}{n}}}
\leq \liminf_{n \to \infty}  \frac{K(B_{X_n})}{\sqrt{\frac{\log n}{n}}} \leq 
\liminf_{n \to \infty} \frac{K(B_{X_n},J)}{\sqrt{\frac{\log n}{n}}} \leq
\limsup_{n \to \infty} \frac{K(B_{X_n},J)}{\sqrt{\frac{\log n}{n}}} \leq
\limsup_{n \to \infty} \frac{K(B_{X_n},\Lambda_T)}{\sqrt{\frac{\log n}{n}}}\,,
\end{align*}
and hence as above  it remains to check that  the last term is  $\leq 1$. To do so we apply Proposition~\ref{innichen} 
with $r=2$, and get for each $m\leq n$
\begin{align*}
\frac{1}{\boldsymbol{\chimon}\big(\mathcal{P}_{\Lambda_T(m,n)}(X_n)\big)^\frac{1}{m}} & \leq
\left(  \frac{\|\id:X_n\to \ell_2^n\|}{\|\id:X_n\to \ell_1^n\|}\right)^{\frac{m-1}{m}}
\frac{\Big( Cm\,  (\log n)^{\frac{3}{2}} \Big)^{\frac{1}{m}}    \Big(e^{\frac{m}{2}} m^{m/2}  \Big)^{\frac{1}{m}}}
{\Big(\Big( \frac{n}{n-m} \Big)^{n-m}\Big)^{\frac{1}{m}} \Big(  \sqrt{\frac{n}{m(n-m)}}\Big)^{\frac{1}{m}}} \\
& \leq n^{-\frac{m-1}{2m}} \frac{\Big( Cm\,  (\log n)^{\frac{3}{2}} \Big)^{\frac{1}{m}} \Big(e^{\frac{m}{2}}     
m^{m/2}  \Big)^{\frac{1}{m}}}{\Big(\Big( \frac{n}{n-m} \Big)^{n-m}\Big)^{\frac{1}{m}}
\Big(\sqrt{\frac{n}{m(n-m)}}\Big)^{\frac{1}{m}}} \\
& = \Big(  \frac{m}{n} \Big)^\frac{1}{2} n^\frac{1}{2m} m^{-\frac{1}{2}}\frac{\Big( Cm\,(\log n)^{\frac{3}{2}} \Big)^{\frac{1}{m}}\Big(e^{\frac{m}{2}}m^{m/2}  \Big)^{\frac{1}{m}}}
{\Big(\Big( \frac{n}{n-m} \Big)^{n-m}\Big)^{\frac{1}{m}} \Big(  \sqrt{\frac{n}{m(n-m)}}\Big)^{\frac{1}{m}}}\\
& = \Big(  \frac{m}{n} \Big)^\frac{1}{2} n^\frac{1}{2m}  e^{\frac{1}{2}}    \Big( \frac{n-m}{n} \Big)^{\frac{n-m}{m}}
\frac{\Big( Cm\,  (\log n)^{\frac{3}{2}} \Big)^{\frac{1}{m}}}
{\Big(\sqrt{\frac{n}{m(n-m)}}\Big)^{\frac{1}{m}}} \\ 
& =\Big(  \frac{m}{n} \Big)^\frac{1}{2} n^\frac{1}{2m}  e^{\frac{1}{2}}    \Big( 1-\frac{m}{n} \Big)^{\frac{n}{m}}
\frac{n}{n-m}\frac{\Big( Cm\,  (\log n)^{\frac{3}{2}} \Big)^{\frac{1}{m}}}
{\Big(\sqrt{\frac{n}{m(n-m)}}\Big)^{\frac{1}{m}}}\,.
\end{align*}
Since
\[
\lim_{n \to \infty}
n^\frac{1}{2\lceil\log n\rceil}  e^{\frac{1}{2}}\Big( 1-\frac{\lceil\log n\rceil}{n}\Big)^{\frac{n}{\lceil\log n\rceil}} 
= e^{\frac{1}{2}}e^{\frac{1}{2}} e^{-1}= 1\,,
\]
the argument completes precisely like in proof of Theorem~\ref{limits}.
\end{proof}

\smallskip

We conclude with a class of Banach sequence lattices $X$ for which \eqref{BPS} holds, if we there
replace $\ell_\infty$ by $X$ and the full index set by an arbitrary index set including all tetrahedral
indices.

To do this we show how to generate Banach sequence lattices  which satisfy the  assumptions of
Theorem~\ref{limits}. First we observe that the $2$-convexification $X:=E^{(2)}$ of any Banach sequence lattice $E$ is $2$-convex with $M^{(2)}(X)=1$. If in addition $X$ satisfies
\begin{equation} \label{migra}
\varphi_X(n)\,\varphi_{X'}(n) \prec n\,,
\end{equation}
then in fact $X$ is an example we are looking for.

Clearly,  \eqref{migra} holds true whenever $E$ is symmetric. But it is worth noting that this estimate for $X:=E^{(2)}$ is valid whenever $E$ is an arbitrary Banach sequence lattice satisfying 
$\varphi_E(n)\,\varphi_{E'}(n) \prec n$ (so also for non-symmetric spaces $E$). Indeed, to see this observe that $E^{(2)} \equiv E^{\frac{1}{2}}\,(\ell_{\infty})^{\frac{1}{2}}$. Thus by the Lozanovskii formula for the K\"othe dual of the Calder\'on product of Banach lattices one has
\[
X' \equiv (E')^{\frac{1}{2}}\,(\ell_{\infty}')^{\frac{1}{2}}\equiv (E')^{\frac{1}{2}}\,(\ell_1)^{\frac{1}{2}}\,.
\]
This implies that for each $n\in \mathbb{N}$, we have $\varphi_{X'}(n) \leq \varphi_{E'}(n)^{\frac{1}{2}} n^{\frac{1}{2}}$. Combining with the above estimate for $E$ and $\varphi_X(n) = \varphi_E(n)^{\frac{1}{2}}$ for each $n$, we get the required estimate \eqref{migra} for $X$.

Let us collect a~few concrete examples of $2$-convex Banach sequence lattices $X$ with $M^{(2)}(X)=1$, which are generated by the above procedure. These are the symmetric Banach sequence lattices $\ell_r$ with $2 \leq r \leq \infty$  and $\ell_{r,s}$ with $2 < r < \infty$ and $1 \leq s \leq \infty$. Moreover, $\ell_{r}(\ell_{s})$ with $2\leq r, s\leq \infty$ is a~non-symmetric example of this type.


\begin{thebibliography}{100}
	
	\bibitem{aliprantis2006positive}
	C.~D. Aliprantis and O.~Burkinshaw.
	\newblock {\em Positive operators}.
	\newblock Berlin: Springer, reprint of the 1985 original edition, 2006.
	
	\bibitem{alkhaleefah2019bohr}
	S.~A. Alkhaleefah, I.~R. Kayumov, and S.~Ponnusamy.
	\newblock On the {Bohr} inequality with a fixed zero coefficient.
	\newblock {\em Proc. Am. Math. Soc.}, 147(12):5263--5274, 2019.
	
	\bibitem{aron1989analytic}
	R.~M. Aron and J.~Globevnik.
	\newblock Analytic functions on {{\(c_ 0\)}}.
	\newblock {\em Rev. Mat. Univ. Complutense Madr.}, 2:27--33, 1989.
	
	\bibitem{aron1976compact}
	R.~M. Aron and M.~Schottenloher.
	\newblock Compact holomorphic mappings on {Banach} spaces and the approximation
	property.
	\newblock {\em J. Funct. Anal.}, 21:7--30, 1976.
	
	\bibitem{balasubramanian2006bohr}
	R.~Balasubramanian, B.~Calado, and H.~Queff{\'e}lec.
	\newblock The {Bohr} inequality for ordinary {Dirichlet} series.
	\newblock {\em Stud. Math.}, 175(3):285--304, 2006.
	
	\bibitem{Banach}
	S.~Banach.
	\newblock {\em Th{\'e}orie des op{\'e}rations lin{\'e}aires}.
	\newblock Warsaw. 1932.
	
	\bibitem{Banach-Mazur}
	S.~Banach and S.~Mazur.
	\newblock Zur {T}heorie der linearen {D}imension.
	\newblock {\em Stud. Math.}, 4(1):100--112, 1933.
	
	\bibitem{bayart2012maximum}
	F.~Bayart.
	\newblock Maximum modulus of random polynomials.
	\newblock {\em Q. J. Math.}, 63(1):21--39, 2012.
	
	\bibitem{bayart2021convergence}
	F.~Bayart.
	\newblock Convergence and almost sure properties in {Hardy} spaces of
	{Dirichlet} series.
	\newblock {\em Math. Ann.}, 382(3-4):1485--1515, 2022.
	
	\bibitem{bayart2014bohr}
	F.~Bayart, D.~Pellegrino, and J.~B. Seoane-Sep{\'u}lveda.
	\newblock The {Bohr} radius of the {{\(n\)}}-dimensional polydisk is equivalent
	to {{\(\sqrt{(\log n) / n}\)}}.
	\newblock {\em Adv. Math.}, 264:726--746, 2014.
	
	\bibitem{beals2001quantum}
	R.~Beals, H.~Buhrman, R.~Cleve, M.~Mosca, and R.~de~Wolf.
	\newblock Quantum lower bounds by polynomials.
	\newblock {\em J. ACM}, 48(4):778--797, 2001.
	
	\bibitem{beneteau2004remarks}
	C.~B{\'e}n{\'e}teau, A.~Dahlner, and D.~Khavinson.
	\newblock Remarks on the {Bohr} phenomenon.
	\newblock {\em Comput. Methods Funct. Theory}, 4(1):1--19, 2004.
	
	\bibitem{benitez1998lower}
	C.~Ben{\'{\i}}tez, Y.~Sarantopoulos, and A.~Tonge.
	\newblock Lower bounds for norms of products of polynomials.
	\newblock {\em Math. Proc. Camb. Philos. Soc.}, 124(3):395--408, 1998.
	
	\bibitem{BS}
	C.~Bennett and R.~Sharpley.
	\newblock {\em Interpolation of operators}.
	\newblock Boston, MA etc.: Academic Press, Inc., 1988.
	
	\bibitem{BerLof76}
	J.~Bergh and J.~L{\"o}fstr{\"o}m.
	\newblock {\em Interpolation spaces. {An} introduction}, volume 223 of {\em
		Grundlehren Math. Wiss.}
	\newblock Springer, Cham, 1976.
	
	\bibitem{berman1952class}
	D.~Berman.
	\newblock On a class of linear operators.
	\newblock In {\em Dokl. Akad. Nauk SSSR}, volume~85, pages 13--16, 1952.
	
	\bibitem{bhowmik2018bohr}
	B.~Bhowmik and N.~Das.
	\newblock Bohr phenomenon for subordinating families of certain univalent
	functions.
	\newblock {\em J. Math. Anal. Appl.}, 462(2):1087--1098, 2018.
	
	\bibitem{billingsley2013convergence}
	P.~{Billingsley}.
	\newblock {\em {Convergence of probability measures.}}
	\newblock Chichester: Wiley, 1999.
	
	\bibitem{boas2000majorant}
	H.~P. Boas.
	\newblock Majorant series.
	\newblock {\em J. Korean Math. Soc.}, 37(2):321--337, 2000.
	
	\bibitem{khavinson1997bohr}
	H.~P. Boas and D.~Khavinson.
	\newblock Bohr's power series theorem in several variables.
	\newblock {\em Proc. Am. Math. Soc.}, 125(10):2975--2979, 1997.
	
	\bibitem{Bohnenblust}
	H.~F. Bohnenblust.
	\newblock Subspaces of {{\(l_{p,n}\)}} spaces.
	\newblock {\em Am. J. Math.}, 63:64--72, 1941.
	
	\bibitem{bohnenblust1931absolute}
	H.~F. Bohnenblust and E.~Hille.
	\newblock On the absolute convergence of {{D}irichlet} series.
	\newblock {\em Ann. Math.}, pages 600--622, 1931.
	
	\bibitem{bohr1913ueber}
	H.~Bohr.
	\newblock \"{U}ber die {Bedeutung} der {Potenzreihen} unendlich vieler
	{Variablen} in der {Theorie} der {{D}irichletschen} {Reihen} $\sum
	\frac{a_n}{n^s}$.
	\newblock {\em Nachrichten von der Gesellschaft der Wissenschaften zu
		G{\"o}ttingen, Mathematisch-Physikalische Klasse}, 1913:441--488, 1913.
	
	\bibitem{bohr1914theorem}
	H.~Bohr.
	\newblock A theorem concerning power series.
	\newblock {\em Proc. Lond. Math. Soc. (2)}, 13:1--5, 1914.
	
	\bibitem{bombieri2004remark}
	E.~Bombieri and J.~Bourgain.
	\newblock A remark on {Bohr}'s inequality.
	\newblock {\em Int. Math. Res. Not.}, 2004(80):4307--4330, 2004.
	
	\bibitem{bondarenkoseip2016}
	A.~Bondarenko and K.~Seip.
	\newblock Helson's problem for sums of a random multiplicative function.
	\newblock {\em Mathematika}, 62(1):101--110, 2016.
	
	\bibitem{bosznay1986norms}
	A.~P. Bosznay and B.~M. Garay.
	\newblock On norms of projections.
	\newblock {\em Acta Sci. Math.}, 50:87--92, 1986.
	
	\bibitem{Bourgain1983}
	J.~Bourgain.
	\newblock {{\(H^{\infty}\)}} is a {Grothendieck} space.
	\newblock {\em Stud. Math.}, 75:193--216, 1983.
	
	\bibitem{bourgain1989}
	J.~Bourgain.
	\newblock Homogeneous polynomials on the ball and polynomial bases.
	\newblock {\em Isr. J. Math.}, 68(3):327--347, 1989.
	
	\bibitem{calderon1964intermediate}
	A.~P. Calder{\'o}n.
	\newblock Intermediate spaces and interpolation, the complex method.
	\newblock {\em Stud. Math.}, 24:113--190, 1964.
	
	\bibitem{carando2014Dirichlet}
	D.~Carando, A.~Defant, D.~Garc{\'{\i}}a, M.~Maestre, and P.~Sevilla-Peris.
	\newblock The {Dirichlet}-{Bohr} radius.
	\newblock {\em Acta Arith.}, 171(1):23--37, 2015.
	
	\bibitem{cerezo2021cost}
	M.~Cerezo, A.~Sone, T.~Volkoff, L.~Cincio, and P.~J. Coles.
	\newblock Cost function dependent barren plateaus in shallow parametrized
	quantum circuits.
	\newblock {\em Nature communications}, 12(1):1--12, 2021.
	
	\bibitem{ChL}
	B.~L. Chalmers and G.~Lewicki.
	\newblock A proof of the {Gr{\"u}nbaum} conjecture.
	\newblock {\em Stud. Math.}, 200(2):103--129, 2010.
	
	\bibitem{collins2006integration}
	B.~Collins and P.~{\'S}niady.
	\newblock Integration with respect to the {Haar} measure on unitary, orthogonal
	and symplectic group.
	\newblock {\em Commun. Math. Phys.}, 264(3):773--795, 2006.
	
	\bibitem{creekmore1981type}
	J.~Creekmore.
	\newblock Type and cotype in {Lorentz} $l_{pq}$ spaces.
	\newblock {\em Indag. Math.}, 43:145--152, 1981.
	
	\bibitem{defant1992tensor}
	A.~Defant and K.~Floret.
	\newblock {\em Tensor norms and operator ideals}, volume 176 of {\em
		North-Holland Math. Stud.}
	\newblock Amsterdam: North-Holland, 1993.
	
	\bibitem{defant2006logarithmic}
	A.~Defant and L.~Frerick.
	\newblock A logarithmic lower bound for multi-dimensional {Bohr} radii.
	\newblock {\em Isr. J. Math.}, 152:17--28, 2006.
	
	\bibitem{defant2011bohr}
	A.~Defant and L.~Frerick.
	\newblock The {Bohr} radius of the unit ball of {{\(\ell_p^n\)}}.
	\newblock {\em J. Reine Angew. Math.}, 660:131--147, 2011.
	
	\bibitem{defant2011bohnenblust}
	A.~Defant, L.~Frerick, J.~Ortega-Cerd{\`a}, M.~Ouna{\"{\i}}es, and K.~Seip.
	\newblock The {Bohnenblust}-{Hille} inequality for homogeneous polynomials is
	hypercontractive.
	\newblock {\em Ann. Math. (2)}, 174(1):485--497, 2011.
	
	\bibitem{defant2003bohr}
	A.~Defant, D.~Garc{\'{\i}}a, and M.~Maestre.
	\newblock Bohr's power series theorem and local {Banach} space theory.
	\newblock {\em J. Reine Angew. Math.}, 557:173--197, 2003.
	
	\bibitem{defant2004estimates}
	A.~Defant, D.~Garc{\'{\i}}a, and M.~Maestre.
	\newblock Estimates for the first and second {Bohr} radii of {Reinhardt}
	domains.
	\newblock {\em J. Approx. Theory}, 128(1):53--68, 2004.
	
	\bibitem{defant2004maximum}
	A.~Defant, D.~Garc{\'{\i}}a, and M.~Maestre.
	\newblock Maximum moduli of unimodular polynomials.
	\newblock {\em J. Korean Math. Soc.}, 41(1):209--229, 2004.
	
	\bibitem{defant2019libro}
	A.~Defant, D.~Garc{\'{\i}}a, M.~Maestre, and P.~Sevilla-Peris.
	\newblock {\em Dirichlet series and holomorphic functions in high dimensions},
	volume~37 of {\em New Math. Monogr.}
	\newblock Cambridge: Cambridge University Press, 2019.
	
	\bibitem{defant2020subgaussian}
	A.~Defant and M.~Masty{\l}o.
	\newblock Subgaussian {K}ahane-{S}alem-{Z}ygmund inequalities in {B}anach
	spaces.
	\newblock {\em arXiv preprint arXiv:2008.04429}, 2020.
	
	\bibitem{defant2018bohrBoolean}
	A.~Defant, M.~Masty{\l}o, and A.~P{\'e}rez.
	\newblock Bohr's phenomenon for functions on the {Boolean} cube.
	\newblock {\em J. Funct. Anal.}, 275(11):3115--3147, 2018.
	
	\bibitem{defant2019fourier}
	A.~Defant, M.~Masty{\l}o, and A.~P{\'e}rez.
	\newblock On the {Fourier} spectrum of functions on {Boolean} cubes.
	\newblock {\em Math. Ann.}, 374(1-2):653--680, 2019.
	
	\bibitem{defant2018bohr}
	A.~Defant, M.~Masty{\l}o, and S.~Schl{\"u}ters.
	\newblock On {Bohr} radii of finite dimensional complex {Banach} spaces.
	\newblock {\em Funct. Approximatio, Comment. Math.}, 59(2):251--268, 2018.
	
	\bibitem{DefantSchoolmann2018}
	A.~Defant and I.~Schoolmann.
	\newblock Hardy spaces of general {Dirichlet} series --- a survey.
	\newblock In {\em Function spaces XII. Selected papers based on the
		presentations at the 12th conference, Krakow, Poland, July 9--14, 2018},
	pages 123--149. Warsaw: Polish Academy of Sciences, Institute of Mathematics,
	2019.
	
	\bibitem{defantschoolmann2019Hptheory}
	A.~Defant and I.~Schoolmann.
	\newblock {{\(\mathcal{H}_p \)}}-theory of general {Dirichlet} series.
	\newblock {\em J. Fourier Anal. Appl.}, 25(6):3220--3258, 2019.
	
	\bibitem{defant2020riesz}
	A.~Defant and I.~Schoolmann.
	\newblock Riesz means in {Hardy} spaces on {Dirichlet} groups.
	\newblock {\em Math. Ann.}, 378(1-2):57--96, 2020.
	
	\bibitem{defant2020variants}
	A.~Defant and I.~Schoolmann.
	\newblock Variants of a theorem of {Helson} on general {Dirichlet} series.
	\newblock {\em J. Funct. Anal.}, 279(5):36, 2020.
	\newblock Id/No 108569.
	
	\bibitem{diaconis2003patterns}
	P.~Diaconis.
	\newblock Patterns in eigenvalues: the 70th {Josiah} {Willard} {Gibbs} lecture.
	\newblock {\em Bull. Am. Math. Soc., New Ser.}, 40(2):155--178, 2003.
	
	\bibitem{diaconis1994eigenvalues}
	P.~Diaconis and M.~Shahshahani.
	\newblock On the eigenvalues of random matrices.
	\newblock {\em J. Appl. Probab.}, 31A:49--62, 1994.
	
	\bibitem{Diestel}
	J.~Diestel.
	\newblock A survey of results related to the {Dunford}-{Pettis} property.
	\newblock Integration, topology, and geometry in linear spaces, {Proc}.
	{Conf}., {Chapel} {Hill}/{N}.{C}. 1979, {Contemp}. {Math}. 2, 15-60., 1980.
	
	\bibitem{diestel2001operator}
	J.~Diestel, H.~Jarchow, and A.~Pietsch.
	\newblock Operator ideals.
	\newblock In {\em Handbook of the geometry of Banach spaces. Volume 1}, pages
	437--496. Amsterdam: Elsevier, 2001.
	
	\bibitem{diestel1995absolutely}
	J.~Diestel, H.~Jarchow, and A.~Tonge.
	\newblock {\em Absolutely summing operators}, volume~43 of {\em Camb. Stud.
		Adv. Math.}
	\newblock Cambridge: Cambridge University Press, paperback reprint of the
	hardback edition 1995 edition, 2008.
	
	\bibitem{dineen1999complex}
	S.~Dineen.
	\newblock {\em Complex analysis on infinite dimensional spaces}.
	\newblock Springer Monographs in Mathematics. London: Springer, 1999.
	
	\bibitem{dreseler1974charshiladze}
	B.~Dreseler and W.~Schempp.
	\newblock On the {C}harshiladze-{L}ozinski theorem for compact topological
	groups and homogeneous spaces.
	\newblock {\em Manuscripta Mathematica}, 13(4):321--337, 1974.
	
	\bibitem{dreseler1975convergence}
	B.~Dreseler and W.~Schempp.
	\newblock On the convergence and divergence behaviour of approximation
	processes in homogeneous banach spaces.
	\newblock {\em Math. Z.}, 143(1):81--89, 1975.
	
	\bibitem{duren1970}
	P.~L. Duren.
	\newblock Theory of {{\(H^ p\)}} spaces.
	\newblock New {York} and {London}: {Academic} {Press} {XII}., 1970.
	
	\bibitem{ErdoesSarkozy}
	P.~Erd{\H{o}}s and A.~Sark{\"o}zy.
	\newblock On the number of prime factors of integers.
	\newblock {\em Acta Sci. Math.}, 42:237--246, 1980.
	
	\bibitem{faber1914interpolatorische}
	G.~Faber.
	\newblock {\"U}ber die interpolatorische {Darstellung} stetiger {Funktionen}.
	\newblock {\em Jahresber. Dtsch. Math.-Ver.}, 23:192--210, 1914.
	
	\bibitem{fejer1930}
	L.~Fej{\'e}r.
	\newblock Die {Absch{\"a}tzung} eines {Polynoms} in einem {Intervalle}, wenn
	{Schranken} f{\"u}r seine {Werte} und ersten {Ableitungswerte} in einzelnen
	{Punkten} des {Intervalles} gegeben sind, und ihre {Anwendung} auf die
	{Konvergenzfrage} {Hermitescher} {Interpolationsreihen}.
	\newblock {\em Math. Z.}, 32:426--457, 1930.
	
	\bibitem{fejer1921}
	L.~Fej{\'e}r and F.~Riesz.
	\newblock {\"U}ber einige funktionentheoretische {Ungleichungen}.
	\newblock {\em Math. Z.}, 11:305--314, 1921.
	
	\bibitem{fichtenholz1934operations}
	G.~Fichtenholz and L.~Kantorovitch.
	\newblock Sur les op{\'e}rations lin{\'e}aires dans l'espace des fonctions
	bornees.
	\newblock {\em Stud. Math.}, 5:69--98, 1934.
	
	\bibitem{FLM}
	T.~Figiel, J.~Lindenstrauss, and V.~D. Milman.
	\newblock The dimension of almost spherical sections of convex bodies.
	\newblock {\em Acta Math.}, 139:53--94, 1977.
	
	\bibitem{floret1997natural}
	K.~Floret.
	\newblock Natural norms on symmetric tensor products of normed spaces.
	\newblock {\em Note Mat.}, 17:153--188, 1997.
	
	\bibitem{galicer2021monomial}
	D.~Galicer, M.~Mansilla, S.~Muro, and P.~Sevilla-Peris.
	\newblock Monomial convergence on {{\(\ell_r\)}}.
	\newblock {\em Anal. PDE}, 14(3):945--984, 2021.
	
	\bibitem{garlinggordon}
	D.~J.~H. Garling and Y.~Gordon.
	\newblock Relations between some constants associated with finite dimensional
	{Banach} spaces.
	\newblock {\em Isr. J. Math.}, 9:346--361, 1971.
	
	\bibitem{gordon}
	Y.~Gordon.
	\newblock On the projection and {Macphail} constants of {{\(l_{n}^{p}\)}}
	spaces.
	\newblock {\em Isr. J. Math.}, 6:295--302, 1968.
	
	\bibitem{gordon1974absolutely}
	Y.~Gordon and D.~R. Lewis.
	\newblock Absolutely summing operators and local unconditional structures.
	\newblock {\em Acta Math.}, 133:27--48, 1974.
	
	\bibitem{gordon1982some}
	Y.~Gordon and S.~Reisner.
	\newblock Some geometrical properties of {Banach} spaces of polynomials.
	\newblock {\em Isr. J. Math.}, 42:99--116, 1982.
	
	\bibitem{gowers1993unconditional}
	W.~T. Gowers and B.~Maurey.
	\newblock The unconditional basic sequence problem.
	\newblock {\em J. Am. Math. Soc.}, 6(4):851--874, 1993.
	
	\bibitem{gowers1997banach}
	W.~T. Gowers and B.~Maurey.
	\newblock Banach spaces with small spaces of operators.
	\newblock {\em Math. Ann.}, 307(4):543--568, 1997.
	
	\bibitem{grahamhare2013}
	C.~C. Graham and K.~E. Hare.
	\newblock {\em Interpolation and {Sidon} sets for compact groups}.
	\newblock CMS Books Math./Ouvrages Math. SMC. New York, NY: Springer, 2013.
	
	\bibitem{Grothendieck1953}
	A.~Grothendieck.
	\newblock Sur les applications lin{\'e}aires faiblement compactes d'espaces du
	type {{\(C(K)\)}}.
	\newblock {\em Can. J. Math.}, 5:129--173, 1953.
	
	\bibitem{grunbaum}
	B.~Gr{\"u}nbaum.
	\newblock Projection constants.
	\newblock {\em Trans. Am. Math. Soc.}, 95:451--465, 1960.
	
	\bibitem{Guth2015}
	L.~Guth.
	\newblock {\em Polynomial methods in combinatorics}, volume~64 of {\em Univ.
		Lect. Ser.}
	\newblock Providence, RI: American Mathematical Society (AMS), 2016.
	
	\bibitem{Guth2016}
	L.~Guth.
	\newblock A restriction estimate using polynomial partitioning.
	\newblock {\em J. Am. Math. Soc.}, 29(2):371--413, 2016.
	
	\bibitem{Guth2018}
	L.~Guth.
	\newblock Restriction estimates using polynomial partitioning. {II}.
	\newblock {\em Acta Math.}, 221(1):81--142, 2018.
	
	\bibitem{GuthKatz}
	L.~Guth and N.~H. Katz.
	\newblock On the {Erd{\H{o}}s} distinct distances problem in the plane.
	\newblock {\em Ann. Math. (2)}, 181(1):155--190, 2015.
	
	\bibitem{HR2022}
	P.~H{\'a}jek and T.~Russo.
	\newblock Projecting {L}ipschitz functions onto spaces of polynomials.
	\newblock {\em Mediterr. J. Math.}, 19(4):1--22, 2022.
	
	\bibitem{hardy1952inequalities}
	G.~H. Hardy, J.~E. Littlewood, and G.~Polya.
	\newblock Inequalities cambridge univ.
	\newblock {\em Press, Cambridge}, (1988), 1952.
	
	\bibitem{HardyRiesz}
	G.~H. Hardy and M.~Riesz.
	\newblock The general theory of {D}irichlet's series.
	\newblock {\em Cambridge tracts in mathematics and mathematical physics; no.
		18.}, 1915.
	
	\bibitem{Harper}
	A.~J. Harper.
	\newblock Moments of random multiplicative functions. {II}: {High} moments.
	\newblock {\em Algebra Number Theory}, 13(10):2277--2321, 2019.
	
	\bibitem{H}
	L.~A. Harris.
	\newblock Bounds on the derivatives of holomorphic functions of vectors.
	\newblock Analyse fonct. {Appl}., {C}. r. {Colloq}. d'{Analyse}, {Rio} de
	{Janeiro} 1972, 145-163 (1975)., 1975.
	
	\bibitem{harris1997holomorphic}
	L.~A. Harris.
	\newblock Holomorphic mappings of domains in operator spaces.
	\newblock {\em Nonlinear Anal., Theory Methods Appl.}, 30(6):3493--3503, 1997.
	
	\bibitem{hedenmalm1997hilbert}
	H.~Hedenmalm, P.~Lindqvist, and K.~Seip.
	\newblock A {Hilbert} space of {Dirichlet} series and systems of dilated
	functions in {{\(L^ 2(0,1)\)}}.
	\newblock {\em Duke Math. J.}, 86(1):1--37, 1997.
	
	\bibitem{HelsonBook}
	H.~Helson.
	\newblock {\em Dirichlet series}.
	\newblock Berkeley, CA: Henry Helson, 2005.
	
	\bibitem{hewitt2013abstract}
	E.~Hewitt and K.~A. Ross.
	\newblock {\em Abstract harmonic analysis. {Volume} {I}: {Structure} of
		topological groups, integration theory, group representations.}, volume 115
	of {\em Grundlehren Math. Wiss.}
	\newblock Berlin: Springer-Verlag, 2nd ed. edition, 1994.
	
	\bibitem{hiai2000semicircle}
	F.~Hiai and D.~Petz.
	\newblock {\em The semicircle law, free random variables and entropy},
	volume~77 of {\em Math. Surv. Monogr.}
	\newblock Providence, RI: American Mathematical Society (AMS), 2000.
	
	\bibitem{ismagilov2020sharp}
	A.~Ismagilov, I.~R. Kayumov, and S.~Ponnusamy.
	\newblock Sharp {Bohr} type inequality.
	\newblock {\em J. Math. Anal. Appl.}, 489(1):9, 2020.
	\newblock Id/No 124147.
	
	\bibitem{jameson1987summing}
	G.~J.~O. Jameson.
	\newblock {\em Summing and nuclear norms in {Banach} space theory}, volume~8 of
	{\em Lond. Math. Soc. Stud. Texts}.
	\newblock Cambridge University Press, Cambridge, 1987.
	
	\bibitem{Jameson}
	G.~J.~O. Jameson.
	\newblock The {{\(q\)}}-concavity constants of {Lorentz} sequence spaces and
	related inequalities.
	\newblock {\em Math. Z.}, 227(1):129--142, 1998.
	
	\bibitem{johansson1997random}
	K.~Johansson.
	\newblock On random matrices from the compact classical groups.
	\newblock {\em Ann. Math. (2)}, 145(3):519--545, 1997.
	
	\bibitem{kadecsnobar}
	M.~I. Kadets and M.~G. Snobar.
	\newblock Some functionals over a compact {Minkowski} space.
	\newblock {\em Math. Notes}, 10:694--696, 1972.
	
	\bibitem{KM2000}
	A.~Kami{\'n}ska and M.~Masty{\l}o.
	\newblock The {Dunford}-{Pettis} property for symmetric spaces.
	\newblock {\em Can. J. Math.}, 52(4):789--803, 2000.
	
	\bibitem{katznelson2004introduction}
	Y.~Katznelson.
	\newblock {\em An introduction to harmonic analysis}.
	\newblock Camb. Math. Libr. Cambridge: Cambridge University Press, 3rd ed.
	edition, 2004.
	
	\bibitem{kayumov2017bohr}
	I.~R. Kayumov and S.~Ponnusamy.
	\newblock Bohr inequality for odd analytic functions.
	\newblock {\em Comput. Methods Funct. Theory}, 17(4):679--688, 2017.
	
	\bibitem{kayumov2018bohr}
	I.~R. Kayumov and S.~Ponnusamy.
	\newblock Bohr's inequalities for the analytic functions with lacunary series
	and harmonic functions.
	\newblock {\em J. Math. Anal. Appl.}, 465(2):857--871, 2018.
	
	\bibitem{klimek1995metrics}
	M.~Klimek.
	\newblock Metrics associated with extremal plurisubharmonic functions.
	\newblock {\em Proc. Am. Math. Soc.}, 123(9):2763--2770, 1995.
	
	\bibitem{konig1985spaces}
	H.~K{\"o}nig.
	\newblock Spaces with large projection constants.
	\newblock S{\'e}min. d'analyse fonctionnelle, {Paris} 1984/1985, {Publ}.
	{Math}. {Univ}. {Paris} {VII} 26, 31-36 (1985)., 1985.
	
	\bibitem{konig1995projections}
	H.~K{\"o}nig.
	\newblock Projections onto symmetric spaces.
	\newblock {\em Quaest. Math.}, 18(1-3):199--220, 1995.
	
	\bibitem{koniglewis}
	H.~K{\"o}nig and D.~R. Lewis.
	\newblock A strict inequality for projection constants.
	\newblock {\em J. Funct. Anal.}, 74:328--332, 1987.
	
	\bibitem{konig1999projection}
	H.~K{\"o}nig, C.~Sch{\"u}tt, and N.~Tomczak-Jaegermann.
	\newblock Projection constants of symmetric spaces and variants of
	{Khintchine}'s inequality.
	\newblock {\em J. Reine Angew. Math.}, 511:1--42, 1999.
	
	\bibitem{konyagin2001translation}
	S.~V. Konyagin and H.~Queff{\'e}lec.
	\newblock The translation {{\(\frac 12\)}} in the theory of {Dirichlet} series.
	\newblock {\em Real Anal. Exch.}, 27(1):155--175, 2002.
	
	\bibitem{kostenberger2021weingarten}
	G.~K{\"o}stenberger.
	\newblock Weingarten calculus.
	\newblock {\em arXiv preprint arXiv:2101.00921}, 2021.
	
	\bibitem{krasnosel1960convex}
	M.~A. Krasnosel'skii.
	\newblock {\em Convex functions and Orlicz spaces}, volume 4311.
	\newblock US Atomic Energy Commission, 1960.
	
	\bibitem{krivine1978constantes}
	J.~L. Krivine.
	\newblock Constantes de {Grothendieck} et fonctions de type positif sur les
	sph{\`e}res.
	\newblock {\em Adv. Math.}, 31:16--30, 1979.
	
	\bibitem{schuettkwapien}
	S.~Kwapi{\'e}n and C.~Sch{\"u}tt.
	\newblock Some combinatorial and probabilistic inequalities and their
	application to {Banach} space theory.
	\newblock {\em Stud. Math.}, 82:81--106, 1985.
	
	\bibitem{lambert1969minimum}
	P.~V. Lambert.
	\newblock Minimum norm property of the {Fourier} projection in spaces of
	continuous functions.
	\newblock {\em Bull. Soc. Math. Belg.}, 21:359--369, 1969.
	
	\bibitem{lambert1970minimum}
	P.~V. Lambert.
	\newblock On the minimum norm property of the {Fourier} projection in {{\(L^
			1\)}}-spaces and in spaces of continuous functions.
	\newblock {\em Bull. Am. Math. Soc.}, 76:798--804, 1970.
	
	\bibitem{lewickimastylo}
	G.~Lewicki and M.~Masty{\l}o.
	\newblock Asymptotic behavior of factorization and projection constants.
	\newblock {\em J. Geom. Anal.}, 27(1):335--365, 2017.
	
	\bibitem{Lewis1975}
	D.~R. Lewis.
	\newblock A relation between diagonal and unconditonal basis constants.
	\newblock {\em Math. Ann.}, 218:193--198, 1975.
	
	\bibitem{lewis}
	D.~R. Lewis.
	\newblock An upper bound for the projection constant.
	\newblock {\em Proc. Am. Math. Soc.}, 103(4):1157--1160, 1988.
	
	\bibitem{light1986minimal}
	W.~A. Light.
	\newblock Minimal projections in tensor-product spaces.
	\newblock {\em Math. Z.}, 191:633--643, 1986.
	
	\bibitem{Lindenstrauss}
	J.~Lindenstrauss.
	\newblock On complemented subspaces of $m$.
	\newblock {\em Isr. J. Math.}, 5(3):153--156, 1967.
	
	\bibitem{LT71}
	J.~Lindenstrauss and L.~Tzafriri.
	\newblock On the complemented subspaces problem.
	\newblock {\em Isr. J. Math.}, 9:263--269, 1971.
	
	\bibitem{LT1}
	J.~Lindenstrauss and L.~Tzafriri.
	\newblock {\em Classical {Banach} spaces {I}. {Sequence} spaces}, volume~92 of
	{\em Ergeb. Math. Grenzgeb.}
	\newblock Springer-Verlag, Berlin, 1977.
	
	\bibitem{LT}
	J.~Lindenstrauss and L.~Tzafriri.
	\newblock {\em Classical {Banach} spaces. {II}: {Function} spaces}, volume~97
	of {\em Ergeb. Math. Grenzgeb.}
	\newblock Springer-Verlag, Berlin, 1979.
	
	\bibitem{littlewood1930bounded}
	J.~E. Littlewood.
	\newblock On bounded bilinear forms in an infinite number of variables.
	\newblock {\em Q. J. Math., Oxf. Ser.}, 1:164--174, 1930.
	
	\bibitem{Loz}
	G.~J. Lozanovskii.
	\newblock On some {Banach} lattices. {IV}.
	\newblock {\em Sib. Math. J.}, 14:97--108, 1973.
	
	\bibitem{lozinski1948class}
	S.~Lozinski.
	\newblock On a class of linear operators.
	\newblock In {\em Dokl. Akad. Nauk SSSR}, volume~61, pages 193--196, 1948.
	
	\bibitem{17236}
	M.~Lugo.
	\newblock Sum of the first k binomial coefficients for fixed n.
	\newblock MathOverflow.
	\newblock URL:https://mathoverflow.net/q/17236 (version: 2017-10-01).
	
	\bibitem{mansilla2019thesis}
	M.~I. Mansilla.
	\newblock {\em Monomial decomposition and summability for holomorphic
		functions}.
	\newblock PhD thesis, Universidad de Buenos Aires, 2019.
	
	\bibitem{marcinkiewicz1937quelques}
	J.~Marcinkiewicz.
	\newblock Quelques remarques sur l'interpolation.
	\newblock {\em Acta Litt. Sci. Szeged}, 8:127--130, 1937.
	
	\bibitem{mastylo2017kahane}
	M.~Masty{\l}o and R.~Szwedek.
	\newblock Kahane-{Salem}-{Zygmund} polynomial inequalities via {Rademacher}
	processes.
	\newblock {\em J. Funct. Anal.}, 272(11):4483--4512, 2017.
	
	\bibitem{mcgehee1981}
	O.~C. McGehee, L.~Pigno, and B.~Smith.
	\newblock Hardy's inequality and the {{\(L^ 1\)}} norm of exponential sums.
	\newblock {\em Ann. Math. (2)}, 113:613--618, 1981.
	
	\bibitem{mckay1989littlewood}
	B.~D. McKay.
	\newblock On {Littlewood}'s estimate for the binomial distribution.
	\newblock {\em Adv. Appl. Probab.}, 21(2):475--478, 1989.
	
	\bibitem{meckes2008linear}
	E.~Meckes.
	\newblock Linear functions on the classical matrix groups.
	\newblock {\em Trans. Am. Math. Soc.}, 360(10):5355--5366, 2008.
	
	\bibitem{Murray}
	F.~J. Murray.
	\newblock On complementary manifolds and projections in spaces {{\(L_p\)}} and
	{{\(l_p\)}}.
	\newblock {\em Trans. Am. Math. Soc.}, 41:138--152, 1937.
	
	\bibitem{natanson1961constructive}
	I.~P. Natanson.
	\newblock {\em Constructive theory of functions}, volume~1.
	\newblock US Atomic Energy Commission, Office of Technical Information
	Extension, 1961.
	
	\bibitem{nelson1961distinguished}
	E.~Nelson.
	\newblock The distinguished boundary of the unit operator ball.
	\newblock {\em Proc. Am. Math. Soc.}, 12:994--995, 1961.
	
	\bibitem{o2008some}
	R.~O'Donnell.
	\newblock Some topics in analysis of {Boolean} functions.
	\newblock In {\em Proceedings of the 40th annual ACM symposium on theory of
		computing, STOC 2008. Victoria, Canada, May 17--20, 2008}, pages 569--578.
	New York, NY: Association for Computing Machinery (ACM), 2008.
	
	\bibitem{o2014analysis}
	R.~O'Donnell.
	\newblock {\em Analysis of {Boolean} functions}.
	\newblock Cambridge: Cambridge University Press, 2014.
	
	\bibitem{ortega2009sidon}
	J.~Ortega-Cerd{\`a}, M.~Ouna{\"\i}es, and K.~Seip.
	\newblock The {S}idon constant for homogeneous polynomials.
	\newblock {\em arXiv preprint arXiv:0903.1455}, 2009.
	
	\bibitem{pastur2011eigenvalue}
	L.~Pastur and M.~Shcherbina.
	\newblock {\em Eigenvalue distribution of large random matrices}, volume 171 of
	{\em Math. Surv. Monogr.}
	\newblock Providence, RI: American Mathematical Society (AMS), 2011.
	
	\bibitem{paulsen2002bohr}
	V.~I. Paulsen, G.~Popescu, and D.~Singh.
	\newblock On {Bohr}'s inequality.
	\newblock {\em Proc. Lond. Math. Soc. (3)}, 85(2):493--512, 2002.
	
	\bibitem{paulsen2004bohr}
	V.~I. Paulsen and D.~Singh.
	\newblock Bohr's inequality for uniform algebras.
	\newblock {\em Proc. Am. Math. Soc.}, 132(12):3577--3579, 2004.
	
	\bibitem{pelczynski1960projections}
	A.~Pe{\l}czy{\'n}ski.
	\newblock Projections in certain {Banach} spaces.
	\newblock {\em Stud. Math.}, 19:209--228, 1960.
	
	\bibitem{Pfitzner1994}
	H.~Pfitzner.
	\newblock Weak compactness in the dual of a $c^*$-algebra is determined
	commutatively.
	\newblock {\em Math. Ann.}, 298(2):349--371, 1994.
	
	\bibitem{pietsch1986eigenvalues}
	A.~Pietsch.
	\newblock {\em Eigenvalues and s-numbers. ({Licensed} ed.)}, volume~13 of {\em
		Camb. Stud. Adv. Math.}
	\newblock Cambridge University Press, Cambridge, 1987.
	
	\bibitem{pisier1978some}
	G.~Pisier.
	\newblock Some results on {Banach} spaces without local unconditional
	structure.
	\newblock {\em Compos. Math.}, 37:3--19, 1978.
	
	\bibitem{pisier1986factorization}
	G.~Pisier.
	\newblock {\em Factorization of linear operators and geometry of {Banach}
		spaces}, volume~60 of {\em Reg. Conf. Ser. Math.}
	\newblock Providence, RI: American Mathematical Society (AMS), 1986.
	
	\bibitem{queffelec2013diophantine}
	H.~Queffelec and M.~Queffelec.
	\newblock {\em Diophantine approximation and {Dirichlet} series}, volume~80 of
	{\em Texts Read. Math.}
	\newblock New Delhi: Hindustan Book Agency; Singapore: Springer, 2nd extended
	edition edition, 2020.
	
	\bibitem{Reisner}
	S.~Reisner.
	\newblock A factorization theorem in {Banach} lattices and its application to
	{Lorentz} spaces.
	\newblock {\em Ann. Inst. Fourier}, 31(1):239--255, 1981.
	
	\bibitem{rieffel2006lipschitz}
	M.~A. Rieffel.
	\newblock Lipschitz extension constants equal projection constants.
	\newblock In {\em Operator theory, operator algebras, and applications.
		Proceedings of the 25th Great Plains Operator Theory Symposium, University of
		Central Florida, FL, USA, June 7--12, 2005}, pages 147--162. Providence, RI:
	American Mathematical Society (AMS), 2006.
	
	\bibitem{Rosenthal1972}
	H.~P. Rosenthal.
	\newblock On factors of {C}({{\([0,1]\)}}) with non-separable dual.
	\newblock {\em Isr. J. Math.}, 13, 1973.
	
	\bibitem{Rosenthal2003}
	H.~P. Rosenthal.
	\newblock The {Banach} spaces {{\(C(K)\)}}.
	\newblock In {\em Handbook of the geometry of Banach spaces. Volume 2}, pages
	1547--1602. Amsterdam: North-Holland, 2003.
	
	\bibitem{rudin1959some}
	W.~Rudin.
	\newblock Some theorems on {Fourier} coefficients.
	\newblock {\em Proc. Am. Math. Soc.}, 10:855--859, 1959.
	
	\bibitem{rudin1960trigonometric}
	W.~Rudin.
	\newblock Trigonometric series with gaps.
	\newblock {\em J. Math. Mech.}, 9:203--227, 1960.
	
	\bibitem{rudin1962projections}
	W.~Rudin.
	\newblock Projections on invariant subspaces.
	\newblock {\em Proc. Am. Math. Soc.}, 13:429--432, 1962.
	
	\bibitem{rudin1985}
	W.~Rudin.
	\newblock The {Ryll}-{Wojtaszczyk} polynomials.
	\newblock {\em Ann. Pol. Math.}, 46:291--294, 1985.
	
	\bibitem{rudin1986new}
	W.~Rudin.
	\newblock {\em New constructions of functions holomorphic in the unit ball of
		{{\(C^n\)}}}, volume~63 of {\em Reg. Conf. Ser. Math.}
	\newblock Providence, RI: American Mathematical Society (AMS), 1986.
	
	\bibitem{rudin1980}
	W.~Rudin.
	\newblock {\em Function theory in the unit ball of {{\({\mathbb{C}}^ n\)}}}.
	\newblock Class. Math. Berlin: Springer, reprint of the 1980 original edition,
	2008.
	
	\bibitem{rutovitz}
	D.~Rutovitz.
	\newblock Some parameters associated with finite-dimensional {Banach} spaces.
	\newblock {\em J. Lond. Math. Soc.}, 40:241--255, 1965.
	
	\bibitem{ryll1983homogeneous}
	J.~Ryll and P.~Wojtaszczyk.
	\newblock On homogeneous polynomials on a complex ball.
	\newblock {\em Trans. Am. Math. Soc.}, 276:107--116, 1983.
	
	\bibitem{schep2010products}
	A.~R. Schep.
	\newblock Products and factors of {Banach} function spaces.
	\newblock {\em Positivity}, 14(2):301--319, 2010.
	
	\bibitem{schutt1978projection}
	C.~Sch{\"u}tt.
	\newblock The projection constant of finite-dimensional spaces whose
	unconditional basis constant is 1.
	\newblock {\em Isr. J. Math.}, 30:207--212, 1978.
	
	\bibitem{shapiro1952extremal}
	H.~S. Shapiro.
	\newblock {\em Extremal problems for polynomials and power series}.
	\newblock PhD thesis, Massachusetts Institute of Technology, 1952.
	
	\bibitem{Sobczyk}
	A.~Sobczyk.
	\newblock Projections in {Minkowski} and {Banach} spaces.
	\newblock {\em Duke Math. J.}, 8:78--106, 1941.
	
	\bibitem{szarek1999nonsymmetric}
	S.~J. Szarek and E.~Werner.
	\newblock A nonsymmetric correlation inequality for {Gaussian} measure.
	\newblock {\em J. Multivariate Anal.}, 68(2):193--211, 1999.
	
	\bibitem{tenenbaum1995introduction}
	G.~Tenenbaum.
	\newblock {\em Introduction {\`a} la th{\'e}orie analytique et probabiliste des
		nombres}, volume~1 of {\em Cours Sp{\'e}c. (Paris)}.
	\newblock Paris: Soci{\'e}t{\'e} Math{\'e}matique de France, 2{\`e}me {\'e}d.
	edition, 1995.
	
	\bibitem{tomczak1989banach}
	N.~Tomczak-Jaegermann.
	\newblock {\em Banach-{Mazur} distances and finite-dimensional operator
		ideals}, volume~38 of {\em Pitman Monogr. Surv. Pure Appl. Math.}
	\newblock Harlow: Longman Scientific \&| Technical; New York: John Wiley \&|
	Sons, Inc., 1989.
	
	\bibitem{visser1946generalization}
	C.~Visser.
	\newblock A generalization of {Tchebychef}'s inequality to polynomials in more
	than one variable.
	\newblock {\em Proc. Akad. Wet. Amsterdam}, 49:455--456, 1946.
	
	\bibitem{weissler1980logarithmic}
	F.~B. Weissler.
	\newblock Logarithmic {Sobolev} inequalities and hypercontractive estimates on
	the circle.
	\newblock {\em J. Funct. Anal.}, 37:218--234, 1980.
	
	\bibitem{wojtaszczyk1996banach}
	P.~Wojtaszczyk.
	\newblock {\em Banach spaces for analysts}.
	\newblock Cambridge: Cambridge Univ. Press, 1996.
	
	\bibitem{zhang2014matrix}
	L.~Zhang.
	\newblock Matrix integrals over unitary groups: An application of schur-weyl
	duality.
	\newblock {\em arXiv preprint arXiv:1408.3782}, 2014.
	
\end{thebibliography}

\vspace{1.5cm}

\noindent
Andreas~Defant\\
Institut f\"{u}r Mathematik, Carl von Ossietzky Universit\"at, 26111 Oldenburg, Germany\\
E-mail:
\texttt{andreas.defant$@$uni-oldenburg.de}

\smallskip 

\noindent Daniel~Galicer\\
Departamento de Matem\'{a}tica,
Facultad de Cs. Exactas y Naturales, Universidad de Buenos Aires and IMAS-CONICET. Ciudad Universitaria, Pabell\'on I (C1428EGA) C.A.B.A., Argentina\\ 
E-mail:
\texttt{dgalicer$@$dm.uba.ar}

\smallskip

\noindent Mart\'in~Mansilla\\
Departamento de Matem\'{a}tica,
Facultad de Cs. Exactas y Naturales, Universidad de Buenos Aires and IMAS-CONICET. Ciudad Universitaria, Pabell\'on I (C1428EGA) C.A.B.A., Argentina\\ 
E-mail:
\texttt{mmansilla$@$dm.uba.ar}

\smallskip 

\noindent Mieczys{\l}aw~Masty{\l}o\\
Faculty of Mathematics and Computer Science, Adam Mickiewicz University, Pozna{\'n}, Uniwersytetu Pozna{\'n}skiego 4,\linebreak 61-614 Pozna{\'n}, Poland\\ 
E-mail:
\texttt{mieczyslaw.mastylo$@$amu.edu.pl}

\smallskip 

\noindent Santiago~Muro\\
FCEIA, Universidad Nacional de Rosario and CIFASIS, CONICET, Ocampo $\&$ Esmeralda, S2000 Rosario, Argentina\\
E-mail:
\texttt{muro$@$cifasis-conicet.gov.ar}

\end{document}